\newif\ifnotanonymous
\notanonymoustrue
\documentclass[10pt]{article}
\usepackage{hyperref}
\usepackage{amsmath}
\usepackage{amsfonts}
\usepackage{amssymb}
\usepackage{amsthm}

\usepackage[shortlabels]{enumitem}
\usepackage{microtype}
\usepackage{fullpage}
\usepackage{mathrsfs}
\usepackage{mleftright}

\mleftright

\usepackage{xparse}

\NewDocumentCommand{\R}{}{\mathbb{R}}
\NewDocumentCommand{\C}{}{\mathbb{C}}

\NewDocumentCommand{\Z}{}{\mathbb{Z}}

\NewDocumentCommand{\transpose}{}{\top}

\NewDocumentCommand{\sE}{}{\mathcal{E}}
\NewDocumentCommand{\sF}{}{\mathcal{F}}
\NewDocumentCommand{\sG}{}{\mathcal{G}}
\NewDocumentCommand{\sD}{}{\mathcal{D}}
\NewDocumentCommand{\sDt}{}{\widetilde{\sD}}
\NewDocumentCommand{\sDh}{}{\widehat{\sD}}
\NewDocumentCommand{\sEh}{}{\widehat{\sE}}
\NewDocumentCommand{\sEt}{}{\widetilde{\sE}}
\NewDocumentCommand{\sN}{}{\mathcal{N}}
\NewDocumentCommand{\sM}{}{\mathcal{M}}
\NewDocumentCommand{\Compactt}{}{\widetilde{\Compact}}

\NewDocumentCommand{\psih}{}{\hat{\psi}}
\NewDocumentCommand{\phih}{}{\hat{\phi}}
\NewDocumentCommand{\psit}{}{\tilde{\psi}}
\NewDocumentCommand{\phit}{}{\tilde{\phi}}
\NewDocumentCommand{\Dt}{}{\widetilde{D}}
\NewDocumentCommand{\Dh}{}{\widehat{D}}
\NewDocumentCommand{\Fh}{}{\widehat{F}}
\NewDocumentCommand{\Ph}{}{\widehat{P}}
\NewDocumentCommand{\Et}{}{\widetilde{E}}

\NewDocumentCommand{\Eh}{}{\widehat{E}}
\NewDocumentCommand{\Ft}{}{\widetilde{F}}
\NewDocumentCommand{\Psih}{}{\widehat{\Psi}}
\NewDocumentCommand{\xt}{}{\tilde{x}}
\NewDocumentCommand{\yt}{}{\tilde{y}}
\NewDocumentCommand{\fh}{}{\hat{f}}
\NewDocumentCommand{\ah}{}{\hat{a}}
\NewDocumentCommand{\at}{}{\tilde{a}}
\NewDocumentCommand{\bt}{}{\tilde{b}}
\NewDocumentCommand{\ct}{}{\tilde{c}}
\NewDocumentCommand{\ch}{}{\hat{c}}

\NewDocumentCommand{\TestFunctionsZero}{o}{\TestFunctionsSymbol_0\IfValueT{#1}{(#1)}}
\NewDocumentCommand{\TestFunctionsZeroN}{}{\TestFunctionsZero[\ManifoldN]}
\NewDocumentCommand{\DistributionsZero}{o}{\TestFunctionsSymbol'_0\IfValueT{#1}{(#1)}}
\NewDocumentCommand{\DistributionsZeroN}{}{\DistributionsZero[\ManifoldN]}

\NewDocumentCommand{\sS}{}{\mathcal{S}}
\NewDocumentCommand{\sT}{}{\mathcal{T}}
\NewDocumentCommand{\sB}{}{\mathcal{B}}

\NewDocumentCommand{\supp}{}{\mathrm{supp}}

\NewDocumentCommand{\DefnIt}{m}{\textit{#1}}

\NewDocumentCommand{\TangentSpace}{m m}{T_{#1}#2}

\NewDocumentCommand{\Vol}{o}{\mathrm{Vol}\IfValueT{#1}{(#1)}}
\NewDocumentCommand{\Volh}{o}{\widehat{\mathrm{Vol}}\IfValueT{#1}{(#1)}}

\NewDocumentCommand{\Omegah}{}{\widehat{\Omega}}
\NewDocumentCommand{\OmegaClosure}{}{\overline{\Omega}}

\NewDocumentCommand{\Rngeq}{}{\R^n_{\geq 0}}
\NewDocumentCommand{\Rng}{}{\R^n_{> 0}}
\NewDocumentCommand{\Rn}{}{\R^n}
\NewDocumentCommand{\Rnmo}{}{\R^{n-1}}
\NewDocumentCommand{\RN}{}{\R^N}

\NewDocumentCommand{\Zgeq}{}{\Z_{\geq 0}}
\NewDocumentCommand{\Zg}{}{\Z_{> 0}}

\NewDocumentCommand{\TraceMap}{o o}{\mathscr{R}\IfValueT{#1}{_{#1}}\IfValueT{#2}{^{#2}}}
\NewDocumentCommand{\TraceInverseMap}{o o}{\mathscr{C}\IfValueT{#1}{_{#1}}\IfValueT{#2}{^{#2}}}

\NewDocumentCommand{\ManifoldM}{}{\mathfrak{M}}
\NewDocumentCommand{\ManifoldN}{}{\mathfrak{N}}
\NewDocumentCommand{\InteriorN}{}{\mathrm{Int}(\ManifoldN)}
\NewDocumentCommand{\BoundaryN}{}{\partial \ManifoldN}

\NewDocumentCommand{\ManifoldNt}{}{\widetilde{\ManifoldN}}

\NewDocumentCommand{\grad}{}{\nabla}

\NewDocumentCommand{\TripleNorm}{m o}{\left\vert\kern-0.25ex\left\vert\kern-0.25ex\left\vert #1 \right\vert\kern-0.25ex\right\vert\kern-0.25ex\right\vert\IfNoValueF{#2}{_{#2}}}

\NewDocumentCommand{\Bngeq}{m}{B^{n}_{\geq}(#1)}
\NewDocumentCommand{\Bng}{m}{B^{n}_{>}(#1)}
\NewDocumentCommand{\Bnleq}{m}{B^{n}_{\leq}(#1)}
\NewDocumentCommand{\Bnl}{m}{B^{n}_{<}(#1)}

\NewDocumentCommand{\Bn}{m}{B^{n}(#1)}
\NewDocumentCommand{\Bnmo}{m}{B^{n-1}(#1)}
\NewDocumentCommand{\BnClosure}{m}{\overline{\Bn{#1}}}
\NewDocumentCommand{\BnmoClosure}{m}{\overline{\Bnmo{#1}}}
\NewDocumentCommand{\BngeqClosure}{m}{\overline{\Bngeq{#1}}}

\NewDocumentCommand{\Qngeq}{m}{Q^{n}_{\geq}(#1)}

\NewDocumentCommand{\Qn}{m}{Q^{n}(#1)}

\NewDocumentCommand{\Qngeqc}{m m}{Q^{n}_{\geq #2}(#1)}
\NewDocumentCommand{\Qneqc}{m m}{Q^{n}_{= #2}(#1)}
\NewDocumentCommand{\QnClosure}{m}{\overline{\Qn{#1}}}
\NewDocumentCommand{\QngeqClosure}{m}{\overline{\Qngeq{#1}}}

\NewDocumentCommand{\TestFunctionsSymbol}{}{\mathscr{D}}
\NewDocumentCommand{\TestFunctions}{o}{\TestFunctionsSymbol\IfNoValueTF{#1}{}{(#1)}}
\NewDocumentCommand{\Distributions}{o}{\TestFunctionsSymbol'\IfNoValueTF{#1}{}{(#1)}}

\NewDocumentCommand{\TemperedDistributions}{o}{\SchwartzSpace'\IfNoValueTF{#1}{}{(#1)}}

\NewDocumentCommand{\GlobalTestFunctionsSymbol}{}{\mathscr{E}}
\NewDocumentCommand{\GlobalTestFunctions}{o}{\GlobalTestFunctionsSymbol\IfNoValueTF{#1}{}{(#1)}}
\NewDocumentCommand{\DistributionsCpt}{o}{\GlobalTestFunctionsSymbol'\IfNoValueTF{#1}{}{(#1)}}

\NewDocumentCommand{\opP}{}{\mathscr{P}}
\NewDocumentCommand{\opL}{}{\mathscr{L}}

\NewDocumentCommand{\Compact}{}{\mathcal{K}}
\NewDocumentCommand{\Compacth}{}{\widehat{\Compact}}

\NewDocumentCommand{\BesovSymbol}{}{\mathscr{B}}
\NewDocumentCommand{\TLSymbol}{}{\mathscr{F}}
\NewDocumentCommand{\XSymbol}{}{\mathscr{A}}

\NewDocumentCommand{\BesovSpace}{m m m o o}{\BesovSymbol^{#1}_{#2,#3}\IfValueT{#4}{(#4\IfValueT{#5}{,#5})}}
\NewDocumentCommand{\TLSpace}{m m m o o}{\TLSymbol^{#1}_{#2,#3}\IfValueT{#4}{(#4\IfValueT{#5}{,#5})}}
\NewDocumentCommand{\XSpace}{m m m o o}{\XSymbol^{#1}_{#2,#3}\IfValueT{#4}{(#4\IfValueT{#5}{,#5})}}
\NewDocumentCommand{\XCircSpace}{m m m o o}{\mathring{\XSymbol}^{#1}_{#2,#3}\IfValueT{#4}{(#4\IfValueT{#5}{,#5})}}
\NewDocumentCommand{\XSpaceClassical}{m m m o o}{\XSymbol^{#1}_{#2,#3,\mathrm{std}}\IfValueT{#4}{(#4\IfValueT{#5}{,#5})}}

\NewDocumentCommand{\BigBesovSpace}{m m m o o}{\BesovSymbol^{#1}_{#2,#3}\IfValueT{#4}{\left(#4\IfValueT{#5}{,#5}\right)}}
\NewDocumentCommand{\BigTLSpace}{m m m o o}{\TLSymbol^{#1}_{#2,#3}\IfValueT{#4}{\left(#4\IfValueT{#5}{,#5}\right)}}

\NewDocumentCommand{\BesovSpaceCpt}{m m m o o}{\BesovSymbol^{#1}_{#2,#3,\mathrm{cpt}}\IfValueT{#4}{(#4\IfValueT{#5}{,#5})}}
\NewDocumentCommand{\TLSpaceCpt}{m m m o o}{\TLSymbol^{#1}_{#2,#3,\mathrm{cpt}}\IfValueT{#4}{(#4\IfValueT{#5}{,#5})}}

\NewDocumentCommand{\ASpace}{m m m o o}{\XSymbol^{#1}_{#2,#3}\IfValueT{#4}{(#4\IfValueT{#5}{,#5})}}
\NewDocumentCommand{\BigASpace}{m m m o o}{\XSymbol^{#1}_{#2,#3}\IfValueT{#4}{\left(#4\IfValueT{#5}{,#5}\right)}}
\NewDocumentCommand{\ACircSpace}{m m m o o}{\mathring{\XSymbol}^{#1}_{#2,#3}\IfValueT{#4}{(#4\IfValueT{#5}{,#5})}}
\NewDocumentCommand{\ASpaceClassical}{m m m o o}{\XSymbol^{#1}_{#2,#3,\mathrm{std}}\IfValueT{#4}{(#4\IfValueT{#5}{,#5})}}
\NewDocumentCommand{\ASpaceCpt}{m m m o o}{\XSymbol^{#1}_{#2,#3,\mathrm{cpt}}\IfValueT{#4}{(#4\IfValueT{#5}{,#5})}}

\NewDocumentCommand{\ANorm}{m m m m o o}{\|#1\|_{\ASpace{#2}{#3}{#4}[#5][#6]}}
\NewDocumentCommand{\BANorm}{m m m m o o}{\left\|#1\right\|_{\ASpace{#2}{#3}{#4}[#5][#6]}}

\NewDocumentCommand{\NewXSpace}{m  o o}{\mathscr{X}^{#1}\IfValueT{#2}{(#2\IfValueT{#3}{,#3})}}
\NewDocumentCommand{\NewXNorm}{m m  o o}{\|#1\|_{\NewXSpace{#2}[#3][#4]}}
\NewDocumentCommand{\BNewXNorm}{m m  o o}{\left\|#1\right\|_{\NewXSpace{#2}[#3][#4]}}

\NewDocumentCommand{\TLNorm}{m m m m o o}{\left\| #1 \right\|_{\TLSpace{#2}{#3}{#4}[#5][#6]}}
\NewDocumentCommand{\BesovNorm}{m m m m o o}{\left\| #1 \right\|_{\BesovSpace{#2}{#3}{#4}[#5][#6]}}
\NewDocumentCommand{\BigBesovNorm}{m m m m o o}{\left\| #1 \right\|_{\BigBesovSpace{#2}{#3}{#4}[#5][#6]}}

\NewDocumentCommand{\LpNorm}{m m o}{\|#1\|_{\LpSpace{#2}[#3]}}
\NewDocumentCommand{\BLpNorm}{m m o}{\left\|#1\right\|_{\LpSpace{#2}[#3]}}

\NewDocumentCommand{\LpOpNorm}{m m o}{\|#1\|_{\LpSpace{#2}[#3]\rightarrow \LpSpace{#2}[#3]}}
\NewDocumentCommand{\BLpOpNorm}{m m o}{\left\|#1\right\|_{\LpSpace{#2}[#3]\rightarrow \LpSpace{#2}[#3]}}

\NewDocumentCommand{\VectorFields}{m}{\mathscr{X}(#1)}
\NewDocumentCommand{\VectorFieldsN}{}{\VectorFields{\ManifoldN}}

\NewDocumentCommand{\LpSpace}{m o}{L^{#1}\IfValueT{#2}{(#2)}}

\NewDocumentCommand{\lqLpSpace}{m m o}{\ell^{#2}(\Zgeq; L^{#1}\IfValueT{#3}{(#3)} )}
\NewDocumentCommand{\LplqSpace}{m m o}{L^{#1}(\IfValueT{#3}{#3;} \ell^{#2}(\Zgeq) )}
\NewDocumentCommand{\LplqSpaceNoSet}{m m}{L^{#1}( \ell^{#2} )}
\NewDocumentCommand{\lqLpSpaceNoSet}{m m}{\ell^{#2}( L^{#1} )}

\NewDocumentCommand{\lqLpNorm}{m m m o}{\|#1\|_{\lqLpSpace{#2}{#3}[#4]}}
\NewDocumentCommand{\LplqNorm}{m m m o}{\|#1\|_{\LplqSpace{#2}{#3}[#4]}}
\NewDocumentCommand{\lqLpNormNoSet}{m m m}{\|#1\|_{\lqLpSpaceNoSet{#2}{#3}}}
\NewDocumentCommand{\LplqNormNoSet}{m m m}{\|#1\|_{\LplqSpaceNoSet{#2}{#3}}}

\NewDocumentCommand{\BlqLpNorm}{m m m o}{\left\|#1\right\|_{\lqLpSpace{#2}{#3}[#4]}}
\NewDocumentCommand{\BLplqNorm}{m m m o}{\left\|#1\right\|_{\LplqSpace{#2}{#3}[#4]}}

\NewDocumentCommand{\FilteredSheaf}{m o o}{%
  #1_{%
    \IfValueTF{#3}{#3}{\bullet}%
  }%
  \IfValueT{#2}{(#2)}%
}

\NewDocumentCommand{\FilteredSheafF}{o o}{\FilteredSheaf{\mathcal{F}}[#1][#2]}
\NewDocumentCommand{\FilteredSheafG}{o o}{\FilteredSheaf{\mathcal{G}}[#1][#2]}
\NewDocumentCommand{\FilteredSheafFh}{o o}{\FilteredSheaf{\widehat{\mathcal{F}}}[#1][#2]}

\NewDocumentCommand{\FilteredSheafNoSet}{m o}{%
  #1_{%
    \IfValueTF{#2}{#2}{\bullet}%
  }%
}
\NewDocumentCommand{\FilteredSheafNoSetF}{o}{\FilteredSheafNoSet{\mathcal{F}}[#1]}

\NewDocumentCommand{\FilteredSheafGenBy}{m o o}{\left\langle #1\right\rangle_{%
\IfValueTF{#3}{#3}{\bullet}%
}% 
\IfValueT{#2}{(#2)}}

\NewDocumentCommand{\LieFilteredSheaf}{m o o}{%
  \mathrm{Lie}%
    (#1)_{%
    \IfValueTF{#3}{#3}{\bullet}%
  }%
  \IfValueT{#2}{(#2)}
}

\NewDocumentCommand{\LieFilteredSheafNoSet}{m o}{%
  \mathrm{Lie}%
    (#1)_{%
    \IfValueTF{#2}{#2}{\bullet}%
  }%
}

\NewDocumentCommand{\LieFilteredSheafFNoSet}{o}{\LieFilteredSheafNoSet{\FilteredSheafF}[#1]}

\NewDocumentCommand{\LieFilteredSheafF}{o o}{\LieFilteredSheaf{\FilteredSheafF}[#1][#2]}
\NewDocumentCommand{\LieFilteredSheafG}{o o}{\LieFilteredSheaf{\FilteredSheafG}[#1][#2]}
\NewDocumentCommand{\LieFilteredSheafFh}{o o}{\LieFilteredSheaf{\FilteredSheafFh}[#1][#2]}

\NewDocumentCommand{\RestrictFilteredSheaf}{m m o}{%
  #1\big|_{#2}^{\#}\IfValueT{#3}{(#3)}
}

\NewDocumentCommand{\CinftyCptSpace}{o}{C^\infty_{\mathrm{cpt}}\IfValueT{#1}{(#1)}}
\NewDocumentCommand{\CinftySpace}{o o}{C^\infty\IfValueT{#1}{(#1\IfValueT{#2}{;#2})}}

\NewDocumentCommand{\Wh}{}{\widehat{W}}

\NewDocumentCommand{\Wds}{}{d^W}
\NewDocumentCommand{\Zds}{}{d^Z}
\NewDocumentCommand{\Zde}{}{e}
\NewDocumentCommand{\Zdv}{}{d}
\NewDocumentCommand{\Xde}{}{e}
\NewDocumentCommand{\Wde}{}{e}
\NewDocumentCommand{\WWds}{}{(W,\Wds)}
\NewDocumentCommand{\ZZds}{}{(Z,\Zds)}
\NewDocumentCommand{\ZZde}{}{(Z,\Zde)}
\NewDocumentCommand{\ZZdv}{}{(Z,\Zdv)}
\NewDocumentCommand{\Wdv}{}{d}
\NewDocumentCommand{\Vdv}{}{d}
\NewDocumentCommand{\Vde}{}{e}
\NewDocumentCommand{\Xdv}{}{d}
\NewDocumentCommand{\WWdv}{}{(W,\Wdv)}
\NewDocumentCommand{\Xh}{}{\widehat{X}}
\NewDocumentCommand{\Zh}{}{\widehat{Z}}
\NewDocumentCommand{\XhXde}{}{(\Xh,\Xde)}
\NewDocumentCommand{\XhXdv}{}{(\Xh,\Xdv)}
\NewDocumentCommand{\ZhZdv}{}{(\Zh,\Zdv)}
\NewDocumentCommand{\WhWde}{}{(\Wh,\Wde)}
\NewDocumentCommand{\WWde}{}{(W,\Wde)}
\NewDocumentCommand{\XXdv}{}{(X,\Xdv)}
\NewDocumentCommand{\XXde}{}{(X,\Xde)}
\NewDocumentCommand{\WhWdv}{}{(\Wh,\Wdv)}
\NewDocumentCommand{\Who}{}{(\Wh,1)}
\NewDocumentCommand{\VVdv}{}{(V,\Vdv)}
\NewDocumentCommand{\VVde}{}{(V,\Vde)}
\NewDocumentCommand{\WWo}{}{(W,1)}

\NewDocumentCommand{\partialo}{}{(\partial,1)}
\NewDocumentCommand{\WdWithBar}{}{{d\hspace*{-0.08em}\bar{}\hspace*{0.1em}}}

\NewDocumentCommand{\Xt}{}{\widetilde{X}}
\NewDocumentCommand{\XtXdv}{}{(\Xt,\Xdv)}
\NewDocumentCommand{\Vt}{}{\widetilde{V}}
\NewDocumentCommand{\VtVdv}{}{(\Vt,\Vdv)}

\NewDocumentCommand{\DegWdv}{m}{\mathrm{deg}_{\Wdv}(#1)}
\NewDocumentCommand{\DegXdv}{m}{\mathrm{deg}_{\Xdv}(#1)}
\NewDocumentCommand{\DegVdv}{m}{\mathrm{deg}_{\Vdv}(#1)}

\NewDocumentCommand{\BWWdv}{m m}{B_{\WWdv}(#1,#2)}
\NewDocumentCommand{\BWhWdv}{m m}{B_{\WhWdv}(#1,#2)}

\NewDocumentCommand{\BXXdv}{m m}{B_{\XXdv}(#1,#2)}
\NewDocumentCommand{\BZZde}{m m}{B_{\ZZde}(#1,#2)}
\NewDocumentCommand{\MetricWWdv}{o o}{\rho_{\WWdv}\IfValueT{#1}{(#1\IfValueT{#2}{,#2})}}
\NewDocumentCommand{\MetricWWo}{o o}{\rho_{\WWo}\IfValueT{#1}{(#1\IfValueT{#2}{,#2})}}
\NewDocumentCommand{\DistWWdv}{o o}{\mathrm{dist}_{\WWdv}\IfValueT{#1}{(#1\IfValueT{#2}{,#2})}}
\NewDocumentCommand{\DistXXdv}{o o}{\mathrm{dist}_{\XXdv}\IfValueT{#1}{(#1\IfValueT{#2}{,#2})}}
\NewDocumentCommand{\DistXXdvSet}{o o}{\mathrm{dist}_{\XXdv}\IfValueT{#1}{(#1\IfValueT{#2}{,#2})}}

\NewDocumentCommand{\MetricXXdv}{o o}{\rho_{\XXdv}\IfValueT{#1}{(#1\IfValueT{#2}{,#2})}}
\NewDocumentCommand{\MetricXXdvSet}{o o}{\rho_{\XXdv}\IfValueT{#1}{(#1\IfValueT{#2}{,#2})}}

\NewDocumentCommand{\MetricWhWdv}{o o}{\rho_{\WhWdv}\IfValueT{#1}{(#1\IfValueT{#2}{,#2})}}
\NewDocumentCommand{\MetricWWds}{o o}{\rho_{\WWds}\IfValueT{#1}{(#1\IfValueT{#2}{,#2})}}
\NewDocumentCommand{\MetricZZds}{o o}{\rho_{\ZZds}\IfValueT{#1}{(#1\IfValueT{#2}{,#2})}}
\NewDocumentCommand{\MetricZZde}{o o}{\rho_{\ZZde}\IfValueT{#1}{(#1\IfValueT{#2}{,#2})}}
\NewDocumentCommand{\MetricVVde}{o o}{\rho_{\VVde}\IfValueT{#1}{(#1\IfValueT{#2}{,#2})}}

\NewDocumentCommand{\Span}{}{\mathrm{span}}

\NewDocumentCommand{\degBoundaryBngeq}{m o}{\mathrm{deg^{\Bngeq{1}}_{#1}}\IfValueT{#2}{(#2)}}
\NewDocumentCommand{\degBoundaryN}{m o}{\mathrm{deg^{\BoundaryN}_{#1}}\IfValueT{#2}{(#2)}}
\NewDocumentCommand{\degBoundaryNF}{o}{\degBoundaryN{\FilteredSheafF}[#1]}

\NewDocumentCommand{\ManifoldNncF}{o}{\ManifoldN^{\mathrm{nc}}_{\FilteredSheafF}\IfValueT{#1}{(#1)}}
\NewDocumentCommand{\BoundaryNncF}{o}{\BoundaryN^{\mathrm{nc}}_{\FilteredSheafF}\IfValueT{#1}{(#1)}}

\NewDocumentCommand{\ManifoldNtncG}{o}{\ManifoldNt^{\mathrm{nc}}_{\FilteredSheafG}\IfValueT{#1}{(#1)}}

\NewDocumentCommand{\Extension}{o}{\mathscr{E}\IfValueT{#1}{_{#1}}}

\NewDocumentCommand{\VSpace}{m m}{\mathcal{V}_{#1,#2}}
\NewDocumentCommand{\VNorm}{m m m}{\|#1\|_{\VSpace{#2}{#3}}}
\NewDocumentCommand{\BVNorm}{m m m}{\left\|#1\right\|_{\VSpace{#2}{#3}}}
\NewDocumentCommand{\BVNormOmega}{m m m}{\left\|#1\right\|_{\VSpace{#2}{#3}(\Omega)}}
\NewDocumentCommand{\BVNormOmegaN}{m m m}{\left\|#1\right\|_{\VSpace{#2}{#3}(\Omega\cap \ManifoldN)}}
\NewDocumentCommand{\VOpNorm}{m m m}{\|#1\|_{\VSpace{#2}{#3}\rightarrow \VSpace{#2}{#3}}}

\NewDocumentCommand{\SchwartzSpace}{o}{\mathscr{S}\IfValueT{#1}{(#1)}}
\NewDocumentCommand{\SchwartzSpaceR}{}{\SchwartzSpace[\R]}
\NewDocumentCommand{\SchwartzSpaceRn}{}{\SchwartzSpace[\Rn]}
\NewDocumentCommand{\SchwartzSpaceRopq}{}{\SchwartzSpace[\R^{1+q}]}
\NewDocumentCommand{\SchwartzSpacezRopq}{}{\SchwartzSpacez[\R^{1+q}]}

\NewDocumentCommand{\SchwartzSpacez}{o}{\mathscr{S}_0\IfValueT{#1}{(#1)}}
\NewDocumentCommand{\SchwartzSpacezR}{}{\SchwartzSpacez[\R]}
\NewDocumentCommand{\SchwartzSpacezRn}{}{\SchwartzSpacez[\Rn]}

\NewDocumentCommand{\SchwartzSpaceRng}{}{\SchwartzSpace[\Rng]}

\NewDocumentCommand{\Dil}{m m}{\mathrm{Dil}_{#1}[#2]}
\NewDocumentCommand{\Dild}{m m}{\mathrm{Dil}_{#1}^d[#2]}

\NewDocumentCommand{\PElemSymbol}{}{\mathrm{PElem}}
\NewDocumentCommand{\PElemzSymbol}{}{\mathrm{PElem}^{0}}

\NewDocumentCommand{\PElemOnlySubscript}{m}{\mathrm{PElem}_{#1}}
\NewDocumentCommand{\PElemzOnlySubscript}{m}{\mathrm{PElem}_{#1}^{0}}
\NewDocumentCommand{\ElemzOnlySubscript}{m}{\mathrm{Elem}_{#1}^{0}}

\NewDocumentCommand{\Elemz}{m m}{\mathrm{Elem}_{#1}^{0}(#2)}
\NewDocumentCommand{\PElem}{m m}{\mathrm{PElem}_{#1}(#2)}
\NewDocumentCommand{\PElemz}{m m}{\mathrm{PElem}_{#1}^{0}(#2)}
\NewDocumentCommand{\Elemzh}{m m}{\widehat{\mathrm{Elem}^{0}_{#1}}(#2)}

\NewDocumentCommand{\PElemWWdv}{m}{\PElem{\WWdv}{#1}}

\NewDocumentCommand{\PElemF}{m}{\PElem{\FilteredSheafF}{#1}}
\NewDocumentCommand{\PElemzF}{m}{\PElemz{\FilteredSheafF}{#1}}
\NewDocumentCommand{\PElemzFh}{m}{\PElemz{\FilteredSheafFh}{#1}}
\NewDocumentCommand{\ElemzF}{m}{\Elemz{\FilteredSheafF}{#1}}
\NewDocumentCommand{\ElemzFh}{m}{\Elemz{\FilteredSheafFh}{#1}}
\NewDocumentCommand{\ElemzhF}{m}{\Elemzh{\FilteredSheafF}{#1}}
\NewDocumentCommand{\PElemzFhN}{m}{\PElemz{\FilteredSheafFh}{#1,\ManifoldN}}
\NewDocumentCommand{\ElemzFhN}{m}{\Elemz{\FilteredSheafFh}{#1,\ManifoldN}}
\NewDocumentCommand{\ElemzG}{m}{\Elemz{\FilteredSheafG}{#1}}
\NewDocumentCommand{\ElemzXXdv}{m}{\Elemz{\FilteredSheafGenByXXdv}{#1}}
\NewDocumentCommand{\ElemzVVdv}{m}{\Elemz{\FilteredSheafGenByVVdv}{#1}}
\NewDocumentCommand{\PElemXXdv}{m}{\PElem{\FilteredSheafGenByXXdv}{#1}}

\NewDocumentCommand{\ElemzLieFh}{m}{\Elemz{\LieFilteredSheafFh}{#1}}
\NewDocumentCommand{\ElemzLieFhN}{m}{\Elemz{\LieFilteredSheafFh}{#1,\ManifoldN}}

\NewDocumentCommand{\PElemzLieFh}{m}{\PElemz{\LieFilteredSheafFh}{#1}}
\NewDocumentCommand{\PElemzLieFhN}{m}{\PElemz{\LieFilteredSheafFh}{#1,\ManifoldN}}

\NewDocumentCommand{\PElemFh}{m}{\PElem{\FilteredSheafFh}{#1}}

\NewDocumentCommand{\PElemWWdvOmegaVol}{m}{\PElem{\WWdv,\Omega,\Vol}{#1}}
\NewDocumentCommand{\PElemzWWdvOmegaVol}{m}{\PElemz{\WWdv,\Omega,\Vol}{#1}}
\NewDocumentCommand{\ElemzWWdvOmegaVol}{m}{\Elemz{\WWdv,\Omega,\Vol}{#1}}

\NewDocumentCommand{\PairDistributionAndTestFunctions}{m m}{\left\langle#1,#2\right\rangle}

\NewDocumentCommand{\Gen}{m}{\mathrm{Gen}#1}
\NewDocumentCommand{\GenWWdv}{}{\Gen{\WWdv}}
\NewDocumentCommand{\GenZZde}{}{\Gen{\ZZde}}

\NewDocumentCommand{\Mult}{m}{\mathrm{Mult}(#1)}
\NewDocumentCommand{\Multpsi}{}{\Mult{\psi}}

\NewDocumentCommand{\VpqsENorm}{m o o o o}{\left\|#1\right\|_{\IfValueTF{#2}{\VSpace{#2}{#3}}{\VSpace{p}{q}}, \IfValueTF{#4}{#4}{s}, \IfValueTF{#5}{#5}{\sE}} }
\NewDocumentCommand{\VpqOmegaCapNsENorm}{m o o o o}{\left\|#1\right\|_{\IfValueTF{#2}{\VSpace{#2}{#3}(\Omega\cap\ManifoldN)}{\VSpace{p}{q}(\Omega\cap\ManifoldN)}, \IfValueTF{#4}{#4}{s}, \IfValueTF{#5}{#5}{\sE}} }
\NewDocumentCommand{\VpqOmegasENorm}{m o o o o}{\left\|#1\right\|_{\IfValueTF{#2}{\VSpace{#2}{#3}(\Omega)}{\VSpace{p}{q}(\Omega)}, \IfValueTF{#4}{#4}{s}, \IfValueTF{#5}{#5}{\sE}} }

\NewDocumentCommand{\CmSpace}{m o o}{C^{#1}\IfValueT{#2}{(#2\IfValueT{#3}{;#3})} }
\NewDocumentCommand{\CmNorm}{m m o o}{\|#1\|_{\CmSpace{#2}[#3][#4]}}
\NewDocumentCommand{\BCmNorm}{m m o o}{\left\|#1\right\|_{\CmSpace{#2}[#3][#4]}}

\NewDocumentCommand{\Extend}{}{\mathfrak{E}}

\NewDocumentCommand{\ceil}{m}{\lceil #1 \rceil}
\NewDocumentCommand{\floor}{m}{\lfloor #1 \rfloor}

\NewDocumentCommand{\vsig}{}{\varsigma}
\NewDocumentCommand{\vsigt}{}{\tilde{\vsig}}

\NewDocumentCommand{\conv}{}{\mathrm{conv}}
\NewDocumentCommand{\convClosure}{}{\overline{\mathrm{conv}}}

\NewDocumentCommand{\DXSpace}{o}{X^{\IfValueTF{#1}{#1}{p}}}
\NewDocumentCommand{\DXNorm}{m o}{\|#1\|_{\DXSpace[#2]}}

\NewDocumentCommand{\XTraceSpaceShort}{}{\mathscr{X}^s}
\NewDocumentCommand{\XtTraceSpaceShort}{}{\widetilde{\mathscr{X}}^{s'}}
\NewDocumentCommand{\YTraceSpaceShort}{}{\mathscr{Y}^{s-\lambda/p}}

\NewDocumentCommand{\XtTraceSpace}{o}{\XtTraceSpaceShort\left( \IfValueTF{#1}{#1}{\Compact}; \FilteredSheafF \right)}
\NewDocumentCommand{\XTraceSpace}{o}{\mathscr{X}^s\left( \IfValueTF{#1}{#1}{\Compact}; \FilteredSheafF \right)}
\NewDocumentCommand{\XTraceSpacez}{o}{\mathscr{X}^s_0\left( \IfValueTF{#1}{#1}{\Compact}; \FilteredSheafF \right)}
\NewDocumentCommand{\YTraceSpace}{o}{\mathscr{Y}^{s-\lambda/p}\left( \IfValueTF{#1}{#1}{\Compact}\cap \BoundaryN; \RestrictFilteredSheaf{\LieFilteredSheafF}{\BoundaryNncF} \right)}

\NewDocumentCommand{\XTraceSpaceNoSet}{}{\mathscr{X}^s\left(  \FilteredSheafF \right)}
\NewDocumentCommand{\YTraceSpaceNoSet}{}{\mathscr{Y}^{s-\lambda/p}\left(  \RestrictFilteredSheaf{\LieFilteredSheafF}{\BoundaryNncF} \right)}
\NewDocumentCommand{\XTraceNorm}{m}{\left\| #1 \right\|_{\XTraceSpaceNoSet}}
\NewDocumentCommand{\YTraceNorm}{m}{\left\| #1 \right\|_{\YTraceSpaceNoSet}}

\NewDocumentCommand{\XTraceMap}{}{\mathscr{R}[\XTraceSpaceShort]}
\NewDocumentCommand{\XtTraceMap}{}{\mathscr{R}[\XtTraceSpaceShort]}

\NewDocumentCommand{\FilteredSheafGenByXXdv}{o o}{\FilteredSheafGenBy{\XXdv}[#1][#2]}
\NewDocumentCommand{\FilteredSheafGenByVVdv}{o o}{\FilteredSheafGenBy{\VVdv}[#1][#2]}

\NewDocumentCommand{\Maximal}{}{\mathcal{M}}

\NewDocumentCommand{\etaOne}{}{\eta_1}
\NewDocumentCommand{\etaThree}{}{\eta_2}
\NewDocumentCommand{\etaFour}{}{\eta_3}
\NewDocumentCommand{\etaFive}{}{\eta_4}

\NewDocumentCommand{\opQtEPhigamma}{o o o o}{\mathscr{Q}_{\IfValueTF{#1}{#1}{t}}(\IfValueTF{#2}{#2}{\sE'},\IfValueTF{#3}{#3}{\Phi},\IfValueTF{#4}{#4}{\gamma})}

\NewDocumentCommand{\ad}{o}{\mathrm{ad}\IfValueT{#1}{(#1)}}

\NewDocumentCommand{\ZygSymbol}{}{\mathscr{C}}
\NewDocumentCommand{\ZygSpace}{m o o o}{\ZygSymbol^{#1}\IfValueT{#2}{\left( #2\IfValueT{#3}{,#3}\IfValueT{#4}{;#4} \right)}}
\NewDocumentCommand{\ZygNorm}{m m o}{\left\| #1\right\|_{\ZygSpace{#2}[#3]}}

\NewDocumentCommand{\HolderSpace}{m m o o o}{C^{#1,#2}\IfValueT{#3}{\left( #3\IfValueT{#4}{,#4}\IfValueT{#5}{;#5} \right)}}
\NewDocumentCommand{\HolderNorm}{m m m o o o}{\| #1 \|_{\HolderSpace{#2}{#3}[#4][#5][#6]}}
\NewDocumentCommand{\BHolderNorm}{m m m o o o}{\left\| #1 \right\|_{\HolderSpace{#2}{#3}[#4][#5][#6]}}

\NewDocumentCommand{\HolderSpaceCompactF}{m m o}{\HolderSpace{#1}{#2}[\Compact][\FilteredSheafF][#3]}
\NewDocumentCommand{\HolderNormF}{m m m o}{\|#1\|_{\HolderSpace{#2}{#3}[\FilteredSheafF]}}
\NewDocumentCommand{\HolderNormFh}{m m m o}{\|#1\|_{\HolderSpace{#2}{#3}[\FilteredSheafFh]}}

\NewDocumentCommand{\HolderSpaceCompactWWo}{m m o}{\HolderSpace{#1}{#2}[\Compact][\WWo][#3]}
\NewDocumentCommand{\HolderNormWWo}{m m m o}{\|#1\|_{\HolderSpace{#2}{#3}[W,1]}}

\NewDocumentCommand{\ZygSpaceCompactF}{m o}{\ZygSpace{#1}[\Compact][\FilteredSheafF][#2]}
\NewDocumentCommand{\ZygNormF}{m m}{\ZygNorm{#1}{#2}[\FilteredSheafF]}
\NewDocumentCommand{\ZygNormFh}{m m}{\ZygNorm{#1}{#2}[\FilteredSheafFh]}

\NewDocumentCommand{\CSpace}{m}{C(#1)}

\NewDocumentCommand{\ProdZygSpace}{m m o o o o}{\ZygSymbol^{#1, #2}\IfValueT{#3}{\left( #3 \IfValueT{#4}{\times #4}  \IfValueT{#5}{,#5}\IfValueT{#6}{;#6} \right)}}
\NewDocumentCommand{\ProdZygSpaceCompactRNF}{m m}{\ProdZygSpace{#1}{#2}[\Compact][\RN][\FilteredSheafF]}
\NewDocumentCommand{\ProdZygNorm}{m m m o}{\left\| #1\right\|_{\ProdZygSpace{#2}{#3}[#4]}}
\NewDocumentCommand{\ProdZygNormF}{m m m}{\ProdZygNorm{#1}{#2}{#3}[\FilteredSheafF]}

\begin{document}

\newtheorem{theorem}{Theorem}[section]
\newtheorem{corollary}[theorem]{Corollary}
\newtheorem{proposition}[theorem]{Proposition}
\newtheorem{lemma}[theorem]{Lemma}
\newtheorem{conjecture}[theorem]{Conjecture}
\newtheorem{problem}[theorem]{Problem}

\theoremstyle{remark}
\newtheorem{remark}[theorem]{Remark}

\theoremstyle{definition}
\newtheorem{definition}[theorem]{Definition}

\theoremstyle{definition}
\newtheorem{notation}[theorem]{Notation}

\theoremstyle{definition}
\newtheorem{assumption}[theorem]{Assumption}

\theoremstyle{remark}
\newtheorem{example}[theorem]{Example}

\numberwithin{equation}{section}

\title{Function spaces and trace theorems for maximally subelliptic boundary value problems}

\ifnotanonymous
\author{Brian Street\footnote{The author was partially supported by National Science Foundation Grant 2153069.}}
\else
\author{}
\fi

\date{}

\maketitle

\begin{abstract}
    We introduce Besov and Triebel--Lizorkin spaces on a manifold with boundary adapted to H\"ormander vector fields, near a so-called non-characteristic point of the boundary. We prove sharp results in these spaces for the corresponding restriction and trace operators, show these operators are retractions, and other related results. This is the second paper in a forthcoming series devoted to a general theory of maximally subelliptic boundary value problems, and lays the function space foundation for this general theory.
\end{abstract}

\tableofcontents

\section{Introduction}
Let \(\ManifoldN\) be a smooth manifold with boundary; we denote by \(\InteriorN\) its interior and \(\BoundaryN\) its boundary.
A boundary value problem for a partial differential equation (PDE) is an equation of the form:
\begin{equation*}
    \opP u = f \text{ on }\InteriorN, \quad B_j u\big|_{\BoundaryN}=g_j\text{ on }\BoundaryN,
\end{equation*}
where \(\opP\) and \(B_1,\ldots, B_l\) are partial differential operators, \(u\) is the unknown function,
and \(f\) and \(g_j\) are given functions on \(\InteriorN\) and \(\BoundaryN\), respectively.
There is a deep and robust theory of elliptic boundary value problems; see, for example,
\cite{AgranovichEgorovShubinPartialDifferentialEquationsIX} for some general theory 
and \cite[Chapter 7]{FollandIntroductionToPDEs} for a friendly introduction.
This theory includes: sharp results in a variety of function spaces, operators with rough coefficients,
fully nonlinear equations and much more, all done for general elliptic boundary value problems.
Once one leaves the setting of elliptic boundary value problems, much less is known.
There are few general theorems and even fewer sharp results.
This is the second in a series of papers whose goal is to develop a general theory of maximally subelliptic (also known 
as maximally hypoelliptic) boundary value problems (the first paper studied the associated geometries \cite{StreetCarnotCaratheodoryBallsOnManifoldsWithBoundary}).
In this paper, we introduce Besov and Triebel--Lizorkin spaces on a manifold with boundary adapted to a maximally subelliptic PDE;
these generalize the classical Besov and Triebel--Lizorkin spaces. 
For the treatment of boundary value problems, trace maps are particularly important; we introduce these and show they
are retractions on the appropriate spaces (continuous linear maps with continuous, linear right inverses).
The (well-versed) reader
wishing to quickly get an idea of our main results can jump straight to Section \ref{Section::GlobalCor}.
In this introduction, we explain the motivation behind our results,
while in Section \ref{Section::Classical} we remind the reader of some classical results on \(\Rn\) which are
generalized in this paper.
% We also give generalizations of the classical
% restriction and trace maps and show they are retractions (right invertible).

Our story begins with the foundational paper of H\"ormander \cite{HormanderHypoellipticSecondOrderDifferentialEquations}.
Let \(X_0,X_1,\ldots, X_p\) be smooth vector fields on a manifold without boundary, \(\ManifoldM\), satisfying
\textit{H\"ormander's condition}: the Lie algebra generated by \(X_0,X_1,\ldots, X_p\) spans the tangent space at every point
(see Definition \ref{Defn::GlobalCor::HormandersCondition}).  Let
\begin{equation}\label{Eqn::Intro::HorSubLapalce}
    \opL:=X_0+X_1^2+X_2^2+\cdots+X_p^2.
\end{equation}
\(\opL\) (and more general maximally subelliptic operators; see Definition \ref{Defn::Intro::MaximalSub}) arises in a number of settings, including
stochastic calculus and several complex variables--see the introduction of \cite{BramantiBrandoliniHormanderOperators}
for a friendly account.
H\"ormander showed that \(\opL\) is \textit{subelliptic}: roughly speaking if \(u\in \Distributions[\ManifoldM]\) and 
\(\opL u\in L^2_s\) near a point \(x\in \ManifoldM\), then \(u\in L^2_{s+\epsilon}\) near \(x\), where \(L^2_s\) is the
\(L^2\)-Sobolev space of order \(s\in \R\), and \(\epsilon=\epsilon(x)>0\).
Informally: \(u\) is smoother than \(\opL u\).
It seems hopeless at this stage to develop a general theory of ``subelliptic boundary value problems'' which parallels the elliptic theory:
general subelliptic PDEs are just too wild.\footnote{For example, the \(\overline{\partial}\)-Neumann problem is a boundary value
problem which is often subelliptic, has given rise to a huge theory, but for which there are still many difficult open questions.
See \cite{ChenShawPartialDifferentialEquationsInSeveralComplexVariables,StraubeLecturesOnTheL2SobolevTheoryOfTheOverlinePartialNeumannProblem}.}  
Nevertheless, boundary value problems for special cases
of \(\opL\) have been studied by many authors: this began with the work of Jerison \cite{JerisonDirichletProblemForTheKohnLaplacianI},
which was followed by many others.  However, there still currently no general sharp theory, even for the Dirichlet
problem for general operators of the form \eqref{Eqn::Intro::HorSubLapalce}.

Rothschild and Stein \cite{RothschildSteinHypoellipticDifferentialOperatorsAndNilpotentGroups} showed that not only is \(\opL\)
subelliptic, but a stronger result is true: \(\opL\) is \textit{maximally subelliptic}. Let \(\WWdv=\left\{ \left( W_1,\Wdv_1 \right),\ldots, \left( W_r,\Wdv_r \right) \right\}\)
be H\"ormander vector fields on \(\ManifoldM\), each paired with a formal degree \(\Wdv_j\in \Zg=\left\{ 1,2,3,\ldots \right\}\).
In the case of \(\opL\) in \eqref{Eqn::Intro::HorSubLapalce}, we take
\(\WWdv=\left\{ \left( X_0,2 \right),\left( X_1,1 \right),\ldots, \left( X_p,1 \right) \right\}\).
For a list \(\alpha=\left( \alpha_1,\ldots,\alpha_L \right)\), let \(W^{\alpha}=W_{\alpha_1}W_{\alpha_2}\cdots W_{\alpha_L}\)
and \(\DegWdv{\alpha}=\Wdv_{\alpha_1}+\Wdv_{\alpha_2}+\cdots+\Wdv_{\alpha_L}\).

\begin{definition}\label{Defn::Intro::MaximalSub}
    Let \(\kappa\in \Zg\) be such that \(\Wdv_j\) divides \(\kappa\) for every \(1\leq j\leq r\). Let \(\opP\) be a partial differential
    operator of the form
    \begin{equation}\label{Eqn::Intro::MaximalSub}
        \opP = \sum_{\DegWdv{\alpha}\leq \kappa} a_\alpha W^{\alpha},\quad a_\alpha\in \CinftySpace[\ManifoldM].
    \end{equation}
    We say \(\opP\) is \textit{maximally subelliptic}\footnote{Estimates like \eqref{Eqn::Intro::MaximalSub} first appeared
    in the work of Folland and Stein \cite{FollandSteinEstimatesForTheBarPartialBComplex}, and later work of Folland \cite{FollandSubellipticEstimatesAndFunctionSpacesOnNilpotentLieGroups}  
    and Rothschild and Stein \cite{RothschildSteinHypoellipticDifferentialOperatorsAndNilpotentGroups}. An equivalent general definition,
    under the French name \textit{hypoellipticit\'e maximale} was first introduced by Helffer and Nourrigat \cite{HelfferNourrigatHypoellipticiteMaximalePourDesOperateursPolynomesDeChampsDeVecteurs}.} with respect to \(\WWdv\) if for all \(\Omega\Subset \ManifoldM\)
    open and relatively compact,\footnote{We write \(A\Subset B\) if \(A\) is relatively compact in \(B\); i.e., if the closure of \(A\)
    as a subspace of \(B\) is compact.} we have
    \begin{equation*}
        \sum_{j=1}^r \BLpNorm{ W_j^{\kappa/\Wdv_j} f}{2}
        \lesssim\BLpNorm{\opP f}{2} + \BLpNorm{f}{2},\quad \forall f\in \CinftyCptSpace[\Omega].
    \end{equation*}
\end{definition}

Maximally subelliptic operators are subelliptic, but maximal subellipticity is a much stronger condition than subellipticity.
To help explain why, we return to the well-studied elliptic setting.
There is by now a vast theory for elliptic PDEs, spanning countless papers and many textbooks
(see, for example, the three volume series of Taylor 
\cite{TaylorPartialDifferentialEquationsI,TaylorPartialDifferentialEquationsII,TaylorPartialDifferentialEquationsIII});
covering
% There is a deep and far-reaching theory of elliptic PDEs, covering 
linear, rough coefficients, fully nonlinear, boundary value problems,
and much more, where sharp results are known in a variety of the classical function spaces.
Many classical proofs in the elliptic theory seem short at first glance, but rely on a huge foundation
of function spaces (like Besov and Triebel--Lizorkin spaces), operators (like pseudodifferential operators),
geometry (usually Euclidean or Riemannian), and classical transforms like the Fourier transform.

If one takes the H\"ormander vector fields with formal degrees on \(\Rn\) given by
\begin{equation*}
    \partialo=\left\{ \left( \partial_{x_1},1   \right),\ldots, \left( \partial_{x_n},1 \right) \right\},
\end{equation*}
then an operator \(\opP\) is maximally subelliptic with respect to \(\partialo\) if and only if it is (locally) elliptic.
While subellipticity is a \textit{weaker} condition than ellipticity,
%  (if an operator is elliptic, it is subelliptic),
maximal subellipticity is instead a \textit{generalization} of ellipticity: ellipticity is the special case of 
maximally subellipticity with \(\WWdv=\partialo\). By changing \(\WWdv\), maximal subellipticity becomes a different condition,
but not a weaker condition.  When viewed from the perspective of the elliptic theory, maximally subelliptic PDEs
can be very degenerate (they are often nowhere elliptic--see \cite[Remark 8.2.7]{StreetMaximalSubellipticity}); however, they satisfy a different condition which in some
ways is just as strong as ellipticity.

In light of this, given any classical result from the elliptic theory, one can ask if it is possible to prove a more general
result for maximally subelliptic operators which specializes to the classical result without weakening the conclusions.
For manifolds without boundary, this idea has been taken up by many authors, and has been furthered in thousands of papers
and several books (see \cite{BramantiAnInvitationToHypoellipticOperators,BramantiBrandoliniHormanderOperators} for a friendly
introduction and \cite{StreetMaximalSubellipticity} for some general theory). A major difficulty 
is that the foundation used for elliptic operators (classical function spaces, operators, geometry, and transforms)
is not as useful when studying maximally subelliptic PDEs: one needs a more general foundation on which to work.
There are many complications in doing so (some are discussed below). Nevertheless,  a general interior theory
for maximally subelliptic PDEs has been achieved, which largely parallels and generalizes the classical elliptic theory
(see \cite{StreetMaximalSubellipticity}). This includes not only the sharp interior regularity theory for linear operators with
smooth coefficients, but also operators with rough coefficients, and fully nonlinear equations.

Returning to boundary value problems, there is a deep theory of elliptic boundary value problems;
see, for example, \cite{AgranovichEgorovShubinPartialDifferentialEquationsIX}.
While there have been many examples of what might be called maximally subelliptic boundary value problems studied 
(starting with the work of Jerison \cite{JerisonDirichletProblemForTheKohnLaplacianI}),
there is not yet even a definition of \textit{maximally subelliptic boundary value problems} which generalizes the elliptic case;
let alone any general results which generalize the elliptic theory.
The main problem in introducing such a theory is that the foundational theory of function spaces on which elliptic
boundary value problems rest has not been generalized to the maximally subelliptic setting.
This goal of this paper is to fill this gap, setting the stage for the study of general maximally subelliptic boundary value problems.

An important concept, first pointed out by Kohn and Nirenberg \cite{KohnNirenbergNonCoerciveBoundaryValueProblems},
Derridj \cite{DerridjSurUnTheoremeDeTraces}, and Jerison \cite{JerisonDirichletProblemForTheKohnLaplacianI,JerisonDirichletProblemForTheKohnLaplacianII},
is that even in the simplest non-elliptic cases, the interaction between \(\WWdv\) and \(\BoundaryN\) is important.
They showed that whether or not the boundary was ``characteristic'' with respect to \(\WWdv\) had significant impacts on boundary value problems.
When \(\Wdv_j=1\) for \(1\leq j\leq r\), we say \(x_0\in \BoundaryN\) is \(\WWo\)-non-characteristic if \(\exists j\)
with \(W_j(x_0)\not\in \TangentSpace{x_0}{\BoundaryN}\). The definition when some of the \(\Wdv_j\) are not \(1\)
is a bit more complicated; see Definition \ref{Defn::Filtrations::RestrictingFiltrations::NonCharPoints}.
In this paper, we define function spaces and various operations near non-characteristic points of the boundary.
Further highlighting the importance of this concept, some of our definitions are only well-defined
near these non-characteristic points; for more details, see Remark \ref{Rmk::Spaces::LP::NonCharMatters}.
Moreover, the natural generalization of our results to the characteristic setting is false;
see Section \ref{Section::Trace::CharacteristicFailure}--and it is unclear what the correct results
are.

The main goals of this paper are to do the following on a manifold with boundary
endowed with H\"ormander vector fields with formal degrees \(\WWdv\), both on the interior and near the non-characteristic
part of the boundary:
\begin{enumerate}[(i)]
    \item Define Besov, \(\BesovSpace{s}{p}{q}\), and Triebel--Lizorkin, \(\TLSpace{s}{p}{q}\), spaces adapted to \(\WWdv\).
        These specialize to important cases, especially when \(\Wdv_j=1\), \(\forall j\): \(\TLSpace{s}{p}{2}\) (\(1<p<\infty\)) can be viewed as non-isotropic \(\LpSpace{p}\)-Sobolev spaces
        adapted to \(\WWdv\) (see Corollary \ref{Cor::GlobalCor::TLmp2AreSobolevSpaces}), while \(\BesovSpace{s}{\infty}{\infty}\)
        for \(s\in (0,\infty)\setminus \Zg\) can be viewed as H\"older spaces adapted to \(\WWdv\) (see Section \ref{Section::Zyg::Holder}).
        See Chapter \ref{Chapter::Spaces}.
    \item Define trace maps to the boundary, characterize their image, and show that they are retractions (right invertible, with a continuous, linear, right inverse),
        on the appropriate Besov and Triebel--Lizorkin spaces. See Chapter \ref{Chapter::Trace}.
    \item When \(\ManifoldN\) is a closed, embedded, co-dimension \(0\) submanifold of some ambient manifold without boundary \(\ManifoldM\), and \(W_1,\ldots, W_r\)
        are given by restrictions of H\"ormander vector fields on \(\ManifoldM\), we show that the Besov and Triebel--Lizorkin spaces on \(\ManifoldN\) (adapted to \(\WWdv\)) correspond
        with restrictions of the corresponding Besov and Triebel--Lizorkin spaces on \(\ManifoldM\), and that the restriction map is a retraction. See Theorem \ref{Thm::Spaces::Extension}.
    \item For applications to the Dirichlet problem, we characterize the closure of \(\CinftyCptSpace[\InteriorN]\) in certain Besov and Triebel--Lizorkin spaces
        in terms of vanishing traces on the boundary. See Section \ref{Section::Trace::Vanish}.
    \item For applications to nonlinear boundary value problems, we give sharp regularity for some compositions.
        See Section \ref{Section::Zyg::Compositions}.
\end{enumerate}

In introducing this theory, we face several hurdles:
\begin{itemize}
    \item The Fourier transform, which is the central tool when studying classical function spaces, is not adapted to the Carnot--Carath\'eodory
        geometry defined by \(\WWdv\). This means that the Fourier transform is  rarely used in this paper. Instead, we work integral operators
        instead of multipliers to define the Littlewood--Paley theory (see Section \ref{Section::Spaces::LittlewoodPaleyTheory}).
        While these integral operators are sufficiently powerful to develop our theory, they do make the proofs more involved.
        Even the classical analogs of our results (as in \cite{TriebelTheoryOfFunctionSpaces,RunstSickelSobolevSpacesOfFractionalOrder})
        require many detailed estimates. The lack of a Fourier transform, and our more general and abstract setting,
        lead to even more involved proofs.
    \item The most common way to define the classical Besov and Triebel--Lizorkin on a manifold with boundary is as restrictions of Besov and Triebel--Lizorkin
        spaces on an ambient manifold without boundary (see the discussion in Section \ref{Section::Spaces::Classical}). This presents a complication
        in our setting because the 
        H\"ormander vector fields with formal degrees on the ambient manifold are not uniquely determined by those
        on the submanifold with boundary.
        % ambient manifold may could be given many different H\"ormander vector fields with formal degrees which restrict
        % to the given H\"ormander vector fields with formal degrees on the manifold with boundary. 
        Thus one has many different possible Besov and Triebel--Lizorkin
        spaces on the ambient manifold (for each choice of ambient H\"ormander vector fields with formal degrees)
        and it is not a priori clear their restrictions give the same space on the given manifold with boundary. We proceed in a different
        way and introduce the Besov and Triebel--Lizorkin spaces intrinsically on the manifold with boundary, and then show that
        all the different possible restrictions
        (from all possible ambient manifolds endowed with appropriate H\"ormander vector fields with formal degrees)
        give rise to the same space. See Section \ref{Section::Spaces::Classical}
        for a discussion of this intrinsic definition in the classical setting.
    \item Related to the previous point, one common way to extend functions on a manifold with boundary to an ambient manifold involves
        some kind of reflection over the boundary; see, e.g., \cite{SeeleyExtensionOfCInfinityFunctionsDefinedInAHalfSpace}.
        Because the H\"ormander vector fields on the ambient manifold are not determined by those on the manifold with boundary,
        this kind of reflection does not respect our Besov and Triebel--Lizorkin spaces. Instead, a more complicated
        method of extension is used: see Section \ref{Section::Spaces::RestrictionAndExtension}.
    \item Theories of Besov and Triebel--Lizorkin spaces 
        for geometries other than Euclidean or Riemannain (for example on spaces of homogeneous type)
    usually either restrict to the regularity parameter \(|s|\) small
        (for example, \cite{HanSawyerLittlewoodPaleyTheoryOnSpacesOfHomogeneousTypeAndTheClassicalFunctionSpaces}) or use
        some sort of underlying group structure (for example, \cite{FollandSubellipticEstimatesAndFunctionSpacesOnNilpotentLieGroups})--see
        \cite[Section 6.14]{StreetMaximalSubellipticity} for a more detailed history. In this paper we use the special nature
        of Carnot--Carathe\'odory geometry to introduce these spaces for all regularity parameters \(s\). On manifolds
        without boundary, this follows many authors (see \cite[Section 6.14]{StreetMaximalSubellipticity} for a history).
        For manifolds with boundary, we do not know of an analogous theory.
    \item The standard dilation maps on \(\Rn\) play an important role in the study of the classical Besov and Triebel--Lizorkin spaces.
        In the maximally subelliptic setting, these are replaced by much more complicated scaling maps originially
        introduced by Nagel, Stein, and Wainger \cite{NagelSteinWaingerBallsAndMetricsDefinedByVectorFieldsI} on manifolds without boundary;
        see \cite[Section 3.3.1]{StreetMaximalSubellipticity} for a description of how these scaling maps can be used.
        In this paper, we use similar scaling maps on manifolds with boundary introduced in \cite{StreetCarnotCaratheodoryBallsOnManifoldsWithBoundary}.
    \item In this introduction, we worked with a given list of H\"ormander vector fields with formal degrees \(\WWdv\).
        However, many such lists are equivalent for our purposes (see \textit{local weakly equivalent} in  \cite[Definition \ref*{CC::Defn::BasicDefns::StrongWeakEquivalnce}]{StreetCarnotCaratheodoryBallsOnManifoldsWithBoundary}).
        For the interior theory, working with this equivalence is not too unwieldy, and this was the approach taken in
        \cite{StreetMaximalSubellipticity}. However, when studying boundary value problems, working with this equivalence becomes
        much more difficult. For example, the induced H\"ormander vector fields with formal degrees on the non-characteristic boundary are only well-defined
        up to this equivalence. Instead, following \cite[Section \ref*{CC::Section::Sheaves}]{StreetCarnotCaratheodoryBallsOnManifoldsWithBoundary}, we work more directly
        with certain \(\Zg\)-filtrations of sheaves of vector fields on \(\ManifoldN\);
        see Chapter \ref{Chapter::VectorFieldsAndSheaves}. While more technical, these definitions encapsulate the equivalence
        (see \cite[Section \ref*{CC::Section::Sheaves::Control}]{StreetCarnotCaratheodoryBallsOnManifoldsWithBoundary}),
        which streamlines many of our definitions and proofs.
    \item We do not assume an equiregularity hypothesis on the H\"ormader vector fields, and do not assume
        a filtered manifold structure; such hypotheses simplfy the analysis but do not seem to yeild stronger results in our setting.
\end{itemize}

As described above, one major theme of this paper that is trace maps and restriction maps are retractions (have continuous, linear, right inverses).
This right invertability is a central tool in the study of elliptic boundary value problems; and will be a central tool
in our future study of maximally subelliptic boundary value problems. It appears 
implicitly and explicitly throughout the elliptic theory; see \cite[Proposition 7.7]{FollandIntroductionToPDEs} for a simple
version of it, and applications of that proposition throughout \cite[Chapter 7]{FollandIntroductionToPDEs}.
One particularly simple to understand application comes from nonlinear PDEs: the Banach space inverse function theorem is a central
tool in studying nonlinear elliptic equations, and applies because all operations involved are appropriately invertible.
See \cite[Section 1.7.1]{StreetMaximalSubellipticity} for how it can be applied to the interior nonlinear maximally subelliptic theory.
Using ideas like these, the right invertiblity described in this paper will be important in our future study of both linear
and nonlinear maximally subelliptic boundary value problems.

Even in the setting of nilpotent Lie groups, some of our definitions and results are new.
Indeed, when considering a smooth bounded domain in a nilpotent Lie group, the boundary
may have no good local group structure. Since we develop our definitions and results without using a group
structure, they cover this case.

Some of the first results on trace theorems in settings like the ones described in this paper are due to Derridj \cite{DerridjSurUnTheoremeDeTraces}.
For other previous related works, see 
\cite{BahouriCheminXuTraceTheoremOnTheHeisenbergGroup,
BahouriCheminXuTraceTheoremOnTheHeisenbergGroupOnHomogeneousHypersurfaces,
BahouriCheminXuTraceTheoremInWeightedSobolevSpaces,
BerhanuPesensonTheTraceProblemForVectorFieldsSatisfyingHormandersCondition,
DanielliGarofaloNhieuTraceInequalitiesForCarnotCaratheodorySpacesAndApplications,PesensonTheTraceProblemAndHardyOperatorForNonisotropicFunctionSpacesOnTheHeisenbergGroup,DanielliGarofaloNhieuSubellipticBesovSpacesAndTheCharacterizationOfTracesOnLowerDimensionalManifolds,MontiMorbidelliTraceTheoremsForVectorFields}.
In particular, \cite{DanielliGarofaloNhieuSubellipticBesovSpacesAndTheCharacterizationOfTracesOnLowerDimensionalManifolds} 
has some similarities to our results;
it would be interesting to unify the apporaches.

\section{Background: Classical trace theorems for elliptic function spaces}\label{Section::Classical}
We describe some classical results on trace theorems that we generalize in this paper.
See \cite[Sections 2.7.2 and 2.9 and Chapter 3]{TriebelTheoryOfFunctionSpaces} and \cite[Section 2.4]{RunstSickelSobolevSpacesOfFractionalOrder}
for complete presentations, along with more results.

Consider the manifold with boundary\footnote{In this section, we restrict attention to the particular manifold with boundary \(\Rngeq\), even though
later in the paper we work with general smooth manifolds with boundary. The difference is not an essential point,
since the main results of this paper are local.}
\begin{equation*}
    \Rngeq :=\left\{ (x',x_n) : x'\in \Rnmo, x_n\geq 0 \right\},
\end{equation*}
with interior \(\Rng:=\left\{ (x',x_n) : x_n>0 \right\}\)
and boundary \(\Rnmo\cong \left\{ (x',0) :x'\in \Rnmo \right\}\).

Let \(H^s(\Rn)\) and \(H^s(\Rngeq)\) denote the standard \(\LpSpace{2}\)-Sobolev spaces
on \(\Rn\) and \(\Rngeq\), respectively.
The two main maps we consider are \(f\mapsto f\big|_{\Rng}\) (taking \(\Distributions[\Rn]\rightarrow \Distributions[\Rng]\))
and for \(L\in \Zgeq:=\left\{ 0,1,2,\ldots \right\}\), the trace map,
initially defined for smooth functions \(f:\Rngeq\rightarrow \C\),
\begin{equation}\label{Eqn::Intro::Classical::DefineTraceMap}
    \TraceMap[L] f(x')=\left( f(x',0), \partial_{x_n} f(x',0),\ldots, \partial_{x_n}^L f(x',0) \right).
\end{equation}
Two of the most familiar results concerning these maps are:

\begin{itemize}
    \item The map \(f\mapsto f\big|_{\Rng}\) is a  retraction \(H^s(\Rn)\rightarrow H^s(\Rngeq)\);
    i.e., it is a continuous linear map with a continuous, linear, right inverse.\footnote{In fact, one common definition
    for \(H^s(\Rngeq)\) is \(H^s(\Rn)/\sim\) where \(\sim\) is the equivalence relation 
    \(f\sim g\) if \(f\big|_{\Rng}=g\big|_{\Rng}\) (see \cite[Section 2.9.1]{TriebelTheoryOfFunctionSpaces}). There are other 
    equivalent
    definitions which do not use the ambient
    space \(\Rn\)--see Section \ref{Section::Spaces::Classical}.}

    \item For \(s>L+1/2\), \(\TraceMap[L]\) extends to a continuous map
    \begin{equation}\label{Eqn::Intro::Classical::TraceOnL2Sobolev}
        \TraceMap[L]:H^s(\Rngeq)\rightarrow \prod_{j=0}^L H^{s-j-1/2}(\Rnmo),
    \end{equation}
    and this map is a retraction.
\end{itemize}

Unfortunately, to generalize \eqref{Eqn::Intro::Classical::TraceOnL2Sobolev}
to \(\LpSpace{p}\)-Sobolev spaces, \(1<p<\infty\), one must leave the realm of Sobolev
spaces altogether: the range is in terms of Besov spaces (see Remark \ref{Rmk::Intro::Classical::TraceGoesToBesov}).
Thus, it makes sense to proceed more generally and use the standard
Besov, \(\BesovSpace{s}{p}{q}[\Rngeq]\), and Triebel--Lizorkin, \(\TLSpace{s}{p}{q}[\Rngeq]\),
spaces. Of particular interest are the standard \(\LpSpace{p}\)-Sobolev spaces
which equal \(\TLSpace{s}{p}{2}\), \(1<p<\infty\), and the Zygmund--H\"older spaces
\(\BesovSpace{s}{\infty}{\infty}\), \(s>0\).
See \cite[Section 2.9.1]{TriebelTheoryOfFunctionSpaces} and Section \ref{Section::Spaces::Classical}
for the definitions of \(\BesovSpace{s}{p}{q}[\Rngeq]\) and \(\TLSpace{s}{p}{q}[\Rngeq]\).

\begin{remark}\label{Rmk::Intro::Classical::Restrcitpq}
    In what follows, we restrict attention to \(1\leq p,q\leq \infty\)
    when considering \(\BesovSpace{s}{p}{q}\), and \(1<p<\infty\), \(1<q\leq \infty\)
    when considering \(\TLSpace{s}{p}{q}\). In the classical setting, there are results known
    for more 
    \(p,q\in (0,\infty]\) (see, e.g., 
    \cite[Sections 2.7.2 and 2.9.4]{TriebelTheoryOfFunctionSpaces} and
    \cite[Section 2.4]{RunstSickelSobolevSpacesOfFractionalOrder}),
    though they can have a somewhat more complicated form.
    We only generalize these results under the above additional restrictions on \(p,q\).
    It is likely possible to obtain results for more \(p,q\in (0,\infty]\), but this would greatly complicate matters,
    and would not be as useful for the applications we have in mind to subelliptic PDEs.
\end{remark}

\begin{theorem}[{\cite[Section 2.9.4]{TriebelTheoryOfFunctionSpaces}}]\label{Thm::Intro::Classical::Extension}
    The map \(f\mapsto f\big|_{\Rng}\), \(\Distributions[\Rn]\rightarrow \Distributions[\Rng]\),
    is a retraction \(\BesovSpace{s}{p}{q}[\Rn]\rightarrow \BesovSpace{s}{p}{q}[\Rngeq]\)
    and \(\TLSpace{s}{p}{q}[\Rn]\rightarrow \TLSpace{s}{p}{q}[\Rngeq]\), \(\forall s\in \R\)
    and under the restrictions on \(p,q\) in Remark \ref{Rmk::Intro::Classical::Restrcitpq}.
\end{theorem}

\begin{theorem}[{\cite[Sections 2.7.2 and 3.3.3]{TriebelTheoryOfFunctionSpaces} and \cite[Section 2.4.2]{RunstSickelSobolevSpacesOfFractionalOrder}}]
    \label{Thm::Intro::Classical::TraceThm}
    Fix \(L\in \Zgeq\). There exists a unique map
    \begin{equation*}
        \TraceMap[L]:\left( \bigcup_{\substack{1<p\leq \infty \\ 1\leq q\leq \infty\\ s>L+1/p}} \BesovSpace{s}{p}{q}[\Rngeq] \right)
        \bigcup
        \left( \bigcup_{\substack{1<p<\infty \\ 1< q\leq \infty\\ s>L+1/p}} \TLSpace{s}{p}{q}[\Rngeq] \right)
        \rightarrow \Distributions[\Rnmo]^{L+1}
    \end{equation*}
    such that
    \begin{enumerate}[(I)]
        \item For \(f\in \CinftyCptSpace[\Rngeq]\), \(\TraceMap[L]f\) is given by \eqref{Eqn::Intro::Classical::DefineTraceMap}.
        \item\label{Item::Intro::Classical::TraceThm::Besov} \(\TraceMap[L]:\BesovSpace{s}{p}{q}[\Rngeq]\rightarrow \prod_{l=0}^L \BesovSpace{s-l-1/p}{p}{q}[\Rnmo]\)
            is continuous for \(1<p\leq \infty\), \(1\leq q\leq \infty\), \(s>L+1/p\).
        \item\label{Item::Intro::Classical::TraceThm::TL} \(\TraceMap[L]:\TLSpace{s}{p}{q}[\Rngeq]\rightarrow \prod_{l=0}^L \BesovSpace{s-l-1/p}{p}{p}[\Rnmo]\) is continuous
            for \(1<p<\infty\), \(1<q\leq \infty\), \(s>L+1/p\).
    \end{enumerate}
    Furthermore, the maps in \ref{Item::Intro::Classical::TraceThm::Besov} and \ref{Item::Intro::Classical::TraceThm::TL}
    are retractions. Moreover there is a continuous, linear map
    \(\TraceInverseMap[L]:\TemperedDistributions[\Rnmo]^{L+1}\rightarrow \Distributions[\Rng]\)
    such that
    \begin{enumerate}[(i)]
        \item\label{Item::Intro::Classical::TraceThm::InverseBesov} \(\TraceInverseMap[L]:\prod_{l=0}^L \BesovSpace{s-l-1/p}{p}{q}[\Rnmo]\rightarrow \BesovSpace{s}{p}{q}[\Rngeq]\)
            is continuous for \(s\in \R\), \(p,q\in [1,\infty]\).
        \item\label{Item::Intro::Classical::TraceThm::InverseTL} \(\TraceInverseMap[L]:\prod_{l=0}^L \BesovSpace{s-l-1/p}{p}{p}[\Rnmo]\rightarrow \TLSpace{s}{p}{q}[\Rngeq]\)
            is continuous for \(s\in \R\), \(p\in (1,\infty)\), \(q\in (1,\infty]\).
        \item \(\TraceMap[L]\TraceInverseMap[L]=I\) whenever it makes sense; i.e., whenever the range space of
            \(\TraceInverseMap[L]\) in \ref{Item::Intro::Classical::TraceThm::InverseBesov} or \ref{Item::Intro::Classical::TraceThm::InverseTL}
            coincides with the domain space of \(\TraceMap[L]\) in \ref{Item::Intro::Classical::TraceThm::Besov} or \ref{Item::Intro::Classical::TraceThm::TL}.
    \end{enumerate}
\end{theorem}

\begin{remark}
    The repeated \(p\) in the subscript of \(\BesovSpace{s-l-1/p}{p}{p}[\Rnmo]\) in \ref{Item::Intro::Classical::TraceThm::TL}
    and \ref{Item::Intro::Classical::TraceThm::InverseTL} is not a typo; this is the correct space to use
    as shown by Theorem \ref{Thm::Intro::Classical::TraceThm}.
\end{remark}

\begin{remark}\label{Rmk::Intro::Classical::TraceGoesToBesov}
    Setting \(q=2\) in \ref{Item::Intro::Classical::TraceThm::TL} and \ref{Item::Intro::Classical::TraceThm::InverseTL}  shows that the range of \(\TraceMap[L]\)
    when acting on \(\LpSpace{p}\)-Sobolev spaces is in terms of Besov spaces.
    When \(p=2\), \(\BesovSpace{s}{2}{2}=\TLSpace{s}{2}{2}=H^s\) and in this case, the Besov spaces do coincide
    with the \(\LpSpace{2}\)-Sobolev spaces, but for \(p\ne 2\) this is not the case.
\end{remark}

The next result is contained in \cite[Theorem 3 of Section 2.4.4]{RunstSickelSobolevSpacesOfFractionalOrder}
with \(\Rngeq\) replaced by a bounded domain; though the same holds for \(\Rngeq\).

\begin{theorem}[{\cite[Theorem 3 of Section 2.4.4]{RunstSickelSobolevSpacesOfFractionalOrder}}]
    \label{Thm::Intro::Classical::DensityOfSmoothWithCptSupp}
    Fix \(p,q\in (1,\infty)\) and for \(s\in \R\) let \(\XSpace{s}{p}{q}[\Rngeq]\) denote either
    \(\BesovSpace{s}{p}{q}[\Rngeq]\) or \(\TLSpace{s}{p}{q}[\Rngeq]\), and let \(\XCircSpace{s}{p}{q}[\Rngeq]\)
    denote the closure of \(\CinftyCptSpace[\Rng]\) in \(\XSpace{s}{p}{q}[\Rngeq]\).
    \begin{enumerate}[(i)]
        \item\label{Item::Intro::Classical::DensityOfSmoothWithCptSupp::sleq1/p} If \(s\leq 1/p\), \(\XSpace{s}{p}{q}[\Rngeq]=\XCircSpace{s}{p}{q}[\Rngeq]\).
        \item\label{Item::Intro::Classical::DensityOfSmoothWithCptSupp::sg1/p} For \(L\in \Zgeq\), if \(L+1/p<s\leq L+1+1/p\), then
            \begin{equation*}
                \XCircSpace{s}{p}{q}[\Rngeq]=\left\{ f\in \XSpace{s}{p}{q}[\Rngeq] : \TraceMap[L] f=0 \right\}.
            \end{equation*}
    \end{enumerate}
\end{theorem}

\begin{remark}
    In Theorem \ref{Thm::Intro::Classical::DensityOfSmoothWithCptSupp}, the ``endpoint'' estimates
    (\(s=1/p\) in \ref{Item::Intro::Classical::DensityOfSmoothWithCptSupp::sleq1/p}
    and \(s=L+1+1/p\) in \ref{Item::Intro::Classical::DensityOfSmoothWithCptSupp::sg1/p})
    are more difficult and we do not address such endpoints in this paper--see Remark \ref{Rmk::GlobalCor::MissingEndpointsForDensity}.
\end{remark}

\section{Basic definitions and global corollaries}\label{Section::GlobalCor}
The main results of this paper are local and take place on non-compact manifolds. In this chapter, we present some global corollaries of our main
results on compact manifolds, where we restrict to simple special cases to help give the reader the main ideas.
Let \(\ManifoldN\) be a manifold with boundary.

\begin{notation}
    \begin{itemize}
        \item \(\VectorFieldsN\) denotes the 
\(\CinftySpace[\ManifoldN][\R]\)-module of
smooth vector fields on \(\ManifoldN\).
        \item \(\CinftyCptSpace[\ManifoldN]\) denotes the space of smooth functions with compact support on \(\ManifoldN\),
        with the usual topology. These functions
            may be nonzero on \(\BoundaryN\).
        \item \(\Distributions[\ManifoldN]\) denotes dual space of \(\CinftyCptSpace[\ManifoldN]\), with the usual weak topology.
        \item \(\TestFunctionsZeroN\) denotes the space of those \(f(x)\in \CinftyCptSpace[\ManifoldN]\) which vanish to infinite
            order as \(x\rightarrow \BoundaryN\). This is a closed subspace of \(\CinftyCptSpace[\ManifoldN]\) and inherits the topology.
        \item \(\DistributionsZeroN\) denotes the dual space of \(\TestFunctionsZeroN\), with the usual weak topology.
    \end{itemize}
\end{notation}

% and let \(\VectorFieldsN\) denote the 
% \(\CinftySpace[\ManifoldN][\R]\)-module of
% smooth vector fields on \(\ManifoldN\).

\begin{definition}\label{Defn::GlobalCor::HormandersCondition}
    Let \(W=\left\{ W_1,\ldots, W_r \right\}\subset \VectorFieldsN\) be a finite set of smooth vector fields on \(\ManifoldN\).
    We say \(W\) \DefnIt{satisfies H\"ormander's condition of order \(m\in \Zg\) at \(x\in \ManifoldN\)}
    if
    \begin{equation*}
        \underbrace{W_1(x),\ldots, W_r(x)}_{\text{commutators of order }1},\underbrace{\ldots, [W_i,W_j](x),\ldots}_{\text{commutators of order} 2},\underbrace{\ldots,[W_i,[W_j,W_k]](x),\ldots}_{\text{commutators of order }3},\ldots,\text{commutators of order }m
    \end{equation*}
    span \(\TangentSpace{x}{\ManifoldN}\), \(\forall x\in \ManifoldN\).
    We say \(W\) \DefnIt{satisfies H\"ormander's condition of order \(m\) on \(\ManifoldN\)} if \(W\)
    satisfies H\"ormander's condition of order \(m\) at \(x\), \(\forall x\in \ManifoldN\).
    We say \(W\) \DefnIt{staisfies H\"ormander's condition on \(\ManifoldN\)} if \(\forall x\in \ManifoldN\),
    \(\exists m\), such that \(W\) satisfies H\"ormander's condition of order \(m\) at \(x\).
\end{definition}

Given a collection of H\"ormander vector fields \(W=\left\{ W_1,\ldots, W_r \right\}\), we will often assign to them \DefnIt{formal degrees}
\(\Wdv_1,\ldots, \Wdv_r\in \Zg\), and write
\begin{equation*}
    \WWdv:=\left\{ \left( W_1,\Wdv_1 \right),\ldots, \left( W_r, \Wdv_r \right) \right\}.
\end{equation*}
We call \(\WWdv\) a set of \DefnIt{H\"ormander vector fields with formal degrees}.

\begin{example}
    When we assign \(W_j\) the degree \(\Wdv_j\in \Zg\), it means we treat \(W_j\) as a differential operator
    of ``degree'' \(\Wdv_j\) (even though it is a differential operator of order \(1\) in the classical sense).
    For example, when considering the subelliptic heat operator \(\partial_t+\sum_{j=1}^r W_j^{*}W_j\), then we use
    the vector fields with formal degrees \(\left\{ \left( \partial_t,2 \right), \left( W_1,1 \right),\ldots, \left( W_r,1 \right) \right\}\).
    See \cite[Section 1.1]{StreetMaximalSubellipticity} for a further discussion and more examples.
\end{example}

\begin{example}
    The simplest example of H\"ormander vector fields with formal degrees on \(\Rn\) is given by
    \((\partial,1):=\left\{ \left( \partial_{x_1},1 \right),\ldots,\left( \partial_{x_n},1 \right) \right\}\).
    When we choose these vector fields with formal degrees, the definitions and results in this paper
    coincide with the classical elliptic definitions and results. See, for example,
    \cite[Section 6.6.1]{StreetMaximalSubellipticity}. Thus, the definitions and results
    in this paper are a true generalization of the classical results.
\end{example}

To simplify our main results and make the easier to understand,
throughout the rest of this chapter (and only this chapter), 
we assume \(\ManifoldN\) is a \textbf{compact} manifold with boundary and
let \(W=\left\{ W_1,\ldots, W_r \right\}\) be H\"ormander vector fields on 
 \(\ManifoldN\).
We write \(\WWdv=\left\{ (W_1,\Wdv_1),\ldots, (W_r,\Wdv_r) \right\}\) for H\"ormander vector fields with formal degrees
and \(\WWo=\left\{ (W_1,1),\ldots, (W_r,1) \right\}\) for the special case when we have chosen every formal degree
equal to \(1\).
\textit{In this chapter, we restrict to the special case \(\WWo\) as that is easiest to understand. However,
in the rest of the paper we work in the more general case of \(\WWdv\) and present local results on non-compact manifolds.}

\begin{definition}\label{Defn::GlobalCor::SimpleNonChar}
    We say \(x\in \BoundaryN\) is \(\WWo\)-non-characteristic if \(\exists j\) with \(W_j(x)\not\in \TangentSpace{x}{\BoundaryN}\).
    See Definition \ref{Defn::Filtrations::RestrictingFiltrations::NonCharPoints} for the general case for \(\WWdv\).
\end{definition}

\begin{remark}
    Derridj showed that almost every boundary point is \(\WWo\)-non-characteristic \cite[Theorem 1]{DerridjSurUnTheoremeDeTraces}.
\end{remark}

% Fix \(W=\left\{ W_1,\ldots, W_r \right\}\), H\"ormander vector fields on \(\ManifoldN\).
Throughout this chapter, we make the following assumption global assumption (which is replaced by a local
assumption in the main results in this paper):

% \begin{assumption}
\medskip
\noindent\textbf{Temporary Global Assumption:}
    In this chapter (and only this chapter), we assume every \(x\in \BoundaryN\) is \(\WWo\)-non-characteristic.
\medskip
    % \end{assumption}

In Section \ref{Section::Spaces::MainDefns}, we define Besov, \(\BesovSpace{s}{p}{q}[\ManifoldN][\WWo]\), and Triebel-Lizorkin, \(\TLSpace{s}{p}{q}[\ManifoldN][\WWo]\),
spaces 
adapted to \(\WWo\)
(with \(s\in \R\) and the restrictions on \(p,q\) described in Remark \ref{Rmk::Intro::Classical::Restrcitpq}).\footnote{
    We are taking \(\FilteredSheafF=\FilteredSheafGenBy{\WWo}\) in the notation of that section.
    See, also, Definition \ref{Defn::Filtrations_Sheaves::SheafGeneratedBy}.
}
These are Banach spaces and \(s\in \R\) should be thought of as the regularity parameter.
Each \(W_j\) acts as a differential operator of degree \(1\) in the sense:
\begin{equation*}
    W_j:\BesovSpace{s}{p}{q}[\ManifoldN][\WWo]\rightarrow \BesovSpace{s-1}{p}{q}[\ManifoldN][\WWo],
    \quad
    W_j:\TLSpace{s}{p}{q}[\ManifoldN][\WWo]\rightarrow \TLSpace{s-1}{p}{q}[\ManifoldN][\WWo];
\end{equation*}
see Proposition \ref{Prop::Spaces::MappingOfFuncsAndVFs}.
In fact, under appropriate hypotheses, we more generally define the spaces
\(\BesovSpace{s}{p}{q}[\ManifoldN][\WWdv]\) and \(\TLSpace{s}{p}{q}[\ManifoldN][\WWdv]\)
where each \(W_j\) acts as a differential operator of degree \(\Wdv_j\).

\begin{notation}
    For a list \(\alpha=(\alpha_1,\alpha_2,\ldots, \alpha_L)\in \left\{ 1,\ldots, r \right\}^L\), we write
    \(W^\alpha = W_{\alpha_1}W_{\alpha_2}\cdots W_{\alpha_L}\) and \(|\alpha|=L\),
    so that \(W^{\alpha}\) is an \(|\alpha|\)-order partial differential operator.
    %We write \(\DegWdv{\alpha}=\Wdv_{\alpha_1}+\Wdv_{\alpha_2}+\cdots+\Wdv_{\alpha_r}\).
\end{notation}

The spaces \(\TLSpace{s}{p}{2}[\ManifoldN][\WWo]\) can be thought of as the appropriate
(non-isotropic) \(\LpSpace{p}\)-Sobolev spaces
as the next result shows (here, we endow \(\ManifoldN\) with any smooth, strictly positive density--the choice of density
does not matter).

\begin{corollary}\label{Cor::GlobalCor::TLmp2AreSobolevSpaces}
    For \(1<p<\infty\) and \(m\in \Zgeq\),
    \begin{equation*}
        \TLSpace{m}{p}{2}[\ManifoldN][\WWo]=\left\{ f\in \DistributionsZeroN : W^{\alpha}f\in \LpSpace{p}[\ManifoldN],\forall |\alpha|\leq L \right\},
    \end{equation*}
    \begin{equation*}
        \TLNorm{f}{m}{p}{2}[\WWo] \approx \sum_{|\alpha|\leq m} \LpNorm{W^{\alpha}f}{p}[\ManifoldN].
    \end{equation*}
\end{corollary}
\begin{proof}
Since \(\ManifoldN\) is compact, this is the special case of
Proposition \ref{Prop::Spaces::EqualsSobolev} with
\(\FilteredSheafF=\FilteredSheafGenBy{\WWo}\) (see Definition \ref{Defn::Filtrations_Sheaves::SheafGeneratedBy}).
\end{proof}

\begin{remark}
    When \(s\in (0,\infty)\setminus \Zg\), \(\BesovSpace{s}{\infty}{\infty}[\ManifoldN][\WWo]\)
    agrees with the naturally defined H\"older space with respect to the adapted Carnot--Carath\'eodory metric.
    See Corollary \ref{Cor::Zyg::Holder::ZygSpaceEqualsHolderForWWo}.
\end{remark}

Let \(X\in \VectorFieldsN\) and be in the \(\CinftySpace[\ManifoldN][\R]\) module generated by \(W_1,\ldots, W_r\)
and be such that \(X(x)\not \in \TangentSpace{x}{\BoundaryN}\), \(\forall x\in \BoundaryN\).
For \(L\in \Zgeq\) and \(f\in \CinftySpace[\ManifoldN]\) define
\begin{equation}\label{Eqn::GlobalCors::DefineTraceMap}
    \TraceMap[L][X]f:=\left(f\big|_{\BoundaryN}, Xf\big|_{\BoundaryN},\ldots, X^L f\big|_{\BoundaryN} \right).
\end{equation}

\begin{corollary}\label{Cor::GlobalCors::TraceThm}
     There exists H\"ormander vector fields with formal degrees \(\VVdv=\left\{ \left( V_1,\Vdv_1 \right),\ldots, \left( V_q, \Vdv_q \right) \right\}\)
    on \(\BoundaryN\)
    such that
    for \(L\in \Zgeq\)
    there exists a
    unique map
    \begin{equation*}
        \TraceMap[L][X]:\left( \bigcup_{\substack{1<p\leq \infty \\ 1\leq q\leq \infty\\ s>L+1/p}} \BesovSpace{s}{p}{q}[\ManifoldN][\WWo] \right)
        \bigcup
        \left( \bigcup_{\substack{1<p<\infty \\ 1< q\leq \infty\\ s>L+1/p}} \TLSpace{s}{p}{q}[\ManifoldN][\WWo] \right)
        \rightarrow \Distributions[\BoundaryN]^{L+1}
    \end{equation*}
    such that
    \begin{enumerate}[(I)]
        \item For \(f\in \CinftyCptSpace[\Rngeq]\), \(\TraceMap[L][X]f\) is given by \eqref{Eqn::GlobalCors::DefineTraceMap}.
        \item\label{Item::GlobalCors::TraceThm::Besov} \(\TraceMap[L][X]:\BesovSpace{s}{p}{q}[\ManifoldN][\WWo]\rightarrow \prod_{l=0}^L \BesovSpace{s-l-1/p}{p}{q}[\BoundaryN][\VVdv]\)
            is continuous for \(1<p\leq \infty\), \(1\leq q\leq \infty\), \(s>L+1/p\).
        \item\label{Item::GlobalCors::TraceThm::TL} \(\TraceMap[L][X]:\TLSpace{s}{p}{q}[\ManifoldN][\WWo]\rightarrow \prod_{l=0}^L \BesovSpace{s-l-1/p}{p}{p}[\BoundaryN][\VVdv]\) is continuous
            for \(1<p<\infty\), \(1<q\leq \infty\), \(s>L+1/p\).
    \end{enumerate}
    Furthermore, the maps in \ref{Item::GlobalCors::TraceThm::Besov} and \ref{Item::GlobalCors::TraceThm::TL}
    are retractions. Moreover there is a continuous, linear map
    \(\TraceInverseMap[L][X]:\Distributions[\BoundaryN]^{L+1}\rightarrow \DistributionsZeroN\)
    such that
    \begin{enumerate}[(i)]
        \item\label{Item::GlobalCors::TraceThm::InverseBesov} \(\TraceInverseMap[L][X]:\prod_{l=0}^L \BesovSpace{s-l-1/p}{p}{q}[\BoundaryN][\VVdv]\rightarrow \BesovSpace{s}{p}{q}[\ManifoldN][\WWo]\)
            is continuous for \(s\in \R\), \(p,q\in [1,\infty]\).
        \item\label{Item::GlobalCors::TraceThm::InverseTL} \(\TraceInverseMap[L][X]:\prod_{l=0}^L \BesovSpace{s-l-1/p}{p}{p}[\BoundaryN][\VVdv]\rightarrow \TLSpace{s}{p}{q}[\ManifoldN][\WWo]\)
            is continuous for \(s\in \R\), \(p\in (1,\infty)\), \(q\in (1,\infty]\).
        \item \(\TraceMap[L][X]\TraceInverseMap[L][X]=I\) whenever it makes sense; i.e., whenever the range space of
            \(\TraceInverseMap[L][X]\) in \ref{Item::GlobalCors::TraceThm::InverseBesov} or \ref{Item::GlobalCors::TraceThm::InverseTL}
            coincides with the domain space of \(\TraceMap[L][X]\) in \ref{Item::GlobalCors::TraceThm::Besov} or \ref{Item::GlobalCors::TraceThm::TL}.    \end{enumerate}
\end{corollary}
\begin{proof}
    See Theorem \ref{Thm::Trace::ForwardMap} and Corollary \ref{Cor::Trace::InverseMap}.
\end{proof}

\begin{corollary}\label{Cor::GlobalCor::DensityOfSmoothWithCptSupp}
    Fix \(p,q\in (1,\infty)\) and for \(s\in \R\) let \(\XSpace{s}{p}{q}[\ManifoldN][\WWo]\) denote either
    \(\BesovSpace{s}{p}{q}[\ManifoldN][\WWo]\) or \(\TLSpace{s}{p}{q}[\ManifoldN][\WWo]\), and let \(\XCircSpace{s}{p}{q}[\ManifoldN][\WWo]\)
    denote the closure of \(\CinftyCptSpace[\InteriorN]\) in \(\XSpace{s}{p}{q}[\ManifoldN][\WWo]\).
    \begin{enumerate}
        \item\label{Item::GlobalCor::DensityOfSmoothWithCptSupp::sleq1/p} If \(s< 1/p\), \(\XSpace{s}{p}{q}[\ManifoldN][\WWo]=\XCircSpace{s}{p}{q}[\ManifoldN][\WWo]\).
        \item\label{Item::GlobalCor::DensityOfSmoothWithCptSupp::sg1/p} For \(L\in \Zgeq\), if \(L+1/p<s< L+1+1/p\), then
            \begin{equation*}
                \XCircSpace{s}{p}{q}[\ManifoldN][\WWo]=\left\{ f\in \XSpace{s}{p}{q}[\ManifoldN][W] : \TraceMap[L][X] f=0 \right\}.
            \end{equation*}
    \end{enumerate}
\end{corollary}
\begin{proof}
    This is a special case of Corollary \ref{Cor::Trace::Vanish::EveryPointNonChar}, below,
    with \(\lambda=1\) and 
\(\FilteredSheafF=\FilteredSheafGenBy{\WWo}\) (see Definition \ref{Defn::Filtrations_Sheaves::SheafGeneratedBy}).
\end{proof}

\begin{remark}\label{Rmk::GlobalCor::MissingEndpointsForDensity}
    Comparing Corollary \ref{Cor::GlobalCor::DensityOfSmoothWithCptSupp} and Theorem \ref{Thm::Intro::Classical::DensityOfSmoothWithCptSupp},
    we see that Corollary \ref{Cor::GlobalCor::DensityOfSmoothWithCptSupp} is missing the ``endpoints''
    \(s=1/p+L\). While we expect these endpoints estimates to hold, the proofs in the classical setting are delicate
    and we have not yet generalized them to this setting.
\end{remark}

\begin{corollary}\label{Cor::GlobalCor::ExtendToLargerMfld}
    Let \(\ManifoldM\) be a smooth, compact manifold without boundary such that
    \(\ManifoldN\subseteq\ManifoldM\) is a closed, embedded, codimension \(0\) submanifold (with boundary).
    Suppose \(\Wh_1,\ldots, \Wh_r\) are H\"ormander vector fields on \(\ManifoldM\)
    with \(\Wh_j\big|_{\ManifoldN}=W_j\).
    The map \(f\mapsto f\big|_{\TestFunctionsZeroN}\), \(\Distributions[\ManifoldM]\rightarrow \DistributionsZeroN\)
    is a retraction \(\BesovSpace{s}{p}{q}[\ManifoldM][\Who]\rightarrow \BesovSpace{s}{p}{q}[\ManifoldN][\WWo]\)
    and \(\TLSpace{s}{p}{q}[\ManifoldM][\Who]\rightarrow \TLSpace{s}{p}{q}[\ManifoldN][\WWo]\),
    \(\forall s\in \R\) and under the restrictions on \(p,q\) in Remark \ref{Rmk::Intro::Classical::Restrcitpq}.
\end{corollary}
\begin{proof}
    Since \(\ManifoldN\) and \(\ManifoldM\) are compact, this is the special case of
    Theorem \ref{Thm::Spaces::Extension} with
    \(\FilteredSheafF=\FilteredSheafGenBy{\WWo}\)
    and \(\FilteredSheafFh=\FilteredSheafGenBy{\Who}\)
    (see Definition \ref{Defn::Filtrations_Sheaves::SheafGeneratedBy}).
\end{proof}

\begin{remark}\label{Rmk::GlobalCor::IntrinsicWithExtension}
    Corollary \ref{Cor::GlobalCor::ExtendToLargerMfld} shows that:
    \begin{enumerate}[(i)]
        \item We define \(\BesovSpace{s}{p}{q}[\ManifoldN][\WWo]\) and \(\TLSpace{s}{p}{q}[\ManifoldN][\WWo]\)
            intrinsically on \(\ManifoldN\). Similar to the
            classical setting (see \eqref{Eqn::Spaces::Classical::ClassicalDefnOfASpace}) the spaces may be equivalently
            defined as restrictions of spaces on manifolds without boundary (see Corollary \ref{Cor::Spaces::Extend::Consequences::RestrictionSpaceDefn}
            for this equivalence). Such Besov and Triebel--Lizorkin
            spaces on manifolds without boundary have been described (in various levels of generality)
            by many authors--see \cite[Chapter 6]{StreetMaximalSubellipticity} for such a theory and \cite[Section 6.14]{StreetMaximalSubellipticity}
            for many references.
        \item\label{Item::GlobalCor::IntrinsicWithExtension::DoesNotDependOnAmbientVfs} If one were to instead 
        (equivalently)
        define \(\BesovSpace{s}{p}{q}[\ManifoldN][\WWo]\) and \(\TLSpace{s}{p}{q}[\ManifoldN][\WWo]\) in terms
            of restrictions of corresponding spaces on a ambient manifold without boundary (as in Corollary \ref{Cor::Spaces::Extend::Consequences::RestrictionSpaceDefn},
             thereby generalizing the classical definition \eqref{Eqn::Spaces::Classical::ClassicalDefnOfASpace}),
            Corollary \ref{Cor::GlobalCor::ExtendToLargerMfld} shows 
            %that this gives the same space (see, also, Corollary \ref{Cor::Spaces::Extend::Consequences::RestrictionSpaceDefn}), and 
            that
            the definition does not depend on the ambient space
            \(\ManifoldM\) or the values of the ambient H\"ormander vector fields \(\Wh_1,\ldots, \Wh_r\) on \(\ManifoldM\setminus \ManifoldN\)
            (see, also, Corollary \ref{Cor::Spaces::Extend::Consequences::RestrictionSpaceDefn}).
    \end{enumerate}
\end{remark}

    \subsection{Informal description of the norms}\label{Section::GlobalCor::InfomralNorms}
    In this section, we informally describe the norms
\(\BesovNorm{\cdot}{s}{p}{q}[\ManifoldN][\WWo]\) and \(\TLNorm{\cdot}{s}{p}{q}[\ManifoldN][\WWo]\),
under the assumptions of this chapter (which are more restrictive than the assumptions we use in the rest of the paper). 
The rigorous definition, in the more general setting, is contained in Section \ref{Section::Spaces::MainDefns};
in that section it is also explained that none of the choices made here affect the equivalence class of the norm.
The reader might also find useful \cite[Section 6.1]{StreetMaximalSubellipticity}, which describes the norms
on a manifold without boundary.

Fix any smooth, strictly positive density \(\Vol\) on \(\ManifoldN\) (the choice of density is not relevant).
Abstractly, we define operators \(D_j:\DistributionsZeroN\rightarrow \CinftySpace[\ManifoldN]\), \(j\in \Zgeq\), which
play the role of the Littlewood--Paley projectors. The norms are then defined in the usual way:
\begin{equation*}
    \BesovNorm{f}{s}{p}{q}[\ManifoldN][\WWo]:=\lqLpNorm{\left\{2^{js} D_j f \right\}_{j\in \Zgeq}}{p}{q}[\ManifoldN,\Vol],
\end{equation*}
\begin{equation*}
    \TLNorm{f}{s}{p}{q}[\ManifoldN][\WWo]:=\LplqNorm{\left\{2^{js} D_j f \right\}_{j\in \Zgeq}}{p}{q}[\ManifoldN,\Vol].
\end{equation*}
In this section, we describe one of the equivalent ways to define the \(D_j\).
Because multiplication by smooth functions is a continuous map on the spaces (see Proposition \ref{Prop::Spaces::MappingOfFuncsAndVFs} \ref{Item::Spaces::MappingOfFuncsAndVFs::Funcs}),
the spaces are localizable, and it suffices to define operators \(D_j\) near each point \(x_0\in \ManifoldN\).
The more difficult case is \(x_0\in \BoundaryN\), which we focus on in this section (for a similar informal discussion in the easier the case \(x\in \InteriorN\),
see \cite[Section 6.1]{StreetMaximalSubellipticity}).

Fix \(x_0\in \BoundaryN\); by hypothesis, \(x_0\) is \(\WWo\)-non-characteristic. Therefore, there exists \(j_0\) with \(W_{j_0}(x_0)\not\in \TangentSpace{x_0}{\BoundaryN}\).
Without loss of generality, assume \(j_0=1\).  Moreover, by possibly replacing \(W_1\) with \(-W_1\), we may assume \(W_1(x_0)\) points into \(\InteriorN\), in
the sense that \(e^{tW_1}x_0\) exists for \(t\geq 0\) small and for \(t>0\) small \(e^{tW_1} x_0\in \InteriorN\).

Let \(U\) be a small \(\ManifoldN\)-neighborhood of \(x_0\); we allow \(U\) to shrink from line to line, as needed.
By continuity, \(W_1(x)\not \in \TangentSpace{x}{\BoundaryN}\), \(\forall x\in U\cap \BoundaryN\).
Recursively associate formal degrees to vector fields as follow:
\begin{itemize}
    \item \(W_1,\ldots, W_r\) have degree \(1\).
    \item If \(Y_1\) has degree \(d_1\) and \(Y_2\) has degree \(d_2\), then assign to \([Y_1,Y_2]\) degree \(d_1+d_2\) (the same vector field might be assigned more than one formal degree).
\end{itemize}
% For \(2\leq k\leq r\), take \(a_k\in \CinftySpace[U]\) such that if \(\Wt_k:=W_k-a_k W_1\), then
% \(\Wt_k\in \TangentSpace{x}{\BoundaryN}\), \(\forall x\in U\cap \BoundaryN\). Set \(\Wt_1:=W_1\),
% and \(\WtWo=\left\{ \left( \Wt_1,1 \right),\ldots, \left( \Wt_r, 1 \right) \right\}\).
% For the remainder of the construction, we use \(\WtWo\) in place of \(\WWo\).

% Recursively define formal degrees to vector fields as follows:
% \begin{itemize}
%     \item \(\Wt_1,\ldots, \Wt_r\) have degree \(1\).
%     \item If \(Y_1\) has degree \(d_1\) and \(Y_2\) has degree \(d_2\), then assign to \([Y_1,Y_2]\) degree \(d_1+d_2\) (the same vector field might be assigned more than one formal degree).
% \end{itemize}
Let \(\ZZdv=\left\{ \left( Z_0, 1 \right), \left( Z_1, \Zdv_1 \right),\ldots, \left( Z_q,\Zdv_q \right) \right\}\)
be an enumeration of all such vector fields with formal degrees \(\leq m\); where \(W_1,\ldots, W_r\)
satisfying H\"ormander's condition of order \(m\) at \(x_0\); here we have taken \((Z_0,1)=(W_1,1)\).
See \cite[Section \ref*{CC::Section::BndryVfsWithFormaLDegrees}]{StreetCarnotCaratheodoryBallsOnManifoldsWithBoundary} for how the choice of taking vector fields with degree \(\leq m\) is used.

For \(1\leq k\leq q\) take \(a_k\in \CinftySpace[U][\R]\) such that if \(X_k=Z_k-a_kZ_0=Z_k-a_kW_1\), then
\(X_k(x)\in \TangentSpace{x}{\BoundaryN}\), \(\forall x\in U\cap \BoundaryN\), \(\forall 1\leq k\leq q\).  Set \(\left( X_0,1 \right):=(Z_0,1)=(W_1,1)\).
Let \(\XXdv=\left\{ (X_0,1),(X_1,d_1),\ldots, (X_q,d_1) \right\}\).

% Let \(\XXdv=\left\{ (X_0,1),(X_1,d_1),\ldots, (X_q,d_1) \right\}\) be an enumeration of all such vector fields with formal degrees \(\leq m\); where \(W_1,\ldots, W_r\)
% satisfying H\"ormander's condition of order \(m\) at \(x_0\); here we have taken \((X_0,1)=(\Wt_1,1)\). Note that, for \(1\leq j\leq q\), \(X_j(x)\in \TangentSpace{x}{\BoundaryN}\), \(\forall x\in U\cap \BoundaryN\).
% See \cite[Section \ref*{CC::Section::BndryVfsWithFormaLDegrees}]{StreetCarnotCaratheodoryBallsOnManifoldsWithBoundary} for how the choice of taking vector fields with degree \(\leq m\) is used.

Let \(\psi\in \CinftyCptSpace[U]\) equal \(1\) on a neighborhood of \(x_0\).  Fix \(\vsig_0(t_0,t_1,\ldots, t_q)\in \SchwartzSpace[\R^{1+q}]\) satisfying the following:
\begin{itemize}
    \item \(\int \vsig_0=1\) and \(\int t^{\alpha}\vsig_0=0\), \(\forall |\alpha|\geq 1\).
    \item \(\supp(\vsig_0)\subseteq [0,\infty)\times \R^q\).
\end{itemize}
Such a \(\vsig_0\) exists: see Lemma \ref{Lemma::Spaces::Classical::ExistsGoodPsi}.
Set 
\begin{equation*}
    \Dild{2^{j}}{\vsig}(t_0,t_1,\ldots, t_q):= 2^{j(1+d_1+d_2+\dots+d_q)}\vsig(2^{j}t_0, 2^{jd_1}t_1,\ldots, 2^{jd_q}t_q),
\end{equation*}
\begin{equation*}
    \vsig_j:=\vsig_0 - \Dild{2^{-1}}{\vsig_0}, \quad \forall j\geq 1.
\end{equation*}
Note that \(\vsig_j=\vsig_1\), \(\forall j\geq 1\),
and 
\begin{equation}\label{Eqn::GlobalCor::Norms::SumToDelta}
    \sum_{j=0}^L \Dild{2^j}{\vsig_j}=\Dild{2^L}{\vsig_0}\xrightarrow{L\rightarrow \infty} \delta_0(t),
\end{equation}
where \(\delta_0\) is the Dirac \(\delta\) function at \(0\).

Fix \(\eta\in \CinftyCptSpace[\R^{1+q}]\) with \(\eta\) supported on a small neighborhood of \(0\) and \(\eta=1\) near \(0\).
Set,
\begin{equation*}
    D_j f(x) = \psi(x) \int f\left( e^{t_0X_0+t_1X_1+\cdots +t_q X_q} x \right) \eta(t) \Dild{2^j}{\vsig_j}(t)\: dt.
\end{equation*}
By the choice of \(\vsig_0\) and the vector fields \(X_0,X_1,\ldots, X_q\), we have \(e^{t_0X_0+t_1X_1+\cdots +t_q X_q} x\in \InteriorN\),
for \(x\in U\), \(t_0>0\); here we have again possibly shrunk \(U\). Using this, it can be shown that \(D_j f\) makes sense,
\(\forall f\in \DistributionsZeroN\) and yields a smooth function.
Finally, by \eqref{Eqn::GlobalCor::Norms::SumToDelta}, we have
\begin{equation*}
    \sum_{j\in \Zgeq} D_j f(x) = \psi(x) f(x).
\end{equation*}
These operators, \(D_j\), are described in greater detail and generality in Section \ref{Section::Spaces::Multiplication}.
Section \ref{Section::Spaces::Classical} explains how the definitions here generalize the classical setting on Euclidean space.

\section{Vector fields and sheaves}\label{Chapter::VectorFieldsAndSheaves}
Let \(\ManifoldN\) be a smooth manifold with boundary, and let
\(\WWdv=\left\{ \left( W_1,\Wdv_1 \right),\ldots, \left( W_r,\Wdv_r \right) \right\}\subset \VectorFieldsN\times \Zg\)
be H\"ormander vector fields with formal degrees on \(\ManifoldN\).
A main goal of this paper is given \(\WWdv\), define function spaces adapted
to \(\WWdv\), which treat \(W_j\) as an operator of degree \(\Wdv_j\).
Closely tied to these function spaces is the Carnot--Carathe\'odory metric geometry on \(\ManifoldN\)
induced by \(\WWdv\); the balls in this geometry are defined for \(x\in \ManifoldN\) and \(\delta>0\) as:
\begin{equation*}
\begin{split}
        \BWWdv{x}{\delta}:=
        \bigg\{ &\exists \gamma:[0,1]\rightarrow \ManifoldN, \gamma(0)=x,\gamma(1)=y,
        \\&\gamma\text{ is absolutely continuous},
        \\&\gamma'(t)=\sum_{j=1}^r a_j(t) \delta^{\Wdv_j} W_j(\gamma(t))\text{ for almost every }t,
        \\&a_j\in \LpSpace{\infty}[[0,1]], \BLpNorm{\sum |a_j|^2 }{\infty}[[0,1]]<1
        \bigg\},        
\end{split}
\end{equation*}
with corresponding extended metric\footnote{An extended metric satisfies all the usual axioms of a metric, 
but may take the value \(+\infty\). If \(\ManifoldN\)
is not connected, then \(\DistWWdv[x][y]=\infty\) for \(x\) and \(y\) in different connected components.}
\begin{equation}\label{Eqn::VectorFields::DefineDistance}
    \DistWWdv[x][y]=\inf\left\{ \delta>0 : y\in \BWWdv{x}{\delta} \right\}.
\end{equation}
Since the results in this paper are local, it is often more convenient to work with the metric
\begin{equation}\label{Eqn::VectorFields::DefineMetric}
    \MetricWWdv[x][y]:=\min\left\{ \DistWWdv[x][y],1 \right\}.
\end{equation}

    \subsection{Filtrations of sheaves}
    There are many different lists of H\"ormander vector fields with formal degrees
which are equivalent for our purposes. This will be exploited when defining
H\"ormander vector fields with formal degrees on the non-characteristic part of \(\BoundaryN\).
Because of this, it is useful to state our definitions and theorems with this equivalence in mind.
The most natural way to do so is using the language of sheaves--see
\cite[Section \ref*{CC::Section::Sheaves}]{StreetCarnotCaratheodoryBallsOnManifoldsWithBoundary} for a more detailed description.

\begin{definition}\label{Defn::Filtrations_Sheaves::Filtration_Sheaves}
    We say \(\FilteredSheafF\) is a \DefnIt{filtration of sheaves of vector fields on \(\ManifoldN\)}
    if for each \(d\in \Zg\) and \(\Omega\subseteq \ManifoldN\) open,
    \(\FilteredSheafF[\Omega][d]\) is a \(\CinftySpace[\Omega][\R]\)-submodule of \(\VectorFields{\Omega}\),
    for each \(d\in \Zg\), \(\Omega\mapsto \FilteredSheafF[\Omega][d]\) satisfies the sheaf axioms,
    and \(d_1\leq d_2\implies \FilteredSheafF[\Omega][d_1]\subseteq \FilteredSheafF[\Omega][d_2]\).
\end{definition}

\begin{definition}\label{Defn::Filtrations_Sheaves::SheafGeneratedBy}
    Fix \(\Omega\subseteq \ManifoldN\) open. For \(\sS\subset \VectorFields{\Omega}\times \Zg\)
    finite and \(\Omega_1\subseteq \Omega\) open and \(d\in \Zg\) set
    \begin{equation*}
        \FilteredSheafGenBy{\sS}[\Omega_1][d]:=\CinftySpace[\Omega_1][\R]\text{-module generated by } \left\{ X : \exists (X,k)\in \sS, k\leq d \right\}.
    \end{equation*}
    \(\FilteredSheafGenBy{\sS}\) is a filtration of sheaves of vector fields on \(\Omega\) (see \cite[Lemma \ref*{CC::Lemma::Sheaves::Sheaves::ModuleGenByFiniteSetIsSheaf}]{StreetCarnotCaratheodoryBallsOnManifoldsWithBoundary}).
\end{definition}

\begin{definition}\label{Defn::Filtrations_Sheaves::Hormander_Filtration_Sheaves}
    We say \(\FilteredSheafF\) is a \DefnIt{H\"ormander filtration\footnote{In \cite{StreetCarnotCaratheodoryBallsOnManifoldsWithBoundary},
    the longer terminology \textit{filtration of sheaves of vector fields of finite type satisfying H\"ormander's condition} was used,
    as separating the notions of finite type and H\"ormander's condition was relevant in that paper. In this paper, we only work
    with H\"ormander filtrations of sheaves of vector fields, so these distinctions do not matter, and we use this more compact
    terminology.} 
    of sheaves of vector fields on \(\ManifoldN\)} if
    \(\FilteredSheafF\) is a filtration of sheaves of vector fields on \(\ManifoldN\), and for each \(x\in \ManifoldN\),
    \(\exists\) an open neighborhood \(\Omega\subseteq \ManifoldN\) of \(x\), and
    \(\WWdv=\left\{ (W_1,\Wdv_1),\ldots, (W_r,\Wdv_r) \right\}\subset \VectorFields{\Omega}\times \Zg\), H\"ormander
    vector fields with formal degrees on \(\Omega\), such that
    \begin{equation*}
        \FilteredSheafF\big|_{\Omega}=\FilteredSheafGenBy{\WWdv},
    \end{equation*}
    where \(\FilteredSheafF\big|_{\Omega}\) denotes the filtration of sheaves of vector fields on
    \(\Omega\) given by restricting \(\FilteredSheafF\) to open subsets of \(\Omega\).
\end{definition}

\begin{lemma}[{\cite[Proposition \ref*{CC::Prop::Sheaves::LieAlg::HormandersCondiIsEquivToOtherNotions} \ref*{CC::Item::::Sheaves::LieAlg::HormandersCondiIsEquivToOtherNotions::LocallyHorCond}\(\Rightarrow\)\ref*{CC::Item::::Sheaves::LieAlg::HormandersCondiIsEquivToOtherNotions::HorCondOnCptSets}]{StreetCarnotCaratheodoryBallsOnManifoldsWithBoundary}}]
    \label{Lemma::Filtrations::GeneratorsOnRelCptSet}
    Let \(\FilteredSheafF\) be a H\"ormander filtration of sheaves of vector fields on \(\ManifoldN\).
    Then, for each \(\Omega\Subset \ManifoldN\) open and relatively compact,
    there exist 
    \(\WWdv=\left\{ (W_1,\Wdv_1),\ldots, (W_r,\Wdv_r) \right\}\subset \VectorFields{\Omega}\times \Zg\), H\"ormander
    vector fields with formal degrees on \(\Omega\), such that
    \begin{equation*}
        \FilteredSheafF\big|_{\Omega}=\FilteredSheafGenBy{\WWdv}.
    \end{equation*}
\end{lemma}

\begin{remark}
    See \cite[Proposition \ref*{CC::Prop::Sheaves::LieAlg::HormandersCondiIsEquivToOtherNotions}]{StreetCarnotCaratheodoryBallsOnManifoldsWithBoundary}
    for another equivalent characterization of Definition \ref{Defn::Filtrations_Sheaves::Hormander_Filtration_Sheaves}.
\end{remark}

\begin{notation}
    Let \(\FilteredSheafF\) be a filtration of sheaves of vector fields on \(\ManifoldN\).
    We write \(\LieFilteredSheafF\) for the smallest filtration of sheaves of vector fields containing
    \(\FilteredSheafF\) and such that
    \begin{equation*}
        X\in \LieFilteredSheafF[\Omega][j], Y\in \LieFilteredSheafF[\Omega][k]\implies
        [X,Y]\in \LieFilteredSheafF[\Omega][j+k].
    \end{equation*}
\end{notation}

\begin{lemma}[{\cite[Lemma \ref*{CC::Lemma::Sheaves::LieAlg::LieAlgebraFilIsGenByXXd} and Proposition \ref*{CC::Prop::Sheaves::LieAlg::FFiniteTypeAndHorImpliesSameForLie}]{StreetCarnotCaratheodoryBallsOnManifoldsWithBoundary}}]
    \label{Lemma::Filtrations::GeneratorsForLieFiltration}
    Let \(\FilteredSheafF\) be a H\"ormander filtration of sheaves of vector fields on \(\ManifoldN\).
    Then, \(\LieFilteredSheafF\) is also a H\"ormander filtration of sheaves of vector fields on \(\ManifoldN\).
    In fact, for each \(\Omega\Subset \ManifoldN\) open and relatively compact,
    there exist \(\XXde=\left\{ (X_1,\Xde_1),\ldots, (X_q,\Xde_q) \right\}\subset \VectorFields{\Omega}\times \Zg\)
    such that \(\LieFilteredSheafF\big|_{\Omega}=\FilteredSheafGenBy{\XXde}\)
    and \(\forall x\in \Omega\), \(\Span\left\{ X_1(x),\ldots, X_q(x) \right\}=\TangentSpace{x}{\ManifoldN}\).
    
    Such an \(\XXde\) can be defined as follows. 
    Let \(\Omegah\Subset \ManifoldN\) have \(\Omega\Subset\Omegah\) and
    \(\FilteredSheafF\big|_{\Omegah}=\FilteredSheafGenBy{\WWdv}\)
    where \(\WWdv=\left\{ (W_1, \Wdv_1),\ldots, (W_r, \Wdv_r) \right\}\subset \VectorFields{\Omegah}\times \Zg\)
    are H\"ormander vector fields with formal degrees on \(\Omegah\).
    Recursively define a set of vector fields with formal degrees \(\GenWWdv\) to be the smallest set with:
    \begin{itemize}
        \item \(\WWdv\subseteq \GenWWdv\),
        \item If \((V_1,e_1),(V_2,e_2)\in \GenWWdv\), then \(([V_1,V_2],e_1+e_2)\in \GenWWdv\).
    \end{itemize}
    By compactness, \(W_1,\ldots, W_r\) satisfy H\"ormander's condition of order \(m\in \Zg\) on \(\overline{\Omega}\).
    Let \(\XXde=\left\{ (V,f)\in \GenWWdv : f\leq m\max\{\Wdv_1,\ldots, \Wdv_r\} \right\}\).
\end{lemma}

\begin{remark}
    In the sequel, we often begin with a H\"ormander filtration of sheaves of vector fields on \(\ManifoldN\),
    and for some \(\Omega\Subset \ManifoldN\) open and relatively compact we take
    H\"ormander vector fields with formal degrees \(\WWdv\)
    such that \(\FilteredSheafF\big|_{\Omega}=\FilteredSheafGenBy{\WWdv}\), as in Lemma \ref{Lemma::Filtrations::GeneratorsOnRelCptSet}.
    Not only will our definitions not depend on the particular choice of generators \(\WWdv\), but in fact
    will only depend on \(\FilteredSheafF\) through \(\LieFilteredSheafF\). More precisely,
    if \(\FilteredSheafG\) is another H\"ormander filtration of sheaves of vector fields on \(\ManifoldN\)
    such that \(\LieFilteredSheafF=\LieFilteredSheafG\), then our definitions will be the same
    whether we use \(\FilteredSheafF\) or \(\FilteredSheafG\). See, for example Theorem \ref{Thm::Spaces::OnlyDependsOnLieFiltration}.
\end{remark}

\begin{remark}
    In the literature on maximal subellipticity, instead of using
    filtrations of sheaves of vector fields, one often
    takes as given the globally defined
    filtration of \(\CinftySpace[\ManifoldN][\R]\)-modules
    which in our case would be
    \(\FilteredSheafF[\ManifoldN][1]\subseteq \FilteredSheafF[\ManifoldN][2]\subseteq\cdots\).
    % This is an equivalent approach, but the localizable nature of sheaves is more adapted to our results (which are all local),
    % and perhaps more common in mathematics more widely.
See \cite[Remark \ref*{CC::Rmk::Sheaves::SheavesAndDistributionsEquiv}]{StreetCarnotCaratheodoryBallsOnManifoldsWithBoundary}
for some comments on why filtrations of sheaves are more appropriate for our situation.
\end{remark}

    \subsection{Filtrations of sheaves on submanifolds and non-characteristic points}\label{Section::VectorFieldsAndSheaves::RestrictingSheavesAndNonCharPoints}
    We wish to restrict a filtration of sheaves of vector fields
to a submanifolds, keeping only those vector fields which
are tangent to the submanifold.
Moreover, we are only interested in H\"ormander filtrations of sheaves
of vector fields and wish to understand settings where this property
is preserved under this restriction. 
In \cite[Section \ref*{CC::Section::Sheaves::Restriction}]{StreetCarnotCaratheodoryBallsOnManifoldsWithBoundary},
a general definition was given and
two settings were presented where one could preserve the relevant properties:
embedded, co-dimension \(0\) submanifold
with boundary (Proposition \ref{Prop::Filtrations::RestrictingFiltrations::CoDim0Restriction}) 
and restricting to the non-characteristic boundary (Proposition \ref{Prop::Filtrations::RestrictingFiltrations::LieFIsHormanderFiltration}).
We describe these results here.

% Sheaves can be easily restricted to open sets, but restricting a sheaf to a non-open set introduces several difficulties;
% it is not even clear what the right definition should be.
% In \cite[Section \ref*{CC::Section::Sheaves::Restriction}]{StreetCarnotCaratheodoryBallsOnManifoldsWithBoundary}, two settings were
% introduced where one can make good sense of this restriction and it has good properties: restricting to an embedded, co-dimension \(0\) submanifold
% with boundary (Proposition \ref{Prop::Filtrations::RestrictingFiltrations::CoDim0Restriction}) 
% and restricting to the non-characteristic boundary (Proposition \ref{Prop::Filtrations::RestrictingFiltrations::LieFIsHormanderFiltration}).

Throughout this section, \(\ManifoldN\) will be a smooth manifold, possibly with boundary.
We begin with the general definition.
\begin{definition}\label{Defn::Filtrations::RestrictingFiltrations::RestrictedFiltration}
    Let \(\FilteredSheafF\) be a filtration of sheaves of vector fields on \(\ManifoldN\)
    and \(S\subseteq \ManifoldN\) an embedded sub-manifold. We let
    \(\RestrictFilteredSheaf{\FilteredSheafF}{S}\) be the smallest filtration of sheaves
    of vector fields on \(\ManifoldN\) such for \(U\subseteq S\) open,
    \begin{equation}\label{Eqn::Filtrations::RestrictingFiltrations::RestrictedFiltration::Defn}
        \left\{X\big|_U : \exists \Omega\subseteq \ManifoldN\text{ open }, \Omega\cap S=U, X\in \FilteredSheafF[\Omega][d], X(x)\in \TangentSpace{x}{S}, \forall x\in U  \right\}\subseteq \RestrictFilteredSheaf{\FilteredSheafNoSetF[d]}{S}[U].    
    \end{equation}
\end{definition}

\begin{remark}
    In fact, one has equality in \eqref{Eqn::Filtrations::RestrictingFiltrations::RestrictedFiltration::Defn}
    as can be shown with a partition of unity argument. However,
    in the special cases we are interested in, 
    this equality follows from more precise characterizations
    of the local generators \(\RestrictFilteredSheaf{\FilteredSheafF}{S}\) which are required for our proofs.
    See, e.g., Proposition \ref{Prop::Filtrations::RestrictingFiltrations::CoDim0Restriction}.
\end{remark}

\begin{proposition}[{\cite[Proposition \ref*{CC::Prop::Sheaves::Restrict::Codim0::MainRestrictionResult}]{StreetCarnotCaratheodoryBallsOnManifoldsWithBoundary}}]
    \label{Prop::Filtrations::RestrictingFiltrations::CoDim0Restriction}
    Suppose \(\ManifoldM\) is a smooth manifold (possibly with boundary) and \(\ManifoldN\subseteq \ManifoldM\)
    is a co-dimension \(0\) embedded submanifold (possibly with boundary). Let \(\FilteredSheafF\) be a H\"ormander filtration of sheaves of vector fields
    on \(\ManifoldM\).
    Then, \(\RestrictFilteredSheaf{\FilteredSheafF}{\ManifoldN}\) is a H\"ormander filtration of sheaves
    of vector fields on \(\ManifoldN\).
    Moreover, if \(\Omega\subseteq\ManifoldM\) is open and \(\FilteredSheafF\big|_{\Omega}=\FilteredSheafGenBy{\WWdv}\),
    where \(\WWdv=\left\{ (W_1,\Wdv_1),\ldots, (W_r,\Wdv_r) \right\}\) are H\"ormander vector fields with formal degrees
    on \(\Omega\), then
    \(\RestrictFilteredSheaf{\FilteredSheafF}{\ManifoldN}\big|_{\Omega\cap\ManifoldN}=\FilteredSheafGenBy{(W\big|_{\ManifoldN\cap \Omega}, \Wdv)}=\RestrictFilteredSheaf{\FilteredSheafF}{\ManifoldN\cap \Omega}\),
    where
    \((W\big|_{\ManifoldN\cap \Omega}, \Wdv)=\left\{ (W_1\big|_{\ManifoldN\cap \Omega}, \Wdv_1),\ldots, (W_r\big|_{\ManifoldN\cap \Omega}, \Wdv_r) \right\}\).
\end{proposition}

Proposition \ref{Prop::Filtrations::RestrictingFiltrations::CoDim0Restriction} shows that we can restrict
a H\"ormander filtration of sheaves of vector fields to a co-dimension \(0\) submanifold with boundary.
The next proposition shows that every such H\"ormander filtration of sheaves of vector fields
on a manifold with boundary arises in this way.

\begin{proposition}[{\cite[Proposition \ref*{CC::Prop::Sheaves::Restrict::Bndry::NonCharIsNonChar}]{StreetCarnotCaratheodoryBallsOnManifoldsWithBoundary}}]
    \label{Prop::Filtrations::RestrictingFiltrations::CoDim0CoRestriction}
    Let \(\ManifoldN\) be a smooth manifold with boundary and \(\FilteredSheafF\) a H\"ormander filtration
    of sheaves of vector fields on \(\ManifoldN\).
    Then, there exists a smooth manifold \(\ManifoldM\) (without boundary),
    such that \(\ManifoldN\subseteq \ManifoldM\) is a closed, co-dimension \(0\), embedded submanifold,
    and \(\FilteredSheafFh\) a H\"ormander filtration of sheaves of vector fields on \(\ManifoldM\)
    such that 
    \(\RestrictFilteredSheaf{\FilteredSheafFh}{\ManifoldN}=\FilteredSheafF\).
\end{proposition}

Next, we turn to defining the non-characteristic boundary and restrictions of filtrations of sheaves to it.

\begin{definition}
    Let \(\FilteredSheafF\) be a H\"ormander filtration of sheaves of vector fields on \(\ManifoldN\).
    We define \(\degBoundaryNF:\BoundaryN\rightarrow \Zg\) by, for \(x_0\in \BoundaryN\),
    \begin{equation*}
        \degBoundaryNF(x_0):=\min \left\{ d\in \Zg : \exists \Omega\subseteq \ManifoldN\text{ open}, x_0\in \Omega, X\in \LieFilteredSheafF[\Omega][d], X(x_0)\not\in \TangentSpace{x_0}{\BoundaryN}\right\}.
    \end{equation*}
\end{definition}

\(\degBoundaryNF:\BoundaryN\rightarrow \Zg\) is clearly upper semi-continuous, and depends on \(\FilteredSheafF\)
only though \(\LieFilteredSheafF\).

\begin{definition}\label{Defn::Filtrations::RestrictingFiltrations::NonCharPoints}
    We say \(x_0\in \BoundaryN\) is \DefnIt{\(\FilteredSheafF\)-non-characteristic} if \(x\mapsto \degBoundaryNF[x]\),
    \(\BoundaryN\rightarrow \Zg\) is continuous at \(x=x_0\). I.e., if \(\degBoundaryNF[x]\) is constant on a neighborhood
    of \(x_0\).
\end{definition}

\begin{example}\label{Example::Filtrations::RestrictingFiltrations::NonCharExamples}
    Fix \(x_0\in \BoundaryN\), and \(\Omega\subseteq\ManifoldN\) an \(\ManifoldN\)-open
    neighborhood of \(x_0\),
    with \(\FilteredSheafF\big|_{\Omega}=\FilteredSheafGenBy{\WWdv}\)
    where \(\WWdv=\left\{ (W_1,\Wdv_1),\ldots, (W_r,\Wdv_r) \right\}\) are H\"ormander vector fields with
    formal degrees on \(\Omega\).
    \begin{enumerate}[(i)]
        \item If \(\exists j\) with \(\Wdv_{j}=1\) and \(W_{j}(x_0)\not \in \TangentSpace{x_0}{\BoundaryN}\),
            then \(x_0\) is \(\FilteredSheafF\)-non-characteristic.
        \item\label{Item::Filtrations::RestrictingFiltrations::NonCharExamples::CharacterizeWhenAllDegsOne} In the case \(\WWdv=\WWo=\left\{ (W_1,1),\ldots, (W_r,1) \right\}\), then \(x_0\)
            is \(\FilteredSheafF\)-non-characteristic if and only if \(\exists j\) with \(W_j(x_0)\not\in T_{x_0}\BoundaryN\)
            if and only if \(\BoundaryN\) is non-characteristic near \(x_0\) for the sub-Laplacian \(\sum W_j^{*}W_j\).
            This is the notion used by Jerison \cite{JerisonDirichletProblemForTheKohnLaplacianI,JerisonDirichletProblemForTheKohnLaplacianII}
            and Kohn and Nirenberg \cite{KohnNirenbergNonCoerciveBoundaryValueProblems}.
        \item\label{Item::Filtrations::RestrictingFiltrations::NonCharExamples::CharacterizeNonCharInTermsOfWWdv}
        In general, \(x_0\) is \(\FilteredSheafF\)-non-characteristic if and only if \(\exists j\) with \(W_j(x_0)\not\in \TangentSpace{x_0}{\BoundaryN}\)
            and \(\Wdv_k<\Wdv_j\implies W_k(x)\in \TangentSpace{x}{\BoundaryN}\) for all \(x\in \BoundaryN\) near \(x_0\).
        \item If \(\FilteredSheafF[\Omega][d]=\VectorFields{\Omega}\), \(\forall \Omega, d\), then every point of
            \(\BoundaryN\) is \(\FilteredSheafF\)-non-characteristic. When we choose this filtration of sheaves of vector fields,
            then our results correspond with the classical results in the elliptic setting (see Proposition \ref{Prop::Spaces::EqualsClassical}). Thus, in the elliptic
            setting, every boundary point is automatically non-characteristic.
        \item Let \(\ManifoldN=\left\{ (x,y)\in \R^2 : y\geq x^2 \right\}\), \(\WWdv=\left\{ (\partial_x,1), (x\partial_y,1) \right\}\),
            and \(\FilteredSheafF=\FilteredSheafGenBy{\WWdv}\). Then, \((x,x^2)\) is \(\FilteredSheafF\)-non-characteristic if and only if
            \(x\ne 0\). Not only do our trace results not apply at \(x=0\), but the natural generalizations
            of our results to this setting are false. See Section \ref{Section::Trace::CharacteristicFailure} for a discussion.
        \item Suppose \(\ManifoldN=[0,\infty)\times \ManifoldM\) where \(\ManifoldM\) is a manifold without boundary, and we consider the heat equation
            \(\partial_t+\sum_{j=1}^s Y_j^{*}Y_j\), where \(Y_1,\ldots, Y_s\) are H\"ormander vector fields on \(\ManifoldM\).
            We choose \(\WWdv=\left\{ (\partial_t,2), (Y_1,1),\ldots, (Y_s,1) \right\}\) and \(\FilteredSheafF=\FilteredSheafGenBy{\WWdv}\).
            Then every point in the boundary (i.e., points of the form \((0,x)\in \{0\}\times \ManifoldM\)) is \(\FilteredSheafF\)-non-characteristic.
    \end{enumerate}
\end{example}

Set, for \(\lambda\in \Zg\),
\begin{equation*}
    \BoundaryNncF:=\left\{ x\in \BoundaryN : x\text{ is }\FilteredSheafF\text{-non-characteristic} \right\},
    \quad
    \BoundaryNncF[\lambda]=\left\{ x\in \BoundaryNncF : \degBoundaryNF[x]=\lambda \right\}.
\end{equation*}
Note that \(\BoundaryNncF\) is an open subset of \(\BoundaryN\), and \(\BoundaryNncF[\lambda]\), \(\lambda\in \Zg\),
are disjoint clopen subsets of \(\BoundaryNncF\) with \(\bigcup_{\lambda}\BoundaryNncF[\lambda]=\BoundaryNncF\).
Set \(\ManifoldNncF:=\InteriorN\cup \BoundaryNncF\) and \(\ManifoldNncF[\lambda]:=\InteriorN\cup \BoundaryNncF[\lambda]\),
where we give them the manifold structure as open submanifolds of \(\ManifoldN\).
Even if \(\ManifoldN\) is compact, \(\ManifoldNncF\) and \(\ManifoldNncF[\lambda]\) may not be compact.

If we replace \(\ManifoldN\) by \(\ManifoldNncF\) and \(\FilteredSheafF\) by \(\FilteredSheafF\big|_{\ManifoldNncF}\),
then we may assume that every point of \(\BoundaryN\) is \(\FilteredSheafF\)-non-characteristic.
Similarly, if we replace \(\ManifoldN\) by \(\ManifoldNncF[\lambda]\) and \(\FilteredSheafF\) by \(\FilteredSheafF\big|_{\ManifoldNncF[\lambda]}\)
then we may assume that \(\degBoundaryNF[x]=\lambda\), \(\forall x\in \BoundaryN\).
Since the main results of this paper are local results on the non-characteristic boundary, there
is no loss of generality in using either of these reductions (though we will not do so).
Note that if \(\ManifoldN\) does not have boundary, then \(\ManifoldNncF=\ManifoldN\).

\begin{proposition}[{\cite[Proposition \ref*{CC::Prop::Sheaves::Restrict::Bndry::RestritionIsFiniteTypeAndSpandTangent}]{StreetCarnotCaratheodoryBallsOnManifoldsWithBoundary}}]
    \label{Prop::Filtrations::RestrictingFiltrations::LieFIsHormanderFiltration}
    Let \(\FilteredSheafF\) be a H\"ormander filtration of sheaves of vector fields on \(\ManifoldN\).
    Then, \(\RestrictFilteredSheaf{\LieFilteredSheafF}{\BoundaryNncF}\) is a H\"ormander filtration of sheaves
    of vector fields\footnote{In fact, more is true. \(\RestrictFilteredSheaf{\LieFilteredSheafF}{\BoundaryNncF}\)
    is locally generated by finite sets, whose vector fields span the tangent space.} 
    on \(\BoundaryNncF\).
\end{proposition}

    \subsection{Metrics and Volumes}\label{Section::VectorFieldsAndSheaves::MetricsAndVolumes}
    A key tool when studying maximally subelliptic PDEs are scaling maps
adapted to H\"ormander vector fields with formal degrees.
On manifolds without boundary,
such scaling maps were introduced by Nagel, Stein, and Wainger \cite{NagelSteinWaingerBallsAndMetricsDefinedByVectorFieldsI},
and were generalized and refined by many authors (see, for example,
\cite{TaoWrightLpImprovingBoundsForAveragesAlongCurves,
FeffermanPhongSubellipticEigenvalueProblems,
FeffermanSanchezCalleFundamentalSolutionsForSecondOrderSubellipticOperators,
MontanariMorbidelliNonsmoothHormanderVectorFieldsAndTheirControlBalls,
StovaStreetCoordinatesAdaptedToVectorFieldsCanonicalCoordinates,
StreetCoordinatesAdaptedToVectorFieldsII,
StreetCoordinatesAdaptedToVectorFieldsIII}); see \cite[Section 3.3.1]{StreetMaximalSubellipticity}
for a description of the importance of these maps.

In \cite{StreetCarnotCaratheodoryBallsOnManifoldsWithBoundary} similar scaling maps
were introduced not only on \(\InteriorN\), but also on \(\BoundaryNncF\).
We present the scaling maps in Section \ref{Section::Spaces::Scaling}, but describe some consequences here.
% We do not require the full power\footnote{Many of the results in \cite{StreetCarnotCaratheodoryBallsOnManifoldsWithBoundary} which are not used in this paper
% will be used in a future paper on maximally subelliptic boundary value problems.} 
% of this theory, but do require several consequences of the results in
% \cite{StreetCarnotCaratheodoryBallsOnManifoldsWithBoundary}.
% We describe those consequences here.

Let \(\ManifoldN\) be a smooth manifold with boundary, and let
\(\FilteredSheafF\) be a H\"ormander filtration of sheaves of vector fields
on \(\ManifoldN\).
Fix \(\Compact\Subset \ManifoldNncF\), and let \(\Omega\Subset \ManifoldNncF\) be open
with \(\Compact\subseteq \Omega\). Let
\(\FilteredSheafF\big|_{\Omega}=\FilteredSheafGenBy{\WWdv}\)
where
\(\WWdv=\left\{ (W_1, \Wdv_1),\ldots, (W_r,\Wdv_r) \right\}\subset \VectorFields{\Omega}\times \Zg\)
are H\"ormander vector fields with formal degrees on \(\Omega\)--see Lemma \ref{Lemma::Filtrations::GeneratorsOnRelCptSet}.
Let \(\Vol\) be a smooth, strictly positive density on \(\Omega\).

\begin{proposition}\label{Prop::VectorFields::Scaling::SameTopology}
    The metric topology on \(\Compact\) induced by \(\MetricWWdv\) agrees with the usual topology on \(\Compact\)
    as a subspace of \(\ManifoldN\).
\end{proposition}
\begin{proof}
    This is 
    \cite[Theorem \ref*{CC::Thm::Sheaves::Metrics::ExistsEquivalenceClass} \ref*{CC::Item::Sheaves::Metrics::ExistsEquivalenceClass::SameTop}]{StreetCarnotCaratheodoryBallsOnManifoldsWithBoundary}.
    Alternatively, the result for \(\MetricWWdv\) replaced by \(\DistWWdv\) is
    \cite[Theorem \ref*{CC::Thm::Metrics::Results::GivesUsualTopology}]{StreetCarnotCaratheodoryBallsOnManifoldsWithBoundary}.
    Since \(\MetricWWdv=\DistWWdv\wedge 1\), the topology induced by \(\DistWWdv\) and \(\MetricWWdv\) are the same,
    so the result follow for \(\MetricWWdv\) as well.
\end{proof}

\begin{proposition}\label{Prop::VectorFields::Scaling::VolEstimates}
    Let \(\LieFilteredSheafF\big|_{\Omega}=\FilteredSheafGenBy{\XXde}\), where
    \(\XXde=\left\{ (X_1,\Xde_1),\ldots, (X_q,\Xde_q) \right\}\subset \VectorFields{\Omega}\times \Zg\),
    and set, for \(x\in \Omega\), \(\delta>0\),
    \begin{equation}\label{Eqn::VectorFields::Scaling::DefineLambda}
        \Lambda(x,\delta):=\max\left\{ \Vol(x)(\delta^{\Xde_{j_1}}X_{j_1}(x), \ldots, \delta^{\Xde_{j_n}}X_{j_n}(x)) : j_1,\ldots, j_n\in \{1,\ldots, q\} \right\}.
    \end{equation}
    There exists \(\delta_1\in (0,1]\) such that the following hold.
    \begin{enumerate}[(a)]
        \item\label{Item::VectorFields::Scaling::VolEstimates::VolApproxLambda} \(\forall x\in \Compact\), \(\delta\in (0,\delta_1]\),
            \begin{equation*}
                \Vol[\BWWdv{x}{\delta}]\approx \Lambda(x,\delta).
            \end{equation*}
        \item\label{Item::VectorFields::Scaling::VolEstimates::VolWedge1ApproxLambdaWedge1} \(\forall x\in \Compact\), \(\delta>0\),
            \begin{equation*}
                \Vol[\BWWdv{x}{\delta}]\wedge 1 \approx \Lambda(x,\delta)\wedge 1\approx \Vol[\BWWdv{x}{\delta\wedge \delta_1}].
            \end{equation*}
        \item\label{Item::VectorFields::Scaling::VolEstimates::VolDoubling} \(\forall x\in \Compact\), \(\delta\in (0,\delta_1]\),
            \begin{equation*}
                \Vol[\BWWdv{x}{2\delta}]\lesssim \Vol[\BWWdv{x}{\delta}].
            \end{equation*}
        \item\label{Item::VectorFields::Scaling::VolEstimates::VolWedge1Doubling} \(\forall x\in \Compact\), \(\delta>0\),
            \begin{equation*}
                \Vol[\BWWdv{x}{2\delta}]\wedge 1\lesssim \Vol[\BWWdv{x}{\delta}]\wedge 1.
            \end{equation*} 
            
        \item\label{Item::VectorFields::Scaling::VolEstimates::VolPoly} There exists \(Q_1,Q_2\geq 1\) such that \(\forall x\in \Compact\), \(\delta\in (0,\delta_1]\),
            \begin{equation*}
                \delta^{Q_2}\lesssim \Vol[\BWWdv{x}{\delta}]\lesssim \delta^{Q_1}.
            \end{equation*}
    \end{enumerate}
\end{proposition}
\begin{proof}
    This largely follows from \cite[Theorem \ref*{CC::Thm::Scaling::MainResult}]{StreetCarnotCaratheodoryBallsOnManifoldsWithBoundary}.
    In that reference, an explicit choice of \(\XXde\) was made.
    However, using the compactness of \(\Compact\), it is easy to see that the particular
    choice of generators \(\XXde\) for \(\LieFilteredSheafF\big|_{\Omega}\) do not change the result.
    \ref{Item::VectorFields::Scaling::VolEstimates::VolPoly} was not stated in \cite{StreetCarnotCaratheodoryBallsOnManifoldsWithBoundary},
    but follows immediately from \ref{Item::VectorFields::Scaling::VolEstimates::VolApproxLambda}.
\end{proof}

\begin{proposition}\label{Prop::VectorFields::Scaling::VolAndMetricEquivalnce}
    Let \(\FilteredSheafG\) be another H\"ormander filtration of sheaves of vector fields on
    \(\ManifoldN\) such that \(\LieFilteredSheafG=\LieFilteredSheafF\),
    \(\Omega_1\Subset \ManifoldNncF\) open with 
    \(\Compact\subseteq \Omega_1\),
    and
    \(\FilteredSheafG\big|_{\Omega_1}=\FilteredSheafGenBy{\ZZde}\), where
    \(\ZZde\subset \VectorFields{\Omega_1}\times \Zg\) are H\"ormander vector fields with formal degrees on \(\Omega_1\).
    Let \(\Volh\) be a smooth, strictly positive density on \(\Omega_1\)
    Then, there exists \(\delta_1\in (0,1]\) such that the following hold.
    \begin{enumerate}[(a)]
        \item\label{Item::VectorFields::Scaling::VolAndMetricEquivalnce::VolWequalsVolZ} \(\forall x\in \Compact\), \(\delta\in (0,\delta_1]\),
            \begin{equation*}
                \Vol[\BWWdv{x}{\delta}]\approx \Volh[\BZZde{x}{\delta}].
            \end{equation*}
        \item\label{Item::VectorFields::Scaling::VolAndMetricEquivalnce::VolWWedge1equalsVolZWedg1} \(\forall x\in \Compact\), \(\delta>0\),
            \begin{equation*}
                \Vol[\BWWdv{x}{\delta}]\wedge 1 \approx \Volh[\BZZde{x}{\delta}]\wedge 1.
            \end{equation*}
        \item\label{Item::VectorFields::Scaling::VolAndMetricEquivalnce::MetricWequalsMetricZ} \(\forall x,y\in \Compact\), \(\MetricWWdv[x][y]\approx \MetricZZde[x][y]\).
        \item\label{Item::VectorFields::Scaling::VolAndMetricEquivalnce::VolOfMetricWequalsVolOfMetricZ} \(\forall x,y\in \Compact\), \(\delta>0\),
            \begin{equation*}
                \Vol[\BWWdv{x}{\delta + \MetricWWdv[x][y]}]\wedge 1
                \approx \Volh[\BZZde{x}{\delta+\MetricZZde[x][y]}]\wedge 1.
            \end{equation*}
    \end{enumerate}
\end{proposition}
\begin{proof}
    Let \(\Omega_2:=\Omega\cup \Omega_1\), and let
    \(\LieFilteredSheafF\big|_{\Omega_2}=\FilteredSheafGenBy{\XXdv}\).
    Thus, \(\LieFilteredSheafF\big|_{\Omega}=\FilteredSheafGenBy{\XXdv}\big|_{\Omega}\)
    and \(\LieFilteredSheafG\big|_{\Omega_1}=\FilteredSheafGenBy{\XXdv}\big|_{\Omega_1}\).
    It follows that we may use this choice of \(\XXdv\) when defining
    \(\Lambda\) in Proposition \ref{Prop::VectorFields::Scaling::VolEstimates}
    both when applied to \(\WWdv\) and \(\Vol\) and when applied to \(\ZZde\) and \(\Volh\).
    From here, \ref{Item::VectorFields::Scaling::VolAndMetricEquivalnce::VolWequalsVolZ}
    and \ref{Item::VectorFields::Scaling::VolAndMetricEquivalnce::VolWWedge1equalsVolZWedg1}
    follow from Proposition \ref{Prop::VectorFields::Scaling::VolEstimates} \ref{Item::VectorFields::Scaling::VolEstimates::VolApproxLambda}
    and \ref{Item::VectorFields::Scaling::VolEstimates::VolWedge1ApproxLambdaWedge1}, respectively.

    \ref{Item::VectorFields::Scaling::VolAndMetricEquivalnce::MetricWequalsMetricZ} follows from
    \cite[Theorem \ref*{CC::Thm::Sheaves::Metrics::ExistsEquivalenceClass} \ref*{CC::Item::Sheaves::Metrics::ExistsEquivalenceClass::AgreesWithCCInducedByG}]{StreetCarnotCaratheodoryBallsOnManifoldsWithBoundary}.
    There is one slight discrepancy between the notation of this paper and the notation in \cite{StreetCarnotCaratheodoryBallsOnManifoldsWithBoundary}.
    The notation \(\rho_{\WWdv}\) in \cite{StreetCarnotCaratheodoryBallsOnManifoldsWithBoundary}
    is the same as
    \(\DistWWdv[x][y]\) as defined in \eqref{Eqn::VectorFields::DefineDistance},
    while \(\rho_{\WWdv}\wedge 1\) in \cite{StreetCarnotCaratheodoryBallsOnManifoldsWithBoundary} is
    \(\MetricWWdv\) as defined in \eqref{Eqn::VectorFields::DefineMetric}.
    With this translation, \ref{Item::VectorFields::Scaling::VolAndMetricEquivalnce::MetricWequalsMetricZ} follows directly from
    \cite[Theorem \ref*{CC::Thm::Sheaves::Metrics::ExistsEquivalenceClass} \ref*{CC::Item::Sheaves::Metrics::ExistsEquivalenceClass::AgreesWithCCInducedByG}]{StreetCarnotCaratheodoryBallsOnManifoldsWithBoundary}.

    \ref{Item::VectorFields::Scaling::VolAndMetricEquivalnce::VolOfMetricWequalsVolOfMetricZ}:
    Using \ref{Item::VectorFields::Scaling::VolAndMetricEquivalnce::MetricWequalsMetricZ}
    and Proposition \ref{Prop::VectorFields::Scaling::VolEstimates} \ref{Item::VectorFields::Scaling::VolEstimates::VolWedge1Doubling},
    we see
    \begin{equation*}
        \Vol[\BWWdv{x}{\delta + \MetricWWdv[x][y]}]\wedge 1
                \approx \Vol[\BWWdv{x}{\delta + \MetricZZde[x][y]}]\wedge 1, \quad \forall x,y\in \Compact, \quad \delta>0.
    \end{equation*}
    Using \ref{Item::VectorFields::Scaling::VolAndMetricEquivalnce::VolWWedge1equalsVolZWedg1}, we have
    \begin{equation*}
        \Vol[\BWWdv{x}{\delta + \MetricZZde[x][y]}]
        \approx
        \Volh[\BZZde{x}{\delta + \MetricZZde[x][y]}]\wedge 1, \quad \forall x,y\in \Compact, \quad \delta>0.
    \end{equation*}
    Combining these two equations yields \ref{Item::VectorFields::Scaling::VolAndMetricEquivalnce::VolOfMetricWequalsVolOfMetricZ}.
\end{proof}

\begin{proposition}\label{Prop::VectorFields::Scaling::AmbientVolAndMetricEquivalnce}
    Let \(\ManifoldM\) be a smooth manifold without boundary such that \(\ManifoldN\subseteq \ManifoldM\)
    is a closed, embedded, co-dimension \(0\) submanifold with boundary.
    Let \(\FilteredSheafFh\) be a H\"ormander filtration of sheaves of vector fields on \(\ManifoldM\)
    such that \(\RestrictFilteredSheaf{\FilteredSheafFh}{\ManifoldN}=\FilteredSheafF\).
    Fix \(\Omegah\Subset \ManifoldM\) open and relatively compact in \(\ManifoldM\), with \(\Compact\subseteq \Omegah\),
    and \(\FilteredSheafFh\big|_{\Omegah}=\FilteredSheafGenBy{\ZZde}\) where \(\ZZde\subset \VectorFields{\Omegah}\times \Zg\)
    are H\"ormander vector fields with formal degrees on \(\Omegah\).
    Let \(\Volh\) be a smooth, strictly positive density on \(\Omegah\).
    Then, there exists \(\delta_1\in (0,1]\) such that the following hold.
    \begin{enumerate}[(a)]
        \item\label{Item::VectorFields::Scaling::AmbientVolAndMetricEquivalnce::VolWequalsVolZ} \(\forall x\in \Compact\), \(\delta\in (0,\delta_1]\),
            \begin{equation*}
                \Vol[\BWWdv{x}{\delta}]\approx \Volh[\BZZde{x}{\delta}].
            \end{equation*}
        \item\label{Item::VectorFields::Scaling::AmbientVolAndMetricEquivalnce::VolWWedge1equalsVolZWedg1} \(\forall x\in \Compact\), \(\delta>0\),
            \begin{equation*}
                \Vol[\BWWdv{x}{\delta}]\wedge 1 \approx \Volh[\BZZde{x}{\delta}]\wedge 1.
            \end{equation*}
        \item\label{Item::VectorFields::Scaling::AmbientVolAndMetricEquivalnce::MetricWequalsMetricZ} \(\forall x,y\in \Compact\), \(\MetricWWdv[x][y]\approx \MetricZZde[x][y]\).
        \item\label{Item::VectorFields::Scaling::AmbientVolAndMetricEquivalnce::VolOfMetricWequalsVolOfMetricZ} \(\forall x,y\in \Compact\), \(\delta>0\),
            \begin{equation*}
                \Vol[\BWWdv{x}{\delta + \MetricWWdv[x][y]}]\wedge 1
                \approx \Volh[\BZZde{x}{\delta+\MetricZZde[x][y]}]\wedge 1.
            \end{equation*}
    \end{enumerate}
\end{proposition}
\begin{proof}
    Let \(\Omega_1\Subset\ManifoldM\) be an open, relatively compact set in \(\ManifoldM\)
    such that \(\Omega\cup \Omegah\subseteq \Omega_1\). Take \(\XhXde=\left\{ (\Xh_1,\Xde_1),\ldots, (\Xh_q,\Xde_q) \right\}\subset \VectorFields{\Omega_1}\times \Zg\)
    such that \(\LieFilteredSheafFh\big|_{\Omega_1}=\FilteredSheafGenBy{\XhXde}\).
    Set \(X_j:=\Xh_j\big|_{\Omega}\) and \(\XXde=\left\{ (X_1,\Xde_1),\ldots, (X_q,\Xde_q) \right\}\).
    It follows from Proposition \ref{Prop::Filtrations::RestrictingFiltrations::CoDim0Restriction}
    that \(\LieFilteredSheafF\big|_{\Omega}=\FilteredSheafGenBy{\XXde}\).
    For \(x\in \Compact\), the formula for \(\Lambda\) in \eqref{Prop::Filtrations::RestrictingFiltrations::CoDim0Restriction}
    is the same whether we use \(\XhXde\) or \(\XXde\).
    From here, \ref{Item::VectorFields::Scaling::AmbientVolAndMetricEquivalnce::VolWequalsVolZ}
    and \ref{Item::VectorFields::Scaling::AmbientVolAndMetricEquivalnce::VolWWedge1equalsVolZWedg1}
    follow from Proposition \ref{Prop::VectorFields::Scaling::VolEstimates} \ref{Item::VectorFields::Scaling::VolEstimates::VolApproxLambda}
    and \ref{Item::VectorFields::Scaling::VolEstimates::VolWedge1ApproxLambdaWedge1}, respectively.

    \ref{Item::VectorFields::Scaling::AmbientVolAndMetricEquivalnce::MetricWequalsMetricZ}
    follows from \cite[Theorem \ref*{CC::Thm::Sheaves::Metrics::SameMetricFromAmbientSpace}]{StreetCarnotCaratheodoryBallsOnManifoldsWithBoundary}.

    \ref{Item::VectorFields::Scaling::AmbientVolAndMetricEquivalnce::VolOfMetricWequalsVolOfMetricZ}
    follows from \ref{Item::VectorFields::Scaling::AmbientVolAndMetricEquivalnce::VolWWedge1equalsVolZWedg1}
    and \ref{Item::VectorFields::Scaling::AmbientVolAndMetricEquivalnce::MetricWequalsMetricZ}
    using Proposition \ref{Prop::VectorFields::Scaling::VolEstimates} \ref{Item::VectorFields::Scaling::VolEstimates::VolWedge1Doubling}
    just as in the proof of Proposition \ref{Prop::VectorFields::Scaling::VolAndMetricEquivalnce} \ref{Item::VectorFields::Scaling::VolAndMetricEquivalnce::VolOfMetricWequalsVolOfMetricZ}.
\end{proof}

    \subsection{Differential operators}
    Let \(\ManifoldN\) be a smooth manifold with boundary, and let \(\FilteredSheafF\)
be a H\"ormander filtration of sheaves of vector fields on \(\ManifoldN\).

\begin{definition}\label{Defn::Filtrations::DiffOps::Deg}
    Let \(\Omega\Subset \ManifoldN\) be open and relatively compact, and let \(\opL\)
    be a partial differential operator with smooth coefficients defined on \(\Omega\).
    We say \(\opL\) has \(\FilteredSheafF\)-degree \(\leq \kappa\in \Zgeq\) on \(\Omega\)
    if the following holds:
    \begin{itemize}
        \item Let \(\WWdv=\left\{ (W_1,\Wdv_1),\ldots, (W_r, \Wdv_r) \right\}\subset \VectorFields{\Omega}\times \Zg\)
            be H\"ormander vector fields with formal degrees on \(\Omega\) such that
            \(\FilteredSheafF\big|_{\Omega}=\FilteredSheafGenBy{\WWdv}\)--see Lemma \ref{Lemma::Filtrations::GeneratorsOnRelCptSet}.
        \item We assume we may write
            \begin{equation*}
                \opL=\sum_{\DegWdv{\alpha}\leq \kappa} a_\alpha(x) W^{\alpha},\quad a_\alpha\in \CinftySpace[\Omega],
            \end{equation*}
            Here if \(\alpha=(\alpha_1,\ldots, \alpha_L)\in \left\{ 1,\ldots, r \right\}^L\),
            then \(\DegWdv{\alpha}=\Wdv_{\alpha_1}+\Wdv_{\alpha_2}+\cdots+\Wdv_{\alpha_L}\).
    \end{itemize}
    It is not hard to see that this does not depend on the choice of \(\WWdv\).
    If \(\Compact\Subset \ManifoldN\) is compact, we say \(\opL\)
    has \(\FilteredSheafF\)-degree \(\leq \kappa\) on \(\Compact\) if there exists \(\Omega\Subset \ManifoldN\)
    open and relatively compact, with \(\Compact\Subset \Omega\) and \(\opL\)
    has \(\FilteredSheafF\)-degree \(\leq \kappa\) on \(\Omega\).
\end{definition}

\begin{remark}\label{Rmk::Filtrations::DiffOps::DoesntDependOnChoices}
    Definition \ref{Defn::Filtrations::DiffOps::Deg} does not depend on any of the choices made
    in the definition. Namely, any other choice of \(\Omega\), \(\Omega_1\), and \(\WWdv\) lead to the same
    definition. Moreover, \(\opL\) has \(\FilteredSheafF\)-degree \(\leq \kappa\in \Zgeq\) on \(\Compact\)
    if and only if \(\opL\) has \(\LieFilteredSheafF\)-degree \(\leq \kappa\in \Zgeq\) on \(\Compact\).
\end{remark}

\section{Besov and Triebel--Lizorkin spaces}\label{Chapter::Spaces}
Throughout this chapter, let \(\ManifoldN\) be a smooth manifold with boundary,
and let \(\FilteredSheafF\) be a H\"ormander filtration of sheaves of vector fields on \(\ManifoldN\)
as in Definitions \ref{Defn::Filtrations_Sheaves::Filtration_Sheaves} and \ref{Defn::Filtrations_Sheaves::Hormander_Filtration_Sheaves}.
In this chapter, for \(\Compact\Subset \ManifoldNncF\) compact,\footnote{See Section \ref{Section::VectorFieldsAndSheaves::RestrictingSheavesAndNonCharPoints}
for the definition of \(\ManifoldNncF\).} we define:
\begin{itemize}
    \item The Besov spaces, \(\BesovSpace{s}{p}{q}[\Compact][\FilteredSheafF]\), \(s\in \R\), \(1\leq p,q\leq \infty\).
    \item The Triebel--Lizorkin spaces, \(\TLSpace{s}{p}{q}[\Compact][\FilteredSheafF]\), \(s\in \R\), \(1<p<\infty\), \(1<q\leq \infty\).
\end{itemize}
See Section \ref{Section::Spaces::MainDefns} for the definitions; before stating the (somewhat technical) definitions we state our main results
concerning these spaces.

\begin{notation}\label{Notation::Spaces::XSpace}
    Throughout the paper,
    \(\XSpace{s}{p}{q}\) will denote either \(\BesovSpace{s}{p}{q}\) or \(\TLSpace{s}{p}{q}\).
    Here, \(s\in \R\), and when \(\XSpace{s}{p}{q}=\BesovSpace{s}{p}{q}\) we restrict to
    \(1\leq p,q\leq \infty\), and when \(\XSpace{s}{p}{q}=\TLSpace{s}{p}{q}\)
    we restrict to \(1<p<\infty\), \(1<q\leq \infty\) (as in Remark \ref{Rmk::Intro::Classical::Restrcitpq}).
\end{notation}

% Let \(\XSpace{s}{p}{q}[\Compact][\FilteredSheafF]\) denote either \(\BesovSpace{s}{p}{q}[\Compact][\FilteredSheafF]\) or 
%  \(\TLSpace{s}{p}{q}[\Compact][\FilteredSheafF]\) with the above restrictions on \(p,q\). 
 \(\XSpace{s}{p}{q}[\Compact][\FilteredSheafF]\) is a subspace of those distributions in \(\DistributionsZeroN\)
 which are supported
 in \(\Compact\).

 \begin{remark}\label{Rmk::Spaces::DefnOfSupport}
    Here, support is defined in the usual way: for \(u\in \DistributionsZeroN\),
    \(x\not \in \supp(u)\) if and only if \(\exists \Omega\subseteq \ManifoldN\) open with \(x\in \Omega\) and
    such that \(\forall f\in \TestFunctionsZeroN\) with \(\supp(f)\subseteq \Omega\), we have \(u(f)=0\).
 \end{remark}

 The main results
 of this chapter are as follows.

 \begin{theorem}\label{Thm::Spaces::BanachSpace}
    \(\XSpace{s}{p}{q}[\Compact][\FilteredSheafF]\) is a Banach space.
 \end{theorem}
\begin{proof}
    This follows from Proposition \ref{Prop::Spaces::Defns::NormAndBanach}, below.
\end{proof}

 \begin{proposition}\label{Prop::Spaces::Containment}
    Let \(\Compact_1\subseteq \Compact_2\Subset \ManifoldNncF\) be two compact sets.
    Then, \(\XSpace{s}{p}{q}[\Compact_1][\FilteredSheafF]\) is a closed subspace
    of \(\XSpace{s}{p}{q}[\Compact_2][\FilteredSheafF]\)
    and
    \begin{equation}\label{Eqn::Spaces::CompactSubsetImpliesClosedSubspace}
        \XSpace{s}{p}{q}[\Compact_1][\FilteredSheafF]
        =\left\{ f\in \XSpace{s}{p}{q}[\Compact_2][\FilteredSheafF] : \supp(f)\subseteq \Compact_1 \right\}.
    \end{equation}
 \end{proposition}
 \begin{proof}
    The equality 
    \eqref{Eqn::Spaces::CompactSubsetImpliesClosedSubspace} is an immediate consequence of the definition
    (see Definition \ref{Defn::Spaces::Defns::ASpace}, below). That \(\XSpace{s}{p}{q}[\Compact_1][\FilteredSheafF]\) is a closed subspace
    of \(\XSpace{s}{p}{q}[\Compact_2][\FilteredSheafF]\)
    can be see as in Remark \ref{Rmk::Spaces::Defns::EquivClassOfNormWellDefinedAndDoesntDependOnCompact}, below.
 \end{proof}

 \begin{theorem}\label{Thm::Spaces::OnlyDependsOnLieFiltration}
    If \(\FilteredSheafG\) is another H\"ormander filtration of sheaves of vector fields on \(\ManifoldN\)
    such that \(\LieFilteredSheafF=\LieFilteredSheafG\), then
    \(\XSpace{s}{p}{q}[\Compact][\FilteredSheafF]=\XSpace{s}{p}{q}[\Compact][\FilteredSheafG]\),
    with equivalent norms.
 \end{theorem}
 \begin{proof}
    This follows by combining Lemma \ref{Lemma::Spaces::LP::ElemDoesntDependOnChoices}, below,
    with the definitions (see Definition \ref{Defn::Spaces::Defns::VspqENorm} and Notation \ref{Notation::Spaces::Defns::Norm}, below).
 \end{proof}

 \begin{proposition}\label{Prop::Spaces::MappingOfFuncsAndVFs}
    Let \(\Omega\subseteq \ManifoldN\) be open with \(\Compact\Subset\Omega\).
    \begin{enumerate}[(i)]
        \item\label{Item::Spaces::MappingOfFuncsAndVFs::VFsInF} For \(V\in \FilteredSheafF[\Omega][d]\), we have
        \(V:\XSpace{s}{p}{q}[\Compact][\FilteredSheafF]\rightarrow \XSpace{s-d}{p}{q}[\Compact][\FilteredSheafF]\),
        continuously.

        \item\label{Item::Spaces::MappingOfFuncsAndVFs::VFsInLieF} For \(V\in \LieFilteredSheafF[\Omega][d]\), we have
        \(V:\XSpace{s}{p}{q}[\Compact][\FilteredSheafF]\rightarrow \XSpace{s-d}{p}{q}[\Compact][\FilteredSheafF]\),
        continuously.

        \item\label{Item::Spaces::MappingOfFuncsAndVFs::Funcs} Let \(\psi\in \CinftySpace[\Omega]\) and \(\Mult{\psi}:f\mapsto \psi f\).
    Then,
    \(\Mult{\psi}:\XSpace{s}{p}{q}[\Compact][\FilteredSheafF]\rightarrow \XSpace{s}{p}{q}[\Compact][\FilteredSheafF]\),
    continuously.

    \item\label{Item::Spaces::MappingOfFuncsAndVFs::DiffOps} Let \(\opL\) be a partial differential
        operator with smooth coefficients, defined on a neighborhood of \(\Compact\), and suppose
        \(\opL\) has \(\FilteredSheafF\)-degree \(\leq \kappa\) on \(\Compact\).
        Then,
        \(\opL:\XSpace{s}{p}{q}[\Compact][\FilteredSheafF]\rightarrow \XSpace{s-\kappa}{p}{q}[\Compact][\FilteredSheafF]\),
        continuously.
    \end{enumerate}
 \end{proposition}

 Proposition \ref{Prop::Spaces::MappingOfFuncsAndVFs} is established in Section \ref{Section::Spaces::BasicProps}.

 \begin{theorem}\label{Thm::Spaces::Extension}
    Let \(\ManifoldM\) be a smooth manifold such that \(\ManifoldN\subseteq \ManifoldM\)
    is a closed, embedded, codimension \(0\) submanifold with boundary.
    Let \(\FilteredSheafFh\) be a H\"ormander filtration of sheaves of vector fields on \(\ManifoldM\)
    such that \(\RestrictFilteredSheaf{\FilteredSheafFh}{\ManifoldN}=\FilteredSheafF\).
    \begin{enumerate}[(i)]
        \item\label{Item::Spaces::Extension::Restriction} Let \(\Compact\Subset \ManifoldM\) be compact with \(\Compact\cap \ManifoldN\subseteq \ManifoldNncF\).
            The map \(f\mapsto f\big|_{\TestFunctionsZeroN}\) is continuous
            \(\XSpace{s}{p}{q}[\Compact][\FilteredSheafFh]\rightarrow \XSpace{s}{p}{q}[\Compact\cap \ManifoldN][\FilteredSheafF]\).
            \item\label{Item::Spaces::Extension::Extension}
                The map in \ref{Item::Spaces::Extension::Restriction} is right invertible in the following sense.
                Let \(\Compact\Subset \ManifoldNncF\) be compact and let \(\Omega\Subset\ManifoldM\) be \(\ManifoldM\)-open and relatively compact in \(\ManifoldM\)
                with \(\Compact\subseteq \Omega\). Fix \(N_0\in [0,\infty)\). Then,
                there exists a linear extension map
                \begin{equation*}
                    \Extension[N_0]: \left( \bigcup_{\substack{1<p\leq \infty\\ 1\leq q\leq \infty \\ |s|\leq N_0}} \BesovSpace{s}{p}{q}[\Compact][\FilteredSheafF] \right)
                    \bigcup \left( \bigcup_{\substack{1<p<\infty \\ 1<q\leq \infty \\ |s|\leq N_0}} \TLSpace{s}{p}{q}[\Compact][\FilteredSheafF] \right)
                    \rightarrow
                    \Distributions[\ManifoldM]
                \end{equation*}
                such that
                \begin{enumerate}[(a)]
                    \item \(\Extension[N_0]:\XSpace{s}{p}{q}[\Compact][\FilteredSheafF]\rightarrow \XSpace{s}{p}{q}[\overline{\Omega}][\FilteredSheafFh]\), \(|s|\leq N_0\), is continuous.
                    \item \(\Extension[N_0]f\big|_{\TestFunctionsZeroN}=f\), for all \(f\) in the domain of \(\Extension[N_0]\).
                \end{enumerate}
    \end{enumerate}
 \end{theorem}

 See Section \ref{Section::Spaces::RestrictionAndExtension} for the proof of Theorem \ref{Thm::Spaces::Extension}.

 \begin{proposition}\label{Prop::Spaces::EqualsLp}
   For \(1<p<\infty\),
   \begin{equation*}
      \TLSpace{0}{p}{2}[\Compact][\FilteredSheafF]
      =\LpSpace{p}[\Compact],
  \end{equation*}
  and
  \begin{equation*}
      \TLNorm{f}{0}{p}{2}[\FilteredSheafF]\approx \LpNorm{f}{p},\quad \forall f\in \TLSpace{0}{p}{2}[\Compact][\FilteredSheafF]=\LpSpace{p}[\Compact].
  \end{equation*}
  Here, the norm on \(\LpSpace{p}[\Compact]\) is defined by any smooth, strictly positive density
  on an \(\ManifoldN\)-neighborhood of \(\Compact\). The choice of density does not matter--as the equivalence
  class of the \(\LpSpace{p}\)-norm does not depend on the choice.
 \end{proposition}
 Proposition \ref{Prop::Spaces::EqualsLp} is proved in Section \ref{Section::Spaces::RestrictionAndExtension::SimpleConsequences}.

 \begin{proposition}\label{Prop::Spaces::EqualsSobolev}
   Let \(\Omega\subseteq \ManifoldN\) be open and let
   \(\WWo=\left\{ (W_1,1),\ldots, (W_r,1) \right\}\subset \VectorFields{\Omega}\times \Zg\)
   are H\"ormander vector fields with formal degrees on \(\Omega\), where each formal degree equals \(1\).
   Suppose \(\FilteredSheafF\big|_{\Omega}=\FilteredSheafGenBy{\WWo}\). Let \(\Compact\Subset \Omega\cap \ManifoldNncF\)
   be compact. Then, for \(1<p<\infty\), and \(k\in \Zgeq\),
   \begin{equation*}
       \TLSpace{k}{p}{2}[\Compact][\FilteredSheafF] = \left\{ f\in \LpSpace{p}[\Compact] : W^{\alpha}f\in \LpSpace{p}[\Compact], \forall |\alpha|\leq k \right\},
   \end{equation*}
   \begin{equation*}
       \TLNorm{f}{k}{p}{2}[\Compact][\FilteredSheafF]\approx \sum_{|\alpha|\leq k} \LpNorm{W^{\alpha}f}{p},
   \end{equation*}
   where the \(\LpSpace{p}\)-norm is taken with respect to any smooth, strictly positive density on \(\Omega\)
   (the choice of density does not matter).
\end{proposition}
\begin{proof}
   For \(k=0\), this is Proposition \ref{Prop::Spaces::EqualsLp}.
   The result for general \(k\) follows from the result for \(k=0\) and Proposition \ref{Prop::Spaces::BasicProps::OneDerivsGivesNorms},
   below.
\end{proof}

\begin{remark}
   Proposition \ref{Prop::Spaces::EqualsSobolev} focuses on the special case when the formal
   degrees are all equal to \(1\). The general case, for general formal degrees, on a manifold without boundary 
   is covered in \cite[Corollary 6.2.14]{StreetMaximalSubellipticity}. We expect a similar result
   should be true on manifolds with boundary. However, if the proof in \cite[Corollary 6.2.14]{StreetMaximalSubellipticity}
   is an indication, this will require a more detailed study of some maximally subelliptic boundary value problems,
   which will be the focus of future work.
   However, see Proposition \ref{Prop::Spaces::Extend::Consequences::DerivForSobolevNormsOnFunctionsOfCptSupport}
   and Remark \ref{Rmk::Spaces::Extend::Consequences::DerivForSobolevNormsOnFunctionsOfCptSupport::NotAllFunctions} for the more
   general result, when restricted to smooth functions with compact support in \(\InteriorN\).
\end{remark}

\begin{remark}
    For the connection between \(\BesovSpace{s}{\infty}{\infty}[\Compact][\FilteredSheafF]\)
    and the naturally defined H\"older space with respect to the Carnot--Carath\'eodory metric;
    see Section \ref{Section::Zyg::Holder}.
\end{remark}

 Let \(\XSpaceClassical{s}{p}{q}[\Compact]\) denote the Banach space of those distributions
 in \(\DistributionsZeroN\) which lie in the classical Besov, \(\BesovSpace{s}{p}{q}\), or Triebel--Lizorkin,
 \(\TLSpace{s}{p}{q}\), space on \(\ManifoldN\), and are supported in \(\Compact\).
 Since \(\ManifoldN\) is a smooth manifold with boundary and \(\Compact\) is compact, this is a well-defined
 space; see, for example, \cite[Section 2.4]{RunstSickelSobolevSpacesOfFractionalOrder}.

 \begin{proposition}\label{Prop::Spaces::EqualsClassical}
    In the special case \(\FilteredSheafF[\Omega][d]=\VectorFields{\Omega}\), for all \(\Omega\) and \(d\),
    we have \(\XSpace{s}{p}{q}[\Compact][\FilteredSheafF]=\XSpaceClassical{s}{p}{q}[\Compact]\) with
    equivalence of norms.
 \end{proposition}

 Proposition \ref{Prop::Spaces::EqualsClassical} is proved in Section \ref{Section::Spaces::RestrictionAndExtension::SimpleConsequences}.

 \begin{remark}
   An important aspect of classical Besov and Triebel--Lizorkin spaces is that the elements can be well-approximated
   by smooth functions. The same is true in this setting: see Section \ref{Section::Spaces::SmoothApproximation}.
 \end{remark}

 The rest of this chapter is devoted to the relevant definitions and proofs of the above results.

    \subsection{The classical setting}\label{Section::Spaces::Classical}
    There are many possible equivalent characterizations
of the classical spaces \(\BesovSpace{s}{p}{q}[\Rngeq]\)
and \(\TLSpace{s}{p}{q}[\Rngeq]\).
In this section, we present one such characterization (Proposition \ref{Prop::Spaces::Classical::NewDefnAspq})
which is perhaps less common, and which we generalize to obtain
the definitions and results of this paper.

\begin{notation}\label{Notation::Spaces::Classical::VSpacepq}
    We define 
    \begin{equation*}
        \VSpace{p}{q}=
        \begin{cases}
            \lqLpSpace{p}{q}, &\text{if } \ASpace{s}{p}{q}=\BesovSpace{s}{p}{q},\\
            \LplqSpace{p}{q}, &\text{if } \ASpace{s}{p}{q}=\TLSpace{s}{p}{q},
        \end{cases}
    \end{equation*}
    where \(\LpSpace{p}\) is on the ambient space
    (with respect to whatever measure has been given).
    For example, when we are considering \(\Rngeq\) we use
    Lebesgue measure and have:
    \begin{equation*}
        \VSpace{p}{q}=
        \begin{cases}
            \lqLpSpace{p}{q}[\Rngeq], &\text{if } \ASpace{s}{p}{q}[\Rngeq]=\BesovSpace{s}{p}{q}[\Rngeq],\\
            \LplqSpace{p}{q}[\Rngeq], &\text{if } \ASpace{s}{p}{q}[\Rngeq]=\TLSpace{s}{p}{q}[\Rngeq],
        \end{cases}
    \end{equation*}
    and similarly for \(\Rngeq\) replaced with \(\Rn\).
    Whenever \( \VSpace{p}{q}=\lqLpSpace{p}{q}\) we restrict to \(1\leq p,q\leq \infty\),
    and whenever \(\VSpace{p}{q}= \LplqSpace{p}{q}\) we restrict to \(1<p<\infty\) and \(1<q\leq \infty\).
\end{notation}

We begin with one of the most standard definitions of \(\ASpace{s}{p}{q}[\Rn]\).
Let \(\SchwartzSpaceRn\) denote Schwartz space on \(\Rn\).
Fix \(\psi\in \SchwartzSpaceRn\) with \(\int \psi =1\) and \(\int t^{\alpha}\psi(t)\: dt=0\), \(\forall \alpha\ne 0\);
for example one may choose
\(\psi\) by letting
\(\psih\) be equal to one on a neighborhood of \(0\).
For any function \(\phi\), set \(\Dil{2^j}{\phi}(t)=2^{jn}\phi(2^j t)\).
For \(j\in \Zgeq\), set
\begin{equation}\label{Eqn::Spaces::Classical::Definephij}
    \phi_j=\begin{cases}
        \psi, & j=0,\\
        \Dil{2}{\psi}-\psi, & j\geq 1.
    \end{cases}
\end{equation}
So that 
\begin{equation}\label{Eqn::Spaces::Classical::SumDilatesToGetDelta0}
    \sum_{j=0}^L \Dil{2^j}{\phi_j} = \Dil{2^L}{\psi}\xrightarrow{L\rightarrow \infty}\delta_0.  
\end{equation}
Define \(D_j:\TemperedDistributions[\Rn]\rightarrow \CinftySpace[\Rn]\) by
\begin{equation}\label{Eqn::Spaces::Classical::DefineDjOnWholeSpace}
    D_j f :=f*\Dil{2^j}{\phi_j}.
\end{equation}
It follows from \eqref{Eqn::Spaces::Classical::SumDilatesToGetDelta0}
that \(\sum_{j=0}^\infty D_j =I\).

\begin{remark}\label{Rmk::Spaces::Classical::phiIsSchwartzAllOfWhoseMomentsVanish}
    For \(j\geq 1\),
    \(\phi_j\in \SchwartzSpaceRn\) and satisfies
    \(
        \int t^{\alpha}\phi_j(t)\: dt=0,
    \)
    \(\forall \alpha\):
    it is a Schwartz function all of whose moments vanish; see \eqref{Eqn::Spaces::Schwartz::DefnSchwartzSpacez}.
\end{remark}

Set
\begin{equation*}
    \ASpace{s}{p}{q}[\Rn]:=
    \left\{ f\in \TemperedDistributions[\Rn] : \left\{ D_j f \right\}_{j\in \Zgeq}\in \VSpace{p}{q}\right\},
    \quad
    \ANorm{f}{s}{p}{q}[\R^n]:=
    \BVNorm{\left\{ D_j f \right\}_{j\in \Zgeq}}{p}{q}.
\end{equation*}
It is a classical result that \(\ASpace{s}{p}{q}[\Rn]\) is a Banach space,
and the space and equivalence class of the norm
do not depend on the choice of \(\psi\); see, for example, \cite[Section 2.1]{RunstSickelSobolevSpacesOfFractionalOrder}.

A common way to define \(\ASpace{s}{p}{q}[\Rngeq]\) is as restrictions
of elements of \(\ASpace{s}{p}{q}[\Rn]\):
\begin{equation}\label{Eqn::Spaces::Classical::ClassicalDefnOfASpace}
    \ASpace{s}{p}{q}[\Rngeq]
    =\left\{ f\big|_{\Rng} : f\in \ASpace{s}{p}{q}[\Rngeq]  \right\},
    \quad
    \ANorm{g}{s}{p}{q}[\Rngeq] = \inf \left\{ \ANorm{f}{s}{p}{q}[\Rn] : f\in \ASpace{s}{p}{q}[\Rn], f\big|_{\Rng}=g \right\}.
\end{equation}
See, for example, \cite[Section 2.9.1]{TriebelTheoryOfFunctionSpaces}.
However, we provide an alternate characterization.

Let \(\SchwartzSpaceRng:=\left\{ f\in \SchwartzSpaceRn : \supp(f)\subseteq \Rngeq \right\}\);
which is a closed subspace of \(\SchwartzSpaceRn\), and therefore a Fr\'echet space.
Above, we picked any \(\psi\in \SchwartzSpaceRn\) with \(\int \psi =1\) and \(\int t^{\alpha}\psi(t)\: dt=0\), \(\forall \alpha\ne 0\).
The next lemma lets us choose one in \(\SchwartzSpaceRng\).
\begin{lemma}\label{Lemma::Spaces::Classical::ExistsGoodPsi}
    There exists \(\psi\in \SchwartzSpaceRng\) with \(\int \psi =1\) and \(\int t^{\alpha}\psi(t)\: dt=0\), \(\forall \alpha\ne 0\).
    In fact, \(\psi\) can be chosen with \(\supp(\psi)\subseteq (0,\infty)^n\).
\end{lemma}

Lemma \ref{Lemma::Spaces::Classical::ExistsGoodPsi} follows immediately
from the next lemma, by taking \(\psi(t_1,\ldots, t_n)=H(t_1)H(t_2)\cdots H(t_n)\), where \(H\)
is as in the next lemma.

\begin{lemma}\label{Lemma::Spaces::Classical::ExistenceGoodFunc}
    \(\exists H\in \SchwartzSpace[\R]\) with \(\int H(t)\: dt=1\), \(\int t^{k}H(t)\: dt=0\), \(\forall k\geq 1\),
    and \(\supp(H)\subset (0,\infty)\).
\end{lemma}

The author learned of Lemma \ref{Lemma::Spaces::Classical::ExistenceGoodFunc} from a MathOverflow post of Nazarov
\cite{NazarovSpecialSchwartzFunctionOnThePositiveInterval}. We describe the proof here, as we do not know of a more
permanent reference.
%; we refer the reader to \cite{NazarovSpecialSchwartzFunctionOnThePositiveInterval} for some additional comments.
First, we require another lemma taken from \cite{HavinJorickeUncertaintyPrincipleInHarmonicAnalysis}.

\begin{lemma}[\cite{HavinJorickeUncertaintyPrincipleInHarmonicAnalysis}]\label{Lemma::Spaces::Classical::RestrictFuncAndFT}
    Let \(S,\Sigma\subset \R\) be Lebesgue measurable sets with \(|S|,|\Sigma|<\infty\).
    Then, for any \(p,q\in \LpSpace{2}[\R]\), \(\exists r\in \LpSpace{2}[\R]\) with
    \begin{equation}\label{Eqn::Spaces::Classical::ExistR::Tmp}
        r\big|_{S}=p\big|_S, \quad \hat{r}\big|_{\Sigma}= \hat{q}\big|_{\Sigma}.
    \end{equation}
    Furthermore, if \(S=-S\), \(\Sigma=-\Sigma\), and \(p\big|_{S}\) and \(\hat{q}\big|_{\Sigma}\)
    are real valued and even, then \(r\) may be taken to be real valued and even.
\end{lemma}
\begin{proof}[Comments on the proof]
    The existence of \(r\) satisfying \eqref{Eqn::Spaces::Classical::ExistR::Tmp} can be found in a corollary
    on page 100 of \cite{HavinJorickeUncertaintyPrincipleInHarmonicAnalysis}; in fact, this works in any dimension.
    See \cite{NazarovSpecialSchwartzFunctionOnThePositiveInterval} for a simple proof in a special case which is sufficient
    for our purposes.

    To see that \(r\) can be taken to be real valued and even under the above hypotheses, 
    replace \(r(x)\) with \((r(x)+\overline{r(x)}+r(-x)+\overline{r(-x)})/4\), which satisfies \eqref{Eqn::Spaces::Classical::ExistR::Tmp}
    if \(r\) does and \(p\big|_{S}\) and \(\hat{q}\big|_{\Sigma}\)
    are real valued and even.
\end{proof}

\begin{proof}[Proof of Lemma \ref{Lemma::Spaces::Classical::ExistenceGoodFunc} taken from \cite{NazarovSpecialSchwartzFunctionOnThePositiveInterval}]
    Let \(\psi\in \CinftyCptSpace[(-1/100,1/100)]\) be real and even with \(\int \psi=1\) and \(\psih>1/2\) on \((-1/10,1/10)\).
    By Lemma \ref{Lemma::Spaces::Classical::RestrictFuncAndFT}, \(\exists f\in L^2(\R)\) real and even
    with \(f\big|_{(-1/10,1/10)}=0\), \(\hat{f}\big|_{(-1/10,1/10)}=1/\psih\).
    Set \(F=f*\psi\), so that \(F\big|_{(-1/20,1/20)}=0\), \(\hat{F}\big|_{(-1/10,1/10)}=1\),
    and \(\partial_x^{\alpha} F\in \LpSpace{\infty}\), \(\forall \alpha\). Note, \(F\) is real and even.

    Let \(G\) be defined by \(\hat{G}=\hat{F}*\psi\) so that \(G\big|_{(-1/20,1/20)}=0\), \(\hat{G}\big|_{(-1/20,1/20)}=1\),
    \(G\) is real and even, and \(G\in \SchwartzSpace[\R]\). We have \(\int_{\R} G=1\) and \(\int_{\R} t^{2k}G=0\), \(\forall k\geq 1\),
    and therefore, since \(G\) is even, \(\int_0^{\infty} 2G(t)\: dt=1\) and \(\int_0^{\infty} 2G(t) t^{2k}\: dt=0\), \(\forall k\geq 1\).

    Define \(H(t)\) by \(2tH(t^2)=G(t)\) for \(t\geq 0\) and \(H(t)=0\) for \(t<0\).
    Note \(H\in \SchwartzSpace[\R]\) since \(G\in \SchwartzSpace[\R]\) and \(G=0\) near \(0\).
    \(H\) satisfies the conclusions of the lemma.
\end{proof}

Let \(\TemperedDistributions[\Rng]\) denote the dual of \(\SchwartzSpaceRng\).
With the choice of \(\psi\in \SchwartzSpaceRng\) given in Lemma \ref{Lemma::Spaces::Classical::ExistsGoodPsi},
define \(\phi_j\in \SchwartzSpaceRng\) as in \eqref{Eqn::Spaces::Classical::Definephij}.
Define \(D_j:\TemperedDistributions[\Rng]\rightarrow \CinftySpace[\Rngeq]\) by
\begin{equation}\label{Eqn::Spaces::Classical::DefineDjOnHalfSpace}
    D_j f (x) = \int f(y) \Dil{2^j}{\phi_j}(y-x)\: dy,
\end{equation}
where we have abused notation and wrote integration in place of the pairing of distributions and test functions.
If we replace \(\psi(t)\) with \(\psi(-t)\), this corresponds \eqref{Eqn::Spaces::Classical::DefineDjOnWholeSpace}.
However, in this case, we only define \(D_j f(x)\) for \(x\in \Rngeq\).

\begin{proposition}\label{Prop::Spaces::Classical::NewDefnAspq}
    We have
    \begin{equation}\label{Eqn::Spaces::Classical::NewDefnAspq}
        \ASpace{s}{p}{q}[\Rngeq]\cong \left\{ f\in \TemperedDistributions[\Rng] :  \left\{ D_j f \right\}_{j\in \Zgeq}\in \VSpace{p}{q}  \right\},
        \quad
        \ANorm{f}{s}{p}{q}[\Rngeq] \approx \BVNorm{\left\{ D_j f \right\}_{j\in \Zgeq}}{p}{q},
    \end{equation}
    in the following sense.
    Given any \(g\in \ASpace{s}{p}{q}[\Rngeq]\), defined by \eqref{Eqn::Spaces::Classical::ClassicalDefnOfASpace},
    \(g\) is an equivalence class of elements of \(\TemperedDistributions[\Rn]\). The map
    \(g\mapsto g\big|_{\SchwartzSpaceRng}\) is well defined \(\ASpace{s}{p}{q}[\Rngeq]\rightarrow \TemperedDistributions[\Rng]\),
    and is a bijection from \(\ASpace{s}{p}{q}[\Rngeq]\) as defined by \eqref{Eqn::Spaces::Classical::ClassicalDefnOfASpace}
    to the set described in \eqref{Eqn::Spaces::Classical::NewDefnAspq}.
    Moreover, the norms in \eqref{Eqn::Spaces::Classical::ClassicalDefnOfASpace}
    and \eqref{Eqn::Spaces::Classical::NewDefnAspq} are equivalent under this bijection.
    In \eqref{Eqn::Spaces::Classical::NewDefnAspq}, \(\VSpace{p}{q}\) is defined on the underlying space \(\Rngeq\)
    as in Notation \ref{Notation::Spaces::Classical::VSpacepq}.
\end{proposition}

In this paper, we generalize the spaces \(\ASpace{s}{p}{q}[\Rngeq]\)
by taking Proposition \ref{Prop::Spaces::Classical::NewDefnAspq}
as the starting point.
Though we do not prove Proposition \ref{Prop::Spaces::Classical::NewDefnAspq}
directly in this paper (and do not know of a reference for it),
we prove an essentially more general version.
In fact, our proof of Theorem \ref{Thm::Spaces::Extension}
can be easily modified to the simpler case described in this section
to establish Theorem \ref{Thm::Intro::Classical::Extension}
where \(\BesovSpace{s}{p}{q}[\Rngeq]\) and \(\TLSpace{s}{p}{q}[\Rngeq]\)
are defined by \eqref{Eqn::Spaces::Classical::NewDefnAspq}.
This then implies Proposition \ref{Prop::Spaces::Classical::NewDefnAspq}.
See, also, Proposition \ref{Prop::Spaces::EqualsClassical}
and Corollary \ref{Cor::Spaces::Extend::Consequences::RestrictionSpaceDefn}.

\begin{remark}
    Proposition \ref{Prop::Spaces::Classical::NewDefnAspq} gives an interior
    characterization of \(\ASpace{s}{p}{q}[\Rngeq]\): one can directly check if an
    element of \(\TemperedDistributions[\Rng]\) is in \(\ASpace{s}{p}{q}[\Rngeq]\)
    without extending it to \(\Rn\).
    This is particularly useful in our more general setting,
    see Remark \ref{Rmk::GlobalCor::IntrinsicWithExtension} \ref{Item::GlobalCor::IntrinsicWithExtension::DoesNotDependOnAmbientVfs}.
\end{remark}

    \subsection{Littlewood--Paley theory}\label{Section::Spaces::LittlewoodPaleyTheory}
    Central to the classical definitions of \(\ASpace{s}{p}{q}\) are the Littlewood--Paley
projectors \(D_j\) described in \eqref{Eqn::Spaces::Classical::DefineDjOnWholeSpace}
and \eqref{Eqn::Spaces::Classical::DefineDjOnHalfSpace},
which consist of convolution with Schwartz functions all of whose moments vanish
(see Remark \ref{Rmk::Spaces::Classical::phiIsSchwartzAllOfWhoseMomentsVanish}).
In \cite[Chapter 5]{StreetMaximalSubellipticity}
this concept was adapted to the more general setting of Carnot--Carath\'eodory geometries
on a manifold without boundary; where convolution with Schwartz functions was replaced
with ``pre-elementary operators,'' and convolution with Schwartz functions all of whose
moments vanished ws replaced with ``elementary operators.''
In this section, we generalize these concepts to manifold with boundary.

We continue to restrict attention to the non-characteristic boundary.
This allows us to use the results from Section \ref{Section::VectorFieldsAndSheaves::MetricsAndVolumes},
and allows us to show that the definitions do not depend on any of the choices in our
definitions. See, for example, Lemma \ref{Lemma::Spaces::LP::ElemDoesntDependOnChoices} and Remark \ref{Rmk::Spaces::LP::NonCharMatters}.

The difference between the definitions in this section and those in
\cite{StreetMaximalSubellipticity} is that we keep track of which functions
vanish to infinite order at \(\BoundaryN\). This is important, because integration by parts
is an essential technique for our proofs, and when integrating by parts, unless one
of the two functions vanishes on \(\BoundaryN\), then boundary terms arise.
We will see by carefully setting up our definitions, we can largely avoid such boundary terms.

Let \(\ManifoldN\) be a manifold with boundary and let \(\FilteredSheafF\)
be a H\"ormander filtration of sheaves of vector fields on \(\ManifoldN\).
Let \(\Compact\Subset \ManifoldNncF\) be compact.

Pick \(\Omega\Subset \ManifoldNncF\) open with \(\Compact\Subset \Omega\),
H\"ormander vector fields with formal degrees \(\WWdv=\left\{ (W_1, \Wdv_1),\ldots, (W_r,\Wdv_r) \right\}\subset \VectorFields{\Omega}\times \Zg\)
such that \(\FilteredSheafF\big|_{\Omega}=\FilteredSheafGenBy{\WWdv}\) (see Lemma \ref{Lemma::Filtrations::GeneratorsOnRelCptSet}), and a smooth, strictly positive density \(\Vol\)
on \(\Omega\). We will see (for example Lemma \ref{Lemma::Spaces::LP::ElemDoesntDependOnChoices}) that none of our definitions depend the choice of \(\Omega\), \(\WWdv\), or \(\Vol\).

\begin{notation}\label{Notation::Spaces::LP::MultiIndexAndDegree}
We write \(2^{-j\Wdv}W=(2^{-j\Wdv_1}W_1,\ldots, 2^{-j\Wdv_r}W_r)\), so that
with \(\alpha=(\alpha_1,\ldots, \alpha_L)\),
\begin{equation*}
    \left( 2^{-j\Wdv}W \right)^{\alpha}=2^{-j\Wdv_{\alpha_1}} W_{\alpha_1} 2^{-j\Wdv_{\alpha_2}}W_{\alpha_2}\cdots 2^{-j\Wdv_{\alpha_L}}W_{\alpha_L}.
\end{equation*}
We let \(\DegWdv{\alpha}=\Wdv_{\alpha_1}+\cdots+\Wdv_{\alpha_L}\), so that
\begin{equation*}
    \left( 2^{-j\Wdv}W \right)^{\alpha}=2^{-j\DegWdv{\alpha}} W^{\alpha}.
\end{equation*}
\end{notation}

\begin{remark}\label{Rmk::Spaces::LP::FunctionsOrOperators}
    Functions \(E(x,y)\in \CinftyCptSpace[\ManifoldN\times \ManifoldN]\) which vanish to infinite order as
    \(y\rightarrow \BoundaryN\) can be viewed as operators \(\DistributionsZeroN\rightarrow \CinftyCptSpace[\ManifoldN]\)
    by
    \begin{equation*}
        E u(x) = u\left( E(x,\cdot) \right), \quad u\in \DistributionsZeroN,
    \end{equation*}    
    where we have written \(u\left( E(x,\cdot) \right)\) for the distribution \(u\) applied to the function \(E(x,\cdot)\) in the \(\cdot\)
    variable.
    We also identify \(\CinftySpace[\ManifoldN]\hookrightarrow \DistributionsZeroN\) in the usual way:
    \begin{equation*}
        \PairDistributionAndTestFunctions{f}{g}=\int f(x) g(x) \: d\Vol(x),\quad f\in \CinftySpace[\ManifoldN], g\in \TestFunctionsZeroN.  
    \end{equation*}
    Combining the above we may view \(E(x,y)\in \CinftyCptSpace[\ManifoldN\times \ManifoldN]\) as an operator
    \(E:\DistributionsZeroN\rightarrow \DistributionsZeroN\). While the particular operator so obtained
    depends on the choice of \(\Vol\) (since the embedding \(\CinftySpace[\ManifoldN]\hookrightarrow \DistributionsZeroN\)
    depends on the choice of \(\Vol\)), the class of operators we consider at any point in time will not depend on 
    the choice of smooth, strictly positive density \(\Vol\) (see, for example, Lemma \ref{Lemma::Spaces::LP::ElemDoesntDependOnChoices}).
    Thus, whenever we treat \(E\) as an operator, we have picked a smooth, strictly positive density, and the choice of density is not relevant.   
\end{remark}

\begin{definition}\label{Defn::Spaces::LP::PElemWWdv}
    We let \(\PElemF{\Compact}\) be the set of subsets \(\sE\subset \CinftyCptSpace[\ManifoldN\times \ManifoldN]\times (0,1]\)
    satisfying:
    \begin{enumerate}[(i)]
        \item\label{Item::Spaces::LP::PElemWWdv::Support} \(\forall(E, 2^{-j})\in \sE\), \(\supp(E)\subseteq \Compact\times \Compact\).
        \item\label{Item::Spaces::LP::PElemWWdv::Bound} \(\forall \alpha,\beta\), \(\forall m\in \Zgeq\), \(\exists C=C(\sE,\alpha,\beta,m)\geq 1\), \(\forall (E,2^{-j})\in \sE\),
            \begin{equation}\label{Eqn::Spaces::LP::PreElemBound}
                \left| \left( 2^{-j\Wdv}W_x \right)^{\alpha}\left( 2^{-j\Wdv}W_y \right)^{\beta} E(x,y)\right| \leq C
                \frac{ \left( 1+2^j\MetricWWdv[x][y] \right)^{-m}}{\Vol[ \BWWdv{x}{2^{-j} + \MetricWWdv[x][y]} ]\wedge 1},
                \quad \forall x,y\in \Compact.
            \end{equation}
    \end{enumerate}
    We let \(\PElemzF{\Compact}\) consist of those \(\sE\in \PElemWWdv{\Compact}\) such that
    \begin{equation*}
        (E,2^{-j})\in \sE\implies E(x,y)\text{ vanishes to infinite order as }y\rightarrow \BoundaryN.
    \end{equation*}
\end{definition}

\begin{remark}\label{Rmk::Spaces::LP::BoundIsSymmetric}
    The expression \( \frac{ \left( 1+2^j\MetricWWdv[x][y] \right)^{-m}}{\Vol[ \BWWdv{x}{2^{-j} + \MetricWWdv[x][y]} ]\wedge 1}\)
    is actually symmetric in \(x,y\) in the sense that for \(x,y\in \Compact\),
    \begin{equation*}
        \frac{ \left( 1+2^j\MetricWWdv[x][y] \right)^{-m}}{\Vol[ \BWWdv{x}{2^{-j} + \MetricWWdv[x][y]} ]\wedge 1}
        \approx  \frac{ \left( 1+2^j\MetricWWdv[y][x] \right)^{-m}}{\Vol[ \BWWdv{y}{2^{-j} + \MetricWWdv[y][x]} ]\wedge 1}.
    \end{equation*}
    To see this, one need only show 
    \begin{equation}\label{Eqn::Spaces::LP::BoundIsSymmetric::SymmetricBound}
        \Vol[ \BWWdv{x}{2^{-j} + \MetricWWdv[x][y]} ]\wedge 1\approx \Vol[ \BWWdv{y}{2^{-j} + \MetricWWdv[x][y]} ]\wedge 1, \quad \forall x,y\in \Compact.  
    \end{equation}
    Indeed, if \(2^{-j} + \MetricWWdv[x][y]=2^{-j}+ \DistWWdv[x][y]\wedge 1\geq \DistWWdv[x][y]\), then \eqref{Eqn::Spaces::LP::BoundIsSymmetric::SymmetricBound}
    follows from Proposition \ref{Prop::VectorFields::Scaling::VolEstimates} \ref{Item::VectorFields::Scaling::VolEstimates::VolWedge1Doubling}.
    Otherwise, we have \(2^{-j} + \MetricWWdv[x][y]\geq 1\) and Proposition \ref{Prop::VectorFields::Scaling::VolEstimates} \ref{Item::VectorFields::Scaling::VolEstimates::VolWedge1ApproxLambdaWedge1}
    implies both sides of \eqref{Eqn::Spaces::LP::BoundIsSymmetric::SymmetricBound} are \(\approx 1\).
\end{remark}

\begin{definition}\label{Defn::Spaces::LP::ElemWWdv}
    We let \(\ElemzF{\Compact}\) be the largest set of subsets of \(\CinftyCptSpace[\ManifoldN\times \ManifoldN]\times (0,1]\)
    satisfying:
    \begin{enumerate}[(a)]
        \item\label{Item::Spaces::LP::ElemWWdv::ElemArePreElem} \(\ElemzF{\Compact}\subseteq \PElemzF{\Compact}\),
        \item\label{Item::Spaces::LP::ElemWWdv::ElemArePreDerivs} \(\forall \sE\in \ElemzF{\Compact}\), \(\forall (E,2^{-j})\in \sE\), we have
            \begin{equation*}
                E(x,y)=\sum_{|\alpha|,|\beta|\leq 1} 2^{-j(2-|\alpha|-|\beta|)} \left( 2^{-j\Wdv}W_x \right)^{\alpha} \left( 2^{-j\Wdv}W_y \right)^{\beta} E_{\alpha,\beta}(x,y),
            \end{equation*}
            where
            \begin{equation*}
                \left\{ \left( E_{\alpha,\beta},2^{-j} \right) : \left( E,2^{-j} \right)\in \sE, |\alpha|,|\beta|\leq 1 \right\}\in \ElemzF{\Compact}.
            \end{equation*}
    \end{enumerate}
    For \(\Omega\subseteq \ManifoldNncF\) open, we let
    \begin{equation}\label{Eqn::Spaces::LP::ElemzFOmegaIsUnion}
        \ElemzF{\Omega}:=\bigcup_{\substack{\Compact\Subset \Omega\\ \Compact\text{ compact}}} \ElemzF{\Compact}.
    \end{equation}
\end{definition}

\begin{lemma}\label{Lemma::Spaces::LP::ElemDoesntDependOnChoices}
    \(\PElemF{\Compact}\), \(\PElemzF{\Compact}\), 
    and \(\ElemzF{\Compact}\)
    do not depend on the choices of \(\Omega\), \(\WWdv\), and \(\Vol\)
    made above. More strongly, suppose \(\FilteredSheafG\) is another H\"ormander filtration
    of sheaves of vector fields on \(\ManifoldN\)
    with \(\LieFilteredSheafF=\LieFilteredSheafG\), \(\Omegah\Subset \ManifoldNncF\) is another 
    relatively compact open set with \(\Compact\Subset \Omegah\), \(\ZZde\) are H\"ormander vector fields
    with formal degrees with \(\FilteredSheafG\big|_{\Omegah}=\FilteredSheafGenBy{\ZZde}\),
    and \(\Volh\) is a smooth, strictly positive density on \(\Omegah\).
    Then if \(\PElem{\FilteredSheafG}{\Compact}\), \(\PElemz{\FilteredSheafG}{\Compact}\), 
    and \(\Elemz{\FilteredSheafG}{\Compact}\)
    are defined as in 
    Definitions \ref{Defn::Spaces::LP::PElemWWdv} and \ref{Defn::Spaces::LP::ElemWWdv} with these choices, we have
    \begin{equation*}
        \PElemF{\Compact}=\PElem{\FilteredSheafG}{\Compact}, \quad \PElemzF{\Compact}=\PElemz{\FilteredSheafG}{\Compact}, \quad \ElemzF{\Compact}=\Elemz{\FilteredSheafG}{\Compact}.
    \end{equation*}
\end{lemma}
\begin{proof}
    This follows by combining Lemmas \ref{Lemma::Spaces::Elem::PreElem::PreElemDoesntDependOnChoices} and \ref{Lemma::Spaces::Elem::Elem::ElemDoesntDependOnChoices}, below.
\end{proof}

\begin{remark}\label{Rmk::Spaces::LP::NonCharMatters}
    Lemma \ref{Lemma::Spaces::LP::ElemDoesntDependOnChoices} uses that \(\Compact\Subset \ManifoldNncF\).
    Many of our definitions are only well-defined due to Lemma \ref{Lemma::Spaces::LP::ElemDoesntDependOnChoices}
    and therefore are only well-defined near the non-characteristic part of the boundary.
\end{remark}

\begin{remark}
    \(\PElemzF{\Compact}\) is generalization of convolutions with scaled Schwartz function from a bounded subset
    of \(\SchwartzSpaceRn\),
    while \(\ElemzF{\Compact}\) are the same but with \(\SchwartzSpaceRn\) replaced with
    \(\SchwartzSpacezRn\). See \cite[Section 1.2]{StreetMultiParameterSingularIntegrals} for this connection made explicit.
    In this paper, we deal with the added complication of working on a manifold with boundary--analogous to working with the
    function \(\psi\) from Lemma \ref{Lemma::Spaces::Classical::ExistsGoodPsi} in the classical setting.
\end{remark}

Our analog of the Littlewood--Paley projectors from \eqref{Eqn::Spaces::Classical::DefineDjOnHalfSpace} are described in the next
proposition.
\begin{proposition}\label{Prop::Spaces::LP::DjExist}
    Let \(\Omega\subseteq \ManifoldNncF\) be open and \(\psi\in \CinftyCptSpace[\Omega]\).
    There exists \(\left\{ (D_j, 2^{-j}) : j\in \Zgeq \right\}\in \ElemzF{\Omega}\)
    satisfying:
    \begin{enumerate}[(i)]
        \item\label{Item::Spaces::LP::DjExist::SumToMultiplication} \(\sum_{j\in \Zgeq} D_j = \Multpsi\), where \(\Multpsi\) is the multiplication operator \(f\mapsto \psi f\).
    See Proposition \ref{Prop::Spaces::Elem::Elem::ConvergenceOfElemOps} for the convergence of this sum.
        \item\label{Item::Spaces::LP::DjExist::SumToPreElem} Let \(P_j:=\sum_{k=0}^j D_j\). Then, \(\left\{ \left( P_j, 2^{-j} \right):j\in \Zgeq \right\}\in \PElemzF{\Omega}\).
    \end{enumerate}
    
    % with
    % \(\sum_{j\in \Zgeq} D_j = \Multpsi\), where \(\Multpsi\) is the operator given by multiplication by \(\psi\).
    % See Proposition \ref{Prop::Spaces::Elem::Elem::ConvergenceOfElemOps} for the convergence of this sum.
\end{proposition}
We defer the proof of Proposition \ref{Prop::Spaces::LP::DjExist} to Section \ref{Section::Spaces::Multiplication}.

\begin{remark}
    Definition \ref{Defn::Spaces::LP::ElemWWdv} differs slightly from \cite[Definition 5.2.8]{StreetMaximalSubellipticity},
    which addresses the case of manifolds without boundary. This difference is superficial, and the two definitions
    are equivalent in the case of manifolds without boundary as Lemma \ref{Lemma::Spaces::Elem::Elem::ElemzhF} shows.
\end{remark}
    
    \subsection{Main definitions}\label{Section::Spaces::MainDefns}
    Let \(\ASpace{s}{p}{q}\) be as in Notation \ref{Notation::Spaces::XSpace}, and let \(\VSpace{p}{q}\) be as in
Notation \ref{Notation::Spaces::Classical::VSpacepq}. To define \(\VSpace{p}{q}\) we use a smooth, strictly positive
density \(\Vol\) on \(\ManifoldN\)--the definitions which follow do not depend on this choice.

The definitions which follow are generalizations of the single-parameter definitions in \cite[Chapter 6]{StreetMaximalSubellipticity}--which take
place on a manifold without boundary. We use elements of \(\ElemzF{\ManifoldNncF}\) in the definitions here, which were called
``bounded sets of elementary operators'' in \cite{StreetMaximalSubellipticity}. The reader is referred to \cite{StreetMaximalSubellipticity}
for a more leisurely description of the concepts here, along with discussions of motivations.

See Section \ref{Section::Spaces::BasicProps} for a discussion of the proofs of the results stated in this section.

\begin{definition}\label{Defn::Spaces::Defns::VspqENorm}
    For \(\sE\in \ElemzF{\ManifoldNncF}\) and \(f\in \DistributionsZeroN\), we set
    \begin{equation*}
        \VpqsENorm{f}:=\sup_{\left\{ (E_j, 2^{j}) : j\in \Zgeq \right\}\subseteq \sE} \VNorm{\left\{ E_j f \right\}_{j\in \Zgeq}}{p}{q}.
    \end{equation*}
\end{definition}

Recall, we have fixed \(\Compact\Subset \ManifoldNncF\) compact.

\begin{definition}\label{Defn::Spaces::Defns::ASpace}
    For \(s\in \R\), we define
    \(\ASpace{s}{p}{q}[\Compact][\FilteredSheafF]\) to be the space of those
    \(f\in \DistributionsZeroN\) with \(\supp(f)\subseteq \Compact\)
    and \(\forall \sE \in \ElemzF{\ManifoldNncF}\) we have \(\VpqsENorm{f}<\infty\).
\end{definition}

Fix \(\psi\in \CinftyCptSpace[\ManifoldNncF]\) with \(\psi\equiv 1\) on a neighborhood of \(\Compact\).
By Proposition \ref{Prop::Spaces::LP::DjExist}, we may write
\(\Mult{\psi}=\sum_{j\in \Zgeq} D_j\), where
\(\sD_0:=\left\{ \left( D_j, 2^{-j} \right) : j\in \Zgeq \right\}\in \ElemzF{\ManifoldNncF}\).

\begin{proposition}\label{Prop::Spaces::Defns::NormAndBanach}
    \(\VpqsENorm{\cdot}[p][q][s][\sD_0]\) defines a norm on \(\ASpace{s}{p}{q}[\Compact][\FilteredSheafF]\).
    With this norm, \(\ASpace{s}{p}{q}[\Compact][\FilteredSheafF]\) is a Banach space.
\end{proposition}

\begin{proposition}\label{Prop::Spaces::Defns::EquivNorms}
    The equivalence class of the norm \(\VpqsENorm{\cdot}[p][q][s][\sD_0]\) on \(\ASpace{s}{p}{q}[\Compact][\FilteredSheafF]\)
    does not depend on the choices made above. In particular, if \(\psit\in \CinftyCptSpace[\ManifoldNncF]\)
    is another function with \(\psit\equiv 1\) on a neighborhood of \(\Compact\), and if
    \(\Mult{\psit}=\sum_{j\in \Zgeq} \Dt_j\), where \(\sDt_0:=\left\{ \left( \Dt_j,2^{-j} \right) : j\in \Zgeq \right\}\in \ElemzF{\ManifoldNncF}\),
    then we have
    \begin{equation*}
        \VpqsENorm{f}[p][q][s][\sD_0]\approx \VpqsENorm{f}[p][q][s][\sDt_0], \quad \forall f\in \ASpace{s}{p}{q}[\Compact][\FilteredSheafF].
    \end{equation*}
\end{proposition}

Propositions \ref{Prop::Spaces::Defns::NormAndBanach} and \ref{Prop::Spaces::Defns::EquivNorms} are proved in
Section \ref{Section::Spaces::BasicProps}.
In light of Proposition \ref{Prop::Spaces::Defns::EquivNorms}, the following notation makes sense.

\begin{notation}\label{Notation::Spaces::Defns::Norm}
    For \(f\in \ASpace{s}{p}{q}[\Compact][\FilteredSheafF]\), we write
    \begin{equation*}
        \ANorm{f}{s}{p}{q}[\FilteredSheafF]:=\VpqsENorm{f}[p][q][s][\sD_0], \quad f\in \ASpace{s}{p}{q}[\Compact][\FilteredSheafF].
    \end{equation*}
\end{notation}

\begin{remark}\label{Rmk::Spaces::Defns::EquivClassOfNormWellDefinedAndDoesntDependOnCompact}
    Only the equivalence class of the norm \(\ANorm{\cdot}{s}{p}{q}[\FilteredSheafF]\) is well defined on \(\ASpace{s}{p}{q}[\Compact][\FilteredSheafF]\).
    The notation \(\ANorm{f}{s}{p}{q}[\cdot]\) does not involve the compact set \(\Compact\), since given two compact sets
    \(\Compact,\Compact'\Subset \ManifoldNncF\), we may pick \(\psi\in \CinftyCptSpace[\ManifoldNncF]\) so that \(\psi\equiv 1\) on a neighborhood of \(\Compact\cup\Compact'\),
    and we may therefore use the same norm on \(\ASpace{s}{p}{q}[\Compact][\FilteredSheafF]\) and \(\ASpace{s}{p}{q}[\Compact'][\FilteredSheafF]\).
\end{remark}

\begin{remark}\label{Rmk::Spaces::Defns::NormFiniteDoesntMeanInSpace}
    Proposition \ref{Prop::Spaces::Defns::EquivNorms} only shows that the equivalence class of the norm
    \(\ANorm{f}{s}{p}{q}[\FilteredSheafF]\) is well-defined for \(f\in \ASpace{s}{p}{q}[\Compact][\FilteredSheafF]\),
    not for all \(f\in \DistributionsZeroN\) with \(\supp(f)\subseteq \Compact\). In particular,
    \textbf{it is not claimed} that if \(\ANorm{f}{s}{p}{q}[\FilteredSheafF]<\infty\) with \(\supp(f)\subseteq \Compact\),
    then \(f\in\ASpace{s}{p}{q}[\Compact][\FilteredSheafF]\). This issue can be fixed by modifying the norm as described
    in Remark \ref{Rmk::Spaces::MainEst::EquivNormWhichHasFinitenessProperty}.
    When \(s>0\), one does not need to modify the norm; see Section \ref{Section::Spaces::BasicProps::FiniteNorm}.
\end{remark}

In light of Proposition \ref{Prop::Spaces::Containment}, the following definition makes sense.
\begin{definition}\label{Defn::Spaces::Defns::ASpaceCpt}
    For \(\Omega\subseteq \ManifoldNncF\) open, set 
    \begin{equation*}
        \ASpaceCpt{s}{p}{q}[\Omega][\FilteredSheafF]:=\bigcup_{\substack{\Compact\Subset \Omega \\ \Compact\text{ compact}}} \ASpace{s}{p}{q}[\Compact][\FilteredSheafF].
    \end{equation*}
\end{definition}

    \subsection{Scaling}\label{Section::Spaces::Scaling}
    As described in Section \ref{Section::VectorFieldsAndSheaves::MetricsAndVolumes}, a key tool when studying maximally
subelliptic PDEs are scaling maps adapted to the Carnot-Carath\'eodory geometry; and such scaling maps were first introduced
by Nagel, Stein, and Wainger \cite{NagelSteinWaingerBallsAndMetricsDefinedByVectorFieldsI}.
In this section, we describe the scaling maps presented in \cite[Theorem \ref*{CC::Thm::Scaling::MainResult}]{StreetCarnotCaratheodoryBallsOnManifoldsWithBoundary},
which work both on \(\InteriorN\) and near \(\BoundaryNncF\).

For simplicity, we describe the result on the unit ball: \(\Bn{1}=\left\{ x\in \Rn : |x|\leq 1 \right\}\),
though \cite[Theorem \ref*{CC::Thm::Scaling::MainResult}]{StreetCarnotCaratheodoryBallsOnManifoldsWithBoundary}
works on more general manifolds and contains some additional results which are not used in this paper. We set \(\Bngeq{1}=\left\{ x=(x_1,\ldots,x_n)\in \Bn{1} : x_n\geq 0 \right\}\),
and \(\Bng{1}\left\{ x=(x_1,\ldots,x_n)\in \Bn{1} : x_n> 0 \right\}\).
\(\Bngeq{1}\) is our manifold with boundary, \(\Bng{1}\) is the interior, \(\Bnmo{1}\) the boundary,
and \(\Bn{1}\) the ambient manifold. We similarly define \(\Bnleq{1}\) and \(\Bnl{1}\) by replacing \(\geq 0\) and \(> 0\)
with \(\leq 0\) and \(<0\), respectively.

Let \(\WhWdv=\left\{ \left( \Wh,\Wdv_1 \right),\ldots, \left( \Wh, \Wdv_r \right) \right\}\subset \VectorFields{\Bn{1}}\times \Zg\)
be H\"ormander vector fields with formal degrees on \(\Bn{1}\).

\begin{assumption}
    We make the following non-characteristic assumption: there exists \(j_0\in \left\{ 1,\ldots, r \right\}\),
    with \(\Wh_{j_0}(x',0)\not \in \TangentSpace{x'}{\Rnmo}\), \(\forall x'\in \Bnmo{1}\)
    and \(\Wdv_k<\Wdv_{j_0}\) implies \(\Wh_k(x',0)\in \TangentSpace{(x',0)}{\Rnmo}\), \(\forall x'\in \Bnmo{1}\).
    In other words, \(\Wh_{j_0}(x',0)\) has a non-zero \(\partial_{x_n}\) component for \(x'\in \Bnmo{1}\),
    and if \(\Wdv_k<\Wdv_{j_0}\), then \(\Wh_k(x',0)\) does not have a \(\partial_{x_n}\) component for \(x'\in \Bnmo{1}\).
\end{assumption}

We let \(\Qn{1}=\left\{ x\in \Rn : |x|_{\infty}<1 \right\}\) and for \(c\in [-1,0]\),
\(\Qngeqc{1}{c}= \left\{ x=(x_1,\ldots, x_n)\in \Qn{1} : x_n\geq c \right\}\)
and \(\Qneqc{1}{c}= \left\{ x=(x_1,\ldots, x_n)\in \Qn{1} : x_n= c \right\}\).
Note that \(\Qngeqc{1}{-1}=\Qn{1}\).

For a vector field \(V\in \VectorFields{\Qn{1}}\), \(V=\sum_{j=1}^n a_j(x)\partial_{x_j}\), we identify
\(V\) with the vector valued function \(x\mapsto (a_1(x),\ldots, a_n(x))\). Thus, it makes sense to consider norms like
\(\CmNorm{V}{L}[\Qn{1}][\Rn]\), which denotes the usual \(\CmSpace{L}[\Qn{1}]\) norm, for functions taking values in \(\Rn\).

\begin{theorem}[{See \cite[Theorem \ref*{CC::Thm::Scaling::MainResult}]{StreetCarnotCaratheodoryBallsOnManifoldsWithBoundary}}]
    \label{Thm::Spaces::Scaling::MainScalingThm}
    Fix \(r_1,r_2\in (0,1)\) with \(r_1<r_2\). There exists \(\delta_1\in (0,1]\), \(\forall x \in \BnClosure{r_1}\),
    \(\forall \delta\in (0,\delta_1]\), \(\exists \Psi_{x,\delta}:\Qn{1}\rightarrow \BWhWdv{x}{\delta}\) such that
    \(\forall x\in \BnClosure{r_1}\) and \(\delta\in (0,\delta_1]\):
    \begin{enumerate}[(a)]
        \item\label{Item::Spaces::Scaling::ContainedInAmbient}  \(\BWWdv{x}{\delta}\subseteq \Bn{r_2}\).
        \item\label{Item::Spaces::Scaling::Doubling} \(\Vol[\BWhWdv{x}{2\delta}]\lesssim \Vol[\BWhWdv{x}{\delta}]\), where \(\Vol\) denotes Lebesgue measure.
        \item \(\Psi_{x,\delta}(0)=x\).
        \item \(\Psi_{x,\delta}(\Qn{1})\subseteq \Bn{r_2}\) is open and \(\Psi_{x,\delta}:\Qn{1}\xrightarrow{\sim} \Psi_{x,\delta}(\Qn{1})\)
            is a \(\CinftySpace\)-diffeomorphism.
        \item \(\forall x\in \BnClosure{r_1}\), \(\delta\in (0,\delta_1]\), \(\exists c_0=c_0(x,\delta)\in [-1,0]\), such that
            \begin{equation*}
                \Psi_{x,\delta}(\Qn{1})\cap \Bng{1}= \Psi_{x,\delta}(\Qngeqc{1}{c_0}), \quad \Psi_{x,\delta}(\Qn{1})\cap \Bnmo{1}=\Psi_{x,\delta}(\Qneqc{1}{c_0}).
            \end{equation*}
            Here, \(c_0=-1\) corresponds to the case \(\Psi_{x,\delta}(\Qn{1})\subseteq \Bng{1}\) and \(\Psi_{x,\delta}(\Qn{1})\cap \Bnmo{1}=\emptyset\),
            and \(c_0=0\) corresponds to the case \(x\in \Bnmo{1}\).
        \item\label{Item::Spaces::Scaling::ImageContainsBall}
         \(\exists \xi_1\in (0,1]\), \(\forall x\in \BnClosure{r_1}\), \(\delta\in (0,\delta_1]\),
            \begin{equation*}
                \BWhWdv{x}{\xi_1\delta}\subseteq \Psi_{x,\delta}(\Qn{1/4})\subseteq \Psi_{x,\delta}(\Qn{1})\subseteq \BWhWdv{x}{\delta}.
            \end{equation*}
    \end{enumerate}
    Let \(\Wh_j^{x,\delta}:=\Psi_{x,\delta}^{*} \delta^{\Wdv_j}\Wh_j\).
    \begin{enumerate}[(a),resume]
        \item\label{Item::Spaces::Scaling::PullBackSmooth} \(\Wh_1^{x,\delta},\ldots, \Wh_r^{x,\delta}\) are smooth, uniformly in \(x\) and \(\delta\) in the sense that
            \(\forall L\in \Zgeq\),
            \begin{equation*}
                \sup_{j\in \left\{ 1,\ldots, r \right\}}\sup_{\substack{x\in \BnClosure{r_1} \\ \delta\in (0,\delta_1]}} 
                \BCmNorm{\Wh_j^{x,\delta}}{L}[\Qn{1}][\Rn]<\infty.
            \end{equation*}
        \item\label{Item::Spaces::Scaling::PullBackSatisfyHormander} \(\Wh_1^{x,\delta},\ldots, \Wh_r^{x,\delta}\) satisfy H\"ormander's condition, uniformly in \(x\) and \(\delta\)
            in the sense that there exists \(m\in \Zgeq\), such that if
            \(\Xh^{x,\delta}_1,\ldots, \Xh^{x,\delta}_q\) denote the list of commutators of \(\Wh_1^{x,\delta},\ldots, \Wh_r^{x,\delta}\)
            up to order \(m\), then
            \begin{equation*}
                \inf_{\substack{x\in \BnClosure{r_1} \\ \delta\in (0,\delta_1] \\ u\in \Qn{1}}} \max_{\substack{j_1,\ldots, j_n\\\in \left\{ 1,\ldots, q \right\}}}
                \left| 
                    \det
                    \left( \Xh_{j_1}^{x,\delta}(u)| \cdots | \Xh_{j_n}^{x,\delta}(u) \right)
                \right|>0,
            \end{equation*}
            where \(\left( \Xh_{j_1}^{x,\delta}(u)| \cdots | \Xh_{j_n}^{x,\delta}(u) \right)\) denotes the \(n\times n\) matrix
            whose \(k\)-th column is \(\Xh_{j_k}^{x,\delta}(u)\), written as a column vector.
    \end{enumerate}
\end{theorem}

    \subsection{Elementary and pre-elementary operators}
    In this section, we prove the relevant facts we require
about pre-elementary operators (\(\PElemF{\Compact}\) and \(\PElemzF{\Compact}\)) and 
elementary operators (\(\ElemzF{\Compact}\)),
as described in Definitions \ref{Defn::Spaces::LP::PElemWWdv} and \ref{Defn::Spaces::LP::ElemWWdv}.
In particular, we establish Lemma \ref{Lemma::Spaces::LP::ElemDoesntDependOnChoices}
and Proposition \ref{Prop::Spaces::LP::DjExist}; along with other needed results.

The notations \(\PElemF{\Compact}\), \(\PElemzF{\Compact}\), and \(\ElemzF{\Compact}\)
are well-defined\footnote{In the sense that it does not depend on any of the choices made in Definitions \ref{Defn::Spaces::LP::PElemWWdv} and \ref{Defn::Spaces::LP::ElemWWdv}.} in light of Lemma \ref{Lemma::Spaces::LP::ElemDoesntDependOnChoices}.
However, we have not yet established Lemma \ref{Lemma::Spaces::LP::ElemDoesntDependOnChoices}, so we temporarily
introduce notation which makes the choices involved in Definitions \ref{Defn::Spaces::LP::PElemWWdv} and \ref{Defn::Spaces::LP::ElemWWdv}
explicit.

Throughout this section, fix \(\Omega\Subset \ManifoldNncF\) open an relatively compact, with \(\Compact\Subset \Omega\),
and H\"ormander vector fields with formal degrees \(\WWdv=\left\{ \left( W_1,\Wdv_1 \right),\ldots, \left( W_r, \Wdv_r \right) \right\}\subset\VectorFields{\ManifoldNncF}\times \Zg\),
with \(\FilteredSheafF\big|_{\Omega}=\FilteredSheafGenBy{\WWdv}\).
Fix a smooth, strictly positive density \(\Vol\) on \(\Omega\).

\begin{notation}\label{Notation::Spaces::Elem::Intro::VerposeElemNotation}
    We write \(\PElemWWdvOmegaVol{\Compact}\), \(\PElemzWWdvOmegaVol{\Compact}\), and
    \(\ElemzWWdvOmegaVol{\Compact}\) for the corresponding
    sets \(\PElemF{\Compact}\),  \(\PElemzF{\Compact}\), and \(\ElemzF{\Compact}\)
    from Definitions \ref{Defn::Spaces::LP::PElemWWdv} and \ref{Defn::Spaces::LP::ElemWWdv}
    when defined with the above choices. As we will show in Lemma \ref{Lemma::Spaces::LP::ElemDoesntDependOnChoices},
    none of these choices matter and so once Lemma \ref{Lemma::Spaces::LP::ElemDoesntDependOnChoices} is established,
    we may write
    \(\PElemWWdvOmegaVol{\Compact}=\PElemF{\Compact}\),
    \(\PElemzWWdvOmegaVol{\Compact}=\PElemzF{\Compact}\),
    and \(\ElemzWWdvOmegaVol{\Compact}=\ElemzF{\Compact}\).
\end{notation}

The main difference between the setting here and that of the single-parameter special case of
\cite[Sections 5.4 and 5.5]{StreetMaximalSubellipticity} is that for
\(\sE\in \PElemzWWdvOmegaVol{\Compact}\) or \(\sE\in \ElemzWWdvOmegaVol{\Compact}\) and \((E,2^{-j})\in \sE\),
we have \(E(x,y)\) vanishes to infinite order as \(y\rightarrow \BoundaryN\),
while in \cite{StreetMaximalSubellipticity} the manifold had no boundary so this vanishing condition is vacuous.
Because of this, many of the proofs from \cite{StreetMaximalSubellipticity} generalize to this setting
with minimal changes. When this is the case, we describe only the changes needed, and refer the reader
to the corresponding result in \cite{StreetMaximalSubellipticity} for full details.

        \subsubsection{Extending and restricting pre-elementary operators}
        The results in this subsection have technical statements and proofs. It may be beneficial for the reader to skip this on a first reading,
returning when the results are needed.

Let \(\ManifoldM\) be a smooth manifold, without boundary, such that
\(\ManifoldN\subseteq \ManifoldM\) is a co-dimension \(0\), closed, embedded, submanifold with boundary,
and let \(\FilteredSheafFh\) be a H\"ormander filtration of sheaves of vector fields on \(\ManifoldM\)
with \(\RestrictFilteredSheaf{\FilteredSheafFh}{\ManifoldN}=\FilteredSheafF\).
Fix \(\Compacth\Subset \ManifoldM\) compact with \(\Compacth\cap \ManifoldN=\Compact\).
In this subsection, we show that we may restrict pre-elementary operators
from \(\PElemFh{\Compacth}\) to obtain pre-elementary operators in \(\PElemF{\Compact}\) (see Proposition \ref{Prop::Spaces::Elem::Extend::Restrict::NonVerbose}).
And, conversely, this restriction is right-invertible by a corresponding extension (see Proposition \ref{Prop::Spaces::Elem::Extend::Extend::NonVerbose}).
Throughout, fix a smooth, strictly positive density \(\Vol\) on \(\ManifoldM\).

Let \(\Omega\Subset\ManifoldM\) be an  \(\ManifoldM\)-open, relatively compact set, with \(\Omega\cap \ManifoldN\subseteq \ManifoldNncF\),
and \(\Compacth\subseteq \Omega\). 

The main results of this subsection are the next two propositions:

\begin{proposition}[Restriction]\label{Prop::Spaces::Elem::Extend::Restrict::NonVerbose}
    Let \(\sE\in \PElemFh{\Compacth}\). Then,
    \begin{equation*}
        \left\{ \left( E\big|_{\ManifoldN\times \ManifoldN}, 2^{-j} \right) : \left( E,2^{-j} \right)\in \sE \right\}
        \in
        \PElemF{\Compact}.
    \end{equation*}
\end{proposition}
\begin{proof}
    Once Lemma \ref{Lemma::Spaces::LP::ElemDoesntDependOnChoices} is proven,
    this is equivalent to Proposition \ref{Prop::Spaces::Elem::Extend::Restrict::Verbose}, below.
\end{proof}

\begin{proposition}[Extension]\label{Prop::Spaces::Elem::Extend::Extend::NonVerbose}
    Let \(\Omega_1\Subset \Omega\) be \(\ManifoldM\)-open and relatively compact with \(\Compact \Subset\Omega_1\).
    For each \(\delta\in (0,1]\) three exist linear operators
    \begin{equation*}
        \Extend_\delta:\left\{ E(x,y)\in \CinftySpace[\ManifoldN\times \ManifoldN] : \supp(E)\subseteq \Compact\times \Compact \right\}
            \rightarrow \CinftyCptSpace[\Omega_1\times \Omega_1],
    \end{equation*}
    \begin{equation*}
        \begin{split}
        \Extend_\delta^0:
        &\left\{ E(x,y)\in \CinftySpace[\ManifoldN\times \ManifoldN] : \supp(E)\subseteq \Compact\times \Compact, E(x,y)\text{ vanishes to infinite order as }y\rightarrow\BoundaryN \right\}
            \\&\rightarrow 
            \left\{ E\in \CinftyCptSpace[\Omega_1\times \Omega_1] : \supp(E)\subseteq \Omega_1\times \left( \Omega_1\cap \ManifoldN \right) \right\}
        \end{split}
    \end{equation*}
    such that:
    \begin{enumerate}[(i)]
        \item\label{Item::Spaces::Elem::Extend::Extend::NonVerbose::ExtendRestrictIsI} 
        \(\Extend_\delta E\big|_{\ManifoldN\times \ManifoldN} = E\) and \(\Extend_\delta^0 E\big|_{\ManifoldN\times \ManifoldN}=E\),
            for all \(E\) in the respective domains.
        \item \label{Item::Spaces::Elem::Extend::Extend::NonVerbose::ExtendPreElemIsPreElem} 
            If \(\sE\in \PElemF{\Compact}\), then
            \begin{equation*}
                \left\{ \left( \Extend_\delta E, \delta \right) : (E,\delta)\in \sE \right\}
                \in \PElemFh{\overline{\Omega_1}}.
            \end{equation*}
            \item \label{Item::Spaces::Elem::Extend::Extend::NonVerbose::ExtendPreElemzIsPreElem}
            If \(\sE\in \PElemzF{\Compact}\), then
            \begin{equation*}
                \left\{ \left( \Extend_\delta^0 E, \delta \right) : (E,\delta)\in \sE \right\}
                \in \PElemFh{\overline{\Omega_1}}.
            \end{equation*}
    \end{enumerate}
\end{proposition}
\begin{proof}
    Once Lemma \ref{Lemma::Spaces::LP::ElemDoesntDependOnChoices} is proven,
    this is equivalent to 
    Proposition \ref{Prop::Spaces::Elem::Extend::Extend::Verbose}.
\end{proof}

Since we have not yet proved Lemma \ref{Lemma::Spaces::LP::ElemDoesntDependOnChoices},
we rewrite Propositions \ref{Prop::Spaces::Elem::Extend::Restrict::NonVerbose} and \ref{Prop::Spaces::Elem::Extend::Extend::NonVerbose} using Notation \ref{Notation::Spaces::Elem::Intro::VerposeElemNotation}.
Let \(\FilteredSheafFh\big|_{\Omega}=\FilteredSheafGenBy{\WhWdv}\),
where \(\WhWdv=\left\{ \left( \Wh_1, \Wdv_1 \right),\ldots, \left( \Wh_r, \Wdv_r \right) \right\}\subset \VectorFields{\Omega}\times \Zg\)
are H\"ormander vector fields with formal degrees on \(\Omega\).
Let \(W_j:=\Wh_j\big|_{\Omega\cap \ManifoldN}\)
and \(\WWdv=\left\{ \left( W_1,\Wdv_1 \right),\ldots, \left( W_r,\Wdv_r \right) \right\}\). It follows from Proposition \ref{Prop::Filtrations::RestrictingFiltrations::CoDim0Restriction}
that \(\FilteredSheafF\big|_{\Omega\cap \ManifoldN}=\FilteredSheafGenBy{\WWdv}\).

\begin{proposition}[Restriction]\label{Prop::Spaces::Elem::Extend::Restrict::Verbose}
    Proposition \ref{Prop::Spaces::Elem::Extend::Restrict::NonVerbose} holds with
    \(\PElemOnlySubscript{\FilteredSheafFh}\) and 
    \(\PElemOnlySubscript{\FilteredSheafF}\) replaced by 
    \(\PElemOnlySubscript{\WhWdv,\Omega,\Vol}\) and
    \(\PElemOnlySubscript{\WWdv,\Omega \cap \ManifoldN,\Vol}\), respectively.
\end{proposition}
\begin{proof}
    For \((E,2^{-j})\in \sE\), we clearly have
    \( \supp \left( E\big|_{\ManifoldN\times \ManifoldN} \right)\subseteq \Compact\times \Compact\),
    so we need only establish the estimates \eqref{Eqn::Spaces::LP::PreElemBound}.
    Indeed, for \(x,y\in\Compact\), we have \(\forall m,\alpha,\beta\),
    \begin{equation*}
    \begin{split}
         &\left| \left( 2^{-j\Wdv}W_x \right)^{\alpha} \left( 2^{-j\Wdv}W_y \right)^{\beta} E(x,y) \right|
         = \left| \left( 2^{-j\Wdv}\Wh_x \right)^{\alpha} \left( 2^{-j\Wdv}\Wh_y \right)^{\beta} E(x,y) \right|
         \\&\lesssim \frac{  \left( 1+2^{j}\MetricWhWdv[x][y] \right)^{-m} }{\Vol[ \BWhWdv{x}{2^{-j} + \MetricWhWdv[x][u]} ]\wedge 1}
         \approx 
         \frac{  \left( 1+2^{j}\MetricWWdv[x][y] \right)^{-m} }{\Vol[ \BWWdv{x}{2^{-j} + \MetricWWdv[x][u]} ]\wedge 1},
    \end{split}
    \end{equation*}
    where the final \(\approx\) follows from Proposition \ref{Prop::VectorFields::Scaling::AmbientVolAndMetricEquivalnce} 
    \ref{Item::VectorFields::Scaling::AmbientVolAndMetricEquivalnce::MetricWequalsMetricZ} and \ref{Item::VectorFields::Scaling::AmbientVolAndMetricEquivalnce::VolOfMetricWequalsVolOfMetricZ}.
\end{proof}

\begin{proposition}[Extension]\label{Prop::Spaces::Elem::Extend::Extend::Verbose}
    Proposition \ref{Prop::Spaces::Elem::Extend::Extend::NonVerbose} holds with
    \(\PElemOnlySubscript{\FilteredSheafFh}\), 
    \(\PElemOnlySubscript{\FilteredSheafF}\), and
    \(\PElemzOnlySubscript{\FilteredSheafF}\)
    replaced by 
    \(\PElemOnlySubscript{\WhWdv,\Omega,\Vol}\),
    \(\PElemOnlySubscript{\WWdv,\Omega \cap \ManifoldN,\Vol}\), 
    and
    \(\PElemzOnlySubscript{\WWdv,\Omega \cap \ManifoldN,\Vol}\)
    respectively.
\end{proposition}

We remind the reader of a standard result.

\begin{lemma}\label{Lemma::Spaces::Elem::Extend::Classical}
    Let \(\Compact_1,\Compact_2\Subset \ManifoldN\) be compact, and \(\Omega_1,\Omega_2\subseteq \ManifoldM\) be \(\ManifoldM\)-open with
    \(\Compact_1\subseteq \Omega_1\) and \(\Compact_2\subseteq\Omega_2\). Then, there exist linear extension maps:
    \begin{equation*}
        \Extend^1 :\left\{ f(x)\in \CinftySpace[\ManifoldN] : \supp(f)\subseteq \Compact_1\right\}\rightarrow \CinftyCptSpace[\Omega_1],
    \end{equation*}
    \begin{equation*}
        \Extend^2 :\left\{ F(x,y)\in \CinftySpace[\ManifoldN\times \ManifoldN] : \supp(F)\subseteq \Compact_1\times \Compact_2\right\}\rightarrow \CinftyCptSpace[\Omega_1\times \Omega_2],
    \end{equation*}
    such that
    \begin{itemize}
        \item \(\Extend^1 f\big|_{\ManifoldN}=f\) and \(\Extend^2F\big|_{\ManifoldN\times \ManifoldN}=F\) for \(f\) and \(F\) in the domains
            of \(\Extend^1\) and \(\Extend^2\), respectively.
        \item For every \(N\),
            \begin{equation*}
                \CmNorm{\Extend^1 f}{N}\lesssim \CmNorm{f}{N}, \quad \CmNorm{\Extend^2 F}{N}\lesssim \CmNorm{F}{N}.
            \end{equation*}
            Here, one may choose any norm from the equivalence class of norms on the respective \(\CmSpace{N}\) space.
    \end{itemize}
\end{lemma}
\begin{proof}[Comments on the proof]
    This is result is standard.
    Results of this form were first established by Mitjagin \cite{MitjaginApproximateDimensionAndBasesInNuclearSpaces}
    and Seeley \cite{SeeleyExtensionOfCInfinityFunctionsDefinedInAHalfSpace}.
    For example, \cite{SeeleyExtensionOfCInfinityFunctionsDefinedInAHalfSpace} works on a half-space, but a simple
    partition of unity can be used to prove the result on a manifold as is stated here for \(\Extend^1\).
    To obtain results like \(\Extend^2\), one may work in each variable separately and thereby reduce to the methods
    of \cite{SeeleyExtensionOfCInfinityFunctionsDefinedInAHalfSpace}.
\end{proof}

We turn to a reduction.\footnote{This reduction is not necessary, and the proof below can be carried out on an arbitrary
manifold, but becomes notationally more cumbersome.}
\begin{lemma}\label{Lemma::Spaces::Elem::Extend::ReduceToUnitBall}
    It suffices to prove Proposition \ref{Prop::Spaces::Elem::Extend::Extend::Verbose} in the special case \(\ManifoldM=\Bn{1}\), \(\ManifoldN=\Bngeq{1}\),
    \(\Compact=\overline{\Bngeq{1/2}}\), \(\Omega_1=\Bn{3/4}\),  \(\Omega=\Bn{7/8}\),
    \(\Vol\) equal to Lebesgue measure,
    and \(\WhWdv\) satisfying the hypotheses of Section \ref{Section::Spaces::Scaling}.
\end{lemma}
\begin{proof}
    For each \(x\in \Compact\cap \BoundaryN\), let \(U_x\subset \Omega_1\) be an \(\ManifoldM\)-neighborhood of 
    \(x\) such that:
    \begin{itemize}
        \item \(\exists j_0=j_0(x)\in \left\{ 1,\ldots, r \right\}\), \(\forall y\in U_x\cap \BoundaryN\), \(W_{j_0}(y)\not \in \TangentSpace{y}{\BoundaryN}\)
            and \(\Wdv_k<\Wdv_{j_0}\) implies \(W_k(y)\in \TangentSpace{y}{\BoundaryN}\), \(\forall y\in U_x\cap \BoundaryN\) (this is always
            possible, since \(\Compact\cap\BoundaryN\subseteq \BoundaryNncF\)).
        \item There exists a diffeomorphism \(g_x:\Bn{1}\xrightarrow{\sim}U_x\), with \(g_x(0)=x\), \(g_x(\Bnmo{1})=U_x\cap \BoundaryN\),
            and \(g_x(\Bngeq{1})=U_x\cap \ManifoldN\).
    \end{itemize}
    \(\left\{ g_x(\Bn{1/2}) : x\in \Compact\cap \BoundaryN \right\}\) is an open cover for the compact
    set \(\Compact\cap\BoundaryN\).  Let \(\phi_j\), \(j=1,\ldots, M\) be a partition of unity subordinate to this cover,
    so that:
    \begin{itemize}
        \item \(\phi_{j}\in \CinftyCptSpace[g_{x_j}(\Bn{1/2})]\), for some \(x_j\in \Compact\cap \BoundaryN\).
        \item \(\sum_{j=1}^M \phi_j =1\) on a neighborhood of \(\Compact\).
    \end{itemize}
    Take \(\phit_j\in \CinftyCptSpace[g_{x_j}(\Bn{1/2})]\) with \(\phit_j=1\) on a neighborhood of \(\supp(\phi_j)\).
    Take \(\psi,\psit\in \CinftyCptSpace[\InteriorN]\) with \(\psit=1\) on a neighborhood of \(\supp(\psi)\)
    and \(\psi+\sum_{j=1}^M \phi_j=1\) on a neighborhood of \(\Compact\).

    For \(E(x,y)\in \CinftySpace[\ManifoldN\times \ManifoldN]\) with \(\supp(E)\subseteq \Compact\times \Compact\),
    we have
    \begin{equation}\label{Eqn::Spaces::Elem::Extend::DecomposeToReduceToBall}
            E(x,y) = \sum_{j=1}^M \phi_j(x) E(x,y)\phit_j(y) + \sum_{j=1}^M \phi_j(x) E(x,y)(1-\phit_j(y)) + \psi(x)E(x,y).
    \end{equation}
    The goal is to extend \(E\) to \(\ManifoldM\times\ManifoldM\) as in the statement of Proposition \ref{Prop::Spaces::Elem::Extend::Extend::Verbose};
    we do treating the three terms on the right-hand side of \eqref{Eqn::Spaces::Elem::Extend::DecomposeToReduceToBall},
    separately.

    Throughout we assume \((E,\delta)\in \sE\), where
    \(\sE\in \PElemz{\WWdv,\Omega\cap \ManifoldN,\Vol}{\Compact}\) when considering \(\Extend^0_\delta\)
    and \(\sE\in \PElem{\WWdv,\Omega\cap \ManifoldN,\Vol}{\Compact}\) when considering \(\Extend_\delta\).
    
    When considering the term \(\Extend_\delta \psi(x) E(x,y)\), we have \(\supp(\psi(x)E(x,y))\subseteq (\InteriorN\cap \Compact)\times \Compact\),
    and therefore vanish to infinite order as \(x\rightarrow \BoundaryN\).
    By reversing the roles of \(x\) and \(y\), this is of the same form as \(\Extend_\delta^0 E(x,y)\), so 
    once we prove the full result for \(\Extend_\delta^0 E(x,y)\), the term \(\Extend_\delta \psi(x)E(x,y)\) can be handled in the same way.

    We are assuming we already have the result on the unit ball (as described in the statement of the lemma),
    and so by pulling back to \(\Bn{1}\) via \(g_{x_j}\), we may extend \( \phi_j(x) E(x,y)\phit_j(y)\)
    as in the statement of Proposition \ref{Prop::Spaces::Elem::Extend::Extend::Verbose}. Here, we may use Lebesgue measure
    in place of \(g_x^{*}\Vol\), since any smooth, strictly positive density gives equivalent estimates
    for the volumes which are used (see Proposition \ref{Prop::VectorFields::Scaling::VolAndMetricEquivalnce} \ref{Item::VectorFields::Scaling::VolAndMetricEquivalnce::VolOfMetricWequalsVolOfMetricZ}).
    So, it suffices to extend the terms \(\sum_{j=1}^M \phi_j(x) E(x,y)(1-\phit_j(y))\)
    and \( \psi(x)E(x,y)\)
    for \(\Extend_\delta^0\), and the term \(\sum_{j=1}^M \phi_j(x) E(x,y)(1-\phit_j(y))\) for \(\Extend_\delta\).
    For these terms, \(\Extend_\delta^0\) and \(\Extend_\delta\) will not depend on \(\delta\),
    and so we drop the \(\delta\) from our notation in what follows; though our estimates still depend on \(\delta\)
    since \((E,\delta)\in \sE\).

    We begin with \(\Extend^0\) and the terms \(\psi(x) E(x,y)\), \((E,\delta)\in \sE\).
    Since (when considering \(\Extend_\delta^0\)), we have 
    \(E(x,y)\) vanishes to infinite order as \(y\rightarrow \BoundaryN\),
    it follows that
    \(\psi(x) E(x,y)\) vanishes to infinite order as either \(x\rightarrow \BoundaryN\) or \(y\rightarrow \BoundaryN\).
    We set
    \begin{equation*}
        \Extend^0 \psi(x) E(x,y)
        =\begin{cases}
            \psi(x) E(x,y),&x,y\in \InteriorN,\\
            0, &\text{otherwise.}
        \end{cases}
    \end{equation*}
    Clearly, 
    \(\Extend^0 \psi(x) E(x,y)\) is smooth and
    \(\supp(\Extend^0 \psi(x) E(x,y))\subseteq \Compact\times \Compact\subseteq \overline{\Omega_1}\times \overline{\Omega_1}\),
    and we have, \(\forall N\), \(\forall m\),
    \begin{equation*}
    \begin{split}
        &\sum_{|\alpha|,|\beta|\leq N} 
        \left| \left( \delta^{\Wdv} \Wh_x \right)^{\alpha} \left( \delta^{\Wdv} \Wh_y \right)^{\beta} \Extend^0 \psi(x)E(x,y) \right|
        \\&=
        \begin{cases}
            \sum_{|\alpha|,|\beta|\leq N} \left| \left( \delta^{\Wdv} W_x \right)^{\alpha} \left( \delta^{\Wdv} W_y \right)^{\beta} \psi(x)E(x,y) \right|, & x,y\in \InteriorN\cap \Compact,\\
            0, &\text{otherwise}
        \end{cases}
        \\&\lesssim
        \begin{cases}
            \frac{\left( 1+\delta^{-1}\MetricWWdv[x][y] \right)^{-m}}{\Vol[\BWWdv{x}{\delta+\MetricWWdv[x][y]}]\wedge 1}, & x,y\in \InteriorN\cap \Compact,\\
            0, &\text{otherwise}
        \end{cases}
        \\& \lesssim 
        \frac{\left( 1+\delta^{-1}\MetricWhWdv[x][y] \right)^{-m}}{\Vol[\BWhWdv{x}{\delta+\MetricWhWdv[x][y]}]\wedge 1},
    \end{split}
    \end{equation*}
    where the last estimate follows from Proposition \ref{Prop::VectorFields::Scaling::AmbientVolAndMetricEquivalnce} \ref{Item::VectorFields::Scaling::AmbientVolAndMetricEquivalnce::MetricWequalsMetricZ} and \ref{Item::VectorFields::Scaling::AmbientVolAndMetricEquivalnce::VolOfMetricWequalsVolOfMetricZ};
    which establishes that \(\Extend^0 \psi(x) E(x,y)\) is of the desired form.

    We turn to defining \(\Extend^0 \sum_{j=1}^M \phi_j (x) E(x,y) (1-\phit_j(y))\) and
    \(\Extend \sum_{j=1}^M \phi_j (x) E(x,y) (1-\phit_j(y))\) (where for \(\Extend^0\),
    \(\sE\in  \PElemz{\WWdv,\Omega\cap \ManifoldN,\Vol}{\Compact}\) and for \(\Extend\),
    \(\sE\in  \PElem{\WWdv,\Omega\cap \ManifoldN,\Vol}{\Compact}\)).
    Set \(F(x,y):=\sum_{j=1}^M \phi_j (x) E(x,y) (1-\phit_j(y))\).
    We set
    \begin{equation*}
        \Extend^0 F(x,y):=
        \begin{cases}
            \Extend^1 F(\cdot, y)\big|_{\cdot=x}, &y\in \ManifoldN,\\
            0, &\text{otherwise},
        \end{cases}
    \end{equation*}
    \begin{equation*}
        \Extend F(x,y):=
            \Extend^2 F(x, y)
    \end{equation*}
    where \(\Extend^1\) and \(\Extend^2\) are the maps from Lemma \ref{Lemma::Spaces::Elem::Extend::Classical},
    and \(\Extend^1\) is acting in the \(\cdot\) variable
    and \(\Omega_1\) playing the role of both \(\Omega_1\) and \(\Omega_2\) in that lemma.
    Since \(F(x,y)\) vanishes to infinite order as \(y\rightarrow \BoundaryN\), we have \(\Extend^0 F\) is smooth;
    and \(\Extend F(x,y)\) is smooth by definition.
    Clearly \(\supp(\Extend^1 F)\subseteq \Omega_1\times (\Omega_1\cap \ManifoldN)\)
    since
    \(\supp(F)\subseteq \Compact\times \Compact\) in that case,
    and \(\supp(\Extend^2 F)\subseteq \Omega_1\times \Omega_1\) directly from Lemma \ref{Lemma::Spaces::Elem::Extend::Classical}.
    As a result, we only need to verify the estimates \eqref{Eqn::Spaces::LP::PreElemBound}, with \(2^{-j}\) replaced by \(\delta\).

    Next, we claim if \((x,y)\in \supp(F)\), then \(\MetricWWdv[x][y]\gtrsim 1\).
    Indeed, suppose \((x,y)\in \supp(\phi_j(x) E(x,y)(1-\phit_j(y)))\); then \(x\) and \(y\)
    are in disjoint compact subsets of \(\Omega_1\cap \ManifoldN\), and by Proposition \ref{Prop::VectorFields::Scaling::SameTopology}
    this implies \(\MetricWWdv[x][y]\gtrsim 1\). Since \(F(x,y)\) is a finite sum of such terms, we see
    if \((x,y)\in \supp(F)\), then \(\MetricWWdv[x][y]\gtrsim 1\).
    As a result, we have, for every \(L\), with \(Q_2\) as in Proposition \ref{Prop::VectorFields::Scaling::VolEstimates} \ref{Item::VectorFields::Scaling::VolEstimates::VolPoly},
    \begin{equation*}
    \begin{split}
         &\sum_{|\alpha|,|\beta|\leq N}
         \left| \left( \delta^{\Wdv} W_x \right)^{\alpha} \left( \delta^{\Wdv} W_y \right)^\beta F(x,y) \right|
         \lesssim
         \frac{\left( 1+\delta^{-1} \MetricWWdv[x][y] \right)^{-L-Q_2}}{\Vol[\BWWdv{x}{\delta+\MetricWWdv[x][y]}]\wedge 1}
         \\& 
         \lesssim \frac{\delta^{L+Q_2}}{\Vol[\BWWdv{x}{\delta+\MetricWWdv[x][y]}]\wedge 1}
         \lesssim \delta^L,
    \end{split}
    \end{equation*}
    and therefore, using Notation \ref{Notation::Spaces::LP::MultiIndexAndDegree}, we have for every \(L,N\),
    by taking 
    \(M\) large depending on \(N\) (using H\"ormander's condition), and taking
    \(L'\) large depending on \(L\) and \(M\),
    \begin{equation*}
    \begin{split}
         &\CmNorm{F}{N}
         \leq \sup_{x,y\in \Compact} \sum_{|\alpha|,|\beta|\leq M} \left| W_x^{\alpha} W_y^\beta F(x,y) \right|
         \leq \sup_{x,y\in \Compact} \sum_{|\alpha|,|\beta|\leq M} \delta^{-\DegWdv{\alpha}-\DegWdv{\beta}} \left| \left(\delta^{\Wdv} W_x \right)^{\alpha} \left(\delta^{\Wdv} W_y \right)^\beta F(x,y) \right|
         \\&\lesssim  \sum_{|\alpha|,|\beta|\leq M} \delta^{-\DegWdv{\alpha}-\DegWdv{\beta}+L'}
         \lesssim \delta^L.
    \end{split}
    \end{equation*}

    We conclude, for every \(L\), \(N\), using Lemma \ref{Lemma::Spaces::Elem::Extend::Classical},
    % and taking \(Q_1\) as in Proposition \ref{Prop::VectorFields::Scaling::VolEstimates} \ref{Item::VectorFields::Scaling::VolEstimates::VolPoly}
    % when applied to \(\WhWd\),
    \begin{equation*}
    \begin{split}
         &\sum_{|\alpha|,|\beta|\leq N}
         \left| \left( \delta^{\Wdv} W_x \right)^{\alpha} \left( \delta^{\Wdv} W_y \right)^{\beta} \Extend^0 F(x,y) \right|
        \lesssim \sum_{|\beta|\leq N} \BCmNorm{\left( \delta^{\Wdv} W_y \right)^{\beta} \Extend^0 F(\cdot,y)  }{N}
        \\&\lesssim \sum_{|\beta|\leq N} \BCmNorm{\left( \delta^{\Wdv} W_y \right)^{\beta} F(\cdot,y)  }{N}
        \lesssim \BCmNorm{ F  }{N}
        \lesssim \delta^{L}
        \\&\lesssim \left( 1+\delta^{-1}\MetricWhWdv[x][y] \right)^{-L}
        \leq \frac{\left( 1+\delta^{-1}\MetricWhWdv[x][y] \right)^{-L}}{\Vol[\BWhWdv{x}{\delta+\MetricWhWdv[x][y]}]\wedge 1},
    \end{split}
    \end{equation*}
    and similarly,
    \begin{equation*}
    \begin{split}
         &\sum_{|\alpha|,|\beta|\leq N}
         \left| \left( \delta^{\Wdv} W_x \right)^{\alpha} \left( \delta^{\Wdv} W_y \right)^{\beta} \Extend F(x,y) \right|
         \lesssim \BCmNorm{ F  }{N}
         \lesssim \delta^{L}
         \lesssim \frac{\left( 1+\delta^{-1}\MetricWhWdv[x][y] \right)^{-L}}{\Vol[\BWhWdv{x}{\delta+\MetricWhWdv[x][y]}]\wedge 1}, 
    \end{split}
    \end{equation*}
    which are the desired estimates, completing the proof.
\end{proof}

With Lemma \ref{Lemma::Spaces::Elem::Extend::ReduceToUnitBall} in hand, we proceed by working on the unit ball
\(\Bn{1}\) endowed with Lebesgue measure, with the setting as in Section \ref{Section::Spaces::Scaling};
let \(\delta_1,\xi_1\in (0,1]\), and \(\Psi_{x,\delta}\) be as in the statement of Theorem \ref{Thm::Spaces::Scaling::MainScalingThm},
with \(r_1=5/8\) and \(r_2=3/4\).

\begin{proposition}\label{Prop::Spaces::Elem::Extend::NiceCovering}
    \(\exists \delta_2\in (0,1]\), \(\exists M\in \Zg\), \(\forall \delta\in (0,\delta_2]\),
    there exists finite collections \(x_1',\ldots, x_{L_1}'\in \BnmoClosure{5/8}\),
    \(y_1,\ldots, y_{L_2}\in \Bng{5/8}\), \(z_1,\ldots, z_{L_3}\in \Bnl{5/8}\) such that
    \(\forall \delta \in (0,\delta_2]\),
    \begin{enumerate}[(a)]
        \item\label{Item::Spaces::Elem::Extend::NiceCovering::Covering} We have
            \begin{equation*}
                \begin{split}
                &\left\{ \BWhWdv{(x_j',0)}{\xi_1\delta/4} : j=1,\ldots,L_1 \right\}
                \\&\bigcup \left\{ \BWhWdv{y_j}{\xi_1^2\delta/64} : j=1,\ldots, L_2 \right\}  
                \\&\bigcup \left\{ \BWhWdv{z_j}{\xi_1^2\delta/64} : j=1,\ldots, L_3 \right\}  
                \end{split}
            \end{equation*}
            is a cover for \(\BnClosure{9/16}\).
        \item\label{Item::Spaces::Elem::Extend::NiceCovering::SmallOverlap} No point lies in more than \(M\) of the balls
        \begin{equation*}
            \begin{split}
            &\left\{ \BWhWdv{(x_j',0)}{\delta} : j=1,\ldots,L_1 \right\}
            \\&\bigcup \left\{ \BWhWdv{y_j}{\xi_1\delta/16} : j=1,\ldots, L_2 \right\}  
            \\&\bigcup \left\{ \BWhWdv{z_j}{\xi_1\delta/16} : j=1,\ldots, L_3 \right\}  
            \end{split}
        \end{equation*}
        \item\label{Item::Spaces::Elem::Extend::NiceCovering::InB34} \(\BWhWdv{y_j}{\xi_1\delta/16}\subseteq \Bng{3/4}\),
            \(\BWhWdv{z_j}{\xi_1\delta/16}\subseteq \Bnl{3/4}\),
            \(\BWhWdv{(x_j',0)}{\xi_1\delta/4}\subseteq \Bn{3/4}\).
    \end{enumerate}
\end{proposition}

To prove Proposition \ref{Prop::Spaces::Elem::Extend::NiceCovering} we require a lemma;
the proof of which is standard (see, for example, \cite[page 32]{SteinHarmonicAnalysisRealVariableMethodsOrthogonalityOscillatoryIntegrals}).

\begin{lemma}\label{Lemma::Spaces::Elem::Extend::DisjointOverlap}
    Let \(K>1\). Then, there exists \(M=M(K)\in \Zg\) such that if 
    \(\delta\in (0,\delta_1/2^{\ceil{\log_2(K)}+1}]\) and \(\left\{ \BWhWdv{x_j}{\delta} \right\}\) is a disjoint
    collection of balls with \(x_j\in \BnClosure{5/8}\), then no point lies in more than
    \(M\) of the balls \(\left\{ \BWhWdv{x_j}{K\delta} \right\}\).
\end{lemma}
\begin{proof}
    Suppose \(y\in \BWhWdv{x_{j_l}}{K\delta}\), \(l=1,\ldots, M\); we wish to show \(M\lesssim 1\) (where the implicit constant depends on \(K\)).
    Since \(y\in \BWhWdv{y}{K\delta}\cap \BWhWdv{x_{j_l}}{K\delta}\),
    it follows that 
    \begin{equation}\label{Eqn::Spaces::Elem::Extend::DisjointOverlap::Tmp1}
        \BWhWdv{x_{j_l}}{\delta}\subseteq \BWhWdv{x_{j_l}}{K\delta}\subseteq \BWhWdv{y}{2K\delta}
    \end{equation}
    and similarly,
    \begin{equation}\label{Eqn::Spaces::Elem::Extend::DisjointOverlap::Tmp2}
        \BWhWdv{y}{\delta}\subseteq \BWhWdv{x_{j_l}}{2K\delta}.
    \end{equation}
    By \eqref{Eqn::Spaces::Elem::Extend::DisjointOverlap::Tmp1} and \eqref{Eqn::Spaces::Elem::Extend::DisjointOverlap::Tmp2}
    we have
    \begin{equation}\label{Eqn::Spaces::Elem::Extend::DisjointOverlap::Tmp2.5}
        \Vol[\BWWdv{x_{j_l}}{\delta}]\leq \Vol[\BWhWdv{y}{2K\delta}], \quad \Vol[\BWWdv{y}{\delta}]\leq \Vol[\BWhWdv{x_{j_l}}{2K\delta}].
    \end{equation}
    Using repeated applications of Theorem \ref{Thm::Spaces::Scaling::MainScalingThm} \ref{Item::Spaces::Scaling::Doubling},
    we have
    \begin{equation}\label{Eqn::Spaces::Elem::Extend::DisjointOverlap::Tmp3}
        \Vol[\BWhWdv{y}{2K\delta}]\leq \Vol[\BWhWdv{y}{2^{\ceil{\log_2(K)}+1}\delta}]
        \lesssim \Vol[\BWhWdv{y}{\delta}],
    \end{equation}
    and similarly
    \begin{equation}\label{Eqn::Spaces::Elem::Extend::DisjointOverlap::Tmp4}
        \Vol[\BWhWdv{x_{j_l}}{2K\delta}]
        \lesssim \Vol[\BWhWdv{x_{j_l}}{\delta}].
    \end{equation}
    Combining \eqref{Eqn::Spaces::Elem::Extend::DisjointOverlap::Tmp2.5}, \eqref{Eqn::Spaces::Elem::Extend::DisjointOverlap::Tmp3},
    and \eqref{Eqn::Spaces::Elem::Extend::DisjointOverlap::Tmp4}, we see
    \begin{equation}\label{Eqn::Spaces::Elem::Extend::DisjointOverlap::Tmp5}
        \Vol[\BWhWdv{y}{\delta}]
        \approx \Vol[\BWhWdv{x_{j_l}}{\delta}].
    \end{equation}
    Using \eqref{Eqn::Spaces::Elem::Extend::DisjointOverlap::Tmp1}, \eqref{Eqn::Spaces::Elem::Extend::DisjointOverlap::Tmp3},
    \eqref{Eqn::Spaces::Elem::Extend::DisjointOverlap::Tmp5}, and the fact that the balls \(\BWhWdv{x_{j_l}}{\delta}\)
    are disjoint, we have
    \begin{equation}\label{Eqn::Spaces::Elem::Extend::DisjointOverlap::Tmp6}
    \begin{split}
         &\Vol[\BWWdv{y}{\delta}] \gtrsim \Vol[\BWhWdv{y}{2K\delta}] \geq \sum_{j=1}^M \Vol[\BWhWdv{x_{j_l}}{\delta}] 
         \\&\approx \sum_{j=1}^M \Vol[\BWhWdv{y}{\delta}] = M\Vol[\BWhWdv{y}{\delta}].
    \end{split}
    \end{equation}
    Using that \(\Vol[\BWhWdv{y}{\delta}]\ne 0\) 
    (see Proposition \ref{Prop::VectorFields::Scaling::VolEstimates} \ref{Item::VectorFields::Scaling::VolEstimates::VolApproxLambda}),
    dividing both sides of \eqref{Eqn::Spaces::Elem::Extend::DisjointOverlap::Tmp6} by \(\Vol[\BWhWdv{y}{\delta}]\)
    shows \(M\lesssim 1\), completing the proof.
\end{proof}

\begin{proof}[Proof of Proposition \ref{Prop::Spaces::Elem::Extend::NiceCovering}]
    This proof is based on standard covering techniques like those from the Vitali covering lemma
    (see \cite[Chapter I, Section 3]{SteinHarmonicAnalysisRealVariableMethodsOrthogonalityOscillatoryIntegrals}).
    Let \(K=16 \xi_1^{-1}\). We take \(\delta_2\leq \delta/2^{\ceil{\log_2(K)}+1}\) so small that
    for \(w\in \BnClosure{{9/16}}\)
    and \(\delta\in(0,\delta_2]\)
    we have \(\BWhWdv{w}{\delta}\subseteq \Bn{5/8}\)
    (this is always possible--apply Theorem \ref{Thm::Spaces::Scaling::MainScalingThm} \ref{Item::Spaces::Scaling::ContainedInAmbient}
    with \(r_1=9/16\), \(r_2=5/8\)).

    Fix \(\delta\in (0,\delta_2]\).
    From the collection
    \begin{equation*}
        \left\{ \BWhWdv{(y',0)}{\xi_1\delta/16} : y'\in \BnmoClosure{5/8} \right\}
    \end{equation*}
    pick a maximal disjoint subcollection:
    \begin{equation}\label{Eqn::Spaces::Elem::Extend::NiceCovering::XjDisjoint}
        \BWhWdv{(x_j',0)}{\xi_1\delta/16}, \quad x_1',\ldots, x_{L_1}'\in \BnmoClosure{5/8}.
    \end{equation}
    Since \(\delta\leq \delta_1\), it follows from Theorem \ref{Thm::Spaces::Scaling::MainScalingThm} \ref{Item::Spaces::Scaling::ContainedInAmbient}
    that \(\BWhWdv{(x_j',0)}{\delta}\subseteq \Bn{3/4}\), establishing part of \ref{Item::Spaces::Elem::Extend::NiceCovering::InB34}.

    For a set \(S\) and a point \(y\), we let \(\DistWWdv[y][S]:=\inf\left\{ x\in S : \DistWWdv[x][y] \right\}\).
    We claim
    \begin{equation}\label{Eqn::Spaces::Elem::Extend::NiceCovering::CoversNearBoundary}
        \left\{ y\in \Bn{1}:\DistWWdv[y][\BnmoClosure{5/8}] < \xi_1\delta/16 \right\}
        \subseteq \bigcup_{j=1}^{L_1} \BWhWdv{(x_j,0)}{\xi_1\delta/4}.
    \end{equation}
    Indeed, suppose \(\DistWWdv[y][\BnmoClosure{5/8}] < \xi_1\delta/16\). Take \(z'\in \BnmoClosure{5/8}\)
    with \(y\in \BWhWdv{(z',0)}{\xi_1\delta/16}\). By maximality,
    \begin{equation*}
        \BWhWdv{(z',0)}{\xi\delta/16}\cap \BWhWdv{(x_j',0)}{\xi_1\delta/16}\ne \emptyset.
    \end{equation*}
    for some \(j\). Then, we have \(y\in \BWhWdv{(z',0)}{\xi_1\delta/16}\subseteq \BWhWdv{(x_j',0)}{\xi_1\delta/4}\),
    establishing \eqref{Eqn::Spaces::Elem::Extend::NiceCovering::CoversNearBoundary}.

    From the collection
    \begin{equation*}
        \left\{ \BWhWdv{w}{\xi_1^2\delta/256} : y\in \BnClosure{9/16}, \DistWWdv[y][\BnmoClosure{5/8}] \geq \xi_1\delta/16 \right\},
    \end{equation*}
    pick a maximal, disjoint subcollection
    \begin{equation}\label{Eqn::Spaces::Elem::Extend::NiceCovering::WjDisjoint}
        \BWhWdv{w_j}{\xi_1^2\delta/256},\quad j=1,\ldots, L_4.
    \end{equation}
    By maximality,
    \begin{equation}\label{Eqn::Spaces::Elem::Extend::NiceCovering::CoversAwayBoundary}
        \left\{ y\in \BnClosure{9/16} : \DistWWdv[y][\Bnmo{5/8}]\geq \xi_1\delta/16 \right\}
        \subseteq \bigcup_{j=1}^{L_4}\BWhWdv{w_j}{\xi_1^2 \delta/64}.
    \end{equation}
    Indeed, if \(y\in \BnClosure{9/16}\) with \(\DistWWdv[y][\Bnmo{5/8}]\geq \xi_1\delta/16\)
    were not contained in \(\bigcup_{j=1}^{L_4}\BWhWdv{w_j}{\xi_1^2 \delta/64}\), then \(\BWhWdv{y}{\xi_1^2\delta/256}\)
    would be disjoint from the collection \eqref{Eqn::Spaces::Elem::Extend::NiceCovering::WjDisjoint}, contradicting maximality.
    
    By the choice of \(\delta_2\), we have \(\BWhWdv{w_j}{\xi_1\delta/16}\subseteq \Bn{5/8}\) and
    by the choice of \(w_j\), \(\BWhWdv{w_j}{\xi_1\delta/16}\cap \BnmoClosure{5/8}=\emptyset\).
    We conclude \(\BWhWdv{w_j}{\xi_1\delta/16}\cap \Bnmo{1}=\emptyset\).
    Directly from the definition, \(\BWhWdv{w_j}{\xi_1\delta/16}\) is path connected, and we conclude
    that for each \(j\), 
    \begin{equation*}
        \BWhWdv{w_j}{\xi_1\delta/16}\subseteq \Bng{5/8}\text{ or }\BWhWdv{w_j}{\xi_1\delta/16}\subseteq \Bnl{5/8}.
    \end{equation*}
    Let \(y_1,\ldots, y_{L_2}\) be an enumeration of those \(w_j\) with \(\BWhWdv{w_j}{\xi_1\delta/16}\subseteq \Bng{5/8}\)
    and \(z_1,\ldots, z_{l_3}\) be an enumeration of those \(w_j\) with \(\BWhWdv{w_j}{\xi_1\delta/16}\subseteq \Bnl{5/8}\).
    The remainder of \ref{Item::Spaces::Elem::Extend::NiceCovering::InB34} follows immediately.
    \ref{Item::Spaces::Elem::Extend::NiceCovering::Covering} follows from
    \eqref{Eqn::Spaces::Elem::Extend::NiceCovering::CoversNearBoundary} and \eqref{Eqn::Spaces::Elem::Extend::NiceCovering::CoversAwayBoundary}.

    Lemma \ref{Lemma::Spaces::Elem::Extend::DisjointOverlap} applied to the disjoint collection
    \eqref{Eqn::Spaces::Elem::Extend::NiceCovering::XjDisjoint} (with \(K=16 \xi_1^{-1}\))
    shows that no point lies in more than \(M_1\lesssim 1\) of the balls
    \(\BWhWdv{(x_j',0)}{\delta}\). Similarly, Lemma \ref{Lemma::Spaces::Elem::Extend::DisjointOverlap} applied to the disjoint collection
    \eqref{Eqn::Spaces::Elem::Extend::NiceCovering::WjDisjoint}
    shows that no point lies in more than \(M_2\lesssim 1\) of the balls
    \(\BWhWdv{w_j}{\xi_1\delta/16}\). Taking \(M=M_1+M_2\), \ref{Item::Spaces::Elem::Extend::NiceCovering::SmallOverlap} follows.
\end{proof}

\begin{proposition}\label{Prop::Spaces::Elem::Extend::ParitionOfUnity}
    \(\exists \delta_2\in (0,1]\), \(\exists M\in \Zg\),
    \(\forall \delta\in (0,\delta_2]\), 
    \(\exists \psi_\delta\in \CinftyCptSpace[\Bng{3/4}]\), \(\phi_{1,\delta},\ldots, \phi_{L,\delta}\in \CinftyCptSpace[\Bn{3/4}]\),
    and \(x_1',\ldots, x_{L}'\in \BnmoClosure{5/8}\),
    where \(L=L(\delta)\) depends on \(\delta\), such that:
    \begin{enumerate}[(i)]
        \item\label{Item::Spaces::Elem::Extend::ParitionOfUnity::phiSupport} \(\phi_{j,\delta}\in \CinftyCptSpace[\Psi_{(x_j',0),\delta}(\Qn{1/2})]\),
        \item\label{Item::Spaces::Elem::Extend::ParitionOfUnity::SumTo1} \(\psi_\delta+\sum_{j=1}^L \phi_{j,\delta}=1\) on a neighborhood of \(\BngeqClosure{1/2}\),
        \item\label{Item::Spaces::Elem::Extend::ParitionOfUnity::DerivEstimates} \(\forall N\),
            \begin{equation*}
                \sup_{x} \sup_{\delta\in (0,\delta_2]} \sum_{|\alpha|\leq N}
                \left| \left( \delta^{\Wdv} W \right)^{\alpha} \psi_\delta(x) \right|<\infty,
            \end{equation*}
            \begin{equation*}
                \sup_{x} \sup_{\delta\in (0,\delta_2]}\max_{j} \sum_{|\alpha|\leq N}
                \left| \left( \delta^{\Wdv} W \right)^{\alpha} \phi_{j,\delta}(x) \right|<\infty.
            \end{equation*}
        % \item \(\forall x\), \(x\) is in the support of at most \(M\) of the \(\phi_{j,\delta}\).
        % \item\label{Item::Spaces::Elem::Extend::ParitionOfUnity::phiDontOverlap}  No point lies in the support of more than \(M\) of the \(\phi_{j,\delta}\).
        \item\label{Item::Spaces::Elem::Extend::PartitionOfUnity::phiDontOverlapBig} No point lies in more than \(M\)
            of the sets \(\Psi_{(x_j',0),\delta}(\Qn{1})\).
    \end{enumerate}
\end{proposition}
\begin{proof}
    We take \(\delta_2\) as in Proposition \ref{Prop::Spaces::Elem::Extend::NiceCovering},
    and for \(\delta\in (0,\delta_2]\) apply Proposition \ref{Prop::Spaces::Elem::Extend::NiceCovering}
    to obtain 
    \(x_1',\ldots, x_L'\in \BnmoClosure{5/8}\) and \(y_1,\ldots, y_{L_2}\in \Bng{5/8}\)
    as in that result (here we are renaming \(L_1\) from Proposition \ref{Prop::Spaces::Elem::Extend::NiceCovering} to be \(L\)).
    Fix \(\phi\in \CinftyCptSpace[\Qn{1/2}]\) with \(\phi=1\) on a neighborhood of \(\QnClosure{1/4}\) and \(0\leq \phi\leq 1\);
    \(\phi\) is a single fixed function, and does not depend on \(\delta\).

    Set \(\psih_{j,\delta}:=\phi\circ \Psi_{y_j, \xi_1\delta/16}^{-1}\)
    and \(\phih_{j,\delta}:=\phi\circ \Psi_{(x_j',0),\delta}^{-1}\),
    so that by Theorem \ref{Thm::Spaces::Scaling::MainScalingThm} \ref{Item::Spaces::Scaling::ImageContainsBall},
    \begin{equation}\label{Eqn::Spaces::Elem::Extend::ParitionOfUnity::psih1}
        \psih_{j,\delta}=1\text{ on }\BWhWdv{y_j}{\xi_1^2\delta/16},
    \end{equation}
    \begin{equation}\label{Eqn::Spaces::Elem::Extend::ParitionOfUnity::phih1}
        \phih_{j,\delta}=1\text{ on }\BWhWdv{(x_j',0)}{\xi_1\delta},
    \end{equation}
    \begin{equation}\label{Eqn::Spaces::Elem::Extend::ParitionOfUnity::psihAndphihSupport}
        \supp(\psih_{j,\delta})\subseteq \BWhWdv{y_j}{\xi_1\delta/16}, \quad \supp(\phih_{j,\delta})\subseteq \BWhWdv{(x_j',0)}{\delta}.
    \end{equation}

    We have, by Theorem \ref{Thm::Spaces::Scaling::MainScalingThm} \ref{Item::Spaces::Scaling::PullBackSmooth}, \(\forall N\),
    \begin{equation}\label{Eqn::Spaces::Elem::Extend::ParitionOfUnity::Boundpsih}
    \begin{split}
            &\sup_{x}\sum_{|\alpha|\leq N} \left| \left( \delta^{\Wdv} W \right)^{\alpha} \psih_{j,\delta}(x) \right|
            \approx \sup_{x}\sum_{|\alpha|\leq N} \left| \left( \left( \xi_1\delta/16 \right)^{\Wdv} W \right)^{\alpha} \psih_{j,\delta}(x) \right|
            \\&=\sup_{u} \sum_{|\alpha|\leq N} \left| \left( W^{y_j, \xi_1\delta/16}  \right)^{\alpha} \phi(u) \right|
            \lesssim 1,
    \end{split}    
    \end{equation}
    and similarly,
    \begin{equation}\label{Eqn::Spaces::Elem::Extend::ParitionOfUnity::Boundphih}
    \begin{split}
         &\sup_{x}\sum_{|\alpha|\leq N} \left| \left( \delta^{\Wdv} W \right)^{\alpha} \phih_{j,\delta}(x) \right|
         =\sup_{u} \sum_{|\alpha|\leq N} \left| \left( W^{(x_j',0),\delta} \right)^{\alpha}\phi(u) \right|\lesssim 1.
    \end{split}
    \end{equation}

    By \eqref{Eqn::Spaces::Elem::Extend::ParitionOfUnity::psihAndphihSupport} and Proposition \ref{Prop::Spaces::Elem::Extend::NiceCovering} \ref{Item::Spaces::Elem::Extend::NiceCovering::SmallOverlap},
    \(0\leq \sum_{j=1}^{L} \phih_{j,\delta}+\sum_{j=1}^{L_2}\psih_{j,\delta}\leq M\),
    and by \eqref{Eqn::Spaces::Elem::Extend::ParitionOfUnity::psih1}, \eqref{Eqn::Spaces::Elem::Extend::ParitionOfUnity::phih1},
    and Proposition \ref{Prop::Spaces::Elem::Extend::NiceCovering} \ref{Item::Spaces::Elem::Extend::NiceCovering::Covering},
    \(\sum_{j=1}^{L} \phih_{j,\delta}+\sum_{j=1}^{L_2}\psih_{j,\delta}\geq 1\) on \(\BnClosure{9/16}\).
    
    Fix \(\Phi\in \CinftyCptSpace[\Bn{9/16}]\) with \(\Phi=1\) on a neighborhood of \(\BnClosure{1/2}\).
    Here, \(\Phi\) is a fixed function which does not depend on \(\delta\).
    Set,
    \begin{equation}\label{Eqn::Spaces::Elem::Extend::ParitionOfUnity::DefinephiAndpsi}
        \psi_{k,\delta}:=\frac{\Phi \psih_{k,\delta}}{\sum_{j=1}^{L} \phih_{j,\delta}+\sum_{j=1}^{L_2}\psih_{j,\delta}},
        \quad \phi_{k,\delta}:=\frac{\Phi \phih_{k,\delta}}{\sum_{j=1}^{L} \phih_{j,\delta}+\sum_{j=1}^{L_2}\psih_{j,\delta}}.
    \end{equation}
    Combining the above comments with \eqref{Eqn::Spaces::Elem::Extend::ParitionOfUnity::Boundpsih} and \eqref{Eqn::Spaces::Elem::Extend::ParitionOfUnity::Boundphih},
    we see
    \begin{equation}\label{Eqn::Spaces::Elem::Extend::ParitionOfUnity::Boundpsi}
        \sup_{x} \sum_{|\alpha|\leq N} \left|  \left( \delta^{\Wdv} W \right)^{\alpha}\psi_{j,\delta}(x) \right|\lesssim 1,
    \end{equation}
    \begin{equation}\label{Eqn::Spaces::Elem::Extend::ParitionOfUnity::Boundphi}
        \sup_{x} \sum_{|\alpha|\leq N} \left|  \left( \delta^{\Wdv} W \right)^{\alpha}\phi_{j,\delta}(x) \right|\lesssim 1.
    \end{equation}
    Setting \(\psi_\delta:=\sum_{j=1}^{L_2}\psi_{j,\delta}\), and using the fact that no point
    is in the support of more than \(M\lesssim 1\) of the \(\psi_{j,\delta}\) (by \eqref{Eqn::Spaces::Elem::Extend::ParitionOfUnity::psihAndphihSupport}
    and Proposition \ref{Prop::Spaces::Elem::Extend::NiceCovering} \ref{Item::Spaces::Elem::Extend::NiceCovering::SmallOverlap}),
    we have 
    \begin{equation}\label{Eqn::Spaces::Elem::Extend::ParitionOfUnity::BoundpsiSum}
        \sup_{x} \sum_{|\alpha|\leq N} \left|  \left( \delta^{\Wdv} W \right)^{\alpha}\psi_{\delta}(x) \right|\lesssim 1.
    \end{equation}
    \ref{Item::Spaces::Elem::Extend::ParitionOfUnity::DerivEstimates}
    follows from \eqref{Eqn::Spaces::Elem::Extend::ParitionOfUnity::BoundpsiSum} and  \eqref{Eqn::Spaces::Elem::Extend::ParitionOfUnity::Boundphi}.
    We have, from the definitions, \eqref{Eqn::Spaces::Elem::Extend::ParitionOfUnity::psihAndphihSupport}, and 
    Proposition \ref{Prop::Spaces::Elem::Extend::NiceCovering} \ref{Item::Spaces::Elem::Extend::NiceCovering::InB34},
    \begin{equation*}
        \supp(\psi_{j,\delta})\subseteq \supp(\psih_{j,\delta})\Subset \Bng{3/4},
    \end{equation*}
    and therefore \(\psi_\delta\in \CinftyCptSpace[\Bng{3/4}]\).
    We have, from the definition,
    \begin{equation*}
        %\label{Eqn::Spaces::Elem::Extend::ParitionOfUnity::Tmp0}
        \supp(\phi_{j,\delta})\subseteq \supp(\phih_{j,\delta})\Subset \Psi_{(x_j',0),\delta}(\Qn{1/2}),
    \end{equation*}
    establishing \ref{Item::Spaces::Elem::Extend::ParitionOfUnity::phiSupport}.
    We have from \eqref{Eqn::Spaces::Elem::Extend::ParitionOfUnity::DefinephiAndpsi}
    \begin{equation*}
        \psi_\delta + \sum_{j=1}^L \phi_{j,\delta} = \Phi,
    \end{equation*}
    and \ref{Item::Spaces::Elem::Extend::ParitionOfUnity::SumTo1} follows from the choice of \(\Phi\).

    Finally, we have (see Theorem \ref{Thm::Spaces::Scaling::MainScalingThm}),
    \begin{equation*}
        \Psi_{(x_j',0),\delta}(\Qn{1})\subseteq \BWhWdv{(x_j',0)}{\delta},
    \end{equation*}
    and \ref{Item::Spaces::Elem::Extend::PartitionOfUnity::phiDontOverlapBig} follows from 
    Proposition \ref{Prop::Spaces::Elem::Extend::NiceCovering} \ref{Item::Spaces::Elem::Extend::NiceCovering::SmallOverlap}.
\end{proof}

\begin{lemma}\label{Lemma::Spaces::Elem::Extend::PreElemBoundUnderSmallPerterbations}
    For \(x_1,x_2,y_1,y_2\in \Bn{7/8}\) and \(\delta\in (0,1]\), 
    if \(\MetricWhWdv[x_1][x_2]<2\delta\)
    and \(\MetricWhWdv[y_1][y_2]<2\delta\), then
    \begin{equation}\label{Eqn::Spaces::Elem::Extend::PreElemBoundUnderSmallPerterbations::Top}
        1+\delta^{-1}\MetricWhWdv[x_1][y_1]\approx 1+\delta^{-1}\MetricWhWdv[x_2][y_2],
    \end{equation}
    \begin{equation}\label{Eqn::Spaces::Elem::Extend::PreElemBoundUnderSmallPerterbations::Bottom}
        \Vol[\BWhWdv{x_1}{\delta+\MetricWhWdv[x_1][y_1]}]\wedge 1 \approx \Vol[\BWhWdv{x_2}{\delta +\MetricWhWdv[x_2][y_2]}].
    \end{equation}
\end{lemma}
\begin{proof}
    \eqref{Eqn::Spaces::Elem::Extend::PreElemBoundUnderSmallPerterbations::Top} is true for any metric, and only uses
    the triangle inequality.
    Using \eqref{Eqn::Spaces::Elem::Extend::PreElemBoundUnderSmallPerterbations::Top}, we have
    \(\delta+\MetricWhWdv[x_1][y_1] \approx \delta+\MetricWhWdv[x_2][y_2]\),
    and \eqref{Eqn::Spaces::Elem::Extend::PreElemBoundUnderSmallPerterbations::Bottom} follows from
    Proposition \ref{Prop::VectorFields::Scaling::VolEstimates} \ref{Item::VectorFields::Scaling::VolEstimates::VolWedge1Doubling},
    with \(\Compact=\BnClosure{7/8}\).
\end{proof}

\begin{proof}[Proof of Proposition \ref{Prop::Spaces::Elem::Extend::Extend::Verbose} for \(\Extend_\delta^0\)]
    In light of Lemma \ref{Lemma::Spaces::Elem::Extend::ReduceToUnitBall}, we may work on the unit ball,
    with \(\Vol\) given by Lebesgue measure.
    
    Let 
    \begin{equation*}
        \Extend^1:\left\{ f\in \CinftySpace[\Qngeq{1}] : \supp(f)\subseteq \QngeqClosure{1/2} \right\}\rightarrow \CinftyCptSpace[\Qn{3/4}]
    \end{equation*}
    be a classical extension operator as in Lemma \ref{Lemma::Spaces::Elem::Extend::Classical}.
    Let \(\delta_2\), \(\phi_{j,\delta}\), \(j=1,\ldots, L=L(\delta)\), and \(\psi_{\delta}\) be as in that Proposition \ref{Prop::Spaces::Elem::Extend::ParitionOfUnity}.
    For \(E(x,y)\in \CinftyCptSpace[\Bn{1}]\) with \(\supp(E)\subseteq \BngeqClosure{1/2}\times \left( \BngeqClosure{1/2}\cap \Bng{1} \right)\),
    and \(\delta\in (0,\delta_2]\), set
    \begin{equation}\label{Eqn::Spaces::Elem::Extend::Extend::DefineEdelta0}
        \Extend_\delta^0 E(x,y)
        :=
        \begin{cases}
            \sum_{j=1}^L \left( \Psi_{x,\delta} \right)_{*} \Extend^1 \Psi_{x,\delta}^{*} \phi_{j,\delta}(\cdot) E(\cdot,y)\big|_{\cdot=x} + \psi_{\delta}(x) E(x,y), &y\in \Bngeq{1},\\
            0,&\text{otherwise.}
        \end{cases}
    \end{equation}
    Because \(E(x,y)\) vanishes to infinite order as \(y\rightarrow \Bnmo{1}\) (when considering \(\Extend_\delta^0\)), 
    \(\Extend_\delta^0 E(x,y)\) is smooth.
    We have, 
    \begin{equation*}
        \begin{split}
            &\supp(\Extend_0^\delta E(x,y ))
            \subseteq\left( \bigcup_{j=1}^L \left(  \Psi_{x,\delta}(\Qn{1})\times \left( \BngeqClosure{1/2}\cap \Bng{1} \right) \right) \right)
            \bigcup
            \left( \BngeqClosure{1/2}\times \BngeqClosure{1/2} \right)
            \\&\subseteq \Bn{3/4}\times \Bngeq{3/4}.
        \end{split}
    \end{equation*}
    so \(\supp(\Extend^0_\delta E)\in \CinftyCptSpace[\Bn{3/4}\times \Bngeq{3/4}])\).
    For \(\delta\in (\delta_2,1]\), set \(\Extend_\delta^0:=\Extend_{\delta_2}^0\); the estimates for
    \(\delta\in (\delta_2,1]\) follow from those for \(\delta=\delta_2\), so we may assume \(\delta\in (0,\delta_2]\)
    for the rest of the proof.

    To complete the proof, we wish to show that if \(\sE\in \PElemz{\WWdv, \Bngeq{7/8},\Vol}{\Bngeq{1/2}}\),
    then for \((E,\delta)\in \sE\), 
    \begin{equation}\label{Eqn::Spaces::Elem::Extend::Extend::ToShowPElemEst::Edelta0}
        \sum_{|\alpha|,|\beta|\leq N} \left| \left( \delta^{\Wdv}\Wh_x^\alpha \right) \left( \delta^{\Wdv} \Wh_x^{\beta} \right)
        \Extend_0^\delta E(x,y )
        \right|
        \lesssim
        \frac{
            \left( 1+\delta^{-1} \MetricWhWdv[x][y] \right)^{-m}
        }{
            \Vol[\BWhWdv{x}{\delta+\MetricWhWdv[x][y]}\wedge 1]
        },
    \end{equation}
    \(\forall N\), \(\forall m\).
    In light of Proposition \ref{Prop::Spaces::Elem::Extend::ParitionOfUnity} \ref{Item::Spaces::Elem::Extend::PartitionOfUnity::phiDontOverlapBig}
    each \((x,y)\) is in the support of at most
    \(M+1\) of the terms on the right-hand side of \eqref{Eqn::Spaces::Elem::Extend::Extend::DefineEdelta0},
    where \(M\lesssim 1\).
    Thus, it suffices to prove \eqref{Eqn::Spaces::Elem::Extend::Extend::ToShowPElemEst::Edelta0}
    for each term separately.

   In \eqref{Eqn::Spaces::Elem::Extend::Extend::ToShowPElemEst::Edelta0}, we may assume
    \(y\in \BngeqClosure{1/2}\), as otherwise \((x,y)\) is not in the support of \(\Extend_0^\delta E(x,y)\).
    We begin with the term \(\psi_{\delta}(x) E(x,y)\), which is supported on \(\BngeqClosure{1/2}\times \BngeqClosure{1/2}\),
    and so we may assume \((x,y)\in \BngeqClosure{1/2}\times \BngeqClosure{1/2}\), where \(\Wh_j=W_j\).
    We have, using Proposition \ref{Prop::Spaces::Elem::Extend::ParitionOfUnity} \ref{Item::Spaces::Elem::Extend::ParitionOfUnity::DerivEstimates},
    \begin{equation*}
        \begin{split}
            &\sum_{|\alpha|,|\beta|\leq N} \left| \left( \delta^{\Wdv}\Wh_x \right)^{\alpha} \left( \delta^{\Wdv} \Wh_y \right)^{\beta} \psi_\delta(y) E(x,y) \right|
            \lesssim \sum_{|\alpha|,|\beta|\leq N} \left| \left( \delta^{\Wdv}\Wh_x \right)^{\alpha} \left( \delta^{\Wdv} \Wh_y \right)^{\beta}  E(x,y) \right|
            \\&=\sum_{|\alpha|,|\beta|\leq N} \left| \left( \delta^{\Wdv}W_x \right)^{\alpha} \left( \delta^{\Wdv} W_y \right)^{\beta}  E(x,y) \right|
            \lesssim 
            \frac{
            \left( 1+\delta^{-1} \MetricWWdv[x][y] \right)^{-m}
        }{
            \Vol[\BWWdv{x}{\delta+\MetricWWdv[x][y]}\wedge 1]
        }
        \\&\approx
        \frac{
            \left( 1+\delta^{-1} \MetricWhWdv[x][y] \right)^{-m}
        }{
            \Vol[\BWhWdv{x}{\delta+\MetricWhWdv[x][y]}\wedge 1]
        },
        \end{split}
       \end{equation*}
    Where the second to last estimate uses the definition of \(\PElemz{\WhWdv, \Bngeq{7/8},\Vol}{\Bngeq{1/2}}\),
    and the last estimate uses Proposition \ref{Prop::VectorFields::Scaling::AmbientVolAndMetricEquivalnce}
    \ref{Item::VectorFields::Scaling::AmbientVolAndMetricEquivalnce::MetricWequalsMetricZ}
    and
    \ref{Item::VectorFields::Scaling::AmbientVolAndMetricEquivalnce::VolOfMetricWequalsVolOfMetricZ};
    establishing \eqref{Eqn::Spaces::Elem::Extend::Extend::ToShowPElemEst::Edelta0} for this term.

    Finally, we prove the estimate \eqref{Eqn::Spaces::Elem::Extend::Extend::ToShowPElemEst::Edelta0} for the term
    \(\left( \Psi_{x,\delta} \right)_{*} \Extend^1 \Psi_{x,\delta}^{*} \phi_{j,\delta}(\cdot) E(\cdot,y)\big|_{\cdot=x}\).
    This function is zero if \(x\not \in \Psi_{(x_j',0),\delta}(\Qn{1})\),
    so we assume \(x\in \Psi_{(x_j',0),\delta}(\Qn{1})\subseteq \BWhWdv{(x_j',0)}{\delta}\).
    As described above, we may assume \(y\in \Bngeq{1/2}\), where \(\Wh_j=W_j\), and so we may
    replace \(\Wh_y\) with \(W_y\) whenever it appears.
    We have, with \(\Wh^{(x_j',0),\delta}\) as in Theorem \ref{Thm::Spaces::Scaling::MainScalingThm},
    \begin{equation}\label{Eqn::Spaces::Elem::Extend::Extend::Extend0::Tmp1}
    \begin{split}
         &\sum_{|\alpha|,|\beta|\leq N} \left| \left( \delta^{\Wdv} \Wh_x \right)^{\alpha} \left( \delta^{\Wdv} \Wh_y \right)^{\beta} \left( \Psi_{x,\delta} \right)_{*} \Extend^1 \Psi_{x,\delta}^{*} \phi_{j,\delta}(\cdot) E(\cdot,y)\big|_{\cdot=x} \right|
         \\&\leq \sup_{u}\sum_{|\alpha|,|\beta|\leq N} \left| \left( \Wh_u^{(x_j',\delta)} \right)^{\alpha} \Extend^1 \Psi_{(x_j',0),\delta}^{*} \phi_{j,\delta}(\cdot) \left( \delta^{\Wdv} W_y \right)^{\beta}E(\cdot, y)\big|_{\cdot=u} \right|
         \\&\lesssim \sum_{|\beta|\leq N}\BCmNorm{ \Extend^1 \Psi_{(x_j',0),\delta}^{*} \phi_{j,\delta}(\cdot) \left( \delta^{\Wdv} W_y \right)^{\beta}E(\cdot, y)}{N}[\Qn{1}]
         \\&\lesssim \sum_{|\beta|\leq N}\BCmNorm{\Psi_{(x_j',0),\delta}^{*} \phi_{j,\delta}(\cdot) \left( \delta^{\Wdv} W_y \right)^{\beta}E(\cdot, y)}{N}[\Qngeq{1}],
    \end{split}
    \end{equation}
    where the second estimate uses Theorem \ref{Thm::Spaces::Scaling::MainScalingThm} \ref{Item::Spaces::Scaling::PullBackSmooth}
    and the final estimate uses Lemma \ref{Lemma::Spaces::Elem::Extend::Classical}.
    By Theorem \ref{Thm::Spaces::Scaling::MainScalingThm} \ref{Item::Spaces::Scaling::PullBackSmooth} and \ref{Item::Spaces::Scaling::PullBackSatisfyHormander}, we have (with \(m'\)
    given by \(m\) in that result),
    \begin{equation*}
        \BCmNorm{f}{N}[\Qngeq{1}]\lesssim \sum_{|\alpha|\leq m'N} \sup_{u\in \Qngeq{1}} \left| \left( \Wh^{(x_j',0),\delta} \right)^{\alpha} f(u) \right|.
    \end{equation*}
    So we have
    \begin{equation}\label{Eqn::Spaces::Elem::Extend::Extend::Extend0::Tmp2}
    \begin{split}
         &\sum_{|\beta|\leq N}\BCmNorm{\Psi_{(x_j',0),\delta}^{*} \phi_{j,\delta}(\cdot) \left( \delta^{\Wdv} W_y \right)^{\beta}E(\cdot, y)}{N}[\Qngeq{1}]
         \\&\lesssim \sum_{\substack{|\alpha|\leq m'N\\ |\beta|\leq N}}\sup_{u\in \Qngeq{1}} \left| \left( W_u^{(x_j',0),\delta} \right)^\alpha \Psi_{(x_j',0),\delta}^{*} \phi_{j,\delta}(\cdot) \left( \delta^{\Wdv}W_y \right)^{\beta} E(\cdot,y)\big|_{\cdot=u} \right|
         \\&\leq \sum_{\substack{|\alpha|\leq m'N\\ |\beta|\leq N}}\sup_{w\in \BWhWdv{(x_j',0)}{\delta} \cap \Bngeq{1}} \left| \left( \delta^{\Wdv} W_w \right)^{\alpha} \phi_{j,\delta}(w) \left( \delta^{\Wdv} W_y \right)^{\beta} E(w,y)  \right|
         \\&\lesssim \sum_{\substack{|\alpha|\leq m'N\\ |\beta|\leq N}}\sup_{w\in \BWhWdv{(x_j',0)}{\delta} \cap \Bngeq{1}} \left| \left( \delta^{\Wdv} W_w \right)^{\alpha}  \left( \delta^{\Wdv} W_y \right)^{\beta} E(w,y)  \right|
         \\&\lesssim \sup_{w\in \BWhWdv{(x_j',0)}{\delta}\cap \Bngeq{1}}
         \frac{
            \left( 1+\delta^{-1} \MetricWWdv[w][y] \right)^{-m}
        }{
            \Vol[\BWWdv{w}{\delta+\MetricWWdv[w][y]}\wedge 1]
        }
        \\&\lesssim \sup_{w\in \BWhWdv{(x_j',0)}{\delta}}
        \frac{
            \left( 1+\delta^{-1} \MetricWhWdv[w][y] \right)^{-m}
        }{
            \Vol[\BWWdv{w}{\delta+\MetricWhWdv[w][y]}\wedge 1]
        },
    \end{split}
    \end{equation}
    where the third estimate uses Proposition \ref{Prop::Spaces::Elem::Extend::ParitionOfUnity} \ref{Item::Spaces::Elem::Extend::ParitionOfUnity::DerivEstimates},
    the fourth uses the definition of \(\PElemz{\WhWdv, \Bngeq{7/8},\Vol}{\Bngeq{1/2}}\), and the final uses 
    Proposition \ref{Prop::VectorFields::Scaling::AmbientVolAndMetricEquivalnce}
    \ref{Item::VectorFields::Scaling::AmbientVolAndMetricEquivalnce::MetricWequalsMetricZ}
    and
    \ref{Item::VectorFields::Scaling::AmbientVolAndMetricEquivalnce::VolOfMetricWequalsVolOfMetricZ}.
    Since \(x\in\BWhWdv{(x_j',0)}{\delta} \), Lemma \ref{Lemma::Spaces::Elem::Extend::PreElemBoundUnderSmallPerterbations}
    shows
    \begin{equation}\label{Eqn::Spaces::Elem::Extend::Extend::Extend0::Tmp3}
        \sup_{w\in \BWhWdv{(x_j',0)}{\delta}}
        \frac{
            \left( 1+\delta^{-1} \MetricWhWdv[w][y] \right)^{-m}
        }{
            \Vol[\BWWdv{w}{\delta+\MetricWhWdv[w][y]}\wedge 1]
        }
        \approx
        \frac{
            \left( 1+\delta^{-1} \MetricWhWdv[x][y] \right)^{-m}
        }{
            \Vol[\BWWdv{x}{\delta+\MetricWhWdv[x][y]}\wedge 1]
        }.
    \end{equation}
    Combining \eqref{Eqn::Spaces::Elem::Extend::Extend::Extend0::Tmp1}, \eqref{Eqn::Spaces::Elem::Extend::Extend::Extend0::Tmp2},
    and \eqref{Eqn::Spaces::Elem::Extend::Extend::Extend0::Tmp3}
    establishes \eqref{Eqn::Spaces::Elem::Extend::Extend::ToShowPElemEst::Edelta0}
    for the term 
    \(\left( \Psi_{x,\delta} \right)_{*} \Extend^1 \Psi_{x,\delta}^{*} \phi_{j,\delta}(\cdot) E(\cdot,y)\big|_{\cdot=x}\),
    completing the proof.
\end{proof}

\begin{proof}[Proof of Proposition \ref{Prop::Spaces::Elem::Extend::Extend::Verbose} for \(\Extend_\delta\)]
    In light of Lemma \ref{Lemma::Spaces::Elem::Extend::ReduceToUnitBall}, we may work on the unit ball,
    with \(\Vol\) given by Lebesgue measure.
    Let \(\delta_2\in (0,1]\) be as in Proposition \ref{Prop::Spaces::Elem::Extend::ParitionOfUnity},
    and for \(\delta\in (0,\delta_2]\), let \(\phi_{j,\delta}\), \(j=1,\ldots,L=L(\delta)\), and \(\psi_\delta\)
    be as in that proposition.
    We define and study \(\Extend_{\delta}\) for \(\delta\in (0,\delta_2]\).
    For \(\delta\in (\delta_2,1]\), set \(\Extend_\delta:=\Extend_{\delta_2}\), and the required estimates
    for \(\Extend_\delta\) follow from those for \(\Extend_{\delta_2}\).
    We henceforth fix \(\delta\in (0,\delta_2]\).
    
    For \(E\in \CinftyCptSpace[\Bngeq{1}]\) with \(\supp(E)\subseteq \BngeqClosure{1/2}\times \BngeqClosure{1/2}\),
    we have
    \begin{equation*}
    \begin{split}
         &E(x,y)=\psi_\delta(x)E(x,y)+\left( 1-\psi_\delta(x) \right) E(x,y) \psi_\delta(y)
         +\sum_{j,k=1}^L \phi_{j,\delta}(x)E(x,y)\phi_{k,\delta}(y).
    \end{split}
    \end{equation*}
    
    We define the extension of each term separately.
    Since \((1-\psi_\delta(x))E(x,y)\psi_\delta(y)\) is supported in \(\BngeqClosure{1/2}\times \left( \BngeqClosure{1/2}\cap \Bng{1} \right)\),
    and therefore vanishes to infinite order as \(y\rightarrow \Bnmo{1}\),
    we may define
    \begin{equation*}
        \Extend_\delta (1-\psi_\delta(x))E(x,y)\psi_\delta(y) := \Extend_\delta^0 (1-\psi_\delta(x))E(x,y)\psi_\delta(y).
    \end{equation*}
    Using Proposition \ref{Prop::Spaces::Elem::Extend::ParitionOfUnity} \ref{Item::Spaces::Elem::Extend::ParitionOfUnity::DerivEstimates},
    if \(\sE\in \PElem{\WWdv, \Bngeq{7/8},\Vol}{\Bngeq{1/2}}\), then
    \begin{equation*}
        \left\{ ((1-\psi_\delta(x))E(x,y)\psi_\delta(y),\delta) : (E,\delta)\in \sE \right\}\in \PElemz{\WWdv, \Bngeq{7/8},\Vol}{\Bngeq{1/2}}.
    \end{equation*}
    Thus, it follows from the already proved result for \(\Extend_\delta^0\)
    that
    \begin{equation*}
         \left\{ \left( \Extend_\delta (1-\psi_\delta(x))E(x,y)\psi_\delta(y) ,\delta \right) : (E,\delta)\in \sE\right\}\in \PElem{\WWdv, \Bn{7/8},\Vol}{\BnClosure{3/4}},
    \end{equation*}
    as desired.

    Similarly, \(\psi_\delta(x)E(x,y)\) is supported away from \(x_n=0\), so we may define
    \(\Extend_\delta \psi_\delta(x)E(x,y)\) by reversing the roles of \(x\) and \(y\) and applying \(\Extend_\delta^0\).
    The same proof as above shows
    \begin{equation*}
        \left\{ \left( \Extend_\delta \psi_\delta(x)E(x,y) ,\delta \right) : (E,\delta)\in \sE\right\}\in \PElem{\WWdv, \Bn{7/8},\Vol}{\BnClosure{3/4}}.
   \end{equation*}

    Finally, we address the terms \(\phi_{j,\delta}(x)E(x,y)\phi_{k,\delta}(y)\). Set
    \(\Psih_{j,k,\delta}(u,v):=\left( \Psi_{(x_j',0),\delta}(u), \Psi_{(x_k',0),\delta }(v) \right)\),
    where \(x_j'\) is as in  Proposition \ref{Prop::Spaces::Elem::Extend::ParitionOfUnity}.
    Let 
    \begin{equation*}
        \Extend^2:\left\{ F\in \CinftySpace[\Qngeq{1}\times \Qngeq{1}] : \supp(F)\subseteq \QngeqClosure{1/2}\times \QngeqClosure{1/2} \right\}\rightarrow \CinftyCptSpace[\Qn{3/4}\times \Qn{3/4}]
    \end{equation*}
    be a classical extension operator as in Lemma \ref{Lemma::Spaces::Elem::Extend::Classical}.
    Set
    \begin{equation*}
        \Extend_\delta\phi_{j,\delta}(x)E(x,y)\phi_{k,\delta}(y)
        :=\left( \Psih_{j,k,\delta} \right)_{*} \Extend^2 \Psih_{j,k,\delta} ^{*} \phi_{j,\delta}(\cdot_1) E(\cdot_1,\cdot_2) \phi_{k,\delta}(\cdot_2)\big|_{\substack{\cdot_1=x \\ \cdot_2=y}}.
    \end{equation*}
    We have
    \begin{equation*}
        \supp\left( \Extend_\delta\phi_{j,\delta}(x)E(x,y)\phi_{k,\delta}(y) \right)\subseteq 
        \Psi_{(x_j',0),\delta}(\Qn{1})\times \Psi_{(x_j',0),\delta}(\Qn{1})
        \subseteq
        \Bn{3/4}\times \Bn{3/4},
    \end{equation*}
    and so \(\Extend_\delta\phi_{j,\delta}(x)E(x,y)\phi_{k,\delta}(y)\in \CinftyCptSpace[\Bn{3/4}\times \Bn{3/4}]\).

    To complete the proof, we wish to show that if \(\sE\in \PElem{\WWdv, \Bngeq{7/8},\Vol}{\Bngeq{1/2}}\),
    then \(\forall (E,\delta)\in \sE\), \(\forall L\), \(\forall N\),
    \begin{equation}\label{Eqn::Spaces::Elem::Extend::Extenddelta::ToShowWithSum}
    \begin{split}
         &\sum_{|\alpha|,|\beta|\leq N} \left|  \left( \delta^{\Wdv}\Wh_x \right)^{\alpha}\left( \delta^{\Wdv}\Wh_y \right)^{\beta} \sum_{j,k}\Extend_\delta\phi_{j,\delta}(x)E(x,y)\phi_{k,\delta}(y) \right|
         \lesssim
        \frac{
            \left( 1+\delta^{-1} \MetricWhWdv[x][y] \right)^{-L}
        }{
            \Vol[\BWhWdv{x}{\delta+\MetricWhWdv[x][y]}\wedge 1]
        }.
    \end{split}
    \end{equation}
    By Proposition \ref{Prop::Spaces::Elem::Extend::ParitionOfUnity} \ref{Item::Spaces::Elem::Extend::PartitionOfUnity::phiDontOverlapBig},
    each \((x,y)\) only lies in the support of at most \(M^2\) of the terms in the \(\sum_{j,k}\), where \(M\lesssim 1\).
    It therefore suffices to estimate each term separately; i.e., we wish to show
    \begin{equation}\label{Eqn::Spaces::Elem::Extend::Extenddelta::ToShowWithOutSum}
        \begin{split}
             &\sum_{|\alpha|,|\beta|\leq N} \left|  \left( \delta^{\Wdv}\Wh_x \right)^{\alpha}\left( \delta^{\Wdv}\Wh_y \right)^{\beta} \Extend_\delta\phi_{j,\delta}(x)E(x,y)\phi_{k,\delta}(y) \right|
             \lesssim
            \frac{
                \left( 1+\delta^{-1} \MetricWhWdv[x][y] \right)^{-L}
            }{
                \Vol[\BWhWdv{x}{\delta+\MetricWhWdv[x][y]}\wedge 1]
            }.
        \end{split}
        \end{equation}
        We henceforth assume \(x\in \BWhWdv{(x_j',0)}{\delta}\) and \(y\in \BWhWdv{(x_k',0)}{\delta}\),
        as otherwise 
        the left-hand side of \eqref{Eqn::Spaces::Elem::Extend::Extenddelta::ToShowWithOutSum} is \(0\).

    Letting \(W^{(x_j',0),\delta}\) be as in Theorem \ref{Thm::Spaces::Scaling::MainScalingThm},
    we have
    \begin{equation}\label{Eqn::Spaces::Elem::Extend::Extenddelta::FinalEquations1}
    \begin{split}
         &\sum_{|\alpha|,|\beta|\leq N} \left|  \left( \delta^{\Wdv}\Wh_x \right)^{\alpha}\left( \delta^{\Wdv}\Wh_y \right)^{\beta} \Extend_\delta\phi_{j,\delta}(x)E(x,y)\phi_{k,\delta}(y) \right|
         \\&\leq \sup_{u,v\in \Qn{1}} \left|  \left( \Wh_u^{(x_j',0),\delta} \right)^{\alpha}\left( \Wh_v^{(x_k',0),\delta} \right)^{\beta} \Extend^2 \Psih_{j,k,\delta}^{*} \phi_{j,\delta}(\cdot_1)E(\cdot_1,\cdot_2)\phi_{k,\delta}(\cdot_2)\big|_{\substack{u=\cdot_1\\v=\cdot_2}} \right|
         \\&\lesssim \BCmNorm{\Extend^2\Psih_{j,k,\delta}^{*} \phi_{j,\delta}(\cdot_1) E(\cdot_1,\cdot_2)\phi_{k,\delta}(\cdot_2)}{2N}[\Qn{1}\times \Qn{1}]
         \\&\lesssim \BCmNorm{\Psih_{j,k,\delta}^{*}\phi_{j,\delta}(\cdot_1) E(\cdot_1,\cdot_2)\phi_{k,\delta}(\cdot_2)}{2N}[\Qngeq{1}\times \Qngeq{1}],
    \end{split}
    \end{equation}
    where the second to last estimate follows from Theorem \ref{Thm::Spaces::Scaling::MainScalingThm} \ref{Item::Spaces::Scaling::PullBackSmooth},
    and the last estimate follows from Lemma \ref{Lemma::Spaces::Elem::Extend::Classical}.
    By Theorem \ref{Thm::Spaces::Scaling::MainScalingThm} \ref{Item::Spaces::Scaling::PullBackSmooth} and \ref{Item::Spaces::Scaling::PullBackSatisfyHormander}, we have 
    (with \(m\) as in that result),
    \begin{equation}\label{Eqn::Spaces::Elem::Extend::Extenddelta::SwitchToVFs}
    \begin{split}
         &\BCmNorm{F}{2N}[\Qngeq{1}\times \Qngeq{1}]
         \lesssim \sum_{|\alpha|,|\beta|\leq 2mN} \sup_{u,v\in \Qngeq{1}} \left|  \left( \Wh_u^{(x_j',0),\delta} \right)^{\alpha}\left( \Wh_v^{(x_k',0),\delta} \right)^{\beta} F(u,v) \right|.
    \end{split}
    \end{equation}
    Using \eqref{Eqn::Spaces::Elem::Extend::Extenddelta::SwitchToVFs}, we have
    \begin{equation}\label{Eqn::Spaces::Elem::Extend::Extenddelta::FinalEquations2}
    \begin{split}
         &\BCmNorm{\Psih_{j,k,\delta}^{*} \phi_{j,\delta}(\cdot_1) E(\cdot_1,\cdot_2)\phi_{k,\delta}(\cdot_2)}{2N}[\Qngeq{1}\times \Qngeq{1}]
         \\&\lesssim \sum_{|\alpha|,|\beta|\leq 2mN} \sup_{u,v\in \Qngeq{1}} \left|\left( \Wh_u^{(x_j',0),\delta} \right)^{\alpha}\left( \Wh_v^{(x_k',0),\delta} \right)^{\beta}\Psih_{j,k,\delta}^{*} \phi_{j,\delta}(\cdot_1) E(\cdot_1,\cdot_2) \phi_{k,\delta}(\cdot_2)  \right|
         \\&=\sum_{|\alpha|,|\beta|\leq 2mN} \sup_{\xt,\yt\in \Psih_{j,k,\delta}(\Qngeq{1})} 
        \left| 
            \left( \delta^{\Wdv} \Wh_{\xt} \right)^{\alpha}
            \left( \delta^{\Wdv} \Wh_{\yt} \right)^{\beta} 
            \phi_{j,\delta}(\xt) E(\xt,\yt) \phi_{k,\delta}(\yt) 
        \right|
        \\&\lesssim \sum_{|\alpha|,|\beta|\leq 2mN}
        \sup_{\substack{\xt\in \BWhWdv{(x_j',0)}{\delta}\cap \Bngeq{1}\\ \yt\in \BWhWdv{(x_k',0)}{\delta}}\cap \Bngeq{1}}
        \left| 
            \left( \delta^{\Wdv} \Wh_{\xt} \right)^{\alpha}
            \left( \delta^{\Wdv} \Wh_{\yt} \right)^{\beta} 
            E(\xt,\yt)
        \right|
        \\&\lesssim
        \sup_{\substack{\xt\in \BWhWdv{(x_j',0)}{\delta}\cap \Bngeq{1}\\ \yt\in \BWhWdv{(x_k',0)}{\delta}}\cap\Bngeq{1}}
        \frac{
                \left( 1+\delta^{-1} \MetricWWdv[x][y] \right)^{-L}
            }{
                \Vol[\BWWdv{x}{\delta+\MetricWWdv[x][y]}\wedge 1]
            }
        \\&\lesssim 
        \sup_{\substack{\xt\in \BWhWdv{(x_j',0)}{\delta}\\ \yt\in \BWhWdv{(x_k',0)}{\delta}}}
        \frac{
                \left( 1+\delta^{-1} \MetricWhWdv[x][y] \right)^{-L}
            }{
                \Vol[\BWhWdv{x}{\delta+\MetricWhWdv[x][y]}\wedge 1]
            }
    \end{split}
    \end{equation}
    where the third-to-last estimate used Proposition \ref{Prop::Spaces::Elem::Extend::ParitionOfUnity} \ref{Item::Spaces::Elem::Extend::ParitionOfUnity::DerivEstimates},
    the second-to-last used the definition of \(\PElem{\WWdv, \Bngeq{7/8},\Vol}{\Bngeq{1/2}}\),
    and the last used Proposition \ref{Prop::VectorFields::Scaling::AmbientVolAndMetricEquivalnce}
    \ref{Item::VectorFields::Scaling::AmbientVolAndMetricEquivalnce::MetricWequalsMetricZ}
    and
    \ref{Item::VectorFields::Scaling::AmbientVolAndMetricEquivalnce::VolOfMetricWequalsVolOfMetricZ}.

    Since  \(x\in \BWhWdv{(x_j',0)}{\delta}\) and \(y\in \BWhWdv{(x_k',0)}{\delta}\),
    Lemma \ref{Lemma::Spaces::Elem::Extend::PreElemBoundUnderSmallPerterbations} shows
    \begin{equation}\label{Eqn::Spaces::Elem::Extend::Extenddelta::FinalEquations3}
    \begin{split}
         &\sup_{\substack{\xt\in \BWhWdv{(x_j',0)}{\delta}\\ \yt\in \BWhWdv{(x_k',0)}{\delta}}}
         \frac{
                 \left( 1+\delta^{-1} \MetricWhWdv[x][y] \right)^{-L}
             }{
                 \Vol[\BWhWdv{x}{\delta+\MetricWhWdv[x][y]}\wedge 1]
             }
             \approx
             \frac{
                 \left( 1+\delta^{-1} \MetricWhWdv[x][y] \right)^{-L}
             }{
                 \Vol[\BWhWdv{x}{\delta+\MetricWhWdv[x][y]}\wedge 1]
             }.
    \end{split}
    \end{equation}
    Combining \eqref{Eqn::Spaces::Elem::Extend::Extenddelta::FinalEquations1}, \eqref{Eqn::Spaces::Elem::Extend::Extenddelta::FinalEquations2},
    and \eqref{Eqn::Spaces::Elem::Extend::Extenddelta::FinalEquations3}
    establishes \eqref{Eqn::Spaces::Elem::Extend::Extenddelta::ToShowWithOutSum} and completes the proof.
\end{proof}

        \subsubsection{Pre-elemenary operators}

We first prove part of Lemma \ref{Lemma::Spaces::LP::ElemDoesntDependOnChoices}:
\begin{lemma}\label{Lemma::Spaces::Elem::PreElem::PreElemDoesntDependOnChoices}
    Let \(\FilteredSheafG\), \(\Omegah\), \(\ZZde\), and \(\Volh\) be as in Lemma \ref{Lemma::Spaces::LP::ElemDoesntDependOnChoices}.
    Then,
    \begin{equation*}
        \PElem{\WWdv, \Omega,\Vol}{\Compact}=\PElem{\ZZde,\Omegah,\Volh}{\Compact},\quad \PElemz{\WWdv, \Omega,\Vol}{\Compact}=\PElemz{\ZZde,\Omegah,\Volh}{\Compact}.
    \end{equation*}
    % Suppose \(\FilteredSheafG\) is another H\"ormander filtration of sheaves of vector fields on \(\ManifoldN\)
    % with \(\LieFilteredSheafF=\LieFilteredSheafG\), \(\Omegah\Subset \ManifoldNncF\) is another relatively compact open set with
    % \(\Compact\Subset \Omegah\), \(\ZZe\) are H\"romander vector fields with formal degrees
\end{lemma}
\begin{proof}
    The result for \(\PElemzSymbol\) follows from the result for \(\PElemSymbol\), so we prove only the result for \(\PElemSymbol\).
    It suffices to prove the result with \(\FilteredSheafG\) replaced by \(\LieFilteredSheafF\) (see Proposition \ref{Prop::Filtrations::RestrictingFiltrations::LieFIsHormanderFiltration}), 
    and we henceforth make this
    replacement.

    Proposition \ref{Prop::VectorFields::Scaling::VolAndMetricEquivalnce} \ref{Item::VectorFields::Scaling::VolAndMetricEquivalnce::MetricWequalsMetricZ} and \ref{Item::VectorFields::Scaling::VolAndMetricEquivalnce::VolOfMetricWequalsVolOfMetricZ} 
    shows that the right-hand side
    \eqref{Eqn::Spaces::LP::PreElemBound} is equivalent whether we use \(\WWdv\), \(\Omega\), and \(\Vol\),
    or \(\ZZde\), \(\Omegah\), and \(\Volh\).
    Thus, it remains to show that in the left hand side of \eqref{Eqn::Spaces::LP::PreElemBound}, it is equivalent
    to either use \(\WWdv\) or \(\ZZde\).

    Since
    \((W_j,\Wdv_j)\in \FilteredSheafF[\Omega\cap \Omegah][\Wdv_j]\subset \LieFilteredSheafF[\Omega\cap \Omegah][\Wdv_j]=\FilteredSheafGenBy{\ZZde}[\Omega\cap \Omegah][\Wdv_j]\),
    we may replace any term on the left-hand side of \eqref{Eqn::Spaces::LP::PreElemBound} with sums of terms of that form
    with \(\WWdv\) replaced by \(\ZZde\). We conclude  \(\PElem{\WWdv, \Omega,\Vol}{\Compact}\subseteq\PElem{\ZZde,\Omegah,\Volh}{\Compact}\).
    
    For the reverse containment, let \(\GenWWdv\) be as in Lemma \ref{Lemma::Filtrations::GeneratorsForLieFiltration}.
    It follows
    from Lemma \ref{Lemma::Filtrations::GeneratorsForLieFiltration} that there is a neighborhood \(\Omega_1\Subset \ManifoldN\)
    of \(\Compact\) on which each
    \((Z_j,e_j)\) can be written as a \(\CinftySpace[\Omega_1]\) linear combination of 
    \(\left\{ (V,f)\in \GenWWdv : f\leq e_j \right\}\). Using this, the left-hand side of \eqref{Eqn::Spaces::LP::PreElemBound} with \(\WWdv\) replaced by \(\ZZde\)
    can be bounded by a sum of terms of the same form using \(\WWdv\).
    We conclude  \(\PElem{\WWdv, \Omega,\Vol}{\Compact}\supseteq\PElem{\ZZde,\Omegah,\Volh}{\Compact}\), completing the proof.
\end{proof}

With Lemma \ref{Lemma::Spaces::Elem::PreElem::PreElemDoesntDependOnChoices} in hand, 
we no longer require Notation \ref{Notation::Spaces::Elem::Intro::VerposeElemNotation} when referring
to \(\PElemOnlySubscript{\FilteredSheafF}\) and \(\PElemzOnlySubscript{\FilteredSheafF}\).

\begin{lemma}\label{Lemma::Spaces::Elem::PElem::PElemOpsBoundedOnLp}
    Let \(\sE\in \PElemF{\Compact}\). Then,
    \begin{equation*}
        \sup_{1\leq p\leq \infty} \sup_{(E,\delta)\in \sE} \LpOpNorm{E}{p}<\infty.
    \end{equation*}
\end{lemma}

\begin{proposition}\label{Prop::Spaces::Elem::PElem::PElemOpsBoundedOnVV}
    Let \(\VSpace{p}{q}\) be as in Notation \ref{Notation::Spaces::Classical::VSpacepq}
    and let \(\sE\in \PElemF{\Compact}\). For \(\sE'=\left\{ (E_j,\delta_j) : j\in \Zgeq \right\}\subseteq \sE\)
    define
    \begin{equation*}
        \sT_{\sE'} \left\{ f_j \right\}_{j\in \Zgeq} := \left\{ E_j f_j \right\}_{j\in \Zgeq}.
    \end{equation*}
    Then,
    \begin{equation*}
        \sup_{\sE'\subseteq \sE} \VOpNorm{\sT_{\sE'}}{p}{q}<\infty,
    \end{equation*}
    where the supremum is taken over all countable sequences in \(\sE\).
\end{proposition}
\begin{proof}[Proof of Lemma \ref{Lemma::Spaces::Elem::PElem::PElemOpsBoundedOnLp} and Proposition \ref{Prop::Spaces::Elem::PElem::PElemOpsBoundedOnVV}]
    Let \(\ManifoldM\) and \(\FilteredSheafFh\) be as in Proposition \ref{Prop::Filtrations::RestrictingFiltrations::CoDim0CoRestriction},
    and let \(\Extend_\delta\) be as in Proposition \ref{Prop::Spaces::Elem::Extend::Extend::NonVerbose}, where \(\Omega_1\)
    is any set as in the statement of that proposition.
    In light of Proposition \ref{Prop::Spaces::Elem::Extend::Extend::NonVerbose} \ref{Item::Spaces::Elem::Extend::Extend::NonVerbose::ExtendRestrictIsI}
    it suffices to prove the results with
    \(\sE\) replaced by \(\sEh:=\left\{ (\Extend_\delta E, \delta) : (E,\delta)\in \sE \right\}\).
    By Proposition \ref{Prop::Spaces::Elem::Extend::Extend::NonVerbose} \ref{Item::Spaces::Elem::Extend::Extend::NonVerbose::ExtendPreElemIsPreElem},
    \(\sEh\in \PElemFh{\overline{\Omega_1}}\).
    In short, we see that it suffices to prove the result with \(\FilteredSheafF\) replaced by \(\FilteredSheafFh\);
    in other words, it suffices to prove the result on a manifold without boundary.
    On such a manifold without boundary, the results are shown in
    \cite[Corollary 5.4.11 and Lemma 5.4.12]{StreetMaximalSubellipticity}.
    Alternatively, one could recreate the proofs of \cite[Corollary 5.4.11 and Lemma 5.4.12]{StreetMaximalSubellipticity} on a manifold with boundary with almost no changes,
    and avoid using extension operators.
\end{proof}

% To help prove the next results, 
% let \(\ManifoldM\) and \(\FilteredSheafFh\) be as in Proposition \ref{Prop::Filtrations::RestrictingFiltrations::CoDim0CoRestriction}.
% Fix \(\Omegah\Subset \ManifoldM\) open with \(\Compact\Subset \Omega\).
% Let \(\FilteredSheafFh\big|_{\Omegah}=\FilteredSheafGenBy{\WhWdv}\), where \(\WhWdv\) are H\"ormander vector fields
% with formal degrees on \(\Omegah\). Fix any smooth, strictly positive density \(\Volh\) on \(\Omegah\).

\begin{lemma}\label{Lemma::Spaces::Elem::PreElem::ComposePreElemBound}
    \(\forall m\in \Zgeq\), \(\exists m'\in \Zgeq\), \(C\geq 1\), \(\forall x,z\in \Compact\), \(j,k\in [0,\infty)\),
    \begin{equation*}
    \begin{split}
         &\int_{\Compact} 
         \frac{\left( 1+2^{j}\MetricWWdv[x][y] \right)^{-m'}}{\Vol[\BWWdv{x}{2^{-j}+\MetricWWdv[x][y]}]}
         \frac{\left( 1+2^{j}\MetricWWdv[y][z] \right)^{-m'}}{\Vol[\BWWdv{y}{2^{-j}+\MetricWWdv[y][z]}]}
         \:d\Vol[y]
         \\&\leq C
         \frac{\left( 1+2^{j}\MetricWWdv[x][z] \right)^{-m}}{\Vol[\BWWdv{x}{2^{-j}+\MetricWWdv[x][z]}]}.
    \end{split}
    \end{equation*}
\end{lemma}

\begin{lemma}\label{Lemma::Spaces::Elem::PreElem::PreElemBoundDifferentScales}
    \(\forall m\in \Zgeq\), \(\exists M\geq 1\), \(C\geq 1\), \(\forall x,z\in \Compact\), \(\forall j,k\in [0,\infty)\),
    \begin{equation*}
    \begin{split}
         &2^{-M|j-k|} 
         \frac{\left( 1+2^{k}\MetricWWdv[x][z] \right)^{-m}}{\Vol[\BWWdv{x}{2^{-k}+\MetricWWdv[x][z]}]}
        \leq C
        \frac{\left( 1+2^{j}\MetricWWdv[x][z] \right)^{-m}}{\Vol[\BWWdv{x}{2^{-j}+\MetricWWdv[x][z]}]}.
    \end{split}
    \end{equation*}
\end{lemma}

\begin{proof}[Proof of Lemma \ref{Lemma::Spaces::Elem::PreElem::ComposePreElemBound} and Lemma \ref{Lemma::Spaces::Elem::PreElem::PreElemBoundDifferentScales}]
    Let \(\ManifoldM\) and \(\FilteredSheafFh\) be as in Proposition \ref{Prop::Filtrations::RestrictingFiltrations::CoDim0CoRestriction}.
    Fix \(\Omegah\Subset \ManifoldM\) open with \(\Compact\Subset \Omegah\) and \(\Omegah\cap \ManifoldN\subseteq \ManifoldNncF\).
    Let \(\FilteredSheafFh\big|_{\Omegah}=\FilteredSheafGenBy{\WhWdv}\), where \(\WhWdv\) are H\"ormander vector fields
    with formal degrees on \(\Omegah\). Fix any smooth, strictly positive density \(\Volh\) on \(\Omegah\).

    In light of Proposition \ref{Prop::VectorFields::Scaling::AmbientVolAndMetricEquivalnce}
    \ref{Item::VectorFields::Scaling::AmbientVolAndMetricEquivalnce::MetricWequalsMetricZ}
    and \ref{Item::VectorFields::Scaling::AmbientVolAndMetricEquivalnce::VolOfMetricWequalsVolOfMetricZ},
    it suffices to prove the results with \(\WWdv\) and \(\Vol\)
    replaced by \(\WhWdv\) and \(\Volh\), respectively. In short, it suffices to prove the results
    on a manifold without boundary.
    This is contained in \cite[Propostion 5.4.4 and Lemma 5.4.7]{StreetMaximalSubellipticity}.

    Alternatively, one can follow the same proofs as those in \cite[Propostion 5.4.4 and Lemma 5.4.7]{StreetMaximalSubellipticity}
    to directly establish the results in this setting.
\end{proof}

        \subsubsection{Elementary opearators}
        Let \(\sE\in \ElemzF{\Compact}\). For \((E,\delta)\in \sE\), we may view \(E\) as an operator
\(E: \DistributionsZeroN\rightarrow\DistributionsZeroN\), whose range lies in
\(\CinftyCptSpace[\Omega]\), where \(\Compact\Subset \Omega\); see Remark \ref{Rmk::Spaces::LP::FunctionsOrOperators} for details.
In what follows, we switch between viewing \(E\) as a function and \(E\) as an operator, freely. When we treat
\(E\) as a function we will usually write \(E(x,y)\) to make it clear. When we treat \(E\) as an operator,
we merely write \(E\).

\begin{remark}\label{Rmk::Spaces::Elem::Elem::IntegrateByPartsWithoutBoundary}
    Because we view \(E\) as an operator \(E: \DistributionsZeroN\rightarrow\DistributionsZeroN\)
    and \(E(x,y)\) vanishes to infinite order as \(y\rightarrow \BoundaryN\), we may integrate by parts without any boundary terms.
    For example, for a vector field \(V\), we have
    \begin{equation*}
        \int V_y^{*} E(x,y) f(y)\: d\Vol(y) = E V f.
    \end{equation*}
    where \(V^{*}_y\) denotes the formal \(\LpSpace{2}[\Vol]\) adjoint of the vector field \(V\), acting in the \(y\) variable;
    i.e., \(V_y^{*}=-V_y+g(y)\)
    where \(g(y)\) is some smooth function.
\end{remark}

\begin{proposition}\label{Prop::Spaces::Elem::Elem::MainProps}
    Let \(\sE\in \ElemzF{\Compact}\). Then,
    \begin{enumerate}[(a)]
        \item\label{Item::Spaces::Elem::Elem::LinearComb} For \(C\geq 1\) fixed,
            \begin{equation*}
                \left\{ \left( aE_1+bE_2,\delta \right) :(E_1,\delta),(E_2,\delta)\in \sE, a,b\in \C, |a|,|b|\leq C\right\}
                \in \ElemzF{\Compact}.
            \end{equation*}
        \item\label{Item::Spaces::Elem::Elem::InfiniteComb} Suppose \(\left\{ a_l \right\}_{l\in \Zgeq}\) is a sequence of complex numbers with \(\sum|a_l|<\infty\). Then,
            \begin{equation*}
                \left\{ \left( \sum_{l\in \Zgeq} a_l E_l, \delta \right) : (E_l,\delta)\in \sE \right\}
                \in \ElemzF{\Compact}.
            \end{equation*}
        \item\label{Item::Spaces::Elem::Elem::MultBySmooth} Fix \(\Omega'\subseteq \ManifoldN\) open with \(\Compact\subseteq \Omega'\) and \(\psi\in \CinftyCptSpace[\Omega']\).
            Then,
            \begin{equation*}
                \left\{ \left( \psi(x)E(x,y),\delta \right), \left( E(x,y)\psi(y),\delta \right) :(E,\delta)\in \sE \right\}\in \ElemzF{\Compact}.
            \end{equation*}
        \item\label{Item::Spaces::Elem::Elem::DerivFunction} For an ordered multi-index \(\alpha\),
            \begin{equation*}
                \left\{ \left( \left( \delta^{\Wdv}W_x \right)^{\alpha}E(x,y),\delta \right), \left( \left( \delta^{\Wdv}W_y\right)^{\alpha}E(x,y),\delta \right)  : \left( E,\delta \right)\in \sE  \right\}\in \ElemzF{\Compact}.
            \end{equation*}
        \item\label{Item::Spaces::Elem::Elem::DerivOp} For an ordered multi-index \(\alpha\),
            \begin{equation*}
                \left\{ \left( \left( \delta^{\Wdv}W\right)^{\alpha}E,\delta \right), \left( E\left( \delta^{\Wdv}W\right)^{\alpha},\delta \right)  : \left( E,\delta \right)\in \sE  \right\}\in \ElemzF{\Compact}.
            \end{equation*}
        \item\label{Item::Spaces::Elem::Elem::PullOutNDerivs} For each \(N\in \Zgeq\), every \((E,2^{-j})\in \sE\) can be written as
            \begin{equation}\label{Eqn::Spaces::Elem::Elem::PullOutNDerivs::Left}
                E=\sum_{|\alpha|\leq N} 2^{(|\alpha|-N)j} \left( 2^{-j\Wdv}W \right) E_{\alpha},
            \end{equation}
            where \(\left\{ \left( E_\alpha, 2^{-j} \right) : (E,2^{-j})\in \sE,|\alpha|\leq N \right\}\in \ElemzF{\Compact}\).
            Similarly, \(E\) can be written as
            \begin{equation}\label{Eqn::Spaces::Elem::Elem::PullOutNDerivs::Right}
                E=\sum_{|\alpha|\leq N} 2^{(|\alpha|-N)j}  \Et_{\alpha}\left( 2^{-j\Wdv}W \right),
            \end{equation}
            where \(\left\{ \left( \Et_\alpha, 2^{-j} \right) : (E,2^{-j})\in \sE,|\alpha|\leq N \right\}\in \ElemzF{\Compact}\).
        \item\label{Item::Spaces::Elem::Elem::DerivFuncAribtrarySection} Let \(\Omega'\subseteq \ManifoldN\) be open with \(\Compact\Subset \Omega'\), and fix \(Y\in \FilteredSheafF[\Omega'][d_0]\).
            Then,
            \begin{equation*}
                \left\{ \left( 2^{-jd_0}Y_x E(x,y), 2^{-j} \right), \left( 2^{-jd_0}Y_y E(x,y), 2^{-j} \right) \right\}\in \ElemzF{\Compact}.
            \end{equation*}
    \end{enumerate}
\end{proposition}
\begin{proof}
    Once  Lemma \ref{Lemma::Spaces::LP::ElemDoesntDependOnChoices} is established, this result is
    the same as Lemma \ref{Lemma::Spaces::Elem::Elem::MainPropsLemma}, below.
\end{proof}

\begin{remark}\label{Rmk::Spaces::Elem::Elem::HowToProveElem}
    Suppose we are given \(\sE\in \PElemzF{\Compact}\) and we wish to show \(\sE\in \ElemzF{\Compact}\).
    The main way we prove this is to show that there is a subset \(\sB\subseteq \PElemzF{\Compact}\)
    such that \(\sE\in \sB\) and \(\sB\) satisfies
    Definition \ref{Defn::Spaces::LP::ElemWWdv} \ref{Item::Spaces::LP::ElemWWdv::ElemArePreDerivs} (in place of \(\ElemzF{\Compact}\)).
    Since \(\ElemzF{\Compact}\subseteq \PElemzF{\Compact}\) is the largest subset satisfying this property, we get \(\sE\in \sB\subseteq \ElemzF{\Compact}\).
    In short, if we can show that for \((E,\delta)\in \sE\), we have
    \begin{equation*}
        E(x,y)=\sum_{|\alpha|,|\beta|\leq 1} \delta^{2-|\alpha|-|\beta|} \left( \delta^{\Wdv}W_x \right)^{\alpha} \left( \delta^{\Wdv} W_y \right)^{\beta}E_{\alpha,\beta} (x,y)
    \end{equation*}
    where \(E_{\alpha,\beta}\) is ``of the same form'' as \(E\), then \(\sE\in \ElemzF{\Compact}\).
\end{remark}

Since we have not yet proved Lemma \ref{Lemma::Spaces::LP::ElemDoesntDependOnChoices}, we should 
have used Notation \ref{Notation::Spaces::Elem::Intro::VerposeElemNotation}
in the statement of Proposition \ref{Prop::Spaces::Elem::Elem::MainProps}. In fact, we will use Proposition \ref{Prop::Spaces::Elem::Elem::MainProps}
in the proof of Lemma \ref{Lemma::Spaces::LP::ElemDoesntDependOnChoices}, so to avoid circular reasoning, we restate
Proposition \ref{Prop::Spaces::Elem::Elem::MainProps} in the next lemma.

% Fix \(\Omega\Subset \ManifoldN\) open with
% \(\Compact\Subset \Omega\), and let \(\WWdv=\left\{ \left( W_1,\Wdv_1 \right),\ldots, \left( W_r,\Wdv_r \right) \right\}\subset \VectorFields{\Omega}\times \Zg\)
% be H\"ormander vector fields with formal degrees on \(\Omega\).

\begin{lemma}\label{Lemma::Spaces::Elem::Elem::MainPropsLemma}
    Proposition \ref{Prop::Spaces::Elem::Elem::MainProps} holds with
    \(\ElemzOnlySubscript{\FilteredSheafF}\) replaced by \(\ElemzOnlySubscript{\WWdv,\Omega,\Vol}\), throughout.
\end{lemma}
\begin{proof}
    This result is very similar to \cite[Proposition 5.5.5]{StreetMaximalSubellipticity}, though the definitions
    and results here vary slightly from those. Because of this, we include the full proof.

    \ref{Item::Spaces::Elem::Elem::LinearComb} and \ref{Item::Spaces::Elem::Elem::InfiniteComb} follow easily
    from the definitions.
    For \ref{Item::Spaces::Elem::Elem::MultBySmooth}, we prove
    \begin{equation*}
        \left\{ \left( \psi(x)E(x,y),\delta \right) :(E,\delta)\in \sE \right\}\in \ElemzWWdvOmegaVol{\Compact},
    \end{equation*}
    the result with \(\psi(x)\) replaced by \(\psi(y)\) is nearly identical.
    It follows immediately from the definitions that 
    \begin{equation*}
        \left\{ \left( \psi(x)E(x,y),\delta \right) :(E,\delta)\in \sE \right\}\in \PElemzWWdvOmegaVol{\Compact},
    \end{equation*}
    so in light of Remark \ref{Rmk::Spaces::Elem::Elem::HowToProveElem}, we wish to show
    \(\psi(x)E(x,y)\) is a sum of derivatives which are of the same form as \(\psi(x)E(x,y)\).
    By definition, we have
    \begin{equation}\label{Eqn::Spaces::Elem::Elem::MainProperties::PullOutDerivsInProof}
        E=\sum_{|\alpha|,|\beta|\leq 1} \delta^{2-|\alpha|-|\beta|}\left( \delta^{\Wdv}W_x \right)^{\alpha} \left( \delta^{\Wdv}W_y \right)^{\beta}E_{\alpha,\beta}(x,y),
    \end{equation}
    where \(\left\{ \left( E_{\alpha,\beta},\delta \right) : (E,\delta)\in \sE,|\alpha|,|\beta|\leq 1 \right\}\in \ElemzWWdvOmegaVol{\Compact}\).
    We therefore have,
    \begin{equation*}
    \begin{split}
         & \psi(x) E =\sum_{|\alpha|,|\beta|\leq 1} \delta^{2-|\alpha|-|\beta|}\psi(x)\left( \delta^{\Wdv}W_x \right)^{\alpha} \left( \delta^{\Wdv}W_y \right)^{\beta}E_{\alpha,\beta}(x,y)
         \\&=\sum_{|\alpha|,|\beta|\leq 1} \delta^{2-|\alpha|-|\beta|}\left( \delta^{\Wdv}W_x \right)^{\alpha} \left( \delta^{\Wdv}W_y \right)^{\beta}\psi(x)E_{\alpha,\beta}(x,y)
         +\sum_{j=1}^r \delta^{1-|\beta|+\Wdv_j} (W_j\psi)(x) \left( \delta^{\Wdv}W_y \right)^{\beta} E_{\{j\},\beta}(x,y)
         \\&=\sum_{|\alpha|=1,|\beta|\leq 1}\delta^{1-|\beta|}\left( \delta^{\Wdv}W_x \right)^{\alpha} \left( \delta^{\Wdv}W_y \right)^{\beta}\psi(x)E_{\alpha,\beta}(x,y)
         \\&\quad\quad+\sum_{|\beta|\leq 1} \delta^{2-|\beta|}\left( \delta^{\Wdv}W_y \right)^{\beta}\left[  \psi(x)E_{\{\},\beta}(x,y) + \sum_{j=1}^r \delta^{\Wdv_j-1} (W_j \psi(x))E_{\{j\},\beta}(x,y) \right]
    \end{split}
    \end{equation*}
    By \ref{Item::Spaces::Elem::Elem::LinearComb} this sees \(\psi(x)E(x,y)\) as an appropriate sum of derivatives of terms of the same form
    as \(\psi(x)E(x,y)\)
    (see Remark \ref{Rmk::Spaces::Elem::Elem::HowToProveElem}), completing the proof of \ref{Item::Spaces::Elem::Elem::MultBySmooth}.

    It suffices to prove \ref{Item::Spaces::Elem::Elem::DerivFunction} in the case \(|\alpha|=1\), as the general case follows
    from this and a simple induction; so we assume \(|\alpha|=1\).
    We prove
    \begin{equation*}
        \left\{ \left( \left( \delta^{\Wdv}W_x \right)^{\alpha}E(x,y),\delta \right)  : \left( E,\delta \right)\in \sE  \right\}\in \ElemzWWdvOmegaVol{\Compact},
    \end{equation*}
    the proof of \(
        \left\{  \left( \left( \delta^{\Wdv}W_y\right)^{\alpha}E(x,y),\delta \right)  : \left( E,\delta \right)\in \sE  \right\}\in \ElemzWWdvOmegaVol{\Compact}
        \)
    is similar.
    It follows immediately from the definitions that 
    \begin{equation*}
        \left\{ \left( \left( \delta^{\Wdv}W_x \right)^{\alpha}E(x,y),\delta \right)  : \left( E,\delta \right)\in \sE  \right\}\in \PElemzWWdvOmegaVol{\Compact}.
    \end{equation*}
    By definition, we have \eqref{Eqn::Spaces::Elem::Elem::MainProperties::PullOutDerivsInProof}, and therefore,
    \begin{equation*}
    \begin{split}
         &\left( \delta^{\Wdv}W_x \right)^{\alpha}E(x,y)
         =\left( \delta^{\Wdv}W_x \right)^{\alpha}\sum_{|\gamma|,|\beta|\leq 1} \delta^{2-|\gamma|-|\beta|}  \left( \delta^{\Wdv}W_x \right)^{\gamma} \left( \delta^{\Wdv}W_y \right)^{\beta}E_{\gamma,\beta}(x,y)
         \\&=\sum_{|\beta|\leq 1}\delta^{1-|\beta|} \left( \delta^{\Wdv}W_x \right)^{\alpha} \left( \delta^{\Wdv}W_y \right)^{\beta} \left[ \sum_{|\gamma|=1} \left( \delta^{\Wdv}W_x \right)^{\gamma} E_{\gamma,\beta}(x,y)\right]
         +\sum_{|\beta|\leq 1}\delta^{2-|\beta|}\left( \delta^{\Wdv}W_y \right)^{\beta} \left[ \left( \delta^{\Wdv}W_x \right)^{\alpha}E_{\{\},\beta}(x,y) \right].
    \end{split}
    \end{equation*}
    By \ref{Item::Spaces::Elem::Elem::LinearComb} this sees \(\left( \delta^{\Wdv}W_x \right)^{\alpha}E(x,y)\) as an appropriate sum of derivatives of terms of the same form
    as \(\left( \delta^{\Wdv}W_x \right)^{\alpha}E(x,y)\)
    (see Remark \ref{Rmk::Spaces::Elem::Elem::HowToProveElem}), completing the proof of  \ref{Item::Spaces::Elem::Elem::DerivFunction}.

    We turn to \ref{Item::Spaces::Elem::Elem::DerivOp}. That
    \begin{equation*}
        \left\{ \left( \left( \delta^{\Wdv}W\right)^{\alpha}E,\delta \right)  : \left( E,\delta \right)\in \sE  \right\}\in \ElemzWWdvOmegaVol{\Compact}
    \end{equation*}
    is exactly the same statement as 
    \begin{equation*}
        \left\{ \left( \left( \delta^{\Wdv}W_x \right)^{\alpha}E(x,y),\delta \right)  : \left( E,\delta \right)\in \sE  \right\}\in \ElemzWWdvOmegaVol{\Compact},
    \end{equation*}
    which was established in \ref{Item::Spaces::Elem::Elem::DerivFunction}.
    So it suffices to show 
    \begin{equation}\label{Eqn::Spaces::ELem::Elem::DerivFunction::ProofToShow}
        \left\{ \left( E\left( \delta^{\Wdv}W\right)^{\alpha},\delta \right)  : \left( E,\delta \right)\in \sE  \right\}\in \ElemzWWdvOmegaVol{\Compact}.
    \end{equation}
    We prove \eqref{Eqn::Spaces::ELem::Elem::DerivFunction::ProofToShow} in the case \(|\alpha|=1\), as the general case
    follows from this and a simple induction. Thus, we have \(\left( \delta^{\Wdv} W \right)^{\alpha}= \delta^{\Wdv_j}W_j\)
    for some \(j\). We have \(W_j^{*}=-W_j+f_j\), where \(f_j\in \CinftySpace[\Omega]\), and so (using Remark \ref{Rmk::Spaces::Elem::Elem::IntegrateByPartsWithoutBoundary}),
    \begin{equation*}
    \begin{split}
         &\left[ E \left( \delta^{\Wdv_j}W_j \right) \right](x,y)
         =-\delta^{\Wdv_j} W_{j,y}E(x,y)+\delta^{\Wdv_j} f_j(y) E(x,y),
    \end{split}
    \end{equation*}
    where \(W_{j,y}\) denotes the vector field \(W_j\) acting in the \(y\)-variable.
    \eqref{Eqn::Spaces::ELem::Elem::DerivFunction::ProofToShow} now follows from \ref{Item::Spaces::Elem::Elem::LinearComb}
    and \ref{Item::Spaces::Elem::Elem::DerivFunction}.

    We turn to \ref{Item::Spaces::Elem::Elem::PullOutNDerivs}. We begin with \eqref{Eqn::Spaces::Elem::Elem::PullOutNDerivs::Left}.
    We prove \eqref{Eqn::Spaces::Elem::Elem::PullOutNDerivs::Left} in the case \(N=1\);
    the general case follows from this and a simple induction.
    Using \eqref{Eqn::Spaces::Elem::Elem::MainProperties::PullOutDerivsInProof},
    we have for \(\left( E,\delta \right)\in \sE\),
    \begin{equation*}
    \begin{split}
         & E=\sum_{|\alpha|\leq 1} \delta^{1-|\alpha|}\left( \delta^{\Wdv}W \right)^{\alpha}  \left( \sum_{|\alpha|\leq 1} \delta^{1-|\alpha|} \left( \delta^{\Wdv}W_y \right)^{\beta} E_{\alpha,\beta}(x,y)\right)
         \\&=:\sum_{|\alpha|\leq 1} \delta^{1-|\alpha|}\left( \delta^{\Wdv}W \right)^{\alpha} \Et_{\alpha}(x,y).
    \end{split}
    \end{equation*}
    By \ref{Item::Spaces::Elem::Elem::LinearComb} and \ref{Item::Spaces::Elem::Elem::DerivFunction},
    \begin{equation*}
    \begin{split}
         &\left\{ \left( \Et_{\alpha}, \delta \right) : (E,\delta)\in \sE, |\alpha|\leq 1\right\}\in \ElemzWWdvOmegaVol{\Compact},
    \end{split}
    \end{equation*}
    establishing \eqref{Eqn::Spaces::Elem::Elem::PullOutNDerivs::Left}.
    The proof for \eqref{Eqn::Spaces::Elem::Elem::PullOutNDerivs::Right} is similar, using Remark \ref{Rmk::Spaces::Elem::Elem::IntegrateByPartsWithoutBoundary}.

    Finally, we turn to \ref{Item::Spaces::Elem::Elem::DerivFuncAribtrarySection}.
    Since \(Y\in \FilteredSheafF[d_0][\Omega']\subseteq \FilteredSheafF[d_0][\Omega\cap \Omega']\),
    we have \(Y=\sum_{\Wdv_j\leq d_0} c_j W_j\), where \(c_j\in \CinftySpace[\Omega\cap \Omega']\).
    From here,     \ref{Item::Spaces::Elem::Elem::DerivFuncAribtrarySection}
    follows from \ref{Item::Spaces::Elem::Elem::LinearComb}, \ref{Item::Spaces::Elem::Elem::MultBySmooth},
    and \ref{Item::Spaces::Elem::Elem::DerivFunction}.
\end{proof}

The next lemma completes the proof of Lemma \ref{Lemma::Spaces::LP::ElemDoesntDependOnChoices}.

\begin{lemma}\label{Lemma::Spaces::Elem::Elem::ElemDoesntDependOnChoices}
    Let \(\FilteredSheafG\), \(\Omegah\), \(\ZZde\), and \(\Volh\) be as in Lemma \ref{Lemma::Spaces::LP::ElemDoesntDependOnChoices}.
    Then,
    \begin{equation*}
        \Elemz{\WWdv, \Omega,\Vol}{\Compact}=\Elemz{\ZZde,\Omegah,\Volh}{\Compact}.
    \end{equation*}
\end{lemma}
\begin{proof}
    We show 
    \(\Elemz{\WWdv, \Omega,\Vol}{\Compact}\subseteq \Elemz{\ZZde,\Omegah,\Volh}{\Compact}\),
    and the reverse containment follows by symmetry.
    Let \(\sE\in \Elemz{\WWdv, \Omega,\Vol}{\Compact}\); we will show \(\sE\in \Elemz{\ZZde,\Omegah,\Volh}{\Compact}\).
    Lemma \ref{Lemma::Spaces::Elem::PreElem::PreElemDoesntDependOnChoices} shows
    \(\sE\in \PElemz{\ZZde,\Omegah,\Volh}{\Compact}\).
    We proceed as in Remark \ref{Rmk::Spaces::Elem::Elem::HowToProveElem}; so the goal is to show
    that for \((E,\delta)\in \sE\), \(E\) can be written as
    \begin{equation}\label{Eqn::Spaces::Elem::Elem::ElemDoesntDependOnChoices::ToShowForm}
        E(x,y) = \sum_{|\alpha|,|\beta|\leq 1} \delta^{2-|\alpha|-|\beta|} \left( \delta^{\Zde} Z_x \right)^{\alpha} \left( \delta^{\Zde} Z_y \right)^{\beta} E_{\alpha,\beta}(x,y),
    \end{equation}
    where
    \begin{equation*}
        \left\{ (E_{\alpha,\beta},\delta) : (E,\delta)\in \sE, |\alpha|,|\beta|\leq 1 \right\}\in \Elemz{\WWdv, \Omega,\Vol}{\Compact},
    \end{equation*}
    and the result will follow.

    Each 
    \(W_j\in \FilteredSheafF[\Wdv_j][\Omega]\subseteq \LieFilteredSheafF[\Wdv_j][\Omega]=\LieFilteredSheafG[\Wdv_j][\Omega]\).
    Letting \(\GenZZde\) be as in Lemma \ref{Lemma::Filtrations::GeneratorsForLieFiltration} that lemma implies
    each \(W_j\) can be written as a finite sum
    \begin{equation*}
        W_j=\sum a_j^k V_k, \quad \left( V_k,g_k \right)\in \GenZZde,g_k\leq \Wdv_j, a_j^k\in \CinftySpace[\Omega\cap \Omegah]. 
    \end{equation*}
    It follows that, for \(|\alpha|= 1\), we may write
    \begin{equation}\label{Eqn::Spaces::Elem::Elem::ElemDoesntDependOnChoices::WInTermsOfZ}
        \left( \delta^{\Wdv}W \right)^{\alpha} = \sum_{1\leq |\gamma|\leq N} \delta^{c_1(\alpha,\gamma)} a_{\alpha}^{\gamma}\left( \delta^{\Zde}Z \right)^{\gamma},\quad a_\alpha^\gamma\in\CinftySpace[\Omega\cap\Omegah],\: c_1(\alpha,\gamma)\geq 1,
    \end{equation}
    for some fixed \(N\), and similarly by reversing the roles of \(\FilteredSheafF\) and \(\FilteredSheafG\),
    \begin{equation}\label{Eqn::Spaces::Elem::Elem::ElemDoesntDependOnChoices::ZInTermsOfW}
        \left( \delta^{\Zde}Z \right)^{\alpha} = \sum_{1\leq |\gamma|\leq N} \delta^{c_2(\alpha,\gamma)} b_{\alpha}^{\gamma}\left( \delta^{\Wdv}W \right)^{\gamma},\quad b_\alpha^\gamma\in\CinftySpace[\Omega\cap\Omegah],\: c_2(\alpha,\gamma)\geq 1.
    \end{equation}
    % Here, when \(|\alpha|=0\), \eqref{Eqn::Spaces::Elem::Elem::ElemDoesntDependOnChoices::WInTermsOfZ} and 
    % \eqref{Eqn::Spaces::Elem::Elem::ElemDoesntDependOnChoices::ZInTermsOfW} are trivial.

    Let \(g_1,g_2\in \CinftySpace[\Omega\cap \Omegah]\).
    We claim, \(\forall \alpha,\beta\),
    \begin{equation}\label{Eqn::Spaces::Elem::Elem::ElemDoesntDependOnChoices::DerivsKeepBounded}
        \left\{ \left( g_1(x)g_2(y)\left( \delta^{\Zde}Z_x \right)^{\alpha} \left( \delta^{\Zde}Z_y \right)^{\beta} E(x,y),\delta \right) : (E,\delta)\in \sE \right\}\in \ElemzWWdvOmegaVol{\Compact}.
    \end{equation}
    It suffices to prove \eqref{Eqn::Spaces::Elem::Elem::ElemDoesntDependOnChoices::DerivsKeepBounded} for \(|\alpha|=1\), \(|\beta|=0\)
    and \(|\alpha|=0\), \(|\beta|=1\) as the general result then follows by a simple induction.
    In either of these two cases, \eqref{Eqn::Spaces::Elem::Elem::ElemDoesntDependOnChoices::DerivsKeepBounded}
    follows from \eqref{Eqn::Spaces::Elem::Elem::ElemDoesntDependOnChoices::ZInTermsOfW}
    and applying Lemma \ref{Lemma::Spaces::Elem::Elem::MainPropsLemma}
    \ref{Item::Spaces::Elem::Elem::LinearComb}, \ref{Item::Spaces::Elem::Elem::MultBySmooth}, and
    \ref{Item::Spaces::Elem::Elem::DerivFunction}.

    Consider, using Definition \ref{Defn::Spaces::LP::ElemWWdv}, we have for \((E,\delta)\in \sE\),
    using \eqref{Eqn::Spaces::Elem::Elem::ElemDoesntDependOnChoices::WInTermsOfZ}
    \begin{equation}\label{Eqn::Spaces::Elem::Elem::ElemDoesntDependOnChoices::LotsOfZs}
    \begin{split}
         & E(x,y) = \sum_{|\alpha|,|\beta|\leq 1}\delta^{2-|\alpha|-|\beta|} \left( \delta^{\Wdv}W_x \right)^{\alpha} \left( \delta^{\Wdv}W_y \right)^{\beta} E_{\alpha,\beta}(x,y)
         \\&=\sum_{|\alpha|,|\beta|= 1} \sum_{1\leq |\gamma_1|,|\gamma_2|\leq N}  \delta^{c_1(\alpha,\gamma_1)} a_\alpha^{\gamma_1}(x) \left( \delta^{\Zde}Z_x \right)^{\gamma_1} \delta^{c_1(\alpha,\gamma_2)} a_\alpha^{\gamma_2}(y) \left( \delta^{\Zde}Z_y \right)^{\gamma_2}E_{\alpha,\beta}(x,y)
         \\&\quad +\sum_{|\alpha|= 1} \sum_{1\leq |\gamma_1|\leq N}\delta^{1+c_1(\alpha,\gamma_1)} a_\alpha^{\gamma_1}(x) \left( \delta^{\Zde}Z_x \right)^{\gamma_1} E_{\alpha,\{\}}(x,y)
         \\&\quad + \sum_{|\beta|= 1} \sum_{1\leq |\gamma_2|\leq N}  \delta^{1+c_1(\alpha,\gamma_2)} a_\alpha^{\gamma_2}(y) \left( \delta^{\Zde}Z_y \right)^{\gamma_2}E_{\{\},\beta}(x,y)
         \\&\quad+\delta^2 E_{\{\},\{\}}(x,y).
    \end{split}
    \end{equation}
    The fourth term on the right-hand side of \eqref{Eqn::Spaces::Elem::Elem::ElemDoesntDependOnChoices::LotsOfZs}
    is already of the desired form \eqref{Eqn::Spaces::Elem::Elem::ElemDoesntDependOnChoices::ToShowForm}.
    The first three terms on the right-hand side of \eqref{Eqn::Spaces::Elem::Elem::ElemDoesntDependOnChoices::LotsOfZs}
    are all similar, so we describe only the second as that contains the main ideas.
    Fix \(|\alpha|=1\) and \(1\leq |\gamma_1|\leq N\).  Write \(\gamma_1=(k,\gamma')\), where \(|\gamma'|\leq N-1\).
    \begin{equation}\label{Eqn::Spaces::Elem::Elem::ElemDoesntDependOnChoices::Tmp1}
    \begin{split}
        &\delta^{1+c_1(\alpha,\gamma_1)} a_\alpha^{\gamma_1}(x) \left( \delta^{\Zde}Z_x \right)^{\gamma_1} E_{\alpha,\{\}}(x,y)
        \\&=\delta^{1+c_1(\alpha,\gamma_1)} \left( \delta^{\Zde_k}Z_k \right) a_\alpha^{\gamma_1}(x) \left( \delta^{\Zde}Z_x \right)^{\gamma'} E_{\alpha,\{\}}(x,y)
        \\&\quad+\delta^{1+\Zde_k+c_1(\alpha,\gamma_1)} \left( Z_k a_\alpha^{\gamma_1}(x) \right)  \left( \delta^{\Zde}Z_x \right)^{\gamma'} E_{\alpha,\{\}}(x,y).
    \end{split}
    \end{equation}
    Using that \(\Zde_k\geq 1\),
    the second-term on the right-hand side of \eqref{Eqn::Spaces::Elem::Elem::ElemDoesntDependOnChoices::Tmp1}
    is of the desired form \eqref{Eqn::Spaces::Elem::Elem::ElemDoesntDependOnChoices::ToShowForm} (with \(|\alpha|=|\beta|=0\)),
    by \eqref{Eqn::Spaces::Elem::Elem::ElemDoesntDependOnChoices::DerivsKeepBounded}.
    The first term on the right-hand size of \eqref{Eqn::Spaces::Elem::Elem::ElemDoesntDependOnChoices::Tmp1}
    is of the desired form \eqref{Eqn::Spaces::Elem::Elem::ElemDoesntDependOnChoices::ToShowForm} (with \(|\alpha|=1\), \(|\beta|=0\))
    by \eqref{Eqn::Spaces::Elem::Elem::ElemDoesntDependOnChoices::DerivsKeepBounded}.
    % It suffices to prove the result in the case \(\FilteredSheafG=\LieFilteredSheafF\)
    % (see Proposition \ref{Prop::Filtrations::RestrictingFiltrations::LieFIsHormanderFiltration});
    % and we henceforth make this substitution, so that \(\LieFilteredSheafF\big|_{\Omegah}=\FilteredSheafGenBy{\ZZde}\).

    % Let \(\sE\in \Elemz{\WWdv, \Omega,\Vol}{\Compact}\). We will show
    % \(\sE\in \Elemz{\ZZde,\Omegah,\Volh}{\Compact}\) by using the argument in Remark \ref{Rmk::Spaces::Elem::Elem::HowToProveElem}.
    % Since each \(W_j\in \FilteredSheafF[\Wdv_j][\Omega]\subseteq \LieFilteredSheafF[\Wdv_j][\Omega\cap \Omegah]\),
    % we have
    % \begin{equation*}
    %     W_j = \sum_{\Zde_k\leq \Wdv_j} a_j^k Z_k=\sum_{\Zde_k\leq \Wdv_j} Z_k\Mult{a_j^k}+\Mult{f_j},\quad a_j^k,f\in \CinftySpace[\Omega\cap \Omegah].
    % \end{equation*}
    % As a consequence, we have for \(|\alpha|=1\)
    % \begin{equation*}
    %     \left( \delta^{\Wdv} W \right)^{\alpha}=\sum_{|\beta|=1} a_{\alpha,\beta} \delta^{c_1(\alpha,\beta)}\left( \delta^{\Zde} Z \right)^{\beta} a_{\alpha,\beta}
    %     =\sum_{|\beta|=1} \delta^{c_1(\alpha,\beta)}\left( \delta^{\Zde} Z \right)^{\beta} a_{\alpha,\beta} +\delta^{c_2(\alpha,\beta)}f_{\alpha,\beta},
    % \end{equation*}
    % where \(c_1(\alpha,\beta), c_2(\alpha,\beta)\geq 0\), and \(a_{\alpha,\beta},f_{\alpha,\beta}\in \CinftySpace[\Omega\cap \Omegah]\).
    
\end{proof}

The corresponding definition of \(\ElemzF{\Compact}\) in \cite[Definition 5.2.8]{StreetMaximalSubellipticity} is slightly
different than the one here (Definition \ref{Defn::Spaces::LP::ElemWWdv}): in \cite{StreetMaximalSubellipticity}
the definition treats \(E\) as an operator instead of a function.  These two definitions are equivalent, as the next lemma shows.

\begin{lemma}\label{Lemma::Spaces::Elem::Elem::ElemzhF}
    Let \(\ElemzhF{\Compact}\) be the largest set of subsets of \(\CinftyCptSpace[\ManifoldN\times \ManifoldN]\times (0,1]\)
    satisfying
    \begin{enumerate}[(a)]
        \item\label{Item::Spaces::Elem::Elem::ElemzhF::ElemzhFIsPElem} \(\ElemzhF{\Compact}\subseteq \PElemzF{\Compact}\).
        \item\label{Item::Spaces::Elem::Elem::ElemzhF::ElemzhFIsDerivs} \(\forall\sE \in \ElemzhF{\Compact}\), \(\forall (E,2^{-j})\in \sE\), we have
            \begin{equation*}
                E=\sum_{|\alpha|,|\beta|\leq 1} 2^{-j(2-|\alpha|-|\beta|)}\left(2^{-j\Wdv}W  \right)^{\alpha} E_{\alpha,\beta} \left( 2^{-j\Wdv}W \right)^{\beta},
            \end{equation*}
            where
            \begin{equation*}
                \left\{ \left( E_{\alpha,\beta},2^{-j} \right) : (E,2^{-j})\in \sE, |\alpha|,|\beta|\leq 1 \right\}\in \ElemzhF{\Compact}.
            \end{equation*}
    \end{enumerate}
    Then, \(\ElemzF{\Compact}=\ElemzhF{\Compact}\).
\end{lemma}
\begin{proof}
    We begin with \(\ElemzF{\Compact}\subseteq \ElemzhF{\Compact}\), and we do this by showing
    \ref{Item::Spaces::Elem::Elem::ElemzhF::ElemzhFIsPElem} and \ref{Item::Spaces::Elem::Elem::ElemzhF::ElemzhFIsDerivs}
    hold with \(\ElemzhF{\Compact}\) replaced by \(\ElemzF{\Compact}\)--then, since \(\ElemzhF{\Compact}\) is the largest
    set satisfying these properties, we have \(\ElemzF{\Compact}\subseteq \ElemzhF{\Compact}\).
    \ref{Item::Spaces::Elem::Elem::ElemzhF::ElemzhFIsPElem} is just a restatement of 
    Definition \ref{Defn::Spaces::LP::ElemWWdv} \ref{Item::Spaces::LP::ElemWWdv::ElemArePreElem}.
    \ref{Item::Spaces::Elem::Elem::ElemzhF::ElemzhFIsDerivs} follows by applying 
    Proposition \ref{Prop::Spaces::Elem::Elem::MainProps} \ref{Item::Spaces::Elem::Elem::PullOutNDerivs} 
    (using both \eqref{Eqn::Spaces::Elem::Elem::PullOutNDerivs::Left} and \eqref{Eqn::Spaces::Elem::Elem::PullOutNDerivs::Right}, with \(N=1\)),
    completing the proof of \(\ElemzF{\Compact}\subseteq \ElemzhF{\Compact}\).

    For the reverse containment, the proof is similar. We show \(\ElemzhF{\Compact}\subseteq \ElemzF{\Compact}\)
    by showing that \(\ElemzhF{\Compact}\) satisfies   
    Definition \ref{Defn::Spaces::LP::ElemWWdv} \ref{Item::Spaces::LP::ElemWWdv::ElemArePreElem} and \ref{Item::Spaces::LP::ElemWWdv::ElemArePreDerivs},
    and the containment will follows. As in the previous part,  Definition \ref{Defn::Spaces::LP::ElemWWdv} \ref{Item::Spaces::LP::ElemWWdv::ElemArePreElem}
    is just a restatement of \ref{Item::Spaces::Elem::Elem::ElemzhF::ElemzhFIsPElem}.
    By \ref{Item::Spaces::Elem::Elem::ElemzhF::ElemzhFIsDerivs}, for \((E,2^{-j})\in \sE\), 
    using Remark \ref{Rmk::Spaces::Elem::Elem::IntegrateByPartsWithoutBoundary}, we may write
    \begin{equation}\label{Eqn::Spaces::Elem::Elem::ElemzhF::PulledOutSomeDerivs}
    \begin{split}
         &E=\sum_{|\alpha|,|\beta|\leq 1} 2^{-j(2-|\alpha|-|\beta|)}\left(2^{-j\Wdv}W  \right)^{\alpha} E_{\alpha,\beta} \left( 2^{-j\Wdv}W \right)^{\beta}
         \\&=-\sum_{|\alpha|\leq 1,|\beta|=1} 2^{-j(1-|\alpha|)}\left(2^{-j\Wdv}W_x  \right)^{\alpha} \left( 2^{-j\Wdv}W_y \right)^{\beta}E_{\alpha,\beta}(x,y) 
         \\&\quad + \sum_{|\alpha|\leq 1,|\beta|=1} 2^{-j(1-|\alpha|+\DegWdv{\beta})}\left(2^{-j\Wdv}W_x  \right)^{\alpha} f_{\beta}(y)E_{\alpha,\beta}(x,y)
         \\&\quad + \sum_{|\alpha|\leq 1} 2^{-j(2-|\alpha|)} \left( 2^{-j\Wdv}W_x  \right)^{\alpha} E_{\alpha,\{\}}(x,y),
    \end{split}
    \end{equation}
    where \(f_\beta\in \CinftySpace[\Omega]\).  The first and third term on the right-hand side of \eqref{Eqn::Spaces::Elem::Elem::ElemzhF::PulledOutSomeDerivs}
    are of the desired form. For the second, we use that Proposition \ref{Prop::Spaces::Elem::Elem::MainProps} \ref{Item::Spaces::Elem::Elem::MultBySmooth}
    holds with \(\ElemzF{\Compact}\) replaced by \(\ElemzhF{\Compact}\) (with the same proof--see also the proof of 
    \cite[Proposition 5.5.5 (c)]{StreetMaximalSubellipticity}); showing that the second term on the
    right-hand side of \eqref{Eqn::Spaces::Elem::Elem::ElemzhF::PulledOutSomeDerivs} is of the desired form, completing the proof.
\end{proof}

If \(\sE\in \ElemzF{\Compact}\). If \((E_1,2^{-j}),(E_2,2^{-k})\in \sE\), then \(E_1\) and \(E_2\) are
``almost orthogonal'' as the next two results show.

\begin{lemma}\label{Lemma::Spaces::Elem::ELem::ProdOfElemAsSmallSumOfProds}
    Let \(\sE\in \ElemzF{\Compact}\). Then, \(\forall N\in \Zgeq\), \(\exists K=K(N)\in \Zgeq\), 
    \(\exists \sE_N\in \ElemzF{\Compact}\), \(\forall (E_j,2^{-j}), (F_k,2^{-k})\in \sE\),
    \begin{equation*}
        F_k E_j=\sum_{l=1}^K 2^{-N|j-k|} F_{k,l} E_{j,l},
    \end{equation*}
    where \(\left( E_{j,l},2^{-j} \right), \left( F_{k,l},2^{-k} \right)\in \sE_N\).
\end{lemma}
\begin{proof}
    We prove the result for \(k\geq j\); the proof for \(j>k\) is similar and we leave it to the reader.
    By Proposition \ref{Prop::Spaces::Elem::Elem::MainProps} \ref{Item::Spaces::Elem::Elem::PullOutNDerivs}
    we may write
    \begin{equation*}
        F_k=\sum_{|\alpha|\leq N} 2^{(|\alpha|-N)k}F_{k,\alpha} \left( 2^{-k\Wdv}W \right)^{\alpha},
    \end{equation*}
    where \(\left\{ \left( F_{k,\alpha},2^{-k} \right) : (F_k,2^{-k})\in \sE,|\alpha|\leq N \right\}\in \ElemzF{\Compact}\).
    We, therefore have
    \begin{equation*}
    \begin{split}
         &F_k E_j
         =\sum_{|\alpha|\leq N} 2^{(|\alpha|-N)k}F_{k,\alpha} \left( 2^{-k\Wdv}W \right)^{\alpha} E_j
         =\sum_{|\alpha|\leq N} 2^{(|\alpha|-N)k} 2^{(j-k)\DegWdv{\alpha}}F_{k,\alpha} \left( \left( 2^{-j\Wdv}W \right)^{\alpha} E_j \right).
    \end{split}
    \end{equation*}
    Proposition  \ref{Prop::Spaces::Elem::Elem::MainProps} \ref{Item::Spaces::Elem::Elem::DerivOp} shows
    \(\left\{ \left( \left( 2^{-j\Wdv}W \right)^{\alpha} E_j ,2^{-j} \right) : (E,2^{-j})\in \sE, |\alpha|\leq N \right\}\in \ElemzF{\Compact}\).
    Since \(2^{2^{(|\alpha|-N)k} 2^{(j-k)\DegWdv{\alpha}}}\leq 2^{-N|j-k|}\), the claim follows.
\end{proof}

\begin{lemma}\label{Lemma::Spaces::ELem::Elem::ComposeElem}
    Let \(\sE\in \ElemzF{\Compact}\). Then, \(\forall N\geq 0\),
    \begin{equation*}
        \left\{ \left( 2^{N|j-k|}E_jE_k, 2^{-j} \right), \left( 2^{N|j-k|}E_jE_k, 2^{-k} \right) : (E_j,2^{-j}), (E_k,2^{-k})\in \sE \right\}\in \ElemzF{\Compact}.
    \end{equation*}
\end{lemma}
\begin{proof}
    This follows just as in \cite[Proposition 5.5.11]{StreetMaximalSubellipticity},
    using Lemma \ref{Lemma::Spaces::Elem::Elem::ElemzhF}
    and using
    Proposition \ref{Prop::Spaces::Elem::Elem::MainProps} 
    and Lemmas \ref{Lemma::Spaces::Elem::PreElem::ComposePreElemBound} and \ref{Lemma::Spaces::Elem::PreElem::PreElemBoundDifferentScales}
    in place of the corresponding results in \cite{StreetMaximalSubellipticity}.
\end{proof}

\begin{proposition}\label{Prop::Spaces::Elem::Elem::ConvergenceOfElemOps}
    Let \(t\in \R\) and \(\left\{ \left( E_j,2^{-j} \right) : j\in \Zgeq \right\}\in \ElemzF{\Compact}\).
    Then, for \(f\in \CinftySpace[\ManifoldN]\), the sum
    \(\sum_{j\in \Zgeq} 2^{jt}E_j f\) converges in \(\CinftyCptSpace[\ManifoldN]\) and defines a continuous
    linear map \(\CinftySpace[\ManifoldN]\rightarrow \CinftyCptSpace[\ManifoldN]\).
    For \(u\in \DistributionsZeroN\), the sum
    \(\sum_{j\in \Zgeq} 2^{jt}E_j u\) converges in \(\DistributionsZeroN\) and defines a continuous linear map
    \(\DistributionsZeroN\rightarrow \DistributionsZeroN\).
\end{proposition}
\begin{proof}
    Let \(\Omega_1\Subset \Omega\) be open with \(\Compact\Subset\Omega_1\).
    We will show \(\forall \alpha\), \(\forall N\), exists \(M\)
    such that
    \begin{equation}\label{Eqn::Spaces::Elem::Elem::ConvergenceOfElemOps::ToShow1}
        \sup_{x} \left| 2^{jt} W^{\alpha} E_j f(x) \right|\lesssim 2^{-Nj} \CmNorm{f}{M}[\Omega_1], \quad\forall f\in \CinftySpace[\Omega],
    \end{equation}
    \begin{equation}\label{Eqn::Spaces::Elem::Elem::ConvergenceOfElemOps::ToShow1::Transpose}
        \sup_{x} \left| 2^{jt} W^{\alpha} E_j^{\transpose} f(x) \right|\lesssim 2^{-Nj} \CmNorm{f}{M}[\Omega_1], \quad\forall f\in \TestFunctionsZero[\Omega],
    \end{equation}
    where \(E_j^{\transpose}\) denotes the transpose of \(E_j\)
     with respect to the usual pairing of \(\DistributionsZeroN\) and \(\TestFunctionsZeroN\).

    Using that \(W_1,\ldots, W_r\) satisfy H\"ormander's condition on a neigborhood of \(\Omega\), and 
    since \(\supp(E_j)\subseteq \Compact\times \Compact\subseteq \Omega\times \Omega\), \(\forall j\), the convergence in \(\CinftyCptSpace[\ManifoldN]\)
    and the fact that the sum defines a continuous
    linear map \(\CinftySpace[\ManifoldN]\rightarrow \CinftyCptSpace[\ManifoldN]\) 
    follows immediately from \eqref{Eqn::Spaces::Elem::Elem::ConvergenceOfElemOps::ToShow1} (with \(N=1\)).
    Similarly, that  \(\sum_{j\in \Zgeq} 2^{jt}E_j u\) converges in \(\DistributionsZeroN\) and defines a continuous
    \(\DistributionsZeroN\rightarrow \DistributionsZeroN\) follows easily from
    \eqref{Eqn::Spaces::Elem::Elem::ConvergenceOfElemOps::ToShow1::Transpose} (with \(N=1\)).

    Note that
    \begin{equation*}
        2^{jt}W^{\alpha}E_j = 2^{j(t+\DegWdv{\alpha})} \left( 2^{-j\Wdv}W \right)^{\alpha}E_j
    \end{equation*}
    and by Proposition \ref{Prop::Spaces::Elem::Elem::MainProps} \ref{Item::Spaces::Elem::Elem::DerivOp},
    \(\left\{ \left(\left( 2^{-j\Wdv}W \right)^{\alpha}E_j, 2^{-j}  \right) : j\in \Zgeq \right\}\in \ElemzF{\Compact}\).
    And so, by replacing \(t\) with \(t+\DegWdv{\alpha}\), it suffices to prove \eqref{Eqn::Spaces::Elem::Elem::ConvergenceOfElemOps::ToShow1}
    in the case \(|\alpha|=0\).
    Similarly, since
    \begin{equation*}
        \left[ \left( 2^{-j\Wdv}W \right)^{\alpha} E_j^{\transpose} \right](y,x)
        =\left( 2^{-j\Wdv}W_y \right) E_j(x,y),
    \end{equation*}
    it follows from Proposition \ref{Prop::Spaces::Elem::Elem::MainProps} \ref{Item::Spaces::Elem::Elem::DerivFunction}
    that it also suffices to prove \eqref{Eqn::Spaces::Elem::Elem::ConvergenceOfElemOps::ToShow1::Transpose} in the case
    \(|\alpha|=0\). I.e., we wish to show
    \begin{equation}\label{Eqn::Spaces::Elem::Elem::ConvergenceOfElemOps::ToShow1::alpha0}
        \sup_{x} \left| 2^{jt}  E_j f(x) \right|\lesssim 2^{-Nj} \CmNorm{f}{M}[\Omega_1], \quad\forall f\in \CinftySpace[\Omega],
    \end{equation}
    \begin{equation}\label{Eqn::Spaces::Elem::Elem::ConvergenceOfElemOps::ToShow1::Transpose::alpha0}
        \sup_{x} \left| 2^{jt} E_j^{\transpose} f(x) \right|\lesssim 2^{-Nj} \CmNorm{f}{M}[\Omega_1], \quad\forall f\in \TestFunctionsZero[\Omega],
    \end{equation}

    Next, we claim that it suffices to prove \eqref{Eqn::Spaces::Elem::Elem::ConvergenceOfElemOps::ToShow1::alpha0}
    and \eqref{Eqn::Spaces::Elem::Elem::ConvergenceOfElemOps::ToShow1::Transpose::alpha0}
    in the case \(N=t=0\).
    Indeed, by Proposition \ref{Prop::Spaces::Elem::Elem::MainProps} \ref{Item::Spaces::Elem::Elem::PullOutNDerivs},
    for \(M\geq 1\) fixed,
    we may write for \(f\in \CinftySpace[\Omega]\),
    \begin{equation*}
        E_j f = \sum_{|\alpha|\leq M}2^{-j(M-|\alpha|)} \Et_{j,\alpha} \left( 2^{-j\Wdv}W \right)^{\alpha} f 
        =\sum_{|\alpha|\leq M} 2^{-j(M-|\alpha|+\DegWdv{\alpha})} \Et_{j,\alpha} W^{\alpha}f,
    \end{equation*}
    where \(\left\{ \left( \Et_{j,\alpha},2^{-j} \right) : j\in \Zgeq \right\}\in \ElemzF{\Compact}\).
    Since \(W^{\alpha}f\in \CinftySpace[\Omega]\) and \(2^{-j(M-|\alpha|+\DegWdv{\alpha})}\leq 2^{-Mj}\),
    we see that it suffices to prove \eqref{Eqn::Spaces::Elem::Elem::ConvergenceOfElemOps::ToShow1::alpha0} with
    \(N\) replaced by \(N-M\). Taking \(M\geq N+t\), we see that it 
    suffices to prove \eqref{Eqn::Spaces::Elem::Elem::ConvergenceOfElemOps::ToShow1::alpha0}
    with \(N=t=0\), as claimed.
    Similarly, again using Proposition \ref{Prop::Spaces::Elem::Elem::MainProps} \ref{Item::Spaces::Elem::Elem::PullOutNDerivs} we may write
    \begin{equation*}
        E_j^{\transpose} f = \sum_{|\alpha|\leq M}2^{-j(M-|\alpha|)} E_{j,\alpha}^{\transpose} \left[ \left( 2^{-j\Wdv}W \right)^{\alpha} \right]^{\transpose} f 
        =\sum_{|\alpha|\leq M} 2^{-j(M-|\alpha|+\DegWdv{\alpha})} E_{j,\alpha}^{\transpose} \left[ W^{\alpha} \right]^{\transpose}f,
    \end{equation*}
    where \(\left\{ \left( E_{j,\alpha},2^{-j} \right) : j\in \Zgeq \right\}\in \ElemzF{\Compact}\).
    Here we have used that \(f\in \TestFunctionsZero[\Omega]\), so there are no boundary terms in the integration by parts
    defining \(\left[ W^{\alpha} \right]^{\transpose}\).
    As before, by taking \(M\geq t+N\), we see that it suffices to prove 
    \eqref{Eqn::Spaces::Elem::Elem::ConvergenceOfElemOps::ToShow1::Transpose::alpha0}
    in the case \(N=t=0\).

    Finally, \eqref{Eqn::Spaces::Elem::Elem::ConvergenceOfElemOps::ToShow1::alpha0} and 
    \eqref{Eqn::Spaces::Elem::Elem::ConvergenceOfElemOps::ToShow1::Transpose::alpha0} in the case \(N=t=0\)
    follow from Lemma \ref{Lemma::Spaces::Elem::PElem::PElemOpsBoundedOnLp}, completing the proof.
\end{proof}
    
    \subsection{Schwartz space}
    In this section, we describe some  subspaces of 
Schwartz space (\(\SchwartzSpaceRn\)) which are important for our proofs.
For \(f\in \SchwartzSpaceRn\), we let \(\fh\in \SchwartzSpaceRn\) denote the Fourier transform.
It is a standard results that the fourier transform is an automorphism of 
the Fr\'echet space
\(\SchwartzSpaceRn\).
Let
\begin{equation}\label{Eqn::Spaces::Schwartz::DefnSchwartzSpacez}
    \SchwartzSpacezRn=\left\{ f\in \SchwartzSpaceRn:\int t^{\alpha} f(t)\: dt=0,\: \forall \alpha \right\},
\end{equation}
the space of Schwartz functions, all of whose moments vanish.

\begin{lemma}\label{Lemma::Spaces::Schwartz::FTOfSchwartzSpacez}
    For \(f\in \SchwartzSpaceRn\),
    we have \(f\in \SchwartzSpacezRn\) if and only if \(\partial_\xi^{\alpha}\fh(0)=0\), \(\forall \alpha\).
\end{lemma}
\begin{proof}
    This follows by taking the Fourier transform of the definition \eqref{Eqn::Spaces::Schwartz::DefnSchwartzSpacez}.
\end{proof}

\begin{lemma}\label{Lemma::Spaces::Schwartz::SchwartzSpacezPullOutDerivsOneWay}
    For \(f\in \SchwartzSpacezRn\), \(f=\sum_{j=1}^n \partial_{t_j} f_j\), where \(f_j\in \SchwartzSpacezRn\).
\end{lemma}
\begin{proof}
    We have,
    \begin{equation*}
        \fh(\xi)=\sum_{j=1}^n \xi_j \left( \frac{\xi_j}{|\xi|^2} \fh(\xi) \right),
    \end{equation*}
    where \(\frac{\xi_j}{|\xi|^2} \fh(\xi) \in \SchwartzSpaceRn\) and vanishes to infinite order at \(\xi=0\)
    by Lemma \ref{Lemma::Spaces::Schwartz::FTOfSchwartzSpacez}. Taking the inverse Fourier transform,
    and applying Lemma \ref{Lemma::Spaces::Schwartz::FTOfSchwartzSpacez}, completes the proof.
\end{proof}

Lemma \ref{Lemma::Spaces::Schwartz::SchwartzSpacezPullOutDerivsOneWay} is actually a characterization
of \(\SchwartzSpacezRn\), as the next lemma shows.

\begin{lemma}\label{Lemma::Spaces::Schwartz::SchwartzSpacezPullOutDerivsCharacterization}
    Let \(\sS\subseteq \SchwartzSpaceRn\) be the largest subset such that \(\forall f\in \sS\),
    \(f=\sum_{j=1}^n \partial_{t_j}f_j\), where \(f_j\in \sS\). Then, \(\sS=\SchwartzSpacezRn\).
\end{lemma}
\begin{proof}
    Lemma \ref{Lemma::Spaces::Schwartz::SchwartzSpacezPullOutDerivsOneWay} shows
    \(\SchwartzSpacezRn\subseteq \sS\).
    For the converse, let \(f\in \sS\) and fix \(\alpha\).
    Applying the definition of \(\sS\) \(|\alpha|+1\) times, we may write
    \(f=\sum_{|\beta|=|\alpha|+1} \partial_t^{\beta} f_\beta\), where \(f_\beta\in \sS\subseteq \SchwartzSpaceRn\).
    We have,
    \begin{equation*}
        \int t^{\alpha} f = \sum_{|\beta|=|\alpha|+1} \int t^{\alpha} \partial_t^{\beta} f_\beta(t)\: dt=0.
    \end{equation*}
    Since \(\alpha\) was arbitrary, we see \(f\in \SchwartzSpacezRn\), completing the proof.
\end{proof}

For the remainder of the section, we take \(n=1+q\) and work on \(\R^{1+q}=\R\times \R^q\), with coordinates
\(t=(t_0,t')\in \R\times \R^q\).
Lemma \ref{Lemma::Spaces::Schwartz::SchwartzSpacezPullOutDerivsCharacterization}
leads us to a convenient definition of our purposes.

\begin{definition}\label{Defn::Spaces::Schwartz::sT}
    Let \(\sT\subseteq \SchwartzSpaceRopq\)  be the largest subset such that \(\forall f\in \sT\),
    \begin{enumerate}[(i)]
        \item\label{Item::Spaces::Schwartz::sT::Support} \(\supp(f)\subseteq [0,\infty)\times \R^q\),
        \item\label{Item::Spaces::Schwartz::sT::Deriv} \(f=\sum_{j=0}^{q} \partial_{t_j} f_j\), \(f_j\in \sT\).
    \end{enumerate}
\end{definition}

\begin{proposition}\label{Prop::Spaces::Schwartz::CharacterizesT}
    \(\sT=\left\{ f\in \SchwartzSpacezRopq : \supp(f)\subseteq[0,\infty)\times \R^{q} \right\}\).
\end{proposition}

We prove Proposition \ref{Prop::Spaces::Schwartz::CharacterizesT} at the end of this section.
In the meantime, we state and prove several preliminary results which are useful for the proof of
Proposition \ref{Prop::Spaces::Schwartz::CharacterizesT} and later in the paper.

% To prove Proposition \ref{Prop::Spaces::Schwartz::CharacterizesT}, we need several preliminary results.

% Clearly \(\sT\) is a subspace of \(\SchwartzSpaceRn\).

\begin{lemma}\label{Lemma::Spaces::Schwartz::PuttingElementsInsT}
    We have,
    \begin{enumerate}[(i)]
        \item\label{Item::Spaces::Schwartz::PuttingElementsInsT::sTInSchwartzSpacez} \(\sT\subseteq \SchwartzSpacezRopq\).
        \item\label{Item::Spaces::Schwartz::PuttingElementsInsT::SzRtimesSRqInsT} If \(F_1\in \SchwartzSpacezR\) with \(\supp(F_1)\subseteq [0,\infty)\) and \(F_2\in \SchwartzSpace[\R^q]\), then \(F_1(t_0)F_2(t')\in \sT\).
        \item\label{Item::Spaces::Schwartz::PuttingElementsInsT::SRtimesSzRqInsT} If \(F_1\in \SchwartzSpaceR\) with \(\supp(F_1)\subseteq [0,\infty)\) and \(F_2\in \SchwartzSpacez[\R^q]\), then \(F_1(t_0)F_2(t')\in \sT\).
    \end{enumerate}
\end{lemma}

To prove Lemma \ref{Lemma::Spaces::Schwartz::PuttingElementsInsT}, we use the next two lemmas.

\begin{lemma}\label{Lemma::Spaces::Schwartz::IntegrateSchwartzSpacez}
    Let \(\phi \in \SchwartzSpacezR\). Then, \(\int_{-\infty}^t \phi(s)\: ds\in \SchwartzSpacezR\).
\end{lemma}
\begin{proof}
    By Lemma \ref{Lemma::Spaces::Schwartz::SchwartzSpacezPullOutDerivsOneWay}, \(\phi=\partial_t\phi_1\), where
    \(\phi_1\in \SchwartzSpacezR\).  We have \(\int_{-\infty}^t \phi(s)\: ds=\phi_1(t)\in \SchwartzSpacezR\).
\end{proof}
% \begin{proof}
%     First we claim \(\int_{-\infty}^t \phi(s)\: ds\in \SchwartzSpaceR\).
%     Indeed, for \(t\leq 0\) and \(N\geq 2\), we have
%     \begin{equation*}
%         \left| \int_{-\infty}^t \phi(s)\: ds \right|
%         \lesssim \int_{-\infty}^t \left( 1+|s| \right)^{-N}\: ds
%         \lesssim \left( 1+|t| \right)^{-N+1},
%     \end{equation*}
%     and for \(t>0\), using \(\int_{-\infty}^{\infty} \phi(s)\: ds=0\), we have
%     \begin{equation*}
%         \left| \int_{-\infty}^t \phi(s)\: ds \right|= \left| \int_t^{\infty}\phi(s) \right|\: ds \lesssim \left( 1+|t| \right)^{-N+1},
%     \end{equation*}
%     as before.
%     Clearly, for \(j\geq 1\), we have
%     \begin{equation*}
%         \left| \partial_t^j  \int_{-\infty}^t \phi(s)\: ds \right| = \left| \partial_t^{j-1}\phi(t) \right|\lesssim \left( 1+|t| \right)^{-N}.
%     \end{equation*}
%     Combining the above establishes \(\int_{-\infty}^t \phi(s)\: ds\in \SchwartzSpaceR\).

%     Now consider, for any \(j\geq 0\), using \(\phi \in \SchwartzSpacezR\),
%     \begin{equation*}
%         \int_{-\infty}^{\infty} t^{j} \int_{-\infty}^t \phi(s)\: ds\: dt
%         =\frac{-1}{j+1}\int_{-\infty}^{\infty} t^{j+1}\phi(t)\: dt =0.
%     \end{equation*}
%     The result follows.
% \end{proof}

\begin{lemma}\label{Lemma::Spaces::Schwartz::CharacterizesT0}
    Let \(\sT_0\subseteq \SchwartzSpaceR\) be the largest subset such that \(\forall f\in \sT_0\),
    \begin{enumerate}[(i)]
        \item\label{Item::Spaces::Schwartz::CharacterizesT0::Support} \(\supp(f)\subseteq [0,\infty)\),
        \item\label{Item::Spaces::Schwartz::CharacterizesT0::Deriv} \(f=\partial_t f_1\), \(f_0\in \sT_0\).
    \end{enumerate}
    Then, \(\sT_0=\left\{ f\in \SchwartzSpacezR : \supp(f)\subseteq [0,\infty) \right\}\).
\end{lemma}
\begin{proof}
    Lemma \ref{Lemma::Spaces::Schwartz::SchwartzSpacezPullOutDerivsCharacterization} 
    implies \(\sT_0\subseteq \SchwartzSpacezR\), and therefore, 
    \(\sT_0\subseteq \left\{ f\in \SchwartzSpacezR : \supp(f)\subseteq [0,\infty) \right\}\).
    Let \(g\in  \left\{ f\in \SchwartzSpacezR : \supp(f)\subseteq [0,\infty) \right\}\).
    By Lemma \ref{Lemma::Spaces::Schwartz::IntegrateSchwartzSpacez}, we have
    \(\int_{-\infty}^t g(s)\: ds\in \left\{ f\in \SchwartzSpacezR : \supp(f)\subseteq [0,\infty) \right\}\).
    Since \(g=\partial_t \int_{-\infty}^t g(s)\: ds\), this shows 
    \(\left\{ f\in \SchwartzSpacezR : \supp(f)\subseteq [0,\infty) \right\}\)
    satisfies the axiom \ref{Item::Spaces::Schwartz::CharacterizesT0::Deriv}.
    Since it clearly satisfies the axiom \ref{Item::Spaces::Schwartz::CharacterizesT0::Support},
    we have \(\left\{ f\in \SchwartzSpacezR : \supp(f)\subseteq [0,\infty) \right\}\subseteq \sT_0\),
    completing the proof.
\end{proof}

\begin{proof}[Proof of Lemma \ref{Lemma::Spaces::Schwartz::PuttingElementsInsT}]
    \ref{Item::Spaces::Schwartz::PuttingElementsInsT::sTInSchwartzSpacez} follows from 
    Lemma \ref{Lemma::Spaces::Schwartz::SchwartzSpacezPullOutDerivsCharacterization}.

    Let \(\sT_1:=\left\{ F_1(t_0)F_2(t') : F_1\in \SchwartzSpacezR, \supp(F_1)\subseteq [0,\infty), F_2\in \SchwartzSpace[\R^q] \right\}\).
    Lemma \ref{Lemma::Spaces::Schwartz::CharacterizesT0} applied to \(F_1\)
    shows that \(\sT_1\)
    satisfies the axiom Definition \ref{Defn::Spaces::Schwartz::sT} \ref{Item::Spaces::Schwartz::sT::Deriv}.
    Since \(\sT_1\) clearly also satisfy the axiom Definition \ref{Defn::Spaces::Schwartz::sT} \ref{Item::Spaces::Schwartz::sT::Support},
    we have \(\sT_1\subseteq \sT\), establishing \ref{Item::Spaces::Schwartz::PuttingElementsInsT::SzRtimesSRqInsT}.
    
    Let \(\sT_2:=\left\{ F_1(t_0)F_2(t') : F_1\in \SchwartzSpaceR, \supp(F_1)\subseteq [0,\infty), F_2\in \SchwartzSpacez[\R^q] \right\}\).
    Lemma \ref{Lemma::Spaces::Schwartz::SchwartzSpacezPullOutDerivsCharacterization} applied to \(F_2\)
    shows that \(\sT_2\) satisfies the axiom Definition \ref{Defn::Spaces::Schwartz::sT} \ref{Item::Spaces::Schwartz::sT::Deriv}.
    Since \(\sT_2\) clearly also satisfy the axiom Definition \ref{Defn::Spaces::Schwartz::sT} \ref{Item::Spaces::Schwartz::sT::Support},
    we have \(\sT_2\subseteq \sT\), establishing \ref{Item::Spaces::Schwartz::PuttingElementsInsT::SRtimesSzRqInsT}.
\end{proof}

% \begin{lemma}\label{Lemma::Spaces::Schwartz::MultiplysTbyPoly}
%     For \(\vsig(t)\in \sT\), we have \(t^{\alpha}\vsig(t)\in \sT\), \(\forall \alpha\).
% \end{lemma}
% \begin{proof}
%     It suffices to prove the result for \(|\alpha|=1\), as the general result follows from this and a simple induction.
%     We prove the result for \(t^{\alpha}=t_j\) for some fixed \(j\in \left\{ 0,\ldots, q \right\}\).

%     Set \(\sT_3:=\left\{ t_j f_1+ f_2: f_1,f_2\in \sT \right\}\). Note \(\sT\subseteq \sT_3\), since \(0\in \sT\).
%     We will show \(\sT_3=\sT\).  For \(g_1,g_2\in \sT\), we may write \(g_k=\sum_{l=0}^q \partial_{t_l} g_{k,l}\), \(k=1,2\),
%     where \(g_{k,l}\in \sT\). We also have \(g_{1,j}=\sum_{l=1}^q \partial_{t_l} g_{3,l}\), where \(g_{3,l}\in \sT\).
%     Combining this, we have
%     \begin{equation*}
%         t_j g_1 + g_2
%         =\sum_{l=0}^q \partial_{t_l} \left( t_j g_{1,l} + g_{2,l}  \right) - g_{1,j}
%         =\sum_{l=0}^q \partial_{t_l} \left( t_j g_{1,l} + g_{2,l} - g_{3,l}  \right).
%     \end{equation*}
%     Since \(\sT\) is a vector space, this shows \(\sT_3\) satisfies the axiom Definition \ref{Defn::Spaces::Schwartz::sT} \ref{Item::Spaces::Schwartz::sT::Deriv}.
%     Since \(\sT_3\) clearly also satisfy the axiom Definition \ref{Defn::Spaces::Schwartz::sT} \ref{Item::Spaces::Schwartz::sT::Support},
%     we have \(\sT_3\subseteq \sT\).

%     Finally, since \(t_j\vsig \in \sT_3=\sT\) for \(\vsig\in \sT\), this completes the proof.
% \end{proof}

Fix \(\Xdv_0, \Xdv_1,\ldots, \Xdv_q\in \Zg\) and for \(t\in \R^{1+q}\), set
\(
    2^{j\Xdv}(t_0,t_1,\ldots, t_q):=(2^{j\Xdv_0}t_0, 2^{j\Xdv_1}t_1,\ldots, 2^{j\Xdv_q}t_q),
\)
and for a function \(\vsig(t)\) set
\(\Dild{2^j}{\vsig}(t)=2^{j(\Xdv_0+\Xdv_1+\cdots+\Xdv_q)}\vsig(2^{jd}t)\).
Note that \(\int \Dild{2^j}{\vsig} = \int \vsig\).

%About to embark on changes to multiplication result to make it work for Zygmund spaces.
\begin{proposition}\label{Prop::Spaces::Schwartz::SumToIdentity}
    There exists \(\vsig_0\in \SchwartzSpaceRopq\) such that the following holds.
    \begin{enumerate}[(i)]
        \item\label{Item::Spaces::Schwartz::SumToIdentity::Support} \(\supp(\vsig_0)\subseteq [0,\infty)\times \R^q\).
        \item\label{Item::Spaces::Schwartz::SumToIdentity::sT} Let \(\vsig_1:=\Dild{2}{\vsig_0}-\vsig_0\). Then \(\vsig_1\in \sT\).
        \item\label{Item::Spaces::Schwartz::SumToIdentity::SumToDelta} \(\delta_0(t)=\vsig_0(t)+\sum_{j=1}^\infty \Dild{2^j}{\vsig_1}(t)\), where \(\delta_0(t)\)
            denotes the Dirac \(\delta\) function at \(0\) and the sum is taken in the sense of tempered distributions.
        \item\label{Item::Spaces::Schwartz::SumToIdentity::IdentityOnSchwartz} For \(f\in \SchwartzSpacezRopq\), we have \(f=\sum_{j=-\infty}^{\infty} f*\Dild{2^j}{\vsig_1}\)
    with convergence in \(\SchwartzSpacezRopq\).
    \end{enumerate}
    % There exists \(\vsig_0\in \SchwartzSpaceRopq\) with \(\supp(\vsig_0)\subseteq [0,\infty)\times \R^q\)
    % and \(\vsig_1\in \sT\) such that \(\delta_0(t)=\vsig_0(t)+\sum_{j=1}^\infty \Dild{2^j}{\vsig_1}(t)\),
    % where the convergence is in the sense of tempered distributions, and \(\delta_0(t)\)
    % denotes the Dirac \(\delta\) function at \(0\).
    % Furthermore, for \(f\in \SchwartzSpacezRopq\), we have \(f=\sum_{j=-\infty}^{\infty} f*\Dild{2^j}{\vsig_1}\)
    % with convergence in \(\SchwartzSpacezRopq\).
\end{proposition}
\begin{proof}
    Take any \(F_2(t')\in \SchwartzSpace[\R^q]\) with \(\int F_2(t')\: dt'=1\) and \(\int \left( t' \right)^{\alpha}F_2(t')\: dt'=0\), \(\forall \alpha\ne 0\)
    (for example take \(\widehat{F_2}(\xi)=1\) on a neighborhood of \(0\)).
    Let \(F_1(t_0)\in \SchwartzSpace[\R]\) satisfy \(\supp(F_1)\subseteq [0,\infty)\), \(\int F_1 =1\), and \(\int t_0^j F_1(t_0)\: dt_0\),
    \(\forall j\geq 1\)--see Lemma \ref{Lemma::Spaces::Classical::ExistenceGoodFunc} for the existence of \(F_1\).

    Set \(\vsig_0(t_0,t'):=F_1(t_0)F_2(t')\in \SchwartzSpaceRopq\) and set
    \(\vsig_1:=\Dild{2}{\vsig_0}-\vsig_0\).
    Note that \(\vsig_0+\sum_{j=1}^N \Dild{2^j}{\vsig_1}=\Dild{2^{N+1}}{\vsig_0}\xrightarrow{N\rightarrow\infty}\delta_0\),
    since \(\int \vsig_0=1\).
    \ref{Item::Spaces::Schwartz::SumToIdentity::Support} and \ref{Item::Spaces::Schwartz::SumToIdentity::SumToDelta}
    follow.
    %All that remains to show is \(\vsig_1\in \sT\).

    We have,
    \begin{equation}\label{Eqn::Spaces::Schwartz::SumToIdentity::DoubleDifference}
    \begin{split}
         &\vsig_1(t_0,t_1,\ldots, t_q) = 2^{\Xdv_0} F_1(2^{\Xdv_0}t_0)\left[  2^{\Xdv_1+\cdots +\Xdv_q} F_2(2^{\Xdv_1} t_1,\ldots, 2^{\Xdv_q} t_q) - F_2(t_1,\ldots, t_q) \right]
         \\&\quad\quad\quad+\left[ 2^{\Xdv_0} F_1(2^{\Xdv_0} t_0)- F_1(t_0) \right]F_2(t').
    \end{split}
    \end{equation}
    The first term on the right-hand side of \eqref{Eqn::Spaces::Schwartz::SumToIdentity::DoubleDifference}
    is of the form \(G_1(t_0)G_2(t')\) where \(G_1\in \SchwartzSpaceR\), \(\supp(G_1)\subseteq [0,\infty)\),
    and \(G_2\in \SchwartzSpacez[\R^q]\). 
    The second term on the right-hand side of \eqref{Eqn::Spaces::Schwartz::SumToIdentity::DoubleDifference}
    is similarly of the form \(H_1(t_0)H_2(t')\) where \(H_1\in \SchwartzSpacezR\), \(\supp(H_1)\subseteq [0,\infty)\),
    and \(H_2\in \SchwartzSpace[\R^q]\)
    
    From here, Lemma \ref{Lemma::Spaces::Schwartz::PuttingElementsInsT} \ref{Item::Spaces::Schwartz::PuttingElementsInsT::SzRtimesSRqInsT}
    and \ref{Item::Spaces::Schwartz::PuttingElementsInsT::SRtimesSzRqInsT}
    show \(\vsig_1\in \sT\); i.e., \ref{Item::Spaces::Schwartz::SumToIdentity::sT} holds.

    Finally, we show \ref{Item::Spaces::Schwartz::SumToIdentity::IdentityOnSchwartz}.
    %\(f=\sum_{j=-\infty}^{\infty} f*\Dild{2^j}{\vsig_1}\).
    Taking the Fourier transform, it suffices to show
    \(\sum_{j=-N}^N \fh(\xi)\widehat{\Dild{2^j}{\vsig_1}}(\xi)\rightarrow \fh(\xi)\), with convergence in \(\SchwartzSpaceRopq\).
    But, 
    \begin{equation*}
        \sum_{j=-N}^N \fh(\xi)\widehat{\Dild{2^j}{\vsig_1}}(\xi)
        =\fh(\xi)\widehat{\Dild{2^{N+1}}{\vsig_0}}(\xi) - \fh(\xi)\widehat{\Dild{2^{-N}}{\vsig_0}}(\xi)
        =\fh(\xi)\hat{\vsig}_0(2^{(-N-1)\Xdv}\xi) - \fh(\xi)\hat{\vsig}_0(2^{N\Xdv}\xi).
    \end{equation*}
    Using Lemma \ref{Lemma::Spaces::Schwartz::FTOfSchwartzSpacez}, it is easy to see
    \(\fh(\xi)\hat{\vsig}_0(2^{N\Xdv}\xi)\rightarrow 0\) in \(\SchwartzSpacezRopq\),
    and using the choice of \(\vsig_0\)
    (in particular, that \(\int \vsig_0=1\)),
    it is easy to see \(\fh(\xi)\hat{\vsig}_0(2^{(-N-1)\Xdv}\xi)\rightarrow \fh(\xi)\)
    in \(\SchwartzSpaceRopq\), completing the proof.
\end{proof}

\begin{lemma}\label{Lemma::Spaces::Schwartz::sT4InsT}
    Consider the set:
    \begin{equation*}
        \sT_4:=\left\{ \sum_{j=-\infty}^{\infty} 2^{Nj} f*\Dild{2^j}{\vsig} : N\in \Z, f\in \SchwartzSpacezRopq, \supp(f)\subseteq[0,\infty)\times \R^{q}, \vsig\in \sT \right\}.
    \end{equation*}
    Then, \(\sT_4\subseteq \sT\).
\end{lemma}
\begin{proof}
    For any \(f,\vsig\in \SchwartzSpacezRopq\), we have
    \begin{equation}\label{Eqn::Spaces::Schwartz::sT4InsT::SumInSchwartz}
        \sum_{j=-\infty}^{\infty} 2^{Nj} f*\Dild{2^j}{\vsig}\in \SchwartzSpaceRopq.
    \end{equation}
    Indeed, this follows easily by taking the Fourier transform and applying Lemma \ref{Lemma::Spaces::Schwartz::FTOfSchwartzSpacez}.
    Since \(\sT\subseteq \SchwartzSpacezRopq\) by Lemma \ref{Lemma::Spaces::Schwartz::PuttingElementsInsT} \ref{Item::Spaces::Schwartz::PuttingElementsInsT::sTInSchwartzSpacez},
    it follows that \(\sT_4\subseteq \SchwartzSpaceRopq\).

    For any \(f,\vsig\in \SchwartzSpacezRopq\), with \(\supp(f),\supp(\vsig)\subseteq[0,\infty)\times \R^q\),
    we have (directly from the definition), \(\supp\left( \sum_{j=-\infty}^{\infty} 2^{Nj} f*\Dild{2^j}{\vsig} \right)\subseteq [0,\infty)\times \R^q\).
    Thus, for \(G\in \sT_4\), \(\supp(G)\subseteq [0,\infty)\times \R^q\).

    Finally, take \(G=\sum_{j=-\infty}^{\infty} 2^{Nj} f*\Dild{2^j}{\vsig}\in \sT_4\), 
    where \(f\in \SchwartzSpacezRopq\), \(\supp(f)\subseteq[0,\infty)\times \R^q\),
    and \(\vsig\in \sT\).  Then, \(\vsig=\sum_{l=0}^q \partial_{t_l} \vsig_l\), where \(\vsig_l\in \sT\),
    and therefore,
    \begin{equation*}
        G=\sum_{j=-\infty}^{\infty} 2^{Nj} f*\Dild{2^j}{\vsig}
        =\sum_{l=0}^q\partial_{t_l}\sum_{j=-\infty}^{\infty} 2^{(N-\Xdv_l)j} f*\Dild{2^j}{\vsig_l}
        =:\sum_{l=0}^q \partial_{t_l} G_l,
    \end{equation*}
    where \(G_l\in \sT_4\).
    We have shown \(\sT_4\) satisfies the axioms of \(\sT\) from Definition \ref{Defn::Spaces::Schwartz::sT},
    and it follows that \(\sT_4\subseteq \sT\), completing the proof.
\end{proof}

\begin{proof}[Proof of Proposition \ref{Prop::Spaces::Schwartz::CharacterizesT}]
    Lemma \ref{Lemma::Spaces::Schwartz::PuttingElementsInsT} \ref{Item::Spaces::Schwartz::PuttingElementsInsT::sTInSchwartzSpacez}
    combined with Definition \ref{Defn::Spaces::Schwartz::sT} \ref{Item::Spaces::Schwartz::sT::Support}
    shows \(\sT\subseteq \left\{ f\in \SchwartzSpacezRopq : \supp(f)\subseteq [0,\infty)\times \R^q \right\}\).
    For \(f\in \left\{ f\in \SchwartzSpacezRopq : \supp(f)\subseteq [0,\infty)\times \R^q \right\}\),
    we have  \(f=\sum_{j=-\infty}^{\infty} f*\Dild{2^j}{\vsig_1}\), by Proposition \ref{Prop::Spaces::Schwartz::SumToIdentity}.
    Since \(\vsig_1\in \sT\), Lemma \ref{Lemma::Spaces::Schwartz::sT4InsT} shows \(f\in \sT_4\subseteq \sT\),
    completing the proof.
\end{proof}

\begin{lemma}\label{Lemma::Spaces::Schwartz::MultiplysTbyPoly}
    For \(\vsig(t)\in \sT\), we have \(t^{\alpha}\vsig(t)\in \sT\), \(\forall \alpha\).
\end{lemma}
\begin{proof}
    Since \(t^{\alpha}f\in \SchwartzSpacezRopq\), \(\forall f\in \SchwartzSpacezRopq\), this follows from 
    Proposition \ref{Prop::Spaces::Schwartz::CharacterizesT}.
\end{proof}

    \subsection{Multiplication operators}\label{Section::Spaces::Multiplication}
    In this section, we prove Proposition \ref{Prop::Spaces::LP::DjExist}; we do so by proving a more detailed
results in the case when \(\ManifoldN\) is embedded in a submanifold without boundary. To describe this, we need some definitions.
Let \(\ManifoldM\) be a smooth manifold (without boundary) such that \(\ManifoldN\subseteq \ManifoldM\)
 is a closed, embedded, co-dimension \(0\) submanifold (with boundary).
 Let \(\FilteredSheafFh\) be a H\"ormander filtration of sheaves of vector fields on \(\ManifoldM\)
 such that \(\RestrictFilteredSheaf{\FilteredSheafFh}{\ManifoldN}=\FilteredSheafF\)
 (such an \(\ManifoldM\) and \(\FilteredSheafFh\) always exist--see Proposition \ref{Prop::Filtrations::RestrictingFiltrations::CoDim0CoRestriction}).

 Fix a compact set \(\Compact\Subset \ManifoldM\), and \(\Omega\Subset \ManifoldM\) open and relatively compact
 with \(\Compact\Subset \Omega\).  Let \(\WhWdv=\left\{ (\Wh_1,\Wdv_1),\ldots, (\Wh_r,\Wdv_r) \right\}\subset \VectorFields{\Omega}\times \Zg\)
 be H\"ormander vector fields with formal degrees on \(\Omega\), such that
 \(\FilteredSheafFh\big|_{\Omega}=\FilteredSheafGenBy{\WhWdv}\).
 Set \(W_j:=\Wh_j\big|_{\Omega\cap \ManifoldN}\) and \(\WWdv=\left\{ (W_1,\Wdv_1),\ldots, (W_r,\Wdv_r) \right\}\)
 so that by Proposition \ref{Prop::Filtrations::RestrictingFiltrations::CoDim0Restriction},
 \(\FilteredSheafF\big|_{\Omega\cap \ManifoldN}=\FilteredSheafGenBy{\WWdv}\).

 \begin{definition}\label{Defn::Spaces::Multiplication::PElemzFhN}
    Let \(\PElemzFhN{\Compact}\) be the set of those \(\sE\in \PElemzFh{\Compact}\)
    such that \(\forall(E,\delta)\in \sE\), \(\supp(E)\cap \left( \ManifoldN\times \ManifoldM \right)\subseteq \ManifoldN\times \ManifoldN\).
 \end{definition}

 \begin{definition}\label{Defn::Spaces::Multiplication::ElemzFhN}
    Let \(\ElemzFhN{\Compact}\) be the largest set of subsets of \(\CinftyCptSpace[\ManifoldM\times\ManifoldM]\times (0,1]\)
    satisfying:
    \begin{enumerate}[(i)]
        \item \(\ElemzFhN{\Compact}\subseteq \PElemzFhN{\Compact}\).
        \item\label{Item::Spaces::Multiplication::ElemzFhN::PullOutDerivs} \(\forall \sE\in \ElemzFhN{\Compact}\), \(\forall (E,2^{-j})\in \sE\),
            \begin{equation*}
                E(x,y)=\sum_{|\alpha|,|\beta|\leq 1} 2^{-j(2-|\alpha|-|\beta|)} \left( 2^{-j\Wdv}\Wh_x \right)^{\alpha}
                \left( 2^{-j\Wdv}\Wh_y \right)^{\beta} E_{\alpha,\beta}(x,y),
            \end{equation*}
            where
            \begin{equation*}
                \left\{ \left( E_{\alpha,\beta},2^{-j} \right) : (E,2^{-j})\in \sE, |\alpha|,|\beta|\leq 1 \right\}\in \ElemzFhN{\Compact}.
            \end{equation*}
    \end{enumerate}
    For \(\Omegah\subseteq \ManifoldM\) open, set
    \begin{equation*}
        \ElemzFhN{\Omegah}:=\bigcup_{\substack{\Compact\Subset \Omegah \\ \Compact\text{ compact}}} \ElemzFhN{\Compact}.
    \end{equation*}
 \end{definition}

 Clearly, \(\ElemzFhN{\Compact}\subseteq \ElemzFh{\Compact}\).

 \begin{lemma}\label{Lemma::Spaces::Multiplication::OldElemResultsApply}
    Lemma \ref{Lemma::Spaces::LP::ElemDoesntDependOnChoices}, Proposition \ref{Prop::Spaces::Elem::Elem::MainProps},
    and Lemma \ref{Lemma::Spaces::Elem::Elem::ElemzhF} hold with \(\ElemzF{\cdot}\) and \(\PElemzF{\cdot}\)
    replaced by \(\ElemzFhN{\cdot}\) and \(\PElemzFhN{\cdot}\), throughout.
 \end{lemma}
 \begin{proof}
    The only difference between the setting here, and the setting of those results,
    is the additional condition \(\supp(E)\cap \left( \ManifoldN\times \ManifoldM \right)\subseteq \ManifoldN\times \ManifoldN\)
    in Definition \ref{Defn::Spaces::Multiplication::PElemzFhN}. Keeping track of this additional condition,
    the proofs of the results go through otherwise unchanged.
 \end{proof}

 The main result of this section is the following.
 \begin{theorem}\label{Thm::Spaces::Multiplication::MainAmbientTheorem}
    Let \(\Omegah\subseteq \ManifoldM\) be open with
    \(\Omegah\cap \ManifoldN\subseteq \ManifoldNncF\), and \(\psih\in \CinftyCptSpace[\Omegah]\).
    Then, there exists \(\left\{ \left( \Dh_j,2^{-j} \right) :j\in \Zgeq \right\}\in \ElemzFhN{\Omegah}\)
    such that the following hold:
    \begin{enumerate}[(i)]
        \item\label{Item::Spaces::Multiplication::MainAmbientTheorem::SumToMultplication} \(\sum_{j\in \Zgeq} \Dh_j =\Mult{\psih}\).
        \item\label{Item::Spaces::Multiplication::MainAmbientTheorem::RestrictedAreElem} Set \(D_j:=\Dh_j\big|_{\ManifoldN\times \ManifoldN}\). Then, \(\left\{ (D_j,2^{-j}) : j\in \Zgeq \right\}\in \ElemzF{\Omegah\cap \ManifoldN}\).
        \item\label{Item::Spaces::Multiplication::MainAmbientTheorem::RestrictedSumToMultiplication} \(\sum_{j\in \Zgeq} D_j = \Mult{\psih\big|_{\ManifoldN}}\).
        \item\label{Item::Spaces::Multiplication::MainAmbientTheorem::PhArePreElem} Set \(\Ph_j:=\sum_{k=0}^j \Dh_j\). Then, \(\left\{ \left( \Ph_j, 2^{-j} \right) : j\in \Zgeq \right\}\in \PElemzFhN{\Omegah}\).
        \item\label{Item::Spaces::Multiplication::MainAmbientTheorem::PArePreElem} Set \(P_j:=\Ph_j\big|_{\ManifoldN\times \ManifoldN}=\sum_{k=0}^j D_j\). Then, \(\left\{ \left( P_j, 2^{-j} \right) : j\in \Zgeq \right\}\in \PElemzF{\Omegah\cap \ManifoldN}\).
    \end{enumerate}
    See Proposition \ref{Prop::Spaces::Elem::Elem::ConvergenceOfElemOps} for the convergence of the sums in
    \ref{Item::Spaces::Multiplication::MainAmbientTheorem::SumToMultplication}
    and
    \ref{Item::Spaces::Multiplication::MainAmbientTheorem::RestrictedSumToMultiplication}.
 \end{theorem}

 First we see how Theorem \ref{Thm::Spaces::Multiplication::MainAmbientTheorem}
 implies Proposition \ref{Prop::Spaces::LP::DjExist}.

 \begin{proof}[Proof of Proposition \ref{Prop::Spaces::LP::DjExist}]
    In Proposition \ref{Prop::Spaces::LP::DjExist}, \(\FilteredSheafFh\) and \(\ManifoldM\) are not given,
    but there always exists a choice of such data--see Proposition \ref{Prop::Filtrations::RestrictingFiltrations::CoDim0CoRestriction}.
    Let \(\Omega\subseteq \ManifoldNncF\) and \(\psi\) be as in Proposition \ref{Prop::Spaces::LP::DjExist}.
    Pick any \(\Omegah\subseteq \ManifoldM\) open with \(\Omegah\cap \ManifoldN=\Omega\), and any
    \(\psih\in \CinftyCptSpace[\Omegah]\) with \(\psih\big|_{\ManifoldN}=\psi\)
    (such a choice of \(\psih\) always exists--see Lemma \ref{Lemma::Spaces::Elem::Extend::Classical}).
    From here, Theorem \ref{Thm::Spaces::Multiplication::MainAmbientTheorem} directly
    implies the result.
 \end{proof}

 The rest of this section is devoted to the proof of Theorem \ref{Thm::Spaces::Multiplication::MainAmbientTheorem}.

 \begin{lemma}\label{Lemma::Spaces::Multiplication::RestrictElemGivesElem}
    Suppose \(\Compact\cap \ManifoldN\subseteq \ManifoldNncF\). 
    \begin{enumerate}[(i)]
        \item\label{Item::Spaces::Multiplication::RestrictElemGivesElem::PElem} For \(\sE\in \PElemzFhN{\Compact}\),
            \begin{equation*}
                \sE\big|_{\ManifoldN\times \ManifoldN}:=\left\{ \left( E\big|_{\ManifoldN\times \ManifoldN},\delta \right) : (E,\delta)\in \sE \right\}
        \in \PElemzF{\Compact\cap \ManifoldN}.
            \end{equation*}
        \item\label{Item::Spaces::Multiplication::RestrictElemGivesElem::Elem} For \(\sE\in \ElemzFhN{\Compact}\),
        \begin{equation*}
                \sE\big|_{\ManifoldN\times \ManifoldN}
                % :=\left\{ \left( E\big|_{\ManifoldN\times \ManifoldN},\delta \right) : (E,\delta)\in \sE \right\}
        \in \ElemzF{\Compact\cap \ManifoldN}.
            \end{equation*}
    \end{enumerate}
    
    % For \(\sE\in \ElemzFhN{\Compact}\),
    % we have
    % \begin{equation*}
    %     \sE\big|_{\ManifoldN\times \ManifoldN}:=\left\{ \left( E\big|_{\ManifoldN\times \ManifoldN},\delta \right) : (E,\delta)\in \sE \right\}
    %     \in \ElemzFhN{\Compact\cap \ManifoldN}.
    % \end{equation*}
 \end{lemma}
 \begin{proof}
    \ref{Item::Spaces::Multiplication::RestrictElemGivesElem::PElem}:
    Proposition \ref{Prop::Spaces::Elem::Extend::Restrict::NonVerbose} shows
    \(\sE\big|_{\ManifoldN\times \ManifoldN}\in \PElemF{\Compact\cap \ManifoldN}\).
    Since \(\supp(E)\cap \left( \ManifoldN\times \ManifoldM\right)\subseteq \ManifoldN\times \ManifoldN\), for \(x\in \ManifoldN\)
    and \((E,\delta)\in \sE\), we have \(E\big|_{\ManifoldN\times\ManifoldN}(x,y)\) vanishes to infinite order as \(y\rightarrow \BoundaryN\).
    We conclude \(\sE\big|_{\ManifoldN\times \ManifoldN}\in \PElemzF{\Compact\cap \ManifoldN}\).

    \ref{Item::Spaces::Multiplication::RestrictElemGivesElem::Elem}:
    For \(\left( E,2^{-j} \right)\in \sE\), we have using Definition \ref{Defn::Spaces::Multiplication::ElemzFhN} \ref{Item::Spaces::Multiplication::ElemzFhN::PullOutDerivs},
    \begin{equation*}
    \begin{split}
         E\big|_{\ManifoldN\times\ManifoldN}(x,y)
         &=\sum_{|\alpha|,|\beta|\leq 1} 2^{-j(2-|\alpha|-|\beta|)}
         \left[ \left( 2^{-j\Wdv}\Wh_x \right)^{\alpha}\left( 2^{-j\Wdv}\Wh_y \right)^{\beta} E_{\alpha,\beta}(x,y) \right]\Bigg|_{\ManifoldN\times \ManifoldN}
         \\&=\sum_{|\alpha|,|\beta|\leq 1} 2^{-j(2-|\alpha|-|\beta|)}
         \left( 2^{-j\Wdv} W_x \right)^{\alpha}\left( 2^{-j\Wdv} W_y^{\beta} \right) E_{\alpha,\beta}\big|_{\ManifoldN\times\ManifoldN}(x,y),
    \end{split}
    \end{equation*}
    where 
    \begin{equation*}
        \left\{ \left( E_{\alpha,\beta},2^{-j} \right) : (E,2^{-j})\in \sE, |\alpha|,|\beta|\leq 1 \right\}\in \ElemzFhN{\Compact}.
    \end{equation*}
    This shows that 
    \begin{equation*}
        \left\{ \sE\big|_{\ManifoldN\times \ManifoldN} : \sE\in \ElemzFhN{\Compact} \right\}
    \end{equation*}
    satisfies the axioms of Definition \ref{Defn::Spaces::LP::ElemWWdv} and therefore
    \(\left\{ \sE\big|_{\ManifoldN\times \ManifoldN} : \sE\in \ElemzFhN{\Compact} \right\}\subseteq \ElemzF{\Compact\cap \ManifoldN}\).
 \end{proof}

 \begin{lemma}\label{lemma::Spaces::Multiplication::RestrictMultiplicationGivesMultiplicationRestrict}
    Suppose \(\Omegah\subseteq \ManifoldM\) with \(\Omegah\cap \ManifoldN\subseteq \ManifoldNncF\),
    \(\psih\in \CinftyCptSpace[\Omegah]\), and
    \(\left\{ \left( \Dh_j, 2^{-j} \right) :j\in \Zgeq\right\}\in \ElemzFhN{\Omegah}\)
    satisfies \(\sum_{j\in \Zgeq}\Dh_j=\Mult{\psih}\).
    Let \(D_j:=\Dh_j\big|_{\ManifoldN\times \ManifoldN}\).
    Then, \(\sum_{j\in \Zgeq} D_j=\Mult{\psih\big|_{\ManifoldN}}\).
    See  Proposition \ref{Prop::Spaces::Elem::Elem::ConvergenceOfElemOps} (and Lemma \ref{Lemma::Spaces::Multiplication::RestrictElemGivesElem})
    for the convergence of the sums.
 \end{lemma}
 \begin{proof}
    By Lemma \ref{Lemma::Spaces::Multiplication::RestrictElemGivesElem} and Proposition \ref{Prop::Spaces::Elem::Elem::ConvergenceOfElemOps},
    \(\sum_{j\in \Zgeq} D_j=\Mult{\psih\big|_{\ManifoldN}}\) converges as continuous linear operators
    \(\DistributionsZeroN\rightarrow \DistributionsZeroN\).
    For \(f\in \CinftyCptSpace[\ManifoldM]\), we have
    \begin{equation*}
        \psih\big|_{\ManifoldN} f\big|_{\ManifoldN}=\left( \psih f \right)\big|_{\ManifoldN}  = \left( \sum_{j\in \Zgeq} \Dh_j f \right)\big|_{\ManifoldN}
        =\sum_{j\in \Zgeq} \left( \Dh_j f \right)\big|_{\ManifoldN} = \sum_{j\in \Zgeq} D_j \left( f\big|_{\ManifoldN} \right).
    \end{equation*}
    Thus, we see for all \(u\) of the form \(u=f\big|_{\ManifoldN}\), where \(f\in \CinftyCptSpace[\ManifoldM]\), we have
    \begin{equation}\label{Eqn::Spaces::Multiplication::RestrictMultiplicationGivesMultiplicationRestrict::AppliedToSmooth}
        \psih\big|_{\ManifoldN} u = \sum_{j\in \Zgeq} D_j u.
    \end{equation}
    Since every element of \(\DistributionsZeroN\) can be approximated (in \(\DistributionsZeroN\)) by such \(u\),
    \eqref{Eqn::Spaces::Multiplication::RestrictMultiplicationGivesMultiplicationRestrict::AppliedToSmooth} holds
    \(\forall u\in \DistributionsZeroN\), completing the proof.
 \end{proof}

 Next, we state a local version of Theorem \ref{Thm::Spaces::Multiplication::MainAmbientTheorem}.

 \begin{lemma}\label{Lemma::Spaces::Multiplication::LocalVersion}
    \(\forall x\in \left( \ManifoldM\setminus \ManifoldN \right)\bigcup \ManifoldNncF\),
    there exists an \(\ManifoldM\)-open set \(\Omegah_0\subseteq \ManifoldM\), with \(x\in \Omegah_0\)
    and such that \(\forall \Omegah\subseteq \Omegah_0\) open, \(\forall \psih\in \CinftyCptSpace[\Omegah]\),
    \(\exists \left\{ \left( \Dh_j, 2^{-j} \right) : j\in \Zgeq \right\} \in \ElemzFhN{\Omegah}\) with
    \(\sum_{j\in \Zgeq}\Dh_j = \Mult{\psih}\) and such that if \(\Ph_j:=\sum_{k=0}^j \Dh_j\),
    then \(\left\{ \left( \Ph_j, 2^{-j} \right):j\in \Zgeq \right\}\in \PElemzFhN{\Omegah}\).

    % \(\forall x\in \left( \ManifoldM\setminus \ManifoldN \right)\bigcup \ManifoldNncF\), there exists an \(\ManifoldM\)-open set,
    % \(\Omegah\subseteq \ManifoldM\), with \(x\in \Omegah\), and such that \(\forall \psih\in \CinftyCptSpace[\Omegah]\),
    % there exists
    % \(\left\{ (\Dh_j,2^{-j}) : j\in \Zgeq\right\}\in \ElemzFhN{\Omegah}\),
    % with \(\sum_{j\in \Zgeq}\Dh_j=\Mult{\psih}\).
 \end{lemma}

 Before we prove Lemma \ref{Lemma::Spaces::Multiplication::LocalVersion}, we see how it
 completes the proof of Theorem \ref{Thm::Spaces::Multiplication::MainAmbientTheorem}.

 \begin{proof}[Proof of Theorem \ref{Thm::Spaces::Multiplication::MainAmbientTheorem}]
    Let \(\Omegah\subseteq \ManifoldM\) be \(\ManifoldM\)-open, with \(\Omegah\cap \ManifoldN\subseteq \ManifoldNncF\),
    and \(\psih\in \CinftyCptSpace[\Omegah]\).
    For each \(x\in \supp(\psih)\), let \(\Omegah_{x,0}\) be the neighborhood of \(x\) from
    Lemma \ref{Lemma::Spaces::Multiplication::LocalVersion} and let \(\Omegah_x:=\Omegah_{x,0}\cap \Omega\).
    % by replacing \(\Omegah_x\) with \(\Omegah_x\cap \Omega\),
    % we may assume
    % \(\Omegah_x\subseteq \Omega\). 
    \(\left\{ \Omegah_x : x\in \supp(\psih) \right\}\) is an open cover for 
    the compact set \(\supp(\psih)\);
    extract a finite subcover, \(\Omegah_{x_1},\ldots, \Omegah_{x_L}\).
    Let \(\phi_l\in \CinftyCptSpace[\Omegah_{x_l}]\) satisfy \(\sum_{l=1}^L \phi_l=1\) on a neighborhood of \(\supp(\psih)\).
    Set \(\psih_l:=\phi_l \psih\in \CinftyCptSpace[\Omega_{x_l}]\), so that \(\sum_{l=1}^L \psih_l=\psih\).

    Applying Lemma \ref{Lemma::Spaces::Multiplication::LocalVersion} to \(\psih_l\), we may write
    \(\Mult{\psih_l}=\sum_{j\in \Zgeq} \Dh_{j,l}\),
    where \(\left\{ \left( \Dh_{j,l}, 2^{-j} \right) :j\in \Zgeq\right\}\in \ElemzFhN{\Omegah_{x_l}}\subseteq \ElemzFhN{\Omegah}\).
    Letting \(\Ph_{j,l}:=\sum_{k=0}^j \Dh_{j,l}\), Lemma \ref{Lemma::Spaces::Multiplication::LocalVersion} shows
    \(\left\{ \left( \Ph_{j,l}, 2^{-j} \right) :j\in \Zgeq\right\}\in \PElemzFhN{\Omegah_{x_l}}\subseteq \PElemzFhN{\Omegah}\).
    Setting \(\Dh_j:=\sum_{l=1}^M \Dh_{j,l}\), 
    Lemma \ref{Lemma::Spaces::Multiplication::OldElemResultsApply} (see Proposition \ref{Prop::Spaces::Elem::Elem::MainProps} \ref{Item::Spaces::Elem::Elem::LinearComb})
    implies \(\left\{ (\Dh_j, 2^{-j}) : j\in \Zgeq \right\}\in \ElemzFhN{\Omega}\).
    Note \(\Ph_j:=\sum_{k=0}^j \Dh_j=\sum_{l=1}^M \Ph_{j,l}\), and therefore
    \(\left\{ (\Ph_j, 2^{-j}) : j\in \Zgeq \right\}\in \PElemzFhN{\Omega}\), establishing \ref{Item::Spaces::Multiplication::MainAmbientTheorem::PhArePreElem}.
    We also have \(\sum_{j\in \Zgeq} \Dh_j =\sum_{l=1}^M\sum_{j\in \Zgeq}\Dh_{j,l}=\sum_{l=1}^M \Mult{\psih_l}=\Mult{\psih}\),
    establishing \ref{Item::Spaces::Multiplication::MainAmbientTheorem::SumToMultplication}.
    \ref{Item::Spaces::Multiplication::MainAmbientTheorem::RestrictedAreElem} and \ref{Item::Spaces::Multiplication::MainAmbientTheorem::PArePreElem}
    follow from Lemma \ref{Lemma::Spaces::Multiplication::RestrictElemGivesElem}.
    \ref{Item::Spaces::Multiplication::MainAmbientTheorem::RestrictedSumToMultiplication} follows from
    \ref{Item::Spaces::Multiplication::MainAmbientTheorem::SumToMultplication}
    and Lemma \ref{lemma::Spaces::Multiplication::RestrictMultiplicationGivesMultiplicationRestrict}.
 \end{proof}

 Lemma \ref{Lemma::Spaces::Multiplication::LocalVersion}
 follows immediately from the next two lemmas.
 
 \begin{lemma}\label{Lemma::Spaces::Multiplication::LocalVersion::Interior}
    Lemma \ref{Lemma::Spaces::Multiplication::LocalVersion} holds for \(x\in \InteriorN\cup \left( \ManifoldM\setminus \ManifoldN \right)\).
 \end{lemma}
% \begin{proof}
%     If \(x\in \InteriorN\), take \(\Omegah\subseteq \InteriorN\) a connected \(\ManifoldM\)-open set with \(x\in \Omegah\).
%     If \(x\in \ManifoldM\setminus \ManifoldN\), take \(\Omegah\subseteq \ManifoldM\setminus \ManifoldN\) a connected \(\ManifoldM\)-open set with \(x\in \Omegah\).

%     In either case, for \(\psih\in \CinftyCptSpace[\Omegah]\), \cite[Corollary 5.6.2]{StreetMaximalSubellipticity}
%     shows that there exists \(\left\{ \left( \Dh_j,2^{-j} \right) : j\in \Zgeq \right\}\in \ElemzFh{\Omegah}\)
%     with \(\sum_{j\in \Zgeq}\Dh_j=\Mult{\psih}\).
%     Also, in either case, it follows easily from the definitions that
%     \( \ElemzFh{\Omegah}\subseteq \ElemzFhN{\Omegah}\). Indeed, to see this, we need to show that the important additional condition
%     from Definition \ref{Defn::Spaces::Multiplication::PElemzFhN}
%     that \(\supp(E)\cap \left( \ManifoldN\times \ManifoldM \right)\subseteq \ManifoldN\times \ManifoldN\)
%     holds for all operators which arise.
%     When \(\Omegah\subseteq \InteriorN\), all operators which arise are supported in \(\InteriorN\times \InteriorN\),
%     and this is clear.  When \(\Omegah\subseteq \ManifoldM\setminus \ManifoldN\) all operators which arise
%     are supported in \(\left( \ManifoldM\setminus \ManifoldN \right)\times \left( \ManifoldM\setminus \ManifoldN \right)\)
%     and the intersection \(\supp(E)\cap \left( \ManifoldN\times \ManifoldM \right)\) is empty.
% \end{proof}

\begin{lemma}\label{Lemma::Spaces::Multiplication::LocalVersion::Boundary}
    Lemma \ref{Lemma::Spaces::Multiplication::LocalVersion} holds for \(x\in \BoundaryNncF\).
 \end{lemma}

 The rest of this section is devoted to the proofs of Lemmas \ref{Lemma::Spaces::Multiplication::LocalVersion::Interior} and \ref{Lemma::Spaces::Multiplication::LocalVersion::Boundary}.
The proof of Lemma \ref{Lemma::Spaces::Multiplication::LocalVersion::Interior} is a simpler reprise of the proof
of Lemma \ref{Lemma::Spaces::Multiplication::LocalVersion::Boundary}; so we first prove 
 Lemma \ref{Lemma::Spaces::Multiplication::LocalVersion::Boundary} and then explain the modifications necessary to establish
 Lemma \ref{Lemma::Spaces::Multiplication::LocalVersion::Interior}.

Fix \(x_0\in \BoundaryNncF\) for which we want to prove Lemma \ref{Lemma::Spaces::Multiplication::LocalVersion::Boundary}.
It will be convenient to pick certain generators of \(\LieFilteredSheafFh\) near \(x_0\). The next lemma gives such generators.

\begin{lemma}\label{Lemma::Spaces::Multiplication::NiceGenerators}
    There exists a connected, \(\ManifoldM\)-open neighborhood \(\Omegah\Subset \ManifoldM\) of \(x_0\)
    such that \(\Omegah\cap \ManifoldN\subseteq \ManifoldNncF\) and
    \(\XhXdv=\left\{ \left( \Xh_0,\Xdv_0 \right),\left( \Xh_1,\Xdv_1 \right)\ldots, \left( \Xh_q,\Xdv_q \right) \right\}\subset \VectorFields{\Omegah}\times \Zg\)
    H\"ormander vector fields with formal degrees on \(\Omegah\), such that:
    \begin{enumerate}[(i)]
        \item\label{Item::Spaces::Multiplication::NiceGenerators::Span} \(\forall x\in \Omegah\), \(\Span\left\{ \Xh_0(x),\ldots, \Xh_q(x) \right\}=\TangentSpace{x}{\Omegah}\).
        \item\label{Item::Spaces::Multiplication::NiceGenerators::AreGenerators} \(\LieFilteredSheafFh\big|_{\Omegah}=\FilteredSheafGenBy{\XhXdv}\).
        \item\label{Item::Spaces::Multiplication::NiceGenerators::RestrictGenerators} Let \(X_j=\Xh_j\big|_{\ManifoldN\cap \Omegah}\), and \(\XXdv=\left\{ (X_0,\Xdv_0),\ldots, (X_q,\Xdv_q) \right\}\). Then, 
            \(\LieFilteredSheafF\big|_{\ManifoldN\cap \Omegah}=\FilteredSheafGenBy{\XXdv}\).
        \item\label{Item::Spaces::Multiplication::NiceGenerators::NSWCommutators} \(\left[ \Xh_j,\Xh_k \right]=\sum_{\Xdv_l\leq \Xdv_j+\Xdv_k}c_{j,k}^l \Xh_l\), \(c_{j,k}^l\in \CinftySpace[\Omegah]\).
        \item\label{Item::Spaces::Multiplication::NiceGenerators::Xh0PointsIn} 
            \(\Xdv_0=\degBoundaryNF[x_0]\),
        \(\Xh_0(x)\not \in \TangentSpace{x}{\BoundaryN}\), 
            \(\forall x\in \BoundaryN\cap \Omegah\), and \(\Xh_0\) points
            into \(\ManifoldN\) in the sense 
            if \(\Compact\Subset \Omegah\) is compact,
            then there exists \(a>0\) small such that for \(0<t<a\), \(e^{t\Xh_0}x\in \InteriorN\),
            for \(x\in \BoundaryN\cap \Compact\).
        \item\label{Item::Spaces::Multiplication::NiceGenerators::XhjAreTangentToBdry} \(\Xh_j(x)\in \TangentSpace{x}{\BoundaryN}\), \(\forall 1\leq j\leq q\), \(x\in \BoundaryN\cap \Omegah\).
    \end{enumerate}
\end{lemma}
\begin{proof}
    First we create \(\ZhZdv\) satisfying all the properties of \(\XhXdv\) except 
    \ref{Item::Spaces::Multiplication::NiceGenerators::XhjAreTangentToBdry}.
    Let \(\Omegah_1\Subset \ManifoldM\) be \(\ManifoldM\)-open and relatively compact with
    \(x_0\in \Omegah_1\) and \(\Omegah_1 \cap \ManifoldN\subseteq \ManifoldNncF\).
    By Lemma \ref{Lemma::Filtrations::GeneratorsOnRelCptSet},
    there exists \(\WhWde\) H\"ormander vector fields with formal degrees with
    \(\FilteredSheafFh\big|_{\Omegah_1}=\FilteredSheafGenBy{\WhWde}\),
    and by Proposition \ref{Prop::Filtrations::RestrictingFiltrations::CoDim0Restriction}
     if \(\WWde=(\Wh\big|_{\Omegah_1\cap\ManifoldN}, \Wde)\) we have
    \(\FilteredSheafF\big|_{\Omegah_1\cap \ManifoldN}=\FilteredSheafGenBy{\WWde}\).

    By Lemma \ref{Lemma::Filtrations::GeneratorsForLieFiltration}, we make take a finite set
    \(\ZhZdv\subset\Gen{\WhWde}\) satisfying
    \ref{Item::Spaces::Multiplication::NiceGenerators::Span} and \ref{Item::Spaces::Multiplication::NiceGenerators::AreGenerators},
    with \(\XhXdv\) replaced by \(\ZhZdv\) and \(\Omegah\) replaced by \(\Omegah_1\).
    Since \(\left[ \Zh_j,\Zh_k \right]\in \LieFilteredSheafFh[\Omegah_1][\Zdv_j+\Zdv_k]=\FilteredSheafGenBy{\ZhZdv}[\Omegah_1][\Zdv_j+\Zdv_k]\),
    \ref{Item::Spaces::Multiplication::NiceGenerators::NSWCommutators} holds as well.

    Let \(\ZZdv=(\Zh\big|_{\Omegah_1\cap \ManifoldN},\Zdv)\).
    Since \(\ZhZdv\subset\Gen{\WhWde}\), we have \(\ZZdv\subset \Gen{\WWde}\) and therefore
    \(\FilteredSheafGenBy{\ZZdv}\subseteq \LieFilteredSheaf{\FilteredSheafGenBy{\WWde}}=\LieFilteredSheafF\big|_{\Omegah_1\cap \ManifoldN}\).
    Conversely, we have \(\LieFilteredSheafF\big|_{\Omegah_1\cap \ManifoldN}\subseteq \RestrictFilteredSheaf{\LieFilteredSheafFh}{\Omegah_1\cap\ManifoldN}=\FilteredSheafGenBy{\ZZdv}\),
    where the final equality follows from Proposition \ref{Prop::Filtrations::RestrictingFiltrations::CoDim0Restriction}.
    We conclude \ref{Item::Spaces::Multiplication::NiceGenerators::RestrictGenerators} holds for \(\ZhZdv\).

    Since \(\LieFilteredSheafF\big|_{\Omegah_1\cap \ManifoldN}=\FilteredSheafGenBy{\ZZdv}\),
    there exists \(j\in \{0,\ldots,q\}\) with \(\Zh_j(x_0)\not \in \TangentSpace{x_0}{\BoundaryN}\)
    and \(\Zdv_j=\degBoundaryNF[x_0]\). 
    Without loss of generality, re-order \(\ZhZdv\) so that \(j=0\).
    Letting \(\Omegah\Subset \Omegah_1\) be a sufficiently small, connected, \(\ManifoldM\)-neighborhood
    of \(x_0\), and using that \(x_0\in \BoundaryNncF\),
    we have \(\degBoundaryNF[x]=\Zdv_0\)
    and \(\Zh_0(x)\not \in \TangentSpace{x}{\BoundaryN}\), \(\forall x\in \Omegah\cap \BoundaryN\).
    By possibly replacing \(\Zh_0\) with \(-\Zh_0\) and possibly shrinking \(\Omegah\), we see that \ref{Item::Spaces::Multiplication::NiceGenerators::Xh0PointsIn}
    holds for \(\ZhZdv\).
    Since \(\degBoundaryNF[x]=\Zdv_0\), \(\forall x\in \Omegah\cap \BoundaryN\), if
    \(\Zdv_j<\Zdv_0\) then \(\Zh_j(x)\in \TangentSpace{x}{\BoundaryN}\), \(\forall x\in \Omegah\cap\BoundaryN\).

    We define \(\XhXdv=\left\{ \left( \Xh_0,\Xdv_0 \right),\left( \Xh_1,\Xdv_1 \right),\ldots, \left( \Xh_1,\Xdv_q \right) \right\}\)
    as follows.
    \(\Xh_0:=\Zh_0\). If \(\Xdv_j<\Xdv_0\), then \(\Xh_j:=\Zh_j\).
    If \(\Xdv_j\geq \Xdv_0\), \(j\geq 1\), pick \(a_j\in \CinftySpace[\Omegah]\) such that with \(\Xh_j(x):=\Zh_j(x)-a_j(x)\Xh_0(x)\)
    we have \(\Xh_j(x)\in \TangentSpace{x}{\BoundaryN}\), \(\forall x\in \Omegah\cap \BoundaryN\).
    \ref{Item::Spaces::Multiplication::NiceGenerators::XhjAreTangentToBdry} holds for \(\XhXdv\),
    and the other properties follow directly from those for \(\ZhZdv\).
\end{proof}

Let \(\Omegah_0\) be \(\Omegah\) from Lemma \ref{Lemma::Spaces::Multiplication::NiceGenerators}
and 
fix \(\XhXdv\), \(\XXdv\) as in that lemma.
Let \(\sT\) be as in Definition \ref{Defn::Spaces::Schwartz::sT}\footnote{By Proposition \ref{Prop::Spaces::Schwartz::CharacterizesT}, \(\sT=\left\{ f\in \SchwartzSpacezRopq : \supp(f)\subseteq[0,\infty)\times \R^{q} \right\}\),
but that characterization is not important for what follows.}.
Let \(\vsig_0\) and \(\vsig_1\) be as in Proposition \ref{Prop::Spaces::Schwartz::SumToIdentity},
where we use the degrees \(\Xdv_0,\Xdv_1,\ldots, \Xdv_q\) in that proposition.

Fix \(\Omegah\subseteq \Omegah_0\) open and
Fix \(\psih\in \CinftyCptSpace[\Omega]\), \(a>0\) a small constant to be chosen later, and
\(\eta\in \CinftyCptSpace[B^{1+q}(a)]\), with \(\eta=1\) on a neighborhood of \(0\).
Set,
\begin{equation}\label{Eqn::Spaces::Multiplication::DefineDh0}
    \Dh_0 f(x) := \psih(x) \int f\left( e^{t\cdot \Xh}x \right)\eta(t)\vsig_0(t)\: dt,
\end{equation}
\begin{equation}\label{Eqn::Spaces::Multiplication::DefineDhj}
    \Dh_j f(x) := \psih(x) \int f\left( e^{t\cdot \Xh}x \right)\eta(t)\Dild{2^{j}}{\vsig_1}(t)\: dt,\quad j\geq 1,
\end{equation}
\begin{equation}\label{Eqn::Spaces::Multiplication::DefinePhj}
    \Ph_j f(x) := \sum_{k=0}^j \Dh_k f(x)=
    \psih(x) \int f\left( e^{t\cdot \Xh}x \right)\eta(t)\Dild{2^{j+1}}{\vsig_0}(t)\: dt,\quad j\geq 0.
\end{equation}
We endow \(\ManifoldM\) with a smooth, strictly positive density \(\Vol\) (the choice of \(\Vol\) is not relevant).
Once we do so, we can view \(\Dh_j\) either as an operator, or as a function of two variables (see Remark \ref{Rmk::Spaces::LP::FunctionsOrOperators}).
We move between these two interpretations freely in what follows.

\begin{lemma}\label{Lemma::Spaces::Multiplication::ExistsDhj}
    If \(a>0\) is sufficiently small,
    \(\left\{ \left( \Dh_j,2^{-j} \right):j\in \Zgeq \right\}\in \ElemzFhN{\Omega}\)
    and satisfies \(\sum_{j\in \Zgeq} \Dh_j=\Mult{\psih}\).
    Moreover, if \(\Ph_j=\sum_{k=0}^j \Dh_j\), then \(\left\{ \left( \Ph_j, 2^{-j} \right) : j\in \Zgeq \right\}\in \PElemzFhN{\Omega}\).
\end{lemma}

\begin{proof}[Proof of Lemma \ref{Lemma::Spaces::Multiplication::LocalVersion::Boundary}]
    This follows immediately from Lemma \ref{Lemma::Spaces::Multiplication::ExistsDhj}.
\end{proof}

To prove Lemma \ref{Lemma::Spaces::Multiplication::ExistsDhj}, we require two additional lemmas.

\begin{lemma}\label{Lemma::Spaces::Multiplication::ExpDerivs}
    Fix a relatively compact, open set \(\Omegah_1\Subset \Omegah\).
    There exists \(L\in \Zg\) and \(a>0\) small such that the following holds.
    For \(l\in \{0,1,\ldots, q\}\), \(|\alpha|\leq L\), \(\Xdv_k\leq \DegWdv{\alpha}-\Xdv_l\),
    there exist \(g_{1,l}^{\alpha,k}(t,x),g_{2,l}^{\alpha,k}(t,x)\in \CinftySpace[B^{1+q}(a)\times \Omegah_1]\)
    such that
    \begin{equation}\label{Eqn::Spaces::Multiplication::ExpDerivs::XhOnInside}
        \partial_{t_l} f\left( e^{t\cdot \Xh} x \right)=
        \sum_{\substack{|\alpha|\leq L \\ \Xdv_k\leq \DegXdv{\alpha}+\Xdv_l}} t^{\alpha} g_{1,l}^{\alpha,k}(t,x) \left( \Xh_k f \right)\left( e^{t\cdot \Xh}x \right),
    \end{equation}
    \begin{equation}\label{Eqn::Spaces::Multiplication::ExpDerivs::XhOnOutside}
        \partial_{t_l} f\left( e^{t\cdot \Xh} x \right)=
        \sum_{\substack{|\alpha|\leq L \\ \Xdv_k\leq \DegXdv{\alpha}+\Xdv_l}} t^{\alpha} g_{2,l}^{\alpha,k}(t,x) \Xh_k  \left( f \left( e^{t\cdot \Xh}x \right) \right).
    \end{equation}
    Here, for a multi-index \(\alpha=(\alpha_0,\ldots, \alpha_q)\in \Zgeq^q\), we have written
    \(\DegXdv{\alpha}=\alpha_0\Xdv_0+\alpha_1\Xdv_1+\cdots+\alpha_q\Xdv_q\).
\end{lemma}
\begin{proof}
    Using Lemma \ref{Lemma::Spaces::Multiplication::NiceGenerators} \ref{Item::Spaces::Multiplication::NiceGenerators::NSWCommutators},
    this follows from \cite[Proposition 4.2.2 (iii) and (iv)]{StreetMaximalSubellipticity}.
\end{proof}

\begin{lemma}\label{Lemma::Spaces::Multiplcation::EhAreElem}
    Fix \(\Omegah_1 \Subset \Omegah\) open and relatively compact. For \(a>0\) sufficiently small, the following holds
    for \(h(t,x)\in \CinftySpace[B^{1+q}(a)\times \Omegah_1]\), \(\psih\in \CinftyCptSpace[\Omegah_1]\),
    and \(\eta\in \CinftyCptSpace[B^{1+q}(a)]\).
    \begin{enumerate}[(i)]
        \item\label{Item::Spaces::Multiplcation::EhAreElem::PElem} For \(\vsigt_0\in \SchwartzSpaceRopq\) with \(\supp(\vsigt_0)\subseteq [0,\infty)\times \R^q\) set
            \begin{equation*}
                \Fh_j f(x):=\psih(x) \int f\left( e^{t\cdot \Xh}x \right) h(t,x) \eta(t) \Dild{2^j}{\vsigt_0}(t)\: dt.
            \end{equation*}
            Then \(\left\{ \left( \Fh_j, 2^{-j} \right) : j\in \Zgeq \right\}\in \PElemzLieFhN{\Omegah_1}\).
        \item\label{Item::Spaces::Multiplcation::EhAreElem::Elem} For \(\vsigt_0\in \SchwartzSpaceRopq\) with \(\supp(\vsigt_0)\subseteq [0,\infty)\times \R^q\) and
            \(\vsigt_1\in \sT\), set
            \begin{equation}\label{Eqn::Spaces::Multiplcation::EhAreElem::Eh0Defn}
        \Eh_0 f(x)=\psih(x) \int f\left( e^{t\cdot \Xh}x \right) h(t,x) \eta(t) \vsigt_0(t)\: dt,
    \end{equation}
    \begin{equation}\label{Eqn::Spaces::Multiplcation::EhAreElem::EhjDefn}
        \Eh_j f(x)=\psih(x)\int f\left( e^{t\cdot \Xh}x \right) h(t,x)\eta(t)\Dild{2^j}{\vsigt_1}(t)\: dt,\quad j\geq 1.
    \end{equation}
    Then, \(\left\{ \left( \Eh_j,2^{-j} \right) :j\in \Zgeq \right\}\in \ElemzLieFhN{\Omegah_1}\).
    \end{enumerate}

    % Fix \(\Omega_1\Subset \Omega\) open and relatively compact. For \(a>0\) sufficiently small, the following holds.
    % Let \(\vsigt_0\in \SchwartzSpacezRopq\) with \(\supp(\vsigt_0)\subseteq [0,\infty)\times \R^q\),
    % \(\vsigt_1\in \sT\), \(h(t,x)\in \CinftySpace[B^{1+q}(a)\times \Omega_1]\), \(\psi\in \CinftyCptSpace[\Omega_1]\),
    % \(\eta\in \CinftyCptSpace[B^{1+q}(a)]\) set
    % \begin{equation}\label{Eqn::Spaces::Multiplcation::EhAreElem::Eh0Defn}
    %     \Eh_0 f(x)=\psi(x) \int f\left( e^{t\cdot \Xh}x \right) h(t,x) \eta(t) \vsigt_0(t)\: dt,
    % \end{equation}
    % \begin{equation}\label{Eqn::Spaces::Multiplcation::EhAreElem::EhjDefn}
    %     \Eh_j f(x)=\psi(x)\int f\left( e^{t\cdot \Xh}x \right) h(t,x)\eta(t)\Dild{2^j}{\vsigt_1}(t)\: dt,\quad j\geq 1.
    % \end{equation}
    % Then, \(\left\{ \left( \Eh_j,2^{-j} \right) :j\in \Zgeq \right\}\in \ElemzLieFhN{\Omega}\).
\end{lemma}
\begin{proof}
    % We first show the result holds with \(\ElemzSymbol\) replaced with \(\PElemzSymbol\).
    \ref{Item::Spaces::Multiplcation::EhAreElem::PElem}:
    That \(\left\{ \left( \Fh_j, 2^{-j} \right):j\in \Zgeq \right\}\in \PElemzLieFh{\Omega}\)
    follows from \cite[Lemma 5.6.6]{StreetMaximalSubellipticity} (applied with \(\nu=1\), and \((X, \vec{d})\) and \((W, \vec{\WdWithBar})\)
    in that lemma replaced by \(\XhXdv\)).

    % By Lemma \ref{Lemma::Spaces::LP::ElemDoesntDependOnChoices}, \(\PElemzLieFh{\Omega}=\PElemzFh{\Omega}\),
    % and therefore
    % \(\left\{  \right\}\)
    % and therefore Proposition \ref{Prop::Spaces::Elem::Extend::Restrict::NonVerbose}
    % shows \(\left\{ \left( \Eh_{\ManifoldN\times \ManifoldN}, 2^{-j} \right) :j\in \Zgeq\right\}\in \PElemzF{\Omega}=\PElemzLieF{\Omega}\).
    If \(a>0\) is sufficiently small, then for \(t\in B^{1+q}(a)\) with \(t_0>0\) and \(x\in \Omegah_1\cap \ManifoldN\),
    it follows from Lemma \ref{Lemma::Spaces::Multiplication::NiceGenerators} \ref{Item::Spaces::Multiplication::NiceGenerators::Xh0PointsIn}
    that \(e^{t\cdot \Xh}x\in \InteriorN\).
    % and in any coordinate system the distance between \(e^{t\cdot \Xh}x\)
    % is \(\gtrsim t_0\).
    Thus, since \(\supp(\vsigt_0),\supp(\vsigt_1)\subseteq [0,\infty)\times \R^q\), and \(\supp(\eta)\subseteq B^{1+q}(a)\),
    the formulas \eqref{Eqn::Spaces::Multiplcation::EhAreElem::Eh0Defn} and \eqref{Eqn::Spaces::Multiplcation::EhAreElem::EhjDefn}
    show 
    \(\supp(\Fh_j)\cap \left( \ManifoldN\times \ManifoldM \right)\subseteq \ManifoldN\times \ManifoldN\).
    We conclude \(\left\{ \left( \Fh_{j}, 2^{-j} \right) :j\in \Zgeq\right\}\in \PElemzLieFhN{\Omegah_1}\)
    % \(\Eh_j\big|_{\ManifoldN\times \ManifoldN}(x,y)\) vanishes to infinite order as \(y\rightarrow \BoundaryN\).
    % We conclude
    % \(\left\{ \left( \Eh_{\ManifoldN\times \ManifoldN}, 2^{-j} \right) :j\in \Zgeq\right\}\in \PElemzLieF{\Omega}\).

    \ref{Item::Spaces::Multiplcation::EhAreElem::Elem}:
    Having 
    % show the result with \(\ElemzSymbol\) replaced with \(\PElemzSymbol\),
    established  \ref{Item::Spaces::Multiplcation::EhAreElem::PElem},
    we follow the proof method from Remark \ref{Rmk::Spaces::Elem::Elem::HowToProveElem}
    to complete the proof. We also use the characterization of \(\ElemzLieFhN{\cdot}\) given in Lemmas \ref{Lemma::Spaces::Elem::Elem::ElemzhF} and \ref{Lemma::Spaces::Multiplication::OldElemResultsApply}.
    We will show
    \begin{equation}\label{Eqn::Spaces::Multiplcation::TwoShowPullOutDerivs}
        \Eh_j = \sum_{|\alpha|,|\beta|\leq 1} 2^{-j(2-|\alpha|-|\beta|)} \left( 2^{-j\Xdv} \Xh \right)^{\alpha} \Et_{j,\alpha,\beta} \left( 2^{-j\Xdv}\Xh \right)^{\beta},
    \end{equation}
    where \(\Et_{j,\alpha,\beta}\) is ``of the same form'' as \(\Eh_j\).
    In light of Remark \ref{Rmk::Spaces::Elem::Elem::HowToProveElem}, \eqref{Eqn::Spaces::Multiplcation::TwoShowPullOutDerivs}
    completes the proof by showing that \(\left\{ \left( \Eh_j,2^{-j} \right) :j\in \Zgeq \right\}\in \ElemzLieFhN{\Omegah_1}\).
    % Moreover, since we have \(\Xh_j\big|_{\Omega\cap \ManifoldN}=X_j\), \eqref{Eqn::Spaces::Multiplcation::TwoShowPullOutDerivs}
    % also shows
    % \begin{equation*}
    %     \Eh_j\big|_{\ManifoldN\times \ManifoldN} = \sum_{|\alpha|,|\beta|\leq 1} 2^{-j(2-|\alpha|-|\beta|)} \left( 2^{-j\Xdv} X \right)^{\alpha} \Et_{j,\alpha,\beta}\big|_{\ManifoldN\times \ManifoldN} \left( 2^{-j\Xdv}X \right)^{\beta},
    % \end{equation*}
    % which similarly completes the proof of 
    % \(\left\{ \left( \Eh_j\big|_{\ManifoldN\times \ManifoldN},2^{-j} \right) : j\in \Zgeq \right\}\in \ElemzLieF{\Omega\cap\ManifoldN}\).

    We turn to \eqref{Eqn::Spaces::Multiplcation::TwoShowPullOutDerivs}.  In fact, we establish \eqref{Eqn::Spaces::Multiplcation::TwoShowPullOutDerivs}
    where \(\Et_{j,\alpha,\beta}\) is a finite linear combination of terms of the same form as \(\Eh_j\),
    where the number of terms in this linear combination, and the coefficients, are bounded independent of \(j\).
    This is sufficient to carry out the above arguments and complete the proof.

    We will show
    \begin{equation}\label{Eqn::Spaces::Multiplication::PullOutLeft}
        \Eh_j=\sum_{|\alpha|\leq 1} 2^{-j(1-|\alpha|)} \left( 2^{-j\Xdv}\Xh \right)^{\alpha} F_{j,\alpha},
    \end{equation}
    \begin{equation}\label{Eqn::Spaces::Multiplication::PullOutRight}
        \Eh_j=\sum_{|\alpha|\leq 1} 2^{-j(1-|\alpha|)}  \Ft_{j,\alpha}\left( 2^{-j\Xdv}\Xh \right)^{\alpha},
    \end{equation}
    where \(F_{j,\alpha}\) and \(\Ft_{j,\alpha}\) are linear combinations of terms of the same form as \(\Eh_j\),
    where the number of terms in this linear combination, and the coefficients, are bounded independent of \(j\).
    An application of \eqref{Eqn::Spaces::Multiplication::PullOutLeft} followed by applying
    \eqref{Eqn::Spaces::Multiplication::PullOutRight} to each \(F_{j,\alpha}\) establishes \eqref{Eqn::Spaces::Multiplcation::TwoShowPullOutDerivs}.

    We turn to proving \eqref{Eqn::Spaces::Multiplication::PullOutLeft} and \eqref{Eqn::Spaces::Multiplication::PullOutRight}.
    Take \(F_{0,\{\}}=\Ft_{0,\{\}}=\Eh_0\) and \(F_{0,\alpha}=\Ft_{0,\alpha}=0\) for \(\alpha\ne \{\}\),
    which establishes 
    \eqref{Eqn::Spaces::Multiplication::PullOutLeft} and \eqref{Eqn::Spaces::Multiplication::PullOutRight}
    for \(j=0\).

    We turn to \(j\geq 1\).  Since \(\vsigt_1\in \sT\) we may write
    \(\vsigt_1=\sum_{l=0}^q \partial_{t_l} \vsigt_{1,l}\), \(\vsigt_{1,l}\in \sT\), and so
    \begin{equation}\label{Eqn::Spaces::Multiplication::EhAreElem::DoIntByParts}
    \begin{split}
         \Eh_j f(x) &= \sum_{j=0}^q \psih(x) \int f\left( e^{t\cdot \Xh}x \right) h(t,x)\eta(t) \Dild{2^j}{\partial_{t_l} \vsigt_{1,l}}(t)\: dt
         \\&= \sum_{l=0}^q 2^{-j\Xdv_l}\psih(x)  \int f\left( e^{t\cdot \Xh}x \right) h(t,x)\eta(t) \partial_{t_l} \Dild{2^j}{\vsigt_{1,l}}(t)\: dt
         \\&=\sum_{l=0}^q -2^{-j\Xdv_l}\psih(x) \int f\left( e^{t\cdot \Xh}x \right) \left( \partial_{t_l}h(t,x) \right) \eta(t)\Dild{2^j}{\vsigt_{1,l}}(t)\: dt
         \\&\quad -\sum_{l=0}^q 2^{-j\Xdv_l}\psih(x) \int f\left( e^{t\cdot \Xh}x \right) h(t,x) \left( \partial_{t_l}\eta(t) \right)\Dild{2^j}{\vsigt_{1,l}}(t)\: dt
         \\&\quad -\sum_{l=0}^q \psih(x) \int \left[ 2^{-j\Xdv_l}\partial_{t_l} f\left( e^{t\cdot \Xh}x \right) \right] h(t.x) \eta(t)\Dild{2^j}{\vsigt_{1,l}}(t)\: dt
    \end{split}
    \end{equation}
    Since \(\Xdv_l\geq 1\) and \(\partial_{t_l}h(t,x)\) and \(\partial_{t_l}\eta(t)\) are of the same form as \(h(t,x)\) and \(\eta(t)\), respectively,
    the first two sums on the right-hand side of \eqref{Eqn::Spaces::Multiplication::EhAreElem::DoIntByParts}
    are of the form \eqref{Eqn::Spaces::Multiplication::PullOutLeft} and \eqref{Eqn::Spaces::Multiplication::PullOutRight}
    (here we take \(\alpha=\{\}\)).

    We turn to the third sum on the right-hand side of \eqref{Eqn::Spaces::Multiplication::EhAreElem::DoIntByParts}.
    Using \eqref{Eqn::Spaces::Multiplication::ExpDerivs::XhOnInside}, we have
    \begin{equation}\label{Eqn::Spaces::Multiplication::EhAreElem::DerivOnInside}
    \begin{split}
         &\psih(x)
         \int \left[ 2^{-j\Xdv_l} \partial_{t_l} f\left( e^{t\cdot \Xh} x \right) \right] h(x,t)\eta(t) \Dild{2^j}{\vsigt_{1,l}}(t)\: dt
         \\&=
            \sum_{\substack{|\alpha|\leq L \\\Xdv_k\leq \DegXdv{\alpha}+\Xdv_l }} 2^{-j(\Xdv_l-\Xdv_k)}\psih(x) \int  \left( 2^{-j\Xdv_k}\Xh_k f \right)\left( e^{t\cdot \Xh}x \right) \left( g_{1,l}^{\alpha,k}(t,x) h(t,x) \right) \eta(t)  t^{\alpha} \Dild{2^j}{\vsigt_{1,l}}(t)\: dt
        \\&=\sum_{\substack{|\alpha|\leq L \\\Xdv_k\leq \DegXdv{\alpha}+\Xdv_l }}2^{-j(\DegXdv{\alpha}+\Xdv_l-\Xdv_k)} \psih(x) \int  \left( 2^{-j\Xdv_k}\Xh_k f \right)\left( e^{t\cdot \Xh}x \right) \left( g_{1,l}^{\alpha,k}(t,x) h(t,x) \right) \eta(t)   \Dild{2^j}{t^{\alpha}\vsigt_{1,l}}(t)\: dt.
    \end{split}
    \end{equation}
    Since \(2^{-j(\DegXdv{\alpha}+\Xdv_l-\Xdv_k)}\leq 1\) and \(t^{\alpha}\vsigt_{1,l}\in \sT\) (see Lemma \ref{Lemma::Spaces::Schwartz::MultiplysTbyPoly}),
    the right-hand side of \eqref{Eqn::Spaces::Multiplication::EhAreElem::DerivOnInside} is of the form \eqref{Eqn::Spaces::Multiplication::PullOutRight},
    which completes the proof of \eqref{Eqn::Spaces::Multiplication::PullOutRight}.
    Similarly, using \eqref{Eqn::Spaces::Multiplication::ExpDerivs::XhOnOutside}, we have
    \begin{equation}\label{Eqn::Spaces::Multiplication::EhAreElem::DerivOnOutside}
        \begin{split}
             &\psih(x)
             \int \left[ 2^{-j\Xdv_l} \partial_{t_l} f\left( e^{t\cdot \Xh} x \right) \right] h(x,t)\eta(t) \Dild{2^j}{\vsigt_{1,l}}(t)\: dt
             \\&=\sum_{\substack{|\alpha|\leq L \\\Xdv_k\leq \DegXdv{\alpha}+\Xdv_l }} 2^{-j(\Xdv_l-\Xdv_k)}\psih(x) \int \left[  2^{-j\Xdv_k}\Xh_k \left( f \left( e^{t\cdot \Xh}x \right) \right) \right] \left( g_{2,l}^{\alpha,k}(t,x) h(t,x) \right) \eta(t)  t^{\alpha} \Dild{2^j}{\vsigt_{1,l}}(t)\: dt
             \\&=\sum_{\substack{|\alpha|\leq L \\\Xdv_k\leq \DegXdv{\alpha}+\Xdv_l }} 2^{-j(\DegXdv{\alpha}+\Xdv_l-\Xdv_k)}\psih(x) \int  \left[2^{-j\Xdv_k}\Xh_k  \left( f \left( e^{t\cdot \Xh}x \right) \right) \right] \left( g_{2,l}^{\alpha,k}(t,x) h(t,x) \right) \eta(t)   \Dild{2^j}{t^{\alpha}\vsigt_{1,l}}(t)\: dt
             \\&=\sum_{\substack{|\alpha|\leq L \\\Xdv_k\leq \DegXdv{\alpha}+\Xdv_l }} 2^{-j(\DegXdv{\alpha}+\Xdv_l-\Xdv_k)} \left( 2^{-j\Xdv_k}\Xh_k \right) \left[ \psih(x) \int  f \left( e^{t\cdot \Xh}x \right)  \left( g_{2,l}^{\alpha,k}(t,x) h(t,x) \right) \eta(t)   \Dild{2^j}{t^{\alpha}\vsigt_{1,l}}(t)\: dt \right]
             \\&\quad-\sum_{\substack{|\alpha|\leq L \\\Xdv_k\leq \DegXdv{\alpha}+\Xdv_l }} 2^{-j\Xdv_k} 2^{-j(\DegXdv{\alpha}+\Xdv_l-\Xdv_k)}  \left( \Xh_k\psih(x) \right) \int  f \left( e^{t\cdot \Xh}x \right)  \left( g_{2,l}^{\alpha,k}(t,x) h(t,x) \right) \eta(t)   \Dild{2^j}{t^{\alpha}\vsigt_{1,l}}(t)\: dt 
             \\&\quad-\sum_{\substack{|\alpha|\leq L \\\Xdv_k\leq \DegXdv{\alpha}+\Xdv_l }} 2^{-j\Xdv_k} 2^{-j(\DegXdv{\alpha}+\Xdv_l-\Xdv_k)}  \psih(x)  \int  f \left( e^{t\cdot \Xh}x \right)  \left( \Xh_k\left( g_{2,l}^{\alpha,k}(t,x) h(t,x) \right) \right) \eta(t)   \Dild{2^j}{t^{\alpha}\vsigt_{1,l}}(t)\: dt 
        \end{split}
        \end{equation}
        Since \(2^{-j(\DegXdv{\alpha}+\Xdv_l-\Xdv_k)}\leq 1\) and \(t^{\alpha}\vsigt_{1,l}\in \sT\) (see Lemma \ref{Lemma::Spaces::Schwartz::MultiplysTbyPoly}),
        the first sum on the right-hand side of \eqref{Eqn::Spaces::Multiplication::EhAreElem::DerivOnOutside}
        is of the form \eqref{Eqn::Spaces::Multiplication::PullOutLeft} (with \(|\alpha|=1\)).
        Similarly, since \(2^{-j\Xdv_k}\leq 2^{-j}\), the second and third sum on the right-hand side of
        of \eqref{Eqn::Spaces::Multiplication::EhAreElem::DerivOnOutside} are of the form \eqref{Eqn::Spaces::Multiplication::PullOutLeft}
        with \(\alpha=\{\}\). This establishes \eqref{Eqn::Spaces::Multiplication::PullOutLeft} and completes the proof.
\end{proof}

\begin{proof}[Proof of Lemma \ref{Lemma::Spaces::Multiplication::ExistsDhj}]
    Lemmas \ref{Lemma::Spaces::LP::ElemDoesntDependOnChoices} and \ref{Lemma::Spaces::Multiplication::OldElemResultsApply}
    show that \(\ElemzFhN{\Omegah}=\ElemzLieFhN{\Omegah}\)
    and \(\PElemzFhN{\Omegah}=\PElemzLieFhN{\Omegah}\).
    Lemma \ref{Lemma::Spaces::Multiplcation::EhAreElem} \ref{Item::Spaces::Multiplcation::EhAreElem::Elem} then implies
    \(\left\{ \left( \Dh_j,2^{-j} \right):j\in \Zgeq \right\}\in \ElemzFhN{\Omegah}\).
    Also, by \eqref{Eqn::Spaces::Multiplication::DefinePhj},
    \begin{equation*}
    \begin{split}
         &\Ph_j f(x) := \sum_{k=0}^j \Dh_k f(x)=
    \psih(x) \int f\left( e^{t\cdot \Xh}x \right)\eta(t)\Dild{2^{j+1}}{\vsig_0}(t)\: dt
    =\psih(x) \int f\left( e^{t\cdot \Xh}x \right)\eta(t)\Dild{2^{j}}{ \Dild{2}{ \vsig_0 }}(t)\: dt
    \end{split}
    \end{equation*}
    and Lemma \ref{Lemma::Spaces::Multiplcation::EhAreElem} \ref{Item::Spaces::Multiplcation::EhAreElem::PElem} implies
    \(\left\{ \left( \Ph_j,2^{-j} \right):j\in \Zgeq \right\}\in \PElemzFhN{\Omegah}\).

    Using \eqref{Eqn::Spaces::Multiplication::DefineDh0} and \eqref{Eqn::Spaces::Multiplication::DefineDhj},
    and the fact that \(\vsig_0(t)+\sum_{j=1}^\infty \Dild{2^j}{\vsig_1}(t)=\delta_0(t)\) (Proposition \ref{Prop::Spaces::Schwartz::SumToIdentity})
    we have for \(f\in \CinftySpace[\ManifoldM]\)
    \begin{equation*}
    \begin{split}
         &\sum_{j=0}^\infty \Dh_j f(x) = \psih(x)\int f\left( e^{t\cdot \Xh}x \right) \eta(t) \delta_0(t)\: dt
         =\psih(x)f\left( e^{0\cdot \Xh}x \right)\eta(0) = \psih(x)f(x).
    \end{split}
    \end{equation*}
    By continuity (see Proposition \ref{Prop::Spaces::Elem::Elem::ConvergenceOfElemOps}), we have
    for all \(f\in \Distributions[\ManifoldM]\), \(\sum_{j=0}^\infty f= \psih f\).
\end{proof}

\begin{proof}[Proof of Lemma \ref{Lemma::Spaces::Multiplication::LocalVersion::Interior}]
    This is a simpler version of Lemma \ref{Lemma::Spaces::Multiplication::LocalVersion::Interior}, because we can stay away
    from \(\BoundaryN\). It is also essentially proved in \cite[Section 5.6]{StreetMaximalSubellipticity} 
    (see \cite[Proposition 5.6.3]{StreetMaximalSubellipticity}). We make some comments on the proof here.

    Let \(x\in \InteriorN\cup \left( \ManifoldM\setminus\ManifoldN \right)\) be as in the statement of the lemma.
    If \(x\in \InteriorN\) let \(\Omegah_0\Subset \InteriorN\) be a relatively compact, connected neighborhood of \(x\).
    If instead \(x\in \ManifoldM\setminus \ManifoldN\), take \(\Omegah_0\Subset \ManifoldM\setminus\ManifoldN\)
    a relatively compact, connected neighborhood of \(x\).
    Let \(\Omegah\subseteq \Omegah_0\) be open and \(\psih\in \CinftyCptSpace[\Omegah]\).

    Because \(\Omegah\cap \BoundaryN=\emptyset\), \(\ElemzFhN{\Omega}=\ElemzFh{\Omega}\) and \(\PElemzFhN{\Omega}=\PElemzFh{\Omega}\);
    indeed Definitions \ref{Defn::Spaces::Multiplication::PElemzFhN} and \ref{Defn::Spaces::Multiplication::ElemzFhN}
    only differ from Definitions \ref{Defn::Spaces::LP::PElemWWdv} and \ref{Defn::Spaces::LP::ElemWWdv}
    in the requirement that \(\supp(E)\cap \left( \ManifoldN\times \ManifoldM \right)\subseteq \ManifoldN\times \ManifoldN\).
    However, this holds trivially in this case, since either \(\Omegah\subseteq \InteriorN\) or \(\Omegah\cap \ManifoldN=\emptyset\).

    Let \(\Omegah_1\Subset \Omegah\) be such that \(\psih\in \CinftyCptSpace[\Omegah_1]\).
    Let \(\XhXdv=\left\{ \left( \Xh_0,\Xdv_1 \right),\ldots, \left( \Xh_{q},\Xdv_{q+1} \right) \right\}\subset \VectorFields{\Omegah_1}\times \Zg\)
    be such that \(\LieFilteredSheafFh\big|_{\Omegah_1}=\FilteredSheafGenBy{\XhXdv}\).
    In the proof of Lemma \ref{Lemma::Spaces::Multiplication::LocalVersion::Interior}, we used
    Lemma \ref{Lemma::Spaces::Multiplication::NiceGenerators} to make a choice of \(\XhXdv\) that behaved well at the boundary;
    in this proof, any choice of \(\XhXdv\) will do, since we stay away from \(\BoundaryN\).

    Let \(\vsig_0\in \SchwartzSpaceRopq\) satisfy \(\int \vsig_0 =1\) and \(\int t^{\alpha} \vsig_0(t)=0\), \(\forall |\alpha|\geq 1\).
    Define \(\vsig_1(t)=\Dild{2}{\vsig_0}-\vsig_0\in \SchwartzSpacezRopq\).
    In the proof of Lemma \ref{Lemma::Spaces::Multiplication::LocalVersion::Interior}, we used \(\vsig_0\) and \(\vsig_1\)
    from Proposition \ref{Prop::Spaces::Schwartz::SumToIdentity}. One can still use these choices, though we do not need
    the additional information provided by that proposition.
    Note that we have \(\vsig_0+\sum_{j=1}^\infty \Dild{2^j}{\vsig_1}=\delta_0\).
    For \(a>0\) small, define \(\Dh_j\) and \(\Ph_j\) as in \eqref{Eqn::Spaces::Multiplication::DefineDh0},
    \eqref{Eqn::Spaces::Multiplication::DefineDhj}, and \eqref{Eqn::Spaces::Multiplication::DefinePhj}.
    As before, \(\sum_{j\in \Zgeq} \Dh_j f= \psi f\).

    The same proof as above shows
    \(\left\{ \left( \Dh_j, 2^{-j} \right) : j\in \Zgeq \right\}\in \ElemzFh{\Omegah_1}\subseteq \ElemzFh{\Omegah}\)
    and \(\left\{ \left( \Ph_j, 2^{-j} \right) : j\in \Zgeq \right\}\in \PElemzFh{\Omegah_1}\subseteq \PElemzFh{\Omegah}\).
    Alternatively,
    that  \(\left\{ \left( \Ph_j, 2^{-j} \right) : j\in \Zgeq \right\}\in \PElemzLieFh{\Omegah_1}= \PElemzFh{\Omegah_1}\)
    follows from \cite[Lemma 5.6.6]{StreetMaximalSubellipticity}
    and that \(\left\{ \left( \Dh_j, 2^{-j} \right) : j\in \Zgeq \right\}\in \ElemzLieFh{\Omegah_1}=\ElemzFh{\Omegah_1}\)
    follows from \cite[Proposition 5.6.3]{StreetMaximalSubellipticity}.
\end{proof}

    \subsection{The main estimate}
    Many basic results for \(\ASpace{s}{p}{q}[\Compact][\FilteredSheafF]\) follow from
a single estimate (Proposition \ref{Prop::Spaces::MainEst::MainEst}, below).
The results in this section have similar proofs to corresponding results in
the special case of manifolds without boundary covered in \cite[Chapter 6]{StreetMaximalSubellipticity}.
We describe any needed changes, and refer the reader to that reference for detailed proofs.

Recall, in Section \ref{Section::Spaces::MainDefns}, 
we have fixed \(\psi\in \CinftyCptSpace[\ManifoldNncF]\) with \(\psi\equiv 1\) on a neighborhood of \(\Compact\Subset \ManifoldNncF\) (where
\(\Compact\) is compact),
and by 
Proposition \ref{Prop::Spaces::LP::DjExist}, we may write
\(\Mult{\psi}=\sum_{j\in \Zgeq} D_j\), where
\(\sD_0:=\left\{ \left( D_j, 2^{-j} \right) : j\in \Zgeq \right\}\in \ElemzF{\ManifoldNncF}\).

\begin{notation}\label{Notation::Spaces::MainEst::ZeroForNegIndices}
    For any sequence of operators indexed by \(j\in \Zgeq\), \(E_j\), we define \(E_j:=0\) for \(j<0\).
    For example, we have \(\Mult{\psi}=\sum_{j\in \Z} D_j\).
\end{notation}

For \(j,k,l\in \Z\) and \(N\geq 1\), set
\begin{equation*}
    F_{N,j,k,l}:=2^{N|k|+|l|} D_{j+k} D_{j+k+l}.
\end{equation*}
Set
\begin{equation*}
    \sD_N:=\left\{ \left( F_{N,j,k,l},2^{-(j+k)} \right) : j,k,l\in \Z, j+k\geq 0, |l|>|k| \right\}.
\end{equation*}

\begin{lemma}\label{Lemma::Spaces::MainEst::sDNIsElem}
    For \(N\geq 1\), \(\sD_N\in \ElemzF{\ManifoldNncF}\). Moreover, if \(\Omega\subseteq \ManifoldNncF\)
    is open and \(\sD_0\in \ElemzF{\Omega}\), then \(\sD_N\in \ElemzF{\Omega}\).
\end{lemma}
\begin{proof}
    This follows from Lemma \ref{Lemma::Spaces::ELem::Elem::ComposeElem}; see \cite[Lemma 6.4.2]{StreetMaximalSubellipticity}.
\end{proof}

\begin{proposition}\label{Prop::Spaces::MainEst::MainEst}
    Fix \(K_0\Subset \R\) compact and set \(N:=\sup\left\{ |s| : s\in K_0 \right\}+1\). Suppose
    \(s\in \Compact\), \(f\in \DistributionsZeroN\) with \(\supp(f)\subseteq \Compact\),
    and \(\VpqsENorm{f}[p][q][s][\sD_N]<\infty\). Then, for every \(\sE\in \ElemzF{\ManifoldNncF}\),
    there exists \(C=C(\sE, K_0,p,q)\geq 0\) such that
    \begin{equation*}
        \VpqsENorm{f}[p][q][s][\sE]\leq C \VpqsENorm{f}[p][q][s][\sD_0],
    \end{equation*}
    where if the right-hand side is finite, so is the left-hand side. Here, \(C\) does not depend on \(f\),
    \(\VpqsENorm{f}[p][q][s][\sD_N]\), or the particular \(s\in K_0\).
\end{proposition}
\begin{proof}
    Using 
    Proposition \ref{Prop::Spaces::Elem::PElem::PElemOpsBoundedOnVV} and 
    Lemma \ref{Lemma::Spaces::ELem::Elem::ComposeElem} this follows  as in
    \cite[Theorem 6.4.3]{StreetMaximalSubellipticity}.
\end{proof}

\begin{corollary}\label{Cor::Spaces::MainEst::VpqsESeminormIsContinuous}
    Let \(s\in \R\) and let \(\sE\in \ElemzF{\ManifoldNncF}\). Then,
    \(\VpqsENorm{\cdot}\) defines a continuous semi-norm on \(\ASpace{s}{p}{q}[\Compact][\FilteredSheafF]\).
    In particular, there is a constant \(C\geq 0\) with
    \begin{equation*}
        \VpqsENorm{f}[p][q][s][\sE]\leq C \VpqsENorm{f}[p][q][s][\sD_0]=C \ANorm{f}{s}{p}{q}[\FilteredSheafF], \quad \forall f\in \ASpace{s}{p}{q}[\Compact][\FilteredSheafF]
    \end{equation*}
    Furthermore, if \(K_0\Subset \R\) is compact, \(C\geq 0\) may be chosen independent of \(s\in K_0\).
\end{corollary}
\begin{proof}
    Using Lemma \ref{Lemma::Spaces::MainEst::sDNIsElem}, this follows from 
    Proposition \ref{Prop::Spaces::MainEst::MainEst}. See the proof of \cite[Corollary 6.4.4]{StreetMaximalSubellipticity}.
\end{proof}

\begin{corollary}\label{Cor::Spaces::MainEst::ChoiceOfNormWhichGivesFinite}
    For \(f\in \DistributionsZeroN\), we have \(f\in \ASpace{s}{p}{q}[\Compact][\FilteredSheafF]\)
    if and only if \(\VpqsENorm{f}[p][q][s][\sD_0]+\VpqsENorm{f}[p][q][s][\sD_{|s|+1}]<\infty\) and \(\supp(f)\subseteq \Compact\).
    Moreover, the norm \(\VpqsENorm{\cdot}[p][q][s][\sD_0]+\VpqsENorm{\cdot}[p][q][s][\sD_{|s|+1}]\)
    is equivalent to the norm \(\VpqsENorm{\cdot}[p][q][s][\sD_0]=\ANorm{\cdot}{s}{p}{q}[\FilteredSheafF]\) on \(\ASpace{s}{p}{q}[\Compact]\).
\end{corollary}
\begin{proof}
    This follows from Lemma \ref{Lemma::Spaces::MainEst::sDNIsElem} and Corollary \ref{Cor::Spaces::MainEst::VpqsESeminormIsContinuous}.
    See the proof of \cite[Corollary 6.4.5]{StreetMaximalSubellipticity}.
\end{proof}

\begin{remark}\label{Rmk::Spaces::MainEst::EquivNormWhichHasFinitenessProperty}
    As described in Remark \ref{Rmk::Spaces::Defns::NormFiniteDoesntMeanInSpace}, if we have
    \(f\in \DistributionsZeroN\) with \(\supp(f)\subseteq \Compact\)
    and \(\ANorm{f}{s}{p}{q}[\FilteredSheafF]=\VpqsENorm{f}[p][q][s][\sD_0]<\infty\), we have not shown
    that
    \(f\in \ASpace{s}{p}{q}[\Compact][\FilteredSheafF]\).
    However, Corollary \ref{Cor::Spaces::MainEst::ChoiceOfNormWhichGivesFinite} shows that if we replace
    \(\ANorm{f}{s}{p}{q}[\FilteredSheafF]\) with the equivalent norm \(\VpqsENorm{f}[p][q][s][\sD_0]+\VpqsENorm{f}[p][q][s][\sD_{|s|+1}]\),
    then it does follow that \(f\in \ASpace{s}{p}{q}[\Compact][\FilteredSheafF]\).
    Alternatively, when \(s>0\), one can use \(\VpqsENorm{f}[p][q][s][\sD_0]\)--see
    Section \ref{Section::Spaces::BasicProps::FiniteNorm}.
\end{remark}

\begin{corollary}\label{Cor::Spaces::MainEst::SufficesToCheckInSpaceLocally}
    Let \(\Omega\subseteq \ManifoldNncF\) be open with \(\Compact\Subset \Omega\).
    Then for \(f\in \DistributionsZeroN\) 
    the following are equivalent:
    \begin{enumerate}[(i)]
        \item\label{Item::Spaces::MainEst::SufficesToCheckInSpaceLocally::InSpace} \(f\in \ASpace{s}{p}{q}[\Compact][\FilteredSheafF]\),
        \item\label{Item::Spaces::MainEst::SufficesToCheckInSpaceLocally::LocalElem} \(\supp(f)\subseteq \Compact\) and \(\forall \sE\in \ElemzF{\Omega}\) we have \(\VpqsENorm{f}<\infty\).
    \end{enumerate}
\end{corollary}
\begin{proof}
    This follows from Lemma \ref{Lemma::Spaces::MainEst::sDNIsElem} and Corollary \ref{Cor::Spaces::MainEst::ChoiceOfNormWhichGivesFinite}.
    See the proof of \cite[Corollary 6.4.7]{StreetMaximalSubellipticity} for details.
\end{proof}

\begin{corollary}\label{Cor::Spaces::MainEst::SpacesAreDefinedLocally}
    Let \(\ManifoldNt\) be an open subset of \(\ManifoldN\), and set \(\FilteredSheafG:=\FilteredSheafF\big|_{\ManifoldNt}\).
    Then, \(\FilteredSheafG\) is a H\"ormander filtration of sheaves of vector fields on \(\ManifoldNt\)
    and \(\ManifoldNtncG=\ManifoldNncF\cap \ManifoldNt\). Furthermore, if \(\Compact\Subset \ManifoldNtncG\) is compact
    then,
    \begin{equation}\label{Eqn::Spaces::MainEst::SpacesAreDefinedLocally::SpacesAreEqual}
        \ASpace{s}{p}{q}[\Compact][\FilteredSheafF]=\ASpace{s}{p}{q}[\Compact][\FilteredSheafG],
    \end{equation}
    with equality of topologies.
\end{corollary}
\begin{proof}
    That \(\FilteredSheafG\) is a H\"ormander filtration of sheaves of vector fields on \(\ManifoldNt\)
    and \(\ManifoldNtncG=\ManifoldNncF\cap \ManifoldNt\) follows from the definitions.
    Suppose \(\Compact\Subset \ManifoldNtncG\) is compact and let \(\Omega\Subset \ManifoldNtncG\)
    be open and relatively compact with \(\Compact\Subset \Omega\).
    Note that \(\ElemzF{\Omega}=\ElemzG{\Omega}\), by definition.
    Then, by Corollary \ref{Cor::Spaces::MainEst::SufficesToCheckInSpaceLocally},
    we have \(f\in \ASpace{s}{p}{q}[\Compact][\FilteredSheafF]\)
    if and only if Corollary \ref{Cor::Spaces::MainEst::SufficesToCheckInSpaceLocally} \ref{Item::Spaces::MainEst::SufficesToCheckInSpaceLocally::LocalElem} holds
    if and only if Corollary \ref{Cor::Spaces::MainEst::SufficesToCheckInSpaceLocally} \ref{Item::Spaces::MainEst::SufficesToCheckInSpaceLocally::LocalElem} holds with \(\FilteredSheafF\)
    replaced by \(\FilteredSheafG\)
    if and only if \(f\in \ASpace{s}{p}{q}[\Compact][\FilteredSheafG]\), proving \eqref{Eqn::Spaces::MainEst::SpacesAreDefinedLocally::SpacesAreEqual}.
    To see that the topologies are equal, choose \(\sD_0\in \ElemzF{\Omega}=\ElemzG{\Omega}\),
    so that the chosen norm (\(\VpqsENorm{f}[p][q][s][\sD_0]\)) is the same on both spaces.
\end{proof}

By definition, \(\ASpace{s}{p}{q}[\Compact][\FilteredSheafF]\subseteq \DistributionsZeroN\); the next result
shows that this embedding is continuous and makes it quantitative.

\begin{proposition}\label{Prop::Spaces::MainEst::QuantitativeDistributions}
    Fix \(\Omega_1\Subset \Omega\Subset \ManifoldNncF\), relatively compact, open sets with \(\Compact\Subset \Omega_1\).
    Let \(\WWdv=\left\{ (W_1,\Wdv_1),\ldots, (W_r,\Wdv_r) \right\}\subset \VectorFields{\Omega}\times \Zg\)
    be H\"ormander vector fields with formal degrees on \(\Omega\)
    such that \(\FilteredSheafF\big|_{\Omega}=\FilteredSheafGenBy{\WWdv}\) (see Lemma \ref{Lemma::Filtrations::GeneratorsOnRelCptSet}).
    Fix \(N\in \Zgeq\). 
    There exists \(C_{N}\geq 1\)
    such that for \(s\in \R\) with \(-s+1\leq N\), we have
    \begin{equation*}
        \left| \PairDistributionAndTestFunctions{f}{g} \right| \leq C_{N} \left( \sum_{|\alpha|\leq N} \LpNorm{W^\alpha g}{\infty}[\Omega_1] \right)
        \ANorm{f}{s}{p}{q}[\FilteredSheafF],\quad \forall f\in \ASpace{s}{p}{q}[\Compact][\FilteredSheafF], g\in \TestFunctionsZeroN,
    \end{equation*}
    where \(\PairDistributionAndTestFunctions{f}{g}\) denotes the pairing of the distribution \(f\in \DistributionsZeroN\) with the
    test function \(g\in \TestFunctionsZeroN\).
    Here, \(C_N\) may depend on \(p,q\), \(\FilteredSheafF\), \(\Compact\), and \(N\), but does not depend on \(s\) with \(|s|+1\leq N\).
\end{proposition}
\begin{proof}
    Let \(\psi\in \CinftyCptSpace[\Omega_1]\) equal \(1\) on a neighborhood of \(\Compact\).
    Using Proposition \ref{Prop::Spaces::LP::DjExist}, write \(\Mult{\psi}=\sum_{j\in \Zgeq} D_j\),
    where \(\left\{ \left( D_j, 2^{-j} \right):j\in \Zgeq \right\}\in \ElemzF{\Omega_1}\).
    Using Proposition \ref{Prop::Spaces::Elem::Elem::MainProps} \ref{Item::Spaces::Elem::Elem::PullOutNDerivs},
    we may write
    \begin{equation*}
        D_{j}= \sum_{|\alpha|\leq N} 2^{-j(N-|\gamma|)}\left( 2^{-j\Wdv} W \right)^{\gamma} D_{j,\gamma},
    \end{equation*}
    where \(\sE:=\left\{ \left( D_{j,\gamma},2^{-j} \right): j\in \Zgeq, |\gamma|\leq N  \right\}\in \ElemzF{\Omega_1}\).
    We have, for \(f\in \ASpace{s}{p}{q}[\Compact][\FilteredSheafF]\) and \(g\in \TestFunctionsZeroN\),
    \begin{equation}\label{Eqn::Spaces::MainEst::QuantitativeDistributions::EstimatePairing}
    \begin{split}
         &\left| \PairDistributionAndTestFunctions{f}{g} \right|
         =\left| \PairDistributionAndTestFunctions{\psi f}{g} \right|
         \leq \sum_{j\in \Zgeq} \left| \PairDistributionAndTestFunctions{D_j f}{g} \right|
         \leq \sum_{j\in \Zgeq}\sum_{|\gamma|\leq N} 2^{-j(N-|\gamma|)} \left| \PairDistributionAndTestFunctions{(2^{-j\Wdv}W)^\gamma D_{j,\gamma} f}{g} \right|
         \\&=\sum_{j\in \Zgeq}\sum_{|\gamma|\leq N} 2^{-j(N-|\gamma|+\DegWdv{\alpha}+(s-1))} \left| \PairDistributionAndTestFunctions{ 2^{j(s-1)}D_{j,\gamma}f}{ \left[ W^{\gamma} \right]^{\transpose}g} \right|
         \\&\leq \left( \sum_{|\gamma|\leq N} \BLpNorm{\left[ W^{\gamma} \right]^{\transpose}g}{\infty}[\Omega_1] \right) \sum_{j\in \Zgeq} \BLpNorm{2^{j(s-1)}D_{j,\gamma}f}{1}[\Omega_1],
    \end{split}
    \end{equation}
    where in the last estimate we have used \(2^{-j(N-|\gamma|+\DegWdv{\alpha}+(s-1))}\leq 1\) due to the choice of \(N\),
    and we are using a smooth, strictly positive density on \(\Omega\) to define the \(\LpSpace{1}\) norm (as usual, the particular choice of
    \(\Vol\) is not relevant for our purposes).
    We have,
    \begin{equation}\label{Eqn::Spaces::MainEst::QuantitativeDistributions::Tmp1}
        \sum_{|\gamma|\leq N} \BLpNorm{\left[ W^{\gamma} \right]^{\transpose}g}{\infty}[\Omega_1] \lesssim \sum_{|\gamma|\leq N} \BLpNorm{ W^{\gamma}g}{\infty}[\Omega_1].
    \end{equation}
    And, for all \(p,q\),
    \begin{equation}\label{Eqn::Spaces::MainEst::QuantitativeDistributions::Tmp2}
    \begin{split}
         &\sum_{j\in \Zgeq} \BLpNorm{2^{j(s-1)}D_{j,\gamma}f}{1}[\Omega_1]
         \lesssim \BVNorm{\left\{ 2^{js}D_{j,\gamma}f \right\}}{p}{q}
         =\VpqsENorm{f}
         \lesssim \ANorm{f}{s}{p}{q}[\FilteredSheafF],
    \end{split}
    \end{equation}
    where the final estimate uses Corollary \ref{Cor::Spaces::MainEst::VpqsESeminormIsContinuous}.
    Combining \eqref{Eqn::Spaces::MainEst::QuantitativeDistributions::EstimatePairing}, \eqref{Eqn::Spaces::MainEst::QuantitativeDistributions::Tmp1},
    and \eqref{Eqn::Spaces::MainEst::QuantitativeDistributions::Tmp2}, completes the proof.
\end{proof}

    \subsection{Basic properties}\label{Section::Spaces::BasicProps}
    \begin{proof}[Proof of Proposition \ref{Prop::Spaces::Defns::EquivNorms}]
    The result is symmetric in \(\sD_0\) and \(\sDt_0\), so it suffices to show
    \(\VpqsENorm{f}[p][q][s][\sDt_0]\lesssim \VpqsENorm{f}[p][q][s][\sD_0]\), \(\forall f\in \ASpace{s}{p}{q}[\Compact][\FilteredSheafF]\).
    This follows immediately from Corollary \ref{Cor::Spaces::MainEst::VpqsESeminormIsContinuous}.
\end{proof}

% We turn to the proof 
% of Proposition \ref{Prop::Spaces::Defns::NormAndBanach}.
% This largely follows  the proof of \cite[Proposition 6.3.6]{StreetMaximalSubellipticity}
% in \cite[Section 6.5]{StreetMaximalSubellipticity}, and we describe the modifications needed.

\begin{proof}[Proof of Proposition \ref{Prop::Spaces::Defns::NormAndBanach}]
    This is largely the same as the proof of \cite[Proposition 6.3.6]{StreetMaximalSubellipticity}
    in \cite[Section 6.5]{StreetMaximalSubellipticity}; we describe the modifications necessary.

    For \(g\in \TestFunctionsZeroN\), by Proposition \ref{Prop::Spaces::MainEst::QuantitativeDistributions}, the map
    \(\lambda_g:\DistributionsZeroN\rightarrow \C\) given by \(\lambda_g(u)=u(g)\)
    restricts to a bounded linear functional on \(\ASpace{s}{p}{q}[\Compact][\FilteredSheafF]\).
    % We first claim that for \(g\in \TestFunctionsZeroN\), the map
    % \(\lambda_g:\DistributionsZeroN\rightarrow \C\) given by \(\lambda_g(u)=u(g)\)
    % restricts to a bounded linear functional on \(\ASpace{s}{p}{q}[\Compact][\FilteredSheafF]\).
    % This follows as in \cite[Lemma 6.5.4]{StreetMaximalSubellipticity}, we make a few comments as to the necessary
    % modifications. In \cite[Lemmas 6.5.1, 6.5.2, 6.5.3, and 6.5.4]{StreetMaximalSubellipticity},
    % we take \(\nu=1\) and replace the space 
    % \(C^\infty_{W,0}(\ManifoldM)\) with \(\TestFunctionsZeroN\).
    % From here, the proofs of \cite[Lemmas 6.5.1, 6.5.2, 6.5.3, and 6.5.4]{StreetMaximalSubellipticity}
    % go through unchanged, by using
    % Corollary \ref{Cor::Spaces::MainEst::VpqsESeminormIsContinuous} and Proposition
    % \ref{Prop::Spaces::Elem::Elem::MainProps} \ref{Item::Spaces::Elem::Elem::PullOutNDerivs} in place of the corresponding
    % results from \cite{StreetMaximalSubellipticity}. In particular, there are no boundary
    % terms in the integration by parts in the last equation in the proof of
    % \cite[Lemma 6.5.3]{StreetMaximalSubellipticity}, because \(g\in \TestFunctionsZeroN\).
    From here, the proof of \cite[Proposition 6.3.6]{StreetMaximalSubellipticity} goes through
    by taking \(g\in \TestFunctionsZeroN\) in the proof, and using
    Corollaries \ref{Cor::Spaces::MainEst::VpqsESeminormIsContinuous} and
    \ref{Cor::Spaces::MainEst::ChoiceOfNormWhichGivesFinite} in place of the corresponding results in \cite{StreetMaximalSubellipticity}.
\end{proof}

\begin{proof}[Proof of Proposition \ref{Prop::Spaces::MappingOfFuncsAndVFs}]
    We begin with  \ref{Item::Spaces::MappingOfFuncsAndVFs::VFsInF}.
    Let \(\sE\in \ElemzF{\ManifoldNncF}\). Take \(\phi\in \CinftyCptSpace[\Omega]\) with \(\phi=1\) on \(\Compact\).
    We have,
    \begin{equation*}
    \begin{split}
         &\VpqsENorm{Yf}[p][q][s-d] = \sup_{\left\{ (E_j,2^{-j}) : j\in \Zgeq \right\}\subseteq \sE} \BVNorm{ \left\{ 2^{j(s-d)} E_j \phi Y f \right\}}{p}{q}
         =\sup_{\left\{ (E_j,2^{-j}) : j\in \Zgeq \right\}\subseteq \sE} \BVNorm{ \left\{ 2^{js} E_j \phi 2^{-jd}Y f \right\}}{p}{q}.
    \end{split}
    \end{equation*}
    Set \(\sEt:=\left\{ \left( 2^{-jd}E_j \phi Y, 2^{-j} \right) : (E_j,2^{-j})\in \sE \right\}\).
    Since \(\phi\in \CinftyCptSpace[\Omega]\), we have \(\phi Y\in \FilteredSheafF[\ManifoldN][d]\).
    Using Remark \ref{Rmk::Spaces::Elem::Elem::IntegrateByPartsWithoutBoundary}, we have
    \begin{equation*}
        \left[ E\phi Y \right](x,y)= -\phi(y) Y_y E(x,y) + \phi(y)g(y)E(x,y),
    \end{equation*}
    for some \(g\in \CinftySpace[\Omega]\),
    it follows from Proposition \ref{Prop::Spaces::Elem::Elem::MainProps} 
    \ref{Item::Spaces::Elem::Elem::MultBySmooth} and \ref{Item::Spaces::Elem::Elem::DerivFuncAribtrarySection}
    that \(\sEt\in \ElemzF{\ManifoldNncF}\).
    We conclude
    \begin{equation*}
    \begin{split}
         &\VpqsENorm{Yf}[p][q][s-d] = \VpqsENorm{f}[p][q][s][\sEt]\lesssim \ANorm{f}{s}{p}{q}[\FilteredSheafF]<\infty,
    \end{split}
    \end{equation*}
    where the final inequality uses Corollary \ref{Cor::Spaces::MainEst::VpqsESeminormIsContinuous}.
    Since \(\sE\in \ElemzF{\ManifoldNncF}\) was arbitrary, and \(\supp(Yf)\subseteq \Compact\),
    we have \(Yf\in \ASpace{s}{p}{q}[\Compact][\FilteredSheafF]\). Taking \(\sE=\sD_0\)
    we have
    \(\ANorm{Yf}{s-d}{p}{q}[\FilteredSheafF]=\VpqsENorm{Yf}[p][q][s-d][\sD_0]\lesssim \ANorm{f}{s}{p}{q}[\FilteredSheafF]\),
    establishing  \ref{Item::Spaces::MappingOfFuncsAndVFs::VFsInF}.

    Since \(\ASpace{s}{p}{q}[\Compact][\FilteredSheafF]=\ASpace{s}{p}{q}[\Compact][\LieFilteredSheafF]\)
    (see Theorem \ref{Thm::Spaces::OnlyDependsOnLieFiltration} and Lemma \ref{Lemma::Filtrations::GeneratorsForLieFiltration}),
    \ref{Item::Spaces::MappingOfFuncsAndVFs::VFsInLieF} follows from \ref{Item::Spaces::MappingOfFuncsAndVFs::VFsInF}.

    \ref{Item::Spaces::MappingOfFuncsAndVFs::Funcs} has a nearly identical proof to \ref{Item::Spaces::MappingOfFuncsAndVFs::VFsInF},
    where \(d\) is taken to be \(0\), and we use Proposition \ref{Prop::Spaces::Elem::Elem::MainProps} \ref{Item::Spaces::Elem::Elem::MultBySmooth}
    to directly conclude \(\sEt\in \ElemzF{\ManifoldNncF}\).

    \ref{Item::Spaces::MappingOfFuncsAndVFs::DiffOps} is an immediate consequence of repeated
    applications of \ref{Item::Spaces::MappingOfFuncsAndVFs::VFsInF}, followed by \ref{Item::Spaces::MappingOfFuncsAndVFs::Funcs}.
\end{proof}

\begin{proposition}\label{Prop::Spaces::BasicProps::OneDerivsGivesNorms}
    Fix \(s\in \R\), \(N\in \Zgeq\).
    Let \(\Omega\subseteq \ManifoldN\) be open
    with \(\Compact\Subset \Omega\) (where \(\Compact\Subset \ManifoldNncF\) is compact)
    and suppose
    \(\FilteredSheafF\big|_{\Omega}=\FilteredSheafGenBy{\WWo}\),
    where \(\left\{ (W_1,1),\ldots, (W_r,1) \right\}\subset \VectorFields{\Omega}\times \Zg\)
    are H\"ormander vector fields on \(\Omega\) all with formal degree \(1\).
    For \(f\in \DistributionsZeroN\)
    the following are equivalent:
    \begin{enumerate}[(i)]
        \item\label{Item::Spaces::BasicProps::OneDerivsGivesNorms::fInAsplusN} \(f\in \ASpace{s+N}{p}{q}[\Compact][\FilteredSheafF]\),
        \item\label{Item::Spaces::BasicProps::OneDerivsGivesNorms::DerivsInAs} \(\forall |\alpha|\leq N\), \(W^{\alpha} f\in \ASpace{s}{p}{q}[\Compact][\FilteredSheafF]\).
    \end{enumerate}
    In this case, we have
    \begin{equation}\label{Eqn::Spaces::BasicProps::OneDerivsGivesNorms::Estimate}
        \ANorm{f}{s+N}{p}{q}[\FilteredSheafF]\approx \sum_{|\alpha|\leq N} \ANorm{W^{\alpha}f}{s}{p}{q}[\FilteredSheafF],
    \end{equation}
    where the implicit constants do not depend on \(f\), but may depend on any of the other ingredients.
\end{proposition}
\begin{proof}
    Suppose \ref{Item::Spaces::BasicProps::OneDerivsGivesNorms::fInAsplusN} holds. Then, by applying
    Proposition \ref{Prop::Spaces::MappingOfFuncsAndVFs} \ref{Item::Spaces::MappingOfFuncsAndVFs::VFsInF} (with \(d=1\))
    \(|\alpha|\leq N\) times, we see \(W^{\alpha}f\in \ASpace{s+N-|\alpha|}{p}{q}[\Compact][\FilteredSheafF]\subseteq \ASpace{s}{p}{q}[\Compact][\FilteredSheafF]\),
    and
    \begin{equation*}
        \ANorm{W^{\alpha}f}{s}{p}{q}[\FilteredSheafF]\leq \ANorm{W^{\alpha}f}{s+N-|\alpha|}{p}{q}[\FilteredSheafF]
        \lesssim \ANorm{f}{s+N}{p}{q}[\FilteredSheafF].
    \end{equation*}
    This establishes \ref{Item::Spaces::BasicProps::OneDerivsGivesNorms::DerivsInAs}
    and the \(\gtrsim\) part of \eqref{Eqn::Spaces::BasicProps::OneDerivsGivesNorms::Estimate}.

    We prove \ref{Item::Spaces::BasicProps::OneDerivsGivesNorms::DerivsInAs}\(\implies\)\ref{Item::Spaces::BasicProps::OneDerivsGivesNorms::fInAsplusN}
    and the \(\lesssim\) part of \eqref{Eqn::Spaces::BasicProps::OneDerivsGivesNorms::Estimate} in the case \(N=1\).
    The result for general \(N\) follows from this and a simple induction.

    Suppose \ref{Item::Spaces::BasicProps::OneDerivsGivesNorms::DerivsInAs} holds for \(N=1\);
    in particular since \(f\in \ASpace{s}{p}{q}[\Compact][\FilteredSheafF]\), we have \(\supp(f)\subseteq \Compact\).
    Let \(\sE\in \ElemzF{\Omega\cap \ManifoldNncF}\). For \((E_j,2^{-j})\in \sE\), Proposition \ref{Prop::Spaces::Elem::Elem::MainProps}
    \ref{Item::Spaces::Elem::Elem::PullOutNDerivs} shows that we may write
    \begin{equation*}
        E_j=\sum_{|\alpha|\leq 1} 2^{(|\alpha|-1)j} \Et_{j,\alpha} \left( 2^{-j}W \right)^{\alpha}
        =\sum_{|\alpha|\leq 1} 2^{-j} \Et_{j,\alpha} W^{\alpha},
    \end{equation*}
    where \(\sEt:=\left\{ (\Et_{j,\alpha}, 2^{-j}) : (E_j,2^{-j})\in \sE, |\alpha|\leq 1 \right\}\in \ElemzF{\Omega\cap\ManifoldNncF}\).
    We have,
    \begin{equation}\label{Eqn::Spaces::BasicProps::OneDerivsGivesNorms::DerivsBasicEst}
    \begin{split}
         &\VpqsENorm{f}[p][q][s+1][\sE]
         =\sup_{\left\{ (E_j, 2^{-j}) : j\in \Zgeq \right\}\subseteq \sE} \BVNorm{\left\{ 2^{j(s+1)} E_j f \right\}}{p}{q}
         \\&\leq \sum_{|\alpha|\leq 1} \sup_{\sup_{\left\{ (\Et_{j,\alpha}, 2^{-j}) : j\in \Zgeq \right\}\subseteq \sEt}}
         \BVNorm{\left\{ 2^{js} \Et_{j,\alpha} W^{\alpha}f \right\}}{p}{q}
         =\sum_{|\alpha|\leq 1} \VpqsENorm{W^{\alpha}f}[p][q][s][\sEt]
         \\&\lesssim \sum_{|\alpha|\leq 1} \ANorm{W^{\alpha}f}{s}{p}{q}[\FilteredSheafF]<\infty.
    \end{split}
    \end{equation}
    Using that \(\supp(f)\subseteq \Compact\),
    Corollary \ref{Cor::Spaces::MainEst::SufficesToCheckInSpaceLocally} (applied with \(\Omega\) replaced by \(\Omega\cap \ManifoldNncF\))
    shows that \eqref{Eqn::Spaces::BasicProps::OneDerivsGivesNorms::DerivsBasicEst} implies 
    \(f\in \ASpace{s+1}{p}{q}[\Compact][\FilteredSheafF]\); i.e., \ref{Item::Spaces::BasicProps::OneDerivsGivesNorms::fInAsplusN} holds
    with \(N=1\).
    Pick \(\psi\in \CinftyCptSpace[\Omega\cap \ManifoldNncF]\) with \(\psi=1\) on a neighborhood of \(\Compact\).
    By Proposition \ref{Prop::Spaces::LP::DjExist}, we may write \(\Mult{\psi}=\sum_{j\in \Zgeq} \Dt_j\),
    where \(\sDt_0:=\left\{ \left( \Dt_j, 2^{-j} \right) : j\in \Zgeq \right\}\in \ElemzF{\Omega\cap \ManifoldNncF}\).
    Using Proposition \ref{Prop::Spaces::Defns::EquivNorms} and \eqref{Eqn::Spaces::BasicProps::OneDerivsGivesNorms::DerivsBasicEst},
    we see
    \begin{equation*}
    \begin{split}
         &\ANorm{f}{s+1}{p}{q}[\FilteredSheafF] \approx \VpqsENorm{f}[p][q][s+1][\sDt_0] \lesssim  \sum_{|\alpha|\leq 1} \ANorm{W^{\alpha}f}{s}{p}{q}[\FilteredSheafF],
    \end{split}
    \end{equation*}
    establishing the \(\lesssim\) part of \eqref{Eqn::Spaces::BasicProps::OneDerivsGivesNorms::Estimate} in the case \(N=1\),
    and completing the proof.
\end{proof}

        \subsubsection{Diffeomorphism Invariance}\label{Section::Spaces::BasicProps::DiffeoInv}
        Our definitions are all ``coordinate-free,'' and therefore behave well under diffeomorphims.
In this subsection, we describe this.
Let \(\ManifoldN\) and \(\ManifoldNt\) be smooth manifolds with boundary and let
\(\Phi:\ManifoldN\xrightarrow{\sim} \ManifoldNt\) be a smooth diffeomorphism.

\begin{definition}
    Let \(\FilteredSheafF\) be a filtration of sheaves of vector fields on \(\ManifoldN\).
    We define \(\Phi_{*}\FilteredSheafF\) a filtration of sheaves of vector fields on \(\ManifoldNt\) by
    \begin{equation*}
        \left( \Phi_{*}\FilteredSheafNoSetF[d] \right)(\Phi(\Omega)) = \left\{ \Phi_{*} X: X\in \FilteredSheafF[\Omega][d] \right\}.
    \end{equation*}
\end{definition}

\begin{lemma}\label{Lemma::Spaces::BasicProps::DiffeoInv::PushForwardHormanderFiltration}
    \(\FilteredSheafF\) is a H\"ormander filtration of sheaves of vector fields on \(\ManifoldN\) if and only if
    \(\FilteredSheafG:=\Phi_{*}\FilteredSheafF\) is a H\"ormander filtration of sheaves of vector fields on \(\ManifoldNt\).
    Moreover, \(\Phi\left( \ManifoldNncF \right)=\ManifoldNtncG\).
\end{lemma}
\begin{proof}
    This follows from the definitions.
\end{proof}

\begin{proposition}\label{Prop::Spaces::BasicProps::DiffeoInv::DiffeosInduceIsomorphism}
    Let \(\FilteredSheafF\) be a H\"ormander filtration of sheaves of vector fields on \(\ManifoldN\)
    and set \(\FilteredSheafG:=\Phi_{*}\FilteredSheafF\) (so that \(\FilteredSheafG\)
    is a H\"ormander filtration of sheaves of vector fields on \(\ManifoldNt\)
    by Lemma \ref{Lemma::Spaces::BasicProps::DiffeoInv::PushForwardHormanderFiltration}).
    Then,
    \(\Phi^{*}:\TestFunctionsZero[\ManifoldNt]\rightarrow \TestFunctionsZeroN\) restricts to an isomorphism
    \(\forall \Compact\Subset \ManifoldNtncG\),
    \begin{equation*}
        \ASpace{s}{p}{q}[\Compact][\FilteredSheafG]\xrightarrow{\sim} \ASpace{s}{p}{q}[\Phi^{-1}(\Compact)][\FilteredSheafF].
    \end{equation*}
\end{proposition}
\begin{proof}
    This follows from the definitions.
\end{proof}

        \subsubsection{Distributions of finite norm}\label{Section::Spaces::BasicProps::FiniteNorm}
        In Notation \ref{Notation::Spaces::Defns::Norm}, we defined 
\(\ANorm{f}{s}{p}{q}[\FilteredSheafF]:=\VpqsENorm{f}[p][q][s][\sD_0]\).
However, if \(f\in \TestFunctionsZeroN\) with \(\supp(f)\subseteq \Compact\)
and \(\ANorm{f}{s}{p}{q}[\FilteredSheafF]<\infty\), we have not shown
\(f\in \ASpace{s}{p}{q}[\Compact][\FilteredSheafF]\).
Corollary \ref{Cor::Spaces::MainEst::ChoiceOfNormWhichGivesFinite} (see Remark \ref{Rmk::Spaces::MainEst::EquivNormWhichHasFinitenessProperty})
shows that we may replace \(\VpqsENorm{\cdot}[p][q][s][\sD_0]\) with an equivalent norm
whose finiteness does imply belonging to the space \(\ASpace{s}{p}{q}[\Compact][\FilteredSheafF]\).
For many applications, \(\VpqsENorm{\cdot}[p][q][s][\sD_0]\) is the most natural and convenient norm to use.
The next result shows that if \(s>0\), the finiteness of this norm does imply belonging to
\(\ASpace{s}{p}{q}[\Compact][\FilteredSheafF]\).

\begin{proposition}\label{Prop::Spaces::BasicProps::FiniteNorm}
    Let \(\ManifoldN\) be a smooth manifold with boundary, \(\FilteredSheafF\) a H\"ormander filtration of sheaves
    of vector fields on \(\ManifoldN\), and \(\Compact\Subset \ManifoldNncF\) compact.
    Fix \(s>0\). For \(f\in \DistributionsZeroN\) with \(\supp(f)\subseteq \Compact\), the following are equivalent:
    \begin{enumerate}[(i)]
        \item \(f\in \ASpace{s}{p}{q}[\Compact][\FilteredSheafF]\).
        \item \(\ANorm{f}{s}{p}{q}[\FilteredSheafF]<\infty\).
    \end{enumerate}
\end{proposition}
\begin{proof}
    Using Lemmas \ref{Lemma::Spaces::ELem::Elem::ComposeElem} and \ref{Lemma::Spaces::Elem::PElem::PElemOpsBoundedOnLp},
    this follows just as in the proof of \cite[Proposition 6.10.1]{StreetMaximalSubellipticity}.
\end{proof}

    \subsection{Bounded operators}
    In \cite[Chapter 5]{StreetMaximalSubellipticity}, several equivalent characterizations of singular integrals were given,
which were shown to be bounded on the adapted Besov and Triebel--Lizorkin spaces (see \cite[Theorem 6.3.10]{StreetMaximalSubellipticity}).
In this section, we adapt one version of this to the current (more general) setting of manifolds with boundary.

\begin{theorem}\label{Thm::Spaces::BoundedOps::SumOfElemIsBoundedOp}
    Let \(\sE\in \ElemzF{\Compact}\) and \(t\in \R\). For \(\sE':=\left\{ (E_j,2^{-j}):j\in \Zgeq \right\}\subseteq \sE\)
    set \(T_{\sE'}f:=\sum_{j\in \Zgeq} 2^{jt}E_jf\) (see Proposition \ref{Prop::Spaces::Elem::Elem::ConvergenceOfElemOps}
    for the convergence of this sum). Then,
    \(T_{\sE'}:\ASpace{s+t}{p}{q}[\Compact][\FilteredSheafF]\rightarrow \ASpace{s}{p}{q}[\Compact][\FilteredSheafF]\) and
    \begin{equation}\label{Eqn::Spaces::BoundedOps::SumOfElemIsBoundedOp::UniformEstimate}
        \sup_{\sE'\subseteq \sE} \|T_{\sE'}\|_{\ASpace{s+t}{p}{q}[\Compact][\FilteredSheafF]\rightarrow \ASpace{s}{p}{q}[\Compact][\FilteredSheafF]}
        <\infty,
    \end{equation}
    where the supremum is taken over all \(\sE'\) as described above.
\end{theorem}
\begin{proof}
    Let \(\sEt\in \ElemzFh{\ManifoldNncF}\). We claim
    \begin{equation}\label{Eqn::Spaces::BoundedOps::SumOfElemIsBoundedOp::ToShowBoundedSet}
        \sEh:=\left\{ \left( 2^{-jt}F_j T_{\sE'},2^{-j} \right) : \left( F_j, 2^{-j} \right)\in \sEt, \sE'\subseteq \sE \right\}\in \ElemzFh{\ManifoldNncF}.
    \end{equation}
    Indeed,
    \begin{equation}\label{Eqn::Spaces::BoundedOps::SumOfElemIsBoundedOp::ElemOpAsSum}
        2^{-jt} F_j T_{\sE'} =\sum_{k\in \Zgeq} 2^{(k-j)t} F_j E_k = \sum_{k\in \Zgeq}2^{-|j-k|} \Ft_{k,j},
    \end{equation}
    where \(\Ft_{k,j}=2^{|k-j|+t(k-j)} F_j E_k\).
    By Lemma \ref{Lemma::Spaces::ELem::Elem::ComposeElem} and Proposition \ref{Prop::Spaces::Elem::Elem::MainProps} \ref{Item::Spaces::Elem::Elem::LinearComb}
    we have 
    \begin{equation*}
        \left\{ \left( \Ft_{k,j},2^{-j} \right) : \left( F_j,2^{-j} \right)\in \sEt, \sE'\subseteq \sE, k\in \Zgeq \right\}\in \ElemzFh{\ManifoldNncF}.
    \end{equation*}
    \eqref{Eqn::Spaces::BoundedOps::SumOfElemIsBoundedOp::ToShowBoundedSet} now follows from \eqref{Eqn::Spaces::BoundedOps::SumOfElemIsBoundedOp::ElemOpAsSum}
    and Proposition \ref{Prop::Spaces::Elem::Elem::MainProps} \ref{Item::Spaces::Elem::Elem::InfiniteComb}.

    Consider, for \(f\in \ASpace{s}{p}{q}[\Compact][\FilteredSheafF]\),
    \begin{equation}\label{Eqn::Spaces::BoundedOps::SumOfElemIsBoundedOp::EstimateSup}
    \begin{split}
         &\sup_{\sE'\subseteq \sE} \VpqsENorm{T_{\sE'}f}[p][q][s][\sEt]
         =\sup_{\substack{ \left\{ (F_j,2^{-j}) : j\in \Zgeq \right\}\subseteq \sEt \\ \sE'\subseteq \sE }} \BVNorm{ \left\{ 2^{j(s+t)} 2^{-jt} F_j T_{\sE'}f \right\}_{j\in \Zgeq} }{p}{q}
         \\&\leq \sup_{\left\{ (G_j, 2^{-j}) : j\in \Zgeq \right\}\subseteq \sEh} \BVNorm{\left\{ 2^{j(s+t)}G_j f \right\}_{j\in \Zgeq}}{p}{q}
         =\VpqsENorm{f}[p][q][s+t][\sEh]
         \lesssim \ANorm{f}{s+t}{p}{q}[\FilteredSheafF]<\infty,
    \end{split}
    \end{equation}
    where the \(\lesssim\) uses Corollary \ref{Cor::Spaces::MainEst::VpqsESeminormIsContinuous}.
    Since \(\sE\in \ElemzF{\Compact}\), we have \(\supp(T_{\sE'}f)\subseteq \Compact\),
    \eqref{Eqn::Spaces::BoundedOps::SumOfElemIsBoundedOp::EstimateSup} shows \(T_{\sE'}f\in \ASpace{s}{p}{q}[\Compact][\FilteredSheafF]\).
    Taking \(\sEt=\sD_0\), \eqref{Eqn::Spaces::BoundedOps::SumOfElemIsBoundedOp::EstimateSup} gives \eqref{Eqn::Spaces::BoundedOps::SumOfElemIsBoundedOp::UniformEstimate}
    and completes the proof.
\end{proof}

It will be useful to have finer control of the convergence of the sum \(\sum_{j\in \Zgeq} 2^{jt}E_j\) in Theorem \ref{Thm::Spaces::BoundedOps::SumOfElemIsBoundedOp}.

\begin{proposition}\label{Prop::Spaces::BoundedOps::ConvergeInStrongOpTop}
    Let \(\left\{ \left( E_j,2^{-j} \right) : j\in \Zgeq \right\}\in \ElemzFh{\Compact}\) and fix \(t\in \R\).
    Then, 
    if \(q<\infty\),
    \(\sum_{j\in \Zgeq} 2^{jt}E_j\) converges in the strong operator topology as operators
    \(\ASpace{s+t}{p}{q}[\Compact][\FilteredSheafF]\rightarrow \ASpace{s}{p}{q}[\Compact][\FilteredSheafF]\).
\end{proposition}

To prove Proposition \ref{Prop::Spaces::BoundedOps::ConvergeInStrongOpTop}, we use the next lemma.

\begin{lemma}\label{Lemma::Spaces::BoundedOps::LimitOfVNorm}
    Suppose \(\left\{ \left( E_{j,l},2^{-j} \right) : j\in \Zgeq, l\in \Z \right\}\in \ElemzF{\ManifoldNncF}\)
    and \(q<\infty\).
    Then, \(\forall f\in \ASpace{s}{p}{q}[\Compact][\FilteredSheafF]\),
    \begin{equation*}
        \lim_{L\rightarrow \infty} \sum_{l\in \Z} 2^{-|l|}\BVNorm{\left\{ 2^{js} \chi_{\{j>L\}} E_{j,l}f \right\}_{j\in \Zgeq}}{p}{q}
        =0,
    \end{equation*}
    where \(\chi_{\{j>L\}}=1\) if \(j>L\) and \(0\) otherwise.
\end{lemma}
\begin{proof}
    This follows just as in \cite[Lemma 6.5.18]{StreetMaximalSubellipticity}.
\end{proof}

\begin{proof}[Proof of Proposition \ref{Prop::Spaces::BoundedOps::ConvergeInStrongOpTop}]
    Theorem \ref{Thm::Spaces::BoundedOps::SumOfElemIsBoundedOp} shows
    the partial sums \(\sum_{j=0}^L 2^{jt}E_j\) define bounded operators \(\ASpace{s+t}{p}{q}[\Compact][\FilteredSheafF]\rightarrow \ASpace{s}{p}{q}[\Compact][\FilteredSheafF]\),
    so the goal is to show that \(\sum_{j>L} 2^{jt}E_j\xrightarrow{L\rightarrow\infty}0\)  in the strong operator topology
    \(\ASpace{s+t}{p}{q}[\Compact][\FilteredSheafF]\rightarrow \ASpace{s}{p}{q}[\Compact][\FilteredSheafF]\).
    Since Theorem \ref{Thm::Spaces::BoundedOps::SumOfElemIsBoundedOp} shows
    \(\sum_{j>L} 2^{jt}E_j\) is bounded     \(\ASpace{s+t}{p}{q}[\Compact][\FilteredSheafF]\rightarrow \ASpace{s}{p}{q}[\Compact][\FilteredSheafF]\)
    for each \(L\), it suffices to show
    \begin{equation}\label{Eqn::Spaces::BoundedOps::ConvergeInStrongOpTop::ToShow}
        \BANorm{\sum_{j>L} 2^{jt}E_j f}{s}{p}{q}[\FilteredSheafF]\xrightarrow{L\rightarrow \infty} 0,\quad \forall f\in \ASpace{s+t}{p}{q}[\Compact][\FilteredSheafF].
    \end{equation}

    Take \(\sD_0=\left\{ \left( D_j, 2^{-j} \right) : j\in \Zgeq \right\}\in \ElemzF{\ManifoldNncF}\) as in Section \ref{Section::Spaces::MainDefns}.
    Set, for \(l\in \Z\), \(F_{j+l,l}:=2^{-ls+|l|} D_j E_{j+l}\), where we define \(E_k=0\) for \(k<0\),
    and therefore \(F_{k,l}=0\) if \(k<0\).
    By Lemma \ref{Lemma::Spaces::ELem::Elem::ComposeElem} and Proposition \ref{Prop::Spaces::Elem::Elem::MainProps} \ref{Item::Spaces::Elem::Elem::LinearComb},
    \(\left\{ \left( F_{j,l}, 2^{-j} \right) : j\in \Zgeq \right\}\in \ElemzF{\ManifoldNncF}\).
    We have, for \(f\in \ASpace{s+t}{p}{q}[\Compact][\FilteredSheafF]\),
    \begin{equation*}
    \begin{split}
         &\BANorm{\sum_{k>L} 2^{kt}E_k f}{s}{p}{q}[\FilteredSheafF]
         =\BVNorm{\left\{ \sum_{k>L} 2^{js} D_j 2^{kt}E_k f  \right\}_{j\in \Zgeq}}{p}{q}
         =\BVNorm{\left\{ \sum_{l\in \Z} 2^{js} \chi_{\{j+l>L\}} D_j 2^{(j+l)t}E_{j+l} f  \right\}_{j\in \Zgeq}}{p}{q}
         \\&=\BVNorm{\left\{ \sum_{l\in \Z} 2^{-|l|} 2^{(j+l)(s+t)} \chi_{\{j+l>L\}} F_{j+l,l} f  \right\}_{j\in \Zgeq}}{p}{q}
         \leq \sum_{l\in \Z} 2^{-|l|} \BVNorm{\left\{  2^{j(s+t)} \chi_{\{j>L\}} F_{j,l} f  \right\}_{j\in \Zgeq}}{p}{q}
         \\&\xrightarrow{L\rightarrow \infty}0,
    \end{split}
    \end{equation*}
    where we have applied Lemma \ref{Lemma::Spaces::BoundedOps::LimitOfVNorm} in the final step.
    This establishes \eqref{Eqn::Spaces::BoundedOps::ConvergeInStrongOpTop::ToShow} and completes the proof.
\end{proof}

    \subsection{Approximation by smooth functions}\label{Section::Spaces::SmoothApproximation}
    One can approximate elements in the classical Besov and Triebel--Lizorkin spaces by using the standard
approximation of the identity given by convolving with a smooth function with compact support. This argument does not
work in the current setting, because convolution operators do not respect the underlying Carnot--Carath\'eodory geometry.
However, the ideas can be generalized as we show in this section.

\begin{theorem}\label{Thm::Spaces::Approximation::BasicPlProps}
    Let \(\Omega\Subset \ManifoldNncF\) be open and relatively compact and fix
    \(\psi\in \CinftyCptSpace[\Omega]\).  Then, for each \(l\in \Zgeq\), there are continuous
    maps \(P_l:\DistributionsZeroN\rightarrow \CinftyCptSpace[\Omega]\) such that:
    \begin{enumerate}[(i)]
        \item\label{Item::Spaces::Approximation::BasicPlProps::ConvergeInDist} \(\forall f\in \DistributionsZeroN\), \(P_lf\xrightarrow{l\rightarrow \infty} \psi f\) in \(\DistributionsZeroN\).
        \item\label{Item::Spaces::Approximation::BasicPlProps::ConvergeInSOT} If \(q<\infty\), \(\forall f\in \ASpace{s}{p}{q}[\OmegaClosure][\FilteredSheafF]\),
            \(P_l f \xrightarrow{l\rightarrow \infty} \psi f\) in the strong topology on \(\ASpace{s}{p}{q}[\OmegaClosure][\FilteredSheafF]\).
        \item\label{Item::Spaces::Approximation::BasicPlProps::UnifBdd} \(\forall s\in \R\) and for all \(p,q\) (see Notation \ref{Notation::Spaces::XSpace}),
            \begin{equation*}
                \sup_{l\in \Zgeq} \left\| P_l  \right\|_{\ASpace{s}{p}{q}[\OmegaClosure][\FilteredSheafF]\rightarrow \ASpace{s}{p}{q}[\OmegaClosure][\FilteredSheafF]}<\infty.
            \end{equation*}
    \end{enumerate}
\end{theorem}
\begin{proof}
    Let \(\left\{ \left( D_j, 2^{-j} \right) : j\in \Zgeq \right\}\in \ElemzF{\Omega}\) be as in Proposition \ref{Prop::Spaces::LP::DjExist},
    so that for \(f\in \DistributionsZeroN\) we have (by Propositions \ref{Prop::Spaces::Elem::Elem::ConvergenceOfElemOps} and \ref{Prop::Spaces::LP::DjExist}),
    \(\Mult{\psi}=\sum_{j\in \Zgeq} D_j f=\psi f\), with convergence in \(\DistributionsZeroN\).
    We set \(P_l=\sum_{j=0}^l D_j\in \CinftyCptSpace[\Omega\times \Omega]\),
    so that \(P_l(x,y)\) vanishes to infinite order as \(y\rightarrow \BoundaryN\).
    We conclude  \(P_l:\DistributionsZeroN\rightarrow \CinftyCptSpace[\Omega]\)  is continuous (see Remark \ref{Rmk::Spaces::LP::FunctionsOrOperators})  and
    \ref{Item::Spaces::Approximation::BasicPlProps::ConvergeInDist} holds.

    \ref{Item::Spaces::Approximation::BasicPlProps::ConvergeInSOT} can be equivalently restated as
    \(P_l\xrightarrow{l\rightarrow \infty}\Mult{\psi}\) in the strong operator topology
    as operators \(\ASpace{s}{p}{q}[\OmegaClosure][\FilteredSheafF]\rightarrow \ASpace{s}{p}{q}[\OmegaClosure][\FilteredSheafF]\).
    Since \(P_l\) are the partial sums of \(\sum_{j=0}^\infty D_j\) that \(P_l\) converges in the strong operator topology follows from
    Proposition \ref{Prop::Spaces::BoundedOps::ConvergeInStrongOpTop}; that the limit of this convergent sequence must be \(\Mult{\psi}\)
    follows from \ref{Item::Spaces::Approximation::BasicPlProps::ConvergeInDist}.

    Since \(P_l=\sum_{j=0}^l D_j\) and \(\left\{ (D_j, 2^{-j}), (0,2^{-j}) : j\in \Zgeq \right\}\in \ElemzF{\Omega}\),
    \ref{Item::Spaces::Approximation::BasicPlProps::UnifBdd} follows from Theorem \ref{Thm::Spaces::BoundedOps::SumOfElemIsBoundedOp}
    (with \(t=0\)).
\end{proof}

\begin{corollary}[Smooth functions are dense]\label{Cor::Spaces::Approximation::SmoothFunctionsAreDense}
    Fix \(\Omega\Subset \ManifoldNncF\) with \(\Compact\Subset \Omega\) and
    let \(f\in \ASpace{s}{p}{q}[\Compact][\FilteredSheafF]\).
    \begin{enumerate}[(i)]
        \item\label{Item::Spaces::Approximation::SmoothFunctionsAreDense::ApproxInDist} 
        There exists \(f_j\in \CinftyCptSpace[\Omega]\) with \(\left\{ f_j :j\in \Zgeq\right\}\subset \ASpace{s}{p}{q}[\OmegaClosure][\FilteredSheafF]\)
            a bounded set, with \(f_j\rightarrow f\) in \(\DistributionsZeroN\).
        \item\label{Item::Spaces::Approximation::SmoothFunctionsAreDense::ApproxInST} If \(q<\infty\), there exists \(f_j\in \CinftyCptSpace[\Omega]\) with \(f_j\rightarrow f\) in \(\ASpace{s}{p}{q}[\OmegaClosure][\FilteredSheafF]\).
    \end{enumerate}
\end{corollary}
\begin{proof}
    Take \(\psi\in \CinftyCptSpace[\Omega]\) with \(\psi=1\) on a neighborhood of \(\Compact\).
    Let \(P_l\) be as in Theorem \ref{Thm::Spaces::Approximation::BasicPlProps} with this choice of \(\psi\)
    and set \(f_l:=P_l f\).
    \ref{Item::Spaces::Approximation::SmoothFunctionsAreDense::ApproxInDist} follows from 
    Theorem \ref{Thm::Spaces::Approximation::BasicPlProps} \ref{Item::Spaces::Approximation::BasicPlProps::ConvergeInDist} and \ref{Item::Spaces::Approximation::BasicPlProps::UnifBdd},
    while 
    \ref{Item::Spaces::Approximation::SmoothFunctionsAreDense::ApproxInST}
    follows from
    Theorem \ref{Thm::Spaces::Approximation::BasicPlProps} \ref{Item::Spaces::Approximation::BasicPlProps::ConvergeInSOT}.
\end{proof}

Corollary \ref{Cor::Spaces::Approximation::SmoothFunctionsAreDense} \ref{Item::Spaces::Approximation::SmoothFunctionsAreDense::ApproxInST} gives 
a stronger result than 
Corollary
\ref{Cor::Spaces::Approximation::SmoothFunctionsAreDense} \ref{Item::Spaces::Approximation::SmoothFunctionsAreDense::ApproxInDist},
but requires \(q<\infty\).
However, Corollary
\ref{Cor::Spaces::Approximation::SmoothFunctionsAreDense} \ref{Item::Spaces::Approximation::SmoothFunctionsAreDense::ApproxInDist} is
often a sufficiently strong approximation for our purposes, 
as the next proposition shows.

\begin{proposition}\label{Prop::Spaces::Approximation::DistributionConvgToStronger}
    Suppose \(\left\{ u_j  \right\}_{j\in \Zgeq}\subset \ASpace{s}{p}{q}[\Compact][\FilteredSheafF]\) is a bounded sequence such that
    \(u_j\rightarrow u\) in \(\DistributionsZeroN\), for some \(u\in \DistributionsZeroN\). Then,
    \begin{enumerate}[(i)]
        \item\label{Item::Spaces::Approximation::DistributionConvgToStronger::NormBound} \(u\in \ASpace{s}{p}{q}[\Compact][\FilteredSheafF]\) with 
        \(
            \ANorm{u}{s}{p}{q}[\FilteredSheafF]\leq \liminf_{j\rightarrow \infty} \ANorm{u_j}{s}{p}{q}[\FilteredSheafF].
        \)
        \item\label{Item::Spaces::Approximation::DistributionConvgToStronger::StrongConvg} If \(1<p,q<\infty\), then \(u\in \convClosure\left\{ u_j : j\in \Zgeq \right\}\), where \(\convClosure\) denotes the closure of
            the convex hull, and the closure is taken in the norm topology on \(\ASpace{s}{p}{q}[\Compact][\FilteredSheafF]\).
    \end{enumerate}
\end{proposition}
\begin{proof}
    % Let \(\Omega\Subset \ManifoldNncF\) be open with \(\Compact\Subset \Omega\).
    % We begin with \ref{Item::Spaces::Approximation::DistributionConvgToStronger::NormBound}.
    Let \(\sE\in \ElemzF{\ManifoldNncF}\).  For \((E_j, 2^{-j})\in \sE\), we have
    \(E_j(x,\cdot)\in \TestFunctionsZeroN\) for every \(x\), and therefore,
    \begin{equation}\label{Eqn::Spaces::Approximation::DistributionConvgToStronger::PtWise}
        E_j u(x) = u(E_j(x,\cdot))=\lim_{k\rightarrow \infty} u_k(E_j(x,\cdot)) = \lim_{k\rightarrow \infty} E_j u_k(x),
    \end{equation}
    where we have written \(u(E_j(x,\cdot))\) for the distribution \(u\in \DistributionsZeroN\) applied to the function
    \(E_j(x,\cdot)\). In other words, \(E_j u_k\rightarrow E_j u\) pointwise.

    We begin with \ref{Item::Spaces::Approximation::DistributionConvgToStronger::NormBound}.
    Let \(\sE\in \ElemzF{\ManifoldNncF}\).
    Using \eqref{Eqn::Spaces::Approximation::DistributionConvgToStronger::PtWise},  an elementary argument if \(p\), \(q\), or both are \(\infty\), and Fatou's lemma when \(p<\infty\) or \(q<\infty\), we have
    \begin{equation}\label{Eqn::Spaces::Approximation::DistributionConvgToStronger::LimitBound}
    \begin{split}
         &\VpqsENorm{u}[p][q][s][\sE] 
         =\sup_{\left\{ \left( E_j, 2^{-j} \right) : j\in \Zgeq \right\}\subseteq \sE} \BVNorm{\left\{ 2^{js}E_j u \right\}_{j\in \Zgeq}}{p}{q}
         \\&=\sup_{\left\{ \left( E_j, 2^{-j} \right) : j\in \Zgeq \right\}\subseteq \sE} \BVNorm{\left\{ \lim_{k\rightarrow \infty}2^{js}E_j u_k \right\}_{j\in \Zgeq}}{p}{q}
         \leq \liminf_{k\rightarrow \infty}\sup_{\left\{ \left( E_j, 2^{-j} \right) : j\in \Zgeq \right\}\subseteq \sE} \BVNorm{\left\{ 2^{js}E_j u_k \right\}_{j\in \Zgeq}}{p}{q}
         \\&=\liminf_{k\rightarrow\infty} \VpqsENorm{u_k}[p][q][s][\sE] \lesssim \liminf_{k\rightarrow\infty}\ANorm{u_k}{s}{p}{q}[\FilteredSheafF]<\infty,
    \end{split}
    \end{equation}
    where the \(\lesssim\) uses Corollary \ref{Cor::Spaces::MainEst::VpqsESeminormIsContinuous}.
    Since \(\supp(u_k)\subseteq \Compact\), \(\forall \Compact\), we have \(\supp(u)\subseteq \Compact\).
    It follows that \(u\in \ASpace{s}{p}{q}[\Compact][\FilteredSheafF]\).
    Using \eqref{Eqn::Spaces::Approximation::DistributionConvgToStronger::LimitBound},
    with \(\sE=\sD_0\), we have
    \begin{equation*}
        \ANorm{u}{s}{p}{q}[\FilteredSheafF]=\VpqsENorm{u}[p][q][s][\sD_0] \leq \liminf_{k\rightarrow\infty} \VpqsENorm{u_k}[p][q][s][\sD_0]
        =\liminf_{k\rightarrow\infty}\ANorm{u_k}{s}{p}{q}[\FilteredSheafF],
    \end{equation*}
    completing the proof of \ref{Item::Spaces::Approximation::DistributionConvgToStronger::NormBound}.

    Turning to \ref{Item::Spaces::Approximation::DistributionConvgToStronger::StrongConvg}, we know by 
    \ref{Item::Spaces::Approximation::DistributionConvgToStronger::NormBound} that \(u\in \ASpace{s}{p}{q}[\Compact][\FilteredSheafF]\).
    With \(\sD_0=\left\{ \left( D_j, 2^{-j} \right) : j\in \Zgeq \right\}\in \ElemzF{\ManifoldNncF}\) as in Section \ref{Section::Spaces::MainDefns},
    set \(\sT:\ASpace{s}{p}{q}[\Compact][\FilteredSheafF]\rightarrow \VSpace{p}{q}\) by
    \(\sT v = \left\{ 2^{js}D_j v \right\}_{j\in \Zgeq}\), so that
    \begin{equation}\label{Eqn::Spaces::Approximation::DistributionConvgToStronger::NormInTermsOfVV}
        \VNorm{\sT v}{p}{q}=\ANorm{v}{s}{p}{q}[\FilteredSheafF], \quad v\in \ASpace{s}{p}{q}[\Compact][\FilteredSheafF].
    \end{equation}

    By assumption and \eqref{Eqn::Spaces::Approximation::DistributionConvgToStronger::NormInTermsOfVV}, we have
    \(\left\{ \sT u_k : k\in \Zgeq \right\}\subset \VSpace{p}{q}\) is a bounded set. Since \(\VSpace{p}{q}\)
    is reflexive (see \cite[Section 2.11.1]{TriebelTheoryOfFunctionSpaces}--this uses \(1<p,q<\infty\)), there exists a
    subsequence \(u_{k_l}\) such that \(\sT u_{k_l}\) converges weakly in \(\VSpace{p}{q}\) to some
    \(\left\{ f_j \right\}_{j\in \Zgeq}\in \VSpace{p}{q}\). But we have already shown
    \(D_j u_{k}\rightarrow D_j u\) pointwise, and therefore \(f_j=\lim_{l\rightarrow \infty} 2^{js}D_j u_{k_l}=2^{js}D_j u\);
    i.e., \(\left\{ f_j\right\}_{j\in \Zgeq}=\sT u\).

    By Mazur's Lemma, there exists \(v_m\in \conv\left\{ u_{k_1},\ldots, u_{k_m} \right\}\) such that
    \(\lim_{m\rightarrow\infty} \sT v_m = \left\{ f_j \right\}_{j\in \Zgeq}=\sT u\), where the limit is taken in the
    norm topology on \(\VSpace{p}{q}\).  By \eqref{Eqn::Spaces::Approximation::DistributionConvgToStronger::NormInTermsOfVV},
    and using \(u, v_m\in \ASpace{s}{p}{q}[\Compact][\FilteredSheafF]\), \(\forall m\),
    this is equivalent to \(v_m\rightarrow u\) in \(\ASpace{s}{p}{q}[\Compact][\FilteredSheafF]\), completing the proof.
\end{proof}

    \subsection{Restrictions and extensions}\label{Section::Spaces::RestrictionAndExtension}
    In this section, we prove Theorem \ref{Thm::Spaces::Extension}.

For \(\Omega\subseteq \ManifoldM\), 
we write
\(\VSpace{p}{q}(\Omega)\) for
\begin{equation*}
    \VSpace{p}{q}=
    \begin{cases}
        \lqLpSpace{p}{q}[\Omega], &\text{if } \ASpace{s}{p}{q}=\BesovSpace{s}{p}{q},\\
        \LplqSpace{p}{q}[\Omega], &\text{if } \ASpace{s}{p}{q}=\TLSpace{s}{p}{q},
    \end{cases}
\end{equation*}
as in
Notation \ref{Notation::Spaces::Classical::VSpacepq}, where we make \(\Omega\) explicit.
We similarly define \(\VSpace{p}{q}(\Omega\cap \ManifoldN)\).

\begin{proof}[Proof of Theorem \ref{Thm::Spaces::Extension} \ref{Item::Spaces::Extension::Restriction}]
    Let \(\Omega\Subset \ManifoldM\) be \(\ManifoldM\)-open and relatively compact
    with \(\Omega\cap \ManifoldN\subseteq \ManifoldNncF\) and \(\Compact\Subset \Omega\).
    Fix \(\psih\in \CinftyCptSpace[\Omega]\) with \(\psih=1\) on a neighborhood of \(\Compact\).
    Let \(\left\{ \left( \Dh_j,2^{-j} \right) : j\in \Zgeq\right\}\in \ElemzFhN{\Omega}\) be 
    as in Theorem \ref{Thm::Spaces::Multiplication::MainAmbientTheorem} (with \(\Omegah\) replaced by \(\Omega\))
    so that \(\sum_{j\in \Zgeq}=\Mult{\psih}\) and if \(D_j=\Dh_j\big|_{\ManifoldN\times \ManifoldN}\),
    then Theorem \ref{Thm::Spaces::Multiplication::MainAmbientTheorem} \ref{Item::Spaces::Multiplication::MainAmbientTheorem::RestrictedAreElem}
    and \ref{Item::Spaces::Multiplication::MainAmbientTheorem::RestrictedSumToMultiplication} hold as well.

    To prove the result, we use Corollary \ref{Cor::Spaces::MainEst::ChoiceOfNormWhichGivesFinite}.
    For \(j,k,l\in \Zgeq\) and \(N:=|s|+1\), set
    \begin{equation*}
        \Fh_{N,j,k,l}:=2^{N|k|+|l|} \Dh_{j+k}\Dh_{j+k+l}, \quad F_{N,j,k,l}:=2^{N|k|+|l|} D_{j+k}D_{j+k+l},
    \end{equation*}
    where we have used Notation \ref{Notation::Spaces::MainEst::ZeroForNegIndices} (any operator with a negative
    index is defined to be zero).
    Note that \(\supp(\Fh_{N,j,k,l})\cap \left( \ManifoldN\times \ManifoldM \right)\subseteq \ManifoldN\times \ManifoldN\)
    as the same is true for \(\Dh_j\) (see Definition \ref{Defn::Spaces::Multiplication::PElemzFhN}).
    Also, we have \(\Fh_{N,j,k,l}\big|_{\ManifoldN\times \ManifoldN}=F_{N,j,k,l}\).

    Let
    \begin{equation*}
        \sDh_N:=\left\{ \left( \Fh_{j,k,l},2^{-(j+k)} \right) :j+k\geq 0\right\}, \quad \sD_N:=\left\{ \left( F_{j,k,l},2^{-(j+k)} \right) :j+k\geq 0\right\},
    \end{equation*}
    \begin{equation*}
        \sDh_0:=\left\{ \left( \Dh_j,2^{-j} \right) :j\in \Zgeq\right\}, \quad \sD_0:=\left\{ \left( D_j,2^{-j} \right) :j\in \Zgeq\right\}.
    \end{equation*}
    By Theorem \ref{Thm::Spaces::Multiplication::MainAmbientTheorem},
    \(\sDh_0\in \ElemzFhN{\Omega}\subseteq \ElemzFh{\Omega}\) and \(\sD_0\in \ElemzF{\Omega\cap \ManifoldN}\).
    By Lemma \ref{Lemma::Spaces::MainEst::sDNIsElem},
    \(\sDh_N\in \ElemzFh{\Omega}\) and \(\sD_N\in \ElemzF{\Omega\cap \ManifoldN}\).

    Due to the fact that \(\supp(\Dh_j)\cap \left( \ManifoldN\times \ManifoldM \right)\subseteq \ManifoldN\times \ManifoldN\)
    and \(D_j=\Dh_j\big|_{\ManifoldN\times \ManifoldN}\),
    we have, for \(f\in \DistributionsZeroN\),
    \begin{equation*}
        \VpqOmegaCapNsENorm{f\big|_{\TestFunctionsZeroN}}[p][q][s][\sD_0] \leq \VpqOmegasENorm{f}[p][q][s][\sDh_0],
    \end{equation*}
    and similarly for \(\sD_0\) and \(\sDh_0\) replaced by \(\sD_N\) and \(\sDh_N\).

    Suppose \(f\in \ASpace{s}{p}{q}[\Compact][\FilteredSheafFh]\).  We have,
    \begin{equation}\label{Eqn::Spaces::RestrictExtend::ProofOfRestrictionTmp1}
    \begin{split}
         & \VpqOmegaCapNsENorm{f\big|_{\TestFunctionsZeroN}}[p][q][s][\sD_0] + \VpqOmegaCapNsENorm{f\big|_{\TestFunctionsZeroN}}[p][q][s][\sD_N]
         \leq \VpqOmegasENorm{f}[p][q][s][\sDh_0]+\VpqOmegasENorm{f}[p][q][s][\sDh_N]\approx \ANorm{f}{s}{p}{q}[\FilteredSheafFh]<\infty,
    \end{split}
    \end{equation}
    where the \(\approx\) uses Corollary \ref{Cor::Spaces::MainEst::ChoiceOfNormWhichGivesFinite}.
    Since \(\supp(f\big|_{\DistributionsZeroN})\subseteq \supp(f)\cap \ManifoldN\subseteq \Compact\cap \ManifoldN\),
    \eqref{Eqn::Spaces::RestrictExtend::ProofOfRestrictionTmp1} and Corollary \ref{Cor::Spaces::MainEst::ChoiceOfNormWhichGivesFinite}
    imply \(f\big|_{\DistributionsZeroN}\in \ASpace{s}{p}{q}[\Compact\cap \ManifoldN][\FilteredSheafF]\),
    and again by \eqref{Eqn::Spaces::RestrictExtend::ProofOfRestrictionTmp1} and Corollary \ref{Cor::Spaces::MainEst::ChoiceOfNormWhichGivesFinite},
    we have 
    \begin{equation*}
        \BANorm{f\big|_{\DistributionsZeroN}}{s}{p}{q}[\FilteredSheafF]\approx  \VpqOmegaCapNsENorm{f\big|_{\TestFunctionsZeroN}}[p][q][s][\sD_0] + \VpqOmegaCapNsENorm{f\big|_{\TestFunctionsZeroN}}[p][q][s][\sD_N] \lesssim \ANorm{f}{s}{p}{q}[\FilteredSheafFh],
    \end{equation*}
    completing the proof.
\end{proof}

We turn to Theorem \ref{Thm::Spaces::Extension} \ref{Item::Spaces::Extension::Extension}.
We do this by introducing a more general class of operators, of which our extension operator is a special case.
Let \(\Compact\) and \(\Omega\) be as in Theorem \ref{Thm::Spaces::Extension} \ref{Item::Spaces::Extension::Extension}.
Let \(\Omega_1\Subset \Omega_2\Subset \Omega\) be \(\ManifoldM\)-open 
and relatively compact with
with \(\Compact\Subset \Omega_1\) and \(\Omega_2\cap \ManifoldN\subseteq \ManifoldNncF\).

Let \(\WhWdv=\left\{ \left( \Wh_1,\Wdv_1 \right),\ldots, \left( \Wh_r,\Wdv_r \right) \right\}\subset \VectorFields{\Omega_2}\times \Zgeq\)
be H\"ormander vector fields with formal degrees such that \(\FilteredSheafFh\big|_{\Omega_2}=\FilteredSheafGenBy{\WhWdv}\)
(see Lemma \ref{Lemma::Filtrations::GeneratorsOnRelCptSet}).
Set \(W_j:=\Wh_j\big|_{\Omega_2\cap \ManifoldN}\) and \(\WWdv:=\left\{ \left( W_1,\Wdv_1 \right),\ldots, \left( W_r,\Wdv_r \right)\right\}\)
so that \(\FilteredSheafF\big|_{\Omega_2\cap \ManifoldN}=\FilteredSheafGenBy{\WWdv}\) (see Proposition \ref{Prop::Filtrations::RestrictingFiltrations::CoDim0Restriction}).

Let \(\left\{ \left( E_j, 2^{-j} \right) : j\in \Zgeq \right\}\in \PElemzFh{\overline{\Omega_1}}\)
be such that \(\supp(E_j)\subseteq \ManifoldM\times \ManifoldN\), \(\forall j\).
Fix 
\(s_0\in \R\),
\(N\in \Zgeq\), and \(\alpha,\beta\) with \(|\alpha|,|\beta|\leq N\). Define
an operator (informally, taking functions on \(\ManifoldN\) to functions on \(\ManifoldM\)) by
\begin{equation*}
    T =\sum_{j\in \Zgeq} 2^{-j(2N-|\alpha|-|\beta|)} 2^{js_0} \left( 2^{-j\Wdv}\Wh \right)^{\alpha} E_j \left( 2^{-j\Wdv}W \right)^{\beta}.
\end{equation*}
Our main result regarding \(T\) is the following.

\begin{proposition}\label{Prop::Spaces::Extend::TDefinesBoundedOperator}
    For \(s\in \R\) with \(N\geq 2|s|+|s_0|+1\) and \(\max\{\Wdv_k :1\leq k\leq r\} \max\{ \ceil{ 1-s-s_0},0\}+s_0<N\), and \(f\in \ASpace{s}{p}{q}[\Compact][\FilteredSheafF]\), the sum
    \begin{equation}\label{Eqn::Spaces::Extend::TDefinesBoundedOperator::WriteDownT}
        Tf:=\sum_{j\in \Zgeq} 2^{-j(2N-|\alpha|-|\beta|)} 2^{js_0} \left( 2^{-j\Wdv}\Wh \right)^{\alpha} E_j \left( 2^{-j\Wdv}W \right)^{\beta}f
    \end{equation}
    converges in \(\Distributions[\ManifoldM]\), and defines a bounded operator
    \(T:\ASpace{s+s_0}{p}{q}[\Compact][\FilteredSheafF]\rightarrow \ASpace{s}{p}{q}[\FilteredSheafFh][\overline{\Omega_1}]\).
\end{proposition}

Note that each term of the sum \eqref{Eqn::Spaces::Extend::TDefinesBoundedOperator::WriteDownT} makes sense,
because
\(E_j\in \CinftyCptSpace[\ManifoldM\times \ManifoldM]\) with 
\(\supp(E_j)\subseteq \ManifoldM\times \ManifoldN\), and therefore \(E_j(x,\cdot)\in \TestFunctionsZeroN\) for each \(x\in \ManifoldM\).
Thus, for \(f\in \DistributionsZeroN\), we have
\(\left( 2^{-j\Wdv}\Wh \right)^{\alpha} E_j \left( 2^{-j\Wdv}W \right)^{\beta}f\in \CinftyCptSpace[\ManifoldM]\).

We turn to the proof of Proposition \ref{Prop::Spaces::Extend::TDefinesBoundedOperator}. We begin with two lemmas.

\begin{lemma}\label{Lemma::Spaces::Extend::ConvergeInDist}
    For \(s\in \R\) such that \(\max\{\Wdv_k :1\leq k\leq r\} \max\{ \ceil{ 1-s-s_0},0\}+s_0<N\), and
    \(f\in \ASpace{s}{p}{q}[\Compact][\FilteredSheafF]\), the sum
    \begin{equation*}
        Tf:=\sum_{j\in \Zgeq} 2^{-j(2N-|\alpha|-|\beta|)} 2^{js_0} \left( 2^{-j\Wdv}\Wh \right)^{\alpha} E_j \left( 2^{-j\Wdv}W \right)^{\beta}f
    \end{equation*}
    converges in  \(\Distributions[\ManifoldM]\) and defines a continuous linear map
    \(\ASpace{s}{p}{q}[\Compact][\FilteredSheafF]\rightarrow \Distributions[\ManifoldM]\).
\end{lemma}
\begin{proof}
    Let \(N_0:=\max\{ \ceil{ 1-s-s_0},0\}\).
    Set \(\Et_j:=2^{-j(N-|\beta|)} \Et_j \left( 2^{-j\Wdv}W \right)^{\beta}=2^{-j(N-|\beta|)} \Et_j \left( 2^{-j\Wdv}\Wh \right)^{\beta}\),
    so that \(\left\{ \left( \Et_j, 2^{-j} \right) : j\in \Zgeq \right\}\in \PElemzFh{\overline{\Omega_1}}\),
    with \(\supp(\Et_j)\subseteq \ManifoldM\times \ManifoldN\), \(\forall j\).

    For \(f\in \ASpace{s+s_0}{p}{q}[\Compact][\FilteredSheafF]\) and \(g\in \CinftyCptSpace[\ManifoldM]\), and using
    Proposition \ref{Prop::Spaces::MainEst::QuantitativeDistributions} (with \(\Omega\) replaced by \(\Omega_2\)),
    \begin{equation}\label{Eqn::Spaces::Extend::ConvergeInDist::Tmp1}
    \begin{split}
         &\sum_{j\in \Zgeq} 2^{-j(N-|\alpha|)+js_0} \left| \PairDistributionAndTestFunctions{g}{ \left( 2^{-j\Wdv}W \right)^{\alpha} \Et_j f} \right| 
         \\&= \sum_{j\in \Zgeq} 2^{-j(N-|\alpha|+\DegWdv{\alpha})+js_0} \left| \PairDistributionAndTestFunctions{ \Et_j^{\transpose} \left( \Wh^{\alpha} \right)^{\transpose} g  }{f} \right|
         \\&\lesssim \sum_{j\in \Zgeq} 2^{-j(N-s_0)} \ANorm{f}{s+s_0}{p}{q}[\FilteredSheafF] \sum_{|\gamma|\leq N_0} \BLpNorm{W^{\gamma} \Et_j^{\transpose} \left( \Wh^{\alpha} \right)^{\transpose}g}{\infty}
         \\&\leq \sum_{j\in \Zgeq} 2^{-j(N-s_0-\max\{\Wdv_k\}N_0)} \ANorm{f}{s+s_0}{p}{q}[\FilteredSheafF] \sum_{|\gamma|\leq N_0} \BLpNorm{\left(2^{-j\Wdv} W \right)^{\gamma} \Et_j^{\transpose} \left( \Wh^{\alpha} \right)^{\transpose}g}{\infty}.
    \end{split}
    \end{equation}
    We have \(\left[ \left(2^{-j\Wdv} W \right)^{\gamma} \Et_j^{\transpose} \right](x,y)= \left(2^{-j\Wdv} W_x \right)^{\gamma}\Et_j(y,x)\)
    and therefore, \(\left\{  \left( \left( 2^{-j\Wdv}W \right)^{\gamma}\Et_j^{\transpose},2^{-j} \right)  : j\in \Zgeq \right\}\in \PElemFh{\overline{\Omega_1}}\) (see Remark \ref{Rmk::Spaces::LP::BoundIsSymmetric})
    and
    Lemma \ref{Lemma::Spaces::Elem::PElem::PElemOpsBoundedOnLp} implies
    \(\BLpNorm{\left(2^{-j\Wdv} W \right)^{\gamma} \Et_j^{\transpose} \left( \Wh^{\alpha} \right)^{\transpose}g}{\infty} \lesssim \BLpNorm{\left( \Wh^{\alpha} \right)^{\transpose}g}{\infty}[\overline{\Omega_1}]\).
    Plugging this into \eqref{Eqn::Spaces::Extend::ConvergeInDist::Tmp1}, we see
    \begin{equation}\label{Eqn::Spaces::Extend::ConvergeInDist::Tmp2}
        \begin{split}
             &\sum_{j\in \Zgeq} 2^{-j(N-|\alpha|)+js_0} \left| \PairDistributionAndTestFunctions{g}{ \left( 2^{-j\Wdv}W \right)^{\alpha} \Et_j f} \right|              
             \\&\lesssim 
             \sum_{j\in \Zgeq}2^{-j(N-s_0-\max\{\Wdv_k\}N_0)} \ANorm{f}{s+s_0}{p}{q}[\FilteredSheafF]  \BLpNorm{ \left( \Wh^{\alpha} \right)^{\transpose}g}{\infty}[\overline{\Omega_1}].
        \end{split}
        \end{equation}
    Since \(N-s_0-\max\{\Wdv_k\}N_0>0\), by assumption, the convergence follows. Moreover, \eqref{Eqn::Spaces::Extend::ConvergeInDist::Tmp2} also implies
    \begin{equation*}
        \left| \PairDistributionAndTestFunctions{g}{Tf} \right| \lesssim \ANorm{f}{s+s_0}{p}{q}[\FilteredSheafF] \BLpNorm{ \left( \Wh^{\alpha} \right)^{\transpose}g}{\infty}[\overline{\Omega_1}]
    \end{equation*}
    giving the desired continuity as well.
\end{proof}

\begin{lemma}\label{Lemma::Spaces::Extend::ReduceToVVIneq}
    Let \(\sE,\sF\in \PElemFh{\ManifoldM}\), with \(\supp(E_j)\subseteq \ManifoldM\times \ManifoldN\), \(\forall (E,2^{-j})\in \sE\),
    and \(\sD\in \ElemzF{\ManifoldNncF}\). For any operator \(E_j\) we set \(E_j=0\) for \(j<0\).
    Let \(a_{l_1,l_2}\geq 0\) be a sequence with \(\sum_{l_1,l_2\in \Z}a_{l_1,l_2}<\infty\).
    Then,
    \begin{equation*}
    \begin{split}
         &\sum_{l_1,l_2\in \Z} a_{l_1,l_2}
         \sup_{
            \substack{\left\{ (E_j,2^{-j}) : j\in \Zgeq \right\}\subseteq \sE\\
            \left\{ (F_j,2^{-j}) : j\in \Zgeq \right\}\subseteq \sF\\
            \left\{ (D_j,2^{-j}) : j\in \Zgeq \right\}\subseteq \sD
            }}
        \BVNormOmega{\left\{ F_j E_{j+l_1} 2^{(j+l_1+l_2)s} D_{j+l_1+l_2} f \right\}_{j\in \Zgeq}}{p}{q}
        \lesssim \ANorm{f}{s}{p}{q}[\FilteredSheafF]\sum_{l_1,l_2\in \Z} a_{l_1,l_2},
    \end{split}
    \end{equation*}
    for all \(f\in \ASpace{s}{p}{q}[\Compact][\FilteredSheafF]\).
\end{lemma}
\begin{proof}
    It suffices to show
    \begin{equation}\label{Eqn::Spaces::Extend::ReduceToVVIneq::ToShow}
        \sup_{
            \substack{\left\{ (E_j,2^{-j}) : j\in \Zgeq \right\}\subseteq \sE\\
            \left\{ (F_j,2^{-j}) : j\in \Zgeq \right\}\subseteq \sF\\
            \left\{ (D_j,2^{-j}) : j\in \Zgeq \right\}\subseteq \sD
            }}
        \BVNormOmega{\left\{ F_j E_{j+l_1} 2^{(j+l_1+l_2)s} D_{j+l_1+l_2} f \right\}_{j\in \Zgeq}}{p}{q}
        \lesssim \ANorm{f}{s}{p}{q}[\FilteredSheafF],
    \end{equation}
    for all \(f\in \ASpace{s}{p}{q}[\Compact][\FilteredSheafF]\), \(\forall l_1,l_2\in \Z\).
    Applying Lemma \ref{Prop::Spaces::Elem::PElem::PElemOpsBoundedOnVV} twice, we see
    \begin{equation*}
    \begin{split}
         &\sup_{
            \substack{\left\{ (E_j,2^{-j}) : j\in \Zgeq \right\}\subseteq \sE\\
            \left\{ (F_j,2^{-j}) : j\in \Zgeq \right\}\subseteq \sF\\
            \left\{ (D_j,2^{-j}) : j\in \Zgeq \right\}\subseteq \sD
            }}
            \BVNormOmega{\left\{ F_j E_{j+l_1} 2^{(j+l_1+l_2)s} D_{j+l_1+l_2} f \right\}_{j\in \Zgeq}}{p}{q}
        \\&\lesssim 
        \sup_{            
            \left\{ (D_j,2^{-j}) : j\in \Zgeq \right\}\subseteq \sD
            }
            \BVNormOmegaN{\left\{ 2^{(j+l_1+l_2)s} D_{j+l_1+l_2} f \right\}_{j\in \Zgeq}}{p}{q}
        \\&\leq 
        \sup_{            
            \left\{ (D_j,2^{-j}) : j\in \Zgeq \right\}\subseteq \sD
            }
            \BVNormOmegaN{\left\{ 2^{js} D_{j} f \right\}_{j\in \Zgeq}}{p}{q}
        \\&\lesssim \ANorm{f}{s}{p}{q}[\FilteredSheafF],
    \end{split}
    \end{equation*}
    where the final estimate uses Corollary \ref{Cor::Spaces::MainEst::VpqsESeminormIsContinuous}.
\end{proof}

\begin{proof}[Proof of Proposition \ref{Prop::Spaces::Extend::TDefinesBoundedOperator}]
    Let \(\sF\in \ElemzFh{\ManifoldM}\). 
    For \(s\) as in the statement of the result,
    we will show
    \begin{equation}\label{Eqn::Spaces::Extend::TDefinesBoundedOperator::MainEst}
        \VpqsENorm{ Tf}[p][q][s][\sF]\lesssim \ANorm{f}{s+s_0}{p}{q}[\FilteredSheafF], \quad \forall f\in \ASpace{s+s_0}{p}{q}[\Compact][\FilteredSheafF].
    \end{equation}

    First we see why \eqref{Eqn::Spaces::Extend::TDefinesBoundedOperator::MainEst} completes the proof.
    Since \(\sF\in \ElemzFh{\ManifoldM}\) is arbitrary, and \(\supp(Tf)\subseteq \overline{\Omega_1}\),
    \eqref{Eqn::Spaces::Extend::TDefinesBoundedOperator::MainEst} implies \(Tf\in \ASpace{s}{p}{q}[\overline{\Omega_1}][\FilteredSheafFh]\).
    Moreover, by picking \(\sF\) as in Notation \ref{Notation::Spaces::Defns::Norm}, \eqref{Eqn::Spaces::Extend::TDefinesBoundedOperator::MainEst}
    shows \(\ANorm{Tf}{s}{p}{q}[\FilteredSheafFh]\lesssim \ANorm{f}{s+s_0}{p}{q}[\FilteredSheafF]\), establishing the boundedness
    of \(T\).

    We turn to proving \eqref{Eqn::Spaces::Extend::TDefinesBoundedOperator::MainEst}.
    Fix \(\psi\in \CinftyCptSpace[\Omega_1\cap \ManifoldN]\) with \(\psi=1\) on a neighborhood of \(\Compact\).
    Using Proposition \ref{Prop::Spaces::LP::DjExist}
    let \(\left\{ \left( D_j, 2^{-j} \right) :j\in \Zgeq \right\}\in \ElemzF{\Omega_1\cap \ManifoldN}\)
    be such that \(\sum_{j\in \Zgeq}D_j =\Mult{\psi}\).
    Let \(\left\{ \left( F_j,2^{-j} \right):j\in \Zgeq \right\}\in \sF\); all estimates below will be independent
    of the particular subset chosen from \(\sF\).

    In what follows, we use the convention that any operator \(E_j\) is defined to be zero for \(j<0\).
    We consider,
    \begin{equation}\label{Eqn::Spaces::Extend::TDefinesBoundedOperator::Tmp1}
    \begin{split}
         & \BVNorm{  \left\{ 2^{js} F_j T f \right\}_{j\in \Zgeq}}{p}{q}
         =\BVNorm{\left\{ 2^{js} F_j T \sum_{k\in \Zgeq}  D_k f \right\}_{j\in \Zgeq}}{p}{q}
         \\&=\BVNorm{\left\{ \sum_{l_1,l_2\in \Z} 2^{js} F_j 2^{-(j+l_1)(2N-|\alpha|-|\beta|-s_0)} \left( 2^{-(j+l_1)\Wdv}\Wh \right)^{\alpha}  E_{j+l_1}\left( 2^{-(j+l_1)\Wdv} W \right)^{\beta} D_{j+l_1+l_2} f   \right\}}{p}{q}
         \\&\leq \sum_{l_1,l_2\in \Z}\BVNorm{\left\{  2^{js} F_j 2^{-(j+l_1)(2N-|\alpha|-|\beta|-s_0)} \left( 2^{-(j+l_1)\Wdv}\Wh \right)^{\alpha}  E_{j+l_1}\left( 2^{-(j+l_1)\Wdv} W \right)^{\beta} D_{j+l_1+l_2} f   \right\}}{p}{q}.
    \end{split}
    \end{equation}

    Our goal is to show that the right-hand side of \eqref{Eqn::Spaces::Extend::TDefinesBoundedOperator::Tmp1}
    is \(\lesssim \ANorm{f}{s+s_0}{p}{q}[\FilteredSheafF]\). To establish this, 
    we manipulate terms to see
    that the right-hand side of \eqref{Eqn::Spaces::Extend::TDefinesBoundedOperator::Tmp1}
    in the form covered by
    Lemma \ref{Lemma::Spaces::Extend::ReduceToVVIneq}.

    We separate the sum on the right-hand side of \eqref{Eqn::Spaces::Extend::TDefinesBoundedOperator::Tmp1}
    into four parts: \(l_1\geq |l_2|\), \(-l_1\geq |l_2|\), \(l_2\geq |l_1|\), and \(-l_2\geq |l_1|\).
    We have,
    \begin{equation}\label{Eqn::Spaces::Extend::TDefinesBoundedOperator::l1geql2::1}
    \begin{split}
         &\sum_{l_1\geq |l_2|}\BVNorm{\left\{  2^{js} F_j 2^{-(j+l_1)(2N-|\alpha|-|\beta|-s_0)} \left( 2^{-(j+l_1)\Wdv}\Wh \right)^{\alpha}  E_{j+l_1}\left( 2^{-(j+l_1)\Wdv} W \right)^{\beta} D_{j+l_1+l_2} f   \right\}}{p}{q}
         \\&= \sum_{l_1\geq |l_2|} 2^{-(l_1+l_2)s-l_2s_0-l_1(2N-|\alpha|-|\beta|+\DegWdv{\alpha})} 
         \\&\quad\quad\quad\BVNorm{ \left\{ 2^{-j(2N-|\alpha|-|\beta|)} \left[ F_j \left( 2^{j\Wdv}\Wh \right)^{\alpha} \right] \left[ E_{j+l_1} \left( 2^{-(j+l_1)\Wdv}W \right)^{\beta} \right] 2^{(s+s_0)(j+l_1+l_2)} D_{j+l_1+l_2}f \right\} }{p}{q}.
    \end{split}
    \end{equation}
    Since \(2^{-j(2N-|\alpha|-|\beta|)}\leq 1\), we have
    \begin{equation}\label{Eqn::Spaces::Extend::TDefinesBoundedOperator::l1geql2::2}
    \begin{split}
         &\sum_{l_1\geq |l_2|} 2^{-(l_1+l_2)s-l_2s_0-l_1(2N-|\alpha|-|\beta|+\DegWdv{\alpha})} 
         \\&\quad\quad\quad\BVNorm{ \left\{ 2^{-j(2N-|\alpha|-|\beta|)} \left[ F_j \left( 2^{j\Wdv}\Wh \right)^{\alpha} \right] \left[ E_{j+l_1} \left( 2^{-(j+l_1)\Wdv}W \right)^{\beta} \right] 2^{(s+s_0)(j+l_1+l_2)} D_{j+l_1+l_2}f \right\} }{p}{q}
         \\&\lesssim \ANorm{f}{s+s_0}{p}{q}[\FilteredSheafF]\sum_{l_1\geq |l_2|} 2^{-(l_1+l_2)s-l_2s_0-l_1(2N-|\alpha|-|\beta|+\DegWdv{\alpha})}.
    \end{split}
    \end{equation}
    by Lemma \ref{Lemma::Spaces::Extend::ReduceToVVIneq} since it is of the form covered in that lemma;
    more precisely,
    \begin{equation*}
        \left\{ \left(  F_j \left( 2^{j\Wdv}\Wh \right)^{\alpha},2^{-j} \right) : (F_j, 2^{-j})\in \sF \right\}\in \PElemFh{\ManifoldM},
    \end{equation*}
    \(\supp\left(  E_{j} \left( 2^{-(j)\Wdv}W \right)^{\beta} \right)\subseteq \ManifoldM\times \ManifoldN\)
    and
    \begin{equation*}
        \left\{ \left(  E_{j} \left( 2^{-(j)\Wdv}W \right)^{\beta},2^{-j}\right) :j\in \Zgeq\right\}\in \PElemFh{\ManifoldM},
    \end{equation*}
    and
    \begin{equation*}
        \left\{ \left( D_j,2^{-j} \right) : j\in \Zgeq \right\}\in \ElemzF{\ManifoldNncF}.
    \end{equation*}
    Also, we have 
    \begin{equation}\label{Eqn::Spaces::Extend::TDefinesBoundedOperator::l1geql2::3}
    \begin{split}
         &\sum_{l_1\geq |l_2|} 2^{-(l_1+l_2)s-l_2s_0-l_1(2N-|\alpha|-|\beta|+\DegWdv{\alpha})} 
         \leq \sum_{l_1\geq |l_2|} 2^{-(l_1+l_2)s-l_2s_0-l_1N} 
         \sum_{l_1\geq |l_2|} 2^{-l_1}\lesssim 1,
    \end{split}
    \end{equation}
    by the choice of \(N\).
    Combining \eqref{Eqn::Spaces::Extend::TDefinesBoundedOperator::l1geql2::1}, \eqref{Eqn::Spaces::Extend::TDefinesBoundedOperator::l1geql2::2},a
    and \eqref{Eqn::Spaces::Extend::TDefinesBoundedOperator::l1geql2::3}
    shows \(\sum_{l_1\geq |l_2|} \cdot \lesssim \ANorm{f}{s+s_0}{p}{q}[\FilteredSheafF]\), as desired.

    For \(\sum_{-l_1\geq |l_2|}\), we use Proposition \ref{Prop::Spaces::Elem::Elem::MainProps} \ref{Item::Spaces::Elem::Elem::PullOutNDerivs}
    to write
    \begin{equation*}
        F_j=\sum_{|\gamma|\leq N} 2^{-j(N-|\gamma|)} F_{j,\gamma} \left( 2^{-j}\Wh \right)^{\gamma},
    \end{equation*}
    where \(\left\{ \left( F_{j,\gamma},2^{-j} \right) : (F_j, 2^{-j})\in \sF \right\}\in \ElemzFh{\ManifoldM}\).
    Then, we have
    \begin{equation*}%\label{Eqn::Spaces::Extend::TDefinesBoundedOperator::ml1geql2::1}
    \begin{split}
         &\sum_{-l_1\geq |l_2|}\BVNorm{\left\{  2^{js} F_j 2^{-(j+l_1)(2N-|\alpha|-|\beta|-s_0)} \left( 2^{-(j+l_1)\Wdv}\Wh \right)^{\alpha}  E_{j+l_1}\left( 2^{-(j+l_1)\Wdv} W \right)^{\beta} D_{j+l_1+l_2} f   \right\}}{p}{q}
         \\&\leq \sum_{|\gamma|\leq N} \sum_{-l_1\geq |l_2|}
         2^{l_1(N-|\gamma|+\DegWdv{\gamma}) - (l_1+l_2)s-l_2s_0}
         \\&\quad\quad\BVNorm{
            \left\{  F_{j,\gamma} 
            \left[ 2^{-(j+l_1) (2N-|\alpha|-|\beta|)} \left( 2^{-(j+l_1)\Wdv}\Wh \right)^{\gamma}\left( 2^{-(j+l_1)\Wdv}\Wh \right)^{\alpha} E_{j+l_1} \left( 2^{-(j+l_1)}W \right)^{\beta} \right]
            2^{(j+l_1+l_2)(s+s_0)}D_{j+l_1+l_2}f
         \right\}}{p}{q}
         \\&\lesssim \ANorm{f}{s+s_0}{p}{q}[\FilteredSheafF] \sum_{-l_1\geq |l_2|} 2^{l_1(N-|\gamma|+\DegWdv{\gamma}) - (l_1+l_2)s-l_2s_0},
    \end{split}
    \end{equation*}
    where we have again used Lemma \ref{Lemma::Spaces::Extend::ReduceToVVIneq} (which applies for the same reasons as before), and the fact that
    \(2^{-(j+l_1)(2N-|\alpha|-|\beta|)}\leq 1\) in the above, because \(E_{j+l_1}=0\) when \(j+l_1<0\).
    Using the choice of \(N\), we have
    \begin{equation*}
    \begin{split}
         &\sum_{-l_1\geq |l_2|} 2^{l_1(N-|\gamma|+\DegWdv{\gamma}) - (l_1+l_2)s-l_2s_0}
         \leq \sum_{-l_1\geq |l_2|}2^{-N|l_1| - (l_1+l_2)s-l_2s_0}
         \leq \sum_{-l_1\geq |l_2|}2^{-|l_1|} \lesssim 1.
    \end{split}
    \end{equation*}
    Combining the previous two equations gives \(\sum_{-l_1\geq |l_2|} \cdot \lesssim \ANorm{f}{s+s_0}{p}{q}[\FilteredSheafF]\), as desired.

    We have,
    \begin{equation}\label{Eqn::Spaces::Extend::TDefinesBoundedOperator::ml2geql1::1}
    \begin{split}
         &\sum_{-l_2\geq |l_1|}\BVNorm{\left\{  2^{js} F_j 2^{-(j+l_1)(2N-|\alpha|-|\beta|-s_0)} \left( 2^{-(j+l_1)\Wdv}\Wh \right)^{\alpha}  E_{j+l_1}\left( 2^{-(j+l_1)\Wdv} W \right)^{\beta} D_{j+l_1+l_2} f   \right\}}{p}{q}
         \\&\leq \sum_{-l_2\geq |l_1|}
         2^{l_2(N-|\beta|+\DegWdv{\beta}) + (l_1+l_2)s+l_2s_0}
         \\&\BVNorm{
            \left\{ F_j \left[ 2^{-(j+l_1)(N-|\alpha|)} 2^{-(j+l_1)} E_{j+l_1} \right] 
            \left[ 2^{-(j+l_1+l_2)(N-|\beta|)} 2^{-(j+l_1+l_2)(s+s_0)}\left( 2^{-(j+l_1+l_2)\Wdv}W \right)^{\beta}D_{j+l_1+L_2}\right] f \right\}
         }{p}{q} 
         \\&\lesssim \ANorm{f}{s+s_0}{p}{q}[\FilteredSheafF] \sum_{-l_2\geq |l_1|}
         2^{l_2(N-|\beta|+\DegWdv{\beta}) + (l_1+l_2)s+l_2s_0},
    \end{split}
    \end{equation}
    where we have again applied Lemma \ref{Lemma::Spaces::Extend::ReduceToVVIneq};
    and have used \(2^{-(j+l_1+l_2)(N-|\beta|)}\leq 1\) (because \(D_{j+l_1+l_2}=0\) if \(j+l_1+l_2<0\))
    and have applied Proposition \ref{Prop::Spaces::Elem::Elem::MainProps} \ref{Item::Spaces::Elem::Elem::DerivOp} to see
    \begin{equation*}
        \left\{ \left( \left( 2^{-(j)\Wdv}W \right)^{\beta}D_{j}, 2^{-j} \right) : j\in \Zgeq \right\}\in \ElemzF{\ManifoldNncF}.
    \end{equation*}
    We have, by the choice of \(N\),
    \begin{equation}\label{Eqn::Spaces::Extend::TDefinesBoundedOperator::ml2geql1::2}
        \sum_{-l_2\geq |l_1|}
         2^{l_2(N-|\beta|+\DegWdv{\beta}) + (l_1+l_2)s+l_2s_0}
         \leq \sum_{-l_2\geq |l_1|}
         2^{l_2N + (l_1+l_2)s+l_2s_0}
         \leq \sum_{-l_2\geq |l_1|}2^{-|l_2|}\lesssim 1.
    \end{equation}
    Plugging \eqref{Eqn::Spaces::Extend::TDefinesBoundedOperator::ml2geql1::2} in to \eqref{Eqn::Spaces::Extend::TDefinesBoundedOperator::ml2geql1::1}
    gives   \(\sum_{-l_2\geq |l_1|} \cdot \lesssim \ANorm{f}{s+s_0}{p}{q}[\FilteredSheafF]\), as desired.

    For \(\sum_{l_2\geq |l_1|}\), we use Proposition \ref{Prop::Spaces::Elem::Elem::MainProps} \ref{Item::Spaces::Elem::Elem::PullOutNDerivs}
    to write
    \(D_j=\sum_{|\gamma|\leq N} 2^{-j(N-|\gamma|)} \left( 2^{-j\Wdv}W \right)^{\gamma} D_{j,\gamma}\),
    where
    \begin{equation*}
        \left\{ \left( D_{j,\gamma},2^{-j} \right) : j\in \Zgeq \right\}\in \ElemzF{\ManifoldNncF}.
    \end{equation*}
    We have,
    \begin{equation*}
    \begin{split}
         &\sum_{l_2\geq |l_1|}\BVNorm{\left\{  2^{js} F_j 2^{-(j+l_1)(2N-|\alpha|-|\beta|-s_0)} \left( 2^{-(j+l_1)\Wdv}\Wh \right)^{\alpha}  E_{j+l_1}\left( 2^{-(j+l_1)\Wdv} W \right)^{\beta} D_{j+l_1+l_2} f   \right\}}{p}{q}
         \\&\leq 
         \sum_{|\gamma|\leq N}\sum_{l_2\geq |l_1|}
         2^{-l_2(N-|\gamma|+\DegWdv{\gamma}) + (l_1+l_2)s+l_2s_0}
         \\&
         \BVNorm{
            \left\{ 
                F_j 
                \left[  2^{-(j+l_1)(3N-|\alpha|-|\beta|-|\gamma|)} \left( 2^{-(j+l_1)\Wdv}\Wh \right)^{\alpha} E_{j+l_1} \left( 2^{-(j+l_1)\Wdv}W \right)^{\beta}\left( 2^{-(j+l_1)\Wdv}W \right)^{\gamma} \right]
                2^{(j+l_1+l_2)(s+s_0)} D_{j+l_1+l_2,\gamma}f
             \right\}
         }{p}{q}
         \\&\lesssim \ANorm{f}{s+s_0}{p}{q}[\FilteredSheafF]\sum_{l_2\geq |l_1|}
         2^{-l_2(N-|\gamma|+\DegWdv{\gamma}) + (l_1+l_2)s+l_2s_0},
    \end{split}
    \end{equation*}
    where we have  again applied Lemma \ref{Lemma::Spaces::Extend::ReduceToVVIneq}
    and used \( 2^{-(j+l_1)(3N-|\alpha|-|\beta|-|\gamma|)}\leq 1\), since \(E_{j+l_1}=0\) if \(j+l_1<0\).
    Using our choice of \(N\), we have
    \begin{equation*}
        \sum_{l_2\geq |l_1|}
         2^{-l_2(N-|\gamma|+\DegWdv{\gamma}) + (l_1+l_2)s+l_2s_0}
         \leq 2^{-l_2N+ (l_1+l_2)s+l_2s_0}\leq \sum_{l_2\geq |l_1|}2^{-|l_2|}\lesssim 1.
    \end{equation*}
    Combining the previous two equations gives \(\sum_{l_2\geq |l_1|} \cdot \lesssim \ANorm{f}{s+s_0}{p}{q}[\FilteredSheafF]\), as desired.

    Plugging the above into \eqref{Eqn::Spaces::Extend::TDefinesBoundedOperator::Tmp1}, we see
    \begin{equation*}
        \BVNorm{  \left\{ 2^{js} F_j T f \right\}_{j\in \Zgeq}}{p}{q}
        \lesssim \ANorm{f}{s+s_0}{p}{q}[\FilteredSheafF],
    \end{equation*}
    and taking the supremum over all \(\left\{ (F_j, 2^{-j}) : j\in \Zgeq \right\}\subseteq \sF\)
    yields \eqref{Eqn::Spaces::Extend::TDefinesBoundedOperator::MainEst} and completes the proof.
\end{proof}

\begin{proof}[Proof of Theorem \ref{Thm::Spaces::Extension} \ref{Item::Spaces::Extension::Extension}]
    Fix \(\Omega_0\Subset \Omega_1\) an \(\ManifoldM\)-open set with \(\Compact\Subset \Omega_0\).
    Let \(\psi\in \CinftyCptSpace[\Omega_0\cap \ManifoldN]\) be equal to \(1\) on an  \(\ManifoldN\)-neighborhood of \(\Compact\).
    Using Proposition \ref{Prop::Spaces::LP::DjExist}, write \(\Mult{\psi}=\sum_{j\in \Zgeq}D_j\),
    where \(\left\{ \left( D_j,2^{-j} \right) : j\in \Zgeq \right\}\in \ElemzF{\Omega_0\cap \ManifoldN}\).
    Fix \(N\in \Zg\) with
    \(N\geq 2N_0+1+\max\{\Wdv_k\}(2+N_0)\).
    By repeated applications of Definition \ref{Defn::Spaces::LP::ElemWWdv}, we may write
    \begin{equation*}
        D_j=\sum_{|\alpha|,|\beta|\leq N} 2^{-j(2N-|\alpha|-|\beta|)} \left( 2^{-j\Wdv}W \right)^{\alpha} D_{j,\alpha,\beta} \left( 2^{-j\Wdv}W \right)^{\beta},
    \end{equation*}
    where 
    \(\left\{ \left( D_{j,\alpha,\beta},2^{-j} \right) : j\in \Zgeq, |\alpha|,|\beta|\leq N \right\}\in \ElemzF{\Omega_0\cap \ManifoldN}\).
    
    Let \(\Extend_\delta^0\) be as in Proposition \ref{Prop::Spaces::Elem::Extend::Extend::NonVerbose},
    with \(\Compact\) replaced by \(\overline{\Omega_0}\cap \ManifoldN\) and \(\Omega_1\) playing the role of 
    the set of the same name
    in that proposition.
    Set \(E_{j,\alpha,\beta}:=\Extend_\delta^0 D_{j,\alpha,\beta}\)
    so that \(\supp(E_{j,\alpha,\beta})\subseteq \ManifoldM\times \ManifoldN\),
    \(E_{j,\alpha,\beta}\big|_{\ManifoldN\times \ManifoldN}=D_{j,\alpha,\beta}\),
    and (using Proposition \ref{Prop::Spaces::Elem::Extend::Extend::NonVerbose} \ref{Item::Spaces::Elem::Extend::Extend::NonVerbose::ExtendPreElemzIsPreElem})
    \(\left\{ \left( E_{j,\alpha,\beta} ,2^{-j} \right) : j\in \Zgeq, |\alpha|,|\beta|\leq N \right\}\in \PElemFh{\overline{\Omega_1}}\).

    Set,
    \begin{equation*}
        \Extension[N_0]:=\sum_{|\alpha|,|\beta|\leq N}\sum_{j\in \Zgeq} 2^{-j(2N-|\alpha|-|\beta|)}\left( 2^{-j\Wdv}\Wh \right)^{\alpha} E_{j,\alpha,\beta} \left( 2^{-j\Wdv}W \right)^{\beta}.
    \end{equation*}
    By Proposition \ref{Prop::Spaces::Extend::TDefinesBoundedOperator},
    for \(s\in [-N_0,N_0]\), we have
    \(\Extension[N_0]\) defines a bounded operator \(\ASpace{s}{p}{q}[\Compact][\FilteredSheafF]\rightarrow \ASpace{s}{p}{q}[\overline{\Omega_1}][\FilteredSheafFh]\),
    and since the inclusion \(\ASpace{s}{p}{q}[\overline{\Omega_1}][\FilteredSheafFh]\hookrightarrow \ASpace{s}{p}{q}[\overline{\Omega}][\FilteredSheafFh]\)
    is continuous (see Proposition \ref{Prop::Spaces::Containment}), we see that
    \(\Extension[N_0]:\ASpace{s}{p}{q}[\Compact][\FilteredSheafF]\rightarrow \ASpace{s}{p}{q}[\overline{\Omega}][\FilteredSheafFh]\)
    is bounded.

    Finally, we have for \(f\in \ASpace{s}{p}{q}[\Compact][\FilteredSheafF]\), with \(s\in [-N_0,N_0]\),
    \begin{equation*}
    \begin{split}
         &\left[ \Extension[N_0] f \right]\big|_{\TestFunctionsZeroN}
         =\sum_{|\alpha|,|\beta|\leq N}\sum_{j\in \Zgeq} 2^{-j(2N-|\alpha|-|\beta|)}\left[ \left( 2^{-j\Wdv}\Wh \right)^{\alpha} E_{j,\alpha,\beta} \left( 2^{-j\Wdv}W \right)^{\beta}f \right]\Big|_{\TestFunctionsZeroN}
         \\&=\sum_{|\alpha|,|\beta|\leq N}\sum_{j\in \Zgeq} 2^{-j(2N-|\alpha|-|\beta|)}\left( 2^{-j\Wdv}W \right)^{\alpha} D_{j,\alpha,\beta} \left( 2^{-j\Wdv}W \right)^{\beta}f
         =\sum_{j\in \Zgeq} D_j f = \psi f = f,
    \end{split}
    \end{equation*}
    completing the proof.
\end{proof}

        \subsubsection{Simple consequences of the extension theorem}\label{Section::Spaces::RestrictionAndExtension::SimpleConsequences}
        In this subsection, we describe some simple consequences of
Theorem \ref{Thm::Spaces::Extension}.

\begin{corollary}\label{Cor::Spaces::Extend::Consequences::RestrictionSpaceDefn}
    Let \(\Compact\Subset \ManifoldNncF\) be compact and \(\Omega\Subset \ManifoldM\) be \(\ManifoldM\)-open and relatively
    compact with 
    \(\Compact\subseteq \Omega\). 
    % and
    % \(\overline{\Omega}\cap \ManifoldN\subseteq \ManifoldNncF\). 
    Then,
    \begin{equation}\label{Eqn::Spaces::Extend::Consequences::RestrictionSpaceDefn::Qual}
        \ASpace{s}{p}{q}[\Compact][\FilteredSheafF]
        =\left\{ g\big|_{\TestFunctionsZeroN} : g\in \ASpace{s}{p}{q}[\overline{\Omega}][\FilteredSheafFh] \right\}
    \end{equation}
    and
    \begin{equation}\label{Eqn::Spaces::Extend::Consequences::RestrictionSpaceDefn::Quant}
        \ANorm{f}{s}{p}{q}[\FilteredSheafF]
        \approx \inf \left\{ \ANorm{g}{s}{p}{q}[\FilteredSheafFh] : g\in \ASpace{s}{p}{q}[\overline{\Omega}][\FilteredSheafFh], g\big|_{\TestFunctionsZeroN}=f \right\},
        \quad \forall f\in \ASpace{s}{p}{q}[\Compact][\FilteredSheafF].
    \end{equation}
\end{corollary}
\begin{proof}
    Fix \(\Omega_1\Subset \Omega\) an \(\ManifoldM\)-open, relatively compact set with
    \(\Compact\Subset \Omega_1\) and \(\overline{\Omega_1}\cap \ManifoldN\subseteq \ManifoldNncF\).
    Fix \(\psi\in \CinftyCptSpace[\Omega_1]\) with \(\psi=1\) on an \(\ManifoldM\)-neighborhood of \(\Compact\).
    By Proposition \ref{Prop::Spaces::MappingOfFuncsAndVFs} \ref{Item::Spaces::MappingOfFuncsAndVFs::Funcs},
    \(\Mult{\psi}:\ASpace{s}{p}{q}[\overline{\Omega}][\FilteredSheafFh]\rightarrow \ASpace{s}{p}{q}[\overline{\Omega}][\FilteredSheafFh]\)
    is continuous, and therefore, by Proposition \ref{Prop::Spaces::Containment},
    \(\Mult{\psi}:\ASpace{s}{p}{q}[\overline{\Omega}][\FilteredSheafFh]\rightarrow \ASpace{s}{p}{q}[\overline{\Omega_1}][\FilteredSheafFh]\)
    is continuous.

    Let \(g\in \ASpace{s}{p}{q}[\overline{\Omega}][\FilteredSheafFh]\)
    with \(\supp(g\big|_{\TestFunctionsZeroN})\subseteq \Compact\). Then,
    since \(\psi g\in \ASpace{s}{p}{q}[\overline{\Omega_1}][\FilteredSheafFh]\),
    Theorem \ref{Thm::Spaces::Extension} \ref{Item::Spaces::Extension::Restriction} gives
    \(g\big|_{\TestFunctionsZeroN}= \left( \psi g \right)\big|_{\TestFunctionsZeroN} \in \ASpace{s}{p}{q}[\overline{\Omega_1}\cap \ManifoldN][\FilteredSheafF]\),
    with 
    \begin{equation*}
        \ANorm{g\big|_{\TestFunctionsZeroN}}{s}{p}{q}[\FilteredSheafF]
        \lesssim \ANorm{\psi g}{s}{p}{q}[\FilteredSheafFh]
        \lesssim \ANorm{g}{s}{p}{q}[\FilteredSheafFh].
    \end{equation*}
    Since \(\supp(g\big|_{\TestFunctionsZeroN})\subseteq \Compact\), that \(g\big|_{\TestFunctionsZeroN}\in \ASpace{s}{p}{q}[\overline{\Omega_1}\cap \ManifoldN][\FilteredSheafF]\)
    implies \(\ASpace{s}{p}{q}[\Compact][\FilteredSheafF]\) (see Proposition \ref{Prop::Spaces::Containment}).
    This establishes the \(\supseteq\) part of \eqref{Eqn::Spaces::Extend::Consequences::RestrictionSpaceDefn::Qual}
    and the \(\lesssim\) part of \eqref{Eqn::Spaces::Extend::Consequences::RestrictionSpaceDefn::Quant}.

    Let \(f\in \ASpace{s}{p}{q}[\Compact][\FilteredSheafF]\).
    Take \(N_0=|s|\), and set \(g=\Extension[N_0] f\), where \(\Extension[N_0]\) is as in 
    Theorem \ref{Thm::Spaces::Extension} \ref{Item::Spaces::Extension::Extension}.
    Then, \(g\in \ASpace{s}{p}{q}[\overline{\Omega}][\FilteredSheafFh]\) with
    \(g\big|_{\TestFunctionsZeroN}=f\) and \(\ANorm{g}{s}{p}{q}[\FilteredSheafFh]\lesssim \ANorm{f}{s}{p}{q}[\FilteredSheafF]\).
    This establishes the \(\subseteq\) part of \eqref{Eqn::Spaces::Extend::Consequences::RestrictionSpaceDefn::Qual}
    and the \(\gtrsim\) part of \eqref{Eqn::Spaces::Extend::Consequences::RestrictionSpaceDefn::Quant},
    competing the proof.
\end{proof}

\begin{proof}[Proof of Proposition \ref{Prop::Spaces::EqualsLp}]
    In Proposition \ref{Prop::Spaces::EqualsLp}, there is no given ambient manifold \(\ManifoldM\)
    or ambient H\"ormander filtration of sheaves of vector fields \(\FilteredSheafFh\).
    However, such data always exists: see Proposition \ref{Prop::Filtrations::RestrictingFiltrations::CoDim0CoRestriction}.

    In light of Corollary \ref{Cor::Spaces::Extend::Consequences::RestrictionSpaceDefn}, it suffices
    to prove the result on \(\ManifoldM\) with \(\FilteredSheafF\) replaced by
    \(\FilteredSheafFh\). This is exactly \cite[Proposition 6.2.13]{StreetMaximalSubellipticity}.
\end{proof}

\begin{proposition}\label{Prop::Spaces::Extend::Consequences::DerivForSobolevNormsOnFunctionsOfCptSupport}
    Let \(\Omega_1\Subset\Omega_2\Subset \ManifoldM\) be relatively compact, \(\ManifoldM\)-open sets with \(\Omega_2\cap\ManifoldN\Subset \ManifoldNncF\).
    Let \(\WhWdv=\left\{ \left( \Wh_1,\Wdv_1 \right),\ldots, \left( \Wh_r,\Wdv_r \right) \right\}\subset \VectorFields{\Omega_2}\times \Zg\)
    be H\"ormander vector fields with formal degrees such that \(\FilteredSheafFh\big|_{\Omega_2}=\FilteredSheafGenBy{\WhWdv}\)
    (see Lemma \ref{Lemma::Filtrations::GeneratorsOnRelCptSet}). Set \(W_j:=\Wh_j\big|_{\Omega_2\cap\ManifoldN}\)
    and set \(\WWdv:=\left\{ \left( W_1,\Wdv_1 \right),\ldots, \left( W_r,\Wdv_r \right) \right\}\)
    so that \(\FilteredSheafF\big|_{\Omega_2\cap \ManifoldN}=\FilteredSheafGenBy{\WWdv}\) (see Proposition \ref{Prop::Filtrations::RestrictingFiltrations::CoDim0Restriction}).
    Fix \(\kappa\in \Zg\) such that \(\Wdv_j\) divides \(\kappa\), for every \(j\). Then, for \(1<p<\infty\),
    \begin{equation*}
        \TLNorm{f}{\kappa}{p}{2}[\FilteredSheafF]\approx \TLNorm{f}{\kappa}{p}{2}[\FilteredSheafFh]
        \approx \sum_{\DegWdv{\alpha}\leq \kappa} \LpNorm{W^{\alpha}f}{p},\quad \forall f\in \CinftyCptSpace[\Omega_1\cap\InteriorN].
        % \ANorm{f}{s}{p}{q}[\FilteredSheafFh]\approx \ANorm{f}{s}{p}{q}[\FilteredSheafF]
    \end{equation*}
\end{proposition}
\begin{proof}
    Since \(W^{\alpha}:\TLSpace{\kappa}{p}{2}[\overline{\Omega_1}][\FilteredSheafF]\rightarrow \TLSpace{\kappa-\DegWdv{\alpha}}{p}{2}[\overline{\Omega_1}][\FilteredSheafF]\hookrightarrow \TLSpace{0}{p}{2}[\overline{\Omega_1}][\FilteredSheafF]\)
    is continuous for \(\DegWdv{\alpha}\leq \kappa\) (see Proposition \ref{Prop::Spaces::MappingOfFuncsAndVFs} \ref{Item::Spaces::MappingOfFuncsAndVFs::VFsInF}),
    we have
    \begin{equation*}
        \sum_{\DegWdv{\alpha}\leq \kappa} \TLNorm{W^{\alpha}f}{0}{p}{2}[\FilteredSheafF]
        \lesssim \TLNorm{f}{\kappa}{p}{2}[\FilteredSheafF], \quad \forall f\in \CinftyCptSpace[\Omega_1\cap \ManifoldN].
    \end{equation*}
    Combining this with Proposition \ref{Prop::Spaces::EqualsLp} shows 
    \begin{equation}\label{Spaces::Extend::Consequences::DerivForSobolevNormsOnFunctionsOfCptSupport::Tmp1}
        \sum_{\DegWdv{\alpha}\leq \kappa} \LpNorm{W^{\alpha}f}{p}
        \lesssim \TLNorm{f}{\kappa}{p}{2}[\FilteredSheafF], \quad \forall f\in \CinftyCptSpace[\Omega_1\cap \ManifoldN].
    \end{equation}
    Corollary \ref{Cor::Spaces::Extend::Consequences::RestrictionSpaceDefn} shows
    \begin{equation*}
        \TLNorm{f\big|_{\ManifoldN}}{\kappa}{p}{2}[\FilteredSheafF]\lesssim \TLNorm{f}{\kappa}{p}{2}[\FilteredSheafFh],\quad \forall f\in \CinftyCptSpace[\Omega_1],
    \end{equation*}
    and therefore
    \begin{equation}\label{Spaces::Extend::Consequences::DerivForSobolevNormsOnFunctionsOfCptSupport::Tmp2}
        \TLNorm{f}{\kappa}{p}{2}[\FilteredSheafF]\lesssim \TLNorm{f}{\kappa}{p}{2}[\FilteredSheafFh],\quad \forall f\in \CinftyCptSpace[\Omega_1\cap \InteriorN].
    \end{equation}
    By \cite[Corollary 6.2.14]{StreetMaximalSubellipticity}, we have
    \begin{equation*}
        \TLNorm{f}{\kappa}{p}{2}[\FilteredSheafFh]
        \approx \sum_{\DegWdv{\alpha}\leq \kappa} \LpNorm{\Wh^{\alpha}f}{p},\quad \forall f\in \CinftyCptSpace[\Omega_1],
    \end{equation*}
    and therefore,
    \begin{equation}\label{Spaces::Extend::Consequences::DerivForSobolevNormsOnFunctionsOfCptSupport::Tmp3}
        \TLNorm{f}{\kappa}{p}{2}[\FilteredSheafFh]
        \approx \sum_{\DegWdv{\alpha}\leq \kappa} \LpNorm{W^{\alpha}f}{p},\quad \forall f\in \CinftyCptSpace[\Omega_1\cap\InteriorN].
    \end{equation}
    Combining \eqref{Spaces::Extend::Consequences::DerivForSobolevNormsOnFunctionsOfCptSupport::Tmp1}, 
    \eqref{Spaces::Extend::Consequences::DerivForSobolevNormsOnFunctionsOfCptSupport::Tmp2},
    and \eqref{Spaces::Extend::Consequences::DerivForSobolevNormsOnFunctionsOfCptSupport::Tmp3}
    completes the proof.
\end{proof}

\begin{remark}\label{Rmk::Spaces::Extend::Consequences::DerivForSobolevNormsOnFunctionsOfCptSupport::NotAllFunctions}
    Proposition \ref{Prop::Spaces::Extend::Consequences::DerivForSobolevNormsOnFunctionsOfCptSupport}
    addresses only functions \(f\in \CinftyCptSpace[\Omega_1\cap\InteriorN]\). While we expect
    the result to be true \(\forall f\in \TLSpace{\kappa}{p}{2}[\overline{\Omega_1}][\FilteredSheafF]\),
    smooth functions with compact support in \(\InteriorN\) are not dense (see Theorem \ref{Thm::Trace::Vanish::MainVanishThm}).
    Because of this, new ideas will be needed, and we do not pursue this here.
    For a special case where all functions can be addressed, see Proposition \ref{Prop::Spaces::EqualsSobolev}.
\end{remark}

\begin{proof}[Proof of Proposition \ref{Prop::Spaces::EqualsClassical}]
    % In Proposition \ref{Prop::Spaces::EqualsClassical}, there is no given ambient manifold \(\ManifoldM\)
    % or ambient H\"ormander filtration of sheaves of vector fields \(\FilteredSheafFh\).
    Let \(\ManifoldM\) be the double of \(\ManifoldN\), and set \(\FilteredSheafFh[d][\Omega]=\VectorFields{\Omega}\)
    for every \(d\in \Zg\) and \(\Omega\subseteq \ManifoldM\) open.
    Note that \(\RestrictFilteredSheaf{\FilteredSheafFh}{\ManifoldN}=\FilteredSheafF\).
    % However, such data always exists: see Proposition \ref{Prop::Filtrations::RestrictingFiltrations::CoDim0CoRestriction}.
    
    The classical Besov and Triebel--Lizorkin spaces are traditionally defined as the restrictions
    of the spaces on an ambient manifold (as in Corollary \ref{Cor::Spaces::Extend::Consequences::RestrictionSpaceDefn}).
    Thus, it suffices to show that for \(\Compact\Subset \ManifoldM\), compact,
    we have 
    \(\ASpace{s}{p}{q}[\Compact][\FilteredSheafFh]\) coincides with the space of those distributions
    supported in \(\Compact\) which lie in the classical Besov or Triebel--Lizorkin space on \(\ManifoldM\)
    (with equivalence of norms).
    This is \cite[Theorem 6.6.7]{StreetMaximalSubellipticity}.
\end{proof}

\section{Trace theorems}\label{Chapter::Trace}
Throughout this chapter,  \(\ManifoldN\) is a smooth manifold, with boundary, and and \(\FilteredSheafF\)
is a H\"ormander filtration of sheaves of vector fields.

In this chapter, we state and prove the main trace theorems regarding
the spaces
\(\BesovSpace{s}{p}{q}[\Compact][\FilteredSheafF]\) and \(\TLSpace{s}{p}{q}[\Compact][\FilteredSheafF]\);
namely, we generalize Theorem \ref{Thm::Intro::Classical::TraceThm} and Corollary \ref{Cor::GlobalCors::TraceThm}
to this more general setting; see Theorems  \ref{Thm::Trace::ForwardMap} and \ref{Thm::Trace::Dirichlet::MainInverseThm}.  We also study subspaces
of functions which vanish at the boundary; see Section \ref{Section::Trace::Vanish}.

Trace theorems (like Theorem \ref{Thm::Intro::Classical::TraceThm} and Corollary \ref{Cor::GlobalCors::TraceThm})
have two parts: the trace map (the forward map), and its right inverse (the inverse map).
We state these as separate theorems.
Fix \(\lambda\in \Zg\) such that
\(\BoundaryNncF[\lambda]\ne \emptyset\), and \(\Omega\Subset \ManifoldNncF[\lambda]\) open and relatively compact.
Recall the spaces \(\ASpaceCpt{s}{p}{q}[\Omega][\FilteredSheafF]\) from Definition \ref{Defn::Spaces::Defns::ASpaceCpt}.
The first main theorem of this chapter is:

\begin{theorem}[The forward map]
    \label{Thm::Trace::ForwardMap}
    Let \(B\) be a partial differential operator with smooth coefficients defined on a neighborhood of \(\OmegaClosure\),
    and suppose \(B\) has \(\FilteredSheafF\)-degree \(\leq \kappa\) on \(\OmegaClosure\)
    (see Definition \ref{Defn::Filtrations::DiffOps::Deg}).
    Then, there exists a unique linear map:
    \begin{equation*}
        \TraceMap[B]:\left( \bigcup_{\substack{ 1< p\leq \infty\\1\leq q\leq \infty \\ s>\kappa+\lambda/p }} \BesovSpaceCpt{s}{p}{q}[\Omega][\FilteredSheafF]\right)
        \bigcup
        \left( \bigcup_{\substack{ 1< p< \infty\\1< q\leq \infty \\ s>\kappa+\lambda/p }} \TLSpaceCpt{s}{p}{q}[\Omega][\FilteredSheafF]\right)
        \rightarrow \Distributions[\Omega\cap \BoundaryN]
    \end{equation*}
    such that:
    \begin{enumerate}[(I)]
        \item\label{Item::Trace::ForwardMap::TraceOnSmooth} \(\forall f\in \CinftyCptSpace[\Omega]\), \(\TraceMap[B]f=Bf\big|_{\BoundaryN}\),
        \item\label{Item::Trace::ForwardMap::Continuous} \(\forall \Compact\Subset \Omega\) compact, the following are continuous:
            \begin{equation*}
                \TraceMap[B]:\BesovSpace{s}{p}{q}[\Compact][\FilteredSheafF] \rightarrow \BigBesovSpace{s-\kappa-\lambda/p}{p}{q}[\Compact\cap \BoundaryN][\RestrictFilteredSheaf{\LieFilteredSheafF}{\BoundaryNncF}],\quad 1<p\leq \infty, \: 1\leq q\leq\infty,\: s>\kappa+\lambda/p,
            \end{equation*}
            \begin{equation*}
                \TraceMap[B]:\TLSpace{s}{p}{q}[\Compact][\FilteredSheafF] \rightarrow \BigBesovSpace{s-\kappa-\lambda/p}{p}{p}[\Compact\cap \BoundaryN][\RestrictFilteredSheaf{\LieFilteredSheafF}{\BoundaryNncF}],\quad 1<p< \infty, \: 1< q\leq\infty,\: s>\kappa+\lambda/p.
            \end{equation*}
    \end{enumerate}
\end{theorem}

After two reductions (in Sections \ref{Section::Trace::ReductionToLocal} and \ref{Section::Trace::ReductionToCoord}),
the proof of Theorem \ref{Thm::Trace::ForwardMap} is completed in Section \ref{Section::Trace::Forward}.

We turn to the inverse map. Since \(\ManifoldN\) does not have a Riemannian metric, it does not make sense
to talk about normal derivatives at the boundary. Instead, we proceed more generally, and generalize the 
notion of a \textit{Dirichlet system} to the maximally subelliptic setting;
see \cite[Part I, Section 4.1]{AgranovichEgorovShubinPartialDifferentialEquationsIX} for the definitions in
the classical setting. We take this approach in Section \ref{Section::Trace::Dirichlet}, and present an indicative special case here,
which is a corollary of Theorem \ref{Thm::Trace::Dirichlet::MainInverseThm}, below.

Let \(\Omega\Subset \ManifoldNncF[\lambda]\) be a relatively compact, open set, and let 
\(V\in \FilteredSheafF[\Omega][\lambda]\) with \(V(x)\not \in \TangentSpace{x}{\BoundaryN}\),
\(\forall x\in \Omega\cap \BoundaryN\) (\(V\) will play the role of a ``normal derivative''). Such a \(V\)
always exists by the definition of \(\BoundaryNncF[\lambda]\), and a simple partition of unity argument.
For each \(L\in \Zgeq\), we consider the map, for \(f\in \CinftyCptSpace[\Omega]\), and \(x\in \BoundaryN\),
\begin{equation*}
    \TraceMap[L][V] f(x) = \left( f\big|_{\BoundaryN}(x), V f\big|_{\BoundaryN}(x),\ldots, V^L f\big|_{\BoundaryN}(x)\right).
\end{equation*}
By Theorem \ref{Thm::Trace::ForwardMap}, we may extend \(\TraceMap[L][V]\) uniquely to a map
such that \(\forall \Compact \Subset \Omega\) compact,
\begin{equation*}
    \TraceMap[L][V]:\BesovSpace{s}{p}{q}[\Compact][\FilteredSheafF]
    \rightarrow \prod_{l=0}^L \BigBesovSpace{s-l\lambda-\lambda/p}{p}{q}[\Compact\cap\BoundaryN][\RestrictFilteredSheaf{\LieFilteredSheafF}{\BoundaryNncF}],\quad 1<p\leq \infty, \: 1\leq q\leq \infty,\: s>L\lambda+\lambda/p,
\end{equation*}
\begin{equation*}
    \TraceMap[L][V]:\TLSpace{s}{p}{q}[\Compact][\FilteredSheafF]
    \rightarrow \prod_{l=0}^L \BigBesovSpace{s-l\lambda-\lambda/p}{p}{p}[\Compact\cap\BoundaryN][\RestrictFilteredSheaf{\LieFilteredSheafF}{\BoundaryNncF}],\quad 1<p< \infty, \: 1< q\leq \infty,\: s>L\lambda+\lambda/p.
\end{equation*}
Our next result says that these maps are right invertible; we state this as a corollary of Theorem \ref{Thm::Trace::Dirichlet::MainInverseThm}, below.
% We write \(\DistributionsCpt[\Omega]\) for the dual of \(\CinftySpace[\Omega]\)--i.e., the space of distributions with
% compact support in \(\Omega\).

\begin{corollary}[The inverse map]
    \label{Cor::Trace::InverseMap}
    Fix \(\Omega_1\Subset \Omega_2\Subset \Omega\), relatively compact, open sets.
    There is a continuous, linear map
    \begin{equation*}
        \TraceInverseMap[L][V]:\Distributions[\Omega_2\cap \BoundaryN]^{L+1}\rightarrow \DistributionsZeroN
    \end{equation*}
    such that
    \begin{enumerate}[(i)]
        \item \(\TraceInverseMap[L][V]:\CinftySpace[\Omega_2\cap \BoundaryN]\rightarrow \CinftyCptSpace[\Omega]\), continuously.
        \item \(\TraceInverseMap[L][V]\) is continuous:
            \begin{equation*}
                \TraceInverseMap[L][V]:
                \prod_{l=0}^L
                \BigBesovSpace{s-l\lambda-\lambda/p}{p}{q}[\overline{\Omega_1}\cap \BoundaryN][\RestrictFilteredSheaf{\LieFilteredSheafF}{\BoundaryNncF}]
                \rightarrow
                \BesovSpace{s}{p}{q}[\overline{\Omega_2}][\FilteredSheafF],\quad 1\leq p\leq \infty, \: 1\leq q\leq\infty, \: s\in \R,
            \end{equation*}
            \begin{equation*}
                \TraceInverseMap[L][V]:
                \prod_{l=0}^L
                \BigBesovSpace{s-l\lambda-\lambda/p}{p}{p}[\overline{\Omega_1}\cap \BoundaryN][\RestrictFilteredSheaf{\LieFilteredSheafF}{\BoundaryNncF}]
                \rightarrow
                \TLSpace{s}{p}{q}[\overline{\Omega_2}][\FilteredSheafF],\quad 1< p< \infty, \: 1< q\leq\infty, \: s\in \R.
            \end{equation*}
        \item \(\TraceMap[L][V]\TraceInverseMap[L][V]=I\) when \(s>L\lambda+\lambda/p\) and \(p>1\); i.e., whenever the range
            space of \(\TraceInverseMap[L][V]\) is a domain space of \(\TraceMap[L][V]\).
    \end{enumerate}
\end{corollary}
\begin{proof}
    This is an immediate corollary of Theorem \ref{Thm::Trace::Dirichlet::MainInverseThm}, below; see Example \ref{Example::Dirichlet::VpowersIsDirichletSystem}.
\end{proof}

Our final main result of this chapter locally characterizes the closure of \(\CinftyCptSpace[\InteriorN]\) in
some of the above spaces, in terms of vanishing traces on the boundary. We present this in Section \ref{Section::Trace::Vanish}.

    \subsection{Dirichlet systems and the inverse map}\label{Section::Trace::Dirichlet}
    Let \(\Omega\Subset \ManifoldNncF[\lambda]\) be a relatively compact, open set, and let 
\(V\in \LieFilteredSheafF[\Omega][\lambda]\) with \(V(x)\not \in \TangentSpace{x}{\BoundaryN}\),
\(\forall x\in \Omega\cap \BoundaryN\). As described above, we treat \(V\) as the
``normal derivative;'' though none of the definitions or results which follow depend on the
choice of \(V\). One may always choose \(V\in \FilteredSheafF[\Omega][\lambda]\).

\begin{remark}\label{Rmk::Trace::Dirichlet::SubtractOffNonTangentPart}
    Let \(Y\in \LieFilteredSheafF[\Omega][d]\).
    \begin{itemize}
        \item If \(d<\lambda\), then \(Y(x)\in \TangentSpace{x}{\BoundaryN}\), \(\forall x\in \Omega\cap\BoundaryN\),
            by the definition of \(\BoundaryNncF[\lambda]\).
        \item If \(d\geq \lambda\), let \(a(x)\in \CinftySpace[\Omega\cap \BoundaryN]\) be the unique function
            such that \(Y(x)-a(x)V(x)\in \TangentSpace{x}{\BoundaryN}\), \(\forall x\in \Omega\cap\BoundaryN\),
            and let \(\at\in \CinftySpace[\Omega]\) be any extension of \(a\). Then,
            \(Y-\at V\in \LieFilteredSheafF[\Omega][d]\), and is tangent to \(\BoundaryN\) on \(\BoundaryN\cap \Omega\).
            Thus, modulo \(V\), \(Y\) is tangent to \(\BoundaryN\) on \(\Omega\cap\BoundaryN\).
    \end{itemize}
\end{remark}

\begin{definition}\label{Defn::Trace::Dirichlet::NormalOperator}
    Let \(B\) be a partial differential operator with smooth coefficients defined on \(\Omega\).
    Fix \(\kappa\in \Zgeq\) such that \(\lambda\) divides \(\kappa\).
    We say \(B\) is an \(\FilteredSheafF\)-normal 
    partial differential operator of degree 
    \(\kappa\) on \(\Omega\)
    if \(B\) can be written as
    \begin{equation*}
        B=a_0(x) V^{\kappa/\lambda} + \sum_{j=1}^L P_j
    \end{equation*}
    where,
    \(a_0\in \CinftySpace[\Omega]\), \(a_0(x)\ne 0\), \(\forall x\in \Omega\cap \BoundaryN\), and
    for each \(j\), \(P_j=b_j(x)X_{j,1}\cdots X_{j,L_j}\) where:
    \begin{itemize}
        \item \(b_j\in \CinftySpace[\Omega]\),
        \item For each \(j\),
        \(X_{j,k}\in \LieFilteredSheafF[\Omega][d_{j,k}]\), with \(d_{j,1}+\cdots+d_{j,L_j}\leq \kappa\),
        \item For each \(j,k\), either \(X_{j,k}(x)\in \TangentSpace{x}{\BoundaryN}\), \(\forall x\in \Omega\cap\BoundaryN\), or \(X_{j,k}=V\),
        \item For each \(j\), the number of \(k\) such that \(X_{j,k}=V\) is \(<\kappa/\lambda\).
    \end{itemize}
\end{definition}

\begin{remark}
    Remark \ref{Rmk::Trace::Dirichlet::SubtractOffNonTangentPart} can be used to easily show that
    Definition \ref{Defn::Trace::Dirichlet::NormalOperator} does not depend on the choice of \(V\in \LieFilteredSheafF[\Omega][\lambda]\) with \(V(x)\not \in \TangentSpace{x}{\BoundaryN}\),
\(\forall x\in \Omega\cap \BoundaryN\).
\end{remark}

\begin{remark}\label{Rmk::Trace::Dirichlet::NormalOperatorHasDeg}
    If \(B\) an \(\FilteredSheafF\)-normal 
    partial differential operator of degree 
    \(\kappa\) then \(B\) has \(\FilteredSheafF\)-degree \(\leq \kappa\)
    in the sense of Definition \ref{Defn::Filtrations::DiffOps::Deg}; see also Remark \ref{Rmk::Filtrations::DiffOps::DoesntDependOnChoices}.
\end{remark}

\begin{remark}\label{Rmk::Trace::Dirichlet::NormalOperatorWhenLambda1}
    Definition \ref{Defn::Trace::Dirichlet::NormalOperator}
    is easier to understand when \(\lambda=1\).
    If \(\lambda=1\), then \(B\) is an 
    \(\FilteredSheafF\)-normal 
    partial differential operator of degree 
    \(\kappa\) on \(\Omega\)
    if and only if
    \begin{itemize}
        \item \(B\)  has \(\FilteredSheafF\)-degree \(\leq \kappa\)
    in the sense of Definition \ref{Defn::Filtrations::DiffOps::Deg}
     (see Remark \ref{Rmk::Trace::Dirichlet::NormalOperatorHasDeg}).

     \item \(\BoundaryN\cap \Omega\) is non-characteristic for \(B\) (in the classical sense).
    \end{itemize} 
\end{remark}

\begin{definition}
    Let \(\sB=\left( B_0, B_1,\ldots, B_L \right)\) be such that \(B_j\)
    is an \(\FilteredSheafF\)-normal partial differential operator  of degree \(j\lambda\) on \(\Omega\).
    We say \(\sB\) is a \(\FilteredSheafF\)-Dirichlet system on \(\Omega\).
\end{definition}

\begin{example}\label{Example::Dirichlet::VpowersIsDirichletSystem}
    \(\sB=(1, V, V^2,\ldots, V^L)\) is a Dirichlet system on \(\Omega\).
\end{example}

Let \(\sB=\left( B_0,B_1,\ldots, B_L \right)\) be an \(\FilteredSheafF\)-Dirichlet system on \(\Omega\).
For \(f\in \CinftySpace[\Omega]\), set
\begin{equation*}
    \TraceMap[\sB]f(x)=
    \left( 
        B_0 f\big|_{\Omega\cap\BoundaryN}(x), 
        B_1 f\big|_{\Omega\cap\BoundaryN}(x),\ldots, B_L f\big|_{\Omega\cap \BoundaryN}(x) \right), \quad x\in \Omega\cap\BoundaryN.
\end{equation*}
By Theorem \ref{Thm::Trace::ForwardMap}, we may extend \(\TraceMap[\sB]\) uniquely to a map
such that \(\forall \Compact \Subset \Omega\) compact,
\begin{equation*}
    \TraceMap[\sB]:\BesovSpace{s}{p}{q}[\Compact][\FilteredSheafF]
    \rightarrow \prod_{l=0}^L \BigBesovSpace{s-l\lambda-\lambda/p}{p}{q}[\Compact\cap\BoundaryN][\RestrictFilteredSheaf{\LieFilteredSheafF}{\BoundaryNncF}],\quad 1<p\leq \infty, \: 1\leq q\leq \infty,\: s>L\lambda+\lambda/p,
\end{equation*}
\begin{equation*}
    \TraceMap[\sB]:\TLSpace{s}{p}{q}[\Compact][\FilteredSheafF]
    \rightarrow \prod_{l=0}^L \BigBesovSpace{s-l\lambda-\lambda/p}{p}{p}[\Compact\cap\BoundaryN][\RestrictFilteredSheaf{\LieFilteredSheafF}{\BoundaryNncF}],\quad 1<p< \infty, \: 1< q\leq \infty,\: s>L\lambda+\lambda/p.
\end{equation*}
Our main inverse result says that these maps are right invertible.

\begin{theorem}[The inverse map]
    \label{Thm::Trace::Dirichlet::MainInverseThm}
    Fix \(\Omega_1\Subset \Omega_2\Subset \Omega\), relatively compact, open sets.
    There is a continuous, linear map
    \begin{equation*}
        \TraceInverseMap[\sB]:\Distributions[\Omega_2\cap \BoundaryN]^{L+1}\rightarrow \DistributionsZeroN
    \end{equation*}
    such that
    \begin{enumerate}[(i)]
        \item\label{Item::Trace::Dirichlet::MainInverseThm::ContinuousOnSmoothFuncs} 
        \(\TraceInverseMap[\sB]:\CinftySpace[\Omega_2\cap \BoundaryN]^{L+1}\rightarrow \CinftyCptSpace[\Omega]\), continuously.
        \item\label{Item::Trace::Dirichlet::MainInverseThm::ContinuousOnSpaces} \(\TraceInverseMap[\sB]\) is continuous:
            \begin{equation*}
                \TraceInverseMap[\sB]:
                \prod_{l=0}^L
                \BigBesovSpace{s-l\lambda-\lambda/p}{p}{q}[\overline{\Omega_1}\cap \BoundaryN][\RestrictFilteredSheaf{\LieFilteredSheafF}{\BoundaryNncF}]
                \rightarrow
                \BesovSpace{s}{p}{q}[\overline{\Omega_2}][\FilteredSheafF],\quad 1\leq p\leq \infty, \: 1\leq q\leq\infty, \: s\in \R,
            \end{equation*}
            \begin{equation*}
                \TraceInverseMap[\sB]:
                \prod_{l=0}^L
                \BigBesovSpace{s-l\lambda-\lambda/p}{p}{p}[\overline{\Omega_1}\cap \BoundaryN][\RestrictFilteredSheaf{\LieFilteredSheafF}{\BoundaryNncF}]
                \rightarrow
                \TLSpace{s}{p}{q}[\overline{\Omega_2}][\FilteredSheafF],\quad 1< p< \infty, \: 1< q\leq\infty, \: s\in \R.
            \end{equation*}
        \item\label{Item::Trace::Dirichlet::MainInverseThm::IsInverse} \(\TraceMap[\sB]\TraceInverseMap[\sB]=I\) when \(s>L\lambda+\lambda/p\) and \(p>1\)
        in \ref{Item::Trace::Dirichlet::MainInverseThm::ContinuousOnSpaces}; i.e., whenever the range
            space of \(\TraceInverseMap[\sB]\) is a domain space of \(\TraceMap[\sB]\).
    \end{enumerate}
\end{theorem}

After two reductions (in Sections \ref{Section::Trace::ReductionToLocal} and \ref{Section::Trace::ReductionToCoord}),
the proof of Theorem \ref{Thm::Trace::Dirichlet::MainInverseThm} is completed in Section \ref{Section::Trace::Inverse}.

    \subsection{Functions which vanish at the boundary}\label{Section::Trace::Vanish}
    In this section, we present results which show that if a function has enough derivatives
vanish at the boundary, \(\BoundaryN\), it can be approximated by elements of \(\CinftyCptSpace[\InteriorN]\).
The first result, which is the main theorem of this section, is a local result which does not require \(\ManifoldN\)
to be compact.

\begin{theorem}\label{Thm::Trace::Vanish::MainVanishThm}
    Let \(\Omega\Subset \ManifoldNncF[\lambda]\)
     be a relatively compact, open set and \(\Compact\Subset \Omega\) compact.
    Fix \(p,q\in (1,\infty)\), \(s\in \R\), and \(u\in \ASpace{s}{p}{q}[\Compact][\FilteredSheafF]\).
    Then,
    \begin{enumerate}[(A)]
        \item\label{Item::Trace::Vanish::MainVanishThm::Smalls} If \(s<\lambda/p\), 
            \(\exists \{f_j\}_{j\in \Zgeq}\subset \CinftyCptSpace[\Omega \cap \InteriorN]\) with \(f_j\rightarrow u\)
            in \(\ASpace{s}{p}{q}[\OmegaClosure][\FilteredSheafF]\).

        \item\label{Item::Trace::Vanish::MainVanishThm::Larges} If \(s\in (L\lambda+\lambda/p, (L+1)\lambda+\lambda/p)\), where \(L\in \Zgeq\), let \(\sB=\left( B_0,B_1,\ldots, B_L \right)\)
            be an \(\FilteredSheafF\)-Dirichlet system on \(\Omega\). The following are equivalent:
            \begin{enumerate}[(i)]
                \item\label{Item::Trace::Vanish::MainVanishThm::Larges::Trace0} \(\TraceMap[\sB]u=0\).
                \item\label{Item::Trace::Vanish::MainVanishThm::Larges::SmoothApprox} \(\exists \{f_j\}_{j\in \Zgeq}\subset \CinftyCptSpace[\Omega\cap \InteriorN]\) with 
                    \(f_j\rightarrow u\) in \(\ASpace{s}{p}{q}[\OmegaClosure][\FilteredSheafF]\).
            \end{enumerate}
    \end{enumerate}
\end{theorem}

After two reductions (in Sections \ref{Section::Trace::ReductionToLocal} and \ref{Section::Trace::ReductionToCoord}),
the proof of Theorem \ref{Thm::Trace::Vanish::MainVanishThm} is completed in Section \ref{Section::Trace::CharacterizeVanish}.

Suppose \(\ManifoldN\) is compact and \(\BoundaryN=\BoundaryNncF\).
Then, \(\ManifoldN\Subset \ManifoldNncF\) and it makes sense to talk about the spaces
\(\ASpace{s}{p}{q}[\ManifoldN][\FilteredSheafF]\).
In this case, we define \(\ACircSpace{s}{p}{q}[\ManifoldN][\FilteredSheafF]\) to be the
closure of \(\CinftyCptSpace[\InteriorN]\) in \(\ASpace{s}{p}{q}[\ManifoldN][\FilteredSheafF]\).

\begin{corollary}\label{Cor::Trace::Vanish::EveryPointNonChar}
    Suppose \(\ManifoldN\) is compact and \(\BoundaryN=\BoundaryNncF[\lambda]\) for some \(\lambda\).
    Fix \(p,q\in (1,\infty)\), \(s\in \R\).
    %, and \(u\in \ASpace{s}{p}{q}[\ManifoldN][\FilteredSheafF]\).
    Then,
    \begin{enumerate}[(A)]
        \item If \(s<\lambda/p\), \(\ASpace{s}{p}{q}[\ManifoldN][\FilteredSheafF]=\ACircSpace{s}{p}{q}[\ManifoldN][\FilteredSheafF]\).
        %\(u\in \ACircSpace{s}{p}{q}[\ManifoldN][\FilteredSheafF]\).
        \item If \(s\in (L\lambda+\lambda/p, (L+1)\lambda+\lambda/p)\), where \(L\in \Zgeq\), let \(\sB=\left( B_0,B_1,\ldots, B_L \right)\)
            be an \(\FilteredSheafF\)-Dirichlet system on \(\ManifoldN\).  Then,
            \begin{equation*}
                \ACircSpace{s}{p}{q}[\ManifoldN][\FilteredSheafFh]
                =\left\{ f\in \ASpace{s}{p}{q}[\ManifoldN][\FilteredSheafF] : \TraceMap[\sB]f=0 \right\}.
            \end{equation*}
            % The following are equivalent:
            % \begin{enumerate}[(i)]
            %     \item \(\TraceMap[\sB]u=0\).
            %     \item \(u\in \ACircSpace{s}{p}{q}[\ManifoldN][\FilteredSheafF]\).
            % \end{enumerate}
    \end{enumerate}
\end{corollary}
\begin{proof}
    This is the special case of Theorem \ref{Thm::Trace::Vanish::MainVanishThm}
    with \(\Omega=\Compact=\ManifoldN\).
\end{proof}

        \subsubsection{Dirichlet problem}\label{Section::Trace::Vanish::DirichletProblem}
        Theorem \ref{Thm::Trace::Vanish::MainVanishThm} can be used to help understand maximally subelliptic
Dirichlet problems, and this subsection is devoted to explaining this connection.

Let \(\ManifoldM\) be a smooth manifold without boundary, \(\Vol\) a smooth, strictly positive density on 
\(\ManifoldM\), \(\ManifoldN\subseteq \ManifoldM\) a closed, embedded, co-dimension \(0\) submanifold (with boundary) of \(\ManifoldM\).
Let \(\FilteredSheafFh\) be a H\"ormander filtration of sheaves of vector fields on \(\ManifoldM\)
and set \(\FilteredSheafF:=\RestrictFilteredSheaf{\FilteredSheafFh}{\ManifoldN}\), so that
\(\FilteredSheafF\) is a H\"ormander filtration of sheaves of vector fields on \(\ManifoldN\)
by Proposition \ref{Prop::Filtrations::RestrictingFiltrations::CoDim0Restriction}.

For simplicity, we assume there are globally defined H\"ormander vector fields
with formal degrees \(\WhWdv=\left\{ \left( \Wh_1,\Wdv_1 \right),\ldots, \left( \Wh_r, \Wdv_r \right) \right\}\subset \VectorFields{\ManifoldM}\times \Zg\)
such that
\(\FilteredSheafFh=\FilteredSheafGenBy{\WhWdv}\);
such a \(\WhWdv\) always exists if \(\ManifoldM\) is compact--see Lemma \ref{Lemma::Filtrations::GeneratorsOnRelCptSet}.

Fix \(\kappa\in \Zg\) such that \(\Wdv_j\) divides \(\kappa\) for every \(j\). We consider partial differential operators
of the form
\begin{equation}\label{Eqn::Trace::Vanish::DirichletProblem::opPFormula}
    \opP=\sum_{\DegWdv{\alpha}\leq \kappa} a_\alpha(x) \Wh^{\alpha},\quad a_\alpha\in \CinftySpace[\ManifoldM].
\end{equation}

\begin{definition}[{\cite[Definition 1.1.7]{StreetMaximalSubellipticity}}]
    We say \(\opP\) given by \eqref{Eqn::Trace::Vanish::DirichletProblem::opPFormula} is maximally subelliptic
    of degree \(\kappa\) with respect to \(\FilteredSheafF\) on \(\ManifoldM\) if for all relatively compact, open sets
    \(\Omegah\Subset \ManifoldM\), \(\exists C_{\Omegah}\geq 1\), 
    \begin{equation*}
        \sum_{j=1}^r \BLpNorm{\Wh_j^{\kappa/\Wdv_j} f}{2}[\ManifoldM,\Vol]
        \leq C_{\Omegah}
        \left( \LpNorm{\opP f}{2}[\ManifoldM,\Vol] + \LpNorm{f}{2}[\ManifoldM,\Vol] \right),
        \quad \forall f\in \CinftyCptSpace[\Omegah].
    \end{equation*}
\end{definition}

With \(\opP\) given by \eqref{Eqn::Trace::Vanish::DirichletProblem::opPFormula}, set
\(\opL=\opP^{*}\opP\), where \(\opP\) denotes the formal \(\LpSpace{2}[\ManifoldM,\Vol]\) adjoint
of \(\opP\). We think of \(\opL\) as a symmetric operator on \(\LpSpace{2}[\ManifoldN,\Vol]\)
with dense domain \(\CinftyCptSpace[\ManifoldN]\), and we wish to consider the self-adjoint
extension of \(\opL\) (on \(\LpSpace{2}[\ManifoldN,\Vol]\)) corresponding to Dirichlet boundary conditions.

More precisely, consider the quadratic form
\begin{equation*}
    Q(f)=\BLpNorm{\opP f}{2}^2,\quad f\in \CinftyCptSpace[\InteriorN].
\end{equation*}
The self-adjoint extension of \(\opL\) corresponding Dirichlet boundary conditions is,
by definition, the self-adjoint extension corresponding to the form closure
\(\overline{Q}\) of \(Q\); namely the Friedrichs extension of \(\opL\)
(see \cite[Section X.3]{ReedSimonMethodsOfModernMathematicalPhysicsII}).
The domain of \(\overline{Q}\) is the space \(\DXSpace[2]\) given by the completion
of \(\CinftyCptSpace[\InteriorN]\) in the norm
\(\DXNorm{f}[2]^2:=Q(f)+\LpNorm{f}{2}^2=\LpNorm{\opP f}{2}^2+\LpNorm{f}{2}^2\).
More generally, for \(1<p<\infty\), let \(\DXSpace\) denote the completion
of \(\CinftyCptSpace[\InteriorN]\) in the norm
\(\DXNorm{f}^p:=\LpNorm{\opP f}{p}^p+\LpNorm{f}{p}^p\).
Note that \(\DXSpace\subseteq \LpSpace{p}[\ManifoldN,\Vol]\), and we may therefore identify
elements of \(\DXSpace\) with distributions in \(\DistributionsZeroN\); we henceforth make this identification.

The main result of this subsection (Theorem \ref{Thm::Trace::Vanish::DirichletProblem::MainCharacterizationThm}) characterizes \(\DXSpace\) near a point of \(\BoundaryNncF\).
First, we need that the space \(\DXSpace\) is localizable.

\begin{proposition}\label{Prop::Trace::Vanish::DirichletProblem::MultPsi}
    Suppose \(\opP\) is 
    maximally subelliptic of degree \(\kappa\) with respect to \(\FilteredSheafFh\) on \(\ManifoldM\).
    For \(\psi\in \CinftyCptSpace[\ManifoldN]\), the map \(u\mapsto \psi u\)
    is continuous \(\DXSpace\rightarrow \DXSpace\).
\end{proposition}
We defer the proof of Proposition \ref{Prop::Trace::Vanish::DirichletProblem::MultPsi} to later in this subsection.
The main result of this subsection is the following theorem.

\begin{theorem}\label{Thm::Trace::Vanish::DirichletProblem::MainCharacterizationThm}
    Suppose \(\opP\) is 
    maximally subelliptic of degree \(\kappa\) with respect to \(\FilteredSheafFh\) on \(\ManifoldM\).
    Fix \(1<p<\infty\), and \(\Omega\Subset \ManifoldNncF[\lambda]\)
    for some fixed \(\lambda\in \Zg\), with \(\Omega\cap \BoundaryN\ne \emptyset\).
    Set\footnote{To see \(\kappa/\lambda\in \Zg\),
    note that \(\lambda=\Wdv_j\) for some \(j\)--see Example \ref{Example::Filtrations::RestrictingFiltrations::NonCharExamples} \ref{Item::Filtrations::RestrictingFiltrations::NonCharExamples::CharacterizeNonCharInTermsOfWWdv}. 
    By hypothesis, \(\Wdv_j\) divides \(\kappa\), for every \(j\).} 
    \(L:=\kappa/\lambda-1\in \Zgeq\) and let \(\sB=\left( B_0, B_1,\ldots, B_L \right)\)
    be an \(\FilteredSheafF\)-Dirichlet system on \(\Omega\). For \(u\in \DistributionsZeroN\) and
    \(\psi\in \CinftyCptSpace[\Omega]\), the following are equivalent:
    \begin{enumerate}[(i)]
        \item\label{Item::Trace::Vanish::DirichletProblem::MainCharacterizationThm::uInTLSpace} \(\psi u \in \TLSpace{\kappa}{p}{2}[\overline{\Omega}][\FilteredSheafF]\) and \(\TraceMap[\sB]\psi u=0\).
        \item\label{Item::Trace::Vanish::DirichletProblem::MainCharacterizationThm::uXSpace} \(\psi u \in \DXSpace\).
    \end{enumerate}
\end{theorem}

Before we prove Theorem \ref{Thm::Trace::Vanish::DirichletProblem::MainCharacterizationThm}, we present an important
corollary.

\begin{corollary}\label{Cor::Trace::Vanish::DirichletProblem::GlobalCharacterization}
    Suppose \(\opP\) is 
    maximally subelliptic of degree \(\kappa\) with respect to \(\FilteredSheafFh\) on \(\ManifoldM\).
    Suppose \(\ManifoldN\) is compact and \(\BoundaryN=\BoundaryNncF[\lambda]\) for some fixed \(\lambda\)
    (in particular, every boundary point is assumed \(\FilteredSheafF\)-non-characteristic).
    Set \(L:=\kappa/\lambda-1\in \Zgeq\) and let  \(\sB=\left( B_0, B_1,\ldots, B_L \right)\)
    be an \(\FilteredSheafF\)-Dirichlet system on \(\Omega\).
    Then,
    \begin{equation*}
        \DXSpace=\left\{ u\in \TLSpace{\kappa}{p}{2}[\ManifoldN][\FilteredSheafF] : \TraceMap[\sB] u =0 \right\}, \quad \forall 1<p<\infty.
    \end{equation*}
\end{corollary}
\begin{proof}
    This is the special case of Theorem \ref{Thm::Trace::Vanish::DirichletProblem::MainCharacterizationThm}
    with \(\Omega=\ManifoldN\) and \(\psi=1\in \CinftyCptSpace[\ManifoldN]\).
\end{proof}

We turn to the proofs of Proposition \ref{Prop::Trace::Vanish::DirichletProblem::MultPsi}
and Theorem \ref{Thm::Trace::Vanish::DirichletProblem::MainCharacterizationThm}.
Because \(\ManifoldM\) is a manifold without boundary, the results from
\cite{StreetMaximalSubellipticity} apply to \(\ManifoldM\) and \(\FilteredSheafFh\).

\begin{lemma}\label{Lemma::Trace::Vanish::DirichletProblem::EquivMaxSub}
    Let \(\opP\) be given by \eqref{Eqn::Trace::Vanish::DirichletProblem::opPFormula} and
    \(\opL=\opP^{*}\opP\), so that \(\opL\) is a polynomial of degree \(2\kappa\) in \(\WhWdv\)--see \cite[(8.4)]{StreetMaximalSubellipticity}.
    The following are equivalent:
    \begin{enumerate}[(i)]
        \item\label{Item::Trace::Vanish::DirichletProblem::EquivMaxSub::opLMaxSub} \(\opL\) is maximally subelliptic of degree \(2\kappa\) with respect to \(\FilteredSheafFh\) on \(\ManifoldM\).
        \item\label{Item::Trace::Vanish::DirichletProblem::EquivMaxSub::opPMaxSub} \(\opP\) is maximally subellitpic of degree \(\kappa\) with respect to \(\FilteredSheafFh\) on \(\ManifoldM\).
        \item\label{Item::Trace::Vanish::DirichletProblem::EquivMaxSub::EstimateLpNorms} \(\forall \Omegah\Subset \ManifoldM\), a relatively compact, \(\ManifoldM\)-open set, we have for \(1<p<\infty\),
            \begin{equation*}
                \sum_{\DegWdv{\alpha}\leq \kappa} \BLpNorm{\Wh^{\alpha} f}{p}^p\approx 
                \TLNorm{f}{\kappa}{p}{2}[\FilteredSheafFh]^p
                \approx \BLpNorm{\opP f}{p}^p+\LpNorm{f}{p}^p,\quad \forall f\in \CinftyCptSpace[\Omegah],
            \end{equation*}
            where the implicit constants do not depend on \(f\), but may depend on anything else.
        \item\label{Item::Trace::Vanish::DirichletProblem::EquivMaxSub::LocalizeBesovAndTLSpace} For all \(\psih_1,\psih_2, \psih_3\in \CinftyCptSpace[\ManifoldM]\) with \(\psih_{j+1}\) equal to \(1\)
            on a neighborhood of \(\supp(\psih_j)\) and all \(s\in \R\),
            we have
            \begin{equation*}
                \psih_3\opP u\in \ASpaceCpt{s+\kappa}{p}{q}[\ManifoldM][\FilteredSheafFh]
                \implies 
                \psih_1 u\in \ASpaceCpt{s}{p}{q}[\ManifoldM][\FilteredSheafFh], \quad \forall u\in \Distributions[\ManifoldM].
            \end{equation*}
            Moreover, in this case, we have \(\forall N\geq 0\)
            \begin{equation}\label{Eqn::Trace::Vanish::DirichletProblem::EquivMaxSub::LocalizeBesovAndTLSpace::Estimate}
                \ANorm{\psih_1 u}{s+\kappa}{p}{q}[\FilteredSheafFh]
                \lesssim \ANorm{\psih_3 \opP f}{s}{p}{q}[\FilteredSheafFh]
                +\ANorm{\psih_2 u}{s-N}{p}{q}[\FilteredSheafFh],
            \end{equation}
            where the implicit constant does not depend on \(u\), but may depend on anything else.

        \item\label{Item::Trace::Vanish::DirichletProblem::EquivMaxSub::LocalizeSobolevSpace} \(\forall \psih_1, \psih_2\in \CinftyCptSpace[\ManifoldM]\) with \(\psih_2\) equal to \(1\) on a neighborhood
            of \(\supp(\psih_1)\), we have for \(1<p<\infty\),
            \begin{equation}\label{Eqn::Trace::Vanish::DirichletProblem::EquivMaxSub::LocalizeSobolevSpace::Estimate}
                \TLNorm{\psih_1 f}{\kappa}{p}{2}[\FilteredSheafFh] \lesssim
                \BLpNorm{\psih_2 \opP f}{p}+\BLpNorm{\psih_2 f}{p},\quad \forall f\in \CinftySpace[\ManifoldM],
            \end{equation}
            where the implicit constant does not depend on \(f\), but may depend on anything else.
    \end{enumerate}
\end{lemma}
\begin{proof}
    \ref{Item::Trace::Vanish::DirichletProblem::EquivMaxSub::opLMaxSub}\(\Leftrightarrow\)
    \ref{Item::Trace::Vanish::DirichletProblem::EquivMaxSub::opPMaxSub}\(\Leftrightarrow\)
    \ref{Item::Trace::Vanish::DirichletProblem::EquivMaxSub::LocalizeBesovAndTLSpace}
    is \cite[Theorem 8.1.1 (vi)\(\Leftrightarrow\)(i)\(\Leftrightarrow\)(v)]{StreetMaximalSubellipticity}.

    \ref{Item::Trace::Vanish::DirichletProblem::EquivMaxSub::LocalizeBesovAndTLSpace}\(\Rightarrow\)
    \ref{Item::Trace::Vanish::DirichletProblem::EquivMaxSub::LocalizeSobolevSpace}: 
    Since \(\CinftyCptSpace[\ManifoldM]\subseteq \ASpaceCpt{s}{p}{q}[\ManifoldM][\FilteredSheafFh]\), \(\forall s\)
    (see \cite[Proposition 6.5.5]{StreetMaximalSubellipticity}), 
    \eqref{Eqn::Trace::Vanish::DirichletProblem::EquivMaxSub::LocalizeBesovAndTLSpace::Estimate}
    holds with \(u\) replace by 
    \(f\in \CinftySpace[\ManifoldM]\). Taking the special case \(N=s=0\) and \(\ASpace{s}{p}{q}=\TLSpace{s}{p}{q}\),
    we see for \(f\in \CinftyCptSpace[\ManifoldM]\),
    \begin{equation*}
    \begin{split}
         &\TLNorm{\psih_1 f}{\kappa}{p}{2}[\FilteredSheafFh] 
         \lesssim \TLNorm{\psih_3 \opP f}{0}{p}{q}[\FilteredSheafFh]
                +\TLNorm{\psih_2 u}{0}{p}{q}[\FilteredSheafFh]
        \approx \BLpNorm{\psih_3 \opP f}{p}+\BLpNorm{\psih_2 f}{p}
        \lesssim \BLpNorm{\psih_3 \opP f}{p}+\BLpNorm{\psih_3 f}{p},
    \end{split}
    \end{equation*}
    where the \(\approx\) follows from \cite[Proposition 6.2.13]{StreetMaximalSubellipticity}.
    By renaming \(\psih_3\) to be \(\psih_2\), this establishes \ref{Item::Trace::Vanish::DirichletProblem::EquivMaxSub::LocalizeSobolevSpace}.

    \ref{Item::Trace::Vanish::DirichletProblem::EquivMaxSub::LocalizeSobolevSpace}\(\Rightarrow\)
    \ref{Item::Trace::Vanish::DirichletProblem::EquivMaxSub::EstimateLpNorms}:
    Taking \(\psih_1\) equal to \(1\) on a neighborhood of \(\overline{\Omegah}\) in 
    \eqref{Eqn::Trace::Vanish::DirichletProblem::EquivMaxSub::LocalizeSobolevSpace::Estimate}
    establishes
    \begin{equation*}
        \sum_{\DegWdv{\alpha}\leq \kappa} \BLpNorm{\Wh^{\alpha} f}{p}
        \lesssim \BLpNorm{\opP f}{p}+\LpNorm{f}{p},\quad \forall f\in \CinftyCptSpace[\Omegah].
    \end{equation*}
    That
    \begin{equation*}
        \BLpNorm{\opP f}{p}+\LpNorm{f}{p}
        \lesssim \sum_{\DegWdv{\alpha}\leq \kappa} \BLpNorm{\Wh^{\alpha} f}{p},
        \quad \forall f\in \CinftyCptSpace[\Omegah],
    \end{equation*}
    follows immediately from \eqref{Eqn::Trace::Vanish::DirichletProblem::opPFormula}.
    Finally, that
    \begin{equation*}
        \sum_{\DegWdv{\alpha}\leq \kappa} \BLpNorm{\Wh^{\alpha} f}{p}
        \approx \TLNorm{f}{\kappa}{p}{2}[\FilteredSheafFh],\quad \forall f\in \CinftyCptSpace[\Omegah]
    \end{equation*}
    follows from \cite[Corollary 6.2.14]{StreetMaximalSubellipticity}. Combining these estimates establishes
    \ref{Item::Trace::Vanish::DirichletProblem::EquivMaxSub::EstimateLpNorms}.

    Taking \(p=2\) in \ref{Item::Trace::Vanish::DirichletProblem::EquivMaxSub::EstimateLpNorms}
    gives \ref{Item::Trace::Vanish::DirichletProblem::EquivMaxSub::opPMaxSub}.
\end{proof}

\begin{lemma}\label{Lemma::Trace::Vanish::DirichletProblem::DxNormIsEquivToTLNorm}
    Suppose \(\opP\) is maximally subelliptic of degree \(\kappa\) with respect to \(\FilteredSheafFh\) on \(\ManifoldM\).
    Fix \(1<p<\infty\) and \(\Omegah\Subset\ManifoldM\) open and relatively compact with \(\Omegah\cap \ManifoldN\Subset \ManifoldNncF\).
    Then,
    \begin{equation*}
        \DXNorm{f}\approx \TLNorm{f}{\kappa}{p}{2}[\FilteredSheafFh]\approx \TLNorm{f}{\kappa}{p}{2}[\FilteredSheafF]
        \approx \sum_{\DegWdv{\alpha}\leq \kappa}\BLpNorm{W^{\alpha}f}{p},
        \quad \forall f\in \CinftyCptSpace[\Omegah\cap \InteriorN].
    \end{equation*}
\end{lemma}
\begin{proof}
    That \(\TLNorm{f}{\kappa}{p}{2}[\FilteredSheafFh]\approx \TLNorm{f}{\kappa}{p}{2}[\FilteredSheafF]\approx \sum_{\DegWdv{\alpha}\leq \kappa}\BLpNorm{W^{\alpha}f}{p}\),
    \(\forall f\in \CinftyCptSpace[\Omegah\cap \InteriorN]\) follows from Proposition \ref{Prop::Spaces::Extend::Consequences::DerivForSobolevNormsOnFunctionsOfCptSupport}
    (with \(\Omega_1=\Omegah\)).
    That
    \begin{equation*}
        \TLNorm{f}{\kappa}{p}{2}[\FilteredSheafFh]
        \approx \DXNorm{f},\quad \forall f\in \CinftyCptSpace[\Omegah]
    \end{equation*}
    follows from Lemma \ref{Lemma::Trace::Vanish::DirichletProblem::EquivMaxSub}
    \ref{Item::Trace::Vanish::DirichletProblem::EquivMaxSub::opPMaxSub}\(\Rightarrow\)\ref{Item::Trace::Vanish::DirichletProblem::EquivMaxSub::EstimateLpNorms}.
\end{proof}

\begin{proof}[Proof of Proposition \ref{Prop::Trace::Vanish::DirichletProblem::MultPsi}]
    Let \(\psih_1\in \CinftyCptSpace[\ManifoldM]\) be such that \(\psih_1\big|_{\ManifoldN}=\psi\)
    (see Lemma \ref{Lemma::Spaces::Elem::Extend::Classical}) and let \(\psih_2\in \CinftyCptSpace[\ManifoldM]\)
    be equal to \(1\) on a neighborhood of \(\supp(\psih_1)\).
    Fix \(\Omegah\Subset \ManifoldM\) an \(\ManifoldM\)-open, relatively compact set with
    \(\psih_2\in \CinftyCptSpace[\Omegah]\).
    Then, for \(f\in \CinftyCptSpace[\InteriorN]\) we have,
    using Lemma \ref{Lemma::Trace::Vanish::DirichletProblem::EquivMaxSub} 
    \ref{Item::Trace::Vanish::DirichletProblem::EquivMaxSub::opPMaxSub}\(\Rightarrow\)
    \ref{Item::Trace::Vanish::DirichletProblem::EquivMaxSub::EstimateLpNorms},\ref{Item::Trace::Vanish::DirichletProblem::EquivMaxSub::LocalizeSobolevSpace},
    \begin{equation*}
    \begin{split}
         &\DXNorm{\psi f}
         \approx \BLpNorm{\opP \psi f}{p} + \LpNorm{\psi f}{p}
         =\BLpNorm{\opP \psih_1 f}{p} + \LpNorm{\psih_1 f}{p}
         \approx \TLNorm{\psih_1 f}{\kappa}{p}{2}[\FilteredSheafFh]
         \\&\lesssim 
         \BLpNorm{\psih_2 \opP f}{p} + \BLpNorm{\psih_2 f}{p}
         \lesssim \BLpNorm{\opP f}{p} + \BLpNorm{f}{p}
         \approx \DXNorm{f}.
    \end{split}
    \end{equation*}
    The result follows.
\end{proof}

\begin{proof}[Proof of Theorem \ref{Thm::Trace::Vanish::DirichletProblem::MainCharacterizationThm}]
    Note that \(L\) is chosen so that \(\kappa\in (L\lambda+\lambda/p, (L+1)\lambda+\lambda/p)\),
    which means that we may apply Theorem \ref{Thm::Trace::Vanish::MainVanishThm} \ref{Item::Trace::Vanish::MainVanishThm::Larges}
    with \(s=\kappa\).

    \ref{Item::Trace::Vanish::DirichletProblem::MainCharacterizationThm::uInTLSpace}
    \(\Rightarrow\)
    \ref{Item::Trace::Vanish::DirichletProblem::MainCharacterizationThm::uXSpace}:
    Let \(\Omega_1\Subset \Omega\) be \(\ManifoldN\)-open and relatively compact with
    \(\psi\in \CinftyCptSpace[\Omega_1]\). Then, by Proposition \ref{Prop::Spaces::Containment}
    we have \(\psi u \in \TLSpace{\kappa}{p}{2}[\overline{\Omega_1}][\FilteredSheafF]\).
    By assumption we have \(\TraceMap[\sB] \psi u =0\).
    Theorem \ref{Thm::Trace::Vanish::MainVanishThm} \ref{Item::Trace::Vanish::MainVanishThm::Larges} shows
    that there exists \(f_j\in \CinftyCptSpace[\Omega\cap \InteriorN]\) with \(f_j\rightarrow \psi u\)
    in \(\TLSpace{\kappa}{p}{2}[\overline{\Omega}][\FilteredSheafF]\). By Lemma \ref{Lemma::Trace::Vanish::DirichletProblem::DxNormIsEquivToTLNorm}
    we have
    \begin{equation*}
        \DXNorm{f_j-f_k}\approx \TLNorm{f_j-f_k}{\kappa}{p}{2}[\FilteredSheafF],
    \end{equation*}
    and therefore \(f_j\) is Cauchy in \(\DXSpace\). We conclude \(f_j\rightarrow \psi u\)
    in \(\DXSpace\), proving \(\psi u\in \DXSpace\).

    \ref{Item::Trace::Vanish::DirichletProblem::MainCharacterizationThm::uXSpace}
    \(\Rightarrow\)
    \ref{Item::Trace::Vanish::DirichletProblem::MainCharacterizationThm::uInTLSpace}:
    Suppose \(\psi u\in \DXSpace\). Fix \(\psi_2\in \CinftyCptSpace[\Omega]\) with \(\psi_2\)
    equal to \(1\) on a neighborhood of \(\supp(\psi)\).
    By definition, there exists \(f_j\in \CinftyCptSpace[\InteriorN]\)
    with \(f_j\rightarrow \psi u\) in \(\DXSpace\). By  Proposition \ref{Prop::Trace::Vanish::DirichletProblem::MultPsi},
    \(\psi_2 f_j\) is Cauchy in \(\DXSpace\), and therefore convergent in \(\DXSpace\). Clearly, \(\psi_2 f_j\rightarrow \psi u\)
    in \(\DXSpace\). By Lemma \ref{Lemma::Trace::Vanish::DirichletProblem::DxNormIsEquivToTLNorm},
    \begin{equation*}
        \DXNorm{\psi_2 f_j-\psi_2 f_k}\approx \TLNorm{\psi_2 f_j-\psi_2 f_k}{\kappa}{p}{2}[\FilteredSheafF],
    \end{equation*}
    and therefore \(\psi_2 f_j\) is Cauchy in \(\TLSpace{\kappa}{p}{2}[\overline{\Omega}][\FilteredSheafF]\).
    We conclude \(\psi_2 f_j\rightarrow \psi u\) in \(\TLSpace{\kappa}{p}{2}[\overline{\Omega}][\FilteredSheafF]\),
    and in particular, \(\psi u\in \TLSpace{\kappa}{p}{2}[\overline{\Omega}][\FilteredSheafF]\).
    Theorem \ref{Thm::Trace::Vanish::MainVanishThm} \ref{Item::Trace::Vanish::MainVanishThm::Larges} shows
    \(\TraceMap[\sB]\psi u=0\), completing the proof.
\end{proof}

    \subsection{Reduction to local results}\label{Section::Trace::ReductionToLocal}
    The main results of this chapter (Theorems \ref{Thm::Trace::ForwardMap}, \ref{Thm::Trace::Dirichlet::MainInverseThm}, and \ref{Thm::Trace::Vanish::MainVanishThm})
are implicitly local.
In this section, we replace them with explicitly local (and slightly simpler) versions and then, in Section \ref{Section::Trace::ReductionToCoord},
we restate these local versions in a convenient coordinate system.

% \begin{proposition}[Local version of Theorem \ref{Thm::Trace::ForwardMap}]
%     \label{Prop::Trace::Reduction::ForwardMap}
%     Let \(x_0\in \BoundaryNncF[\lambda]\) for some \(\lambda\). There are open neighborhoods \(\Omega_1\Subset\Omega_2\Subset \ManifoldNncF[\lambda]\)
%     of \(x_0\)
%     such that Theorem \ref{Thm::Trace::ForwardMap} holds with \(\Omega=\Omega_2\)
%     and \(\Compact=\overline{\Omega_1}\).
% \end{proposition}

\begin{proposition}[Local version of Theorem \ref{Thm::Trace::ForwardMap}]
    \label{Prop::Trace::Reduction::ForwardMap}
    Let \(x_0\in \BoundaryNncF[\lambda]\) for some \(\lambda\). There is an open neighborhood \(\Omega\Subset \ManifoldNncF[\lambda]\)
    of \(x_0\)
    such that for all \(1<p\leq \infty\), \(1\leq q\leq \infty\), \(s>\lambda/p\)
    \begin{equation*}
        \BigBesovNorm{f\big|_{\BoundaryN}}{s-\lambda/p}{p}{q}[\RestrictFilteredSheaf{\LieFilteredSheafF}{\BoundaryNncF}]
        \lesssim \BesovNorm{f}{s}{p}{q}[\FilteredSheafF],\quad \forall f\in \CinftyCptSpace[\Omega],
    \end{equation*}
    and for all \(1<p<\infty\), \(1<q\leq \infty\), \(s>\lambda/p\)
    \begin{equation*}
        \BigBesovNorm{f\big|_{\BoundaryN}}{s-\lambda/p}{p}{p}[\RestrictFilteredSheaf{\LieFilteredSheafF}{\BoundaryNncF}]
        \lesssim \TLNorm{f}{s}{p}{q}[\FilteredSheafF],\quad \forall f\in \CinftyCptSpace[\Omega].
    \end{equation*}
\end{proposition}

Theorems \ref{Thm::Trace::Dirichlet::MainInverseThm} and \ref{Thm::Trace::Vanish::MainVanishThm}
involve an \(\FilteredSheafF\)-Dirichlet system \(\sB\). For the next two local version, we use
a particular Dirichlet system. 
Given a point \(x_0\in \BoundaryNncF[\lambda]\), choose a neighborhood \(U\Subset \ManifoldNncF[\lambda]\)
and a vector field \(V\in \LieFilteredSheafF[U][\lambda]\) with \(V(x)\not \in \TangentSpace{x}{\BoundaryN}\),
\(\forall x\in U\cap \BoundaryN\). Take
\(\sN_L=\left( 1,V, V^2,\ldots, V^L \right)\); so that \(\sN\) is an \(\FilteredSheafF\)-Dirichlet system as pointed out in
Example \ref{Example::Dirichlet::VpowersIsDirichletSystem}.

\begin{proposition}[Local version of Theorem \ref{Thm::Trace::Dirichlet::MainInverseThm}]
    \label{Prop::Trace::Reduction::InverseMap}
    Let \(x_0\in \BoundaryNncF[\lambda]\) for some \(\lambda\). There are open neighborhoods
    \(\Omega_1\Subset\Omega_2\Subset \Omega\Subset \ManifoldNncF[\lambda]\) of \(x_0\)
    such that there exists a continuous, linear map
    \begin{equation*}
        \TraceInverseMap[\sN_L]:\Distributions[\Omega_2\cap\BoundaryN]^{L+1}\rightarrow \DistributionsZero[\ManifoldN]
    \end{equation*}
    satisfying
    \begin{enumerate}[(i)]
        \item\label{Item::Trace::Reduction::Dirichlet::MainInverseThm::ContinuousOnSmoothFuncs} 
        \(\TraceInverseMap[\sN_L]:\CinftySpace[\Omega_2\cap \BoundaryN]^{L+1}\rightarrow \CinftyCptSpace[\Omega]\), continuously.
        \item\label{Item::Trace::Reduction::Dirichlet::MainInverseThm::ContinuousOnSpaces} \(\TraceInverseMap[\sN_L]\) is continuous:
            \begin{equation*}
                \TraceInverseMap[\sN_L]:
                \prod_{l=0}^L
                \BigBesovSpace{s-l\lambda-\lambda/p}{p}{q}[\overline{\Omega_1}\cap \BoundaryN][\RestrictFilteredSheaf{\LieFilteredSheafF}{\BoundaryNncF}]
                \rightarrow
                \BesovSpace{s}{p}{q}[\overline{\Omega_2}][\FilteredSheafF],\quad 1\leq p\leq \infty, \: 1\leq q\leq\infty, \: s\in \R,
            \end{equation*}
            \begin{equation*}
                \TraceInverseMap[\sN_L]:
                \prod_{l=0}^L
                \BigBesovSpace{s-l\lambda-\lambda/p}{p}{p}[\overline{\Omega_1}\cap \BoundaryN][\RestrictFilteredSheaf{\LieFilteredSheafF}{\BoundaryNncF}]
                \rightarrow
                \TLSpace{s}{p}{q}[\overline{\Omega_2}][\FilteredSheafF],\quad 1< p< \infty, \: 1< q\leq\infty, \: s\in \R.
            \end{equation*}
        \item\label{Item::Trace::Reduction::Dirichlet::MainInverseThm::IsInverse} 
        \(\TraceMap[\sN_L]\TraceInverseMap[\sN_L]f=f\), \(\forall f\in \CinftyCptSpace[\Omega_1]^{L+1}\).

    \end{enumerate}
    
    % Theorem \ref{Thm::Trace::Dirichlet::MainInverseThm} holds for these choices
    % of \(\Omega_1\), \(\Omega_2\), and \(\Omega\), with \(\sB=\sN_L\).
\end{proposition}

% \begin{proposition}[Local version of Theorem \ref{Thm::Trace::Vanish::MainVanishThm}]
%     \label{Prop::Trace::Reduction::VanishingChar}
%     Let \(x_0\in \BoundaryNncF[\lambda]\) for some \(\lambda\).
%     There are open neighborhoods \(\Omega_1\Subset\Omega\Subset \ManifoldNncF[\lambda]\)
%     of \(x_0\) such that Theorem \ref{Thm::Trace::Vanish::MainVanishThm} holds for this choice of \(\Omega\),
%     with \(\Compact=\overline{\Omega_1}\),
%     and \(\sB=\sN_L\) in \ref{Item::Trace::Vanish::MainVanishThm::Larges}.
% \end{proposition}

\begin{proposition}[Local version of Theorem \ref{Thm::Trace::Vanish::MainVanishThm}]
    \label{Prop::Trace::Reduction::VanishingChar}
    Let \(x_0\in \BoundaryNncF[\lambda]\) for some \(\lambda\).
    There are open neighborhoods \(\Omega_1\Subset\Omega\Subset \ManifoldNncF[\lambda]\)
    of \(x_0\) such that the following holds.
    Fix \(p,q\in (1,\infty)\), \(s\in \R\), and \(u\in \ASpace{s}{p}{q}[\overline{\Omega_1}][\FilteredSheafF]\).
    \begin{enumerate}[(A)]
        \item\label{Item::Trace::Reduction::VanishingChar::Smalls} If \(s<\lambda/p\), \(\exists \left\{ f_j \right\}_{j\in \Zgeq}\subset \CinftyCptSpace[\Omega\cap \InteriorN]\)
            with \(f_j\rightarrow u\) in \(\ASpace{s}{p}{q}[\overline{\Omega}][\FilteredSheafF]\).
        \item\label{Item::Trace::Reduction::VanishingChar::Larges} If \(s\in (L\lambda+\lambda/p, (L+1)\lambda+\lambda/p)\), where \(L\in \Zgeq\), and if 
            \(\TraceMap[\sN_L]u=0\), then \(\exists \left\{ f_j \right\}_{j\in \Zgeq}\subset \CinftyCptSpace[\Omega\cap \InteriorN]\)
            with \(f_j\rightarrow u\) in \(\ASpace{s}{p}{q}[\overline{\Omega}][\FilteredSheafF]\).
    \end{enumerate}
    
    % Theorem \ref{Thm::Trace::Vanish::MainVanishThm} holds for this choice of \(\Omega\),
    % with \(\Compact=\overline{\Omega_1}\),
    % and \(\sB=\sN_L\) in \ref{Item::Trace::Vanish::MainVanishThm::Larges}.
\end{proposition}

        \subsubsection{Proof of the reduction to local propositions}\label{Section::Trace::ReductionToLocal::Proof}
        In this subsection, we show how
Propositions \ref{Prop::Trace::Reduction::ForwardMap}, \ref{Prop::Trace::Reduction::InverseMap}, and \ref{Prop::Trace::Reduction::VanishingChar}
imply Theorems \ref{Thm::Trace::ForwardMap}, \ref{Thm::Trace::Dirichlet::MainInverseThm}, and \ref{Thm::Trace::Vanish::MainVanishThm}, respectively.

\begin{proof}[Proof of Theorem \ref{Thm::Trace::ForwardMap}]
    We begin with uniqueness.  Let \(\TraceMap[B]^1\) and \(\TraceMap[B]^2\) be two maps
    satisfying the conclusions of Theorem \ref{Thm::Trace::ForwardMap}.
    Let \(\Compact \Subset \Omega\) be compact and pick \(\Omega_1\Subset \Omega\) open and relatively compact
    with \(\Compact\Subset \Omega_1\).
    Let \(s>\kappa+\lambda/p\) and \(f\in \ASpace{s}{p}{q}[\Compact][\FilteredSheafF]\).
    We wish to show \(\TraceMap[B]^1 f= \TraceMap[B]^2f\).
    
    Suppose \(q<\infty\). By Corollary \ref{Cor::Spaces::Approximation::SmoothFunctionsAreDense} \ref{Item::Spaces::Approximation::SmoothFunctionsAreDense::ApproxInST},
    there exist \(f_j\in \CinftyCptSpace[\Omega_1]\) with \(f_j\rightarrow f\) in \(\ASpace{s}{p}{q}[\overline{\Omega_1}][\FilteredSheafF]\).
    Since we are assuming \ref{Item::Trace::ForwardMap::TraceOnSmooth} and \ref{Item::Trace::ForwardMap::Continuous}
    hold, we have
    \begin{equation*}
    \begin{split}
         & \TraceMap[B]^1 f = \lim_{j\rightarrow\infty} \TraceMap[B]^1 f_j 
         =\lim_{j\rightarrow\infty} f_j\big|_{\BoundaryN}
         = \lim_{j\rightarrow\infty} \TraceMap[B]^2 f_j 
         =\TraceMap[B]^2 f,
    \end{split}
    \end{equation*}
    as desired.

    For \(q=\infty\), take \(\epsilon>0\) so small \(s-\epsilon>\kappa+\lambda/p\).
    Then, \(\ASpace{s}{p}{\infty}[\Compact][\FilteredSheafF]\subseteq\ASpace{s-\epsilon}{p}{2}[\Compact][\FilteredSheafF]\).
    Since we have already shown \(\TraceMap[B]^1\) and \(\TraceMap[B]^2\) agree on 
    \(\ASpace{s-\epsilon}{p}{2}[\Compact][\FilteredSheafF]\), they must also agree on 
    \(\ASpace{s}{p}{\infty}[\Compact][\FilteredSheafF]\), completing the proof of uniqueness.

    With uniqueness proved, we turn to existence.
    Let \(\TraceMap=\TraceMap[1]\) be the map with \(B=1\); i.e., \(\TraceMap f = f\big|_{\BoundaryN}\) for smooth functions.
    By Proposition \ref{Prop::Spaces::MappingOfFuncsAndVFs} \ref{Item::Spaces::MappingOfFuncsAndVFs::DiffOps},
    \(B:\ASpace{s}{p}{q}[\Compact][\FilteredSheafF]\rightarrow \ASpace{s-\kappa}{p}{q}[\Compact][\FilteredSheafF]\)
    continuously (for all \(\Compact\Subset \Omega\), compact).
    Thus, if we prove the existence of \(\TraceMap\) as in the theorem (with \(B=1\) and \(\kappa=0\)), then setting
    \(\TraceMap[B]:=\TraceMap\circ B\) establishes the result for general \(B\).
    We henceforth assume \(B=1\) and \(\kappa=0\).

    Let \(\Compact\Subset \Omega\), informally, our main goal is to show that \(\TraceMap:f\mapsto f\big|_{\BoundaryN}\)
    extends to a continuous map
    \begin{equation*}
        \XTraceSpace[\Compact]\rightarrow \YTraceSpace[\Compact],
    \end{equation*}
    where either
    \begin{itemize}
        \item \(\XTraceSpaceShort=\BesovSpace{s}{p}{q}\), \(1<p\leq \infty\), \(1\leq q\leq \infty\), \(s>\lambda/p\),
            and \(\YTraceSpaceShort=\BesovSpace{s-\lambda/p}{p}{q}\); or,
        \item \(\XTraceSpaceShort=\TLSpace{s}{p}{q}\), \(1<p<\infty\), \(1<q\leq \infty\), \(s>\lambda/p\),
            and \(\YTraceSpaceShort=\BesovSpace{s-\lambda/p}{p}{p}\).
    \end{itemize}

    Let \(U\Subset \ManifoldNncF[\lambda]\) be a relatively compact, open set with \(\Omega\Subset U\).
    For each \(x\in \overline{U}\cap \BoundaryN\), let \(\Omega_x\Subset \ManifoldNncF[\lambda]\)
    be the neighborhood of \(x\) guaranteed by Proposition \ref{Prop::Trace::Reduction::ForwardMap}.
    \(\left\{ \Omega_x : x\in \overline{U} \cap \BoundaryN\right\}\) is an open cover for the compact
    set \(\overline{U}\cap\BoundaryN\); extract a finite subcover \(\Omega_{x_1},\ldots, \Omega_{x_L}\).
    Let \(\phi_j \in \CinftyCptSpace[\Omega_{x_j}]\) be such that \(\sum_{j=1}^L \phi_j =1\) on a neighborhood of \(\overline{U}\cap \BoundaryN\).
    We have, for \(f\in \CinftyCptSpace[U]\),
    \begin{equation}\label{Eqn::Trace::Reduction::ForwardMap::Tmp1}
        \TraceMap f = f\big|_{\BoundaryN}  = \sum_{j=1}^L \left( \phi_j f  \right)\big|_{\BoundaryN}  = \sum_{j=1}^L\TraceMap \phi_j f.
    \end{equation}
    Proposition \ref{Prop::Trace::Reduction::ForwardMap} gives
    \begin{equation}\label{Eqn::Trace::Reduction::ForwardMap::Tmp2}
        \YTraceNorm{\TraceMap\phi_j f} \lesssim \XTraceNorm{\phi_j f}\lesssim \XTraceNorm{f}, \quad f\in \CinftyCptSpace[U],
    \end{equation}
    where the final \(\lesssim\) uses Proposition \ref{Prop::Spaces::MappingOfFuncsAndVFs} \ref{Item::Spaces::MappingOfFuncsAndVFs::Funcs}.
    Combining \eqref{Eqn::Trace::Reduction::ForwardMap::Tmp1} and \eqref{Eqn::Trace::Reduction::ForwardMap::Tmp2}
    gives
    \begin{equation}\label{Eqn::Trace::Reduction::ForwardMap::APrioriNormEstimate}
        \YTraceNorm{\TraceMap f}\leq \sum_{j=1}^L \YTraceNorm{\TraceMap \phi_j f} \lesssim \XTraceNorm{f}, \quad \forall f\in \CinftyCptSpace[U].
    \end{equation}

    We now address the case \(q<\infty\).
    Let \(\XTraceSpacez[\overline{U}]\) denote the closure of \(\CinftyCptSpace[U]\) in 
    \(\XTraceSpace[\overline{U}]\).
    \eqref{Eqn::Trace::Reduction::ForwardMap::APrioriNormEstimate} shows that
    \(\TraceMap\) extends uniquely to a continuous map
    \begin{equation*}
        \XTraceSpacez[\overline{U}] \rightarrow \YTraceSpace[\overline{U}].
    \end{equation*}
    A priori, this map may depend on the particular choice of \(\XTraceSpaceShort\); so we write
    \begin{equation*}
        \XTraceMap:\XTraceSpacez[\overline{U}] \rightarrow \YTraceSpace[\overline{U}].
    \end{equation*}
    By Corollary \ref{Cor::Spaces::Approximation::SmoothFunctionsAreDense} \ref{Item::Spaces::Approximation::SmoothFunctionsAreDense::ApproxInST},
    \(\XTraceSpace[\overline{\Omega}]\subseteq \XTraceSpacez[\overline{U}]\), and by Proposition \ref{Prop::Spaces::Containment}
    it is a closed subspace. We conclude that \(\XTraceMap\) restricts to a continuous map
    \begin{equation}\label{Eqn::Trace::Reduction::ForwardMap::TmpContinuous}
        \XTraceMap:\XTraceSpace[\overline{\Omega}] \rightarrow \YTraceSpace[\overline{U}].
    \end{equation}
    We next claim that the map in \eqref{Eqn::Trace::Reduction::ForwardMap::TmpContinuous}
    does not depend on the choice of \(\XTraceSpaceShort\). More precisely,
    Let \(\XtTraceSpaceShort\) be another of the relevant Besov of Triebel--Lizorkin spaces we are studying
    (still with \(q<\infty\)), and let \(\XtTraceMap\) be the map from \eqref{Eqn::Trace::Reduction::ForwardMap::TmpContinuous}
    on \(\XtTraceSpace[\overline{\Omega}]\). Let \(f\in \XTraceSpace[\overline{\Omega}]\cap \XtTraceSpace[\overline{\Omega}]\).
    Fix \(U_0\Subset U\) open and relatively compact, with \(\Omega\Subset U_0\).
    By Theorem \ref{Thm::Spaces::Approximation::BasicPlProps} \ref{Item::Spaces::Approximation::BasicPlProps::ConvergeInSOT},
    applied with some \(\psi\in \CinftyCptSpace[U_0]\) equal to \(1\) on a neighborhood of \(\overline{\Omega}\),
    there exists \(f_l\in \CinftyCptSpace[U_0]\) such that \(f_l\rightarrow f\) in both
    \(\XTraceSpace[\overline{U_0}]\) and in \(\XtTraceSpace[\overline{U_0}]\).
    Therefore, with limits taken in the sense of distributions,
    \begin{equation*}
        \XTraceMap f = \lim_{l\rightarrow \infty} \XTraceMap f_l = \lim_{l\rightarrow \infty}  f_l\big|_{\BoundaryN} 
        = \lim_{l\rightarrow \infty} \XtTraceMap f_l = \XtTraceMap f.
    \end{equation*}
    Thus, we may drop \(\XTraceSpaceShort\) from the notation \(\XTraceMap\), and have a continuous map
    \begin{equation}\label{Eqn::Trace::Reduction::ForwardMap::Continuous}
        \TraceMap:\XTraceSpace[\overline{\Omega}] \rightarrow \YTraceSpace[\overline{U}],
    \end{equation}
    defined for all such \(\XTraceSpaceShort\), and agreeing on the intersections of any two such spaces.

    We next claim that for \(f\in \XTraceSpace[\overline{\Omega}]\), \(\supp(\TraceMap f)\subseteq \supp(f)\cap \BoundaryN\).
    Indeed, take any \(U_0\Subset U\) open and relatively compact with \(\supp(f)\Subset U_0\).
    By Corollary \ref{Cor::Spaces::Approximation::SmoothFunctionsAreDense} \ref{Item::Spaces::Approximation::SmoothFunctionsAreDense::ApproxInST},
    there exist \(f_j\in \CinftyCptSpace[U_0]\) such that \(f_j\rightarrow f\) in
    \(\XTraceSpace[\overline{U_0}]\).  We therefore have
    \begin{equation*}
        \TraceMap f = \XTraceMap f = \lim_{j\rightarrow \infty} \XTraceMap f_j = \lim_{j\rightarrow \infty} f_j\big|_{\BoundaryN},
    \end{equation*}
    with convergence in \(\YTraceSpace[\overline{U}]\), and therefore in \(\Distributions[\BoundaryN]\).
    Since \(\supp(f_j\big|_{\BoundaryN})\subseteq U_0\cap \BoundaryN\), we conclude
    \(\supp(\TraceMap f )\subseteq \overline{U_0}\cap \BoundaryN\).
    As \(U_0\) was an arbitrary set with \(\supp(f)\subseteq U_0\), it follows that
    \(\supp(\TraceMap f)\subseteq \supp(f)\cap \BoundaryN\), as desired.
    Combining this with \eqref{Eqn::Trace::Reduction::ForwardMap::Continuous}, Proposition \ref{Prop::Spaces::Containment}
    shows that for \(\Compact\Subset \Omega\) compact,
    \begin{equation*}
        \TraceMap:\XTraceSpace[\Compact] \rightarrow \YTraceSpace[\Compact].
    \end{equation*}

    This completes the proof for \(q<\infty\), and we turn to the case \(q=\infty\).
    In this case, \(\XTraceSpaceShort\) is either \(\BesovSpace{s}{p}{\infty}\)
    or \(\TLSpace{s}{p}{\infty}\); we present the proof for \(\BesovSpace{s}{p}{\infty}\),
    and the proof for \(\TLSpace{s}{p}{\infty}\) is nearly identical, and we leave the details to the reader.
    The main difficulty we must address is that the conclusions of 
    Theorem \ref{Thm::Spaces::Approximation::BasicPlProps} and Corollary \ref{Cor::Spaces::Approximation::SmoothFunctionsAreDense}
    are weaker when \(q=\infty\).

    First we note that \(\TraceMap\) is defined on \(\BesovSpace{s}{p}{\infty}[\Compact][\FilteredSheafF]\),
    for \(\Compact\Subset \Omega\) compact.  Indeed, take \(s'\) so that  \(s>s'>\kappa+\lambda/p\).
    Then \(\BesovSpace{s}{p}{\infty}[\Compact][\FilteredSheafF]\subseteq \BesovSpace{s'}{p}{2}[\Compact][\FilteredSheafF]\).
    Since \(\TraceMap\) is defined on \(\BesovSpace{s'}{p}{2}[\Compact][\FilteredSheafF]\)
    (as described above), it is also defined on \(\BesovSpace{s}{p}{\infty}[\Compact][\FilteredSheafF]\).

    Fix \(\Omega_1\Subset \Omega\) with open and relatively compact with \(\Compact\Subset \Omega_1\).
    Let \(P_l\) be as in Theorem \ref{Thm::Spaces::Approximation::BasicPlProps} with \(\Omega\) replaced by
    \(\Omega_1\), and with \(\psi\in \CinftyCptSpace[\Omega_1]\) chosen equal to \(1\) on a neighborhood of \(\Compact\).
    For \(f\in \BesovSpace{s}{p}{\infty}[\Compact][\FilteredSheafF]\subseteq \BesovSpace{s'}{p}{2}[\Compact][\FilteredSheafF]\),
    we have by \eqref{Eqn::Trace::Reduction::ForwardMap::APrioriNormEstimate}
    and Theorem \ref{Thm::Spaces::Approximation::BasicPlProps} \ref{Item::Spaces::Approximation::BasicPlProps::UnifBdd}
    \begin{equation}\label{Eqn::Trace::Reduction::ForwardMap::UniformBoundForApprox}
    \begin{split}
         &\sup_{l\in \Zgeq} \BigBesovNorm{\TraceMap P_l f}{s-\lambda/p}{p}{\infty}[\RestrictFilteredSheaf{\LieFilteredSheafF}{\BoundaryNncF}]
         \lesssim \sup_{l\in \Zgeq} \BesovNorm{P_l f}{s}{p}{\infty}[\FilteredSheafF]
         \lesssim \BesovNorm{f}{s}{p}{\infty}[\FilteredSheafF].
    \end{split}
    \end{equation}
    Since \(f\in \BesovSpace{s'}{p}{2}[\Compact][\FilteredSheafF]\),
    Theorem \ref{Thm::Spaces::Approximation::BasicPlProps} \ref{Item::Spaces::Approximation::BasicPlProps::ConvergeInSOT}
    shows \(P_l f\rightarrow \psi f=f\) in \(\BesovSpace{s'}{p}{2}[\overline{\Omega_1}][\FilteredSheafF]\).
    Thus, by the already established case of \(q<\infty\), we see
    \(\TraceMap P_l f\rightarrow \TraceMap f\) in \(\BesovSpace{s'}{p}{2}[\overline{\Omega_1}][\FilteredSheafF]\),
    and in particular,
    \begin{equation}\label{Eqn::Trace::Reduction::ForwardMap::ConvergeApprox}
        \TraceMap P_l f\rightarrow \TraceMap f \text{ in }\Distributions[\BoundaryN].
    \end{equation}
    Using Proposition \ref{Prop::Spaces::Approximation::DistributionConvgToStronger} \ref{Item::Spaces::Approximation::DistributionConvgToStronger::NormBound}
    we see \(\TraceMap f \in \BigBesovSpace{s-\lambda/p}{p}{\infty}[\overline{\Omega_1}\cap\BoundaryN][\RestrictFilteredSheaf{\LieFilteredSheafF}{\BoundaryNncF}]\)
    with
    \begin{equation}\label{Eqn::Trace::Reduction::ForwardMap::InftyTraceNorm}
        \BigBesovNorm{\TraceMap f}{s-\lambda/p}{p}{\infty}[\RestrictFilteredSheaf{\LieFilteredSheafF}{\BoundaryNncF}]
        \lesssim \BesovNorm{f}{s}{p}{\infty}[\FilteredSheafF].
    \end{equation}
    Finally, since we have already shown
    \(\supp(\TraceMap f)\subseteq \supp(f)\cap \BoundaryN\subseteq \Compact\cap \BoundaryN\), 
    Proposition \ref{Prop::Spaces::Containment} shows
    \begin{equation*}
        \TraceMap: \BesovSpace{s}{p}{\infty}[\Compact][\FilteredSheafF]
        \rightarrow 
        \BigBesovSpace{s-\lambda/p}{p}{\infty}[\Compact\cap\BoundaryN][\RestrictFilteredSheaf{\LieFilteredSheafF}{\BoundaryNncF}],
    \end{equation*}
    continuously, completing the proof.
\end{proof}

In Propositions \ref{Prop::Trace::Reduction::InverseMap} and \ref{Prop::Trace::Reduction::VanishingChar},
we use the particular Dirichlet system \(\sN_L\), while in Theorems \ref{Thm::Trace::Dirichlet::MainInverseThm} and \ref{Thm::Trace::Vanish::MainVanishThm}
we use a general Dirichlet system. The next proposition helps us reduce general Dirichlet systems to this particular choice.

Fix, for the rest of the subsection, \(\Omega\Subset \ManifoldNncF[\lambda]\) (for some fixed \(\lambda\in \Zg\)) and let
\begin{equation}\label{Eqn::Trace::Reduction::DirichletSystemAsColumn}
    \sB=\begin{bmatrix}
B_0 \\
B_1 \\
B_2\\
\vdots \\
B_L
\end{bmatrix},\quad
\sN_L=
\begin{bmatrix}
1 \\
V \\
V^2\\
\vdots \\
V^L
\end{bmatrix},
\end{equation}
where \(\sB\) is an \(\FilteredSheafF\)-Dirichlet system on \(\Omega\), and \(\sN_L\) is the \(\FilteredSheafF\)-Dirichlet
system discussed earlier for a particular choice of \(V\); see, also, Example \ref{Example::Dirichlet::VpowersIsDirichletSystem}.
In particular, 
 \(V\in \LieFilteredSheafF[\Omega][\lambda]\) with \(V(x)\not \in \TangentSpace{x}{\BoundaryN}\),
\(\forall x\in \Omega\cap \BoundaryN\).
In this subsection, it will be convenient to treat \(\sB\) and \(\sN_L\) as column vectors.

\begin{proposition}\label{Prop::Trace::Reduction::TurnBIntoN}
    There exists an \((L+1)\times (L+1)\) matrix, \(M\), with operator coefficients (indexed by \(0,\ldots, L\))
    given by 
    \begin{equation*}
        M_j^k
        =\begin{cases}
            a_j(x)\in \CinftySpace[\Omega\cap\BoundaryN], &j=k,\\
            P_j^k, & k<j,\\
            0, &k>j,
        \end{cases}
    \end{equation*}
    where \(a_j(x)\ne 0\), \(\forall x\in \Omega\cap\BoundaryN\), \(P_j^k\)
    is a
    partial differential operator which is 
    \(\RestrictFilteredSheaf{\LieFilteredSheafF}{\BoundaryNncF}\)-degree \(\leq j\lambda-k\lambda\) on \(\Omega\cap\BoundaryN\),
    and such that \(M\TraceMap[\sN_L]=\TraceMap[\sB]\). 
\end{proposition}

Before we prove Proposition \ref{Prop::Trace::Reduction::TurnBIntoN},
we show how it can be used to reduce 
Theorems \ref{Thm::Trace::Dirichlet::MainInverseThm} and \ref{Thm::Trace::Vanish::MainVanishThm}
to the case \(\sB=\sN_L\).

\begin{proof}[Reduction of Theorems \ref{Thm::Trace::Dirichlet::MainInverseThm} and \ref{Thm::Trace::Vanish::MainVanishThm}
    to the case \(\sB=\sN_L\)]
    Let \(M\) be the matrix from Proposition \ref{Prop::Trace::Reduction::TurnBIntoN}.
    Since \(M\) is a triangular matrix (with never vanishing smooth functions on the diagonal), it is easy to see that \(M^{-1}\) is of the same form as \(M\).
    It follows from Proposition \ref{Prop::Spaces::MappingOfFuncsAndVFs} \ref{Item::Spaces::MappingOfFuncsAndVFs::Funcs}
    and \ref{Item::Spaces::MappingOfFuncsAndVFs::DiffOps} that for \(\Compact\Subset \Omega\) compact, 
    \(M\) and \(M^{-1}\) are bounded on
    \begin{equation}\label{Eqn::Trace::Reduction::MBoundedOnSpace}
        \BigASpace{s}{p}{q}[\Compact\cap\BoundaryN][\RestrictFilteredSheaf{\LieFilteredSheafF}{\BoundaryNncF}]
        \times 
        \BigASpace{s-\lambda}{p}{q}[\Compact\cap\BoundaryN][\RestrictFilteredSheaf{\LieFilteredSheafF}{\BoundaryNncF}]
        \times
        \cdots 
        \times 
        \BigASpace{s-L\lambda}{p}{q}[\Compact\cap\BoundaryN][\RestrictFilteredSheaf{\LieFilteredSheafF}{\BoundaryNncF}].
    \end{equation}
    Thus, if we have established Theorem \ref{Thm::Trace::Dirichlet::MainInverseThm} for \(\TraceInverseMap[\sN_L]\),
    then setting \(\TraceInverseMap[\sB]:=\TraceInverseMap[\sN_L]M^{-1}\) satisfies the conclusions
    of Theorem \ref{Thm::Trace::Dirichlet::MainInverseThm} \ref{Item::Trace::Dirichlet::MainInverseThm::ContinuousOnSmoothFuncs}
    and \ref{Item::Trace::Dirichlet::MainInverseThm::ContinuousOnSpaces}. Furthermore,
    we have
    \begin{equation*}
        \TraceMap[\sB] \TraceInverseMap[\sB]= M\TraceMap[\sN_L] \TraceInverseMap[\sN_L]M^{-1}  = MM^{-1}=I,
    \end{equation*}
    establishing Theorem \ref{Thm::Trace::Dirichlet::MainInverseThm} \ref{Item::Trace::Dirichlet::MainInverseThm::IsInverse}
    for \(\TraceInverseMap[\sB]\).

    Since \(\TraceMap[\sB]=M\TraceMap[\sN_L]\), we also have \(\TraceMap[\sN_L]=M^{-1}\TraceMap[\sB]\).
    Thus, \(\TraceMap[\sB]f=0\) if and only if \(\TraceMap[\sN_L]f=0\).
    This shows that if Theorem \ref{Thm::Trace::Vanish::MainVanishThm} holds with \(\sB\) replaced by \(\sN_L\),
    it holds for \(\sB\) as well.
\end{proof}

We turn to the proof of Proposition \ref{Prop::Trace::Reduction::TurnBIntoN}. To do so we introduce
two definitions which are only used in the proof of the proposition; since these definitions
are only used here, we leave implicit the dependence on things like \(\FilteredSheafF\).

\begin{definition}
    For \(X\in \LieFilteredSheafF[d][\Omega]\) with \(X(x)\in \TangentSpace{x}{\BoundaryN}\), \(\forall x\in \BoundaryN\cap \Omega\),
    we say \(X\) is a vector field of order \(d\) on \(\Omega\) tangent to the boundary.
\end{definition}

\begin{definition}\label{Defn::Trace::Reduction::TangentPDO}
    We say \(P\) is a differential operator or oder \(\leq \kappa\) on \(\Omega\) tangent to the boundary
    if
    \begin{equation}\label{Eqn::Trace::Reduction::TangentPDO::Formula}
        P=\sum_{j=0}^M b_j(x) X_{j,1}X_{j,2}\cdots X_{j,K_j}
    \end{equation}
    where \(b_j\in \CinftySpace[\Omega]\), and \(X_{j,k}\) is a vector field of order \(\Xdv_{j,k}\) on \(\Omega\)
    tangent to the boundary, and \(\Xdv_{j,1}+\cdots+\Xdv_{j,K_j}\leq \kappa\), \(\forall j\).
\end{definition}

\begin{lemma}\label{Lemma::Trace::Reduction::RestrictTangentVf}
    If \(X\) is a vector field of order \(d\) on \(\Omega\) tangent to the boundary,
    then \(X\big|_{\BoundaryN\cap \Omega}\in \RestrictFilteredSheaf{\LieFilteredSheafFNoSet[d]}{\BoundaryNncF}[\Omega\cap \BoundaryN]\)
    and \(Xf\big|_{\BoundaryN\cap \Omega} = X\big|_{\BoundaryN\cap \Omega} f\big|_{\BoundaryN\cap \Omega}\),
    \(\forall f\in \CinftySpace[\Omega]\).
\end{lemma}
\begin{proof}
    This follows immediately from the definitions.
\end{proof}

\begin{lemma}\label{Lemma::Trace::Reduction::RestrictTangentOp}
    Suppose \(P\) is a differential operator of order \(\leq \kappa\) on \(\Omega\) tangent to the boundary.
    Then, for \(f\in \CinftySpace[\Omega]\),
    \begin{equation*}
        \left( P f \right)\big|_{\BoundaryN\cap \Omega} = P\big|_{\BoundaryN\cap \Omega} f\big|_{\BoundaryN\cap \Omega},
    \end{equation*}
    where \(P\big|_{\BoundaryN\cap \Omega}\) is a partial differential operator on \(\BoundaryN\cap \Omega\)
    with \(\RestrictFilteredSheaf{\LieFilteredSheafF}{\BoundaryNncF}\) degree \(\leq \kappa\) on 
    \(\Omega\cap\BoundaryN\) (see Definition \ref{Defn::Filtrations::DiffOps::Deg}).
\end{lemma}
\begin{proof}
    With \(P\) given by \eqref{Eqn::Trace::Reduction::TangentPDO::Formula}, we have by 
    Lemma \ref{Lemma::Trace::Reduction::RestrictTangentVf},
    \begin{equation*}
        P\big|_{\BoundaryN\cap \Omega}= \sum_{j=1}^M b_j\big|_{\BoundaryN\cap \Omega} X_{j,1}\big|_{\BoundaryN\cap \Omega}\cdots X_{j,K_j}\big|_{\BoundaryN\cap \Omega}.
    \end{equation*}
    From here the result follows from Lemma \ref{Lemma::Trace::Reduction::RestrictTangentVf} and the definitions
    (see Definition \ref{Defn::Filtrations::DiffOps::Deg}).
\end{proof}

\begin{lemma}\label{Lemma::Trace::Reduction::NormalOpSpecialForm}
    Let \(B\) be an \(\FilteredSheafF\)-normal partial differential operator of degree \(\kappa\) on \(\Omega\)
    (see Definition \ref{Defn::Trace::Dirichlet::NormalOperator}). Then, \(B\)
    can be written in the from
    \begin{equation*}
        B=a(x) V^{\kappa/\lambda} + \sum_{r=0}^{\kappa/\lambda-1} \opP_j V^r,
    \end{equation*}
    where 
    \(a(x)\ne 0\), \(\forall x\in \BoundaryN\cap \Omega\) and 
    \(\opP_j\) is a differential operator of order \(\leq \kappa-r\lambda\) on \(\Omega\)
    tangent to the boundary.
\end{lemma}

\begin{remark}
    In short, Lemma \ref{Lemma::Trace::Reduction::NormalOpSpecialForm} says that in Definition \ref{Defn::Trace::Dirichlet::NormalOperator},
    one may put all the \(V\)'s on the right-hand side.
\end{remark}

\begin{proof}[Proof of Lemma \ref{Lemma::Trace::Reduction::NormalOpSpecialForm}]
    Fix \(\kappa_0\) and \(m\) and
    let \(P\) be a partial differential operator of the form
    \(P=b(x)Y_1\cdots Y_K\) where
    \begin{itemize}
        \item \(b\in \CinftySpace[\Omega]\),
        \item \(Y_k\in \LieFilteredSheafF[d_k][\Omega]\), with \(d_1+\cdots+d_K\leq \kappa_0\).
        \item \(\forall k\), either \(Y_k(x)\in \TangentSpace{x}{\BoundaryN}\), \(\forall x\in \BoundaryN\cap \Omega\)
            or \(Y_k=V\),
        \item \(\# \left\{ k : Y_k=V \right\}=m\); note that this implies \(m\lambda\leq \kappa_0\).
    \end{itemize}
    We claim, that in this case, we may write
    \begin{equation*}
        P=\sum_{r=0}^{m} \opP_j V^r,
    \end{equation*}
    where \(\opP_j\) is a differential operator of order \(\leq \kappa_0-r\lambda\) on \(\Omega\) tangent to the boundary.
    Once this claim is proved, the result follows from the definitions.

    We prove the claim by induction on \(K\). The cases \(K=0,1\) are trivial.
    For \(K=2\), the only non-trivial case is \(P=b VY\), where \(Y\in \LieFilteredSheafF[\Omega][d]\),
    \(d+\lambda\leq \kappa_0\), and \(Y(x)\in \TangentSpace{x}{\BoundaryN}\), \(\forall x\in \BoundaryN\cap \Omega\).
    We have
    \begin{equation*}
        b VY= b YV + b [V,Y].
    \end{equation*}
    The first term is of the desired form, so we consider the second. We have, by definition,
    \([V,Y]\in \LieFilteredSheafF[\Omega][d+\lambda]\). Take \(c(x)\in \CinftySpace[\Omega]\)
    to be any function so that \([V,Y](x)-c(x)V(x)\in \TangentSpace{x}{\BoundaryN}\), \(\forall x\in \BoundaryN\cap \Omega\).
    Note that \([V,Y]-cV\in \LieFilteredSheafF[\Omega][d+\lambda]\)
    since \(V\in \LieFilteredSheafF[\Omega][\lambda]\).  We conclude
    \begin{equation*}
        b[V,Y] = b\left( [V,Y]-cV \right) + bcV
    \end{equation*}
    is of the desired form.

    The claim for \(K\geq 3\) follows from the case \(K=2\) and a simple induction, which we leave to the reader.
\end{proof}

\begin{lemma}\label{Lemma::Trace::Reduction::MatrixAppledToNEqualsB}
    Let \(\sB\) and \(\sN_L\) the \(\FilteredSheafF\)-Dirichlet systems on \(\Omega\)
    as in \eqref{Eqn::Trace::Reduction::DirichletSystemAsColumn}.
    There exists an \((L+1)\times (L+1)\) matrix, \(\sM\), with components
    \begin{equation*}
        \sM_j^k
        =
        \begin{cases}
            a_j(x)\in \CinftySpace[\Omega], & j=k,\\
            \opP_j^k, & k<j,\\
            0, k>j,
        \end{cases}
    \end{equation*}
    where \(\opP_j^k\) is a partial differential operator of degree \(\leq j\lambda-k\lambda\) tangent
    to the boundary, and \(a_j(x)\ne 0\), \(\forall x\in \Omega\cap \BoundaryN\)
    and such that \(\sM \sN_L = \sB\).
\end{lemma}
\begin{proof}
    This follows by applying Lemma \ref{Lemma::Trace::Reduction::NormalOpSpecialForm} to each component of \(\sB\).
\end{proof}

\begin{proof}[Proof of Propositon \ref{Prop::Trace::Reduction::TurnBIntoN}]
    Let \(\sM\) be as in Lemma \ref{Lemma::Trace::Reduction::MatrixAppledToNEqualsB} and
    set \(M=\sM\big|_{\BoundaryN\cap \Omega}\); i.e.,
    \begin{equation*}
        M_j^k=
        \begin{cases}
            a_j\big|_{\BoundaryN\cap \Omega}, & j=k,\\
            \opP_j^k\big|_{\BoundaryN\cap \Omega}, & k<j,\\
            0, & k>j,
        \end{cases}
    \end{equation*}
    where \(a_j\) and \(\opP_j^k\) are as in Lemma \ref{Lemma::Trace::Reduction::MatrixAppledToNEqualsB}
    and \(\opP_j^k\big|_{\BoundaryN\cap \Omega}\) is defined as in Lemma \ref{Lemma::Trace::Reduction::RestrictTangentOp}.
    Then, \(M\) is of the desired form by Lemma \ref{Lemma::Trace::Reduction::RestrictTangentOp} and we have
    for \(f\in \CinftySpace[\Omega]\),
    \begin{equation*}
        \TraceMap[\sB] f = \sB f\big|_{\BoundaryN}
        =\left( \sM \sN_L f \right)\big|_{\BoundaryN}
        =  M \left( \sN_L f \right)\big|_{\BoundaryN}
        =M \TraceMap[\sN_L] f.
    \end{equation*}
    The uniqueness of \(\TraceMap[\sB]\) described in Theorem \ref{Thm::Trace::ForwardMap}
    (along with the fact that \(M\) is bounded on \eqref{Eqn::Trace::Reduction::MBoundedOnSpace})
    proves \(\TraceMap[\sB]=\TraceMap[\sN_L]\) on their domain.
\end{proof}

\begin{proof}[Proof of Theorem \ref{Thm::Trace::Dirichlet::MainInverseThm}]
    As shown above, we may assume \(\sB=\sN_L\).  Let \(\Omega_1\), \(\Omega_2\), and \(\Omega\) be
    as in the statement of the theorem.
    For each \(x\in \overline{\Omega_1}\cap \BoundaryN\Subset \BoundaryNncF[\lambda]\), 
    Proposition \ref{Prop::Trace::Reduction::InverseMap} gives 
    neighborhoods \(\Omega_1^x\Subset\Omega_2^x\Subset \Omega^x \Subset \ManifoldNncF[\lambda]\)
    of \(x\)
    and a map
    \begin{equation*}
        \TraceInverseMap[\sN_L]^x : \Distributions[\Omega_2^x \cap \BoundaryN]^{L+1}
        \rightarrow \DistributionsZeroN
    \end{equation*}
    satisfying the conclusions of Proposition \ref{Prop::Trace::Reduction::InverseMap}
    % Theorem \ref{Thm::Trace::Dirichlet::MainInverseThm}
    (with \(\Omega_j=\Omega_j^x\) and \(\Omega=\Omega^x\)).

    \(\left\{ \Omega_1^x \cap \Omega_2 : x\in \overline{\Omega_1}\cap \BoundaryN \right\}\)
    is an open cover for the compact set \(\overline{\Omega_1}\cap \BoundaryN\); let
    \(\Omega_1^{x_1}\cap \Omega_2, \ldots, \Omega_1^{x_K}\cap \Omega_2\)
    be a finite subcover.
    Pick \(\phi_k\in \CinftyCptSpace[\Omega_1^{x_k}\cap \Omega_2\cap \BoundaryN]\) such that
    \(\sum_{k=1}^K\phi_k=1\) on a neighborhood of \(\overline{\Omega_1}\cap \BoundaryN\)
    and fix \(\psi\in \CinftyCptSpace[\Omega_2]\) with \(\psi=1\) on a neighborhood \(U\Subset \Omega_2\) of \(\overline{\Omega_1}\).
    For \(u\in \Distributions[\Omega_2\cap \BoundaryN]^{L+1}\), set
    \(\TraceInverseMap[\sN_L]u=\psi \sum_{k=1}^K \TraceInverseMap[\sN_L]^{x_k} \phi_k u\).
    The mapping property in \ref{Item::Trace::Dirichlet::MainInverseThm::ContinuousOnSmoothFuncs}
    follows from Proposition \ref{Prop::Trace::Reduction::InverseMap} \ref{Item::Trace::Reduction::Dirichlet::MainInverseThm::ContinuousOnSmoothFuncs}.
    Using Proposition \ref{Prop::Trace::Reduction::InverseMap} \ref{Item::Trace::Reduction::Dirichlet::MainInverseThm::ContinuousOnSpaces},
    Proposition \ref{Prop::Spaces::MappingOfFuncsAndVFs} \ref{Item::Spaces::MappingOfFuncsAndVFs::Funcs},
    and Proposition \ref{Prop::Spaces::Containment}, we see
    \(\TraceInverseMap[\sN_L]\) is continuous 
                \begin{equation}\label{Eqn::Trace::Reduction::ToLocal::ContinuousOnBesov}
                \TraceInverseMap[\sN_L]:
                \prod_{l=0}^L
                \BigBesovSpace{s-l\lambda-\lambda/p}{p}{q}[\overline{U}\cap \BoundaryN][\RestrictFilteredSheaf{\LieFilteredSheafF}{\BoundaryNncF}]
                \rightarrow
                \BesovSpace{s}{p}{q}[\overline{\Omega_2}][\FilteredSheafF],\quad 1\leq p\leq \infty, \: 1\leq q\leq\infty, \: s\in \R,
            \end{equation}
            \begin{equation*}
                \TraceInverseMap[\sN_L]:
                \prod_{l=0}^L
                \BigBesovSpace{s-l\lambda-\lambda/p}{p}{p}[\overline{U}\cap \BoundaryN][\RestrictFilteredSheaf{\LieFilteredSheafF}{\BoundaryNncF}]
                \rightarrow
                \TLSpace{s}{p}{q}[\overline{\Omega_2}][\FilteredSheafF],\quad 1< p< \infty, \: 1< q\leq\infty, \: s\in \R.
            \end{equation*}
    Proposition \ref{Prop::Spaces::Containment} then gives \ref{Item::Trace::Dirichlet::MainInverseThm::ContinuousOnSpaces}
    (since \(\overline{\Omega_1}\subseteq \overline{U}\)).

    % The mapping properties in \ref{Item::Trace::Dirichlet::MainInverseThm::ContinuousOnSmoothFuncs}
    % and \ref{Item::Trace::Dirichlet::MainInverseThm::ContinuousOnSpaces} follow from those
    % for \(\TraceInverseMap[\sN_L]^{x_j}\), Proposition \ref{Prop::Spaces::MappingOfFuncsAndVFs} \ref{Item::Spaces::MappingOfFuncsAndVFs::Funcs},
    % and Proposition \ref{Prop::Spaces::Containment}.

    We turn to \ref{Item::Trace::Dirichlet::MainInverseThm::IsInverse}.
    % Fix \(U\Subset \Omega_2\) open with \(\overline{\Omega_1}\Subset U\)
    % and \(\psi=1\) on a neighborhood of \(\overline{U}\).
    For \(u\in \CinftyCptSpace[U\cap \BoundaryN]^{L+1}\), we have
    by  Proposition \ref{Prop::Trace::Reduction::InverseMap} \ref{Item::Trace::Reduction::Dirichlet::MainInverseThm::IsInverse}
    \begin{equation}\label{Eqn::Trace::Reduction::ToLocal::TraceTraceInverseIsI}
    \begin{split}
         &\TraceMap[\sN_L] \TraceInverseMap[\sN_L] u
         =\sum_{k=1}^L \TraceMap[\sN_L] \psi \TraceInverseMap[\sN_L]^{x_k} \phi_k u
         =\sum_{k=1}^L \psi\big|_{\BoundaryN}\TraceMap[\sN_L]  \TraceInverseMap[\sN_L]^{x_k} \phi_k u
         =\sum_{k=1}^L \psi\big|_{\BoundaryN} \phi_k u = u.
    \end{split}
    \end{equation}
    % where the final equality used \(\supp(u)\subseteq \overline{\Omega_1}\cap\BoundaryN\) (see Definition \ref{Defn::Spaces::Defns::ASpace}).
    Next, suppose \(u\in \prod_{l=0}^L
                \BigBesovSpace{s-l\lambda-\lambda/p}{p}{q}[\overline{\Omega_1}\cap \BoundaryN][\RestrictFilteredSheaf{\LieFilteredSheafF}{\BoundaryNncF}]\),
                where \(s>L\lambda+\lambda/p\) and \(p>1\).
    If \(q<\infty\),
    Corollary \ref{Cor::Spaces::Approximation::SmoothFunctionsAreDense} shows that there exists
    \(u_j \in \CinftyCptSpace[U\cap \BoundaryN]^{L+1}\) with \(u_j\rightarrow u\)
    in \(\BigBesovSpace{s-l\lambda-\lambda/p}{p}{q}[\overline{U}\cap \BoundaryN][\RestrictFilteredSheaf{\LieFilteredSheafF}{\BoundaryNncF}]\).
    Since \(\TraceMap[\sN_L]\TraceInverseMap[\sN_L]u_j=u_j\) (by \eqref{Eqn::Trace::Reduction::ToLocal::TraceTraceInverseIsI}),
    the continuity established in \eqref{Eqn::Trace::Reduction::ToLocal::ContinuousOnBesov}
    and Theorem \ref{Thm::Trace::ForwardMap}
    shows \(\TraceMap[\sN_L]\TraceInverseMap[\sN_L]u=u\).
    If \(q=\infty\), take \(s'\in (L\lambda+\lambda/p,s)\) and use the fact that
    \begin{equation*}
    \begin{split}
         &u\in \prod_{l=0}^L
                \BigBesovSpace{s-l\lambda-\lambda/p}{p}{\infty}[\overline{\Omega_1}\cap \BoundaryN][\RestrictFilteredSheaf{\LieFilteredSheafF}{\BoundaryNncF}]
        \subseteq\prod_{l=0}^L\BigBesovSpace{s'-l\lambda-\lambda/p}{p}{2}[\overline{\Omega_1}\cap \BoundaryN][\RestrictFilteredSheaf{\LieFilteredSheafF}{\BoundaryNncF}],
    \end{split}
    \end{equation*}
    to see that \(\TraceMap[\sN_L]\TraceInverseMap[\sN_L]u=u\) by the already proved case \(q<\infty\).
    This establishes \ref{Item::Trace::Dirichlet::MainInverseThm::IsInverse} and completes the proof.
    % If \(u\in \prod_{l=0}^L
    %             \BigBesovSpace{s-l\lambda-\lambda/p}{p}{q}[\overline{\Omega_1}\cap \BoundaryN][\RestrictFilteredSheaf{\LieFilteredSheafF}{\BoundaryNncF}]\),
    %             where \(s>L\lambda+\lambda/p\),
    % then
    % \begin{equation*}
    % \begin{split}
    %      &\TraceMap[\sN_L] \TraceInverseMap[\sN_L] u
    %      =\sum_{k=1}^L \TraceMap[\sN_L] \psi \TraceInverseMap[\sN_L]^{x_k} \phi_k u
    %      =\sum_{k=1}^L \psi\big|_{\BoundaryN}\TraceMap[\sN_L]  \TraceInverseMap[\sN_L]^{x_k} \phi_k u
    %      =\sum_{k=1}^L \psi\big|_{\BoundaryN} \phi_k u = u,
    % \end{split}
    % \end{equation*}
    % where the final equality used \(\supp(u)\subseteq \overline{\Omega_1}\cap\BoundaryN\) (see Definition \ref{Defn::Spaces::Defns::ASpace}).
\end{proof}

\begin{proof}[Proof of Theorem \ref{Thm::Trace::Vanish::MainVanishThm}]
    The proofs of \ref{Item::Trace::Vanish::MainVanishThm::Smalls} and \ref{Item::Trace::Vanish::MainVanishThm::Larges}
    \ref{Item::Trace::Vanish::MainVanishThm::Larges::Trace0}
    \(\Rightarrow\)
    \ref{Item::Trace::Vanish::MainVanishThm::Larges::SmoothApprox}
     are similar. We prove \ref{Item::Trace::Vanish::MainVanishThm::Larges}; the proof for \ref{Item::Trace::Vanish::MainVanishThm::Smalls}
     follows in the same way, by ignoring the condition \(\TraceMap[\sN_L]u=0\) throughout, and we leave the necessary
     simple modifications to the reader.
    Suppose \(s\in (L\lambda+\lambda/p, (L+1)\lambda+\lambda/p)\) and \(p,q\in (1,\infty)\) as in \ref{Item::Trace::Vanish::MainVanishThm::Larges}.
    We have already shown that it suffices to prove the result with \(\sB=\sN_L\) (see the proof following
    Proposition \ref{Prop::Trace::Reduction::TurnBIntoN}).

    \ref{Item::Trace::Vanish::MainVanishThm::Larges::Trace0}
    \(\Rightarrow\)
    \ref{Item::Trace::Vanish::MainVanishThm::Larges::SmoothApprox}:
    For each \(x\in \Compact\cap \BoundaryN\subseteq \BoundaryNncF[\lambda]\),
    Proposition \ref{Prop::Trace::Reduction::VanishingChar}
    gives \(\Omega_1^x\Subset \Omega^x\Subset \ManifoldNncF[\lambda]\), neighborhoods of \(x\),
    such that the result holds with \(\Omega=\Omega^x\), \(\Compact=\overline{\Omega_1^x}\).

    \(\left\{ \Omega_1^x  : x\in \Compact\cap \BoundaryN \right\}\) is an open cover for 
    the compact set \(\Compact\cap\BoundaryN\). Let \(\Omega_1^{x_1},\ldots, \Omega_x^{x_K}\)
    be a finite subcover.
    Take \(\phi_k\in \CinftyCptSpace[\Omega_1^{x_k}]\) such that \(\sum_{k=1}^K \phi_k =1\)
    on an \(\ManifoldN\)-neighborhoood of \(\Compact\cap \BoundaryN\).
    Let \(\phi_0\in \CinftyCptSpace[\Omega\cap\InteriorN]\) be such that \(\sum_{k=0}^K \phi_k =1\)
    on an \(\ManifoldN\)-neighborhood of \(\Compact\).

    Let \(\Compact_0:=\supp(\phi_0)\) and \(\Compacth=\Compact_0\bigcup\left( \bigcup_{k=1}^K \overline{\Omega_1}^{x_k} \right)\).
    For \(u\in \ASpace{s}{p}{q}[\Compact][\FilteredSheafF]\subseteq \ASpace{s}{p}{q}[\Compacth][\FilteredSheafF]\)
    (see Proposition \ref{Prop::Spaces::Containment}), 
    Proposition \ref{Prop::Spaces::MappingOfFuncsAndVFs} \ref{Item::Spaces::MappingOfFuncsAndVFs::Funcs}
    shows \(\phi_j u \in \ASpace{s}{p}{q}[\Compacth][\FilteredSheafF]\).
    For \(k\geq 1\), Proposition \ref{Prop::Spaces::Containment} then implies
    \(\phi_k u \in \ASpace{s}{p}{q}[\overline{\Omega_1^{x_k}}][\FilteredSheafF]\)
    and \(\phi_0 u\in \ASpace{s}{p}{q}[\Compact_0][\FilteredSheafF]\).

    Let \(\Omega_0\Subset \Omega\cap \InteriorN\) be a relatively compact, open set with \(\Compact_0\Subset \Omega_0\).
    Since \(q<\infty\), Corollary \ref{Cor::Spaces::Approximation::SmoothFunctionsAreDense}
    \ref{Item::Spaces::Approximation::SmoothFunctionsAreDense::ApproxInST}
    shows that there exist \(f_r^0\in \CinftyCptSpace[\Omega_0]\)
    with \(f_r^0\xrightarrow{r\rightarrow \infty} \phi_0 u\) in \(\ASpace{s}{p}{q}[\overline{\Omega_0}][\FilteredSheafF]\).
    Set \(\Compactt:=\overline{\Omega}\cup \bigcup_{k=1}^k \overline{\Omega^{x_k}}\)
    and therefore by Proposition \ref{Prop::Spaces::Containment},
    \(f_r^0\xrightarrow{r\rightarrow \infty} \phi_0 u\) in \(\ASpace{s}{p}{q}[\Compactt][\FilteredSheafF]\).

    Suppose \(\TraceMap[\sN_L ]u=0\).  Since \(\sN_L \phi_k u = \left(\phi_k u, V\phi_k u,\ldots, V^L \phi_k u  \right)\),
    the product rule shows that \(\sN_L \phi_k =0 \), for \(1\leq k\leq K\).
    Proposition \ref{Prop::Trace::Reduction::VanishingChar}
    \ref{Item::Trace::Reduction::VanishingChar::Larges}
    shows that for \(1\leq k\leq K\), there exist
    \(f_r^k\in \CinftyCptSpace[\Omega^{x_k}\cap \InteriorN]\) such that
    \(f_r^k\xrightarrow{r\rightarrow \infty} \phi_k u\) in \(\ASpace{s}{p}{q}[\overline{\Omega^{x_k}}][\FilteredSheafF]\),
    and therefore by Proposition \ref{Prop::Spaces::Containment},
    \(f_r^k\xrightarrow{r\rightarrow \infty} \phi_k u\) in \(\ASpace{s}{p}{q}[\Compactt][\FilteredSheafF]\).
    
    Fix \(\psi\in \CinftyCptSpace[\Omega]\) with \(\psi=1\) on a neighborhood of \(\Compact\).
    Set \(f_r:=\psi \sum_{k=0}^K f_r^k\in \CinftyCptSpace[\Omega\cap \InteriorN]\). 
    Proposition \ref{Prop::Spaces::MappingOfFuncsAndVFs} \ref{Item::Spaces::MappingOfFuncsAndVFs::Funcs} shows
    \(f_r\xrightarrow{r\rightarrow \infty} \psi \sum_{k=0}^K \phi_k u =u \)
    in \(\ASpace{s}{p}{q}[\Compacth][\FilteredSheafF]\), where the final equality used
    \(\supp(u)\subseteq \Compact\), since \(u\in \ASpace{s}{p}{q}[\Compact][\FilteredSheafF]\)
    (see Definition \ref{Defn::Spaces::Defns::ASpace}).
    Since \(\supp(f_r)\subseteq \Omega\), \(\forall r\), Proposition \ref{Prop::Spaces::Containment}
    shows
\(f_r\xrightarrow{r\rightarrow \infty}  =u \)
    in \(\ASpace{s}{p}{q}[\overline{\Omega}][\FilteredSheafF]\), completing the proof of \ref{Item::Trace::Vanish::MainVanishThm::Larges::SmoothApprox}.

    \ref{Item::Trace::Vanish::MainVanishThm::Larges::SmoothApprox}
    \(\Rightarrow\)
    \ref{Item::Trace::Vanish::MainVanishThm::Larges::Trace0}:
    Suppose \( f_j\in \CinftyCptSpace[\Omega\cap \InteriorN] \) is such that \(f_j\rightarrow u\)
    in \(\ASpace{s}{p}{q}[\overline{\Omega}][\FilteredSheafF]\).
    Clearly \(\TraceMap[\sN_L] f_j=0\), \(\forall j\), and the continuity of
    \(\TraceMap[\sN_L]\) (as in Theorem \ref{Thm::Trace::ForwardMap}) shows \(\TraceMap[\sN_L] u=0\).
\end{proof}

    \subsection{Reduction to a convenient coordinate system}\label{Section::Trace::ReductionToCoord}
    In Section \ref{Section::Trace::ReductionToLocal}, we reduced
Theorems \ref{Thm::Trace::ForwardMap}, \ref{Thm::Trace::Dirichlet::MainInverseThm}, and \ref{Thm::Trace::Vanish::MainVanishThm}
to simpler local propositions. In this section, we further reduce these results to a convenient coordinate system.
We being by stating the results in \(\R^n\), and then in Subsection \ref{Section::Trace::ReductionToCoord::Proof}
we show how these local results imply Propositions \ref{Prop::Trace::Reduction::ForwardMap},
\ref{Prop::Trace::Reduction::InverseMap},
and \ref{Prop::Trace::Reduction::VanishingChar}
(and therefore imply Theorems \ref{Thm::Trace::ForwardMap}, \ref{Thm::Trace::Dirichlet::MainInverseThm}, and \ref{Thm::Trace::Vanish::MainVanishThm}).

We work on \(\R^n\) with coordinates \(x=(x',x_n)\in \R^{n-1}\times \R\cong \R^{n+1}\).
Let \(\ManifoldN=\Bngeq{1}=\left\{ x=(x',x_n)\in\R^{n}: |x|<1, x_n\geq 0  \right\}\),
with \(\InteriorN=\Bng{1}=\left\{ (x',x_n)\in \Bngeq{1} : x_n>0 \right\}\)
and \(\BoundaryN=\left\{ (x',0) : |x'|<1 \right\}\cong \Bnmo{1}\).
We take \(\Vol\) to be Lebesgue measure on \(\Bngeq{1}\).

We suppose we are given smooth vector fields with formal degrees
\begin{equation*}
    \XXdv=\left\{ \left( X_1,\Xdv_1 \right),\ldots, \left( X_{q+1},\Xdv_{q+1} \right) \right\}
    \subset \VectorFields{\Bngeq{1}}\times \Zg
\end{equation*}
satisfying:
\begin{enumerate}[(A)]
    \item\label{Item::Traces::CoordReduction::IsPartialxn} \(X_{q+1}=\partial_{x_n}\). Set \(\lambda:=\Xdv_{q+1}\).
    \item\label{Item::Traces::CoordReduction::TangentToBoundary} \(X_j(x',0)\) has no \(\partial_{x_n}\) component for \(1\leq j\leq q\), \(\forall x\in \Bnmo{1}\).
    \item\label{Item::Traces::CoordReduction::NSW} \([X_j, X_k]=\sum_{\Xdv_l\leq \Xdv_j+\Xdv_k}c_{j,k}^l X_l\), \(c_{j,k}^l\in \CinftyCptSpace[\Bngeq{1}]\).
    \item\label{Item::Traces::CoordReduction::Span} \(\Span\left\{ X_1(x),\ldots, X_{q+1}(x) \right\}=\TangentSpace{x}{\R^n}\), \(\forall x\in \Bngeq{1}\).
\end{enumerate}
Using \ref{Item::Traces::CoordReduction::Span},
it is easy to see that \(\FilteredSheafGenByXXdv\) is a H\"ormander filtration of sheaves of vector fields on \(\Bngeq{1}\),
and using \ref{Item::Traces::CoordReduction::NSW}, we see
\(\LieFilteredSheaf{\FilteredSheafGenByXXdv}=\FilteredSheafGenByXXdv\).
Using \ref{Item::Traces::CoordReduction::IsPartialxn} and \ref{Item::Traces::CoordReduction::TangentToBoundary},
we see \(\degBoundaryBngeq{\FilteredSheafGenByXXdv}[x']=\lambda\), \(\forall x'\in \Bnmo{1}\),
and so every point of \(\partial \Bngeq{1}\) is \(\FilteredSheafGenByXXdv\)-non-characteristic with degree 
\(\lambda\).

For \(1\leq j\leq q\), set \(V_j:=X_j\big|_{\Bnmo{1}}\). Note that \ref{Item::Traces::CoordReduction::TangentToBoundary}
implies \(V_j\in \VectorFields{\Bnmo{1}}\). 
Set \(\VVdv:=\left\{ \left( V_1,\Vdv_1 \right),\ldots, \left( V_q,\Vdv_1 \right) \right\}\).
It follows from the definitions that
\begin{equation*}
    \RestrictFilteredSheaf{\LieFilteredSheaf{\FilteredSheafGenByXXdv}}{\Bnmo{1}}=\RestrictFilteredSheaf{\FilteredSheafGenByXXdv}{\Bnmo{1}}=\FilteredSheafGenByVVdv.
\end{equation*}

\begin{proposition}\label{Prop::Traces::CoordReduction::ForwardMap}
    For all \(1<p\leq \infty\), \(1\leq q\leq \infty\), \(s>\lambda/p\), \(\exists C\geq 0\), \(\forall f\in \CinftyCptSpace[\Bngeq{1/2}]\),
    \begin{equation}\label{Eqn::Traces::CoordReduction::ForwardMap::Besov}
        \BesovNorm{f\big|_{\Bnmo{1}}}{s-\lambda/p}{p}{q}[\FilteredSheafGenByVVdv]
        \leq C
        \BesovNorm{f}{s}{p}{q}[\FilteredSheafGenByXXdv].
    \end{equation}
    For all \(1<p< \infty\), \(1< q\leq \infty\), \(s>\lambda/p\), \(\exists C\geq 0\), \(\forall f\in \CinftyCptSpace[\Bngeq{1/2}]\),
    \begin{equation}\label{Eqn::Traces::CoordReduction::ForwardMap::TL}
        \BesovNorm{f\big|_{\Bnmo{1}}}{s-\lambda/p}{p}{p}[\FilteredSheafGenByVVdv]
        \leq C
        \TLNorm{f}{s}{p}{q}[\FilteredSheafGenByXXdv].
    \end{equation}
\end{proposition}
Proposition \ref{Prop::Traces::CoordReduction::ForwardMap} is proved in Section \ref{Section::Trace::Forward}.

For the next two results, we use the \(\FilteredSheafGenByXXdv\)-Dirichlet system
\(\sN_L:=\left( 1, \partial_{x_n},\partial_{x_n}^2,\ldots, \partial_{x_n}^L \right)\).

\begin{proposition}\label{Prop::Traces::CoordReduction::InverseMap}
    There is a continuous, linear map
    \begin{equation}\label{Eqn::Traces::CoordReduction::InverseMap::BasicMapping}
        \TraceInverseMap[\sN_L]:\Distributions[\Bngeq{1/2}]^{L+1}\rightarrow \DistributionsZero[\Bngeq{1}]
    \end{equation}
    satisfying
    \begin{enumerate}[(i)]
        \item\label{Item::Traces::CoordReduction::InverseMap::MappingOnSmoothFunctions} \(\TraceInverseMap[\sN_L]: \CinftySpace[\Bnmo{1/2}]^{L+1}\rightarrow \CinftyCptSpace[\Bngeq{3/4}]\), continuously.
        \item\label{Item::Traces::CoordReduction::InverseMap::MappingOnBesovAndTL} \(\TraceInverseMap[\sN_L]\) is continuous:
                    \begin{equation*}
                \TraceInverseMap[\sN_L]:
                \prod_{l=0}^L
                \BigBesovSpace{s-l\lambda-\lambda/p}{p}{q}[\BnmoClosure{1/4}][\FilteredSheafGenByVVdv]
                \rightarrow
                \BesovSpace{s}{p}{q}[\BngeqClosure{3/4}][\FilteredSheafGenByXXdv],\quad 1\leq p\leq \infty, \: 1\leq q\leq\infty, \: s\in \R,
            \end{equation*}
            \begin{equation*}
                \TraceInverseMap[\sN_L]:
                \prod_{l=0}^L
                \BigBesovSpace{s-l\lambda-\lambda/p}{p}{p}[\BnmoClosure{1/4}][\FilteredSheafGenByVVdv]
                \rightarrow
                \TLSpace{s}{p}{q}[\BngeqClosure{3/4}][\FilteredSheafGenByXXdv],\quad 1< p< \infty, \: 1< q\leq\infty, \: s\in \R.
            \end{equation*}
        \item\label{Item::Traces::CoordReduction::InverseMap::IsInverse} \(\TraceMap[\sN_L]\TraceInverseMap[\sN_L]f=f\), \(\forall f\in \CinftyCptSpace[\Bnmo{1/4}]^{L+1}\).
    \end{enumerate}
\end{proposition}
Proposition \ref{Prop::Traces::CoordReduction::InverseMap} is proved in Section \ref{Section::Trace::Inverse}.

\begin{proposition}\label{Prop::Traces::CoordReduction::VanishingChar}
    Fix \(p,q\in (1,\infty)\), \(s\in \R\), and \(u\in \ASpace{s}{p}{q}[\BngeqClosure{1/2}][\FilteredSheafGenByXXdv]\).
    \begin{enumerate}[(i)]
        \item\label{Item::Traces::CoordReduction::VanishingChar::smalls} If \(s<\lambda/p\), \(\exists \left\{ f_j \right\}_{j\in \Zgeq}\subset \CinftyCptSpace[\Bng{3/4}]\)
            with \(f_j\rightarrow u\) in \(\ASpace{s}{p}{q}[\BngeqClosure{3/4}][\FilteredSheafGenByXXdv]\).
        \item\label{Item::Traces::CoordReduction::VanishingChar::larges} If \(s\in (L\lambda+\lambda/p, (L+1)\lambda+\lambda/p)\), where \(L\in \Zgeq\), and if
            \(\TraceMap[\sN_L] u=0\), then \(\exists \left\{ f_j \right\}_{j\in \Zgeq}\subset \CinftyCptSpace[\Bng{3/4}]\)
            with \(f_j\rightarrow u\) in \(\ASpace{s}{p}{q}[\BngeqClosure{3/4}][\FilteredSheafGenByXXdv]\).
    \end{enumerate}
\end{proposition}
Proposition \ref{Prop::Traces::CoordReduction::VanishingChar} is proved in Section \ref{Section::Trace::CharacterizeVanish}.
    
        \subsubsection{Proof of the reduction to a convenient coordinate system}\label{Section::Trace::ReductionToCoord::Proof}
        In this subsection, we reduce  Propositions \ref{Prop::Trace::Reduction::ForwardMap},
\ref{Prop::Trace::Reduction::InverseMap},
and \ref{Prop::Trace::Reduction::VanishingChar} to
Propositions \ref{Prop::Traces::CoordReduction::ForwardMap},
\ref{Prop::Traces::CoordReduction::InverseMap},
and \ref{Prop::Traces::CoordReduction::VanishingChar}, respectively;
thereby reducing Theorems \ref{Thm::Trace::ForwardMap}, \ref{Thm::Trace::Dirichlet::MainInverseThm}, and \ref{Thm::Trace::Vanish::MainVanishThm}
to Propositions \ref{Prop::Traces::CoordReduction::ForwardMap},
\ref{Prop::Traces::CoordReduction::InverseMap},
and \ref{Prop::Traces::CoordReduction::VanishingChar}.
The key is the existence of a good coordinate system, which the next proposition gives.

\begin{proposition}\label{Prop::Trace::CoordReduction::Proof::ExistCoordSystem}
    Let \(\ManifoldN\) be a smooth manifold with boundary and let \(\FilteredSheafF\)
    be a H\"ormander filtration of sheaves of vector fields on \(\ManifoldN\).
    Fix \(x_0\in \BoundaryNncF[\lambda]\) and let \(U\) be a neighborhood of \(x_0\)
    and \(Y\in \LieFilteredSheafF[U][\lambda]\) with \(Y(x_0)\not\in \TangentSpace{x_0}{\BoundaryN}\).
    There exist vector fields with formal degrees
    \(\XXdv=\left\{ \left( X_1,\Xdv_1 \right),\ldots, \left( X_{q+1},\Xdv_{q+1} \right) \right\}\subset \VectorFields{\Bngeq{1}}\times \Zg\)
    satisfying the properties \ref{Item::Traces::CoordReduction::IsPartialxn}--\ref{Item::Traces::CoordReduction::Span},
    and a map \(\Phi:\Bngeq{1}\rightarrow \ManifoldN\)
    satisfying:
    \begin{enumerate}[(i)]
        \item\label{Item::Trace::CoordReduction::Proof::ExistCoordSystem::Phi0} \(\Phi(0)=x_0\).
        \item\label{Item::Trace::CoordReduction::Proof::ExistCoordSystem::PhiDiffeo} \(\Phi(\Bngeq{1})\Subset \ManifoldNncF[\lambda]\) is open and \(\Phi:\Bngeq{1}\xrightarrow{\sim} \Phi\left( \Bngeq{1} \right)\)
            is a smooth diffeomorphism.
        \item\label{Item::Trace::CoordReduction::Proof::ExistCoordSystem::PhiMapsBoundarys} \(\Phi(\Bnmo{1})=\Phi\left( \Bngeq{1} \right)\cap \BoundaryN\).
        \item\label{Item::Trace::CoordReduction::Proof::ExistCoordSystem::PhiTakesXToF} Let \(\left( \Phi_{*}X, \Xdv \right)=\left\{ \left( \Phi_{*}X_1,\Xdv_1 \right),\ldots, \left( \Phi_{*}X_{q+1},\Xdv_{q+1} \right) \right\}\).
            Then, \(\LieFilteredSheafF\big|_{\Phi(\Bngeq{1})}=\FilteredSheafGenBy{\left( \Phi_{*}X, \Xdv \right)}\).
        \item\label{Item::Trace::CoordReduction::Proof::ExistCoordSystem::PullBackY} \(\Phi^{*}Y=\sigma \partial_{x_n}\), for some constant \(\sigma\ne 0\).
        \item\label{Item::Trace::CoordReduction::Proof::ExistCoordSystem::PhiTakesVToBoundaryF} Define \(\VVdv\) as above (i.e., \(V_j=X_j\big|_{\Bnmo{1}}\in \VectorFields{\Bnmo{1}}\), \(1\leq j\leq q\)),
            and set
            \(\left( \Phi_{*}V, \Vdv \right)=\left\{ \left( \Phi_{*}V_1, \Vdv_1 \right),\ldots, \left( \Phi_{*}V_{q}, \Vdv_q \right) \right\}\subset \VectorFields{\Phi(\Bnmo{1})}\times \Zg\).
            Then,
            \begin{equation*}
                \RestrictFilteredSheaf{\LieFilteredSheafF}{\Phi(\Bnmo{1})}=\FilteredSheafGenBy{(\Phi_{*}V, \Vdv)}.
            \end{equation*}
        \item\label{Item::Trace::CoordReduction::Proof::ExistCoordSystem::PhiIsIsoOnXSpace} 
            For \(\Compact \Subset \Bngeq{1}\) compact, 
            \(\Phi^{*}: \DistributionsZeroN\rightarrow \DistributionsZero[\Bngeq{1}]\) restricts to an isomorphism:
            \begin{equation*}
                \ASpace{s}{p}{q}[\Phi(\Compact)][\FilteredSheafF] \xrightarrow{\sim} \ASpace{s}{p}{q}[\Compact][\FilteredSheafGenByXXdv].
            \end{equation*}
        \item\label{Item::Trace::CoordReduction::Proof::ExistCoordSystem::PhiIsIsoOnVSpace} For \(\Compact_0\Subset \Bnmo{1}\) compact, \(\Phi\big|_{\Bnmo{1}}^{*}:\Distributions[\BoundaryN]\rightarrow \Distributions[\Bnmo{1}]\)
            restricts to an isomorphism:
            \begin{equation*}
                \BigASpace{s}{p}{q}[\Phi(\Compact_0)][\RestrictFilteredSheaf{\LieFilteredSheafF}{\BoundaryNncF}]
                \xrightarrow{\sim}
                \ASpace{s}{p}{q}[\Compact_0][\FilteredSheafGenByVVdv].
            \end{equation*}
    \end{enumerate}
\end{proposition}

To prove Proposition \ref{Prop::Trace::CoordReduction::Proof::ExistCoordSystem}, we use the next lemma, which
gives convenient generators for \(\LieFilteredSheafF\).

\begin{lemma}\label{Lemma::Trace::CoordReduction::Proof::NiceGenerators}
        Let \(\ManifoldN\) be a smooth manifold with boundary and let \(\FilteredSheafF\)
    be a H\"ormander filtration of sheaves of vector fields on \(\ManifoldN\).
    Fix \(x_0\in \BoundaryNncF[\lambda]\) and let \(U\) be a neighborhood of \(x_0\)
    and \(Y\in \LieFilteredSheafF[U][\lambda]\) with \(Y(x_0)\not\in \TangentSpace{x_0}{\BoundaryN}\).
    There exists a neighborhood \(\Omega\Subset U\) of \(x_0\) and
    \((\Xt,\Xdv)=\left\{ \left( \Xt_1, \Xdv_1 \right),\ldots, \left( \Xt_{q+1},\Xdv_{q+1} \right) \right\}\subset \VectorFields{\Omega}\times \Zg\),
    such that:
    \begin{enumerate}[(i)]
        \item\label{Item::Trace::CoordReduction::Proof::NiceGenerators::Xqplus1} \(\Xt_{q+1}=\pm Y\).
        \item\label{Item::Trace::CoordReduction::Proof::NiceGenerators::Span} \(\forall x\in \Omega\), \(\Span\left\{ \Xt_1(x),\ldots, \Xt_{q+1}(x) \right\}=\TangentSpace{x_0}{\Omega}\).
        \item\label{Item::Trace::CoordReduction::Proof::NiceGenerators::AreGenerators} \(\LieFilteredSheafF\big|_{\Omega}=\FilteredSheafGenBy{(\Xt,\Xdv)}\).
        \item\label{Item::Trace::CoordReduction::Proof::NiceGenerators::NSW} \([\Xt_j,\Xt_k]=\sum_{\Xdv_l\leq \Xdv_j+\Xdv_k} \ct_{j,k}^l \Xt_l\), \(\ct_{j,k}^l\in \CinftySpace[\Omega]\).
        \item\label{Item::Trace::CoordReduction::Proof::NiceGenerators::PointsInside} \(\Xdv_{q+1}=\lambda\), \(\Xt_{q+1}(x)\not \in \TangentSpace{x}{\BoundaryN}\), \(\forall x\in \BoundaryN\),
            and \(\Xt_{q+1}\) points into \(\ManifoldN\), in the sense that if \(\forall \Compact\Subset \Omega\) is compact,
            \(\exists a>0\), \(\forall t\in (0,a)\), we have \(e^{t\Xt_{q+1}}x\in \InteriorN\), for \(x\in \BoundaryN\cap \Compact\).
        \item\label{Item::Trace::CoordReduction::Proof::NiceGenerators::TangentToBdry} \(\Xt_j(x)\in \TangentSpace{x}{\BoundaryN}\), \(\forall 1\leq j\leq q\), \(\forall x\in \Omega\cap\BoundaryN\).
        \item\label{Item::Trace::CoordReduction::Proof::NiceGenerators::GenerateBoundarySheaf} For \(1\leq j\leq q\), set \(\Vt_j:=\Xt_j\big|_{\Omega\cap\BoundaryN}\in \VectorFields{\Omega\cap\BoundaryN}\), and 
            \(\VtVdv:=\left\{ \left( \Vt_1,\Vdv_1\right),\ldots, \left( \Vt_{q},\Vdv_q \right) \right\}\).
            Then, \(\RestrictFilteredSheaf{\LieFilteredSheafF}{\BoundaryN\cap\Omega}=\FilteredSheafGenBy{\VtVdv}\).
    \end{enumerate}
\end{lemma}
\begin{proof}
    This is essentially a consequence of Lemma \ref{Lemma::Spaces::Multiplication::NiceGenerators}; however
    the proof in this setting is simpler, and we present it.
    First, we prove the result for some \(\ZZdv\) in place of \(\XtXdv\), satisfying all the conclusions
    except for \ref{Item::Trace::CoordReduction::Proof::NiceGenerators::Xqplus1}, \ref{Item::Trace::CoordReduction::Proof::NiceGenerators::TangentToBdry}, and 
    \ref{Item::Trace::CoordReduction::Proof::NiceGenerators::GenerateBoundarySheaf}.

    Let \(\Omega_1\Subset U\) be a relatively compact neighborhood of \(x_0\).
    By Lemma \ref{Lemma::Filtrations::GeneratorsForLieFiltration},
    there exists \(\ZZdv=\left\{ \left( Z_1,\Zdv_1 \right),\ldots, \left( Z_{q},\Zdv_{q} \right) \right\}\subset \VectorFields{\Omega_1}\times \Zg\)
    satisfying \ref{Item::Trace::CoordReduction::Proof::NiceGenerators::Span}
    and \ref{Item::Trace::CoordReduction::Proof::NiceGenerators::AreGenerators} (with \(\XtXdv\) replaced by \(\ZZdv\)).
    Since \([Z_j,Z_k]\in \LieFilteredSheafF[\Omega_1][\Zdv_j+\Zdv_k]\), \ref{Item::Trace::CoordReduction::Proof::NiceGenerators::NSW}
    follows from \ref{Item::Trace::CoordReduction::Proof::NiceGenerators::AreGenerators}.

    % By the definition of \(x_0\in \BoundaryNncF[\lambda]\).
    Let \(\Omega\Subset \Omega_1\) be a (small) connected neighborhood of \(x_0\) with \(\BoundaryN\cap\Omega\) connected
    and such that \(Y(x)\not \in \TangentSpace{x}{\BoundaryN}\), \(\forall x\in \Omega\cap \BoundaryN\).
    By letting \(\Xt_{q+1}\) be either \(Y\) or \(-Y\),
    \ref{Item::Trace::CoordReduction::Proof::NiceGenerators::PointsInside} holds (and clearly
    \ref{Item::Trace::CoordReduction::Proof::NiceGenerators::Xqplus1} holds as well).

    Since \(\Omega\Subset \ManifoldNncF[\lambda]\), if \(\Zdv_j< \lambda\) we have
    \(Z_j(x)\in \TangentSpace{x}{\BoundaryN}\), \(\forall x\in \Omega\cap \BoundaryN\).
    For \(1\leq j\leq q\) with \(\Zdv_j\geq \lambda\), pick \(a_j(x)\in \CinftySpace[\Omega]\)
    such that \(Z_j(x)-a_j(x)\Xt_{q+1}(x)\in \TangentSpace{x}{\BoundaryN}\), \(\forall x\in \Omega\cap \BoundaryN\).
    Set
    \begin{equation*}
        \Xt_j:=
        \begin{cases}
            Z_j, &\Zdv_j<\Zdv_{q+1},\\
            Z_j-a_j\Xt, & \Zdv_j\geq \Zdv_{q+1}, 1\leq j\leq q+1,\\
            \Xt_{q+1}, &j=q+1,
        \end{cases}
    \end{equation*} 
    and \(\XtXdv:=\left\{ \left( \Xt_1, \Zdv_1 \right),\ldots, \left( \Xt_{q+1}, \Zdv_{q+1} \right) \right\}\).
    \ref{Item::Trace::CoordReduction::Proof::NiceGenerators::Xqplus1},
    \ref{Item::Trace::CoordReduction::Proof::NiceGenerators::Span}, \ref{Item::Trace::CoordReduction::Proof::NiceGenerators::AreGenerators},
    \ref{Item::Trace::CoordReduction::Proof::NiceGenerators::NSW},
    and \ref{Item::Trace::CoordReduction::Proof::NiceGenerators::PointsInside}
    follow from the above construction and the corresponding results for \(\ZZdv\).
    \ref{Item::Trace::CoordReduction::Proof::NiceGenerators::TangentToBdry} holds by construction.
    Finally, \ref{Item::Trace::CoordReduction::Proof::NiceGenerators::GenerateBoundarySheaf}
    follows from \ref{Item::Trace::CoordReduction::Proof::NiceGenerators::AreGenerators},
    \ref{Item::Trace::CoordReduction::Proof::NiceGenerators::PointsInside},
    \ref{Item::Trace::CoordReduction::Proof::NiceGenerators::TangentToBdry}, and 
    Definition \ref{Defn::Filtrations::RestrictingFiltrations::RestrictedFiltration}.
\end{proof}

\begin{proof}[Proof of Proposition \ref{Prop::Trace::CoordReduction::Proof::ExistCoordSystem}]
    Let \(\XtXdv\) and \(\Omega\) be as in Lemma \ref{Lemma::Trace::CoordReduction::Proof::NiceGenerators}.
    Using Lemma \ref{Lemma::Trace::CoordReduction::Proof::NiceGenerators} \ref{Item::Trace::CoordReduction::Proof::NiceGenerators::PointsInside}
    and standard results from differential topology, for some \(\sigma>0\) small, we may pick
    \(\Phi_0:\Bngeq{\sigma}\xrightarrow{\sim}\Phi_{0}\left( \Bngeq{\sigma} \right)\subset \Omega\) a coordinate chart
    near \(x_0\), with \(\Phi_0^{*} X_{q+1}=\partial_{x_n}\), \(\Phi_0(0)=x_0\), and \(\Phi_0\left( \Bnmo{\sigma} \right)=\Phi_0\left( \Bngeq{\sigma} \right)\cap\BoundaryN\).
    Set \(\Psi_\sigma(t):=\sigma t\), and \(\Phi:=\Phi_0\circ \Psi_{\sigma}:\Bngeq{1}\xrightarrow{\sim}\Phi(\Bngeq{1})\).
    Note that \(\Phi(0)=x_0\), \(\Phi\) is a coordinate chart near \(x_0\), and \(\Phi\left( \Bnmo{1} \right)=\Phi\left( \Bngeq{1} \right)\cap \BoundaryN\).
    Set
    \begin{equation*}
        X_j:=
        \begin{cases}
            \Phi^{*} \Xt_j, &1\leq j\leq q,\\
            \partial_{x_n}=\sigma\Phi^{*} \Xt_{q+1}, & j=q+1,
        \end{cases}
    \end{equation*}
    and \(\XXdv:=\left\{ \left( X_1,\Xdv_1 \right),\ldots, \left( X_{q+1},\Xdv_{q+1} \right) \right\}\).

    \(\XXdv\) satisfies the assumptions \ref{Item::Traces::CoordReduction::IsPartialxn}--\ref{Item::Traces::CoordReduction::Span}
    by construction and the corresponding properties of \(\XtXdv\)
    (see Lemma  \ref{Lemma::Trace::CoordReduction::Proof::NiceGenerators} \ref{Item::Trace::CoordReduction::Proof::NiceGenerators::Span}, \ref{Item::Trace::CoordReduction::Proof::NiceGenerators::NSW}, and \ref{Item::Trace::CoordReduction::Proof::NiceGenerators::TangentToBdry}).

    \(\Phi\) satisfies \ref{Item::Trace::CoordReduction::Proof::ExistCoordSystem::Phi0},
    \ref{Item::Trace::CoordReduction::Proof::ExistCoordSystem::PhiDiffeo},
    and \ref{Item::Trace::CoordReduction::Proof::ExistCoordSystem::PhiMapsBoundarys} by construction;
    and satisfies \ref{Item::Trace::CoordReduction::Proof::ExistCoordSystem::PhiTakesXToF} and \ref{Item::Trace::CoordReduction::Proof::ExistCoordSystem::PhiTakesVToBoundaryF} by
    construction and 
    Lemma \ref{Lemma::Trace::CoordReduction::Proof::NiceGenerators} 
    \ref{Item::Trace::CoordReduction::Proof::NiceGenerators::AreGenerators} and \ref{Item::Trace::CoordReduction::Proof::NiceGenerators::GenerateBoundarySheaf}.
    Using Lemma \ref{Lemma::Trace::CoordReduction::Proof::NiceGenerators} \ref{Item::Trace::CoordReduction::Proof::NiceGenerators::Xqplus1},
    we have \(\Phi^{*}Y=\pm \Phi^{*}\Xt_{q+1}=\pm \sigma^{-1}\partial_{x_n}\),
    establishing \ref{Item::Trace::CoordReduction::Proof::ExistCoordSystem::PullBackY} with \(\sigma\)
    replaced by \(\pm \sigma^{-1}\).

    \ref{Item::Trace::CoordReduction::Proof::ExistCoordSystem::PhiIsIsoOnXSpace}:
    Let \(\Compact\Subset \Bngeq{1}\) be compact. Proposition \ref{Prop::Spaces::BasicProps::DiffeoInv::DiffeosInduceIsomorphism}
    shows
    \begin{equation*}
        \Phi^{*}:\BigASpace{s}{p}{q}[\Phi(\Compact)][\FilteredSheafGenBy{\XtXdv}]\xrightarrow{\sim} \BigASpace{s}{p}{q}[\Compact][\FilteredSheafGenBy{\XXdv}]
    \end{equation*}
    is an isomorphism.
    Corollary \ref{Cor::Spaces::MainEst::SpacesAreDefinedLocally}
    combined with Lemma \ref{Lemma::Trace::CoordReduction::Proof::NiceGenerators} \ref{Item::Trace::CoordReduction::Proof::NiceGenerators::AreGenerators}
    shows
    \begin{equation*}
        \BigASpace{s}{p}{q}[\Phi(\Compact)][\FilteredSheafGenBy{\XtXdv}]=\ASpace{s}{p}{q}[\Phi(\Compact)][\LieFilteredSheafF],
    \end{equation*}
    with equality of topologies. Finally, Theorem \ref{Thm::Spaces::OnlyDependsOnLieFiltration}
    (see, also, Lemma \ref{Lemma::Filtrations::GeneratorsForLieFiltration}) shows
    \begin{equation*}
        \ASpace{s}{p}{q}[\Phi(\Compact)][\LieFilteredSheafF]=\ASpace{s}{p}{q}[\Phi(\Compact)][\FilteredSheafF],
    \end{equation*}
    with equality of topologies. \ref{Item::Trace::CoordReduction::Proof::ExistCoordSystem::PhiIsIsoOnXSpace} follows.

    \ref{Item::Trace::CoordReduction::Proof::ExistCoordSystem::PhiIsIsoOnVSpace}:
    Let \(\Compact_0\Subset \Bngeq{1}\) be compact.
    Since \(\left( \Phi\big|_{\Bnmo{1}} \right)_{*} V_j=\Vt_j\), where \(\Vt_j\) are as in 
    Lemma \ref{Lemma::Trace::CoordReduction::Proof::NiceGenerators} \ref{Item::Trace::CoordReduction::Proof::NiceGenerators::GenerateBoundarySheaf},
    Proposition \ref{Prop::Spaces::BasicProps::DiffeoInv::DiffeosInduceIsomorphism}
    shows \(\Phi\big|_{\Bnmo{1}}^{*}\) is an isomorphism
    \begin{equation*}
                \BigASpace{s}{p}{q}[\Phi(\Compact_0)][\FilteredSheafGenBy{\VtVdv}]
                \xrightarrow{\sim}
                \ASpace{s}{p}{q}[\Compact_0][\FilteredSheafGenByVVdv].
    \end{equation*}
     Corollary \ref{Cor::Spaces::MainEst::SpacesAreDefinedLocally}
    combined with Lemma \ref{Lemma::Trace::CoordReduction::Proof::NiceGenerators} \ref{Item::Trace::CoordReduction::Proof::NiceGenerators::GenerateBoundarySheaf}
    shows
    \begin{equation*}
        \BigASpace{s}{p}{q}[\Phi(\Compact_0)][\FilteredSheafGenBy{\VtVdv}]=\BigASpace{s}{p}{q}[\Phi(\Compact_0)][\RestrictFilteredSheaf{\LieFilteredSheafF}{\BoundaryNncF}]
    \end{equation*}
    with equality of topologies, completing the proof of \ref{Item::Trace::CoordReduction::Proof::ExistCoordSystem::PhiIsIsoOnVSpace}.
\end{proof}

\begin{proof}[Proof of Propositions \ref{Prop::Trace::Reduction::ForwardMap},
\ref{Prop::Trace::Reduction::InverseMap},
and \ref{Prop::Trace::Reduction::VanishingChar}]
    We rename the vector field \(V\)
    in Propositions \ref{Prop::Trace::Reduction::ForwardMap},
\ref{Prop::Trace::Reduction::InverseMap},
and \ref{Prop::Trace::Reduction::VanishingChar} to be \(Y\).
    Applying the coordinate change given in Proposition \ref{Prop::Trace::CoordReduction::Proof::ExistCoordSystem},
    we see that Propositions \ref{Prop::Trace::Reduction::ForwardMap},
\ref{Prop::Trace::Reduction::InverseMap},
and \ref{Prop::Trace::Reduction::VanishingChar}
    are directly implied by Propositions \ref{Prop::Traces::CoordReduction::ForwardMap},
\ref{Prop::Traces::CoordReduction::InverseMap},
and \ref{Prop::Traces::CoordReduction::VanishingChar}, respectively if 
 \(\sN_L\) is replaced  in Propositions 
\ref{Prop::Traces::CoordReduction::InverseMap}
and \ref{Prop::Traces::CoordReduction::VanishingChar} by
    \(\sN_L^{\sigma}=(1, \sigma \partial_{x_n}, \sigma^2 \partial_{x_n}^2, \ldots, \sigma^L \partial_{x_n}^L )\),
    where \(\sigma\) is the nonzero constant from 
    Proposition \ref{Prop::Trace::CoordReduction::Proof::ExistCoordSystem} \ref{Item::Trace::CoordReduction::Proof::ExistCoordSystem::PullBackY}.
    Since
    \begin{equation*}
        \TraceMap[\sN_L^{\sigma}]=
\begin{pmatrix}
1 & & & \\
& \sigma & & \\
& & \ddots & \\
& & & \sigma^L
\end{pmatrix}
\TraceMap[\sN_L],
\quad
\TraceInverseMap[\sN_L^{\sigma}]
=\TraceInverseMap[\sN_L]
\begin{pmatrix}
1 & & & \\
& \sigma^{-1} & & \\
& & \ddots & \\
& & & \sigma^{-L}
\end{pmatrix},
    \end{equation*}
    the results for general (fixed) \(\sigma\) follow from those for \(\sigma=1\).
\end{proof}

    \subsection{Proof of the forward result (Proposition \texorpdfstring{\ref{Prop::Traces::CoordReduction::ForwardMap}}{\ref*{Prop::Traces::CoordReduction::ForwardMap}})}
    \label{Section::Trace::Forward}
    In this section, we prove Proposition \ref{Prop::Traces::CoordReduction::ForwardMap}.
Fix \(0<\eta_1<\eta_2<\eta_3<1\); we will prove Proposition \ref{Prop::Traces::CoordReduction::ForwardMap}
more generally with \(1/2\) replaced by \(\eta_1\).
For \(x\in \Bngeq{1}\), set
\begin{equation*}
    \MetricXXdvSet[\Bnmo{\eta_3}][x]:=\inf\left\{ \delta>0 : \exists x_0'\in \Bnmo{\eta_3}, \MetricXXdv[(x_0',0)][x]<\delta \right\}.
\end{equation*}
and 
\begin{equation*}
    \DistXXdvSet[\Bnmo{\eta_3}][x]:=\inf\left\{ \delta>0 : \exists x_0'\in \Bnmo{\eta_3}, \DistXXdv[(x_0',0)][x]<\delta \right\},
\end{equation*}
where \(\DistXXdv[x][y]\) and \(\MetricXXdv[x][y]=\DistXXdv[x][y]\wedge 1\)  are defined in \eqref{Eqn::VectorFields::DefineMetric} and \eqref{Eqn::VectorFields::DefineDistance}.

\begin{lemma}\label{Lemma::Trace::Forward::MetricAndDistEquiv}
    \(\MetricXXdv[x][y]\approx \DistXXdv[x][y]\), \(\forall x,y\in \BngeqClosure{\eta_3}\),
    where the implicit constant may depend on \(\XXdv\) and \(\eta_3\).
\end{lemma}
\begin{proof}
    The statement of the lemma is equivalent to \(\DistXXdv[x][y]\lesssim 1\), \(\forall x,y\in \BngeqClosure{\eta_3}\).
    This follows from the definitions and the assumption \ref{Item::Traces::CoordReduction::Span}.
\end{proof}

\begin{lemma}\label{Lemma::Trace::Forward::SimpleDistanceToBoundary}
    For \(x=(x',x_n)\in \Bngeq{\eta_3}\),
    \begin{equation*}
        \MetricXXdvSet[\Bnmo{\eta_3}][x]
        \approx \MetricXXdv[(x',0)][x]\approx x_n^{1/\lambda}
        \approx \DistXXdv[(x',0)][x]
        \approx \DistXXdvSet[\Bnmo{\eta_3}][x].
    \end{equation*}
\end{lemma}
\begin{proof}
    Using Lemma \ref{Lemma::Trace::Forward::MetricAndDistEquiv} once we show
    \(\DistXXdvSet[\Bnmo{\eta_3}][x]\approx \DistXXdv[(x',0)][x]\approx x_n^{1/\lambda}\), it immediately follows
    that \(\MetricXXdvSet[\Bnmo{\eta_3}][x]
        \approx \MetricXXdv[(x',0)][x]\approx x_n^{1/\lambda}\).

    Clearly, \(\DistXXdvSet[\Bnmo{\eta_3}][x]\leq \DistXXdv[(x',0)][x]\).
    By considering the path \(\gamma(t)=(x',tx_n)\) we see
    \begin{equation*}
        \gamma'(t)=x_n \partial_{x_n}= \left( x_n^{1/\lambda} \right)^{\lambda} X_{q+1}(\gamma(t)),
    \end{equation*}
    and therefore \(\DistXXdv[(x',0)][x]\leq x_n^{1/\lambda}\).

    Finally, we show \(x_n^{1/\lambda}\lesssim \DistXXdvSet[\Bnmo{\eta_3}][x]\).
    Suppose \(x_0'\in \Bnmo{\eta_3}\) with \(x\in \BXXdv{(x_0',0)}{\delta}\)'
    we will show \(x_n\lesssim \delta^{\lambda}\), which will complete the proof.
    Since \(x_n\leq 1\), it suffices to prove this when \(\delta>0\) is small (depending on \(\eta_3\)).
    In particular, we may assume \(\BXXdv{(x_0',0)}{\delta}\) lies in a fixed compact subset of \(\Bngeq{1}\).

    Since \(X_j(x_0',0)\) has no \(\partial_{x_n}\) component, for \(1\leq j\leq q\), we may write
    \begin{equation*}
        X_j(x)=Z_j(x)+b_j(x)x_n\partial_{x_n}, \quad 1\leq j\leq q,
    \end{equation*}
    where \(Z_j\) has no \(\partial_{x_n}\) component and \(b_j\in \CinftySpace[\Bngeq{1}]\).
    Since \(x\in \BXXdv{(x_0',0)}{\delta}\), \(\exists \gamma:[0,1]\rightarrow \Bngeq{1}\) absolutely continuous,
    with 
    \(\gamma(0)=(x_0',0)\), \(\gamma(1)=x\), and
    (setting \(\gamma_n(t)\) to be the last component of \(\gamma(t)\))
    \begin{equation*}
        \gamma'(t)=\sum_{j=1}^{q+1} a_j(t) \delta^{\Xdv_j} X_j(\gamma(t))
        =\sum_{j=1}^q a_j(t) \delta^{\Xdv_j} Z_j(\gamma(t)) +\sum_{j=1}^q a_j(t)\delta^{\Xdv_j}b_j(\gamma(t)) \gamma_n(t) \partial_{x_n}
        +a_{q+1}(t)\delta^{\lambda}\partial_{x_n},
    \end{equation*}
    where \(\sum |a_j|^2<1\), almost everywhere.  Since \(\gamma(t)\in \BXXdv{(x_0',0)}{\delta}\),
    which lies in a fixed compact subset of \(\Bngeq{1}\), we have \(|b_j(\gamma(t))|\lesssim 1\).
    We conclude
    \begin{equation*}
        |\gamma_n(t)|\lesssim t\delta^{\lambda} + \sum_{j=1}^q \delta^{\Xdv_j} \int_0^t |\gamma_n(t)|\: dt
        \lesssim t\delta^{\lambda} +  \int_0^t |\gamma_n(t)|\: dt.
    \end{equation*}
    Gr\"onwall's inequality gives \(|\gamma_n(t)|\leq C t\delta^{\lambda} e^{Ct}\) for some \(C\approx 1\).
    Setting \(t=1\) gives \(x_n = |\gamma_n(t)|\lesssim \delta^{\lambda}\), completing the proof.
\end{proof}

For \(f\in \LpSpace{1}[\Bngeq{\eta_3}]\) and \(x\in \Bngeq{1}\), set
\begin{equation}\label{Eqn::Trace::Forward::DefineMaximal}
    \Maximal f(x):=
    \begin{cases}
        \sup_{\delta>0} \sup_{\substack{x\in \BXXdv{z}{\delta} \\ z\in \Bngeq{\eta_3}}}
            \Vol[\BXXdv{z}{\delta}]^{-1} \int_{\BXXdv{z}{\delta}} |f(y)|\: dy + \int_{\Bngeq{\eta_3}}|f(y)|\: dy, & x\in \Bngeq{\eta_3},\\
            0, &\text{otherwise}.
    \end{cases}
\end{equation}
By Proposition \ref{Prop::VectorFields::Scaling::VolEstimates} \ref{Item::VectorFields::Scaling::VolEstimates::VolWedge1Doubling}
(with \(\WWdv\) replaced by \(\XXdv\)), we see
\begin{equation}\label{Eqn::Trace::Forward::DoublingWithWedge}
    \Vol[\BXXdv{z}{2\delta}]\wedge 1\lesssim \Vol[\BXXdv{z}{\delta}]\wedge 1,\quad \forall z\in \Bngeq{\eta_3},\: \delta>0.
\end{equation}
However, since \(\Vol[\BXXdv{z}{\delta}]\leq \Vol[\Bngeq{1}]\lesssim 1\), \(\forall z\in \Bngeq{1}\), \(\delta>0\),
we see
\begin{equation}\label{Eqn::Trace::Forward::Doubling}
    \Vol[\BXXdv{z}{2\delta}]\lesssim \Vol[\BXXdv{z}{\delta}],\quad \forall z\in \Bngeq{\eta_3},\: \delta>0.
\end{equation}
Due to \eqref{Eqn::Trace::Forward::Doubling}, we have the Fefferman--Stein vector valued inequalities:
\begin{equation}\label{Eqn::Trace::Forward::VVMaximalIneq}
\begin{split}
     &\lqLpNorm{\left\{ \Maximal f_j \right\}_{j\in \Zgeq}}{p}{q}[\Bngeq{\eta_3}]
     \leq \lqLpNorm{\left\{ f_j \right\}_{j\in \Zgeq}}{p}{q}[\Bngeq{\eta_3}], \quad 1<p\leq \infty,\: 1\leq q\leq \infty,
     \\&\LplqNorm{\left\{ \Maximal f_j \right\}_{j\in \Zgeq}}{p}{q}[\Bngeq{\eta_3}]
     \leq \LplqNorm{\left\{ f_j \right\}_{j\in \Zgeq}}{p}{q}[\Bngeq{\eta_3}], \quad 1<p< \infty,\: 1< q\leq \infty.
\end{split}
\end{equation}
See \cite[Theorem 1.2]{GrafaosLiuYangVectorValuedSingularIntegralsAndMaximalFunctionsOnSpacesOfHomogeneousType}
which is based on \cite{FeffermanSteinSomeMaximalInequalities}.
In particular, using Notation \ref{Notation::Spaces::Classical::VSpacepq},
\eqref{Eqn::Trace::Forward::VVMaximalIneq} holds for the spaces \(\VSpace{p}{q}\).

Let \(\sE\in \ElemzXXdv{\BngeqClosure{\eta_2}}\). For \((E_j,2^{-j})\in \sE\), and \(L\in \Zgeq\), set
\begin{equation*}
    E_j^{*,L} f(x):=
    \begin{cases}
        \sup_{z\in \Bngeq{1}} \left(   1+2^{j}\DistXXdv[x][z] \right)^{-L} \left| E_j f(z) \right|, &x\in \Bngeq{\eta_3},\\
        0,&\text{otherwise}.
    \end{cases}
\end{equation*}

\begin{proposition}\label{Prop::Trace::Forward::BoundEjStarByNorm}
    \(\exists L\in \Zgeq\), \(\forall \sE\in \ElemzXXdv{\BngeqClosure{\eta_2}}\),
    \begin{equation*}
        \sup_{\left\{ \left( E_j,2^{-j} \right) : j\in \Zgeq \right\}\subseteq \sE} 
        \BVNorm{\left\{ 2^{js} E_j^{*,L}f \right\}_{j\in \Zgeq}}{p}{q}
        \lesssim \ANorm{f}{s}{p}{q}[\FilteredSheafGenByXXdv],\quad \forall f\in \CinftyCptSpace[\Bngeq{\eta_2}].
    \end{equation*}
\end{proposition}

To prove Proposition \ref{Prop::Trace::Forward::BoundEjStarByNorm}, we need several preliminary lemmas.

\begin{lemma}\label{Lemma::Trace::Forward::BoundQuotientVolByPoly}
    \(\exists L\in \Zgeq\), \(\forall z\in \BngeqClosure{\eta_2}\), \(\forall \delta_1,\delta_2>0\),
    \begin{equation}\label{Eqn::Trace::Forward::BoundQuotientVolByPoly::MainEqn}
        \frac{\Vol[\BXXdv{z}{\delta_2}]}{\Vol[\BXXdv{z}{\delta_1}]}
        \lesssim \left( 1+\frac{\delta_2}{\delta_1} \right)^L.
    \end{equation}
\end{lemma}
\begin{proof}
    If \(\delta_2\leq \delta_1\), then the left-hand side of \eqref{Eqn::Trace::Forward::BoundQuotientVolByPoly::MainEqn}
    is \(\leq 1\) and the result is trivial.  If \(\delta_2/\delta_1=2^M\) for some \(M\in \Zgeq\),
    then repeated applications of \eqref{Eqn::Trace::Forward::Doubling} show
    \begin{equation*}
        \Vol[\BXXdv{z}{\delta_2}]
        =\Vol[\BXXdv{z}{(\delta_2/\delta_1)\delta_1}]
        \leq 2^{LM}\Vol[\BXXdv{z}{\delta_1}]
        =\left( \delta_2/\delta_1 \right)^{L}\Vol[\BXXdv{z}{\delta_1}],
    \end{equation*}
    for some fixed \(L\in \Zgeq\), establishing the result in this case. Finally,
    the case of general \(\delta_2\) can be easily reduced to the case when \(\delta_2/\delta_1=2^M\)
    for some \(M\in \Zgeq\).
\end{proof}

\begin{lemma}\label{Lemma::Trace::Forward::BoundAverageByMaximal}
    \(\exists L\in \Zgeq\), \(\forall z\in \BngeqClosure{\eta_2}\), \(\forall f\in \Bngeq{\eta_2}\),
    \(\forall \delta>0\), \(\forall x\in \Bngeq{1}\),
    \begin{equation*}
        \frac{1}{\Vol[\BXXdv{z}{\delta}]} \int_{\BXXdv{z}{\delta}} |f(y)|\: dy 
        \lesssim \left( 1+\delta^{-1}\DistXXdv[x][z] \right)^L \Maximal f(x).
    \end{equation*}
\end{lemma}
\begin{proof}
    We separate into two cases. When \(\delta\geq \DistXXdv[x][z]\), then \(x\in \BXXdv{z}{2\delta}\)
    and therefore using \eqref{Eqn::Trace::Forward::Doubling}, we have
    \begin{equation}\label{Eqn::Trace::Forward::BoundAverageByMaximal::BigDelta}
        \frac{1}{\Vol[\BXXdv{z}{\delta}]} \int_{\BXXdv{z}{\delta}} |f(y)|\: dy
        \lesssim \frac{1}{\Vol[\BXXdv{z}{2\delta}]} \int_{\BXXdv{z}{2\delta}} |f(y)|\: dy
        \leq \Maximal f(x),
    \end{equation}
    by the definition of \(\Maximal f(x)\), see \eqref{Eqn::Trace::Forward::DefineMaximal}.

    When \(\delta<\DistXXdv[x][z]\), we have using Lemma \ref{Lemma::Trace::Forward::BoundQuotientVolByPoly},
    \begin{equation*}
    \begin{split}
         &\frac{1}{\Vol[\BXXdv{z}{\delta}]} \int_{\BXXdv{z}{\delta}} |f(y)|\: dy
         \\&\leq \frac{\Vol[\BXXdv{z}{\DistXXdv[x][z]}]}{\Vol[\BXXdv{z}{\delta}]} \frac{1}{\Vol[\BXXdv{z}{\DistXXdv[x][z]}]} \int_{\BXXdv{z}{\DistXXdv[x][z]}}|f(y)|\: dy
         \\&\lesssim \left( 1+\delta^{-1} \DistXXdv[x][z] \right)^L \frac{1}{\Vol[\BXXdv{z}{\DistXXdv[x][z]}]} \int_{\BXXdv{z}{\DistXXdv[x][z]}}|f(y)|\: dy
         \\&\lesssim \left( 1+\delta^{-1} \DistXXdv[x][z] \right)^L\Maximal f(x),
    \end{split}
    \end{equation*}
    where the final inequality used \eqref{Eqn::Trace::Forward::BoundAverageByMaximal::BigDelta}.
\end{proof}

\begin{lemma}\label{Lemma::Trace::Forward::BoundPreElemByMaximal}
    \(\exists L\in \Zgeq\), \(\forall \sE\in \PElemXXdv{\BngeqClosure{\eta_2}}\),
    \(\exists C\geq 0\), \(\forall (E_j, 2^{-j})\in \sE\), \(\forall f\in \LpSpace{1}[\Bngeq{\eta_2}]\),
    \(\forall x,z\in \Bngeq{1}\),
    \begin{equation}\label{Eqn::Trace::Forward::BoundPreElemByMaximal}
        \left| E_j f(x) \right|\leq C \left( 1+2^{j} \DistXXdv[x][z] \right)^L \Maximal f(x).
    \end{equation}
\end{lemma}
\begin{proof}
    When \(z\not \in \BngeqClosure{\eta_2}\), the left hand side of \eqref{Eqn::Trace::Forward::BoundPreElemByMaximal}
    is zero (see Definition \ref{Defn::Spaces::LP::PElemWWdv} \ref{Item::Spaces::LP::PElemWWdv::Support}),
    so we may assume \(z\in \BngeqClosure{\eta_2}\).

    Using Lemma \ref{Lemma::Trace::Forward::MetricAndDistEquiv} and \eqref{Eqn::Trace::Forward::Doubling}, we have
    \begin{equation*}
        \Vol[\BXXdv{x}{\delta+\MetricXXdv[x][y]}] \approx \Vol[\BXXdv{x}{\delta+\DistXXdv[x][y]}],\quad \forall x,y\in \Bngeq{\eta_3},\: \delta>0.
    \end{equation*}
    And therefore,  using Lemma \ref{Lemma::Trace::Forward::MetricAndDistEquiv} and 
    Definition \ref{Defn::Spaces::LP::PElemWWdv} \ref{Item::Spaces::LP::PElemWWdv::Bound},
    we have \(\forall m\in \Zgeq\),
    \begin{equation*}
        |E_j(x,y)|\lesssim \frac{\left( 1+2^{j}\MetricXXdv[x][y] \right)^{-m}}{\Vol[\BXXdv{x}{2^{-j}+\MetricXXdv[x][y]}]\wedge 1}
        \approx \frac{\left( 1+2^{j}\DistXXdv[x][y] \right)^{-m}}{\Vol[\BXXdv{x}{2^{-j}+\DistXXdv[x][y]}]}.
    \end{equation*}

    Taking \(L\in \Zgeq\) as in Lemma \ref{Lemma::Trace::Forward::BoundAverageByMaximal} and \(m= L+1\),
    we have (using \eqref{Eqn::Trace::Forward::Doubling})
    \begin{equation*}
    \begin{split}
         &\left| E_j f(z) \right|
         \lesssim \int_{\Bngeq{\eta_2}} \frac{\left( 1+2^j \DistXXdv[z][y] \right)^L}{\Vol[\BXXdv{z}{2^{-j}+\DistXXdv[z][y]}]}|f(y)|\: dy
         \\&\lesssim \sum_{k=0}^\infty 2^{-k(L+1)} \frac{1}{\Vol[\BXXdv{z}{2^{-j+k}}]} \int_{y\in \BXXdv{z}{2^{-j+k}}} |f(y)|\: dy
         \\&\lesssim \sum_{k=0}^\infty 2^{-k(L+1)} \left( 1+2^{j-k}\DistXXdv[x][z] \right)^L \Maximal f(x)
         \lesssim \sum_{k=0}^\infty 2^{-k} \left( 1+2^{j}\DistXXdv[x][z] \right)^L \Maximal f(x)
         \\&\lesssim \left( 1+2^{j}\DistXXdv[x][z] \right)^L \Maximal f(x).
    \end{split}
    \end{equation*}
\end{proof}

\begin{lemma}\label{Lemma::Trace::Forward::BoundProdElemByMaximal}
    \(\exists L\in \Zgeq\), \(\forall \sE\in \ElemzXXdv{\Bngeq{\eta_2}}\), \(\forall N\in \Zgeq\),
    \(\exists K=K(N)\in \Zgeq\), \(\exists \sG_N\in \ElemzXXdv{\Bngeq{\eta_2}}\),
    \(\exists C\geq 0\),
    \(\forall (E_j, 2^{-j}), (F_k, 2^{-k})\in \sE\), 
    \(\exists (G_j^1,2^{-j}),\ldots (G_j^K, 2^{-j})\in \sG_N\),
    \(\forall f\in \DistributionsZero[\Bngeq{1}]\),
    \(\forall x,z\in \Bngeq{1}\)
    \begin{equation*}
        \left| F_k E_j f(z) \right|
        \leq C \sum_{l=1}^K 2^{-N|j-k|} \left( 1+2^j \DistXXdv[x][z] \right)^L \Maximal\left( G_j^l f \right)(x).
    \end{equation*}
\end{lemma}
\begin{proof}
    By replacing \(N\) with \(N+L\) it suffices to prove
    \begin{equation}\label{Eqn::Trace::Forward::BoundProdElemByMaximal::ToShow}
        \left| F_k E_j f(z) \right|
        \leq C \sum_{l=1}^K 2^{-N|j-k|} \left( 1+2^k \DistXXdv[x][z] \right)^L \Maximal\left( G_j^l f \right)(x).
    \end{equation}
    In light of Lemma \ref{Lemma::Spaces::Elem::ELem::ProdOfElemAsSmallSumOfProds},
    it suffices to prove \eqref{Eqn::Trace::Forward::BoundProdElemByMaximal::ToShow} in the case \(N=0\).

    Using Lemma \ref{Lemma::Trace::Forward::BoundPreElemByMaximal}, we have
    \begin{equation*}
        \left| F_k E_j f(z) \right|
        \lesssim \left( 1+2^k \DistXXdv[x][z] \right)^L \Maximal\left( E_j f \right)(x),
    \end{equation*}
    which is of the form \eqref{Eqn::Trace::Forward::BoundProdElemByMaximal::ToShow} with \(N=0\),
    completing the proof.
\end{proof}

\begin{lemma}\label{Lemma::Trace::Forward::BoundElemStarByMaximal}
    \(\exists L\in \Zgeq\), \(\forall \sE\in \ElemzXXdv{\Bngeq{\eta_2}}\), \(\forall N\in \Zgeq\),
    \(\exists K=K(N)\in \Zgeq\), \(\exists \sG_N\in \ElemzXXdv{\Bngeq{\eta_2}}\),
    \(\forall (E_j, 2^{-j})\in \sE\), 
    \(\exists \left\{ \left( G_{j,k}^l, 2^{-j} \right) : k\in \Zgeq, 1\leq l\leq K \right\}\subseteq \sG_N\),
    \(\exists C\geq 0\),
    \(\forall f\in \DistributionsZero[\Bngeq{1}]\), \(\forall x\in \Bngeq{1}\),
    \begin{equation}\label{Eqn::Trace::Forward::BoundElemStarByMaximal}
        \left| E_j^{*,L} f(x) \right|
        \leq C \sum_{k=0}^\infty \sum_{l=1}^K 2^{-N|j-k|} \Maximal\left( G_{j,k}^l f \right)(x).
    \end{equation}
\end{lemma}
\begin{proof}
    We will show, for \((E_j, 2^{-j})\in \sE\),
    \begin{equation}\label{Eqn::Trace::Forward::BoundElemStarByMaximal::ToShow}
        \left| E_j f(z) \right|
        \leq C \sum_{k=0}^{\infty} \sum_{l=1}^K 2^{-N|j-k|} \left( 1+2^{j}\DistXXdv[x][z] \right)^L \Maximal\left( G_{j,k}^l f \right)(x).
    \end{equation}
    The result then follows by dividing both sides of \eqref{Eqn::Trace::Forward::BoundElemStarByMaximal::ToShow}
    by \(\left( 1+2^{j}\DistXXdv[x][z] \right)^L\) and taking the supremum \(z\).

    Using the definition of \(\ElemzXXdv{\Bngeq{\eta_2}}\) (see \eqref{Eqn::Spaces::LP::ElemzFOmegaIsUnion}),
    there is a compact set \(\Compact\Subset \Bngeq{\eta_2}\) such that
    \(\supp(E_j)\subseteq \Compact\times\Compact\), \(\forall (E_j, 2^{-j})\in \sE\).
    Take \(\psi\in \CinftyCptSpace[\Bngeq{\eta_2}]\) with \(\psi=1\) on a neighborhood of \(\Compact\).
    Using Proposition \ref{Prop::Spaces::LP::DjExist}, we may write
    \begin{equation*}
        \Mult{\psi}=\sum_{k\in \Zgeq} D_k,
    \end{equation*}
    where \(\left\{ \left( D_k,2^{-k} \right):k\in \Zgeq \right\}\in \ElemzXXdv{\Bngeq{\eta_2}}\),
    and the convergence is described in Proposition \ref{Prop::Spaces::Elem::Elem::ConvergenceOfElemOps}.

    Using Lemma \ref{Lemma::Trace::Forward::BoundProdElemByMaximal}, we have
    \begin{equation*}
    \begin{split}
         &\left| E_j f(z) \right|
         =\left| \Mult{\psi} E_j f(z) \right|
         \leq \sum_{k=0}^\infty \left| D_k E_j f(z) \right|
         \leq C \sum_{k=0}^\infty \sum_{l=1}^K 2^{-N|j-k|}\left( 1+2^j \DistXXdv[x][z] \right)^L \Maximal\left( G_{j,k}^l f \right)(x),
    \end{split}
    \end{equation*}
    where the above are as in the statement of the lemma. This establishes \eqref{Eqn::Trace::Forward::BoundElemStarByMaximal::ToShow}
    and completes the proof.
\end{proof}

\begin{lemma}\label{Lemma::Trace::Forward::BoundElemStarVVByElemVV}
    \(\exists L\in \Zgeq\), \(\forall \sE\in \ElemzXXdv{\Bngeq{\eta_2}}\), \(\exists \sG\in \ElemzXXdv{\Bngeq{\eta_2}}\),
    \(\exists C\geq 0\),
    \(\forall f_j \in \DistributionsZero[\Bngeq{1}]\),
    \begin{equation*}
        \sup_{\left\{ \left( E_j, 2^{-j} \right) : j\in \Zgeq \right\}\subseteq \sE} 
        \BVNorm{ \left\{ E_j^{*,L} f_j \right\}_{j\in \Zgeq} }{p}{q} 
        \lesssim
        \sup_{\left\{ \left( G_j, 2^{-j} \right):j\in \Zgeq \right\}\subseteq \sG} \BVNorm{ \left\{ G_j f_j \right\}_{j\in \Zgeq}}{p}{q}.
    \end{equation*}
\end{lemma}
\begin{proof}
    Let \(\sG_1\in \ElemzXXdv{\Bngeq{\eta_2}}\) be as in Lemma \ref{Lemma::Trace::Forward::BoundElemStarByMaximal}
    with \(N=1\).
    Then, using \eqref{Eqn::Trace::Forward::BoundElemStarByMaximal} and \eqref{Eqn::Trace::Forward::VVMaximalIneq},
    we have for \(\left\{ \left( E_j, 2^{-j} \right) : j\in \Zgeq \right\}\subseteq \sE\),
    using the notation of Lemma \ref{Lemma::Trace::Forward::BoundElemStarByMaximal},
    \begin{equation*}
    \begin{split}
         &\BVNorm{\left\{ E_j^{*,L}f_j \right\}_{j\in \Zgeq}}{p}{q}
         \lesssim \sum_{k=0}^\infty \sum_{l=1}^K 2^{-|j-k|} \BVNorm{\left\{ \sM\left( G_{j,k}^l f_j \right) \right\}_{j\in \Zgeq}}{p}{q}
         \\&\leq \sum_{k=0}^\infty \sum_{l=1}^K 2^{-|j-k|} \sup_{\left\{ \left( G_j, 2^{-j} \right) : j\in \Zgeq \right\}\subseteq \sG_1} \BVNorm{\left\{ G_j f_j \right\}_{j\in \Zgeq}}{p}{q}
         \approx \sup_{\left\{ \left( G_j, 2^{-j} \right) : j\in \Zgeq \right\}\subseteq \sG_1} \BVNorm{\left\{ G_j f_j \right\}_{j\in \Zgeq}}{p}{q},
    \end{split}
    \end{equation*}
    where we have used \(K\approx 1\) in the estimate. Taking the supremum over \(\left\{ \left( E_j, 2^{-j} \right) : j\in \Zgeq \right\}\subseteq \sE\)
    completes the proof.
\end{proof}

\begin{proof}[Proof of Proposition \ref{Prop::Trace::Forward::BoundEjStarByNorm}]
    By Lemma \ref{Lemma::Trace::Forward::BoundElemStarVVByElemVV}, \(\exists \sG\in \ElemzXXdv{\Bngeq{\eta_2}}\),
    \(\forall f\in \CinftyCptSpace[\Bngeq{\eta_2}]\),
    \begin{equation*}
        \sup_{\left\{ \left( E_j, 2^{-j}  \right): j\in \Zgeq \right\}\subseteq \sE} 
        \BVNorm{\left\{ 2^{js} E_j^{*,L} f \right\}_{j\in \Zgeq}}{p}{q}
        \lesssim \VpqsENorm{f}[p][q][s][\sG]
        \lesssim \ANorm{f}{s}{p}{q}[\XXdv],
    \end{equation*}
    where the final estimate uses Corollary \ref{Cor::Spaces::MainEst::VpqsESeminormIsContinuous}.
\end{proof}

For the remainder of the proof, fix \(\epsilon_0:=\left( \eta_3-\eta_2 \right)/2\approx 1\).

\begin{lemma}\label{Lemma::Trace::Forward::BoundRestrictElemByElemStar}
    \(\forall K\in \Zgeq\), \(\exists C\geq 0\), \(\forall \sE\in \ElemzXXdv{\Bngeq{\eta_2}}\), \(\forall (E_j, 2^{-j})\in \sE\),
    \(\forall f\in \DistributionsZero[\Bngeq{1}]\), \(\forall (x',x_n)=x\in \Rngeq\),
    \begin{equation}\label{Eqn::Trace::Forward::BoundRestrictElemByElemStar}
        \left| E_j f(x',0) \right|
        \leq C E_j^{*,K} f(x',x_n)\text{ if } 2^{-(j+1)\lambda}\epsilon_0 \leq x_n\leq 2^{-j\lambda} \epsilon_0.
    \end{equation}
\end{lemma}
\begin{proof}
    For \(x'\not \in \Bnmo{\eta_2}\), the left-hand side of \eqref{Eqn::Trace::Forward::BoundRestrictElemByElemStar}
    is zero, so we may assume \(x'\in \Bnmo{\eta_2}\). For the \(x_n\) under consideration (by the choice of \(\epsilon_0\)),
    we have \((x',x_n)\in \Bngeq{\eta_3}\).

    By the definition of \(E_j^{*,K}\), we have
    \begin{equation*}
        \left| E_j f(x',0) \right|
        \leq \left( 1+2^{j}\MetricXXdv[(x',0)][(x',x_n)] \right)^K \left| E_j^{*,K} f(x',x_n) \right|.
    \end{equation*}
    Lemma \ref{Lemma::Trace::Forward::SimpleDistanceToBoundary} shows
    \(\MetricXXdv[(x',0)][(x',x_n)]\approx x_n^{1/\lambda}\approx 2^{-j}\) for the \(x=(x',x_n)\) in question.
    The result follows.
\end{proof}

\begin{lemma}\label{Lemma::Trace::Forward::BoundRestrictElemByIntElemStar}
    \(\forall K\in \Zgeq\), \(\exists C\geq 0\), \(\forall \sE\in \ElemzXXdv{\Bngeq{\eta_2}}\), \(\forall p\in (0,\infty)\),
     \(\forall (E_j, 2^{-j})\in \sE\),
    \begin{equation}\label{Eqn::Trace::Forward::BoundRestrictElemByIntElemStar::pFinite}
        \LpNorm{ E_j f(x',0)}{p}[\Rnmo]^p \leq C^p 2^{j\lambda} \int_{\epsilon_0 2^{-(j+1)\lambda}}^{\epsilon_0 2^{-j\lambda}} 
        \LpNorm{E_j^{*,K} f(x',x_n)}{p}[\Rnmo]^p\: dx_n.
    \end{equation}
    For \(p=\infty\), we have
    \(\forall K\in \Zgeq\), \(\exists C\geq 0\), \(\forall \sE\in \ElemzXXdv{\Bngeq{\eta_2}}\), 
     \(\forall (E_j, 2^{-j})\in \sE\),
    \begin{equation}\label{Eqn::Trace::Forward::BoundRestrictElemByIntElemStar::pInfinite}
        \LpNorm{ E_j f(x',0)}{\infty}[\Rnmo] \leq C \LpNorm{ E_j^{*,K} f(x',x_n)}{\infty}[\Rngeq].
    \end{equation}
\end{lemma}
\begin{proof}
    This follows immediately from Lemma \ref{Lemma::Trace::Forward::BoundRestrictElemByElemStar}.
\end{proof}

Recall, \(V_j(x')=X_{j}(x',0)\) has no \(\partial_{x_n}\) component for \(1\leq j\leq q\).
With \(\TraceMap f(x'):=f(x',0)\), we have for smooth functions \(f\),
\begin{equation}\label{Eqn::Trace::Forward::CommuteVandRestrict}
    V_j \TraceMap f = \TraceMap X_j f,\quad 1\leq j\leq q.
\end{equation}

\begin{lemma}\label{Lemma::Trace::Forward::ElemRestrictElemAsSmallSumOfElemRestrictElem}
    \(\forall N\in \Zgeq\), \(\exists K\in \Zgeq\),
    \(\forall \sE\in \ElemzXXdv{\Bngeq{\eta_2}}\), \(\forall \sE'\in \ElemzVVdv{\Bnmo{\eta_2}}\),
    \(\exists \sG_N\in \ElemzXXdv{\Bngeq{\eta_2}}\), \(\exists \sG_N'\in \ElemzVVdv{\Bngeq{\eta_2}}\),
    \(\forall (E_j, 2^{-j})\in \sE\), \(\forall (E_k',2^{-k})\in \sE'\) with \(k\geq j\),
    \(\exists (E_{j,1},2^{-j}),\ldots, (E_{j,K},2^{-j})\in \sG_N\),
    \(\exists (E_{k,1}',2^{-k}),\ldots, (E_{k,K}',2^{-k})\in \sG_N'\),
    \begin{equation*}
        E_k' \TraceMap E_j = \sum_{l=1}^K 2^{-N|j-k|} E_{k,l}' \TraceMap E_{j,l}.
    \end{equation*}
\end{lemma}
\begin{proof}
    This proof is similar to the proof of Lemma \ref{Lemma::Spaces::Elem::ELem::ProdOfElemAsSmallSumOfProds}.
    It suffices to prove the result for \(N=1\), as the result for any \(N\) follows from
    repeated application of \(N=1\) case.

    Let \((E_j, 2^{-j})\in \sE\), \((E_k', 2^{-k})\in \sE'\) with \(k\geq j\).
    We use Proposition \ref{Prop::Spaces::Elem::Elem::MainProps} \ref{Item::Spaces::Elem::Elem::PullOutNDerivs} (with \(N=1\)) and
    \eqref{Eqn::Trace::Forward::CommuteVandRestrict} to write
    \begin{equation*}
        \begin{split}
        &E_k' \TraceMap E_j 
        = \sum_{|\alpha|\leq 1} 2^{(|\alpha|-1)k} E_{k,\alpha}' \left( 2^{-k\Vdv} V \right)^{\alpha}\TraceMap E_j
        = \sum_{|\alpha|\leq 1} 2^{(|\alpha|-1)k} E_{k,\alpha}' \TraceMap \left( 2^{-k\Xdv}X \right)^{\alpha} E_j
        \\&= \sum_{|\alpha|\leq 1} 2^{(|\alpha|-1)k - (k-j)\DegXdv{\alpha}} E_{k,\alpha}' \TraceMap \left( 2^{-j\Xdv} X \right)^{\alpha}E_j,
        \end{split}
    \end{equation*}
    where \(\left\{ \left( E_{k,\alpha}', 2^{-j} \right) : \left( E_k',2^{-j} \right)\in \sE', |\alpha|\leq 1 \right\}\in \ElemzVVdv{\Bnmo{\eta_2}}\).
    Proposition \ref{Prop::Spaces::Elem::Elem::MainProps} \ref{Item::Spaces::Elem::Elem::DerivOp}
    shows \(\left\{ \left( \left( 2^{-j\Xdv}X \right)^{\alpha}E_j, 2^{-j} \right) : (E_j, 2^{-j})\in \sE, |\alpha|\leq 1 \right\}\in \ElemzXXdv{\Bngeq{\eta_2}}\).
    Finally, using \(2^{(|\alpha|-1)k - (k-j)\DegXdv{\alpha}}\leq 2^{-|k-j|}\) and Proposition \ref{Prop::Spaces::Elem::Elem::MainProps} \ref{Item::Spaces::Elem::Elem::LinearComb}
    completes the proof.
\end{proof}

\begin{lemma}\label{Lemma::Trace::Forward::ElemRestrictBoundedByElemRestrictElem}
    \(\forall \sE'\in \ElemzVVdv{\Bnmo{\eta_2}}\), \(\forall N\in \Zgeq\),
    \(\exists K\in \Zgeq\), \(\exists \sE\in \ElemzXXdv{\Bngeq{\eta_2}}\), \(\exists C\geq 0\),
    \(\forall (E_k', 2^{-k})\in \sE'\), \(\forall j\in \Zgeq\), \(\exists \left( E_{j,1},2^{-j} \right),\ldots, \left( E_{j,K},2^{-j} \right)\in \sE\),
    \(\forall p\in [1,\infty]\),
    \(\forall f\in \CinftyCptSpace[\Bngeq{\eta_1}]\),
    \begin{equation}\label{Eqn::Trace::Forward::ElemRestrictBoundedByElemRestrictElem}
        \BLpNorm{E_k' \TraceMap f}{p}[\Rnmo] 
        \leq C \sum_{l=1}^K \sum_{j=0}^\infty 2^{-N\left( (k-j)\vee 0 \right)} \BLpNorm{\TraceMap E_{j,l}f}{p}[\Rnmo].
    \end{equation}
\end{lemma}
\begin{proof}
    Fix \(\psi\in \CinftyCptSpace[\Bngeq{\eta_2}]\) with \(\psi=1\) on a neighborhood of \(\BngeqClosure{\eta_1}\).
    Using Proposition \ref{Prop::Spaces::LP::DjExist}, we may write
    \begin{equation*}
        \Mult{\psi}=\sum_{k\in \Zgeq} D_k,
    \end{equation*}
    where \(\left\{ \left( D_k,2^{-k} \right):k\in \Zgeq \right\}\in \ElemzXXdv{\Bngeq{\eta_2}}\),
    and the convergence is described in Proposition \ref{Prop::Spaces::Elem::Elem::ConvergenceOfElemOps}.

    We have, for \(f\in \CinftyCptSpace[\Bngeq{\eta_1}]\),
    \begin{equation}\label{Eqn::Trace::Forward::ElemRestrictBoundedByElemRestrictElem::SumToBound}
        \BLpNorm{E_k' \TraceMap f}{p}[\Rnmo] 
        =\BLpNorm{E_k' \TraceMap D_j f}{p}[\Rnmo]
        \leq \sum_{j=0}^\infty \BLpNorm{E_k'\TraceMap D_j f}{p}[\Rnmo].
    \end{equation}
    We separate the \(\sum_{j=0}^\infty\)
    on the right-hand side of \eqref{Eqn::Trace::Forward::ElemRestrictBoundedByElemRestrictElem::SumToBound}
    into \(\sum_{j=0}^k\) and \(\sum_{j=k+1}^\infty\).

    Lemma \ref{Lemma::Spaces::Elem::PElem::PElemOpsBoundedOnLp} shows \(\LpOpNorm{E_k'}{p}[\Rnmo]\lesssim 1\), and therefore,
    \begin{equation*}
        \sum_{j=k+1}^\infty \BLpNorm{E_k'\TraceMap D_j f}{p}[\Rnmo]\lesssim \sum_{j=k+1}^\infty \BLpNorm{\TraceMap D_j f}{p}[\Rnmo],
    \end{equation*}
    which is of the desired form \eqref{Eqn::Trace::Forward::ElemRestrictBoundedByElemRestrictElem} (with \(K=1\)).

    Turning to \(\sum_{j=0}^k\), and using Lemma \ref{Lemma::Trace::Forward::ElemRestrictElemAsSmallSumOfElemRestrictElem},
    we have
    \begin{equation*}
        \sum_{j=0}^k \BLpNorm{E_k' \TraceMap D_j f}{p}[\Rnmo] 
        \leq \sum_{l=1}^K \sum_{j=0}^k 2^{-N|j-k|} \BLpNorm{E_{k,l}' \TraceMap E_{j,l}f}{p}[\Rnmo],
    \end{equation*}
    where \(E_{k,l}'\) and \(E_{j,l}\) are as in Lemma \ref{Lemma::Trace::Forward::ElemRestrictElemAsSmallSumOfElemRestrictElem}.
    Lemma \ref{Lemma::Spaces::Elem::PElem::PElemOpsBoundedOnLp} shows \(\LpOpNorm{E_{k,l}'}{p}[\Rnmo]\lesssim 1\) and so
    \begin{equation*}
        \sum_{j=0}^k \BLpNorm{E_k' \TraceMap D_j f}{p}[\Rnmo] 
        \lesssim \sum_{l=1}^K \sum_{j=0}^k 2^{-N|j-k|} \BLpNorm{\TraceMap E_{j,l}f}{p}[\Rnmo],
    \end{equation*}
    which is of the desired form \eqref{Eqn::Trace::Forward::ElemRestrictBoundedByElemRestrictElem}, completing the proof.
\end{proof}

\begin{proof}[Completion of the proof of Proposition \ref{Prop::Traces::CoordReduction::ForwardMap}]
    Take \(\psi'\in \CinftyCptSpace[\Bnmo{\eta_2}]\) with \(\psi'=1\) on a neighborhood of \(\BnmoClosure{\eta_1}\).
    Using Proposition \ref{Prop::Spaces::LP::DjExist}, we may write
    \begin{equation*}
        \Mult{\psi'}=\sum_{k\in \Zgeq} D_k',
    \end{equation*}
    where \(\left\{ \left( D_k',2^{-k} \right):k\in \Zgeq \right\}\in \ElemzVVdv{\Bnmo{\eta_2}}\),
    and the convergence is described in Proposition \ref{Prop::Spaces::Elem::Elem::ConvergenceOfElemOps}.
    \eqref{Eqn::Traces::CoordReduction::ForwardMap::Besov} can be restated as:
    \begin{equation}\label{Eqn::Trace::Forward::BesovRestated}
        \BlqLpNorm{ \left\{ 2^{k(s-\lambda/p)}D_k' \TraceMap f \right\}_{k\in \Zgeq} }{p}{q}[\Rnmo] 
        \lesssim \BesovNorm{f}{s}{p}{q}[\FilteredSheafGenByXXdv],
    \end{equation}
    while \eqref{Eqn::Traces::CoordReduction::ForwardMap::TL} can be restated as:
    \begin{equation}\label{Eqn::Trace::Forward::TLRestated}
        \BlqLpNorm{ \left\{ 2^{k(s-\lambda/p)}D_k' \TraceMap f \right\}_{k\in \Zgeq} }{p}{p}[\Rnmo] 
        \lesssim \TLNorm{f}{s}{p}{q}[\FilteredSheafGenByXXdv].
    \end{equation}
    Here, and in the rest of the proof, \(f\) is an arbitrary element of \(\CinftyCptSpace[\Bngeq{\eta_1}]\).

    Let \(N\in \Zgeq\) with 
    \(N\geq s+1\). By Lemma \ref{Lemma::Trace::Forward::ElemRestrictBoundedByElemRestrictElem},
    \(\exists K\in \Zgeq\) and \(\sE=\left\{ \left( E_{j,k,l}, 2^{-j} \right) : j,k\in \Zgeq, l=1,\ldots,K \right\}\in \ElemzXXdv{\Bngeq{\eta_2}}\)
    with
    \begin{equation}\label{Eqn::Trace::Forward::FinalProof::Tmp1}
        \BLpNorm{D_k' \TraceMap f}{p}[\Rnmo]
        \lesssim \sum_{l=1}^K \sum_{j\in \Zgeq} 2^{-N((k-j)\vee 0)} \BLpNorm{\TraceMap E_{j,k,l}f}{p}[\Rnmo].
    \end{equation}
    Consider for \(s>\lambda/p\), using \eqref{Eqn::Trace::Forward::FinalProof::Tmp1}, \(N\geq s+1\),
    and Lemma \ref{Lemma::Trace::Forward::BoundRestrictElemByIntElemStar}
    we have for \(p\in [1,\infty)\), with \(L\in \Zgeq\) as in Proposition \ref{Prop::Trace::Forward::BoundEjStarByNorm}
    and \(\epsilon_0>0\) as in Lemma \ref{Lemma::Trace::Forward::BoundRestrictElemByIntElemStar},
    \begin{equation}\label{Eqn::Trace::Forward::FinalProof::InitialBoundFinite}
    \begin{split}
         &2^{k(s-\lambda/p)} \BLpNorm{D_k' \TraceMap f}{p}[\Rnmo] 
         \lesssim \sum_{l=1}^K \sum_{j\in \Zgeq} 2^{k(s-\lambda/p)-N((k-j)\vee 0)} \BLpNorm{\TraceMap E_{j,k,l}f}{p}[\Rnmo]
         \\&\lesssim \sum_{l=1}^K \sum_{j\in \Zgeq} 2^{k(s-\lambda/p)-N((k-j)\vee 0)} \left( 2^{j\lambda}\int_{\epsilon_02^{-(j+1)\lambda}}^{\epsilon_02^{-j\lambda}} \BLpNorm{E_{j,k,l}^{*,L} f(\cdot, x_n)}{p}[\Rnmo]^p\: dx_n  \right)^{1/p}
         \\&= \sum_{l=1}^K \sum_{j=0}^\infty 2^{(k-j)(s-\lambda/p)-N((k-j)\vee 0)+js} \left( \int_{\epsilon_0 2^{-(j+1)\lambda}}^{\epsilon_0 2^{-j\lambda}} \BLpNorm{ E_{j,k,l}^{*,L}f(\cdot, x_n)}{p}[\Rnmo]^p \right)^{1/p}
         \\&\lesssim \sum_{l=1}^K \sum_{j=0}^\infty 2^{-\epsilon|j-k|}2^{js}\left( \int_{\epsilon_0 2^{-(j+1)\lambda}}^{\epsilon_0 2^{-j\lambda}} \BLpNorm{ E_{j,k,l}^{*,L}f(\cdot, x_n)}{p}[\Rnmo]^p \right)^{1/p},
    \end{split}
    \end{equation}
    where \(\epsilon:=(s-\lambda/p)\wedge 1>0\).
    When \(p=\infty\), using \eqref{Eqn::Trace::Forward::BoundRestrictElemByIntElemStar::pInfinite} in place of
    \eqref{Eqn::Trace::Forward::BoundRestrictElemByIntElemStar::pFinite}, the same proof shows
    \begin{equation}\label{Eqn::Trace::Forward::FinalProof::InitialBoundInfinite}
    \begin{split}
         &2^{ks} \BLpNorm{D_k' \TraceMap f}{\infty}[\Rnmo] 
         \lesssim \sum_{l=1}^K \sum_{j=0}^\infty 2^{-\epsilon|j-k|}2^{js} \BLpNorm{ E_{j,k,l}^{*,L}f}{\infty}[\Rngeq],
    \end{split}
    \end{equation}
    with \(\epsilon=s\wedge 1>0\).
    For convenience, in the rest of the proof we write \(E_{j,k,l}^{*,L}=0\) if either \(j<0\) or \(k<0\).

    We turn to \eqref{Eqn::Trace::Forward::BesovRestated}. Using \eqref{Eqn::Trace::Forward::FinalProof::InitialBoundFinite}
    and \eqref{Eqn::Trace::Forward::FinalProof::InitialBoundInfinite}, we have for any \(p\in [1,\infty]\),
    \begin{equation}\label{Eqn::Trace::Forward::FinalProof::InitialBoundBoth}
        2^{k(s-\lambda/p)} \BLpNorm{D_k' \TraceMap f}{p}[\Rnmo] 
        \lesssim \sum_{l=1}^K \sum_{j=0}^\infty 2^{-\epsilon|j-k|}2^{js} \BLpNorm{ E_{j,k,l}^{*,L}f}{p}[\Rngeq].
    \end{equation}
    Using \eqref{Eqn::Trace::Forward::FinalProof::InitialBoundBoth}, we have for \(1<p\leq \infty\), \(1\leq q\leq \infty\),
    we have (with the usual modification when \(q=\infty\)),
    \begin{equation*}
    \begin{split}
         &\left( \sum_{k=0}^\infty 2^{k(s-\lambda/p)q} \BLpNorm{D_k' \TraceMap f}{q}[\Rnmo]^q  \right)^{1/q}
         \lesssim \left( \sum_{k=0}^{\infty} \left( \sum_{j=0}^\infty \sum_{l=1}^K 2^{js} 2^{-\epsilon|j-k|} \BLpNorm{ E_{j,k,l}^{*,L}f}{p}[\Rngeq]  \right)^q  \right)^{1/q}
         \\&= \left( \sum_{k=0}^\infty \left( \sum_{j\in \Z} \sum_{l=1}^K 2^{(j+k)s} 2^{-\epsilon|j|} \BLpNorm{E_{j+k,k,l}^{*,L}f}{p}[\Rngeq]   \right)^q  \right)^{1/q}
         \\&\leq \sum_{j\in \Z} \sum_{l=1}^K 2^{-\epsilon|j|}\left( \sum_{k=0}^\infty \left(  2^{(j+k)s}  \BLpNorm{E_{j+k,k,l}^{*,L}f}{p}[\Rngeq]   \right)^q  \right)^{1/q}
         \\&= \sum_{j\in \Z} \sum_{l=1}^K 2^{-\epsilon|j|}\left( \sum_{k\in \Z} \left(  2^{ks}  \BLpNorm{E_{k,k-j,l}^{*,L}f}{p}[\Rngeq]   \right)^q  \right)^{1/q}
         \\&\leq \sum_{j\in \Z} \sum_{l=1}^K 2^{-\epsilon|j|} 
         \sup_{\left\{ \left( E_k,2^{-k} \right) :k\in \Zgeq \right\}\subseteq \sE} 
            \BlqLpNorm{ \left\{ 2^{ks} E_k^{*,L}f \right\}_{k\in \Zgeq} }{p}{q}[\Rngeq]
        \\&\lesssim \sup_{\left\{ \left( E_k,2^{-k} \right) :k\in \Zgeq \right\}\subseteq \sE} 
            \BlqLpNorm{ \left\{ 2^{ks} E_k^{*,L}f \right\}_{k\in \Zgeq} }{p}{q}[\Rngeq]
        \\&\lesssim \BesovNorm{f}{s}{p}{q}[\FilteredSheafGenByXXdv],
    \end{split}
    \end{equation*}
    where the final estimate uses Proposition \ref{Prop::Trace::Forward::BoundEjStarByNorm},
    establishing \eqref{Eqn::Trace::Forward::BesovRestated}.

    Turning to \eqref{Eqn::Trace::Forward::TLRestated}, because \(\TLNorm{f}{s}{p}{\infty}[\FilteredSheafGenByXXdv]\leq \TLNorm{f}{s}{p}{q}[\FilteredSheafGenByXXdv]\),
    \(\forall q\), it suffices to prove \eqref{Eqn::Trace::Forward::TLRestated} with \(q=\infty\).
    Using \eqref{Eqn::Trace::Forward::FinalProof::InitialBoundFinite} (recall, \(1<p<\infty\) in this case),
    we have
    \begin{equation*}
    \begin{split}
         &\left(\sum_{k=0}^\infty 2^{k(s-\lambda/p)p} \BLpNorm{D_k' \TraceMap f}{p}[\Rnmo]^p  \right)^{1/p}
         \\&\lesssim \left( \sum_{k=0}^\infty \left[  \sum_{j=0}^\infty \sum_{l=1}^K 2^{-\epsilon|j-k|+js} \left( \int_{\epsilon_0 2^{-(j+1)\lambda}}^{\epsilon_02^{-j\lambda}} \BLpNorm{E_{j,k,l}^{*,L}f(\cdot, x_n)}{p}[\Rnmo]^p\: dx_n \right)^{1/p} \right]^p  \right)^{1/p}
         \\&=\left( \sum_{k=0}^\infty \left[ \sum_{j\in \Z} \sum_{l=1}^K 2^{-\epsilon|j|+(j+k)s} \left( \int_{\epsilon_0 2^{-(j+k+1)\lambda}}^{\epsilon_0 2^{-(j+k)\lambda}} \BLpNorm{E_{j+k,k,l}^{*,L}f(\cdot,x_n)}{p}[\Rnmo]^p \:dx_n\right)^{1/p}  \right]^p \right)^{1/p}
         \\&\leq \sum_{j\in \Z} \sum_{l=1}^K 2^{-\epsilon|j|} \left( \sum_{k=0}^\infty 2^{(j+k)sp}\int_{\epsilon_0 2^{-(j+k+1)\lambda}}^{\epsilon_0 2^{-(j+k)\lambda}} \BLpNorm{E_{j+k,k,l}^{*,L}f(\cdot, x_n)}{p}[\Rnmo]^p \: dx_n \right)^{1/p}
         \\&\leq \sum_{j\in \Z} \sum_{l=1}^K 2^{-\epsilon|j|} \left( \sum_{k\in \Z} 2^{ksp}\int_{\epsilon_0 2^{-(k+1)\lambda}}^{\epsilon_0 2^{-k\lambda}} \BLpNorm{E_{k,k-j,l}^{*,L}f(\cdot, x_n)}{p}[\Rnmo]^p \: dx_n \right)^{1/p}
         \\&\lesssim \sup_{\left\{ \left( E_k,2^{-k} \right) : k\in \Zgeq\right\}\subseteq \sE} \left( \sum_{k\in \Z} 2^{ksp}\int_{\epsilon_0 2^{-(k+1)\lambda}}^{\epsilon_0 2^{-k\lambda}} \BLpNorm{E_{k}^{*,L}f(\cdot, x_n)}{p}[\Rnmo]^p \: dx_n \right)^{1/p}
         \\&\leq \sup_{\left\{ \left( E_k,2^{-k} \right) : k\in \Zgeq\right\}\subseteq \sE} \left( \sum_{k\in \Z} \int_{\epsilon_0 2^{-(k+1)\lambda}}^{\epsilon_0 2^{-k\lambda}} \int_{\Rnmo}\sup_{k'\in \Zgeq} \left| 2^{k's} E_{k'}^{*,L}f(x',x_n)  \right|^p  \: dx' \: dx_n \right)^{1/p}
         \\&\leq \sup_{\left\{ \left( E_k,2^{-k} \right) : k\in \Zgeq\right\}\subseteq \sE} \left( \int_0^\infty \int_{\Rnmo} \sup_{k'\in \Zgeq} \left| 2^{k's} E_{k'}^{*,L} f(x',x_n)  \right|^p\: dx'\: dx_n \right)^{1/p}
         \\&=\sup_{\left\{ \left( E_k,2^{-k} \right) : k\in \Zgeq\right\}\subseteq \sE} \BLplqNorm{ \left\{ 2^{ks} E_{k}^{*,L}f \right\}_{k\in \Zgeq} }{p}{\infty}[\Rn]
         \lesssim \TLNorm{f}{s}{p}{\infty}[\FilteredSheafGenByXXdv],
    \end{split}
    \end{equation*}
    where the final estimate uses Proposition \ref{Prop::Trace::Forward::BoundEjStarByNorm},
    establishing \eqref{Eqn::Trace::Forward::TLRestated} and completing the proof.
\end{proof}

    \subsection{Proof of the inverse result (Proposition \texorpdfstring{\ref{Prop::Traces::CoordReduction::InverseMap}}{\ref*{Prop::Traces::CoordReduction::InverseMap}})}
    \label{Section::Trace::Inverse}
    In this section, we prove Proposition \ref{Prop::Traces::CoordReduction::InverseMap}.
Fix \(0
<\etaOne
% <\etaTwo
<\etaThree
<\etaFour
<\etaFive<1\).
We will prove Proposition \ref{Prop::Traces::CoordReduction::InverseMap} 
by specializing to the case  \(\etaFour=3/4\), \(\etaThree=1/2\), and \(\etaOne=1/4\).
% replaced
% by \(\etaFour\), \(\etaThree\), and \(\etaOne\), respectively.
The map \(\TraceInverseMap[\sN_L]\) is constructed from a more general class of operators to which we now turn.

    For \(\sE'=\left\{ \left( E_k', 2^{-k} \right) : k\in \Zgeq \right\}\in \ElemzVVdv{\Bnmo{\etaThree}}\), \(\Phi=\left\{ \phi_j \right\}_{j\in \Zgeq}\subset \SchwartzSpaceR\) a bounded sequence,
    \(\gamma\in \CinftyCptSpace[ {[0, (\etaFour-\etaThree)/4)}]\), and \(t\in \R\) define    
    \begin{equation}\label{Eqn::Trace::InverseProof::DefineQt}
        \opQtEPhigamma f(x',x_n)=\sum_{k=0}^\infty \gamma(x_n) 2^{kt}\phi_k(2^{k\lambda}x_n) E_k' f(x').
    \end{equation}
The main properties we use for \(\opQtEPhigamma\) are contained in the next two propositions.

\begin{proposition}\label{Prop::Trace::InverseProof::QtConverges}
    For each \(f\in \Distributions[\Bnmo{\etaThree}]\), the sum in \eqref{Eqn::Trace::InverseProof::DefineQt}
    converges in \(\DistributionsZero[\Bngeq{1}]\) and defines a continuous operator
    \(\opQtEPhigamma:\Distributions[\Bnmo{\etaThree}]\rightarrow \DistributionsZero[\Bngeq{1}]\)
    with \(\supp(\opQtEPhigamma f)\subseteq \BngeqClosure{\etaFour}\), \(\forall f\in \Distributions[\Bnmo{\etaThree}]\).
    For each \(f\in \CinftySpace[\Bnmo{\etaThree}]\), the sum in \eqref{Eqn::Trace::InverseProof::DefineQt}
    converges in \(\CinftyCptSpace[\Bngeq{\etaFour}]\) and defines a continuous operator
    \(\opQtEPhigamma:\CinftySpace[\Bnmo{\etaThree}]\rightarrow \CinftyCptSpace[\Bngeq{\etaFour}]\).
\end{proposition}

\begin{proposition}\label{Prop::Trace::InverseProof::QtIsContinuous}
    For \(s\in \R\), the following maps are continuous:
    \begin{equation*}
        \opQtEPhigamma: \BesovSpace{s+t}{p}{q}[\BnmoClosure{\etaOne}][\VVdv]
        \rightarrow 
        \BesovSpace{s+\lambda/p}{p}{q}[\BngeqClosure{\etaFour}][\XXdv], \quad 1\leq p,q\leq\infty.
    \end{equation*}
    \begin{equation*}
        \opQtEPhigamma: \BesovSpace{s+t}{p}{p}[\BnmoClosure{\etaOne}][\VVdv]
        \rightarrow 
        \TLSpace{s+\lambda/p}{p}{q}[\BngeqClosure{\etaFour}][\XXdv], \quad 1<p<\infty, \quad 1<q\leq \infty.
    \end{equation*}
\end{proposition}

Before we prove Propositions \ref{Prop::Trace::InverseProof::QtConverges} and \ref{Prop::Trace::InverseProof::QtIsContinuous},
we see how they imply Proposition \ref{Prop::Traces::CoordReduction::InverseMap}.

\begin{proof}[Proof of Proposition \ref{Prop::Traces::CoordReduction::InverseMap}]
    Let \(\phi_0,\phi_1,\ldots, \phi_L\in \SchwartzSpace[\R]\) be such that
    \begin{equation}\label{Eqn::Trace::InverseProof::ChooseFinalphi}
        \partial^j \phi_k(0)
        =
        \begin{cases}
            1, &j=k, 0\leq j,k\leq L,\\
            0,&j\ne k, 0\leq j,k\leq L.
        \end{cases}
    \end{equation}
    Fix \(\psi\in \CinftyCptSpace[\Bnmo{\etaThree}]\) with \(\psi=1\) on a neighborhood of \(\BnmoClosure{\etaOne}\).
        Using Proposition \ref{Prop::Spaces::LP::DjExist}, we may write
    \begin{equation*}
        \Mult{\psi'}=\sum_{k\in \Zgeq} D_k',
    \end{equation*}
    where \(\sD_0':=\left\{ \left( D_k',2^{-k} \right):k\in \Zgeq \right\}\in \ElemzVVdv{\Bnmo{\etaThree}}\),
    and the convergence is described in Proposition \ref{Prop::Spaces::Elem::Elem::ConvergenceOfElemOps}.
    Let \(\gamma\in \CinftyCptSpace[ {[0, (\etaFour-\etaThree)/4)}]\) equal \(1\) on a neighborhood of \(0\),
    and set \(\Phi_l=\{\phi_l\}\) (i.e., \(\Phi_l\) denotes the constant sequence, \(\phi_l,\phi_l,\ldots\)).
    Set,
    \begin{equation*}
        \TraceInverseMap[\sN_L] \left( f_0, f_1,\ldots, f_L \right):= \sum_{l=0}^L \gamma(x_n) 2^{-l\lambda k} \phi_l(2^{k\lambda}x_n) D_k' f_l(x')
        =\sum_{l=0}^L \opQtEPhigamma[-l\lambda][\sD_0'][\Phi_l][\gamma]f_l. 
    \end{equation*}
    Proposition \ref{Prop::Trace::InverseProof::QtConverges} implies
    \(\TraceInverseMap[\sN_L]:\Distributions[\Bnmo{\etaThree}]\rightarrow \DistributionsZero[\Bngeq{\etaFour}]\)
    and \(\TraceInverseMap[\sN_L]:\CinftySpace[\Bnmo{\etaThree}]\rightarrow \CinftyCptSpace[\Bngeq{\etaFour}]\)
    are continuous; this establishes \eqref{Eqn::Traces::CoordReduction::InverseMap::BasicMapping} and 
    \ref{Item::Traces::CoordReduction::InverseMap::MappingOnSmoothFunctions} (by taking \(\etaFour=3/4\) and
    \(\etaThree=1/2\)). Taking \(\etaOne=1/4\), \ref{Item::Traces::CoordReduction::InverseMap::MappingOnBesovAndTL}
    follows from Proposition \ref{Prop::Trace::InverseProof::QtIsContinuous}.
    All that remains to show is \ref{Item::Traces::CoordReduction::InverseMap::IsInverse}.

    For \((f_0,f_1,\ldots, f_L)\in \CinftyCptSpace[\Bnmo{\etaOne}]^{L+1}\), \(0\leq l_0\leq L\), we have using
    Proposition \ref{Prop::Trace::InverseProof::QtConverges} and \eqref{Eqn::Trace::InverseProof::ChooseFinalphi},
    \begin{equation*}
        \begin{split}
        \partial_{x_n}^{l_0}\big|_{x_n=0} \TraceInverseMap[\sN_L] \left( f_0,f_1,\ldots, f_L \right)(x',x_n)
        &=\sum_{k\in \Zgeq} \sum_{l=0}^L 2^{(l_0-l)\lambda k} \left( \partial_{x_n}^{l_0}\phi_l \right)(0) D_k'f_l(x')
        \\&=\sum_{k\in \Zgeq} D_k' f_{l_0}(x') = \psi f_{l_0}=f_{l_0}.
        \end{split}
    \end{equation*}
    We conclude \(\TraceMap[\sN_L]\TraceInverseMap[\sN_L](f_0,f_1,\ldots, f_L)=(f_0,f_1,\ldots, f_L)\), establishing
    \ref{Item::Traces::CoordReduction::InverseMap::IsInverse} and completing the proof.
\end{proof}

We turn to the proof of Proposition \ref{Prop::Trace::InverseProof::QtConverges}.
For it, we need a lemma.

\begin{lemma}\label{Lemma::Trace::InverseProof::ElemOnBoundedSets}
    Let \(\sE'\in \ElemzVVdv{\Bnmo{\etaThree}}\), \(\sB\subset \CinftySpace[\Bngeq{1}]\) be a bounded set, \(\Phi=\left\{ \phi_j \right\}_{j\in \Zgeq}\subset \SchwartzSpace[\R]\) a bounded sequence,
            \(\gamma\in \CinftyCptSpace[{[0, (\etaFour-\etaThree)/4)}]\).  Then, \(\forall N\in \Zgeq\),
            \begin{equation*}
                \left\{ 2^{Nk} \gamma(x_n) \phi_k(2^{k\lambda}x_n) E_k' \zeta(x',x_n) : (E_k',2^{-k})\in \sE', \zeta\in \sB \right\}\subset \CinftyCptSpace[\Bngeq{\etaFour}]
            \end{equation*}
            is a bounded set. Here, \(E_k'\) acts on \(\zeta(x',x_n)\) in just the \(x'\)-variable.
\end{lemma}
\begin{proof}
    Using \(\supp(E_k')\subseteq \Bnmo{\etaThree}\times \Bnmo{\etaThree}\)
    and the support of \(\gamma\), we have
    \begin{equation}\label{Eqn::Trace::InverseProof::ElemOnBoundedSets::SupportEstimate}
        \bigcup_{\substack{(E_k',2^{-k})\in \sE'\\ \zeta\in \sB}} \supp\left( \gamma(x_n) \phi_k(2^{-k}x_n) E_k'\zeta(x',x_n)) \right)\Subset \Bngeq{\etaFour}.
    \end{equation}
    In light of Assumptions \ref{Item::Traces::CoordReduction::IsPartialxn}, \ref{Item::Traces::CoordReduction::TangentToBoundary}, and
    \ref{Item::Traces::CoordReduction::Span}
    in Section \ref{Section::Trace::ReductionToCoord}, since \(V_j=X_j\big|_{\Bnmo{1}}\), \(1\leq j\leq q\),
    we have \(\Span\left\{ V_j(x'):1\leq j\leq q \right\}=\TangentSpace{x'}{\Rnmo}\), \(\forall x'\in \Bnmo{1}\).
    Combining this with \eqref{Eqn::Trace::InverseProof::ElemOnBoundedSets::SupportEstimate},
    to prove the result, it suffices to show
    \(\forall N\in \Zgeq\), \(\forall \alpha,\beta\),
    \begin{equation}\label{Eqn::Trace::InverseProof::ElemOnBoundedSets::ToShowBound}
        \sup_{\zeta\in \sB} \sup_{x',x_n} \sup_{\left( E_k',2^{-k} \right)\in \sE'} \left| 2^{Nk} \partial_{x_n}^{\beta} V^{\alpha} \gamma(x_n) \phi(2^{k\lambda}x_n) E_k' \zeta(x',x_n) \right|<\infty.
    \end{equation}
    
    We claim that it suffices to prove \eqref{Eqn::Trace::InverseProof::ElemOnBoundedSets::ToShowBound}
    in the case \(|\alpha|=|\beta|=0\). Indeed,
    \begin{equation}\label{Eqn::Trace::InverseProof::ElemOnBoundedSets::DistributeDerivs}
        \begin{split}
            &2^{Nk} \partial_{x_n}^{\beta} V^{\alpha} \gamma(x_n) \phi(2^{k\lambda}x_n) E_k' \zeta(x',x_n)
            \\&= \sum_{\beta_1+\beta_2+\beta_3=\beta} c_{\beta_1,\beta_2,\beta_3}
            2^{(N+|\beta_2|\lambda+\DegVdv{\alpha})k} \left( \partial_{x_n}^{\beta_1}\gamma \right)(x_n)
            \left( \partial_{x_n}^{\beta_2}\phi_k \right)(2^{k\lambda}x_n)
            \left( \left( 2^{-k\Vdv}V \right)^{\alpha} E_k' \partial_{x_n}^{\beta_3}\zeta \right)(x',x_n).
        \end{split}
    \end{equation}
    By Proposition \ref{Prop::Spaces::Elem::Elem::MainProps} \ref{Item::Spaces::Elem::Elem::DerivOp}, for fixed \(\alpha\),
    \(\left\{ \left( \left( 2^{-k\Vdv}V \right)^{\alpha}E_k',2^{-k} \right) : \left( E_k',2^{-k} \right)\in \sE' \right\}\in \ElemzVVdv{\Bnmo{\etaThree}}\).
    It follows that the right-hand side of \eqref{Eqn::Trace::InverseProof::ElemOnBoundedSets::DistributeDerivs} is a linear combination
    of terms of the same form as those in \eqref{Eqn::Trace::InverseProof::ElemOnBoundedSets::ToShowBound}, with different
    choices of \(N\), and with \(|\alpha|=|\beta|=0\). Thus, to establish 
    \eqref{Eqn::Trace::InverseProof::ElemOnBoundedSets::ToShowBound} it suffices to instead show
    \begin{equation}\label{Eqn::Trace::InverseProof::ElemOnBoundedSets::ToShowBound::2}
        \sup_{\zeta\in \sB} \sup_{x',x_n} \sup_{\left( E_k',2^{-k} \right)\in \sE'} \left| 2^{Nk}  \gamma(x_n) \phi(2^{k\lambda}x_n) E_k' \zeta(x',x_n) \right|<\infty.
    \end{equation}

    Next, we claim that it suffices to prove \eqref{Eqn::Trace::InverseProof::ElemOnBoundedSets::ToShowBound::2}
    in the case \(N=0\). Indeed, for \(N\in \Zgeq\) fixed,
    Proposition \ref{Prop::Spaces::Elem::Elem::MainProps} \ref{Item::Spaces::Elem::Elem::PullOutNDerivs} shows that we may write
    \begin{equation}\label{Eqn::Trace::InverseProof::ElemOnBoundedSets::ToShowBound::Tmp1}
        E_k' \zeta(x',x_n) = \sum_{|\alpha|\leq N} 2^{(|\alpha|-N-\DegVdv{\alpha})k}E_{k,\alpha}' V^{\alpha} \zeta(x',x_n),
    \end{equation}
    where \(\left\{ \left( E_{k,\alpha}', 2^{-k} \right) : \left( E_k,2^{-k} \right)\in \sE,|\alpha|\leq N \right\}\in \ElemzVVdv{\Bnmo{\etaThree}}\).
    Since \(\left\{ V^{\alpha}\zeta :\zeta\in \sB, |\alpha|\leq N \right\}\subset \CinftySpace[\Bngeq{1}]\)
    is a bounded set, and since \(\DegVdv{\alpha}\geq|\alpha|\), \(\forall \alpha\),
    plugging \eqref{Eqn::Trace::InverseProof::ElemOnBoundedSets::ToShowBound::Tmp1}
    into \eqref{Eqn::Trace::InverseProof::ElemOnBoundedSets::ToShowBound::2} shows that it suffices to prove
    \eqref{Eqn::Trace::InverseProof::ElemOnBoundedSets::ToShowBound::2} with \(N=0\). I.e., we wish to show
    \begin{equation}\label{Eqn::Trace::InverseProof::ElemOnBoundedSets::ToShowBound::3}
        \sup_{\zeta\in \sB} \sup_{x',x_n} \sup_{\left( E_k',2^{-k} \right)\in \sE'} \left|  \gamma(x_n) \phi(2^{k\lambda}x_n) E_k' \zeta(x',x_n) \right|<\infty.
    \end{equation}
    Since \(\supp(E_k')\subseteq \Bnmo{\etaThree}\times \Bnmo{\etaThree}\) (see \eqref{Eqn::Spaces::LP::ElemzFOmegaIsUnion}),
    and since \(|\phi_k|\lesssim 1\), Lemma \ref{Lemma::Spaces::Elem::PElem::PElemOpsBoundedOnLp} (with \(p=\infty\))
    implies
    \begin{equation*}
        \left|  \gamma(x_n) \phi(2^{k\lambda}x_n) E_k' \zeta(x',x_n) \right|
        \lesssim \sup_{\substack{x_n\in [0,\infty) \\ x'\in \Bnmo{\etaThree}}}
        \left| \gamma(x_n)\zeta(x',x_n) \right|\lesssim 1,
    \end{equation*}
    establishing \eqref{Eqn::Trace::InverseProof::ElemOnBoundedSets::ToShowBound::3} and completing the proof.
\end{proof}

\begin{proof}[Proof of Proposition \ref{Prop::Trace::InverseProof::QtConverges}]
    We begin by showing the convergence on \(\CinftySpace[\Bnmo{\etaFour}]\).
    Since \(\sE'\in\ElemzVVdv{\Bnmo{\etaThree}}\), there is a common compact set \(\Compact\Subset \Bnmo{\etaThree}\)
    with \(\supp(E_k')\subset \Compact\times \Compact\), \(\forall (E_k',2^{-k})\in \sE\).
    Fix \(\psi'\in \CinftyCptSpace[\Bnmo{\etaThree}]\) with \(\psi'=1\) on a neighborhood of \(\Compact\).

    Let \(\sB'\subset \CinftySpace[\Bnmo{\etaThree}]\) be a bounded set.
    For \(f\in \sB'\), set
    \begin{equation*}
        G_k[f](x',x_n):=2^{k(t+1)}\gamma(x_n) \phi_k(2^{k\lambda}x_n) E_k' f(x')=2^{k(t+1)}\gamma(x_n) \phi_k(2^{k\lambda}x_n) E_k'\psi'f(x')
    \end{equation*}
    Using Lemma \ref{Lemma::Trace::InverseProof::ElemOnBoundedSets} %\ref{Item::Trace::InverseProof::ElemOnBoundedSets::nmo}
    with \(\zeta(x',x_n)=\psi'(x')f(x')\), we see
    \begin{equation*}
        \left\{ G_k[f] : k\in \Zgeq, f\in \sB \right\}\subset \CinftyCptSpace[\Bngeq{\etaFour}]
    \end{equation*}
    is a bounded set.
    Therefore,
    \begin{equation*}
        \sum_{k=0}^\infty \gamma(x_n) 2^{kt} \phi_k(2^{k\lambda} x_n) E_k' f(x') = \sum_{k=0}^\infty 2^{-k} G_k[f](x',x_n)
    \end{equation*}
    converges in \(\CinftyCptSpace[\Bngeq{\etaFour}]\), uniformly for \(f\in \sB'\).
    The claimed convergence and continuity of  \(\opQtEPhigamma:\CinftySpace[\Bnmo{\etaThree}]\rightarrow \CinftyCptSpace[\Bngeq{\etaFour}]\)
    follows.

    We turn to the convergence on \(\Distributions[\Bnmo{\etaThree}]\).
    Let \(\sB'\subset \Distributions[\Bnmo{\etaThree}]\) and \(\sB\subset \CinftySpace[\Bngeq{1}]\)
    be bounded sets. We will show
    \begin{equation}\label{Eqn::Trace::InverseProof::ShowDistributionConvergence}
        \sum_{k=0}^\infty \iint \overline{\zeta(x',x_n)} \gamma(x_n) \phi_k(2^{k\lambda}x_n) 2^{kt} E_k' f(x')\: dx'\: dx_n
    \end{equation}
    converges absolutely and uniformly for \(f\in \sB'\) and \(\zeta\in \sB\); from this convergence the claimed convergence in the proposition
    and continuity  \(\opQtEPhigamma:\Distributions[\Bnmo{\etaThree}]\rightarrow \DistributionsZero[\Bngeq{\etaFour}]\)
    follow. In fact, this convergence of \eqref{Eqn::Trace::InverseProof::ShowDistributionConvergence} is somewhat stronger than
    is needed (for example, we could take \(\sB\subset \TestFunctionsZero[\Bngeq{1}]\), but do not require this).

    We have
    \begin{equation}\label{Eqn::Trace::InverseProof::ShowDistributionConvergence::Tmp1}
    \begin{split}
         &\sum_{k=0}^\infty \iint \overline{\zeta(x',x_n)} \gamma(x_n) \phi_k (2^{k\lambda}x_n) 2^{kt} E_k'f(x')\: dx'\: dx_n
         \\&=\sum_{k=0}^\infty \iint \overline{ 2^{kt}\overline{\gamma(x_n) \phi(2^{k\lambda}x_n)} \left( E_k' \right)^{*} \zeta(x',x_n)  } f(x')\: dx'\: dx_n.
    \end{split}
    \end{equation}
    We set,
    \begin{equation*}
        G_k[\zeta](x',x_n):= 2^{k(t+1)}\overline{\gamma(x_n) \phi(2^{k\lambda}x_n)} \left( E_k' \right)^{*} \zeta(x',x_n).
    \end{equation*}
    Using \cite[Lemma 5.5.6]{StreetMaximalSubellipticity}, we have \(\left\{ \left( \left( E_k' \right)^{*}, 2^{-k} \right) : (E_k', 2^{-k})\in \sE \right\}\in \ElemzVVdv{\Bnmo{\etaThree}}\)
    (here we are using that \(\Bnmo{1}\) is a manifold without boundary, so the results of \cite{StreetMaximalSubellipticity} apply).
    Therefore, Lemma \ref{Lemma::Trace::InverseProof::ElemOnBoundedSets} shows
    \begin{equation}\label{Eqn::Trace::InverseProof::ShowDistributionConvergence::Tmp2}
        \left\{ G_k[\zeta] : k\in \Zgeq,\zeta\in \sB \right\}\subset \CinftyCptSpace[\Bngeq{\etaFour}]
    \end{equation}
    is a bounded set. Plugging the definition of \(G_k[\zeta]\) into \eqref{Eqn::Trace::InverseProof::ShowDistributionConvergence::Tmp1}
    we see that \eqref{Eqn::Trace::InverseProof::ShowDistributionConvergence} is equal to
    \begin{equation*}
        \sum_{k=0}^\infty 2^{-k}\iint G_k[\zeta](x',x_n)\: dx_n\: f(x')\: dx',
    \end{equation*}
    which converges absolutely and uniformly for \(\zeta\in \sB\) and \(f\in \sB'\) since
    \eqref{Eqn::Trace::InverseProof::ShowDistributionConvergence::Tmp2} is a bounded set, completing the proof
    of the convergence of \eqref{Eqn::Trace::InverseProof::ShowDistributionConvergence}.

    Finally, we claim \(\supp(\opQtEPhigamma f)\subseteq \BngeqClosure{\etaFour}\), \(\forall f\in \Distributions[\Bnmo{\etaThree}]\).
    Indeed, since \(\supp(E_k' f)\subseteq \Bnmo{\etaThree}\) (since \(\sE'\in \ElemzVVdv{\Bnmo{\etaThree}}\)),
    and \(\supp(\gamma(x_n))\subseteq {[0, (\etaFour-\etaThree)/4)}\) this follows from the formula \eqref{Eqn::Trace::InverseProof::DefineQt}.
\end{proof}

We complete this section with the proof of Proposition \ref{Prop::Trace::InverseProof::QtIsContinuous}.
We begin with some formulas relating \(X_j\) and \(V_j\).

\begin{lemma}\label{Lemma::Trace::InverseProof::XjRelatedToVj}
    \(\exists N_0\in \Zgeq\), \(\forall j\in \left\{ 1,\ldots, q \right\}\),
    \begin{equation}\label{Eqn::Trace::InverseProof::XjRelatedToVj::XjInTermsOfVj}
        X_j(x',x_n)=\sum_{l=0}^{N_0} \sum_{\substack{\Vdv_k\leq \Xdv_j+l\lambda \\ 1\leq k,\leq q}} a_{j,l}^k (x',x_n) x_n^l V_k(x') + b_j(x',x_n) x_n\partial_{x_n},
    \end{equation}
    \begin{equation}\label{Eqn::Trace::InverseProof::XjRelatedToVj::VjInTermsOfXj}
        V_j(x')=\sum_{l=0}^{N_0} \sum_{\substack{\Xdv_k\leq \Vdv_j+l\lambda \\ 1\leq k\leq q+1}} X_k(x',x_n) \at_{j,l}^k(x',x_n)x_n^l +g_j(x',x_n),
    \end{equation}
    where \(a_{j,l}^k, \at_{j,l}^k, b_j, g_j\in \CinftySpace[\Bngeq{1}]\), and the functions are treated as multiplication operators.
\end{lemma}
\begin{proof}
    Throughout this proof, all functions of \(x',x_n\) are in \(\CinftySpace[\Bngeq{1}]\), while all functions of \(x'\)
    are in \(\CinftySpace[\Bnmo{1}]\).
    For \(1\leq j\leq q\) and \(N_0=\max\left\{ \Xdv_1,\ldots, \Xdv_{q+1} \right\}\geq \max\left\{ \Vdv_1,\ldots, \Vdv_q \right\}\),
    we have by Taylor's Theorem,
    \begin{equation}\label{Eqn::Trace::InverseProof::XjRelatedToVj::Tmp1}
        X_j(x',x_n)= X_j(a',0)+\sum_{l=1}^{N_0-1} \frac{x_n^l}{l!} \left( \ad[\partial_{x_n}]^l X_j \right) ( x',0) +x_n^{N_0} B_j(x',x_n)\cdot \grad_{(x',x_n)}
    \end{equation}
    where \(B_j\in \CinftySpace[\Bngeq{1};\R^n]\).

    Applying Assumption \ref{Item::Traces::CoordReduction::NSW} at the start of Section \ref{Section::Trace::ReductionToCoord},
    \(l\) times, and using \(\partial_{x_n}=X_{q+1}\) and \(\lambda=\Xdv_{q+1}\) shows for \(l\geq 1\),
    \begin{equation}\label{Eqn::Trace::InverseProof::XjRelatedToVj::Tmp2}
        \left( \ad[\partial_{x_n}]^l X_j \right) (x',0)=\sum_{\substack{\Xdv_k\leq \Xdv_j+l\lambda\\ 1\leq k\leq q+1}} h_{j,l}^k X_k(x',0)
        =\sum_{\substack{\Vdv_k\leq \Xdv_j+l\lambda \\ 1\leq k\leq q}} h_{j,l}^k(x') V_k(x')+ h_{j,l}^{q+1}(x')\partial_{x_n}, 
    \end{equation}
    where in the final equality, we used \(X_k(x',0)=V_k(x')\), \(1\leq k\leq q\).
    Plugging \eqref{Eqn::Trace::InverseProof::XjRelatedToVj::Tmp2} into \eqref{Eqn::Trace::InverseProof::XjRelatedToVj::Tmp1}
    and using \(N_0\geq \max\left\{ \Vdv_1,\ldots, \Vdv_q \right\}\), and \(V_k(x')=X_k(x',0)\) for \(1\leq k\leq q\) gives
    \begin{equation}\label{Eqn::Trace::InverseProof::XjRelatedToVj::Tmp3}
        X_j(x',x_n)= V_j(x')+\sum_{l=1}^{N_0} \sum_{\substack{\Vdv_k\leq \Xdv_j+l\lambda \\ 1\leq k\leq q}} x_n^l a_{j,l}^k(x',x_n) V_k(x')+b_j(x',x_n)x_n\partial_{x_n}.
    \end{equation}
    Rewriting \eqref{Eqn::Trace::InverseProof::XjRelatedToVj::Tmp3} gives
    \begin{equation}\label{Eqn::Trace::InverseProof::XjRelatedToVj::Tmp4}
        V_j(x')= X_j(x',x_n)-\sum_{l=1}^{N_0} \sum_{\substack{\Vdv_k\leq \Xdv_j+l\lambda \\ 1\leq k\leq q}} x_n^l a_{j,l}^k(x',x_n) V_k(x')-b_j(x',x_n)x_n\partial_{x_n}.
    \end{equation}
    \eqref{Eqn::Trace::InverseProof::XjRelatedToVj::Tmp4} gives a formula for \(V_j\); plugging this formula in for
    \(V_k\) on the right-hand side of \eqref{Eqn::Trace::InverseProof::XjRelatedToVj::Tmp4} and repeating this process \(N_0\)
    times gives
    \begin{equation}\label{Eqn::Trace::InverseProof::XjRelatedToVj::Tmp5}
        V_j(x')=X_j(x',x_n) + \sum_{l=1}^{N_0} \sum_{\substack{\Xdv_k\leq \Xdv_j+l\lambda \\ 1\leq k\leq q}} x_n^l \ah_{j,l}^k(x',x_n) X_k(x',x_n) 
        +\sum_{1\leq k\leq q} x_n^{N_0} c_j^k(x',x_n) V_k(x') + \bt(x',x_n)x_n \partial_{x_n}.
    \end{equation}
    Using 
     Assumption \ref{Item::Traces::CoordReduction::Span} at the start of Section \ref{Section::Trace::ReductionToCoord}
     and the fact that \(N_0=\max\left\{ \Xdv_1,\ldots, \Xdv_q \right\}\), we have
     \begin{equation}\label{Eqn::Trace::InverseProof::XjRelatedToVj::Tmp6}
        \sum_{1\leq k\leq q} x_n^{N_0} c_{j}^k(x',x_n)V_k(x') = \sum_{\substack{\Xdv_k\leq N_0\lambda \\ 1\leq k\leq q+1}} x_n^{N_0} \ch_{j}^k(x',x_n) X_k(x',x_n).
    \end{equation}
    Using that \(\partial_{x_n}=X_{q+1}\) and \(\Xdv_{q+1}=\lambda\), 
    and plugging \eqref{Eqn::Trace::InverseProof::XjRelatedToVj::Tmp6} into \eqref{Eqn::Trace::InverseProof::XjRelatedToVj::Tmp5},
    we have
    \begin{equation}\label{Eqn::Trace::InverseProof::XjRelatedToVj::Tmp7}
        V_j(x')=X_j(x',x_n) + \sum_{l=1}^{N_0} \sum_{\substack{\Xdv_k\leq \Xdv_j+l\lambda \\ 1\leq k\leq q+1}} x_n^l \at_{j,l}^k(x',x_n) X_k(x',x_n).
    \end{equation}
    Commuting \(x_n^l \at_{j,l}^k(x',x_n)\) past \(X_k\) gives
    \begin{equation}\label{Eqn::Trace::InverseProof::XjRelatedToVj::Tmp8}
        V_j(x')= X_j(x',x_n) + \sum_{l=1}^{N_0} \sum_{\substack{\Xdv_k\leq \Xdv_j+l\lambda \\ 1\leq k\leq q+1}} X_k(x',x_n) x_n^l \at_{j,l}^k(x',x_n)+g_j(x',x_n).
    \end{equation}
    \eqref{Eqn::Trace::InverseProof::XjRelatedToVj::XjInTermsOfVj} and \eqref{Eqn::Trace::InverseProof::XjRelatedToVj::VjInTermsOfXj}
    follow by extending the sum to \(l=0\) in 
    \eqref{Eqn::Trace::InverseProof::XjRelatedToVj::Tmp3} and \eqref{Eqn::Trace::InverseProof::XjRelatedToVj::Tmp8},
    respectively.
\end{proof}

\begin{notation}\label{Notation::Trace::InverseProof::ZeroIfNegativeIndex}
    For the remainder of this section, if we have an operator or a function indexed by \(j\in \Zgeq\) (for example \(E_j\)),
    then for \(j\in \Z\setminus \Zgeq\) we set \(E_j=0\).
\end{notation}

\begin{lemma}\label{Lemma::Trace::InverseProof::ElemAppliedToOpQ}
    Let \(\sE'=\left\{ \left( E_k', 2^{-k} \right) : k\in \Zgeq \right\}\in \ElemzVVdv{\Bnmo{\etaThree}}\), \(\Phi=\left\{ \phi_j \right\}_{j\in \Zgeq}\subset \SchwartzSpaceR\) a bounded sequence,
    \(\gamma\in \CinftyCptSpace[ {[0, (\etaFour-\etaThree)/4)}]\), and \(t\in \R\)
    be as in the definition of \(\opQtEPhigamma\) (see \eqref{Eqn::Trace::InverseProof::DefineQt}),
    and let \(\sE\in \ElemzXXdv{\Bngeq{\etaFive}}\).  Then, \(\forall N\in \Zgeq\),
    \(\exists L\in \Zgeq\), \(\forall (E_j, 2^{-j})\in \sE\) with \(j\in \Zgeq\),
    \begin{equation*}
        E_j \opQtEPhigamma = \sum_{l=1}^L \sum_{k\in \Z} 2^{-N|k|} 2^{(j+k)t} E_{j,l}\gamma_l(x_n) \phi_{j,k,l}\left( 2^{(j+k)\lambda}x_n \right)E_{j+k,l}',
    \end{equation*}
    where \(\gamma_1,\ldots, \gamma_L\in \CinftyCptSpace[{[0,(\etaFour-\etaThree)/4)}]\),
    \(\left\{ \phi_{j,k,l} : j\in \Zgeq, k\in \Z, l=1,\ldots,L \right\}\subset \SchwartzSpaceR\) is a bounded set,
    \(\left\{ \left( E_{j,l},2^{-j} \right) :(E_j,2^{-j})\in \sE, j\in \Zgeq, l=1,\ldots, L \right\}\in \ElemzXXdv{\Bngeq{\etaFive}}\),
    \(E_{j+k,l}'=0\) if \(j+k<0\) as in Notation \ref{Notation::Trace::InverseProof::ZeroIfNegativeIndex},
    and \(\left\{ \left( E_{j+k,l}', 2^{-j-k} \right) : j\in \Zgeq, k\in \Z, j+k\geq 0, l=1,\ldots, L \right\}\in \ElemzVVdv{\Bnmo{\etaThree}}\).
\end{lemma}
\begin{proof}
    Using the formula for \(\opQtEPhigamma\) (see \eqref{Eqn::Trace::InverseProof::DefineQt}) and 
    Proposition \ref{Prop::Trace::InverseProof::QtConverges}, we have by reindexing the sum \eqref{Eqn::Trace::InverseProof::DefineQt}
    and using Notation \ref{Notation::Trace::InverseProof::ZeroIfNegativeIndex},
    \begin{equation}\label{Eqn::Trace::InverseProof::ElemAppliedToOpQ::RewriteSum}
        E_j \opQtEPhigamma =\sum_{k\in \Z} 2^{(j+k)t} E_j \gamma(x_n) \phi_k\left( 2^{(j+k)\lambda}x_n \right)E_{j+k}',
    \end{equation}
    where both sides of \eqref{Eqn::Trace::InverseProof::ElemAppliedToOpQ::RewriteSum} are thought of as acting on
    functions (or distributions) on \(\Bnmo{\etaThree}\) and yield functions or distributions on \(\Bngeq{\etaFive}\).

    Thus, to prove the result, it suffices to show
    \begin{equation}\label{Eqn::Trace::InverseProof::ElemAppliedToOpQ::ToShow}
        E_j \gamma(x_n) \phi_k\left( 2^{(j+k)\lambda}x_n \right)E_{j+k}'
        =2^{-N|k|} \sum_{l=1}^L E_{j,l}\gamma_l(x_n) \phi_{j,k,l}\left( 2^{(j+k)\lambda}x_n \right) E_{j+k,l}',
    \end{equation}
    where each of the above objects is as in the statement of the lemma; in short, we wish to show
    \( E_j \gamma(x_n) \phi_k\left( 2^{(j+k)\lambda}x_n \right)E_{j+k}'\) equals \(2^{-N|k|}\)
    times a sum of terms of the same form.
    It suffices to prove \eqref{Eqn::Trace::InverseProof::ElemAppliedToOpQ::ToShow} in the case \(N=1\),
    and the general result follows by iterating the case \(N=1\).
    We separate the proof of \eqref{Eqn::Trace::InverseProof::ElemAppliedToOpQ::ToShow} in the case \(N=1\)
    into two parts: \(k<0\) and \(k\geq 0\).

    We begin with the case \(k<0\). When \(|k|>j\), then by Notation \ref{Notation::Trace::InverseProof::ZeroIfNegativeIndex},
    \(E_{j+k}'=0\) and the left-hand side of \eqref{Eqn::Trace::InverseProof::ElemAppliedToOpQ::ToShow} is zero, so there is
    nothing to show.  We henceforth assume \(|k|\leq j\).

    We apply Proposition \ref{Prop::Spaces::Elem::Elem::MainProps} \ref{Item::Spaces::Elem::Elem::PullOutNDerivs} (with \(N=1\))
    to write
    \begin{equation}\label{Eqn::Trace::InverseProof::ElemAppliedToOpQ::Tmp1}
        E_j=\sum_{|\alpha|\leq 1} 2^{(|\alpha|-1)j}E_{j,\alpha} \left( 2^{-j\Xdv}X \right)^{\alpha},
    \end{equation}
    where \(\left\{ \left( E_{j,\alpha},2^{-j} \right) : |\alpha|\leq 1, \left( E_j,2^{-j} \right)\in \sE \right\}\in \ElemzXXdv{\Bngeq{\etaFive}}\).
    We plug \eqref{Eqn::Trace::InverseProof::ElemAppliedToOpQ::Tmp1} into the left-hand side of \eqref{Eqn::Trace::InverseProof::ElemAppliedToOpQ::ToShow}
    and show that each term is of the desired form.

    When \(|\alpha|=0\), using \(|k|\leq j\), we have
    \begin{equation*}
        2^{-j} E_{j,0} \gamma(x_n) \phi_{j+k}\left( 2^{(j+k)\lambda}x_n\right) E_{j+k}'
        =2^{-|k|} E_{j,0}\gamma(x_n) \left(2^{|k|-j} \phi_{j+k} \right)\left( 2^{(j+k)\lambda}x_n\right) E_{j+k}'
    \end{equation*}
    is of the desired form.  
    
    Turning to the case \(|\alpha|=1\), we have
    \begin{equation}\label{Eqn::Trace::InverseProof::ElemAppliedToOpQ::Tmp2}
        \left( 2^{-j\Xdv}X \right)^{\alpha} =
        \begin{cases}
            2^{-j\Xdv_l}X_l, &\text{for some }1\leq l\leq q,\text{ or,}\\
            2^{-j\lambda}\partial_{x_n}.
        \end{cases}
    \end{equation}
    We deal with the two cases in \eqref{Eqn::Trace::InverseProof::ElemAppliedToOpQ::Tmp2} separately.
    We first address the case when \(\left( 2^{-j\Xdv}X \right)^{\alpha}=2^{-j\Xdv_l}X_l\) for some \(1\leq l\leq q\).
    Without loss of generality, we take \(l=1\); so we consider terms of the form
    \begin{equation}\label{Eqn::Trace::InverseProof::ElemAppliedToOpQ::Tmp3}
        E_{j,1} 2^{-j\Xdv_1} X_1 \gamma(x_n) \phi_{j+k}\left( 2^{(j+k)\lambda}x_n \right) E_{j+k}'.
    \end{equation}
    Applying \eqref{Eqn::Trace::InverseProof::XjRelatedToVj::XjInTermsOfVj} to \(X_l\),
    \eqref{Eqn::Trace::InverseProof::ElemAppliedToOpQ::Tmp3} is a sum of terms of two types:
    \begin{equation}\label{Eqn::Trace::InverseProof::ElemAppliedToOpQ::Tmp4}
        E_{j,1}2^{-j\Xdv_1}a(x) x_n^l V_s \gamma(x_n) \phi_{j+k}\left( 2^{(j+k)\lambda}x_n \right) E_{j+k}', \quad \Vdv_s\leq \Xdv_1+l\lambda,\: a\in \CinftySpace[\Bngeq{1}],
    \end{equation}
    \begin{equation}\label{Eqn::Trace::InverseProof::ElemAppliedToOpQ::Tmp5}
        E_{j,1}2^{-j\Xdv_1} b(x) x_n \partial_{x_n} \gamma(x_n) \phi_{j+k}\left( 2^{(j+k)\lambda}x_n \right) E_{j+k}', \quad b\in \CinftySpace[\Bngeq{1}].
    \end{equation}

    Since \(V_s\) does not have a \(\partial_{x_n}\) component, \eqref{Eqn::Trace::InverseProof::ElemAppliedToOpQ::Tmp4} can be rewritten as
    \begin{equation}\label{Eqn::Trace::InverseProof::ElemAppliedToOpQ::Tmp928}
        2^{k\Xdv_1} 2^{-(j+k)\Xdv_1-(j+k)l\lambda+(j+k)\Vdv_s} E_{j,1} a(x) \gamma(x_n) \left( x_n^l\phi_{j+k} \right)\left( 2^{(j+k)\lambda}x_n \right) \left( 2^{-(j+k)\Vdv_s}V_s E_{j+k}' \right)
    \end{equation}
    Since \(\Vdv_s\leq \Xdv_1+l\lambda\) and \(j+k\geq 0\), we have \(2^{-(j+k)\Xdv_1-(j+k)l\lambda+(j+k)\Vdv_s}\leq 1\),
    and we have \(2^{k\Xdv_1}\leq 2^{-|k|}\) since \(k<0\).
    Proposition \ref{Prop::Spaces::Elem::Elem::MainProps} \ref{Item::Spaces::Elem::Elem::MultBySmooth} shows
    \(\left\{ \left( E_{j,1}a,2^{-j} \right) : (E_j, 2^{-j})\in \sE \right\}\in \ElemzXXdv{\Bngeq{\etaFive}}\),
    and Proposition \ref{Prop::Spaces::Elem::Elem::MainProps} \ref{Item::Spaces::Elem::Elem::DerivOp} shows
    \(\left\{ \left( 2^{-k\Vdv_s} V_s E_{k}', 2^{-k} \right) : k\in \Zgeq \right\}\in \ElemzVVdv{\Bnmo{\etaThree}}\).
    It follows that \eqref{Eqn::Trace::InverseProof::ElemAppliedToOpQ::Tmp928} is of the desired form
    and therefore so is \eqref{Eqn::Trace::InverseProof::ElemAppliedToOpQ::Tmp4}.

    Since the expressions in this proof are acting on functions of only \(x'\), we have \(\partial_{x_n}E_{j+k}'=0\), and therefore
    \eqref{Eqn::Trace::InverseProof::ElemAppliedToOpQ::Tmp5} can be written as
    \begin{equation}\label{Eqn::Trace::InverseProof::ElemAppliedToOpQ::Tmp929}
        \begin{split}
            &2^{k\Xdv_1} 2^{-(j+k)\Xdv_1} E_{j,1} b(x) \left( x_n \partial_{x_n}\gamma\right)(x_n) \phi_{j+k}\left( 2^{(j+k)\lambda}x_n \right)E_{j+k}'
            \\&+2^{k\Xdv_1} 2^{-(j+k)\Xdv_1} E_{j,1} b(x) \gamma(x_n) \left( x_n \partial_{x_n}\phi \right)\left( 2^{(j+k)\lambda}x_n \right) E_{j+k}'.
        \end{split}
    \end{equation}
    Since \(k<0\), \(2^{k\Xdv_1}\leq 2^{-|k|}\) and since \(j+k\geq 0\), \(2^{-(j+k)\Xdv_1}\leq 1\).
    Proposition \ref{Prop::Spaces::Elem::Elem::MainProps} \ref{Item::Spaces::Elem::Elem::MultBySmooth}
    shows \(\left\{ \left( E_{j,1}b, 2^{-j} \right) : \left( E_j, 2^{-j} \right)\in \sE \right\}\in \ElemzXXdv{\Bngeq{\etaFive}}\).
    From here, it follows that \eqref{Eqn::Trace::InverseProof::ElemAppliedToOpQ::Tmp929} is of the desired form
    and therefore so is \eqref{Eqn::Trace::InverseProof::ElemAppliedToOpQ::Tmp5}.
    We conclude \eqref{Eqn::Trace::InverseProof::ElemAppliedToOpQ::Tmp3} is of the desired form.

    We turn to the second case in \eqref{Eqn::Trace::InverseProof::ElemAppliedToOpQ::Tmp2}: when \(\left( 2^{-j\Xdv}X \right)^{\alpha}=2^{-j\lambda}\partial_{x_n}\).
    In this case, we are considering the expression
    \begin{equation}\label{Eqn::Trace::InverseProof::ElemAppliedToOpQ::Tmp930}
        E_{j,q+1} 2^{-j\lambda} \partial_{x_n} \gamma(x_n) \phi_{j+k}\left( 2^{(j+k)\lambda}x_n \right) E_{j+k}'.
    \end{equation}
    As before, \(\partial_{x_n}E_{j+k}'=0\) and therefore \eqref{Eqn::Trace::InverseProof::ElemAppliedToOpQ::Tmp930}
    equals (using \(k<0\)),
    \begin{equation*}
    \begin{split}
         &2^{-|k|} E_{j,q+1}' \left( \partial_{x_n} \gamma \right)(x_n) \left( 2^{-(j-|k|)-(j-1)\lambda}\phi_{j+k} \right)\left( 2^{(j+k)\lambda}x_n \right)E_{j+k'}
         \\&+ 2^{-|k|} E_{j,q+1}' \gamma(x_n) \left( 2^{-|k|(\lambda-1)} \partial_{x_n} \phi_{j+k} \right)\left( 2^{(j+k)\lambda}x_n \right)E_{j+k}',
    \end{split}
    \end{equation*}
    which is of the desired form since \(\lambda\geq 1\) and \(|k|\leq j\).
    This completes the proof when \(k<0\).

    We turn to the case \(k\geq 0\). By  Proposition \ref{Prop::Spaces::Elem::Elem::MainProps} \ref{Item::Spaces::Elem::Elem::PullOutNDerivs}
    (with \(N=1\)), we may write
    \begin{equation}\label{Eqn::Trace::InverseProof::ElemAppliedToOpQ::Tmp931}
        E_{j+k}'=\sum_{|\alpha|\leq 1} 2^{(|\alpha|-1)(j+k)} \left( 2^{-(j+k)\Vdv}V \right)^{\alpha} E_{j+k,\alpha}',
    \end{equation}
    where \(\left\{ \left( E_{k,\alpha}',2^{-k}  \right) : k\in \Zgeq,|\alpha|\leq 1 \right\}\in \ElemzVVdv{\Bnmo{\etaThree}}\).
    We plug \eqref{Eqn::Trace::InverseProof::ElemAppliedToOpQ::Tmp931} into the left-hand side of \eqref{Eqn::Trace::InverseProof::ElemAppliedToOpQ::ToShow}
    and show that each term is of the desired form.
    When \(|\alpha|=0\), using that \(2^{-(j+k)}\leq 2^{-|k|}\), the resulting summand is clearly of thd desired form,
    so we treat only \(|\alpha|=1\).
    Without loss of generality, we assume \(\left( 2^{-(j+k)\Vdv}V \right)^{\alpha}=2^{-(j+k)\Vdv_1}V_1\).
    Thus, we are considering the term
    \begin{equation}\label{Eqn::Trace::InverseProof::ElemAppliedToOpQ::Tmp932}
        E_j \gamma(x_n) \phi_{j+k}\left( 2^{(j+k)\lambda}x_n \right) 2^{-(j+k)\Vdv_1}V_1 E_{j+k,1}'.
    \end{equation}
    Using that \(V_1\) has no \(\partial_{x_n}\) component, \eqref{Eqn::Trace::InverseProof::ElemAppliedToOpQ::Tmp932}
    equals
    \begin{equation}\label{Eqn::Trace::InverseProof::ElemAppliedToOpQ::Tmp933}
        E_j  2^{-(j+k)\Vdv_1}V_1 \gamma(x_n) \phi_{j+k}\left( 2^{(j+k)\lambda}x_n \right) E_{j+k,1}'.
    \end{equation}
    Applying \eqref{Eqn::Trace::InverseProof::XjRelatedToVj::VjInTermsOfXj}, we see
    \eqref{Eqn::Trace::InverseProof::ElemAppliedToOpQ::Tmp933} can be written as a sum of terms of two types:
    \begin{equation}\label{Eqn::Trace::InverseProof::ElemAppliedToOpQ::Tmp934}
        E_j 2^{-(j+k)\Vdv_1} X_s a(x) x_n^l \gamma(x_n) \phi_{j+k}\left( 2^{(j+k)\lambda}x_n \right) E_{j+k,1}', \quad \Xdv_s\leq \Vdv_1+l\lambda, a\in \CinftySpace[\Bngeq{1}],
    \end{equation}
    \begin{equation}\label{Eqn::Trace::InverseProof::ElemAppliedToOpQ::Tmp935}
        E_j 2^{-(j+k)\Vdv_1} g(x) a(x) \gamma(x_n) \phi_{j+k}\left( 2^{(j+k)\lambda}x_n \right) E_{j+k,1}', \quad g\in \CinftySpace[\Bngeq{1}].
    \end{equation}

    We have that \(2^{-(j+k)\Vdv_1}\leq 2^{-|k|}\). By 
     Proposition \ref{Prop::Spaces::Elem::Elem::MainProps} \ref{Item::Spaces::Elem::Elem::MultBySmooth}
     \(\left\{ \left( E_j g, 2^{-j} \right) : \left( E_j, 2^{-j} \right)\in \sE \right\}\in \ElemzXXdv{\Bngeq{\etaFive}}\).
     It follows that \eqref{Eqn::Trace::InverseProof::ElemAppliedToOpQ::Tmp935} is of the desired form.

     Finally, we show \eqref{Eqn::Trace::InverseProof::ElemAppliedToOpQ::Tmp934} is of the desired form.
     We rewrite \eqref{Eqn::Trace::InverseProof::ElemAppliedToOpQ::Tmp934} as
     \begin{equation}\label{Eqn::Trace::InverseProof::ElemAppliedToOpQ::Tmp936}
        2^{(j+k)(-\Vdv_1-l\lambda+\Xdv_s)} 2^{-k\Xdv_s} E_j 2^{-j\Xdv_s} X_s a(x) \gamma(x_n) \left(x_n^l \phi_{j+k} \right)\left( 2^{(j+k)\lambda}x_n \right) E_{j+k,1}', \quad \Xdv_s\leq \Vdv_1+l\lambda, a\in \CinftySpace[\Bngeq{1}].
     \end{equation}
     Since \(\Xdv_s\leq \Vdv_1+l\lambda\), we have \(2^{(j+k)(-\Xdv_1-l\lambda+\Xdv_s)}\leq 1\).
     Since \(k\geq\), we have \(2^{-k\Xdv_s}\leq 2^{-|k|}\).
     Proposition \ref{Prop::Spaces::Elem::Elem::MainProps} \ref{Item::Spaces::Elem::Elem::MultBySmooth} and \ref{Item::Spaces::Elem::Elem::DerivOp}
     show
     \(\left\{ \left( E_j 2^{-j\Xdv_s}X_s \Mult{a}, 2^{-j} \right) : (E_j, 2^{-j})\in \sE\right\}\in \ElemzXXdv{\Bngeq{\etaFive}}\).
     It follows that \eqref{Eqn::Trace::InverseProof::ElemAppliedToOpQ::Tmp936} is of the desired form, completing the proof.
\end{proof}

\begin{lemma}\label{Lemma::Trace::InverseProof::VpqNormInequality}
    Let \(\sE'\in \ElemzVVdv{\Bnmo{\etaThree}}\), \(\gamma\in \CinftyCptSpace[{[0,(\etaFour-\etaThree)/2)}]\),
    \(\Phi\subset \SchwartzSpaceR\) a bounded set, and
    \begin{equation*}
        \VSpace{p}{q}
        \in
        \left\{ \lqLpSpace{p}{q}[\Rngeq] : 1\leq p,q\leq \infty \right\}
        \bigcup 
        \left\{ \LplqSpace{p}{q}[\Rngeq] : 1<p<\infty, 1<q\leq \infty \right\}.
    \end{equation*}
    For \(s\in \R\), set
    \begin{equation*}
        \NewXSpace{s}:=
        \begin{cases}
            \BesovSpace{s}{p}{q}, & \VSpace{p}{q}=\lqLpSpace{p}{q}[\Rngeq],\\
            \BesovSpace{s}{p}{p}, & \VSpace{p}{q}=\LplqSpace{p}{q}[\Rngeq].
        \end{cases}
    \end{equation*}
    Then, \(\forall s\in \R\), \(\exists C\geq 0\), \(\forall f\in \NewXSpace{s}[\BnmoClosure{\etaOne}][\FilteredSheafGenByVVdv]\),
    \begin{equation*}
        \sup_{\left\{ \phi_j : j\in \Zgeq \right\}\subseteq \Phi}
        \sup_{\left\{ \left( E_j', 2^{-j} \right) : j\in \Zgeq \right\}\subseteq \sE'} 
        \BVNorm{ \left\{ \gamma(x_n) 2^{j\lambda/p} \phi_j(2^{j\lambda}x_n) 2^{js}E_j' f(x') \right\}_{j\in \Zgeq} }{p}{q} 
        \leq C \NewXNorm{f}{s}[\FilteredSheafGenByVVdv].
    \end{equation*}
\end{lemma}
\begin{proof}
    Let \(\left\{ \left( E_j',2^{-j} \right) :j\in \Zgeq\right\}\subseteq \sE'\)
    and \(\left\{ \phi_j : j\in \Zgeq \right\}\subseteq \Phi\); the implicit constants
    which follow do not depend on the particular choice of subsets.

    We being with the case \(\VSpace{p}{q}=\lqLpSpace{p}{q}[\Rngeq]\).
    We have, (with the usual modification when \(q=\infty\)),
    \begin{equation*}
    \begin{split}
         &\BlqLpNorm{ \left\{ \gamma(x_n) 2^{j\lambda/p} \phi_j(2^{j\lambda}x_n) 2^{js} E_j'f(x') \right\}_{j\in \Zgeq} }{p}{q}[\Rngeq]
         \\&= \left( \sum_{j\in \Zgeq} \BLpNorm{ \gamma(x_n) 2^{j\lambda/p} \phi_j(2^{j\lambda}x_n) 2^{js}E_j'f(x')}{p}[\Rngeq]^q \right)^{1/q}
         \\&=\left( \sum_{j\in \Zgeq} 
        \BLpNorm{ \gamma(x_n) 2^{j\lambda/p} \phi_j(2^{j\lambda}x_n)}{p}[{\lbrack 0,\infty)}]^q 
         \BLpNorm{2^{js}E_j' f(x')}{p}[\Rnmo]^q  
         \right)^{1/q}
         \\&\lesssim \left( \sum_{j\in \Zgeq} 
         \BLpNorm{2^{js}E_j' f(x')}{p}[\Rnmo]^q  
         \right)^{1/q}
         \leq \VpqsENorm{f}[p][q][s][\sE']
         \lesssim \BesovNorm{f}{s}{p}{q}[\FilteredSheafGenByVVdv],
    \end{split}
    \end{equation*}
    where the final estimate used Corollary \ref{Cor::Spaces::MainEst::VpqsESeminormIsContinuous},
    completing the proof in this case.

    Turning to the case \(\VSpace{p}{q}=\LplqSpace{p}{q}[\Rngeq]\), since \(\NewXSpace{s}=\BesovSpace{s}{p}{p}\)
    does not depend on \(q\), and since \(\LplqNormNoSet{\cdot}{p}{q}\leq \LplqNormNoSet{\cdot}{p}{q'}\),
    \(\forall q'<q\), we may without loss of generality assume \(1<q<p\).
    We have
    \begin{equation*}
    \begin{split}
         &\BLplqNorm{\left\{\gamma(x_n) 2^{j\lambda/p} \phi_j(2^{j\lambda}x_n) 2^{js} E_j'f(x')  \right\}_{j\in \Zgeq}}{p}{q}[\Rngeq]^p
         \\&=\int_0^\infty \int_{\Rnmo} \left( \sum_{j\in \Zgeq} \left| \gamma(x_n) 2^{j\lambda/p} \phi_j(2^{j\lambda}x_n) 2^{js} E_j'f(x')  \right|^q \right)^{p/q} \: dx'\: dx_n
         \\&=\int_0^\infty  \BLpNorm{ \sum_{j\in \Zgeq} \left| \gamma(x_n) 2^{j\lambda/p}\phi_j (2^{j\lambda}x_n) 2^{js}E_j'f(x') \right|^q }{p/q}[\Rnmo]^{p/q} \: dx_n
         \\&\leq \int_0^\infty  \left(  \sum_{j\in \Zgeq} \BLpNorm{\left| \gamma(x_n)2^{j\lambda/p} \phi_j(2^{j\lambda}x_n) 2^{js} E_j' f(x') \right|^q}{p/q}[\Rnmo] \right)^{p/q}\: dx_n
         \\&=\int_0^\infty  \left(  \sum_{j\in \Zgeq} \BLpNorm{ \gamma(x_n) 2^{j\lambda/p} \phi_j(2^{j\lambda}x_n) 2^{js}E_j'f(x')}{p}[\Rnmo]^q \right)^{p/q}\: dx_n
         \\&=\int_0^\infty  \left(  \sum_{j\in \Zgeq} |\gamma(x_n)|^q  2^{jq\lambda/p} |\phi_j(2^{j\lambda}x_n)|^q \BLpNorm{2^{js} E_j' f(x')}{p}[\Rnmo]^q \right)^{p/q}\: dx_n
         \\&=: \int_I \cdot\: dx_n + \sum_{l=0}^\infty \int_{I_l} \cdot\: dx_n,
    \end{split}
    \end{equation*}
    where \(I=(1,\infty)\) and \(I_l=\left( 2^{-(l+1)\lambda}, 2^{-l\lambda} \right)\).
    Due to \(\supp(\gamma)\), we have \(\int_I =0\), so we focus on \(\sum_{l=0}^\infty \int_{I_l}\).

    Consider,
    \begin{equation}\label{Eqn::Trace::InverseProof::VpqNormInequality::CutUpSum}
    \begin{split}
         &\sum_{l=0}^\infty \int_{I_l} \left( \sum_{j\in \Zgeq} \left| \gamma(x_n) \right|^q 2^{jq\lambda/p} \left| \phi_j(2^{j\lambda}x_n) \right|^q \BLpNorm{2^{js}E_j' f(x')}{p}[\Rnmo] \right)^{p/q}\: dx_n
         \\&\lesssim \sum_{l=0}^\infty \int_{I_l}\left( \sum_{j=0}^l \cdot \right)^{p/q}\: dx_n+\sum_{l=0}^\infty \int_{I_l} \left( \sum_{j=l+1}^\infty \cdot \right)^{p/q}\: dx_n.
    \end{split}
    \end{equation}
    We estimate the two terms on the right-hand side of \eqref{Eqn::Trace::InverseProof::VpqNormInequality::CutUpSum}, separately.

    For the first term, let \(\epsilon:=q\lambda/2p>0\). Then, using \(|\gamma|,|\phi_j|\lesssim 1\) and H\"older's Inequality,
    we have with \((p/q)'\) denoting the dual exponent to \(p/q\in (1,\infty)\),
    \begin{equation*}
    \begin{split}
         &\sum_{l=0}^\infty \int_{I_l} \left( \sum_{j=0}^l |\gamma(x_n)|^q 2^{j\lambda q/p} |\phi_j(2^{j\lambda}x_n)|^q \BLpNorm{2^{js}E_j' f}{p}[\Rnmo]^q \right)^{p/q}\: dx_n
         \\&\lesssim \sum_{l=0}^\infty  |I_l|\left( \sum_{j=0}^l 2^{j\lambda q/p} \BLpNorm{2^{js}E_j' f}{p}[\Rnmo]^q \right)^{p/q}
         \\&\approx \sum_{l=0}^\infty 2^{-l\lambda} \left(\sum_{j=0}^l 2^{-\epsilon(l-j)}2^{\epsilon(l-j)}2^{j\lambda q/p} \BLpNorm{2^{js}E_j' f}{p}[\Rnmo]^q  \right)^{p/q}
         \\&\leq \sum_{l=0}^\infty 2^{-l\lambda} \left( \sum_{j=0}^l  2^{-\epsilon(l-j)(p/q)'}  \right)^{(p/q)/(p/q)'} \left( \sum_{j=0}^l 2^{\epsilon(l-j)p/q} 2^{j\lambda} \BLpNorm{2^{js}E_j' f}{p}[\Rnmo]^p  \right)
         \\&\lesssim \sum_{l=0}^\infty 2^{-l\lambda} \sum_{j=0}^l 2^{\epsilon(l-j)p/q} 2^{j\lambda} \BLpNorm{2^{js} E_j'f}{p}[\Rnmo]^p
         \\&=\sum_{j=0}^\infty\sum_{l=j}^\infty 2^{-l\lambda}  2^{\epsilon(l-j)p/q} 2^{j\lambda} \BLpNorm{2^{js}E_j'f}{p}[\Rnmo]^p
         \\&\lesssim \sum_{j=0}^\infty \BLpNorm{2^{js}E_j'f}{p}[\Rnmo]^p
         \leq \| f \|^p_{\lqLpSpaceNoSet{p}{p}, s, \sE'}
         \lesssim \BesovNorm{f}{s}{p}{p}[\FilteredSheafGenByVVdv]^p,
    \end{split}
    \end{equation*}
    where the final estimate used Corollary \ref{Cor::Spaces::MainEst::VpqsESeminormIsContinuous}.

    We turn to the second term on the right-hand side of \eqref{Eqn::Trace::InverseProof::VpqNormInequality::CutUpSum}.
    Take \(\rho>0\) such that \((\rho\lambda-1)(p/q)>\lambda\).
    Using that \(\Phi\subset \SchwartzSpaceR\) is a bounded set,
    for \(x_n\in I_l\) and \(j\geq l\), we have \(|\phi_j(2^{j\lambda}x_n)|\lesssim 2^{-(j-l)\lambda\rho}\).
    We have,
    using H\"older's inequality and with \((p/q)'\) the dual exponent of \(p/q\in (1,\infty)\), and using \(|\gamma(x_n|\lesssim 1)\),
    \begin{equation*}
    \begin{split}
         &\sum_{l=0}^\infty \int_{I_l} \left( \sum_{j=l+1}^\infty |\gamma(x_n)|^q 2^{j\lambda q/p} |\phi_j(2^{j\lambda}x_n)|^q \BLpNorm{2^{js}E_j' f}{p}[\Rnmo]^q  \right)^{p/q}\: dx_n
         \\&\lesssim \sum_{l=0}^\infty |I_l| \left( \sum_{j=l+1}^\infty 2^{j\lambda q/p}2^{-\rho(j-l)\lambda} \BLpNorm{2^{js} E_j' f}{p}[\Rnmo]^q \right)^{p/q}
         \\&\lesssim \sum_{l=0}^\infty 2^{-l\lambda} \left( \sum_{j=l+1}^\infty 2^{-(j-l)(p/q)'} \right)^{(p/q)/(p/q)'} \left( \sum_{j=l+1}^\infty 2^{j\lambda} 2^{-(\rho\lambda-1)(j-l)(p/q)} \BLpNorm{2^{js}E_j' f}{p}[\Rnmo]^p \right)
         \\&\lesssim \sum_{j=0}^\infty \sum_{l=0}^{j-1} 2^{(j-l)\lambda} 2^{-(\rho\lambda-1)(j-l)(p/q)} \BLpNorm{2^{js}E_j' f}{p}[\Rnmo]^p
         \\&\lesssim \sum_{j=0}^\infty \BLpNorm{2^{js}E_j' f}{p}[\Rnmo]^p
         \leq  \| f \|^p_{\lqLpSpaceNoSet{p}{p}, s, \sE'}
         \lesssim \BesovNorm{f}{s}{p}{p}[\FilteredSheafGenByVVdv]^p,
    \end{split}
    \end{equation*}
    where the third to last estimate used the choice of \(\rho\) and the  final estimate used Corollary \ref{Cor::Spaces::MainEst::VpqsESeminormIsContinuous}.
    Combining the above estimates completes the proof.
\end{proof}

\begin{proof}[Proof of Propoistion \ref{Prop::Trace::InverseProof::QtIsContinuous}]
    Let \(\ASpace{s}{p}{q}\in \left\{ \BesovSpace{s}{p}{q} : 1\leq p,q\leq \infty \right\}\bigcup \left\{ \TLSpace{s}{p}{q} : 1<p<\infty, 1<q\leq\infty\right\}\).
    Set
    \begin{equation*}
        \VSpace{p}{q}
        :=
        \begin{cases}
            \lqLpSpace{p}{q}[\Rngeq], & \ASpace{s}{p}{q}=\BesovSpace{s}{p}{q},\\
            \LplqSpace{p}{q}[\Rngeq], & \ASpace{s}{p}{q}=\TLSpace{s}{p}{q}.
        \end{cases}
    \end{equation*}
    Let \(\opQtEPhigamma\) be as in \eqref{Eqn::Trace::InverseProof::DefineQt}, with \(\sE'\), \(\Phi\), and \(\gamma\)
    as described there.

    Fix \(\sE\in \ElemzXXdv{\Bngeq{\etaFive}}\).  Fix \(N>|s+\lambda/p|\).
    Let \(\left\{ \left( E_j, 2^{-j} \right) : j\in \Zgeq \right\}\subseteq \sE\).
     For
    \begin{equation*}
        f\in
        \begin{cases}
            \BesovSpace{s+t}{p}{q}[\BnmoClosure{\etaOne}][\FilteredSheafGenByVVdv], & 
            \ASpace{s}{p}{q}=\BesovSpace{s}{p}{q},\\
            \BesovSpace{s+t}{p}{p}[\BnmoClosure{\etaOne}][\FilteredSheafGenByVVdv], & \ASpace{s}{p}{q}=\TLSpace{s}{p}{q},
        \end{cases}
    \end{equation*}
    using Lemma \ref{Lemma::Trace::InverseProof::ElemAppliedToOpQ} (and the notation of that lemma),
    we have (with implicit constants independent of the choice of subset of \(\sE\) and independent of \(f\)),
    \begin{equation}\label{Eqn::Trace::InverseProof::FinalProof::EstVpq1}
    \begin{split}
         & \BVNorm{ \left\{2^{js+j\lambda/p}  E_j \opQtEPhigamma f \right\}_{j\in \Zgeq}}{p}{q}
         \\&\leq \sum_{l=1}^L \sum_{k\in \Z} 2^{-N|k|} \BVNorm{\left\{ 2^{j(s+\lambda/p)+(j+k)t} E_{j,l}\gamma_l(x_n) \phi_{j,k,l}\left( 2^{(j+k)\lambda}x_n \right)E_{j+k}' f  \right\}_{j\in \Zgeq}}{p}{q}
         \\&\lesssim \sum_{l=1}^L \sum_{k\in \Z} 2^{-N|k|} \BVNorm{\left\{ 2^{j(s+\lambda/p)+(j+k)t} \gamma_l(x_n) \phi_{j,k,l}\left( 2^{(j+k)\lambda}x_n \right)E_{j+k}' f  \right\}_{j\in \Zgeq}}{p}{q},
    \end{split}
    \end{equation}
    where the final estimate uses \(\left\{ \left( E_{j,l},2^{-j} \right) : (E_j, 2^{-j})\in \sE, j\in \Zgeq, l=1,\ldots,L \right\}\in \ElemzXXdv{\Bngeq{\etaFive}}\subseteq \PElemXXdv{\BngeqClosure{\etaFive}}\)
    (by Lemma \ref{Lemma::Trace::InverseProof::ElemAppliedToOpQ})
    and Proposition \ref{Prop::Spaces::Elem::PElem::PElemOpsBoundedOnVV}.
    We have, using Lemma \ref{Lemma::Trace::InverseProof::VpqNormInequality} (with \(s\) replaced by \(s+t\)) and Notation \ref{Notation::Trace::InverseProof::ZeroIfNegativeIndex},
    \begin{equation}\label{Eqn::Trace::InverseProof::FinalProof::EstVpq2}
    \begin{split}
         &\sum_{l=1}^L \sum_{k\in \Z} 2^{-N|k|} \BVNorm{\left\{ 2^{j(s+\lambda/p)+(j+k)t}\gamma_l(x_n) \phi_{j,k,l}\left( 2^{(j+k)\lambda}x_n \right)E_{j+k}' f \right\}_{j\in \Zgeq}}{p}{q}
         \\&\leq \sum_{l=1}^L \sum_{k\in \Z} 2^{-N|k|-k(s+\lambda/p)} \BVNorm{ \left\{ 2^{j\lambda/p} \gamma_l(x_n) \phi_{j-k,k,l}\left( 2^{j\lambda}x_n \right) 2^{j(s+t)}E_{j,l}'f   \right\}_{j\in \Zgeq} }{p}{q}
         \\&\lesssim \sum_{l=1}^L \sum_{k\in \Z} 2^{-N|k|-k(s+\lambda/p)} 
         \begin{cases}
            \BesovNorm{f}{s+t}{p}{q}[\FilteredSheafGenByVVdv], &\ASpace{s}{p}{q}=\BesovSpace{s}{p}{q},\\
            \BesovNorm{f}{s+t}{p}{p}[\FilteredSheafGenByVVdv], &\ASpace{s}{p}{q}=\TLSpace{s}{p}{q}.
         \end{cases}
    \end{split}
    \end{equation}
    Combining \eqref{Eqn::Trace::InverseProof::FinalProof::EstVpq1} and \eqref{Eqn::Trace::InverseProof::FinalProof::EstVpq2},
    and taking the supremum over \(\left\{ \left( E_j, 2^{-j} \right) :j\in \Zgeq \right\}\subseteq \sE\), we conclude
    \begin{equation}\label{Eqn::Trace::InverseProof::FinalProof::EstVpq3}
        \VpqsENorm{\opQtEPhigamma f}[p][q][s+\lambda/p][\sE]
        \lesssim
         \begin{cases}
            \BesovNorm{f}{s+t}{p}{q}[\FilteredSheafGenByVVdv], &\ASpace{s}{p}{q}=\BesovSpace{s}{p}{q},\\
            \BesovNorm{f}{s+t}{p}{p}[\FilteredSheafGenByVVdv], &\ASpace{s}{p}{q}=\TLSpace{s}{p}{q}.
         \end{cases}
    \end{equation}
    Since \(\supp(\opQtEPhigamma f)\subseteq \BngeqClosure{\etaFour}\) (see Proposition \ref{Prop::Trace::InverseProof::QtConverges}),
    it follows that \(\opQtEPhigamma f\in \ASpace{s+\lambda/p}{p}{q}[\FilteredSheafGenByVVdv]\)
    (see Definition \ref{Defn::Spaces::Defns::ASpace}).
    By choosing \(\sE=\sD_0\in \ElemzXXdv{\Bnmo{\etaFive}}\) as in Notation \ref{Notation::Spaces::Defns::Norm}, 
    we see
    \begin{equation*}
        \ANorm{\opQtEPhigamma f}{s+\lambda/p}{p}{q}[\FilteredSheafGenByXXdv]
        \lesssim 
        \begin{cases}
            \BesovNorm{f}{s+t}{p}{q}[\FilteredSheafGenByVVdv], &\ASpace{s}{p}{q}=\BesovSpace{s}{p}{q},\\
            \BesovNorm{f}{s+t}{p}{p}[\FilteredSheafGenByVVdv], &\ASpace{s}{p}{q}=\TLSpace{s}{p}{q},
         \end{cases}
    \end{equation*}
    completing the proof.
\end{proof}

    \subsection{Proof of Proposition \texorpdfstring{\ref{Prop::Traces::CoordReduction::VanishingChar}}{\ref*{Prop::Traces::CoordReduction::VanishingChar}}}
    \label{Section::Trace::CharacterizeVanish}
    In this section, we prove Proposition \ref{Prop::Traces::CoordReduction::VanishingChar}.
Fix \(p,q\in (1,\infty)\) and \(0<\sigma_1<\sigma_2<\sigma_3<1\); we will later
specialize to \(\sigma_2=3/4\) and \(\sigma_1=1/2\).
Let \(\Psi\in \CinftyCptSpace[\Bngeq{\sigma_3}]\), with \(\Psi=1\) on a neighborhood
of \(\BngeqClosure{\sigma_2}\). Using Proposition \ref{Prop::Spaces::LP::DjExist}, we may write
\(\Mult{\Psi}=\sum_{j\in \Zgeq} D_j\), where \(\left\{ \left( D_j, 2^{-j} \right) : j\in \Zgeq \right\}\in \ElemzXXdv{\Bngeq{\sigma_3}}\)
and the convergence is taken in the sense of Proposition \ref{Prop::Spaces::Elem::Elem::ConvergenceOfElemOps}.
Using Notation \ref{Notation::Spaces::Defns::Norm} (see, also, Proposition \ref{Prop::Spaces::Defns::EquivNorms}),
we have
\begin{equation*}
    \ANorm{f}{s}{p}{q}[\FilteredSheafGenByXXdv] 
    = \BVNorm{\left\{ 2^{js} D_j f\right\}_{j\in \Zgeq}}{p}{q}, \quad f\in \ASpace{s}{p}{q}[\BngeqClosure{\sigma_2}][\FilteredSheafGenByXXdv].
\end{equation*}
The key estimate we need is the next proposition.

\begin{proposition}\label{Prop::Trace::CharVanish::MainEstimate}
    Let \(\phi_0(x_n)\in \CinftySpace[\lbrack 0,\infty)]\) be equal to \(0\) on \([0,1/2^\lambda]\)
    and be equal to \(1\) on \([1,\infty)\).
    Fix \(\kappa\in \Zgeq\), \(g\in \CinftyCptSpace[\Bngeq{1}]\), and \(s<\lambda\kappa+\lambda/p\).
    We have,
    \begin{equation*}
        \sup_{\substack{K\geq 4 \\ K\in \Z}} \sum_{j\in \Zgeq} 2^{js} \BLpNorm{ D_j \phi_0\left( 2^{\lambda K}x_n \right) x_n^{\kappa}g(x)}{p}[\Rngeq]<\infty.
    \end{equation*}
\end{proposition}

To prove Proposition \ref{Prop::Trace::CharVanish::MainEstimate}, we use the next lemma.

\begin{lemma}\label{Lemma::Trace::CharVanish::MainEstimatelemma}
    Fix \(0<a_1<a_2\) and let \(\sB\subset \CinftyCptSpace[a_1,a_2]\) be a bounded set.
    Fix \(\epsilon>0\), \(R\geq 0\), and \(g\in \CinftyCptSpace[\Bngeq{1}]\).
    There exists \(C=C(\sB, \epsilon, R, g,p)\geq 0\) such that
    for all \(\left\{ \psi_k : k\in \Zgeq \right\}\subset \sB\), we have
    \begin{equation}\label{Eqn::Trace::CharVanish::MainEstimatelemma}
        \sum_{j,k\in \Zgeq} 2^{j(\lambda/p-\epsilon)+R(j-k)} \BLpNorm{D_j \psi_k\left( 2^{\lambda k}x_n \right) g(x)}{p}[\Rngeq]\leq 
        C.
    \end{equation}
\end{lemma}

Before we prove Lemma \ref{Lemma::Trace::CharVanish::MainEstimatelemma}, we see how it implies
Proposition \ref{Prop::Trace::CharVanish::MainEstimate}.

\begin{proof}[Proof of Proposition \ref{Prop::Trace::CharVanish::MainEstimate}]
    Take \(\phi(x_n)\in \CinftyCptSpace[1/4^{\lambda}, 2^\lambda]\) such that \(\phi=1\) on  a neighborhood of
    \([1/2^{\lambda}, 1]\). Set \(\psi(x_n)=\phi(x_n)-\phi(2^{-\lambda}x_n)\), so that 
    \(\psi\in \CinftyCptSpace[1/4^{\lambda}, 2^{\lambda+1}]\).
    By the choice of \(g\), we have \(g(x,x_n)=0\) for \(x_n\geq 1\).  We have, for \(K\geq 4\),
    \begin{equation*}
    \begin{split}
         &\phi_0\left( 2^{\lambda K}x_n \right)x_n^{\kappa} g(x) 
         =\phi_0\left( 2^{\lambda K }x_n \right) \phi\left( 2^{\lambda K}x_n \right) x_n^{\kappa} g(x)
         \\&=\left( \phi(x_n) + \sum_{k=1}^K \psi\left( 2^{k\lambda}x_n \right) \right) \phi_0\left( 2^{\lambda K}x_n \right)x_n^{\kappa} g(x)
         =\sum_{k=0}^K \psit_{K,k}\left( 2^{k\lambda}x_n \right)x_n^{\kappa} g(x),
    \end{split}
    \end{equation*}
    where
    \begin{equation*}
        \psit_{K,k}(x_n)
        =
        \begin{cases}
            \phi(x_n) \phi_0\left( 2^{K\lambda}x_n \right), &k=0,\\
            \psi(x_n) \phi_0\left( 2^{(K-k)\lambda}x_n \right), &1\leq k\leq K
        \end{cases}
        =
        \begin{cases}
            \phi(x_n), &k=0,\\
            \psi(x_n), &1\leq k\leq K-2,\\
            \psi(x_n) \phi\left( 2^{(K-k)\lambda }x_n \right), & K-2\leq k\leq K,
        \end{cases}
    \end{equation*}
    and we have used \(\phi_0(x_n)=1\) for \(x_n\geq 1\).
    In particular, \(\sB:=\left\{ \psit_{K, k} : k\in \Zgeq, K\geq 4 \right\}\) is finite  subset of \(\CinftyCptSpace[1/4^{\lambda}, 2^{\lambda+1}]\).

    Take \(\epsilon:=\lambda/p+\lambda \kappa-s>0\). We have, for every \(K\geq 4\),
    \begin{equation*}
    \begin{split}
         &\sum_{j\in \Zgeq} 2^{js}\BLpNorm{ D_j \phi_0\left( 2^{\lambda K}x_n \right)x_n^{\kappa}g }{p}    
         \leq \sum_{j\in \Zgeq} \sum_{k=0}^K 2^{j(\lambda/p-\epsilon)+\lambda\kappa j} \BLpNorm{D_j \psit_{K,k}\left( 2^{\lambda k}x_n \right)x_n^{\kappa}g}{p}
        \\&=\sum_{j\in \Zgeq} \sum_{k=0}^K 2^{j(\lambda/p-\epsilon)+\lambda \kappa(j-k)} \BLpNorm{D_j \left( x_n^\kappa \psit_{K,k} \right)\left( 2^{\lambda k}x_n \right)g}{p}
        \leq C(\sB, \epsilon, \lambda\kappa, g),
    \end{split}
    \end{equation*}
    where the final inequality uses Lemma \ref{Lemma::Trace::CharVanish::MainEstimatelemma}. Taking the supremum over
    \(K\geq 4\) completes the proof.
\end{proof}

We turn to the proof of Lemma \ref{Lemma::Trace::CharVanish::MainEstimatelemma}. For it, we need two more lemmas.

\begin{lemma}\label{Lemma::Trace::CharVanish::DerivOfBpsi}
    For all \(\alpha\), an ordered multi-index, \(\exists L\in \Zgeq\), \(\forall 0<a_1<a_2<\infty\),
    for all 
    \(\sB_1\subset \CinftyCptSpace[\Bngeq{1}]\) and \(\sB_2\subset \CinftyCptSpace[a_1,a_2]\) bounded sets,
    there exist 
    \(\sB_1'\subset \CinftyCptSpace[\Bngeq{1}]\) and \(\sB_2'\subset \CinftyCptSpace[a_1,a_2]\) bounded sets,
    \(\forall b(x)\in \sB_1\), \(\forall \psi(x_n)\in \sB_2\), \(\forall j,k\geq 0\),
    \begin{equation*}
        \left( 2^{-j\Xdv}X \right)^{\alpha} \left[ b(x) \psi\left( 2^{k\lambda}x_n \right) \right]
        =2^{(k-j)\DegXdv{\alpha}} \sum_{l=1}^L b_l(x) \psi_l\left( 2^{k\lambda}x_n \right),
    \end{equation*}
    where \(b_l\in \sB_1'\) and \(\psi_l\in \sB_2'\).
\end{lemma}
\begin{proof}
    The result for \(|\alpha|=0\) is trivial. We prove the result for \(|\alpha|=1\) and the full result follow by iterating
    the case \(|\alpha|=1\).
    Thus, we consider the case when \(\left( 2^{-j\Xdv}X \right)^{\alpha}=2^{-j\Xdv_m}X_m\) for some \(1\leq m\leq q+1\).

    We have,
    \begin{equation}\label{Eqn::Trace::CharVanish::DerivOfBpsi::Tmp102}
        \begin{split}
            &\left( 2^{-j\Xdv}X \right)^{\alpha} \left[ b(x) \psi\left( 2^{k\lambda}x_n \right) \right]=
            \left( 2^{-j\Xdv_m}X_m \right)\left[ b(x) \psi\left( 2^{k\lambda}x_n \right) \right]
            \\&=
            2^{-j\Xdv_m} \left( X_m b \right) (x) \psi\left( 2^{k\lambda}x_n \right)
            +2^{-j\Xdv_m} b(x)\left( X_m\left[ \psi\left( 2^{k\lambda} x_n \right) \right] \right).
        \end{split}
    \end{equation}
    For the first term on the right-hand side of \eqref{Eqn::Trace::CharVanish::DerivOfBpsi::Tmp102}, we have
    \begin{equation*}
        2^{-j\Xdv_m} \left( X_m b \right)(x)\psi\left( 2^{k\lambda}x_n \right)
        =2^{(k-j)\Xdv_m} \left( 2^{-k\Xdv_m}X_m b \right)(x) \psi\left( 2^{k\lambda}x_n \right).
    \end{equation*}
    Since \(\left\{ 2^{-k\Xdv_m}X_m b :k\geq 0, b\in \sB_1 \right\}\subset \CinftyCptSpace[\Bngeq{1}]\) is a bounded set,
    we see that this term is of the desired form.

    We turn to the second term on the right-hand side of \eqref{Eqn::Trace::CharVanish::DerivOfBpsi::Tmp102},
    which we separate into two cases: \(1\leq m\leq q\) and \(m=q+1\).
    When \(m=q+1\), we have \(\Xdv_m=\lambda\), \(X_m=X_{q+1}=\partial_{x_n}\), and so
    \begin{equation*}
        2^{-j\Xdv_m}b(x) \left( X_m\left[ \psi\left( 2^{k\lambda} x_n \right) \right] \right)
        =2^{(k-j)\lambda} b(x) \left( \partial_{x_n} \psi \right)\left( 2^{k\lambda}x_n \right)
        =2^{(k-j)\Xdv_m} b(x) \left( \partial_{x_n}\psi \right)\left( 2^{k\lambda}x_n \right),
    \end{equation*}
    which is of the desired form.

    Finally, we turn to the case when \(1\leq m\leq q\). Without loss of generality, we make take \(m=1\),
    so that we are considering
    \begin{equation*}
        2^{-j\Xdv_1} b(x)\left( X_1\left[ \psi\left( 2^{k\lambda} x_n \right) \right] \right).
    \end{equation*}
    We use \eqref{Eqn::Trace::InverseProof::XjRelatedToVj::XjInTermsOfVj} to see
    \begin{equation*}
        X_1(x)\equiv b_1(x) x_n \partial_{x_n}\mod \Span\left\{ V_1(x),\ldots, V_q(x) \right\},
    \end{equation*}
    where \(b_1\in \CinftyCptSpace[\Bngeq{1}]\). Since \(V_k\) has no \(\partial_{x_n}\) component,
    we have \(V_k\left[ \psi\left( 2^{k\lambda} x_n \right)\right]=0\), and therefore,
    \begin{equation*}
    \begin{split}
            2^{-j\Xdv_1} b(x)\left( X_1\left[ \psi\left( 2^{k\lambda} x_n \right) \right] \right)
            &=2^{-j\Xdv_1} b_1(x) b(x)  \left( x_n \partial_{x_n} \left[ \psi\left( 2^{k\lambda}x_n \right) \right] \right)
            \\&=2^{(k-j)\Xdv_1}  \left( 2^{-k\Xdv_1} b_1(x) b(x) \right) \left( x_n \partial_{x_n} \psi \right)\left( 2^{k\lambda}x_n \right).
    \end{split}
    \end{equation*}
    Since \(\left\{ 2^{-k\Xdv_1} b_1 b : k\geq 0, b\in \sB_1 \right\}\subset \CinftyCptSpace[\Bngeq{1}]\) and
    \(\left\{ x_n \partial_{x_n}\psi :\psi\in \sB_2 \right\}\subset \CinftyCptSpace[a_1,a_2]\) are bounded sets,
    this is of the desired form, completing the proof.
\end{proof}

\begin{lemma}\label{Lemma::Trace::CharVanish::LpNormDerivOfbPsi}
    \(\forall 0<a_1<a_2<\infty\), for all \(\sB_1\subset \CinftyCptSpace[\Bngeq{1}]\) and \(\sB_2\subset \CinftyCptSpace[a_1,a_2]\)
    bounded sets, \(\forall \alpha\), 
    \(\exists C=C_{a_1,a_2,\sB_1,\sB_2,\alpha,p}\geq 0\),
    \(\forall j,k\geq 0\),
    \begin{equation*}
        \BLpNorm{\left( 2^{-j\Xdv}X \right)^{\alpha} \left[ b(x) \psi\left( 2^{k\lambda}x_n \right) \right]}{p}[\Rngeq]
        \leq C 2^{(k-j)\DegXdv{\alpha}}2^{-k\lambda/p}.
    \end{equation*}
\end{lemma}
\begin{proof}
    In light of Lemma \ref{Lemma::Trace::CharVanish::DerivOfBpsi}, it suffices to prove the result
    in the case \(|\alpha|=0\). In that case, we have
    \begin{equation*}
        \BLpNorm{b(x) \psi\left( 2^{k\lambda}x_n \right)}{p}[\Rngeq]^p 
        \lesssim \int_{2^{-k\lambda}a_1}^{2^{-k\lambda}a_2} \int_{\Bnmo{1}} \: dx' \: dx_n \approx 2^{-k\lambda},
    \end{equation*}
    completing the proof.
\end{proof}

\begin{proof}[Proof of Lemma \ref{Lemma::Trace::CharVanish::MainEstimatelemma}]
    We separate \eqref{Eqn::Trace::CharVanish::MainEstimatelemma} into \(\sum_{j\geq k\geq 0}\) and \(\sum_{k>j\geq 0}\).
    For \(j\geq k\), fix \(N>R+\lambda/p\). Using Proposition \ref{Prop::Spaces::Elem::Elem::MainProps} \ref{Item::Spaces::Elem::Elem::PullOutNDerivs},
    we may write
    \begin{equation*}
        D_j=\sum_{|\alpha|\leq N} 2^{-j(N-|\alpha|)} E_{j,\alpha} \left( 2^{-j\Xdv} X \right)^{\alpha},
    \end{equation*}
    where \(\left\{ \left( E_{j,\alpha},2^{-j}  \right) : j\in \Zgeq, |\alpha|\leq N \right\}\in \ElemzXXdv{\Bngeq{\sigma_3}}\).
    By Lemma \ref{Lemma::Spaces::Elem::PElem::PElemOpsBoundedOnLp}, \(\LpOpNorm{E_{j,\alpha}}{p}\lesssim 1\)
    and therefore using Lemma \ref{Lemma::Trace::CharVanish::LpNormDerivOfbPsi}, we have
    \begin{equation*}
    \begin{split}
         &\sum_{j\geq k\geq 0} 2^{j(\lambda/p-\epsilon)+R(j-k)} \BLpNorm{D_j \psi_k\left( 2^{\lambda k}x_n \right)g}{p}
         \\&\leq \sum_{|\alpha|\leq N} \sum_{j\geq k\geq 0} 2^{j(\lambda/p-\epsilon)+R(j-k)-j(N-|\alpha|)} \BLpNorm{ E_{j,\alpha} \left( 2^{-j\Xdv}X \right)^{\alpha} \psi_k\left( 2^{\lambda k}x_n \right) g}{p}
         \\&\lesssim \sum_{|\alpha|\leq N} \sum_{j\geq k\geq 0} 2^{j(\lambda/p-\epsilon)+R(j-k)-j(N-|\alpha|)} \BLpNorm{\left( 2^{-j\Xdv}X \right)^{\alpha} \psi_k\left( 2^{\lambda k}x_n \right) g}{p}
         \\&\lesssim \sum_{|\alpha|\leq N} \sum_{j\geq k\geq 0} 2^{j(\lambda/p-\epsilon) +R(j-k) -j(N-|\alpha|) -(j-k)\DegXdv{\alpha}-k\lambda/p}
         \\&\lesssim \sum_{j\geq k\geq 0} 2^{-(j-k)(N-R-\lambda/p+\epsilon)}2^{-k\epsilon}
         \lesssim 1.
    \end{split}
    \end{equation*}

    For \(0\leq j<k\), we use  Lemma \ref{Lemma::Spaces::Elem::PElem::PElemOpsBoundedOnLp} to see
    \(\LpOpNorm{D_j}{p}\lesssim 1\), and therefore using Lemma \ref{Lemma::Trace::CharVanish::LpNormDerivOfbPsi}
    with \(|\alpha|=0\), we have
    \begin{equation*}
        \BLpNorm{ D_j \psi_k\left( 2^{\lambda k}x_n \right)g}{p}
        \lesssim \BLpNorm{ \psi_k\left( 2^{\lambda k}x_n \right)g}{p}
        \lesssim 2^{-k\lambda/p},
    \end{equation*}
    and therefore,
    \begin{equation*}
    \begin{split}
         &\sum_{0\leq j<k} 2^{j(\lambda/p-\epsilon)+R(j-k)} 
            \BLpNorm{D_j \psi_k\left( 2^{\lambda k}x_n \right)g}{p}
        \lesssim \sum_{0\leq j<k} 2^{j(\lambda/p-\epsilon) +R (j-k) -k\lambda/p}
        \leq \sum_{0\leq j<k} 2^{(j-k)\lambda/p}2^{-j\epsilon}\lesssim 1.
    \end{split}
    \end{equation*}
\end{proof}

We have completed the proof of Lemma \ref{Lemma::Trace::CharVanish::MainEstimatelemma}
and therefore completed the proof of Proposition \ref{Prop::Trace::CharVanish::MainEstimate}.
With Proposition \ref{Prop::Trace::CharVanish::MainEstimate} in hand, we turn to the proof of
Proposition \ref{Prop::Traces::CoordReduction::VanishingChar}, which we separate into the next two lemmas.

\begin{lemma}\label{Lemma::Trace::CharVanish::ApproxBySmoothWithVanish}
    Let \(\ASpace{s}{p}{q}\) be either \(\BesovSpace{s}{p}{q}\) or \(\TLSpace{s}{p}{q}\).
    \begin{enumerate}[(i)]
        \item\label{Item::Trace::CharVanish::ApproxBySmoothWithVanish::smalls} If \(\ASpace{s}{p}{q}=\BesovSpace{s}{p}{q}\), we assume \(1\leq p\leq \infty\), \(1\leq q<\infty\); if 
            \(\ASpace{s}{p}{q}=\TLSpace{s}{p}{q}\) we assume \(1<p,q<\infty\).
            \(\forall u\in \ASpace{s}{p}{q}[\BngeqClosure{\sigma_1}][\FilteredSheafGenByXXdv]\),
            \(\exists \{f_j\}_{j\in \Zgeq}\subset \CinftyCptSpace[\Bngeq{\sigma_2}]\), with \(f_j\rightarrow u\)
            in \(\ASpace{s}{p}{q}[\BngeqClosure{\sigma_2}][\FilteredSheafGenByXXdv]\).
        \item\label{Item::Trace::CharVanish::ApproxBySmoothWithVanish::larges} If \(\ASpace{s}{p}{q}=\BesovSpace{s}{p}{q}\), we assume \(1< p\leq \infty\), \(1\leq q<\infty\); if 
            \(\ASpace{s}{p}{q}=\TLSpace{s}{p}{q}\) we assume \(1<p,q<\infty\).
            If \(s>L\lambda+\lambda/p\), where \(L\in \Zgeq\), if \(u\in \ASpace{s}{p}{q}[\BngeqClosure{\sigma_1}][\FilteredSheafGenByXXdv]\)
            with \(\TraceMap[\sN_L]u=0\), \(\exists \{f_j\}_{j\in \Zgeq}\subset \CinftyCptSpace[\Bngeq{\sigma_2}]\)
            with \(\partial_{x_n}^l\big|_{x_n=0} f_j=0\), for \(0\leq l\leq L\), and \(f_j\rightarrow u\)
            in \(\ASpace{s}{p}{q}[\BngeqClosure{\sigma_2}][\FilteredSheafGenByXXdv]\).
    \end{enumerate}
\end{lemma}
\begin{proof}
    \ref{Item::Trace::CharVanish::ApproxBySmoothWithVanish::smalls} follows immediately from Corollary \ref{Cor::Spaces::Approximation::SmoothFunctionsAreDense} \ref{Item::Spaces::Approximation::SmoothFunctionsAreDense::ApproxInST}.

    For \ref{Item::Trace::CharVanish::ApproxBySmoothWithVanish::larges} fix \(\sigma_{a},\sigma_b\in (\sigma_1,\sigma_2)\).
    with \(\sigma_1<\sigma_a<\sigma_b<\sigma_2\).
    By Corollary \ref{Cor::Spaces::Approximation::SmoothFunctionsAreDense} \ref{Item::Spaces::Approximation::SmoothFunctionsAreDense::ApproxInST},
    \(\exists \{g_j\}_{j\in \Zgeq}\subset \CinftyCptSpace[\Bngeq{\sigma_{a}}]\) with 
    \(g_j\rightarrow u\) in \(\ASpace{s}{p}{q}[\BngeqClosure{\sigma_{a}}][\FilteredSheafGenByXXdv]\).
    Set \(f_j:=g_j-\TraceInverseMap[\sN_L]\TraceMap[\sN_L]g_j\in \CinftyCptSpace[\Bngeq{\sigma_2}]\),
    where \(\TraceInverseMap[\sN_L]\) is as in Theorem \ref{Thm::Trace::Dirichlet::MainInverseThm}
    with \(\Omega_1=\Bngeq{\sigma_{a}}\), \(\Omega_2=\Bngeq{\sigma_b}\), and \(\Omega=\Bngeq{\sigma_2}\).
    Note that \(\TraceMap[\sN_L]f_j=0\) (by Theorem \ref{Thm::Trace::Dirichlet::MainInverseThm} \ref{Item::Trace::Dirichlet::MainInverseThm::IsInverse}),
    and therefore by Theorem \ref{Thm::Trace::ForwardMap} \ref{Item::Trace::ForwardMap::TraceOnSmooth}, \(\partial_{x_n}^l\big|_{x_n=0}f_j=0\), \(\forall 0\leq l\leq L\).

    By Theorem \ref{Thm::Trace::ForwardMap} \ref{Item::Trace::ForwardMap::Continuous} and
    Theorem \ref{Thm::Trace::Dirichlet::MainInverseThm} \ref{Item::Trace::Dirichlet::MainInverseThm::ContinuousOnSpaces},
    we have \(f_j\rightarrow u-\TraceInverseMap[\sN_L]\TraceMap[\sN_L] u=u\) in  \(\ASpace{s}{p}{q}[\BngeqClosure{\sigma_2}][\FilteredSheafGenByXXdv]\).
\end{proof}

\begin{lemma}\label{Lemma::Trace::CharVanish::ApproxBySmoothWithSupportAwayFromZero}
    Let \(\ASpace{s}{p}{q}\) be either \(\BesovSpace{s}{p}{q}\) or \(\TLSpace{s}{p}{q}\) and fix \(p,q\in (1,\infty)\).
    Fix \(f\in \CinftyCptSpace[\Bngeq{\sigma_2}]\).
    \begin{enumerate}[(i)]
        \item\label{Item::Trace::CharVanish::ApproxBySmoothWithSupportAwayFromZero::smalls} If \(s<\lambda/p\),  \(\exists \{f_j\}_{j\in \Zgeq}\subset \CinftyCptSpace[\Bng{\sigma_2}]\) with
            \(f_j\rightarrow f\) in \(\ASpace{s}{p}{q}[\BngeqClosure{\sigma_2}][\FilteredSheafGenByXXdv]\).
        \item\label{Item::Trace::CharVanish::ApproxBySmoothWithSupportAwayFromZero::larges} If \(s<(L+1)\lambda+\lambda/p\) with \(L\in \Zgeq\) and \(\partial_{x_n}^l\big|_{x_n=0} f=0\) for \(0\leq l\leq L\),
            \(\exists \{f_j\}_{j\in \Zgeq}\subset \CinftyCptSpace[\Bng{\sigma_2}]\) with \(f_j\rightarrow f\)
            in \(\ASpace{s}{p}{q}[\BngeqClosure{\sigma_2}][\FilteredSheafGenByXXdv]\).
    \end{enumerate}
\end{lemma}
\begin{proof}
    \ref{Item::Trace::CharVanish::ApproxBySmoothWithSupportAwayFromZero::smalls} is the case of 
    \ref{Item::Trace::CharVanish::ApproxBySmoothWithSupportAwayFromZero::larges} with \(L=-1\). We proceed by allowing
    \(L=-1\) and prove
    \ref{Item::Trace::CharVanish::ApproxBySmoothWithSupportAwayFromZero::smalls} and
    \ref{Item::Trace::CharVanish::ApproxBySmoothWithSupportAwayFromZero::larges}
    simultaneously.

    Let \(\phi_0(x_n)\in \CinftySpace[\lbrack 0,\infty)]\) be equal to \(0\) on \([0,1/2^{\lambda}]\) and equal to \(1\)
    on \([1,\infty)\). Set,
    \begin{equation*}
        g_k(x', x_n)=\phi_0(2^{\lambda k}x_n) f(x',x_n)\in \CinftyCptSpace[\Bng{\sigma_2}].
    \end{equation*}
    Note that 
    \begin{equation}\label{Eqn::Trace::CharVanish::ApproxBySmoothWithSupportAwayFromZero::ConvgInDist}
        g_k\rightarrow f\text{ in }\DistributionsZero[\Bngeq{1}].
    \end{equation}
    We will show
    \begin{equation}\label{Eqn::Trace::CharVanish::ApproxBySmoothWithSupportAwayFromZero::ToShow}
        \sup_{k\geq 4} \ANorm{g_k}{s}{p}{q}[\FilteredSheafGenByXXdv]<\infty.
    \end{equation}
    The result follows from 
    \eqref{Eqn::Trace::CharVanish::ApproxBySmoothWithSupportAwayFromZero::ConvgInDist}
    and
    \eqref{Eqn::Trace::CharVanish::ApproxBySmoothWithSupportAwayFromZero::ToShow}
    using Proposition \ref{Prop::Spaces::Approximation::DistributionConvgToStronger} \ref{Item::Spaces::Approximation::DistributionConvgToStronger::StrongConvg};
    this is where we use \(1<p,q<\infty\).

    We turn to proving \eqref{Eqn::Trace::CharVanish::ApproxBySmoothWithSupportAwayFromZero::ToShow}.
    Since \(\partial_{x_n}^l\big|_{x_n=0}f=0\) for \(0\leq l\leq L\), we have \(f(x',x_n)=x_n^{L+1} g(x',x_n)\)
    for some \(g\in \CinftyCptSpace[\Bngeq{\sigma_2}]\).  With \(\VSpace{p}{q}\) as in Notation \ref{Notation::Spaces::Classical::VSpacepq},
    \eqref{Eqn::Trace::CharVanish::ApproxBySmoothWithSupportAwayFromZero::ToShow} is equivalent to
    \begin{equation}\label{Eqn::Trace::CharVanish::ApproxBySmoothWithSupportAwayFromZero::Tmp1}
        \sup_{k\geq 4} \BVNorm{\left\{2^{js} D_j \phi_0\left( 2^{\lambda k}x_n \right) x_n^{L+1}g(x) \right\}_{j\in \Zgeq}}{p}{q}<\infty.
    \end{equation}
    Since \(\VNorm{\cdot}{p}{q}\leq \lqLpNormNoSet{\cdot}{p}{1}\), it suffices to show
    \begin{equation}\label{Eqn::Trace::CharVanish::ApproxBySmoothWithSupportAwayFromZero::Tmp2}
        \sup_{k\geq 4} \BlqLpNorm{\left\{2^{js} D_j \phi_0\left( 2^{\lambda k}x_n \right) x_n^{L+1}g(x) \right\}_{j\in \Zgeq}}{p}{1}[\Rngeq]<\infty.
    \end{equation}
    \eqref{Eqn::Trace::CharVanish::ApproxBySmoothWithSupportAwayFromZero::Tmp2} follows immediately from
    Proposition \ref{Prop::Trace::CharVanish::MainEstimate}
    with \(\kappa=L+1\).
\end{proof}

\begin{proof}[Completion of the proof of Proposition \ref{Prop::Traces::CoordReduction::VanishingChar}]
    We take \(\sigma_1=1/2\) and \(\sigma_2=3/4\).
    \ref{Item::Traces::CoordReduction::VanishingChar::smalls} (respectively, \ref{Item::Traces::CoordReduction::VanishingChar::larges}) follows by combining
    Lemma \ref{Lemma::Trace::CharVanish::ApproxBySmoothWithVanish} \ref{Item::Trace::CharVanish::ApproxBySmoothWithVanish::smalls} (respectively, \ref{Item::Trace::CharVanish::ApproxBySmoothWithVanish::larges})
    with Lemma \ref{Lemma::Trace::CharVanish::ApproxBySmoothWithSupportAwayFromZero} \ref{Item::Trace::CharVanish::ApproxBySmoothWithSupportAwayFromZero::smalls} (respectively, \ref{Item::Trace::CharVanish::ApproxBySmoothWithSupportAwayFromZero::larges}).
\end{proof}

    \subsection{Characteristic points and the failure of our results}
    \label{Section::Trace::CharacteristicFailure}
    Let \(\ManifoldN=\left\{ (x,y)\in \R^2 : y\geq x^2 \right\}\), with
boundary \(\BoundaryN=\left\{ (x,x^2)\in \R^2 :x\in \R \right\}\cong \R\).
Consider the H\"ormander vector fields with formal degrees
\(\WWo:=\left\{ \left( \partial_x,1 \right), \left( x\partial_y,1 \right) \right\}\).
In this section, we show that the range of the trace map is not what one might expect,
and moreover it is unclear what the range is. Since the claim of this section is that
a result is ``not true,'' one could choose many slight variations of the same possible result
and the same argument shows that any such slight variation will not work. Thus, we proceed
somewhat informally and just present the main idea; in particular, we do not worry about
localizing the below proofs, and leave any relevant needed localization to the reader.

Let \(\FilteredSheafF=\FilteredSheafGenBy{\WWo}\).
Note that \((x,x^2)\in \BoundaryN\) is \(\FilteredSheafF\)-non-characteristic if and only if
\(x\ne 0\).

Note that \(\TangentSpace{x}{\BoundaryN} = \R \left( \partial_x+2x\partial_y \right)\)
and \(\partial_x+2x\partial_y\in \FilteredSheafF[\ManifoldN][1]\).
Therefore \(\RestrictFilteredSheaf{\FilteredSheafNoSetF}{\BoundaryN}[d]=\VectorFields{\BoundaryN}\),
for every \(d\).
Because of this, the spaces
\(\ASpace{s}{p}{q}[\RestrictFilteredSheaf{\LieFilteredSheafF}{\BoundaryN}]\)
locally equal the classical spaces \(\ASpace{s}{p}{q}[\R]\).

In light of Corollary \ref{Cor::GlobalCors::TraceThm},
away from \(x=0\), the trace map \(f\mapsto f\big|_{\BoundaryN}\)
is locally continuous and surjective as a map \(\TLSpace{s}{2}{2}[\ManifoldN][\FilteredSheafF]\rightarrow \TLSpace{s-1/2}{2}{2}[\R]\),
\(s>1/2\), where \(\TLSpace{s-1/2}{2}{2}[\R]=\BesovSpace{s-1/2}{2}{2}[\R]=H^s(\R)\) 
denotes the classical space on \(\R\) (i.e., the \(\LpSpace{2}\)-Sobolev space of order \(s-1/2\) on \(\R\)).
The main remark of this section is that this map is not continuous near \(x=0\): we lose
at least \(1\)
derivative, not \(1/2\).

To see this, set \(G(x,y)=e^{-x^2-y^2}\) and \(G_{\lambda}(x,y)=G(\lambda x, \lambda^2 y)\).
For \(L\in \Zg\), using Proposition \ref{Prop::Spaces::EqualsSobolev},
we see (for \(\lambda\) large)
\begin{equation*}
    \TLNorm{G_\lambda}{L}{2}{2}[\FilteredSheafF]\approx 
    \sum_{|\alpha|\leq L} \BLpNorm{ \left( \partial_x,x\partial_y \right)^{\alpha}G_\lambda }{2}[\ManifoldN]
    \approx \sum_{|\alpha|\leq L} \lambda^{|\alpha|-3/2}
    \approx
    \lambda^{L-3/2}.
\end{equation*}
Note that \(G_\lambda\big|_{\BoundaryN}=F_\lambda(x)\), where \(F(x)=e^{-x^2-x^4}\) and \(F_\lambda(x)=F(\lambda x)\).
It follows that (for \(s>0\)),
\begin{equation*}
    \TLNorm{G_\lambda\big|_{\BoundaryN}}{s}{2}{2}[\R]
    =
    \TLNorm{F_\lambda}{s}{2}{2}[\R]\approx
    \|F_\lambda\|_{H^s(\R)}
    \approx
    \lambda^{s-1/2}.
\end{equation*}
Thus if \(f\mapsto f\big|_{\BoundaryN}\) were continuous 
\(\TLSpace{L}{2}{2}[\ManifoldN][\FilteredSheafF]\rightarrow \TLSpace{s}{2}{2}[\R]\)
we would need \(\lambda^{s-1/2}\lesssim \lambda^{L-3/2}\) for 
\(\lambda\) large; 
i.e., \(s\leq L-1\).
In other words, the trace map must lose at least \(1\) derivative near \(0\).
Since, as described earlier, the trace map loses only \(1/2\) of a derivative away from \(x=0\),
it is unclear how to characterize the image of the trace map.

\section{Zygmund--H\"older spaces and compositions}
Throughout this chapter, let \(\ManifoldN\) be a smooth manifold with boundary,
\(\FilteredSheafF\) a H\"ormander filtration of sheaves of vector fields on \(\ManifoldN\),
and \(\Compact\Subset \ManifoldNncF\) compact.

\begin{definition}
    For \(s>0\),
    we let \(\ZygSpaceCompactF{s}:=\BesovSpace{s}{\infty}{\infty}[\Compact][\FilteredSheafF]\)
    with equality of norms:
    the \DefnIt{Zygmund--H\"older space of order \(s\) with respect to \(\FilteredSheafF\)}.
\end{definition}

The spaces \(\ZygSpaceCompactF{s}\) are closely related to the naturally defined H\"older spaces with respect
to the Carnot--Carath\'eodory metric. See Section \ref{Section::Zyg::Holder}.
As with the classical elliptic theory, they play a crucial role in the study of maximally subelliptic PDEs--see
\cite[Chapter 7]{StreetMaximalSubellipticity}.

Throughout this chapter, 
fix \(\Omega\Subset \ManifoldN\) open and relatively compact with \(\Compact\Subset \Omega\), and let
\(\WWdv=\left\{ \left( W_1,\Wdv_1 \right),\ldots, \left( W_r, \Wdv_r \right) \right\}\subset \VectorFields{\Omega}\times\Zg\)
be H\"ormander vector fields with formal degrees on \(\Omega\) such that
\(\FilteredSheafF\big|_{\Omega}=\FilteredSheafGenBy{\WWdv}\) (see Lemma \ref{Lemma::Filtrations::GeneratorsOnRelCptSet}).

    \subsection{The norm}\label{Section::Zyg::Norm}
    The norm \(\ZygNormF{\cdot}{s}=\BesovNorm{\cdot}{s}{\infty}{\infty}[\FilteredSheafF]\) was defined
in Notation \ref{Notation::Spaces::Defns::Norm} as follows.
Let \(\psi\in \CinftyCptSpace[\ManifoldNncF]\) equal \(1\) on a neighborhood of \(\Compact\).
Let \(\left\{ \left( D_j, 2^{-j} \right) : j\in \Zgeq \right\}\in \ElemzF{\ManifoldNncF}\)
satisfy \(\sum_{j\in \Zgeq} D_j =\Multpsi\). Then,
\begin{equation*}
    \ZygNormF{f}{s}=\BesovNorm{f}{s}{\infty}{\infty}[\FilteredSheafF]
    =\sup_{j\in \Zgeq} 2^{js} \LpNorm{D_j f}{\infty},
\end{equation*}
and since \(s>0\), Proposition \ref{Prop::Spaces::BasicProps::FiniteNorm}
shows
\begin{equation}\label{Eqn::Zyg::Norm::SpaceEqualsDistWithFiniteNorm}
    \ZygSpaceCompactF{s} = 
    \left\{ f\in \DistributionsZeroN : \supp(f)\subseteq \Compact, \ZygNormF{f}{s}<\infty \right\}.
\end{equation}

Proposition \ref{Prop::Spaces::Defns::EquivNorms} shows that the equivalence class of this norm does not 
depend on the choices made. We use this flexibility to pick \(D_j\) which are useful for
studying \(\ZygSpaceCompactF{s}\).
Namely, assume \(\psi\in \CinftyCptSpace[\Omega]\) and take
\(\left\{ (D_j, 2^{-j}) : j\in \Zgeq \right\}\in \ElemzF{\Omega}\) as in Proposition \ref{Prop::Spaces::LP::DjExist}.
Importantly, if \(P_j=\sum_{k=0}^j D_k\), then
Proposition \ref{Prop::Spaces::LP::DjExist} \ref{Item::Spaces::LP::DjExist::SumToPreElem}
gives \(\left\{ \left( P_j, 2^{-j} \right):j\in \Zgeq \right\}\in \PElemzF{\Omega}\).
Note that \(P_j f=\sum_{k=0}^j D_j f\xrightarrow{j\rightarrow \infty} f\), \(\forall f\in \ZygSpaceCompactF{s}\).
It follows from Lemma \ref{Lemma::Spaces::Elem::PElem::PElemOpsBoundedOnLp} that
\begin{equation}\label{Eqn::Zyg::Norm::DerivPjBounded}
    \sup_{j\in \Zgeq} \BLpOpNorm{\left( 2^{-j\Wdv}W \right)^{\alpha} P_j}{\infty}<\infty, \quad \forall \alpha.
\end{equation}

    \subsection{Decomposition into smooth functions}
    Similar to the classical setting on \(\Rn\), one of the easiest ways to study
functions in \(\ZygSpaceCompactF{s}\) is to decompose them into a sum of smooth functions.
The next proposition gives such a decomposition.

\begin{proposition}\label{Prop::Zyg::Decomp::MainDecomp}
    Let \(s>0\). For \(f\in \DistributionsZeroN\), the following are equivalent.
    \begin{enumerate}[(i)]
        \item\label{Item::Zyg::Decomp::MainDecomp::fInZyg} \(f\in \ZygSpaceCompactF{s}\).
        \item\label{Item::Zyg::Decomp::MainDecomp::Decompf} \(\supp(f)\subseteq \Compact\) and \(\exists \left\{ f_j \right\}\subset \CinftyCptSpace[\Omega]\) such that
            \begin{equation}\label{Eqn::Zyg::Decomp::MainDecomp::Boundfj}
                \sup_{j\in \Zgeq} 2^{js}\BLpNorm{\left( 2^{-j\Wdv}W \right)^{\alpha} f_j}{\infty}<\infty, \quad \forall \alpha,
            \end{equation}
            and \(f=\sum_{j\in \Zgeq} f_j\), with convergence in \(\CSpace{\ManifoldN}\) and in \(\DistributionsZeroN\).
    \end{enumerate}
    In this case, \(\exists M=M(s)\in \Zgeq\) such that
    \begin{equation}\label{Eqn::Zyg::Decomp::MainDecomp::BoundfByfj}
        \ZygNormF{f}{s} \leq C \sum_{|\alpha|\leq M} \sup_{j\in \Zgeq} 2^{js}\BLpNorm{\left( 2^{-j\Wdv}W \right)^{\alpha} f_j}{\infty}.
    \end{equation}
    Furthermore, \(f_j\) may be chosen such that
    \begin{equation}\label{Eqn::Zyg::Decomp::MainDecomp::BoundfjByf}
        \sup_{j\in \Zgeq} 2^{js}\BLpNorm{\left( 2^{-j\Wdv}W \right)^{\alpha} f_j}{\infty}\leq C_\alpha \ZygNormF{f}{s},\quad \forall \alpha.
    \end{equation}
    Here, \(C\) and \(C_\alpha\) do not depend on \(f\), but may depend on any other ingredient in the proposition.
\end{proposition}

To prove Proposition \ref{Prop::Zyg::Decomp::MainDecomp}, we use the next lemma.

\begin{lemma}\label{Lemma::Zyg::Decomp::DerivDjfBound}
    \(\forall s>0\), \(\forall \alpha\), \(\exists C_\alpha\geq 0\), \(\forall f\in \ZygSpaceCompactF{s}\),
    \begin{equation*}
        \sup_{j\in \Zgeq} 2^{js} \BLpNorm{ \left( 2^{-j\Wdv}W \right)^{\alpha} D_j f }{\infty}
        \leq C_{\alpha}\ZygNormF{f}{s},
    \end{equation*}
    where \(D_j\) is as in Section \ref{Section::Zyg::Norm}.
\end{lemma}
\begin{proof}
    Set \(E_{j,\alpha}:=\left( 2^{-j \Wdv}W \right)^{\alpha} D_j\).
    Proposition \ref{Prop::Spaces::Elem::Elem::MainProps} \ref{Item::Spaces::Elem::Elem::DerivOp}
    shows \(\sE_\alpha:=\left\{ \left( E_{j,\alpha},2^{-j} \right) : j\in \Zgeq \right\}\in \ElemzF{\Omega}\).
    Using Corollary \ref{Cor::Spaces::MainEst::VpqsESeminormIsContinuous}, we have
    \begin{equation*}
        \begin{split}
            &\sup_{j\in \Zgeq} 2^{js}\BLpNorm{ \left( 2^{-j\Wdv}W \right)^{\alpha} D_j f }{\infty}
            =\sup_{j\in \Zgeq} 2^{js}\BLpNorm{ E_{j,\alpha} f }{\infty}
            =\VpqsENorm{f}[\infty][\infty][s][\sE_{\alpha}]
            \lesssim \BesovNorm{f}{s}{\infty}{\infty}[\FilteredSheafF]
            =\ZygNormF{f}{s}.
        \end{split}
    \end{equation*}
\end{proof}

\begin{proof}[Proof of Proposition \ref{Prop::Zyg::Decomp::MainDecomp}]
    Suppose \ref{Item::Zyg::Decomp::MainDecomp::fInZyg} holds, i.e.,
    \(f\in \ZygSpaceCompactF{s}\).  Then, by definition (see Definition \ref{Defn::Spaces::Defns::ASpace}), \(\supp(f)\subseteq \Compact\).
    Let \(D_j\) be as Section \ref{Section::Zyg::Norm}, and set \(f_j:=D_j f\).
    Since \(D_j\in \CinftyCptSpace[\Omega\times \Omega]\)
    and \(D_j(x,y)\) vanishes to infinite order as \(y\rightarrow \BoundaryN\)
    (see Definition \ref{Defn::Spaces::LP::PElemWWdv}), 
    we have \(D_j f\in\CinftyCptSpace[\Omega]\).
    By Proposition \ref{Prop::Spaces::Elem::Elem::ConvergenceOfElemOps}, \(\sum_{j\in \Zgeq} f_j=\sum_{j\in \Zgeq} D_j f =f\),
    with convergence in \(\DistributionsZeroN\).
    Lemma \ref{Lemma::Zyg::Decomp::DerivDjfBound}
    shows \eqref{Eqn::Zyg::Decomp::MainDecomp::BoundfjByf} holds (and therefore \eqref{Eqn::Zyg::Decomp::MainDecomp::Boundfj} holds).
    Finally, \eqref{Eqn::Zyg::Decomp::MainDecomp::BoundfjByf} implies \(\sum_{j\in \Zgeq} f_j\) converges in \(\CSpace{\ManifoldN}\);
    completing the proof of \ref{Item::Zyg::Decomp::MainDecomp::Decompf}.

    Suppose \ref{Item::Zyg::Decomp::MainDecomp::Decompf} holds.
    Fix \(M>s\). We claim, \(\forall j,k\in \Zgeq\),
    \begin{equation}\label{Eqn::Zyg::Decomp::MainDecomp::BoundDjfk}
        \BLpNorm{D_j f_k}{\infty}\lesssim 2^{-M|(j\vee k)-k |} \sum_{|\alpha|\leq M} \BLpNorm{ \left( 2^{-k\Wdv}W \right)^{\alpha} f_k}{\infty}.
    \end{equation}
    Indeed, for \(j\leq k\), \eqref{Eqn::Zyg::Decomp::MainDecomp::BoundDjfk} follows immediately
    from Lemma \ref{Lemma::Spaces::Elem::PElem::PElemOpsBoundedOnLp}. Suppose \(j>k\).
    using Proposition \ref{Prop::Spaces::Elem::Elem::MainProps} \ref{Item::Spaces::Elem::Elem::PullOutNDerivs}, we may write
    \(D_j=\sum_{|\alpha|\leq M} 2^{-j(M-|\alpha|)}E_{j,\alpha}\left( 2^{-j\Wdv}W \right)^{\alpha}\), where
    \(\left\{ \left( E_{j,\alpha},2^{j} \right) : j\in \Zgeq, |\alpha|\leq M \right\}\in \ElemzF{\Omega}\).
    Using Lemma \ref{Lemma::Spaces::Elem::PElem::PElemOpsBoundedOnLp}, we have
    \begin{equation*}
    \begin{split}
            &\BLpNorm{D_j f_k}{\infty}
            \leq \sum_{|\alpha|\leq M}2^{-j(M-|\alpha|)} \BLpNorm{E_{j,\alpha} \left( 2^{-j\Wdv}W \right)^{\alpha} f_k}{\infty}
            \\&\lesssim \sum_{|\alpha|\leq M} 2^{-j(M-|\alpha|)-(j-k)\DegWdv{\alpha}} \BLpNorm{\left( 2^{-k\Wdv}W \right)^{\alpha} f_k}{\infty}
            \leq \sum_{|\alpha|\leq M} 2^{-M(j-k)} \BLpNorm{\left( 2^{-k\Wdv}W \right)^{\alpha} f_k}{\infty},
    \end{split}    
    \end{equation*}
    establishing \eqref{Eqn::Zyg::Decomp::MainDecomp::BoundDjfk}.

    Using \eqref{Eqn::Zyg::Decomp::MainDecomp::BoundDjfk}, we have with \(f=\sum_{j\in \Zgeq} f_j\),
    \begin{equation}\label{Eqn::Zyg::Decomp::MainDecomp::BoundDjfTmp}
    \begin{split}
         &\BLpNorm{D_j f}{\infty}
         \leq \sum_{k\in \Zgeq}\BLpNorm{D_j f_k}{\infty}
         \lesssim 2^{-M|(j\vee k)-k |} \sum_{|\alpha|\leq M} \BLpNorm{ \left( 2^{-k\Wdv}W \right)^{\alpha} f_k}{\infty}
         \\& \leq \left( \sum_{|\alpha|\leq M} \sup_{k\in \Zgeq} 2^{ks}\BLpNorm{\left( 2^{-k\Wdv}W \right)^{\alpha} f_k}{\infty}  \right)
         \left( \sum_{k\in \Zgeq} 2^{-M|(j\vee k )-k |-ks}  \right)
         \\&\lesssim 2^{-js}\sum_{|\alpha|\leq M} \sup_{k\in \Zgeq} 2^{ks}\BLpNorm{\left( 2^{-k\Wdv}W \right)^{\alpha} f_k}{\infty}.
    \end{split}
    \end{equation}
    Multiplying both sides of \eqref{Eqn::Zyg::Decomp::MainDecomp::BoundDjfTmp} by \(2^{js}\)
    and taking the supremum over \(j\) gives \eqref{Eqn::Zyg::Decomp::MainDecomp::BoundfByfj}.
    From here, \eqref{Eqn::Zyg::Norm::SpaceEqualsDistWithFiniteNorm}
    implies \ref{Item::Zyg::Decomp::MainDecomp::fInZyg}.
\end{proof}

\begin{corollary}\label{Cor::Zyg::Decomp::ZygAreContinuous}
    For \(s>0\), \(\ZygSpaceCompactF{s}\subseteq \CSpace{\ManifoldN}\), and the inclusion is continuous.
    In particular, for \(f\in \ZygSpaceCompactF{s}\)
    \begin{equation}\label{Eqn::Zyg::Decomp::ZygAreContinuous::BoundLinftyNorm}
        \LpNorm{f}{\infty}\leq C_{s,\Compact,\FilteredSheafF} \ZygNormF{f}{s}.
    \end{equation}
\end{corollary}
\begin{proof}
    Let \(f\in \ZygSpaceCompactF{s}\). Taking \(f_j\in \CinftyCptSpace[\Omega]\) as in Proposition \ref{Prop::Zyg::Decomp::MainDecomp} \ref{Item::Zyg::Decomp::MainDecomp::Decompf}
    so that \(f=\sum_{j\in \Zgeq} f_j\) (with convergence in \(\CSpace{\ManifoldN}\)) proves \(f\in \CSpace{\ManifoldN}\).
    Moreover, picking \(f_j\) so that \eqref{Eqn::Zyg::Decomp::MainDecomp::BoundfjByf} holds
    gives
    \begin{equation*}
        \BLpNorm{f}{\infty} \leq \sum_{k} \BLpNorm{f_k}{\infty} \lesssim \sum_{k} 2^{-ks} \ZygNormF{f}{s}\lesssim \ZygNormF{f}{s},
    \end{equation*}
    establishing \eqref{Eqn::Zyg::Decomp::ZygAreContinuous::BoundLinftyNorm} and
    proving that the inclusion is continuous.
\end{proof}

\begin{remark}\label{Rmk::Zyg::Decomp::TraceOnZyg}
    In light of Corollary \ref{Cor::Zyg::Decomp::ZygAreContinuous}, the restriction map
    \(f\mapsto f\big|_{\BoundaryN}\) is defined in the classical sense, for \(f\in \ZygSpaceCompactF{s}\).
    By continuity of the inclusion in Corollary \ref{Cor::Zyg::Decomp::ZygAreContinuous},
    this agrees \(\TraceMap[1]f\), where \(\TraceMap[1]\) 
    is the trace map from Theorem \ref{Thm::Trace::ForwardMap}.
\end{remark}

    \subsection{Algebra}
    When studying nonlinear equations, an important property is that \(\ZygSpaceCompactF{s}\) is an algebra.
In fact, stronger compositions results are true (see Theorem \ref{Thm::Zyg::Composition::MainComposition}), but in this section we focus on the algebra property.

\begin{proposition}\label{Prop::Zyg::Algebra::ZygIsAlgebra}
    Fix \(s>0\). For \(f,g\in \ZygSpaceCompactF{s}\), we have \(fg\in \ZygSpaceCompactF{s}\)
    and moreover,
    \begin{equation}\label{Eqn::Zyg::Algebra::ZygIsAlgebra::TameEst}
        \ZygNormF{fg}{s}
        \leq C\left( \ZygNormF{f}{s}\LpNorm{g}{\infty}+\LpNorm{f}{\infty}\ZygNormF{g}{s}  \right),
    \end{equation}
    where \(C=C(s,\Compact,\FilteredSheafF)\geq 0\).
\end{proposition}

\begin{remark}\label{Rmk::Zyg::Tame}
    The estimate \eqref{Eqn::Zyg::Algebra::ZygIsAlgebra::TameEst} is stronger than the more obvious
    continuity estimate for multiplication:
    \begin{equation}\label{Eqn::Zyg::Algebra::ZygIsAlgebra::NonTameEst}
        \ZygNormF{fg}{s}\lesssim \ZygNormF{f}{s}\ZygNormF{g}{s}.
    \end{equation}
     \eqref{Eqn::Zyg::Algebra::ZygIsAlgebra::TameEst} is called a \textit{tame} estimate because 
    each summand on the right-hand side of \eqref{Eqn::Zyg::Algebra::ZygIsAlgebra::TameEst}
    involves at most one term involving \(\ZygNorm{\cdot}{s}\); and the other terms do not depend on \(s\)
    (in particular, \eqref{Eqn::Zyg::Algebra::ZygIsAlgebra::NonTameEst} is not a tame estimate).
    Estimates like \eqref{Eqn::Zyg::Algebra::ZygIsAlgebra::TameEst} (and the more general ones described in Theorem \ref{Thm::Zyg::Composition::MainComposition})
    are essential in the study of nonlinear PDEs.
    See \cite{HamiltonInverseFunctionTheoremOfNashAndMoser} for a general framework involving such estimates,
    and \cite[Chapter 9]{StreetMaximalSubellipticity} for applications of tame estimates in the study of nonlinear
    subelliptic equations.
\end{remark}

\begin{proof}[Proof of Proposition \ref{Prop::Zyg::Algebra::ZygIsAlgebra}]
    For \(f,g\in \ZygSpaceCompactF{s}\), we have \((P_j f)(P_j g)\rightarrow fg\) in \(\CSpace{\ManifoldN}\).
    Thus,
    \begin{equation*}
    \begin{split}
            &fg=(P_0f) (P_0 g)+\sum_{j=1}^\infty \left[ (P_{j} f)(P_{j}g)-(P_{j-1}f)(P_{j-1}g) \right] 
            \\&= (P_0f) (P_0 g) + \sum_{j=1}^\infty \left( D_j f \right) \left( P_j g \right) + \sum_{j=1}^\infty \left( P_{j-1} f \right)\left( D_j g \right).
            \\&=:\sum_{j\in \Zgeq} F_j  + \sum_{j\in \Zgeq} G_j,
    \end{split}
    \end{equation*}
    where,
    \begin{equation*}
        F_j =
        \begin{cases}
            (P_0f) (P_0 g),& j=0,\\
             \left( D_j f \right) \left( P_j g \right), &j\geq 1,
        \end{cases},
        \quad\quad
        G_j=
        \begin{cases}
            0, &j=0,\\
            \left( P_{j-1} f \right)\left( D_j g \right),&j\geq 1.
        \end{cases}
    \end{equation*}
    We will show, \(\forall \alpha\),
    \begin{equation}\label{Eqn::Zyg::Algebra::ZygIsAlgebra::BoundFj}
        \sup_{j\in \Zgeq} 2^{js}\BLpNorm{\left( 2^{-j\Wdv}W \right)^{\alpha} F_j}{\infty} \lesssim \ZygNormF{f}{s}\LpNorm{g}{\infty},
    \end{equation}
    \begin{equation}\label{Eqn::Zyg::Algebra::ZygIsAlgebra::BoundGj}
        \sup_{j\in \Zgeq} 2^{js}\BLpNorm{\left( 2^{-j\Wdv}W \right)^{\alpha} G_j}{\infty} \lesssim \LpNorm{f}{\infty}\ZygNormF{g}{s}.
    \end{equation}
    Using \eqref{Eqn::Zyg::Algebra::ZygIsAlgebra::BoundFj} and \eqref{Eqn::Zyg::Algebra::ZygIsAlgebra::BoundGj}
    the result follows directly from Proposition \ref{Prop::Zyg::Decomp::MainDecomp}.

    The proofs of \eqref{Eqn::Zyg::Algebra::ZygIsAlgebra::BoundFj} and \eqref{Eqn::Zyg::Algebra::ZygIsAlgebra::BoundGj}
    are similar, so we prove only \eqref{Eqn::Zyg::Algebra::ZygIsAlgebra::BoundFj}.
    We have, using \eqref{Eqn::Zyg::Norm::DerivPjBounded},
    \begin{equation*}
        \BLpNorm{W^\alpha F_0}{\infty}
        \lesssim \sum_{|\beta_1|+|\beta_2|\leq |\alpha|} \BLpNorm{W^{\beta_1} P_0 f}{\infty}\BLpNorm{W^{\beta_2} P_0 g}{\infty}
        \lesssim \BLpNorm{f}{\infty}\BLpNorm{g}{\infty}.
    \end{equation*}
    Since \(\BLpNorm{f}{\infty}\lesssim \ZygNormF{f}{s}\) (by \eqref{Eqn::Zyg::Decomp::ZygAreContinuous::BoundLinftyNorm}),
    this proves the case \(j=0\) of \eqref{Eqn::Zyg::Algebra::ZygIsAlgebra::BoundFj}.
    For \(j\geq 1\), we have using \eqref{Eqn::Zyg::Norm::DerivPjBounded} and Lemma \ref{Lemma::Zyg::Decomp::DerivDjfBound},
    \begin{equation*}
    \begin{split}
         &\BLpNorm{\left( 2^{-j\Wdv}W \right)^{\alpha}F_j}{\infty}
         \lesssim \sum_{|\beta_1|+|\beta_2|\leq |\alpha|} \BLpNorm{\left( 2^{-j\Wdv}W \right)^{\beta_1}D_j f}{\infty} \BLpNorm{\left( 2^{-j\Wdv}W \right)^{\beta_2}P_j g}{\infty}
         \lesssim 2^{-js} \ZygNormF{f}{s} \BLpNorm{g}{\infty}.
    \end{split}
    \end{equation*}
    \eqref{Eqn::Zyg::Algebra::ZygIsAlgebra::BoundFj} follows, completing the proof.
\end{proof}

\begin{remark}\label{Rmk::Zyg::Algebra::ExtensionDoesntWork}
    It is tempting to use the extension map \(\Extension[N_0]\) (where \(N_0>s\)) from Theorem \ref{Thm::Spaces::Extension}
    to reduce Proposition \ref{Prop::Zyg::Algebra::ZygIsAlgebra} to the simpler situation of manifolds without boundary
    (which was covered in \cite[Proposition 7.4.1 and Theorem 7.5.2 (iii)]{StreetMaximalSubellipticity}). 
    Indeed, using \(\left( \Extension[N_0] f \right)\left( \Extension[N_0] g \right)\big|_{\ManifoldN}=fg\), 
    one can easily use Theorem \ref{Thm::Spaces::Extension} and \cite[Proposition 7.4.1]{StreetMaximalSubellipticity}
    to conclude \(\ZygSpaceCompactF{s}\) is an algebra and that \eqref{Eqn::Zyg::Algebra::ZygIsAlgebra::NonTameEst} holds.
    However, since we have not shown \(\Extension[N_0]:\LpSpace{\infty}[\Compact]\rightarrow \LpSpace{\infty}[\OmegaClosure]\),
    we cannot directly conclude  \eqref{Eqn::Zyg::Algebra::ZygIsAlgebra::TameEst}.
\end{remark}

    \subsection{H\"older spaces}\label{Section::Zyg::Holder}
    The classical Zygmund--H\"older spaces in \(\Rn\) are given by
\(\ZygSpace{s}[\Rn]=\BesovSpace{s}{\infty}{\infty}[\Rn]\). However,
for \(s\in (0,\infty)\setminus \Zg\), they agree with the classical H\"older spaces:
\begin{equation*}
    \ZygSpace{s}[\Rn] = \HolderSpace{\lfloor s \rfloor}{s-\lfloor s\rfloor}[\Rn], \quad s\in  (0,\infty)\setminus \Zg,
\end{equation*}
with equality of topologies; see \cite[Sections 2.2.2 and 2.5.12]{TriebelTheoryOfFunctionSpaces}.
However, when \(s=m+1\in \Zg\), we have
\begin{equation*}
    \HolderSpace{m+1}{0}[\Rn]\subsetneq \HolderSpace{m}{1}[\Rn]\subsetneq \ZygSpace{m+1}[\Rn].
\end{equation*}
Similar remarks hold on the manifold with boundary \(\Rngeq\).

In this section, we describe the connection between \(\ZygSpaceCompactF{s}\) and the H\"older spaces
with respect to the Carnot--Carath\'eodory metric on manifolds with boundary; which is very similar to the above classical setting.
For these results, and more, on a manifold without boundary, see \cite[Section 7.7]{StreetMaximalSubellipticity}.

\begin{definition}[The H\"older Spaces]
    \label{Defn::Zyg::Holder::HolderSpace}
    For \(s\in [0,1]\), and \(f\in \CSpace{\ManifoldN}\) with \(\supp(f)\subseteq \Compact\), set
    \begin{equation}\label{Eqn::Zyg::Holder::HolderSpace::Norm}
        \HolderNormF{f}{0}{s}:=\sup_{x}|f(x)|+\sup_{\substack{x,y\in \Omega \\ x\ne y}} \MetricWWdv[x][y]^{-s} \left| f(x)-f(y) \right|,\quad s\in (0,1],
    \end{equation}
    \begin{equation*}
        \HolderNormF{f}{0}{0}:=\sup_{x}|f(x)|,
    \end{equation*}
    and
    \begin{equation*}
        \HolderSpaceCompactF{0}{s}:=
        \left\{ f\in \CSpace{\ManifoldN} : \supp(f)\subseteq \Compact, \HolderNormF{f}{0}{s}<\infty \right\}.
    \end{equation*}
\end{definition}

First, we show that this Definition \ref{Defn::Zyg::Holder::HolderSpace} is well-defined:

\begin{lemma}\label{Lemma::Zyg::Holder::NormWellDefined}
    Definition \ref{Defn::Zyg::Holder::HolderSpace} does not depend on the choice of \(\Omega\) or \(\WWdv\).
    More precisely, the equivalence class of \(\HolderNormF{f}{0}{s}\), for \(f\in \CSpace{\ManifoldN}\) with \(\supp(f)\subseteq \Compact\),
    does not depend on these choices.
    Moreover, \(\HolderSpaceCompactF{0}{s}=\HolderSpace{0}{s}[\Compact][\FilteredSheafG]\) (with equality of topologies),
    where \(\FilteredSheafG\) is any H\"ormander filtration of sheaves of vector fields on \(\ManifoldN\)
    with \(\LieFilteredSheafF=\LieFilteredSheafG\).
    In particular, we have
    \(\HolderSpaceCompactF{0}{s}=\HolderSpace{0}{s}[\Compact][\LieFilteredSheafF]\) (see Lemma \ref{Lemma::Filtrations::GeneratorsForLieFiltration}).
\end{lemma}
\begin{proof}
    The result is obvious for \(s=0\), so we focus on the case \(s\in (0,1]\).
    Let \(\Omega'\Subset \ManifoldNncF\) be open and relatively compact with \(\Compact\Subset \Omega'\),
    and let \(\FilteredSheafG\big|_{\Omega'}=\FilteredSheafGenBy{\ZZde}\),
    where \(\ZZde=\left\{ \left( Z_1,\Zde_1 \right),\ldots, \left( Z_q,\Zde_1 \right) \right\}\subset \VectorFields{\Omega'}\times \Zg\)
    are H\"ormander vector fields with formal degrees on \(\Omega'\).
    We will show that if \(\WWdv\) and \(\Omega\) in \eqref{Eqn::Zyg::Holder::HolderSpace::Norm} are replaced
    with \(\ZZde\) and \(\Omega'\), one obtains a comparable norm, and the result follows.
    More precisely, we will show for \(f\in \CSpace{\ManifoldN}\) with \(\supp(f)\subseteq \Compact\),
    \begin{equation}\label{Eqn::Zyg::Holder::NormWellDefined::ToShow}
        \sup_{x}|f(x)|+\sup_{\substack{x,y\in \Omega \\ x\ne y}} \MetricWWdv[x][y]^{-s} \left| f(x)-f(y) \right|
        \approx
        \sup_{x}|f(x)|+\sup_{\substack{x,y\in \Omega' \\ x\ne y}} \MetricZZde[x][y]^{-s} \left| f(x)-f(y) \right|.
    \end{equation}

    Let \(\Omega_1\Subset \Omega\cap \Omega'\) be open and relatively compact with \(\Compact\Subset \Omega_1\).
    By compactness, \(\exists \delta_0>0\),
    \begin{equation}\label{Eqn::Zyg::Holder::NormWellDefined::DistanceFromKToComplement}
        \MetricWWdv[x][y]\geq \delta_0,\quad \forall x\in \Compact, y\in \Omega\setminus \Omega_1.
    \end{equation}
    In light of \eqref{Eqn::Zyg::Holder::NormWellDefined::DistanceFromKToComplement}, we have
    \begin{equation}\label{Eqn::Zyg::Holder::NormWellDefined::Tmp1}
        \sup_{x}|f(x)|+\sup_{\substack{x,y\in \Omega \\ x\ne y}} \MetricWWdv[x][y]^{-s} \left| f(x)-f(y) \right|\approx \sup_{x}|f(x)|+\sup_{\substack{x,y\in \Omega_1 \\ x\ne y}} \MetricWWdv[x][y]^{-s} \left| f(x)-f(y) \right|.
    \end{equation}
    and similarly,
    \begin{equation}\label{Eqn::Zyg::Holder::NormWellDefined::Tmp2}
        \sup_{x}|f(x)|+\sup_{\substack{x,y\in \Omega' \\ x\ne y}} \MetricZZde[x][y]^{-s} \left| f(x)-f(y) \right|\approx \sup_{x}|f(x)|+\sup_{\substack{x,y\in \Omega_1 \\ x\ne y}} \MetricZZde[x][y]^{-s} \left| f(x)-f(y) \right|.
    \end{equation}
    Proposition \ref{Prop::VectorFields::Scaling::VolAndMetricEquivalnce} \ref{Item::VectorFields::Scaling::VolAndMetricEquivalnce::MetricWequalsMetricZ}
    shows \(\MetricWWdv[x][y]\approx \MetricZZde[x][y]\), \(\forall x,y\in \overline{\Omega_1}\).
    Plugging this into \eqref{Eqn::Zyg::Holder::NormWellDefined::Tmp1} and \eqref{Eqn::Zyg::Holder::NormWellDefined::Tmp2}
    yields \eqref{Eqn::Zyg::Holder::NormWellDefined::ToShow} and completes the proof.
\end{proof}

% \begin{remark}\label{Rmk::Zyg::Holder::NormWellDefined}
%     There were two choices made in Definition \ref{Defn::Zyg::Holder::HolderSpace}:
%     the choice of \(\Omega\) and the choice of \(\WWdv\); however, neither of these choices affect the definition.
%     Indeed, Proposition \ref{Prop::VectorFields::Scaling::VolAndMetricEquivalnce} \ref{Item::VectorFields::Scaling::VolAndMetricEquivalnce::MetricWequalsMetricZ}
%     shows that the Lipschitz equivalence class of \(\MetricWWdv\big|_{\Compact\times \Compact}\)
%     does not depend on these choices. Moreover, it also shows that
%     \(\HolderSpaceCompactF{0}{s}=\HolderSpace{0}{s}[\Compact][\FilteredSheafG]\) (with equality of topologies),
%     where \(\FilteredSheafG\) is any H\"ormander filtration of sheaves of vector fields on \(\ManifoldN\)
%     with \(\LieFilteredSheafF=\LieFilteredSheafG\). In particular, we have
%     \(\HolderSpaceCompactF{0}{s}=\HolderSpace{0}{s}[\Compact][\LieFilteredSheafF]\) (see Lemma \ref{Lemma::Filtrations::GeneratorsForLieFiltration}).
% \end{remark}

\begin{remark}
    The notation \(\HolderNormF{\cdot}{0}{s}\) does not mention the compact set \(\Compact\).  This is justified by the fact
    that the equivalence class of the norm does not depend on the choice of \(\Compact\).
    More precisely, if \(\Compact_1,\Compact_2\Subset \ManifoldNncF\) are two compact sets, then one may choose \(\Omega\Subset \ManifoldNncF\)
    open and relatively compact with \(\Compact_1\cup\Compact_2\subseteq \Omega\), and one may define
    \(\HolderNormF{\cdot}{0}{s}\) with this choice of \(\Omega\). In light of Lemma \ref{Lemma::Zyg::Holder::NormWellDefined},
    the equivalence class of the norm does not depend on the choice of \(\Omega\).
\end{remark}

\begin{theorem}\label{Thm::Zyg::Holder::MainThm}
    For \(s\in (0,1)\), \(\ZygSpaceCompactF{s}=\HolderSpaceCompactF{0}{s}\), with equality of topologies;
    here we are identifying \(\ZygSpaceCompactF{s}\) with a subspace of \(\CSpace{\ManifoldN}\), as described in
    Corollary \ref{Cor::Zyg::Decomp::ZygAreContinuous}.
\end{theorem}

Before we prove Theorem \ref{Thm::Zyg::Holder::MainThm}, we consider the important special case when
\(\Wdv_1=\cdots=\Wdv_r=1\). I.e., the case when \(\FilteredSheafF\big|_{\Omega}=\FilteredSheafGenBy{\WWo}\),
where \(\WWo=\left\{ \left( W_1,1 \right),\ldots, \left( W_r,1 \right) \right\}\subset \VectorFields{\Omega}\times \Zg\)
and \(W_1,\ldots, W_r\) are H\"ormander vector fields on \(\Omega\).

\begin{definition}
    For \(m\in \Zgeq\), \(s\in [0,1]\), and \(f\in \CSpace{\ManifoldN}\) with \(\supp(f)\subseteq \Compact\),
    \begin{equation*}
        \HolderNormWWo{f}{0}{s}:=\HolderNormF{f}{0}{s}=\sup_{x} |f(x)|+\sup_{\substack{x,y\in \Omega\\x\ne y}} \MetricWWo[x][y]^{-s}\left| f(x)-f(y) \right|, \quad s\in (0,1],
    \end{equation*}
    \begin{equation*}
        \HolderNormWWo{f}{0}{0}:=\HolderNormF{f}{0}{s}=\sup_{x} |f(x)|,
    \end{equation*}
    \begin{equation*}
        \HolderNormWWo{f}{m}{s}:=\sum_{|\alpha|\leq m} \HolderNormWWo{W^\alpha f}{0}{s},
    \end{equation*}
    \begin{equation*}
        \HolderSpaceCompactWWo{m}{s}:=
        \left\{ f\in \CSpace{\ManifoldN} : W^{\alpha}f\in \CSpace{\ManifoldN}, \forall |\alpha|\leq m, \supp(f)\subseteq \Compact, \HolderNormWWo{f}{m}{s}<\infty \right\}.
    \end{equation*}
\end{definition}

\begin{corollary}\label{Cor::Zyg::Holder::ZygSpaceEqualsHolderForWWo}
    Suppose \(\FilteredSheafF\big|_{\Omega}=\FilteredSheafGenBy{\WWo}\). Then, 
    \(\forall m\in \Zgeq\), \(\forall s\in (0,1)\), we have
    \begin{equation*}
        \ZygSpaceCompactF{m+s}=\HolderSpaceCompactWWo{m}{s},
    \end{equation*}
    with equality of topologies.
\end{corollary}
\begin{proof}
    This follows by combining Proposition \ref{Prop::Spaces::BasicProps::OneDerivsGivesNorms}
    and Theorem \ref{Thm::Zyg::Holder::MainThm}.
\end{proof}

We turn to the proof of Theorem \ref{Thm::Zyg::Holder::MainThm}. We proceed by reducing the result
to the simpler case of manifolds without boundary, which was covered in \cite{StreetMaximalSubellipticity}. We
state this case as a lemma.

\begin{lemma}\label{Lemma::Zyg::Holder::ManifoldsWithoutBoundary}
    Theorem \ref{Thm::Zyg::Holder::MainThm} holds when \(\ManifoldN\) does not have a boundary.
\end{lemma}
\begin{proof}
    This is \cite[Theorem 7.7.23]{StreetMaximalSubellipticity}.
\end{proof}

Let \(\ManifoldM\) be a smooth manifold without boundary
such that \(\ManifoldN\subseteq \ManifoldM\) is a closed, co-dimension \(0\), embedded, submanifold, 
and \(\FilteredSheafFh\) a H\"ormander filtration
of sheaves of vector fields on \(\ManifoldM\) such that \(\RestrictFilteredSheaf{\FilteredSheafFh}{\ManifoldN}=\FilteredSheafF\)
(such a choice of \(\ManifoldM\) and \(\FilteredSheafFh\) always exist--see Proposition \ref{Prop::Filtrations::RestrictingFiltrations::CoDim0CoRestriction}).

\begin{lemma}\label{Lemma::Zyg::Holder::ZygInHolder}
    For \(0<s<1\), \(\ZygSpaceCompactF{s}\subseteq \HolderSpaceCompactF{0}{s}\), and the inclusion is continuous.
\end{lemma}
\begin{proof}
    Let \(f\in \ZygSpaceCompactF{s}\). We wish to show
    \begin{equation}\label{Eqn::Zyg::Holder::ZygInHolder::ToShow}
        \HolderNormF{f}{0}{s} \lesssim \ZygNormF{f}{s}.
    \end{equation}
    \eqref{Eqn::Zyg::Holder::ZygInHolder::ToShow} is clearly true for \(f=0\), so we may assume
    \(\ZygNormF{f}{s}>0\).
    Fix \(U_1\Subset U_2\Subset \ManifoldM\), \(\ManifoldM\)-open, relatively compact sets with \(\Omega\Subset U_1\).
    Let \(\FilteredSheafFh\big|_{U_2}=\FilteredSheafGenBy{\ZZde}\), where \(\ZZde\)
    are H\"ormander vector fields with formal degrees on \(U_2\) (see Lemma \ref{Lemma::Filtrations::GeneratorsOnRelCptSet}).
    By Corollary \ref{Cor::Spaces::Extend::Consequences::RestrictionSpaceDefn}, 
    \(\exists g\in \ZygSpace{s}[\overline{U_1}][\FilteredSheafFh]\), with \(g\big|_{\ManifoldN}=f\)
    and \(\ZygNormFh{g}{s}\lesssim \ZygNormF{f}{s}\).
    By Lemma \ref{Lemma::Zyg::Holder::ManifoldsWithoutBoundary},
    \begin{equation}\label{Eqn::Zyg::Holder::ZygInHolder::Tmp1}
        \HolderNormFh{g}{0}{s}\approx \ZygNormFh{g}{s}\lesssim \ZygNormF{f}{s}.
    \end{equation}
    Let \(\Omega_1\Subset \Omega\) be \(\ManifoldN\)-open and relatively compact, with \(\Compact\Subset \Omega\).
    \eqref{Eqn::Zyg::Holder::ZygInHolder::Tmp1} implies
    \begin{equation}\label{Eqn::Zyg::Holder::ZygInHolder::Tmp2}
    \begin{split}
         &\sup_{x}|f(x)|+\sup_{\substack{x,y\in \Omega_1 \\ x\ne y}} \MetricZZde[x][y]^{-s} \left| f(x)-f(y) \right|
         \leq \sup_{x}|g(x)|+\sup_{\substack{x,y\in U_2 \\ x\ne y}} \MetricZZde[x][y]^{-s} \left| g(x)-g(y) \right|
         \\&=\HolderNormFh{g}{0}{s}\lesssim \ZygNormF{f}{s}.
    \end{split}
    \end{equation}
    Proposition \ref{Prop::VectorFields::Scaling::AmbientVolAndMetricEquivalnce} \ref{Item::VectorFields::Scaling::AmbientVolAndMetricEquivalnce::MetricWequalsMetricZ}
    shows \(\MetricZZde[x][y]\approx \MetricWWdv[x][y]\), \(\forall x,y\in \overline{\Omega_1}\).
    Plugging this into the left-hand side of \eqref{Eqn::Zyg::Holder::ZygInHolder::Tmp2}
    shows
    \begin{equation*}
        \sup_{x}|f(x)|+\sup_{\substack{x,y\in \Omega_1 \\ x\ne y}} \MetricWWdv[x][y]^{-s} \left| f(x)-f(y) \right|
        \lesssim \ZygNormF{f}{s}.
    \end{equation*}
    In light of Lemma \ref{Lemma::Zyg::Holder::NormWellDefined} this gives \eqref{Eqn::Zyg::Holder::ZygInHolder::ToShow}.
\end{proof}

Lemma \ref{Lemma::Zyg::Holder::ZygInHolder} proves half of Theorem \ref{Thm::Zyg::Holder::MainThm};
we turn to the reverse containment: \(\HolderSpaceCompactF{0}{s}\subseteq \ZygSpaceCompactF{s}\).
This takes a bit more work to reduce to Lemma \ref{Lemma::Zyg::Holder::ManifoldsWithoutBoundary}.
We begin with a particular closed subspace of \(\HolderSpaceCompactF{0}{s}\).

\begin{lemma}\label{Lemma::Zyg::Holder::VanishHolderInZyg}
    \(\left\{ f\in \HolderSpaceCompactF{0}{s} : f\big|_{\BoundaryN}=0 \right\}\subseteq \ZygSpaceCompactF{s}\)
    and the inclusion is continuous.
\end{lemma}
\begin{proof}
    Let \(f\in \HolderSpaceCompactF{0}{s}\).
    Fix \(\Omega_1\Subset \Omega\) open and relatively compact in \(\Omega\), with \(\Compact\Subset \Omega_1\).
    Take \(U_1\Subset U_2\Subset \ManifoldM\), \(\ManifoldM\)-open, relatively compact sets with 
    \(U_1\cap \ManifoldN=\Omega_1\)  (and therefore \(\Compact\Subset U_1\)).
    Let \(\FilteredSheafFh\big|_{U_2}=\FilteredSheafGenBy{\ZZde}\), where \(\ZZde\)
    are H\"ormander vector fields with formal degrees on \(U_2\) (see Lemma \ref{Lemma::Filtrations::GeneratorsOnRelCptSet}).

    Define \(g:\ManifoldM\rightarrow \C\) by
    \begin{equation*}
        g(x)=
        \begin{cases}
            f(x), &x\in \ManifoldN,\\
            0,&x\in \ManifoldM\setminus \ManifoldN.
        \end{cases}
    \end{equation*}
    We claim \(g\in \HolderSpace{0}{s}[\Compact][\FilteredSheafFh]\) with
    \begin{equation}\label{Eqn::Zyg::Holder::VanishHolderInZyg::BoundExtended}
      \HolderNormFh{g}{0}{s}\lesssim  \HolderNormF{f}{0}{s},
    \end{equation}
    Clearly, \(\supp(g)\subseteq \Compact\) and \(\sup_{x}|g(x)|=\sup_{x}|f(x)|\).
    Fix \(\Omega_2\Subset\Omega\), \(\ManifoldN\)-open, and relatively compact in \(\Omega\), with \(\Omega_1\Subset \Omega_2\).
    We will show for \(x,y\in U_1\), \(x\ne y\),
    \begin{equation}\label{Eqn::Zyg::Holder::VanishHolderInZyg::BoundExtended::ToShow}
        \MetricZZde[x][y]^{-s}\left| g(x)-g(y) \right|
        \lesssim  
        \sup_{x} |f(x)|+
        \sup_{\substack{x',y'\in \Omega_2 \\ x'\ne y'}}\MetricWWdv[x'][y']^{-s}\left| f(x')-f(y') \right|,
    \end{equation}
    and then \eqref{Eqn::Zyg::Holder::VanishHolderInZyg::BoundExtended} will follow (using Lemma \ref{Lemma::Zyg::Holder::NormWellDefined}).

    We separate \eqref{Eqn::Zyg::Holder::VanishHolderInZyg::BoundExtended::ToShow} into three
    cases.
    If \(x,y\in U_1\setminus \Omega_1\), then \(g(x)=g(y)=0\), and \eqref{Eqn::Zyg::Holder::VanishHolderInZyg::BoundExtended::ToShow} is trivial.
    If \(x,y\in \Omega_1\), then \(g(x)=f(y)\) and \(g(y)=f(y)\) and by 
    Proposition \ref{Prop::VectorFields::Scaling::AmbientVolAndMetricEquivalnce} \ref{Item::VectorFields::Scaling::AmbientVolAndMetricEquivalnce::MetricWequalsMetricZ},
    \(\MetricZZde[x][y]\approx \MetricWWdv[x][y]\); \eqref{Eqn::Zyg::Holder::VanishHolderInZyg::BoundExtended::ToShow} follows
    in this case.
    Finally, we consider the case when \(x\in U_1\setminus \Omega\) and \(y\in \Omega\) or \(x\in \Omega\) and \(y\in U_1\setminus \Omega\).
    By symmetry, we may assume \(x\in U_1\setminus \Omega\) and \(y\in \Omega\). 
    
    % If \(\MetricZZde[x][y]=1\),
    % \eqref{Eqn::Zyg::Holder::VanishHolderInZyg::BoundExtended::ToShow} is trivial,
    % so we may assume \(\MetricZZde[x][y]<1\).
    
    Let \(V\Subset U_2\) be \(\ManifoldM\)-open and relatively compact in \(U_2\), with \(U_1\Subset V\) and \(V\cap \ManifoldN\Subset \Omega_2\), and take \(\delta_0>0\) with \(\MetricZZde[x'][y']\geq \delta_0\),
    \(\forall x'\in U_1\), \(y'\in U_2\setminus V\). 
    If \(\MetricZZde[x][y]\geq \delta_0\) then
    \eqref{Eqn::Zyg::Holder::VanishHolderInZyg::BoundExtended::ToShow} is trivial, so we assume  \(\MetricZZde[x][y]< \delta_0\).
    We claim
    \begin{equation}\label{Eqn::Zyg::Holder::VanishHolderInZyg::BoundExtended::ToShowMetric}
      \exists z\in V\cap \BoundaryN,\quad\MetricZZde[x][y]\gtrsim \MetricZZde[z][y].
    \end{equation}
    % If \(\MetricZZde[x][y]\geq \delta_0\), then since \(\MetricZZde[z][y]\leq 1\) by definition (see \eqref{Eqn::VectorFields::DefineMetric}),
    % \eqref{Eqn::Zyg::Holder::VanishHolderInZyg::BoundExtended::ToShowMetric} follows. 
    Suppose \(\MetricZZde[x][y]<\delta\)
    for some \(\delta<\delta_0\). 
    Then \(\exists \gamma:[0,1]\rightarrow U_2\) absolutely continuous, \(\gamma(0)=x\), \(\gamma(1)=y\), \(\gamma'(t)=\sum_{j=1}^q a_j(t) \delta^{\Zde_j} Z_j(t)\),
    \(\sum |a_j|^2<1\), almost everywhere. Note that \(\gamma([0,1])\subseteq \BZZde{x}{\delta}\subseteq \BZZde{x}{\delta_0}\subseteq V\) (since \(x\in U_1\)).
    Since \(\gamma(0)=x\in \ManifoldM\setminus \ManifoldN\) and \(\gamma(1)=y\in \ManifoldN\), \(\exists t_0\in [0,1]\)
    with \(z=\gamma(t_0)\in \BoundaryN\). Also, \(z\in \gamma([0,1])\subseteq V\). By definition, we have \(\MetricZZde[z][y]<\delta\).
    Taking \(\delta\leq 2\MetricZZde[x][y]\), \eqref{Eqn::Zyg::Holder::VanishHolderInZyg::BoundExtended::ToShowMetric} follows.

    With \eqref{Eqn::Zyg::Holder::VanishHolderInZyg::BoundExtended::ToShowMetric} in hand, we complete the proof of 
    \eqref{Eqn::Zyg::Holder::VanishHolderInZyg::BoundExtended::ToShow}.  We have,
    \begin{equation}\label{Eqn::Zyg::Holder::VanishHolderInZyg::BoundExtended::Tmp1}
        \MetricZZde[x][y]^{-s}\left| g(x)-g(y) \right|
        =\MetricZZde[x][y]^{-s}\left| g(y) \right|
        =\MetricZZde[x][y]^{-s}\left| g(z)-g(y) \right|
        \lesssim \MetricZZde[z][y]^{-s}\left| g(z)-g(y) \right|.
    \end{equation}
    By since \(z,y\in V\cap \ManifoldN\subseteq \Omega_2\), 
    Proposition \ref{Prop::VectorFields::Scaling::AmbientVolAndMetricEquivalnce} \ref{Item::VectorFields::Scaling::AmbientVolAndMetricEquivalnce::MetricWequalsMetricZ}
    shows \(\MetricZZde[z][y]\approx \MetricWWdv[z][y]\).
    Plugging this into \eqref{Eqn::Zyg::Holder::VanishHolderInZyg::BoundExtended::Tmp1} yields \eqref{Eqn::Zyg::Holder::VanishHolderInZyg::BoundExtended::ToShow}
    and completes the proof of \eqref{Eqn::Zyg::Holder::VanishHolderInZyg::BoundExtended}.

    By Lemma \ref{Lemma::Zyg::Holder::ManifoldsWithoutBoundary}, we have
    \(g\in \ZygSpace{s}[\Compact][\FilteredSheafFh]\), with
    \begin{equation*}
        \ZygNormFh{g}{s}\lesssim \HolderNormFh{g}{0}{s}\lesssim \HolderNormF{f}{0}{s}.
    \end{equation*}
    By Corollary \ref{Cor::Spaces::Extend::Consequences::RestrictionSpaceDefn},
    \(f=g\big|_{\ManifoldN}\in \ZygSpaceCompactF{s}\) with
    \begin{equation*}
        \ZygNormF{f}{s}\lesssim \ZygNormFh{g}{s}\lesssim \HolderNormF{f}{0}{s},
    \end{equation*}
    completing the proof.
\end{proof}

\begin{lemma}\label{Lemma::Zyg::Holder::TraceHolder}
    For \(s\in [0,1]\), the map \(f\mapsto f\big|_{\BoundaryNncF}\)
    is continuous
    \begin{equation*}
        \HolderSpaceCompactF{0}{s}\rightarrow \HolderSpace{0}{s}[\Compact\cap \BoundaryN][\RestrictFilteredSheaf{\LieFilteredSheafF}{\BoundaryNncF}].
    \end{equation*}
\end{lemma}
\begin{proof}
    The result is trivial for \(s=0\), so we assume \(s\in (0,1]\).
    Let \(\Omega_1\Subset \Omega\) be open and relatively compact with \(\Compact\Subset \Omega\).
    Let \(\Gamma_1:=\Omega_1\cap \BoundaryN\) and \(\Gamma:=\Omega\cap\BoundaryN\),
    so that \(\Compact\cap\BoundaryN\Subset \Gamma_1\Subset\Gamma\Subset\BoundaryNncF\).
    Let \(\VVde:=\left\{ \left( V_1,\Vde_1 \right),\ldots, \left( V_q,\Vde_q \right) \right\}\subset \VectorFields{\Gamma}\times \Zg\)
    be H\"ormander vector fields with formal degrees on \(\Gamma\) such that
    \(\RestrictFilteredSheaf{\LieFilteredSheafF}{\Gamma}=\FilteredSheafGenBy{\VVde}\) 
    (see Lemma \ref{Lemma::Filtrations::GeneratorsOnRelCptSet} and Proposition \ref{Prop::Filtrations::RestrictingFiltrations::LieFIsHormanderFiltration}).
    For \(x,y\in \overline{\Gamma_1}\), we have \(\MetricVVde[x][y]\approx \MetricWWdv[x][y]\)
    (see \cite[Theorem \ref*{CC::Thm::Sheaves::Metrics::BoundaryMetricEquivalence}]{StreetCarnotCaratheodoryBallsOnManifoldsWithBoundary}).
    Therefore, using Lemma \ref{Lemma::Zyg::Holder::NormWellDefined},
    \begin{equation*}
    \begin{split}
         &\BHolderNorm{f\big|_{\BoundaryNncF} }{0}{s}[\RestrictFilteredSheaf{\LieFilteredSheafF}{\BoundaryNncF}]
         \approx \sup_{x\in \Gamma_1} |f(x)| + \sup_{\substack{x,y\in \Gamma_1 \\ x\ne y}} \MetricVVde[x][y]^{-s} \left| f(x)-f(y) \right|
         \\&\approx \sup_{x\in \Gamma_1} |f(x)| + \sup_{\substack{x,y\in \Gamma_1 \\ x\ne y}} \MetricWWdv[x][y]^{-s} \left| f(x)-f(y) \right|
         \\&\leq \sup_{x} |f(x)| + \sup_{\substack{x,y\in \Omega_1 \\ x\ne y}} \MetricWWdv[x][y]^{-s} \left| f(x)-f(y) \right|
         \approx \HolderNormF{f}{0}{s}[\FilteredSheafF].
    \end{split}
    \end{equation*}
\end{proof}

\begin{lemma}\label{Lemma::Zyg::Holder::TraceHolderToZyg}
    For \(s\in (0,1)\), the map \(f\mapsto f\big|_{\BoundaryNncF}\)
    is continuous
    \begin{equation*}
        \HolderSpaceCompactF{0}{s}\rightarrow \ZygSpace{s}[\Compact\cap \BoundaryN][\RestrictFilteredSheaf{\LieFilteredSheafF}{\BoundaryNncF}].
    \end{equation*}
\end{lemma}
\begin{proof}
    Since \(\BoundaryNncF\) is a manifold without boundary, Lemma \ref{Lemma::Zyg::Holder::ManifoldsWithoutBoundary}
    shows
    \begin{equation*}
        \HolderSpace{0}{s}[\Compact\cap \BoundaryN][\RestrictFilteredSheaf{\LieFilteredSheafF}{\BoundaryNncF}]
        =
        \ZygSpace{s}[\Compact\cap \BoundaryN][\RestrictFilteredSheaf{\LieFilteredSheafF}{\BoundaryNncF}],
    \end{equation*}
    with equality of topologies;
    see also Proposition \ref{Prop::Filtrations::RestrictingFiltrations::LieFIsHormanderFiltration}.
    The result now follows from Lemma \ref{Lemma::Zyg::Holder::TraceHolder}.
\end{proof}

\begin{lemma}\label{Lemma::Zyg::Holder::ZygAgreeWithHolderOnBoundary}
    Fix \(s\in (0,1)\) and \(\Omega_1\Subset \Omega\) open and relatively compact with
    \(\Compact\Subset \Omega_1\). For \(f\in \HolderSpaceCompactF{0}{s}\), \(\exists g\in \ZygSpace{s}[\overline{\Omega_1}][\FilteredSheafF]\), 
    with \(\ZygNormF{g}{s}\lesssim \HolderNormF{f}{0}{s}\) and \(f\big|_{\BoundaryN}=g\big|_{\BoundaryN}\).
\end{lemma}
\begin{proof}
    Set \(g:=\TraceInverseMap\left( f\big|_{\BoundaryNncF} \right)\),
    where \(\TraceInverseMap\) is the map from Theorem \ref{Thm::Trace::Dirichlet::MainInverseThm} with \(\sB=(1)\),
    \(\Omega_2\) replaced with \(\Omega_1\) and \(\Omega_1\) replaced with any relatively compact, open subset of \(\Omega_1\)
    containing \(\Compact\).
    That the linear map \(f\mapsto g\) is continuous,
    \(\HolderSpaceCompactF{0}{s}\rightarrow \ZygSpace{s}[\overline{\Omega_1}][\FilteredSheafF]\), 
    follows by combining Lemma \ref{Lemma::Zyg::Holder::TraceHolderToZyg}
    and  Theorem \ref{Thm::Trace::Dirichlet::MainInverseThm} \ref{Item::Trace::Dirichlet::MainInverseThm::ContinuousOnSpaces};
    and therefore \(\ZygNormF{g}{s}\lesssim \HolderNormF{f}{0}{s}\).
    That \(f\big|_{\BoundaryN}=g\big|_{\BoundaryN}\) follows from 
    Theorem \ref{Thm::Trace::Dirichlet::MainInverseThm} \ref{Item::Trace::Dirichlet::MainInverseThm::IsInverse}
    (see also Remark \ref{Rmk::Zyg::Decomp::TraceOnZyg}); recall \(\supp(f\big|_{\BoundaryN})\subseteq \Compact\cap\BoundaryN\subseteq \BoundaryNncF\).
\end{proof}

\begin{lemma}\label{Lemma::Zyg::Holder::HolderAsSumOfZyg}
    Fix \(s\in (0,1)\) and \(\Omega_1\Subset \Omega\) open and relatively compact with
    \(\Compact\Subset \Omega_1\). For \(f\in \HolderSpaceCompactF{0}{s}\),
    \(\exists g_1,g_2\in \ZygSpace{s}[\overline{\Omega_1}][\FilteredSheafF]\), 
    with \(\ZygNormF{g_1}{s},\ZygNormF{g_2}{s}\lesssim \HolderNormF{f}{0}{s}\) and 
    \(f=g_1+g_2\).
\end{lemma}
\begin{proof}
    Let \(g_1\) be \(g\) from Lemma \ref{Lemma::Zyg::Holder::ZygAgreeWithHolderOnBoundary}; so that
    \(g_1\in \ZygSpace{s}[\overline{\Omega_1}][\FilteredSheafF]\) with \(\ZygNormF{g_1}{s}\lesssim \HolderNormF{f}{0}{s}\)
    and \(f\big|_{\BoundaryN}=g_1\big|_{\BoundaryN}\).
    By Lemma \ref{Lemma::Zyg::Holder::ZygInHolder}, \(g_1\in \HolderSpace{0}{s}[\overline{\Omega_1}][\FilteredSheafF]\),
    with \(\HolderNormF{g_1}{0}{s}\lesssim \ZygNormF{g_1}{s}\lesssim \HolderNormF{f}{0}{s}\).

    Set \(g_2=f-g_1\in \HolderSpace{0}{s}[\overline{\Omega_1}][\FilteredSheafF]\)
    so that \(\HolderNormF{g_2}{0}{s}\leq \HolderNormF{f}{0}{s}+ \HolderNormF{g_1}{0}{s}\lesssim \HolderNormF{f}{0}{s}\).
    Since \(g_2\big|_{\BoundaryN}=0\), Lemma \ref{Lemma::Zyg::Holder::VanishHolderInZyg} shows
    \(g_2\in \ZygSpace{s}[\overline{\Omega_1}][\FilteredSheafF]\)
    with \(\ZygNormF{g_2}{s}\lesssim \HolderNormF{f}{0}{s}\), completing the proof.
\end{proof}

\begin{lemma}\label{Lemma::Zyg::Holder::HolderInZyg}
    For \(0<s<1\), \(\HolderSpaceCompactF{0}{s}\subseteq \ZygSpaceCompactF{s}\), and the inclusion is continuous.
\end{lemma}
\begin{proof}
    Let \(\Omega_1\Subset \Omega\) be a relatively compact, open set with \(\Compact\Subset \Omega_1\).
    By Lemma \ref{Lemma::Zyg::Holder::HolderAsSumOfZyg}, \(f\in \ZygSpace{s}[\overline{\Omega_1}][\FilteredSheafF]\)
    with \(\ZygNormF{f}{s}\lesssim \HolderNormF{f}{0}{s}\). Since \(\supp(f)\subseteq \Compact\),
    by definition (see Definition \ref{Defn::Zyg::Holder::HolderSpace}), Proposition \ref{Prop::Spaces::Containment}
    shows \(f\in \ZygSpaceCompactF{s}\), completing the proof.
\end{proof}

\begin{proof}[Proof of Theorem \ref{Thm::Zyg::Holder::MainThm}]
    This follows by combining Lemmas \ref{Lemma::Zyg::Holder::ZygInHolder} and \ref{Lemma::Zyg::Holder::HolderInZyg}.
\end{proof}

    \subsection{Compositions}\label{Section::Zyg::Compositions}
    Fix \(N\in \Zg\). For functions \(u\in \ZygSpaceCompactF{s}[\RN]\), we wish to consider compositions of the form
\(F(x,u(x))\), where \(F(x,\zeta):\ManifoldN\rightarrow \C\) is a sufficiently smooth function.

\begin{remark}
    Such compositions are important for nonlinear boundary value problems. A general nonlinear boundary value problem
    takes the form
    \begin{equation}\label{Eqn::Zyg::Comp::ExampleNonlinearInterior}
        F\left( x, \left\{ W^{\alpha}u(x) \right\}_{\DegWdv{\alpha}\leq \kappa} \right) =f(x) \text{ on }\InteriorN,
    \end{equation}
    \begin{equation}\label{Eqn::Zyg::Comp::ExampleNonlinearBoundary}
        G_j\left(x', \left\{ W^{\alpha} u\big|_{\BoundaryN}(x') \right\}_{\DegWdv{\alpha}\leq \kappa_j}\right)=g_j(x')\text{ on }\BoundaryN,
    \end{equation}
    where \(F(x,\zeta)\), \(f(x)\), \(G_j(x',\zeta)\), and \(g_j(x')\) (\(j=1,\ldots,L\)) are given.
    Let \(u\) be the unknown function, and suppose that near a point \(x_0\in \BoundaryNncF\), \(u\in \ZygSpace{s+\kappa}[\FilteredSheafF]\),
    for some \(s>0\).
    Then, after introducing appropriate cut-off functions, \eqref{Eqn::Zyg::Comp::ExampleNonlinearInterior}
    is of the form studied in this section. Moreover, in light of Theorem \ref{Thm::Trace::ForwardMap},
    if \(\kappa_j\leq \kappa\), \(\forall j\), 
    and again after introducing appropriate cut-off functions,
    \eqref{Eqn::Zyg::Comp::ExampleNonlinearBoundary}  is also of the form studied in this section (with \(\ManifoldN\)
    replaced by the manifold without boundary \(\BoundaryNncF\) and \(\FilteredSheafF\) replaced
    by \(\RestrictFilteredSheaf{\LieFilteredSheafF}{\BoundaryNncF}\)). Using the results in this section,
    such nonlinear boundary value problems will be studied in a future paper.
\end{remark}

Informally, the class of functions \(F(x,\zeta)\) we consider are \(\ZygSpaceCompactF{s}\) in the \(x\)-variable
and \(\ZygSpace{t}[\R^N]\) in the \(\zeta\)-variable, where \(\ZygSpace{t}[\R^N]\) denotes the classic
Zygmund--H\"older space on \(\R^N\) of order \(t>0\) (see \cite[Sections 2.2.2 and 2.5.12]{TriebelTheoryOfFunctionSpaces}).
To rigorously define these product Zygmund--H\"older spaces, let 
\(D_j\) be as in Section \ref{Section::Zyg::Norm}.
For \(l\in \Zgeq\), let \(\Dt_l\) be given by \eqref{Eqn::Spaces::Classical::DefineDjOnWholeSpace} (with \(n\) replaced by \(N\));
so that 
\(\Dt_l\) acts on \(\TemperedDistributions[\RN]\) and 
\(\sum_{l\in \Zgeq} \Dt_l = I\).
Set \(\Dh_{j,l}=D_j\otimes \Dt_l\).

\begin{definition}
    For \(s,t>0\), we let \(\ProdZygSpaceCompactRNF{s}{t}\) denote the space of all
    \(F(x,\zeta)\in \DistributionsZero[\ManifoldN\times \RN]\) such that \(\supp(F)\subseteq \Compact\times \RN\) 
    and the following norm is finite:
    \begin{equation*}
        \ProdZygNormF{F}{s}{t}:=\sup_{j,l\in \Zgeq} 2^{js+lt} \BLpNorm{\Dh_{j,l}F}{\infty}[\ManifoldN\times \RN].
    \end{equation*}
\end{definition}

\begin{theorem}\label{Thm::Zyg::Composition::MainComposition}
    Let \(s,t>0\), \(F\in \ProdZygSpaceCompactRNF{s}{t}\), \(u,v\in \ZygSpaceCompactF{s}[\RN]\), and \(w\in \ZygSpaceCompactF{s}\).
    Then,
    \begin{enumerate}[(i)]
        \item If \(t>\floor{s}+2\), then
            \(F(x,u(x)+\zeta)\in \ProdZygSpaceCompactRNF{s}{t-\floor{s}-2}\) and
            \begin{equation*}
                \ProdZygNormF{ F(x,u(x)+\zeta)}{s}{t-\floor{s}-2}
                \lesssim \ProdZygNormF{F}{s}{t} \left( 1+\ZygNormF{u}{s} \right) \left( 1+\LpNorm{u}{\infty} \right)^{\floor{s}+1}.
            \end{equation*}
        \item If \(t>\floor{s}+2\), then \(F(x,u(x))\in \ZygSpaceCompactF{s}\) and
            \begin{equation*}
                \ZygNormF{F(x,u(x))}{s}
                \lesssim
                \left( 1+\ZygNormF{u}{s} \right)
                \left( 1+\LpNorm{u}{\infty} \right)^{\floor{s}+1}.
            \end{equation*}
        \item\label{Item::Zyg::Composition::MainComposition::DotProduct} \(u\cdot v\in \ZygSpaceCompactF{s}\) with 
            \begin{equation*}
                \ZygNormF{u\cdot v}{s} \lesssim \ZygNormF{u}{s}\LpNorm{v}{\infty}+\LpNorm{u}{\infty}\ZygNormF{v}{s}.
            \end{equation*}
        \item If \(t>\floor{s}+2\), \(F(x,u(x)+\zeta)w(x)\in \ProdZygSpaceCompactRNF{s}{t-\floor{s}-2}\), with
            \begin{equation*}
                \ProdZygNormF{F(x,u(x)+\zeta)w(x)}{s}{t-\floor{s}-2}
                \lesssim \ProdZygNormF{F}{s}{t} \left( \ZygNormF{w}{s} + \ZygNormF{u}{s}\LpNorm{w}{\infty} \right)\left( 1+\LpNorm{u}{\infty} \right)^{\floor{s}+1}.
            \end{equation*}
        \item If \(t>\floor{s}+3\),
            \begin{equation*}
                \begin{split}
                    &\ProdZygNormF{F(x,u(x)+\zeta)-F(x,v(x)+\zeta)}{s}{t-\floor{s}-3}
                    \\&\lesssim \ProdZygNormF{F}{s}{t} 
                    \left( \ZygNormF{u-v}{s} + \left( \ZygNormF{u}{s}+\ZygNormF{v}{s} \right)\LpNorm{u-v}{\infty} \right)
                    \left( 1+\LpNorm{u}{\infty}+\LpNorm{v}{\infty} \right)^{\floor{s}+1}.
                \end{split}
            \end{equation*}
        \item If \(t>\floor{s}+3\)
            \begin{equation*}
            \begin{split}
                 &\ZygNormF{ F(x,u(x))-F(x,v(x))}{s}
                 \\&\lesssim \ProdZygNormF{F}{s}{t} 
                    \left( \ZygNormF{u-v}{s} + \left( \ZygNormF{u}{s}+\ZygNormF{v}{s} \right)\LpNorm{u-v}{\infty} \right)
                    \left( 1+\LpNorm{u}{\infty}+\LpNorm{v}{\infty} \right)^{\floor{s}+1}.                 
            \end{split}
            \end{equation*}
        \item Fix \(L\in \Zgeq\), \(L\geq 2\) and suppose \(\partial_{\zeta}^{\beta} F(x,0)=0\), \(\forall 0<|\beta|<L\)
            and \(t>\floor{s}+2+L\). Then,
            \begin{equation*}
            \begin{split}
                 &\ZygNormF{ F(x,u(x))-F(x,v(x))}{s}
                 \\&\lesssim \ProdZygNormF{F}{s}{t} 
                 \left( 1+\LpNorm{u}{\infty}+\LpNorm{v}{\infty} \right)^{\floor{s}+2}
                 \left( \LpNorm{u}{\infty}+\LpNorm{v}{\infty} \right)^{L-2}
                 \left( \ZygNormF{u}{s} + \ZygNormF{v}{s} \right)\LpNorm{u-v}{\infty}
                 \\&\quad+\ProdZygNormF{F}{s}{t} 
                 \left( \LpNorm{u}{\infty}+\LpNorm{v}{\infty} \right)^{L-1} \ZygNormF{u-v}{s}.
            \end{split}
            \end{equation*}
    \end{enumerate}
    The implicit constants do not depend on \(F\), \(u\), \(v\), or \(w\), but may depend on any other ingredient in the theorem.
\end{theorem}
\begin{proof}[Comments on the proof]
    This is a simple reprise of the proof of \cite[Theorem 7.5.2]{StreetMaximalSubellipticity} (in the case \(\nu=1\)), using 
    Lemma \ref{Lemma::Spaces::Elem::PElem::PElemOpsBoundedOnLp},
    Proposition \ref{Prop::Spaces::Elem::Elem::MainProps},
    and \eqref{Eqn::Zyg::Norm::DerivPjBounded} in place of the corresponding results in \cite{StreetMaximalSubellipticity}.
\end{proof}

\begin{remark}
    As in Remark \ref{Rmk::Zyg::Tame}, the estimates in Theorem \ref{Thm::Zyg::Composition::MainComposition}
    are tame; which is essential for applications to nonlinear PDEs.
\end{remark}

\begin{remark}
    As in Remark \ref{Rmk::Zyg::Algebra::ExtensionDoesntWork}, one can obtain the qualitative
    consequences of Theorem \ref{Thm::Zyg::Composition::MainComposition} by combining
    \cite[Theorem 7.5.2]{StreetMaximalSubellipticity} (the result on manifolds without boundary)
    with the extension map from Theorem \ref{Thm::Spaces::Extension}. However, this does not
    establish the tame estimates. Instead, one can merely recreate the proof from \cite{StreetMaximalSubellipticity}
    in this setting to obtain the desired tame estimates.
\end{remark}

\begin{remark}
    See \cite[Corollary 4.5]{MemanaRegularityTheoremForFullyNonlinearMaximallySubellipticPDE} for a version of
    Theorem \ref{Thm::Zyg::Composition::MainComposition} for Sobolev spaces on manifolds without boundary. 
    It seems likely that a similar result is true on manifolds with boundary near a non-characteristic point of the boundary,
    though we do not pursue this here.
\end{remark}

\bibliographystyle{amsalpha}

\bibliography{bibliography}

\ifnotanonymous
\center{\it{University of Wisconsin-Madison, Department of Mathematics, 480 Lincoln Dr., Madison, WI, 53706}}

\center{\it{street@math.wisc.edu}}
\fi

\center{MSC 2020:  46E35 (Primary), 46E36, 35G15, and 35H20 (Secondary)}

\end{document}